\documentclass[12pt]{amsart}

\usepackage{longtable}
\usepackage{float}
\restylefloat{table}

\newcommand{\hide}[1]{}

\usepackage{amssymb}
\usepackage{amsbsy}
\usepackage{amscd}
\usepackage{amsmath}
\usepackage{amsthm}
\usepackage{bbold}
\usepackage{mathrsfs}
\usepackage{url}
\usepackage{graphicx}

\numberwithin{equation}{subsection}

\theoremstyle{plain}
\newtheorem{thm}{Theorem}[subsection]
\newtheorem{prop}[thm]{Proposition}

\newtheorem{conj}[thm]{Conjecture}
\newtheorem{assumption}[thm]{Assumption}
\newtheorem{thm-defi}[thm]{Theorem/Definition}
\newtheorem{cor}[thm]{Corollary}
\newtheorem{lem}[thm]{Lemma}

\theoremstyle{definition}

\newtheorem{defi}[thm]{Definition}
\newtheorem{rem}[thm]{Remark}

\newtheorem{question}[thm]{Question}
\newtheorem{example}[thm]{Example}

\newtheorem{construction}[thm]{Construction}

\newcommand{\A}{{\mathcal A}}
\newcommand{\B}{{\mathcal B}}
\newcommand{\C}{{\mathcal C}}
\newcommand{\CC}{{\mathbb C}}

\newcommand{\D}{{\mathcal D}}
\newcommand{\E}{{\mathcal E}}
\newcommand{\F}{{\mathcal F}}
\newcommand{\G}{{\mathcal G}}

\renewcommand{\H}{{\mathcal H}}

\newcommand{\I}{{\mathcal I}}
\newcommand{\J}{{\mathcal J}}

\newcommand{\LB}{{\mathcal L}}
\newcommand{\Q}{{\mathcal Q}}
\newcommand{\QQ}{{\mathbb Q}}
\newcommand{\R}{{\mathcal R}}
\newcommand{\RR}{{\mathbb R}}

\newcommand{\T}{{\mathcal T}}
\newcommand{\M}{{\mathcal M}}

\newcommand{\U}{{\mathcal U}}
\newcommand{\V}{{\mathcal V}}

\newcommand{\X}{{\mathcal X}}

\newcommand{\Z}{{\mathcal Z}}
\newcommand{\ZZ}{{\mathbb Z}}
\renewcommand{\P}{{\mathcal P}}
\newcommand{\PP}{{\mathbb P}}
\newcommand{\ann}{{\rm ann}}
\newcommand{\cm}{\eta}

\newcommand{\Integers}{{\mathbb Z}}
\newcommand{\ComplexNumbers}{{\mathbb C}}

\newcommand{\LieAlg}[1]{{\mathfrak #1}}

\newcommand{\Spin}{{\rm Spin}}

\newcommand{\linsys}[1]{{\mid}#1{\mid}}

\newcommand{\fS}{{\mathfrak S}}

\newcommand{\IsomRightArrow}{\stackrel{\cong}{\rightarrow}}

\newcommand{\RightArrowOf}[1]{\stackrel{#1}{\rightarrow}}

\newcommand{\LongRightArrowOf}[1]{\stackrel{#1}{\longrightarrow}}

\newcommand{\StructureSheaf}[1]{{\mathcal O}_{#1}}
\newcommand{\EndProof}{\hfill  $\Box$}
\newcommand{\restricted}[2]{#1_{\mid_{#2}}}

\newcommand{\rank}{{\rm rank}}

\newcommand{\Pic}{{\rm Pic}}

\newcommand{\Sym}{{\rm Sym}}
\newcommand{\Ext}{{\rm Ext}}

\newcommand{\SheafTor}{{\mathcal Tor}}
\newcommand{\Hom}{{\rm Hom}}
\newcommand{\Aut}{{\rm Aut}}
\newcommand{\End}{{\rm End}}

\newcommand{\SheafHom}{{\mathcal H}om}
\newcommand{\SheafEnd}{{\mathcal E}nd}
\newcommand{\SheafExt}{{\mathcal E}xt}

\newcommand{\Ideal}[1]{{\mathcal I}_{#1}}

\newcommand{\Contract}{\rfloor}

\newcommand{\Choose}[2]
{\left(\!\!\begin{array}{c}#1\\#2\end{array}\!\!\right)}

\newcommand{\one}{{\mathbb 1}}
\newcommand{\Spec}{{\rm Spec}}
\renewcommand{\span}{{\rm span}}

\input xy
\xyoption{all}

\begin{document}
\title[Abelian $2n$-folds of Weil type]
{Cycles on abelian $2n$-folds of Weil type from secant sheaves on abelian $n$-folds}
%{Algebraic Weil classes on abelian $2n$-folds from secant sheaves on abelian $n$-folds}
\author{Eyal Markman}
\address{Department of Mathematics and Statistics, 
University of Massachusetts, Amherst, MA 01003, USA}
\email{markman@umass.edu}

\date{\today}
\subjclass[2010]{14C25, 14C30, 14K22}
%\rightline{\today}
%\begin{center}
%\begin{Large}
%{\bf 
%\noindent
%The monodromy of generalized Kummer varieties}
%\end{Large}
%\\
%Eyal Markman
%\end{center}

\begin{abstract}
Let $X$ be an abelian $n$-fold, $n\geq 2$,  and $\hat{X}$ its dual abelian $n$-fold. 
Endow $V\!\!=\!\!H^1(X,\ZZ)\!\oplus \! H^1(\hat{X},\ZZ)$ with the natural symmetric bilinear pairing.
The even cohomology $H^{ev}(X,\ZZ)$ is the half spin representation of $\Spin(V)$. 
Fix an integer $d>0$ and set $K\!:=\!\QQ(\sqrt{-d})$.
A coherent sheaf $F$ on $X$ is a {\em $K$-secant sheaf}, if $ch(F)$ belongs to a $2$-dimensional subspace
$P$ of $H^{ev}(X,\QQ)$ spanned by Hodge classes, such that the line $\PP(P)$ intersects the spinorial variety in
$\PP[H^{ev}(X,K)]$ 
along two distinct complex conjugate points defined over $K$. The $K$-secant  $P$ determines an embedding  
$\eta:K\rightarrow \End_\QQ(X\!\!\times\!\!\hat{X})$ and a rational $(1,1)$-form $h$ on $X\!\!\times\!\!\hat{X}$. The triple $(X\!\!\times\!\!\hat{X},\eta,h)$ is a polarized abelian variety of Weil type, for a non-empty open subset of such $K$-secants.

Let $F_1$, $F_2$ be non-zero coherent sheaves on $X$ with $ch(F_i)$ in such a $K$-secant $P$. 
Orlov constructed a derived equivalence 
$
\Phi:D^b(X\!\!\times\!\! X)\rightarrow D^b(X\!\!\times\!\!\hat{X})
$ 
categorifying  the isomorphism of two $\Spin(V)$-representations,
the tensor square $H^*(X\!\!\times\!\! X,\ZZ)$ 
of the spin  representation, 
%$H^*(X,\ZZ)$, 
and $\wedge^*V$. Assume that the rank of $E:=\Phi(F_1^\vee\boxtimes F_2)$ is non-zero.
We prove that the characteristic class $\exp\left(-\frac{c_1(E)}{\rank(E)}\right)ch(E)$ 
%of the projective bundle $\PP(E)$ 
remains of Hodge type under every deformation of 
$(X\!\!\times\!\!\hat{X},\eta,h)$ as a polarized abelian variety of Weil type.  The algebraicity of the Hodge-Weil classes of a deformation of the triple would follow, if $E$ deforms as well as a possibly twisted object.
%We exhibit a class $\lambda\in H^{1,1}(X\!\!\times\!\!\hat{X},\QQ)$, unique modulo $\QQ h$, such that 
%$\exp(\lambda)ch(E)$ remains of Hodge type under every deformation $(A,\eta',h')$ of 
%$(X\!\!\times\!\!\hat{X},\eta,h)$ as a polarized abelian variety of Weil type.  Furthermore, 
%the Hodge-Weil classes of $(A,\eta',h')$ are algebraic, if $E$ deforms to an object on $A$  (possibly twisted).

When $n=3$ we construct for every positive integer $d$ an example, where $K=\QQ(\sqrt{-d})$ and the derived dual $E$ of $\Phi(F_1^\vee\boxtimes F_2)[1]$ is a simple reflexive sheaf over $X\!\!\times\!\!\hat{X}$. 
%The characteristic class $\exp(-c_1(E)/8d)ch(E)$ of the projective bundle $\PP(E)$ thus remains of Hodge type under every deformation of $(X\!\!\times\!\!\hat{X},\eta,h)$ of Weil type. 
%We show that $\PP(E)$ deforms with $(X\!\!\times\!\!\hat{X},\eta,h)$ to first order in every direction tangent to the $9$-dimensional moduli space of polarized abelian $6$-folds of Weil type.  
We prove that $E$ deforms with $(X\!\!\times\!\!\hat{X},\eta,h)$ locally in the $9$-dimensional moduli space of  polarized abelian $6$-folds of Weil type of discriminant $-1$, using the Semiregularity theorem of Buchweitz-Flenner.
%conditional on a conjecture that an unobstructedness theorem of Buchweitz-Flenner for deformations of semiregular coherent sheaves 
%generalizes to semiregular twisted reflexive sheaves.  

The Hodge Conjecture for abelian fourfolds is known to follow from the above result.
\end{abstract}

\maketitle

\setcounter{tocdepth}{3}
\tableofcontents
%***************************************************************************
% Introduction 
%***************************************************************************
\section{Introduction}
%***************************************************************************
% Introduction 
%***************************************************************************
\subsection{Abelian varieties of Weil-type}
\label{sec-introduction-abelian-varieties-of-Weil-type}
A $2n$-dimensional abelian variety $A$ is of {\em Weil-type}, if it admits an embedding $\eta:K\rightarrow \End_\QQ(A)$, where $K=\QQ(\sqrt{-d})$ for some positive integer $d$, such that each of the eigenspaces $W$ and $\bar{W}$ 
of $\eta(\sqrt{-d})$, with eigenvalues $\sqrt{-d}$ and $-\sqrt{-d}$, intersects $H^{1,0}(A)$ in an $n$-dimensional subspace (see \cite{weil}). 
In that case $\wedge^{2n}W\oplus \wedge^{2n}\bar{W}$ is a $Gal(K/\QQ)$-invariant $2$-dimensional subspace of $H^{2n}(A,K)$, corresponding to 
a $2$-dimensional subspace $\hat{HW}\subset H^{n,n}(A,\QQ)$. The rational $(n,n)$-classes in $\hat{HW}$ are called {\em Hodge-Weil classes}. The Hodge conjecture predicts that $\hat{HW}$ is spanned by classes of algebraic cycles. A {\em polarized abelian variety of Weil type} is a triple $(A,\eta,h)$, where $(A,\eta)$ is an abelian variety of Weil type and $h\in H^{1,1}(A,\QQ)$ is an ample class, such that $\eta(k)\in \End_\QQ(A)$ maps $h$ to $Nm(k)h$, where $Nm:K\rightarrow \QQ$ is the norm map \cite{van-Geemen,weil}. 

Polarized abelian varieties of Weil type admit a natural $K$-valued hermitian form $H:H_1(A,\QQ)\times H_1(A,\QQ)\rightarrow K$.
The determinant of the matrix of $H$, with respect to some $K$-basis of $H_1(A,\QQ)$, is an element of $\QQ^\times$ and its image in $\QQ^\times/Nm(K^\times)$ is called the {\em discriminant} $\det H$ of $(A,\eta,h)$. 
The moduli space of $2n$-dimensional abelian varieties of Weil type is $n^2$-dimensional. The generic abelian variety of Weil type has a cyclic Neron-Severi group, but a $3$-dimensional $H^{n,n}(A,\QQ)$ \cite{weil} (see also \cite[Th. 6.12]{van-Geemen}).
The triple $(n,K,\det H)$, consisting of half the dimension of $A$, an imaginary quadratic number field $K$, and the discriminant $\det H$, is a discrete invariant of a component of the moduli space, which determines it up to isogenies of abelian varieties of Weil type \cite[Th. 5.2(3)]{van-Geemen}.

The algebraicity of the Hodge-Weil classes was proved for abelian fourfolds of Weil type with $K=\QQ[\sqrt{-3}]$ and arbitrary discriminant in \cite{schoen}, where it is also proved for abelian sixfolds with $K=\QQ[\sqrt{-3}]$ and trivial discriminant.
The algebraicity is proved for fourfolds with $K=\QQ[\sqrt{-1}]$ and discriminant $-1$  in \cite{schoen1} (see also \cite{van-Geemen} for another proof).
It was later proved for sixfolds with $K=\QQ[\sqrt{-1}]$ and discriminant $-1$ in \cite{koike}, which implies the result for fourfolds, $K=\QQ[\sqrt{-1}]$, and arbitrary discriminant. The algebraicity is proved in \cite{markman-generalized-kummers} for fourfolds, arbitrary imaginary quadratic number field $K$, and discriminant $1$. 

%***************
% Hide
%***************
\hide{
The construction of \cite{markman-generalized-kummers} starts with an abelian surface $X$, a $2$-dimensional subspace $P$ of $H^{even}(X,\QQ)$, positive definite with respect to the Mukai pairing, and a stable sheaf $F$ with Chern character in $P$. We introduce\footnote{In fact, $A$ is isogenous to the intermediate Jacobian of a hyper-K\"{a}hler variety $Kum(ch(F))$ of generalized Kummer type, which is a fiber 
 of the albanese morphism of the moduli space $M(ch(F))$ of stable sheaves on $X$ with Chern character $ch(F)$. O'Grady observed earlier that the intermediate Jacobian of $Kum(ch(F))$ is an abelian variety of Weil type \cite{ogrady}.
 }
 the structure of an abelian variety of Weil type on $A:=X\times \hat{X}$ with complex multiplication $\eta:K\rightarrow \End_\QQ(A)$ by the imaginary quadratic number field $K$ of definition of the two isotropic lines in $P\otimes_\QQ\CC$. The complex multiplication arrises naturally using the representation theory of an arithmetic group of type $\Spin(4,4)$ arising as the {\em derived monodromy group} of $X$, namely the group generated by the cohomological action on $H^*(X,\ZZ)$ of all monodromy operators and all derived equivalences for all complex structures on $X$. The choice of the stable sheaf $F$ is then used to construct a reflexive coherent sheaf $E$ on $X\times \hat{X}$.
 % of positive rank $r$. We show that the characteristic class $\kappa(E):=\exp(-c_1(E)/r)ch(E)$ remains of Hodge-type over 
 %the $4$-dimensional space of deformations of $(A,\eta)$ as an abelian variety of Weil type.  
 We use techniques from hyper-K\"{a}hler geometry to deform
 the pair $(A,E)$ over a $5$-dimensional moduli space, which contains the $4$-dimensional moduli space of deformations $(A',\eta')$ of $(A,\eta)$ as an abelian variety of Weil type. Finally we show that the $\eta'(K)$ translates of a characteristic class of the deformed $E'$ over $A'$ span the space of Hodge-Weil classes 
 of $(A',\eta')$. The algebraicity of the Hodge-Weil classes of $(A',\eta')$ follows.
 In this paper 
 we generalize the above construction for abelian $n$-folds $X$ without the use of hyper-K\"{a}hler geometry techniques, which seem unavailable for $n>2$.
 %***************
% End Hide
%***************
}
 %***************************************************************************
% 
%***************************************************************************
\subsection{Abelian varieties of Weil-type from $K$-secants}
\label{sec-abelian-varieties-of-Weil-type-introduction}
Let $X$ be a projective abelian $n$-fold and set $\hat{X}:=\Pic^0(X)$. 
Set 
\begin{equation}
\label{eq-V}
V:=H^1(X,\Integers)\oplus H^1(\hat{X},\Integers). 
\end{equation}
Then $V$ is endowed with the symmetric non-degenerate unimodular bilinear pairing 
\begin{equation}
\label{eq-pairing-on-V}
((w_1,\theta_1),(w_2,\theta_2))_V:=\theta_1(w_2)+\theta_2(w_1),
\end{equation}
where we used the natural isomorphism to identify $H^1(\hat{X},\Integers)$ with $H^1(X,\Integers)^*$.
$V$ has rank $4n$ and is isometric to the orthogonal direct sum $U^{\oplus 2n}$ of $2n$ copies of the even unimodular rank $2$ lattice $U$ of signature $(1,1)$. 
Set $S:=H^*(X,\Integers)$, $S^+:=\oplus_{i=0}^{2n} H^{2i}(X,\Integers)$, and $S^-:=\oplus_{i=0}^{2n-1} H^{2i+1}(X,\Integers)$. Then $S^+$ and $S^-$ are the half-spin representations of the integral spin group $\Spin(V)$.
The spin representation $S$ is endowed with a non-degenerate integral bilinear pairing, which is symmetric for even $n$ and anti-symmetric for odd $n$
\begin{equation}
\label{eq-Mukai-pairing}
(s,t)_S:=\int_X\tau(s)\cup t,
\end{equation}
where the {\em main anti-automorphism} $\tau$ acts via multiplication by $(-1)^{i(i-1)/2}$ on $H^i(X,\ZZ)$.
The groups $S^+$ and $S^-$ are both of rank $2^{2n-1}$.

Let $X$ be an abelian $n$-fold and $K$ an imaginary quadratic number field. Set $S^+_K:=S^+\otimes_\ZZ K$. The spinor variety in $\PP(S^+_K)$ parametrizes one of the two connected components of the grassmannian of maximal isotropic subspaces of $V_K$ (see Section \ref{sec-pure-spinors}). Points of the spinor variety in $\PP(S^+_K)$ are called {\em even pure spinors}. We are interested in lines in $\PP(S^+_K)$, which are defined over $\QQ$ and intersect the spinor variety at two distinct complex conjugate points. We will call such a line a {\em rational $K$-secant}, even though its points of intersection with the spinor variety are defined over $K$ and not over $\QQ$. 
When $n=2$, the spinor variety in $\PP(S^+_\CC)$ is the quadric hypersurface associated to the pairing (\ref{eq-Mukai-pairing}), and infinitely many rational secants pass through every point in $\PP(S^+_\QQ)$. The rational secant line in 
$\PP(\oplus_{p=0}^2H^{p,p}(X,\QQ))$ yielding the structure of a polarized abelian variety of Weil type on $X\times\hat{X}$ is one corresponding to a negative definite plane $P$ in $\oplus_{p=0}^2H^{p,p}(X,\QQ)$ with respect to the pairing (\ref{eq-Mukai-pairing}) (see \cite{markman-generalized-kummers}).
When $n=3$, the secant variety of the spinor variety is birational to $\PP(S^+_\CC)$. Through a general point $\lambda$ of 
$\PP(\oplus_{p=0}^3H^{p,p}(X,\QQ))$ passes a unique rational secant line and the latter yields the structure of an abelian variety of Weil type on $X\times\hat{X}$ only if the $\Spin(V_\QQ)$-invariant Igusa quartic, given in (\ref{eq-J}), has positive value at $\lambda$ (see Section \ref{sec-coherent-sheaves-with-positive-Igusa-invariant}, the value of the Igusa invariant determines the field of definition of the two pure spinors). For $n\geq 4$ the image in the $(2^{2n-1}-1)$-dimensional $\PP(S^+_\CC)$ of the $(4n^2-2n+1)$-dimensional
secant variety to the spinor variety is a proper subvariety. 

Following is a construction of a $\tau$-invariant rational $K$-secant line to the spinor variety in all dimensions.
Let  $\Theta$ be an ample class on an abelian $n$-fold $X$, $n\geq 2$, 
and let $d$ be a positive integer. Set $u:=\sqrt{-d}\Theta$ and $K:=\QQ(\sqrt{-d})$. The automorphism of $H^*(X,K)$ of cup product with $\exp(u)$ is an element $g$ of $\Spin(V_K)$. Let $\rho:\Spin(V_K)\rightarrow SO(V_K)$ be the standard representation. 
Let $P\subset \oplus_{p=0}^nH^{p,p}(X,\QQ)$ be the rational plane 
\begin{equation}
\label{eq-P-introduction}
P=\span_\QQ\left\{\exp(u)+\overline{\exp(u)},\frac{\exp(u)-\overline{\exp(u)}}{\sqrt{-d}}\right\}
\end{equation}
corresponding to the conjugation invariant plane 
%\begin{equation}
%\label{eq-P-K}
$P_K:=\span_K\{\exp(u),\overline{\exp(u)}\}.$
%\end{equation}
The secant line $\PP(P)$ intersects the spinor variety at $\tilde{\ell}_1:=\span_K\{\exp(u)\}$ and its complex conjugate $\tilde{\ell}_2:=\span_K\{\overline{\exp(u)}\}$. 
The corresponding maximal isotropic subspaces of $V$ are $W_1:=\rho_g(H^1(\hat{X},K))$ and its complex conjugate $W_2$, which are calculated explicitly in (\ref{eq-W-1-and-2}).
$P$  yields the structure of a polarized abelian variety of Weil type on $X\times \hat{X}$ with imaginary quadratic number field $K$, by Proposition \ref{prop-polarized-abelian-variety-of-Weil-type}. 
The rational endomorphism associated to $k\in K$ via the complex multiplication
\[
\eta:K\rightarrow \End_\QQ(X\times\hat{X})\cong \End(V_\QQ)
\]
acts on $W_1$ by scalar multiplication by $k$ and on $W_2$ by its complex conjugate $\bar{k}$.

The fact that $P$ is spanned by Hodge classes implies that each of the intersections $W_i\cap V^{1,0}$ and $W_i\cap V^{0,1}$ is $n$-dimensional, for $i=1,2$, by
Lemma \ref{lemma-decomposition-into-4-direct-summands}. 
Consequently, $\wedge^{2n}W_i$, $i=1,2$, is a pair of complex conjugate lines in $H^{n,n}(X,K)$, which span the rational $2$-dimensional subspace  $\hat{HW}_P$ 
of {\em Hodge-Weil classes} in $H^{n,n}(X\times\hat{X},\QQ)$. 
%corresponding to the conjugation-invariant subspace 
%$\wedge^{2n}W_1\oplus\wedge^{2n}W_2$ of $H^*(X\times\hat{X}),K)$.

Chevalley constructed an integral isomorphism 
$\tilde{\varphi}:H^*(X\times X,\ZZ)\rightarrow H^*(X\times\hat{X},\ZZ)$, given in (\ref{eq-tilde-varphi}), 
which depends on a choice of a class in $H^2(X\times\hat{X},\ZZ)$ (Remark \ref{remark-non-equivariance-of-varphi-tilde}). 
Chevalley's construction depends on the interpretation of $H^*(X\times X,\ZZ)$ as the tensor square $S\otimes S$ and of $H^*(X\times\hat{X},\ZZ)$ as the exterior algebra $\wedge^*V$ 
%and the $C(V)$-module structure of both 
and is purely representation theoretic.
Consider the composition
\begin{equation}
\label{eq-phi-introduction}
\phi:=(\phi_\P\otimes\phi_\P^{-1})\circ\tilde{\varphi}\circ (id\otimes\tau):H^*(X\times X,\ZZ)\rightarrow H^*(X\times\hat{X},\ZZ),
\end{equation}
where $\tau$ is the main anti-automorphism appearing in (\ref{eq-Mukai-pairing}), $\phi_\P:H^*(X,\ZZ)\rightarrow H^*(\hat{X},\ZZ)$ is the correspondence induced by the Chern character $ch(\P)$ of the Poincar\'{e} line bundle $\P$, and the left isomorphism is
$\phi_\P\otimes\phi_\P^{-1}:H^*(X\!\times\!\hat{X},\ZZ)\rightarrow H^*(\hat{X}\!\times\! X,\ZZ)\cong H^*(X\!\times\!\hat{X},\ZZ)$. The subtle $\Spin(V)$-equivariance properties of $\phi$ will be explained in the next subsection. 

\begin{prop}[Prop. \ref{prop-the-orlov-image-of-HW-P-projects-into-the-3-dimensional-space-of-HW-classes}]
The image $\phi(\tilde{\ell}_i\otimes\tau(\tilde{\ell}_i))$ via $\phi\circ(id\otimes\tau)$, of the tensor square $\tilde{\ell}_i\otimes \tilde{\ell}_i$
of each of the two pure spinor lines in $P$, is contained in the subspace $\oplus_{k=2n}^{4n}H^k(X\times\hat{X},K)$ and its projection to $H^{2n}(X\times\hat{X},K)$ is $\wedge^{2n}W_i$. 
\end{prop}

In particular, the isomorphism $\phi\circ(id\otimes\tau)$ maps the $2$-dimensional rational subspace $HW_P:=[\tilde{\ell}_1\otimes \tilde{\ell}_1] \oplus [\tilde{\ell}_2\otimes \tilde{\ell}_2]$ of $P\otimes P\subset H^{even}(X\times X,\QQ)$ to a $2$-dimensional subspace of $H^{even}(X\times\hat{X},\QQ)$, and the latter projects onto the $2$-dimensional subspace $\hat{HW}_P$ of Hodge-Weil classes in $H^{2n}(X\times\hat{X},\QQ)$.
%***************************************************************************
% 
%***************************************************************************
\subsection{Orlov's equivalence $\Phi:D^b(X\times X)\rightarrow D^b(X\times\hat{X})$}
Let $\P$ be the Poincar\'{e} line bundle over $X\times\hat{X}$. 
Let $\mu:X\times X\rightarrow X\times X$ be given by $\mu(x,y)=(x+y,y)$. Let $id\times \Phi_\P:D^b(X\times \hat{X})\rightarrow D^b(X\times X)$ be the equivalence with Fourier-Mukai kernel $\StructureSheaf{\Delta_X}\boxtimes \P$, where $\Delta_X$ is the diagonal in $X\times X$.
{\em Orlov's equivalence}
\[
\Phi:D^b(X\times X)\rightarrow D^b(X\times\hat{X})
\] 
is the inverse of
\[
\mu_*\circ (id\times\Phi_\P):D^b(X\times \hat{X})\rightarrow D^b(X\times X).
\]
Let $m:\Spin(V)\rightarrow GL[H^*(X,\ZZ)]$ be the spin representation and define
$m^\dagger:Spin(V)\rightarrow GL[H^*(X,\ZZ)]$ by $m^\dagger_g=\tau m_g\tau$. Let
$\rho:\Spin(V)\rightarrow SO(V)$ be the standard representation and denote the induced representation on $H^*(X\times\hat{X},\ZZ)\cong \wedge^*V$ by $\rho$ as well. Define the representation
\[
\rho':\Spin(V)\rightarrow GL[H^*(X\times\hat{X},\ZZ)]
\]
by $\rho'_g=\exp\left(\frac{1}{2}[c_1(\P)-\rho_g(c_1(\P))]\right)\rho_g$, for all $g\in\Spin(V)$. The two integral $\Spin(V)$-representation $\rho$ and $\rho'$ are non-isomorphic, but they are isomorphic once tensored with $\QQ$, as the upper square in  the following diagram is commutative, for all $g\in\Spin(V)$.
\begin{equation}
\label{eq-diagram-of-rho-and-rho-prime}
\xymatrix{
H^*(X\times\hat{X},\QQ) \ar[r]^{\rho_g} & H^*(X\times\hat{X},\QQ)
\\
H^*(X\times\hat{X},\QQ)\ar[r]_{\rho'_g} \ar[u]^{\cup \exp\left(-\frac{1}{2}c_1(\P)\right)} & H^*(X\times\hat{X},\QQ) \ar[u]_{\cup \exp\left(-\frac{1}{2}c_1(\P)\right)} 
\\
H^*(X\times X,\QQ) \ar[u]^{\phi\circ(id\otimes\tau)} \ar[r]_{m_g\otimes m^\dagger_g} \ar@/^9pc/[uu]^{\tilde{\phi}}
& H^*(X\times X,\QQ)\ar[u]_{\phi\circ(id\otimes\tau)} \ar@/_9pc/[uu]_{\tilde{\phi}}
}
\end{equation}
The commutativity of the lower square and the $\Spin(V)$-equivariance of 
\begin{equation}
\label{eq-tilde-phi}
\tilde{\phi}:=\exp\left(-c_1(\P)/2\right)\cup \phi\circ(id\otimes\tau):H^*(X\times X,\QQ)\rightarrow H^*(X\times\hat{X},\QQ)
\end{equation}
are established by the following result.
In particular, $\tilde{\phi}$ is $\Spin(V)$-equivariant, where the domain is the representation $m\otimes m^\dagger$ and the codomain is $\rho$.

\begin{prop}[Proposition \ref{prop-extension-class-of-decreasing-filtration-of-spin-V-representations}]
The isomorphism 
\[
\phi:H^*(X\times X,\ZZ)\rightarrow H^*(X\times\hat{X},\ZZ),
\] 
given in (\ref{eq-phi-introduction}), is equal to the isomorphism  
induced by $\Phi$ and $\phi\circ(id\otimes \tau)$ is $\Spin(V)$-equivariant, where the domain is the representation $m\otimes m^\dagger$ and the codomain is $\rho'$.
\end{prop}

An object $F$ of $D^b(X)$ is called a {\em $P$-secant object}, if $ch(F)$ belongs to a secant plane $P$. Given two $P$-secant objects, we refer to the object $E:=\Phi(F_1\boxtimes F_2^\vee)$ as a {\em $P$-secant$^{\boxtimes 2}$-object} in $D^b(X\times\hat{X})$. Let $\Spin(V)_P$ be the subgroup  of $\Spin(V)$ leaving every element of the plane $P$ invariant. When $P$ is given by (\ref{eq-P-introduction})
the subspace $H^2(X\times\hat{X},\QQ)^{\Spin(V)_P}$ is one-dimensional spanned by an ample class $h$, by Lemma \ref{prop-polarized-abelian-variety-of-Weil-type}. Given an element $k\in K$, the rational endomorphism $\eta(k)\in \End_\QQ(X\times\hat{X})$ maps $h$ to $Nm(k)h$, where $Nm:K\rightarrow \QQ$ is the norm map. Consequently, $(X\times\hat{X},\eta,h)$ is a polarized abelian variety of Weil type. The period domain of deformations of $(X\times\hat{X},\eta,h)$ as a polarized abelian variety of Weil type is the adjoint orbit of the complex structure of $X\times\hat{X}$ in $\Spin(V_\RR)_P$, by Lemma \ref{lemma-period-domain-is-an-adjoint-orbit} and Corollary \ref{corollary-Spin-V-P-invariant-classes-are-Hodge}. 

Given a class $ch$ in $H^{ev}(X\times \hat{X},\QQ)$ with graded summand $ch_i$ in $H^{2i}(X\times \hat{X},\QQ)$ and with $ch_0=r\neq 0$
considered as a rational number, set $\kappa(ch)=\exp(-ch_1/r)ch$. Given an object in $D^b(X\times\hat{X})$
of non-zero rank set $\kappa(E):=\kappa(ch(E))$.
The following observation motivated the current paper. It is 
an immediate corollary of the above proposition.
% \ref{prop-extension-class-of-decreasing-filtration-of-spin-V-representations}. 

\begin{cor}
\label{cor-kappa-class-is-Spin-V-P-invariant}
If the rank $r$ of a {\em $P$-secant$^{\boxtimes 2}$-object} $E:=\Phi(F_1\boxtimes F_2^\vee)$ is non zero, then its characteristic class
$
\kappa(E)
%\exp(-c_1(E)/r)ch(E)
$
is $\Spin(V)_P$-invariant with respect to the representation $\rho$. Consequently, $\kappa(E)$ remains of Hodge-type, under every deformations of 
$(X\times\hat{X},\eta,h)$  as a polarized abelian variety of Weil-type.
\end{cor}

\begin{proof}
Let $g$ be an element of $\Spin(V)_P$. 
Using the fact that $\rho_g$ is an algebra automorphism we have $\rho_g(\exp(\bullet))=\exp(\rho_g(\bullet))$
and $\rho_g(\kappa(\bullet))=\kappa(\rho_g(\bullet))$.
The fact that $F_1$ and $F_2$ are  $P$-secant sheaves implies that $m_g\otimes m_g^\dagger$ leaves $ch(F_1\otimes F_2^\vee)$ invariant. 
Hence $\rho'_g(ch(E))=ch(E)$, 
by Proposition \ref{prop-extension-class-of-decreasing-filtration-of-spin-V-representations}. 
Thus $ch(E)\cup\exp\left(-\frac{1}{2}c_1(\P)\right)$ is $\rho_g$ invariant,
$ch(E)\cup\exp\left(-\frac{1}{2}c_1(\P)\right)=\rho_g(ch(E))\cup\exp\left(-\frac{1}{2}\rho_g(c_1(\P))\right)$, by the commutativity of the upper square in Diagram (\ref{eq-diagram-of-rho-and-rho-prime}).
Applying $\kappa$ to both sides, we get $\kappa(ch(E))=\kappa(\rho_g(ch(E)))=\rho_g(\kappa(ch(E))).$
Hence, $\kappa(E)=\rho_g(\kappa(E))$.

%Set $\lambda_g:=\frac{1}{2}[\rho_g(c_1(\P))-c_1(\P)]$.
%Then $\rho_g(c_1(E))=c_1(E)-r\lambda_g$, by the $\rho'_g$-invariance of $ch(E)$. Using the fact that $\rho_g$ is an algebra automorphism we get
%\begin{eqnarray*}
%\rho_g(\kappa(E))&=&\exp\left(-\frac{\rho_g(c_1(E))}{r}\right)\rho_g(ch(E))=\exp\left(-\frac{c_1(E)}{r}\right)\exp(\lambda_g)\rho_g(ch(E))
%\\
%&=&\exp\left(-\frac{c_1(E)}{r}\right)\rho'_g(ch(E))=\exp\left(-\frac{c_1(E)}{r}\right)ch(E)=\kappa(E).
%\end{eqnarray*}
Finally, $\Spin(V)_P$-invariant classes remain of Hodge-type under every deformation of 
$(X\times\hat{X},\eta,h)$  as a polarized abelian variety of Weil-type, by
Corollary \ref{corollary-Spin-V-P-invariant-classes-are-Hodge}.
\end{proof}

%***************************************************************************
% 
%***************************************************************************
%\subsection{An adjoint orbit in  $\Spin(V_\RR)_P$ as a period domain of abelian varieties of Weil type}

%***************************************************************************
% 
%***************************************************************************
\subsection{Ideal secant sheaves on the Jacobian of a genus $3$ curve}
We consider the example where $X$ is the Jacobian $\Pic^2(C)$ of a genus $3$ non-hyperelliptic curve $C$,  $F_1=\Ideal{\cup_{i=1}^{d+1}C_i}(\Theta)$ is the tensor product of the theta line bundle with the ideal sheaf of the union of $d+1$ generic translates $C_i\subset X$ of the Abel Jacobi image $AJ(C)$ of $C$, for $d\geq 3$, and 
$F_2=\Ideal{\cup_{i=1}^{d+1}\Sigma_i}(\Theta)$, where $\Sigma_i\subset X$ is a generic translate of $-AJ(C)$. 
%In this case $K=\QQ[\sqrt{-d}]$. 
Set $u:=\sqrt{-d}\Theta$, $d\geq 3$. 
We prove the following:
\begin{thm}
\label{main-theorem-introduction}
\begin{enumerate}
\item
The line $\PP$ in $\PP{H^{even}(X,\QQ)}$ through $ch(F_1^\vee)$ and $ch(F_2)$ intersects the spinor variety at the two complex conjugate pure spinors $\exp(u)$ and $\overline{\exp(u)}$
defined over $K=\QQ(\sqrt{-d})$ (Lemma \ref{lemma-ch-F-i-is-on-secant-to-spinor-variety}). 
\item
The dual object 
$\Phi(F_2\boxtimes F_1)^\vee$ is isomorphic to $\E[-2]$, where $\E$ is a simple reflexive sheaf of rank $8d$ over $X\times\hat{X}$ (Proposition \ref{prop-local-freeness}). 
\item
The characteristic class $\kappa(\E):=\exp(-c_1(\E)/\rank(\E))ch(\E)$ remains of Hodge type under every deformation of $(X\times\hat{X},\eta,h)$ as a polarized abelian sixfold of Weil type (Lemma \ref{ch-3-alpha-is-second-partial-of-J}).
\item 
\label{thm-item-K-translates-of-kappa-3-and-h-cube-span-HW}
The $\eta(K)$-translates of the graded summand $\kappa_3(\E)$ of $\kappa(\E)$ in $H^{3,3}(X\times \hat{X},\QQ)$, together with $h^3$, span the $3$-dimensional subspace $\QQ h^3\oplus \hat{HW}_P$ 
(Proposition \ref{prop-the-orlov-image-of-HW-P-projects-into-the-3-dimensional-space-of-HW-classes} and Lemma \ref{lemma-kappa-3-is-linearly-independent-from-h-cube}). 
\item
\label{item-P-E-deforms-to-first-order}
$\E$ deforms with $(X\times\hat{X},\eta,h)$ to first order as a twisted sheaf in every direction in the $9$-dimensional moduli space of   
polarized abelian sixfolds of Weil type (Lemma \ref{lemma-kernel-of-ob-E-is-annihilator-of-ch_E}
(\ref{lemma-item-kernel-of-ob-E-is-annihilator-of-ch_E}) and Remark \ref{rem-interchanging-F-1-and-F-2}).
\end{enumerate}
\end{thm}

Orlov's derived equivalence $\Phi:D^b(X\times X)\rightarrow D^b(X\times\hat{X})$ relates diagonal generalized (including non-commutative and gerby) deformations of $D^b(X\times X)$ along which the objects $F_i,$ $i=1,2$, deform,
to {\em commutative} deformations of $X\times\hat{X}$ as polarized abelian varieties of Weil type (with an additional gerby structure), by 
%Lemma \ref{lemma-Orlov's-equivalence-maps-diagonal-deformations-to-commutative-gerby-ones} and
Corollary \ref{cor-Orlov's-equivalence-maps-diagonal-deformations-to-commutative-gerby-ones}. The proof of (\ref{item-P-E-deforms-to-first-order}) above reduces to showing that $F_i$ deforms to first order along a $9$-dimensional subspace of $HH^2(X)$. This amounts to showing that the obstruction map $ob_{F_i}:HH^2(X)\rightarrow \Ext^2(F_i,F_i)$ has rank $6$ (Proposition \ref{prop-obstruction-map-has-rank-6}).

%***************************************************************************
% 
%***************************************************************************
\subsection{Semiregular $K$-secant sheaves}

Let $E$ be a coherent sheaf on a $N$-dimensional compact K\"{a}hler manifold $Y$. Denote the Atiyah class of $E$ by  $at_E\in \Ext^1(E,E\otimes\Omega^1_Y)$. 
The $q$-th component $\sigma_q$ of the {\em semiregularity map} 
\[
\sigma:=(\sigma_0, \dots,\sigma_{N-2}):\Ext^2(E,E)\rightarrow \prod_{q= 0}^{N-2}H^{q+2}(Y,\Omega^q_Y)
\]
is the composition
$
\Ext^2(E,E)\LongRightArrowOf{(at_E)^q/q!}\Ext^{q+2}(E,E\otimes\Omega^q_Y)\LongRightArrowOf{Tr}H^{q+2}(Y,\Omega^q_Y).
$
The sheaf $E$ is said to be {\em semiregular}, if $\sigma$ is injective. Buchweitz and Flenner proved that when $E$ is semiregular the pair $(Y,E)$ deforms locally in the sublocus in the smooth base of a deformation of $Y$, where $ch(E)$ remains of Hodge-type \cite[Th. 5.1]{buchweitz-flenner}. 

If $E$ is locally free, then the projective bundle $\PP(E)$ often deforms over a larger family relaxing the condition that $c_1(E)$ remains of Hodge-type $(1,1)$.   
Every $\PP^{r-1}$-bundle admits a lift to a locally free sheaf twisted by a \v{C}ech $2$-cocycle with coefficients in the local system $\mu_r$ of $r$-th roots of unity and with trivial determinant. The Buchweitz-Flenner theorem has a generalization for twisted sheaves \cite[Remark 2.6]{pridham}. 
We formulate a conjectural generalization of the Buchweitz-Flenner theorem for such twisted sheaves, which we believe\footnote{One needs to verify that our semiregularity map is injective, if and only if Pridham's is.} follows from Pridham's result. The usual definition of the Atiyah class goes through for such a sheaf (Definition \ref{def-atiyah-class-of-twisted-sheaf}) and the semiregularity map is defined as in the untwisted case. Conjecture \ref{conjecture-semiregular-twisted-sheaves-deform} states that $E$ deforms over the sublocus of the base of the deformation, where $\kappa(E)$ remains of Hodge type. Note that if $E$ is locally free, then $\kappa(E)$ is the trace of the exponential Atiyah class of $\PP(E)$, in precise analogy with the relationship between $ch(E)$ and the Atiyah class of $E$ for an untwisted sheaf $E$. We prove Conjecture \ref{conjecture-semiregular-twisted-sheaves-deform} for twisted sheaves on abelian varieties in Section \ref{sec-the-semiregularity-theorem-for-twisted-sheaves-on-abelian-varieties}.
%replacing $ch(E)$ by the characteristic class $\kappa(E)$ of such a sheaf
%sheaves twisted by a \v{C}ech $2$-cocycle with coefficients in the local system $\mu_r$ of $r$-th roots of unity and with trivial determinant
%J. Pridham informed the author that a version of 
%Conjecture \ref{conjecture-semiregular-twisted-sheaves-deform} follows from \cite{pridham}.

Keep the notation of Theorem \ref{main-theorem-introduction}. The sheaf $\E$ in the theorem is not semiregular. We describe next the equivariant version of the construction, carried out in Section \ref{subsection-a-semiregular-secant-square-sheaf}, which produces a semiregular sheaf.
Let $G_1$ and $G_2$ be cyclic subgroups of $X$ of order $d+1$ with $G_1\cap G_2=(0)$. Choose $C_i\subset X$, $1\leq i \leq d+1$, to be a $G_1$-orbit of translates of the Abel-Jacobi image $AJ(C)$ of $C$. Choose $\Sigma_i\subset X$, $1\leq i \leq d+1$, to be a $G_2$-orbit of translates of  $-AJ(C)$. 
The ideal sheaf $\Ideal{\cup_{i=1}^{d+1} C_i}$ is $G_1$-equivariant, $\Ideal{\cup_{i=1}^{d+1} \Sigma_i}$ is $G_2$-equivariant,
and $\Ideal{\cup_{i=1}^{d+1} C_i}\boxtimes \Ideal{\cup_{i=1}^{d+1} \Sigma_i}$ is $G_1\times G_2$-equivariant. The reflexive sheaf $\E$ is the dual of the image of $\Ideal{\cup_{i=1}^{d+1} C_i}\boxtimes \Ideal{\cup_{i=1}^{d+1} \Sigma_i}$ via an equivalence $\tilde{\Phi}$ (the composition of $\Phi$ with tensorization by $\Theta\boxtimes\Theta$ and a shift) between $D^b(X\times X)$ and $D^b(X\times\hat{X})$
and $\E^\vee$ is thus equivariant with respect to the image $G\subset (X\times \hat{X})\times \Pic^0(X\times \hat{X})$ of $G_1\times G_2$ via the Rouqier isomorphism  associated to $\tilde{\Phi}$ between the identity components of the groups of autoequivalences of $X\times X$ and $X\times\hat{X}$. Denote by $\bar{G}$ the image of $G$ via the projection to the cartesian factor  $X\times\hat{X}$. 
The projection $G\rightarrow\bar{G}$ is an isomorphism, by Lemma \ref{lemma-the-projection-of-G-to-bar-G-is-an-isomorphism}.

Let $q:X\times\hat{X}\rightarrow Y:=(X\times\hat{X})/\bar{G}$ be the quotient morphism.  
When $d$ is even\footnote{If $d$ is odd replace it with $4d$ and note that $\QQ(\sqrt{-4d})=\QQ(\sqrt{-d})$.} the rank $8d$ of $\E$ is relatively prime to the order $(d+1)^2$ of $\bar{G}$. In that case 
we can replace $\E$ by its tensor product with a suitable power of the line bundle $\det(\E)$ to get a $\bar{G}$-equivariant sheaf $\tilde{\E}$. The latter descends to a semiregular reflexive sheaf $\bar{\E}$ over $Y$ satisfying $q^*\bar{\E}=\tilde{\E}$. Finally, we associate to $\bar{\E}$ a reflexive sheaf $\B$, twisted by a \v{C}ech $2$-cocycle with coefficients in $\mu_{8d}$ and with a trivial determinant line bundle, satisfying 
%$\kappa(\bar{\E})=\kappa(\B)$ and so 
$\kappa(\E)=q^*\kappa(\B)$.
Such a sheaf $\B$ is constructed, regardless of the parity of $d$.
The morphism $q$ induces a local isomorphism of the Kuranishi deformation spaces of $X\times\hat{X}$ and $Y$. 
Our verification of Conjecture  \ref{conjecture-semiregular-twisted-sheaves-deform} in the case of abelian varieties 
implies that $(Y,\B)$ deforms locally over the locus where $\kappa(\B)$ remains of Hodge-type. Hence, $(X\times \hat{X},q^*\B)$ deforms locally over the locus where $q^*\kappa(\B)$ remains of Hodge type.  Now $\kappa(\E)=q^*\kappa(\B)$ and $\kappa(\E)$ remains of Hodge type over the locus, where $(X\times\hat{X},\eta,h)$ deforms as an abelian variety of Weil-type, by Corollary \ref{cor-kappa-class-is-Spin-V-P-invariant}, yielding a proof of the the following.

\begin{thm}
\label{thm-algebraicity}
Let $d$ be a positive integer. Set $K:=\QQ(\sqrt{-d})$. The Hodge-Weil classes of 
%in a non-empty open analytic subset of the moduli space of 
polarized abelian sixfolds of Weil type with complex multiplication by $K$ and with discriminant $-1$ are algebraic.
%provided the semiregularity Conjecture \ref{conjecture-semiregular-twisted-sheaves-deform} holds.
\end{thm}

The Theorem is proved in Section \ref{subsection-a-semiregular-secant-square-sheaf}.
%If $\E$ is furthermore Gieseker-Maruyama $h$-stable, then the algebraicity would follow globally in moduli from the relative properness 
%of the relative moduli space of semistable sheaves.
The construction described above is generalized in Example \ref{example-secant-sheaf-on-jacobian-of-genus-4-curve}
to produce secant$^{\boxtimes 2}$-objects over $X\times\hat{X}$, for Jacobians $X$ of higher genus curves, but the proof of their semiregularity is special to genus $3$.

%***************************************************************************
% 
%***************************************************************************
\subsection{Verification of the Hodge conjecture for abelian fourfolds}
The following is a known consequence of Theorem \ref{thm-algebraicity}. 

\begin{cor}
The Hodge conjecture holds for abelian fourfolds.
\end{cor}

\begin{proof}
%It has been known that the Hodge conjecture for abelian fourfolds follows from Theorem \ref{thm-algebraicity}. 
%Conjecture \ref{conjecture-semiregular-twisted-sheaves-deform} 
%This is seen as follows. 
Theorem \ref{thm-algebraicity} 
%and Conjecture \ref{conjecture-semiregular-twisted-sheaves-deform}
implies that the Hodge-Weil classes are algebraic for every abelian fourfold of Weil-type, for all imaginary quadratic number fields, and for all discriminants, by  
degenerating abelian sixfolds of Weil type of discriminant $-1$ to products of abelian fourfolds of Weil type of arbitrary discriminant and abelian surfaces of Weil type \cite[Prop. 10]{schoen}. If $A$ is a simple abelian fourfold, then $H^{2,2}(A,\QQ)$ is spanned by quadratic polynomials in divisor classes and by Hodge-Weil classes (for possibly infinitely many complex multiplications), by \cite[Theorem 2.11]{Moonen-Zarhin-Weil-Hodge-Tate-classes}. The Hodge conjecture for abelian fourfolds is thus reduced to the case of non-simple abelian fourfolds. If the Hodge conjecture is verified for an abelian variety, then it follows for any isogenous abelian variety. 
The Hodge conjecture for the product of two abelian surfaces was proved in \cite[Theorem 4.11]{ramon-mari}. If $A=B\times E$, where $B$ is a simple abelian 3-fold and $E$ is an elliptic curve, then either the Hodge ring is generated by divisor classes, or $E$ and $B$ both have complex multiplication by the same imaginary quadratic number field $K$, by \cite[Prop. 3.8]{moonen-zarhin-low-dimension}, and in the latter case the Hodge ring of $A$ is generated by divisor classes and Hodge-Weil classes, by \cite[Theorem 0.1(i)]{moonen-zarhin-low-dimension}.
%The Hodge conjecture thus follows for abelian fourfolds from Conjecture \ref{conjecture-semiregular-twisted-sheaves-deform}. 
\end{proof}

%***************************************************************************
% 
%***************************************************************************
\subsection{Organization of the paper}
In Sections \ref{section-abelian-2n-folds-of-Weil-type-from-rational-secants} to \ref{sec-period-domains} we reformulate Weil's construction of abelian varieties of Weil type, their special Mumford-Tate group,  and their period domain, in terms of $\Spin(V)$ representations and a rational $K$-secant line $P$. 
In Section \ref{section-abelian-2n-folds-of-Weil-type-from-rational-secants} we associate to a non-degenerate rational $K$-secant  $P\subset H^*(X,\QQ)$, which is spanned by Hodge classes, the structure of a polarized abelian variety of Weil type on $X\times\hat{X}$ and show that $\Spin(V)_P$ plays the role of the special Mumford-Tate group of a generic abelian variety of Weil type in its deformation class. We denote the $\Spin(V)_P$-invariant polarization by $\Xi_P$ and the polarized abelian variety of Weil type by $(X\times\hat{X},\Xi_P,\eta)$.

In Section \ref{sec-Hermitian-form} we construct a $\Spin(V)_P$-invariant $K$-valued hermitian form $H$ 
%of discriminant $-1$ 
on the vector space $H^1(X\times\hat{X},\QQ)$.
In section \ref{sec-period-domains} we show that the period domain of deformations of $(X\times\hat{X},\Xi_P,\eta)$ as a polarized abelian variety of Weil type is the adjoint orbit of the complex structure $I$ of $X\times\hat{X}$ in $\Spin(V_\RR)_P$.

In Section \ref{section-equivalences-of-derived-categories} we review general results about the group of autoequivalences of the derived category of an abelian variety. 

Section \ref{section-Orlov-equivalence} is dedicated to the study of the isomorphism 
$\phi:H^*(X\times X,\ZZ)\rightarrow H^*(X\times\hat{X})$ induced by of Orlov's equivalence $\Phi:D^b(X\times X)\rightarrow D^b(X\times\hat{X})$.
We describe the $\Spin(V)$-equivariance properties of $\phi$, relate it to Chevalley's  isomorphism $S\otimes S\rightarrow \wedge^*V$, and use it to conclude that  $\phi$ maps the $2$-dimensional subspace of $H^*(X\times X,K)$, spanned by tensor squares of pure spinors in $P\otimes_\QQ K$, onto 
the $2$-dimensional space of Hodge-Weil classes.

In Section \ref{section-semiregular-twisted-sheaves} we review the theorem of Buchweitz-Flenner verifying the variational Hodge conjecture for semiregular sheaves. We then formulate a conjectural generalization for semiregular sheaves of rank $r>0$ twisted by a \v{C}ech $2$-cocycle with coefficients in the local system  $\mu_r$ of $r$-th roots of unity (Conjecture\footnote{As mentioned, we expect the conjecture to follow from \cite[Remark 2.26]{pridham}.} 
\ref{conjecture-semiregular-twisted-sheaves-deform}).
We prove the conjecture for families of abelian varieties by reducing it to the original result of Buchweitz-Flenner via an elementary argument. 

In Section \ref{section-secant-sheaves-on-abelian-threefolds} we concentrate on the case where $X$ is the Jacobian of a non-hyperelliptic curve $C$ of genus $3$.
We construct a $\QQ(\sqrt{-d})$-secant sheaf $F$ on $X$, for every positive integer $d$, as the sheaf $\Ideal{\cup_{i=1}^{d+1}C_i}(\Theta)$,
where $\Theta$ is the canonical principal polarization and $C_i$, $1\leq i\leq d+1$, are disjoint translates of the Abel-Jacobi image of $C$ in $X$. Recall that both the Yoneda algebra $\Ext^*(F,F)$ and the cohomology $H^*(X,\CC)$ are modules over the Hochschild cohomology $HH^*(X)$. The space $HH^2(X)$ parametrized generalized deformations of the abelian category of coherent sheaves on $X$. We show that the kernel of the evaluation homomorphism 
$ob_F:HH^2(X)\rightarrow \Ext^2(F,F)$ is equal to the kernel of the evaluation homomorphism $ch(F)\Contract:HH^2(X)\rightarrow H^*(X,\CC)$. In particular, 
$F$ deforms to first order in every direction in which $ch(F)$ remains of Hodge-type. 
We then extend the equality $\ker(ob_F)=\ker(ch_F)$, replacing $X$ by $X\times\hat{X}$ and $F$ by the secant$^{\boxtimes 2}$-object
$\G:=\Phi(F_2\boxtimes F_1)[3]$ over $X\times\hat{X}$, to obtain 
$\ker(ob_{\G})=\ker(ch(\G))$, where $F_i$ are the secant sheaves in Theorem \ref{main-theorem-introduction} (see Lemma \ref{lemma-kernel-of-ob-E-is-annihilator-of-ch_E}).
We conclude that $\G$ deforms to first order in all directions in $H^1(X\times\hat{X},T[X\times\hat{X}])$ tangent to the period domain of deformations of $(X\times\hat{X},\Xi_P,\eta)$ as a polarized abelian variety of Weil type.

In Section \ref{section-Jacobians-of-genus-3-curves} we complete the proof of Theorem \ref{main-theorem-introduction} and prove Theorem \ref{thm-algebraicity}. We show that $\E:=\G^\vee[-1]$ is a reflexive sheaf, under a general position assumption on the curves $C_i$ and $\Sigma_j$ in Theorem \ref{main-theorem-introduction}. We also check that the general position assumption can be achieved with a $G_1$-equivariant $\Ideal{\cup_{i=1}^{d+1}C_i}$
and a $G_2$-equivariant $\Ideal{\cup_{i=1}^{d+1}\Sigma_i}$ in Theorem \ref{main-theorem-introduction}, where $G_1$ and $G_2$ are cyclic subgroups of $X$ of order $d+1$. We show that $\E$ can be then normalized to a $\bar{G}$-equivariant sheaf over $X\times\hat{X}$, which descends to a semiregular simple reflexive  sheaf $\B$ over $(X\times\hat{X})/\bar{G}$, twisted by a \v{C}ech $2$-cocycle with coefficients in $\mu_{8d}$, where $\bar{G}$ is a subgroup of $X\times\hat{X}$ isomorphic to $G_1\times G_2$. The semiregularity of $\B$ follows from the equality $\ker(ob_{\G})=\ker(ch(\G))$ mentioned above and the surjectivity of $ob_\G:HH^2(X\times\hat{X})\rightarrow \Ext^2(\G,\G)^{G_1\times G_2}$ (Lemma \ref{lemma-semiregularity-of-twisted-sheaf-B}).
Finally we derive from the semiregularity theorem 
%Conjecture \ref{conjecture-semiregular-twisted-sheaves-deform} 
the algebraicity of the Hodge-Weil classes for 
%deformations of $(X\times\hat{X},\Xi,\eta)$ in an open subset in the classical topology of the moduli space of polarized abelian sixfolds of Weil-type.
polarized abelian sixfolds of Weil-type with discriminant $-1$ and with complex multiplication by an arbitrary imaginary quadratic number field (Theorem \ref{thm-algebraicity}). 

In Section \ref{sec-Igusa-invariant} we characterize $K$-secant sheaves on abelian $3$-folds, by the value of the Igusa $\Spin(V)$-invariant polynomial on their Chern character.
In the appendix section \ref{appendix} we prove that the derived tensor product of the sheaves $\Ideal{\cup_{i=1}^{d+1}C_i}$ and $\Ideal{\cup_{i=1}^{d+1}\Sigma_i}$
in Theorem \ref{thm-algebraicity} is isomorphic to the ideal sheaf of a $1$-dimensional subscheme, even when the subschemes $\cup_{i=1}^{d+1}C_i$ and $\cup_{i=1}^{d+1}\Sigma_i$ intersect.

%***************************************************************************
% 
%***************************************************************************
\section{Abelian $2n$-folds of Weil type associated to a rational secant to the even spinor variety of an abelian $n$-fold}
\label{section-abelian-2n-folds-of-Weil-type-from-rational-secants}

In Section \ref{sec-clifford-algebra} we recall the definition of the Clifford algebra $C(V)$ and of the Spin group as a subgroup of the group of invertible elements of $C(V)$. We recall also that when $V=H^1(X\times\hat{X},\ZZ)$, then $S:=H^*(X,\ZZ)$ is the Spin representation and $S^+:=H^{ev}(X,\ZZ)$ and $S^-:=H^{odd}(X,\ZZ)$ are its two half-spin subrepresentations. In 
Section \ref{sec-pure-spinors} we recall the notion of pure spinors, which are elements of the half-spin representations corresponding to maximal isotropic subspaces of $V_\CC$. The set of even pure spinors in $\PP(S^+_\CC)$ is the spinorial variety. We show that a rational secant line $P$ to the spinorial variety, which intersects it at non-rational points,  determines a $K$-vector space structure on $V_\QQ$, for a quadratic field extension $K$ of $\QQ$. We determine also the subalgebra of $\wedge^*V$ invariant under the subgroup $\Spin(V)_P$ leaving all vectors in the $2$-dimensional subspace $P\subset S^+_\QQ$ invariant.
In Section \ref{subsection-the-isomorphism-tilde-varphi} we recall the isomorphism  $S\otimes_\ZZ S\cong \wedge^*V$ of $\Spin(V)$-representations.
In Section \ref{sec-polarized-avwt-from-secants} we give, for every imaginary quadratic number field $K$, an example of a secant to the spinorial variety, such that 
the associated $K$-vector space structure on $V_\QQ=H^1(X\times\hat{X},\QQ)$ makes $X\times\hat{X}$ a polarized abelian variety of Weil type.
%***************************************************************************
% 
%***************************************************************************
\subsection{The Clifford algebra and the spin group}
\label{sec-clifford-algebra}

Keep the notation of Section \ref{sec-abelian-varieties-of-Weil-type-introduction}.
Let $V$ be the lattice given in (\ref{eq-V}).
Let $C(V)$ be the Clifford algebra, the quotient of the tensor algebra $\oplus_{i=0}^\infty V^{\otimes i}$ by the relation
\begin{equation}
\label{eq-Clifford-relation}
v_1\cdot v_2+v_2\cdot v_1=(v_1,v_2)_V, 
\end{equation}
where the integer in the right hand side is regarded in $V^{\otimes 0}\cong \Integers.$ Then $C(V)$ is a $\Integers/2\Integers$-graded associative algebra with a unit.

Given $w\in H^1(X,\ZZ)$, define the endomorphism $L_w:S\rightarrow S$ of degree $1$ by $L_w(\bullet)=w\wedge(\bullet)$.
Given $\theta\in H^1(X,\ZZ)^*$ we get the endomorphism $D_\theta:S\rightarrow S$ of degree $-1$ given by contraction with $\theta$.
We get the embedding 
\begin{equation}
\label{eq-embedding-of-V-in-End-S}
m:V\rightarrow\End(S),
\end{equation} 
given by $m_{(w,\theta)}:=L_w+D_\theta$ and 
satisfying the analogue 
\[
m_{v_1}\circ m_{v_2}+m_{v_2}\circ m_{v_1}=(v_1,v_2)_V\cdot id_S
\]
of the Clifford relation (\ref{eq-Clifford-relation}). Hence, $m$ extends an algebra homomorphism
\begin{equation}
\label{eq-m}
m:C(V)\rightarrow \End(S),
\end{equation}
which is in fact an isomorphism, by
\cite[Prop. 3.2.1(e)]{golyshev-luntz-orlov}. The {\em main anti-automorphism} 
\[
%\label{eq-main-anti-automorphism-tau-of-Clifford-algebra}
\tau:C(V)\rightarrow C(V)
\]
sends $v_1\cdot \cdots \cdot v_r$ to $v_r\cdot \cdots\cdot v_1$.
The main involution $\alpha:C(V)\rightarrow C(V)$ acts by multiplication by $-1$ on $C(V)^{odd}$ and as the identity on $C(V)^{even}$. The {\em conjugation} $x\mapsto x^*$ is the composition of $\tau$ and $\alpha$.
Let $C(V)^\times$ be the group of invertible elements in $C(V)$. The integral Clifford group is 
\[
G(V):= \{x\in C(V)^\times \ : \ x\cdot V\cdot x^{-1}\subset V\}.
\]
The integral spin group is its index four subgroup
\[
\Spin(V) := \{x\in C(V)^{even} \ : \ x\cdot x^*=1 \ \mbox{and} \ x\cdot V\cdot x^*\subset V \}.
\]

The standard representation $\rho:G(V)\rightarrow O(V)$ of $G(V)$  is defined by
$\rho(x)(v)=x\cdot v\cdot x^{-1}$. If $(v,v)=\pm 2,$
then $-\rho(v)$ is the reflection with respect to the co-rank one lattice $v^\perp$ orthogonal to $v$,
\[
-\rho(v)(\lambda)=\lambda-\frac{2(\lambda,v)_V}{(v,v)_V}\cdot v, \ \forall \lambda\in V.
\]
If $(v_1,v_1)_V=(v_2,v_2)_V=2$, or  $(v_1,v_1)_V=(v_2,v_2)_V=-2$, then $v_1\cdot v_2$ belongs to $\Spin(V)$ and the  group $\Spin(V)$ is generated by such elements.

An element $x\in C(V)^{odd}$ is mapped via (\ref{eq-m}) to an endomorphism $m_x$ of $S$, which maps $S^+$ to $S^-$ and $S^-$ to $S^+$. 
In particular, we get the $\Spin(V)$ equivariant homomorphisms
$V\otimes S^+\rightarrow S^-$ and $V\otimes S^-\rightarrow S^+$.
Given an element $v\in V$, denote by $m_{v,+-}:S^+\rightarrow S^-$ and by $m_{v,-+}:S^-\rightarrow S^+$ the resulting homomorphisms. The latter are adjoints with respect to the bilinear pairing (\ref{eq-Mukai-pairing}).
More generally, for $s, t\in S$ and $v\in V$ we have
\[
(m_v(s),t)_S=(s,m_v(t))_S,
\] 
by \cite[III.2.2]{chevalley}.
Given an element $w\in S^+$, denote by $m_w:V\rightarrow S^-$ the homomorphism given by
\[
m_w(v):=m_{v,+-}(w).
\]
Given $w\in S^-$, define $m_w:V\rightarrow S^+$ by $m_w(v)=m_{v,-+}(w)$. 
The norm character $N:G(V)\rightarrow \{\pm 1\}$ is given by $N(g)=g\cdot\tau(g)$. 
For $g\in G(V)$ and $s, t\in S$ we have
\[
(g(s),g(t))_S=N(g)(s,t)_S,
\]
by \cite[III.2.1]{chevalley}. In particular, if $v\in V$ and $(v,v)_V=2$, 
then $v$ belongs to $G(V)$, $N(v)=1$,  $m_v:S\rightarrow S$ is an isometry interchanging $S^+$ and $S^-$, and $m_v^2=\one_S$.

Given a field $K$ we set $V_K:=V\otimes_\Integers K$, and define $S_K$, $S^+_K$, and $S^-_K$ similarly.
The same definitions above yield the Clifford algebra $C(V_K)$, the groups $G(V_K)$, $\Spin(V_K)$, and 
$O(V_K)$, the representation $\rho:G(V_K)\rightarrow O(V_K)$, and the isomorphism
$m:C(V_K)\rightarrow \End(S_K)$.

%***************************************************************************
% 
%***************************************************************************
\subsection{Rational lines $K$-secant to the variety of even pure spinors}
\label{sec-pure-spinors}
An element $w\in S^+_\CC$ is called an {\em even pure spinor}, if the kernel of $m_w:V_\CC\rightarrow S^-_\CC$ is a maximal isotropic subspace. An element 
$w\in S^-_\CC$ is called an {\em odd pure spinor}, if the kernel of $m_w:V_\CC\rightarrow S^+_\CC$ is a maximal isotropic subspace \cite[III.1.4]{chevalley}. 
Every maximal isotropic subspace is of this form and so the $(2n^2-n)$-dimensional grassmannian $IGr(2n,V_\CC)$ of $2n$-dimensional isotropic subspaces has two connected components $IGr_+(2n,V_\CC)$ and $IGr_-(2n,V_\CC)$ \cite[III.1.5]{chevalley}. 
We get a $Spin(V)$-equivariant embedding 
\[
IGr_+(2n,V_\CC) \ \rightarrow \PP(S^+_\CC)\cong \PP^{2^{2n-1}-1}
\]
sending $W\in IGr_+(2n,V_\CC)$ to the line spanned by an even pure spinor $s$ with $W=\ker(m_s)$.
The image of the embedding in $\PP(S^+_\CC)$ is called the (even) {\em spinor variety}.

Let $K$ be a purely imaginary quadratic number field. 
Let $\ell_1$, $\ell_2$ be two complex conjugate points in $\PP(S^+_K)$. Assume that the line $\tilde{\ell}_i$ in $S^+_K$ corresponding to $\ell_i$ is spanned by a pure spinor.
% $\lambda_i$, $i=1,2$. 
%We may choose the pure spinors to be complex conjugate $\lambda_2=\bar{\lambda}_1$. 
 Let $W_i\subset V_K$ be the maximal isotropic subspace of $V_K$ corresponding to $\ell_i$. Note that $W_2$ is the complex conjugate of $W_1$. 
 Let $P_K$ be the plane in $S^+_K$ spanned by $\tilde{\ell}_1$ and $\tilde{\ell}_2$. 
 $P_K$ is defined over $\QQ$ and we denote by $P$ the corresponding subspace of $S^+_\QQ$. 

\begin{lem}
\label{lemma-P-is-non-isotropic-iff-W-1-and-W-2-are-transversal}
$P$ is isotropic with respect to the pairing $(\bullet,\bullet)_S$, given in (\ref{eq-Mukai-pairing}), if and only if $W_1\cap W_2\neq (0)$.
The restriction of the pairing $(\bullet,\bullet)_S$ to $P$ is definite, if and only if $W_1\cap W_2= (0)$.
\end{lem}

\begin{proof}
Let $\lambda_1$ be a non-zero element of $\tilde{\ell}_1$ and set $\lambda_2=\bar{\lambda}_1$. 
Then $(\lambda_i,\lambda_i)=0$, for $i=1,2$, by \cite[III.2.4]{chevalley}.
Furthermore, $(\lambda_1,\lambda_2)=0$, if and only if $W_1\cap W_2\neq (0)$, by \cite[III.2.4]{chevalley}.
Write $\lambda_1=a+ib$, with $a,b\in S^+_\QQ$. Then 
\begin{eqnarray*}
(\lambda_1,\lambda_1)&=& (a,a)_S-(b,b)_S+2i(a,b)_S=0,
\\
(\lambda_1,\lambda_2)&=& (a,a)_S+(b,b)_S.
\end{eqnarray*}
The first equation implies that $(a,a)=(b,b)$ and $(a,b)=0$. The second equation thus implies that $(a,a)=0$, if and only if $W_1\cap W_2\neq 0$.
\end{proof}
 
 Assume that $W_1\cap W_2$ is the zero subspace. Then $\{\ell_1,\ell_2\}$ is the set theoretic intersection of $IGr_+(2n,V_\CC)$ and the line through $\ell_1$ and $\ell_2$, provided $n>1$, by \cite[III.1.12]{chevalley}.
Denote by $\Spin(V_K)_{\ell_i}$ the subgroup of $\Spin(V_K)$ stabilizing $\ell_i$. Denote their intersection by
\begin{equation}
\label{eq-spin-V-K-ell-1-ell-2}
\Spin(V_K)_{\ell_1,\ell_2}:=\Spin(V_K)_{\ell_1}\cap \Spin(V_K)_{\ell_2}.
\end{equation}
The quotient $\Spin(V_K)_{\ell_1,\ell_2}/\{\pm 1\}$ 
is isomorphic to $GL(W_1)$, and so to $GL_{2n}(K)$, and the quotient maps injectively into $SO^+(V_K)$, by the proof of \cite[Lemma 1]{igusa}.  The element of $\Spin(V_K)_{\ell_1,\ell_2}/\{\pm 1\}$ corresponding to $g\in GL(W_1)$ acts on $W_2$ via $(g^*)^{-1}$, where $W_2$ is identified with $W_1^*$ via the bilinear pairing of $V_K$. 
Let 
\[
{\det}_i: \Spin(V_K)_{\ell_1,\ell_2}\rightarrow K^\times
\]
be the pullback from $GL(W_i)$ of the determinant character. Then $\det_2(s)=\det_1(s)^{-1}$.
%Let $\tilde{\ell}_i$ be the one-dimensional subspace of $S^+_K$ corresponding to $\ell_i$. 
The characters $\tilde{\ell}_i$ and $\det_i$ satisfy $\tilde{\ell}_i\otimes \tilde{\ell}_i\cong \det_i$,
%An element $s$ of $\Spin(V_K)_{\ell_1,\ell_2}$ acts on $\tilde{\ell}_i$ via multiplication by $(\det_i(s))^2$, 
by the proof of \cite[Lemma 1]{igusa} (see the second displayed formula for $\phi(s_i(\lambda))$ in \cite[Sec. 2]{igusa}. See also \cite[III.3.2]{chevalley}). 

 Denote by $\Spin(V_K)_P$ the subgroup of 
 $\Spin(V_K)$ leaving every vector in $P_K$ invariant. 
 Define 
 \begin{equation}
 \label{eq-Spin(V)_P}
 \Spin(V)_P
 \end{equation} 
 and $\Spin(V_\QQ)_P$ analogously.
 
 \begin{lem}
 \label{lemma-Spin-V-K-is-SL-n-K}
 The group $\Spin(V_K)_P$ is isomorphic to $SL_n(K)$.
 \end{lem}
 
 \begin{proof}
 Note that $\Spin(V_K)_P$ is the kernel of $\det_1$ in 
 $\Spin(V_K)_{\ell_1,\ell_2}$ and so it is isomorphic to $SL_n(K)$, since $-1$ is not contained in $\Spin(V_K)_P$. 
 \end{proof}
 
 \begin{rem} 
 \label{remark-stabilizer-of-w-may-have-two-connected-components}
 (\cite[Lemma 2]{igusa} and the remark following it).
 Assume that $n\geq 3$.
 Let $w$ be a point in $P$. Assume that $w$ does not belong to neither $\tilde{\ell}_1$ nor $\tilde{\ell}_2$. 
 Then $w$ is not a pure spinor, as commented above.
 If $n$ is odd, then the stabilizer of $w$ in $\Spin(V_K)$ is $\Spin(V_K)_P$. If $n$ is even, then the stabilizer has two connected components and the identity component is $\Spin(V_K)_P$. In particular, $w$ determines $P$ and $\PP(P)$ is the unique secant to the spinor variety through $w$.
 \end{rem}
 
 Let d be a rational number, such that $-d$ is not a square of a rational number. Set $K:=\QQ[\sqrt{-d}]$. 
Let $\sigma: K\rightarrow K$ be the involution in $Gal(K/\QQ)$. Denote by $\sigma$ also the induced involution on $S^+_K$, $S^-_K$, and $V_K$. 
Let $Nm:K\rightarrow\QQ$ be the norm map $Nm(\lambda)=\lambda\sigma(\lambda)$.
Denote the group of rational similarities of $V_\QQ$ by
\begin{equation}
\label{eq-group-of-similarities}
\tilde{O}(V_\QQ):=\{g\in GL(V_\QQ) \ : \ (g(v_1),g(v_2))=c(v_1,v_2), \ \mbox{for some} \ c\in Nm(K^\times)\}.
\end{equation}
Assume that $P\subset S^+_\QQ$ is a non-isotropic $2$-dimensional subspace, which intersects the spinor variety in $\PP(S^+_K)$ in two $\sigma$-conjugate points $\ell_1$ and $\ell_2$ corresponding to two maximal isotropic subspaces $W_1$ and $W_2$ of $V_K$. The vanishing $W_1\cap W_2=(0)$ holds, by Lemma \ref{lemma-P-is-non-isotropic-iff-W-1-and-W-2-are-transversal}.

The subset $V_\QQ$ of $V_K$ is equal to $\{v_1+\sigma(v_1) \ : \ v_1\in W_1\}$.
Given $\lambda\in K^\times$, let $\cm_\lambda:V_K\rightarrow V_K$ act on $W_1$ by multiplication by $\lambda$ and on $W_2$ by multiplication by $\sigma(\lambda)$. Then $\cm_\lambda$ leaves $V_\QQ$ invariant and we get the homomorphism 
\begin{equation}
\label{eq-action-by-imaginary-quadratic-number-field}
\cm:K^\times\rightarrow GL(V_\QQ)
\end{equation}
sending $\lambda$ to the restriction of $\cm_\lambda$ to $V_\QQ$.

\begin{lem}
\label{lemma-centralizer-of-rho-Spin-V-P}
The centralizer of $\rho(\Spin(V_\QQ)_P)$ in $\tilde{O}(V_\QQ)$ is $\cm(K^\times)$.
\end{lem}
 
 \begin{proof}
 The image of $\cm$ clearly centralizes $\rho(\Spin(V_\QQ)_P)$.
 %where $e$ is given in (\ref{eq-embedding-e-of-stabilizer-of-w}), 
 %and so it centralizes $\rho(\Spin(V_\QQ)_w)$, by Lemma \ref{lemma-secant-to-spinor-variety}. 
Given $v_1, v_1'\in W_1$, set $v=v_1+\sigma(v_1)$ and $v'=v_1'+\sigma(v_1')$. 
%Every element of $V_\QQ$ is of this form. 
Then 
\begin{eqnarray*}
(\cm_\lambda(v),\cm_\lambda(v'))_V&=&(\lambda v_1+\sigma(\lambda)\sigma(v_1),\lambda v_1'+\sigma(\lambda)\sigma(v_1'))_V
\\
&=&
Nm(\lambda)\left[(v_1,\sigma(v_1'))_V+(v_1',\sigma(v_1))_V\right]
=  Nm(\lambda)(v,v').
\end{eqnarray*}
%Now, $Nm(\lambda)$ is positive, for all $\lambda\in K^\times$. Hence, the image of $\cm$ is contained in $\tilde{O}(V_\QQ)$.

Conversely, if $g\in\tilde{O}(V_\QQ)$  centralizes $\rho(\Spin(V_\QQ)_P)$,
then $W_1$ and $W_2$ are $g$ invariant and $g$ acts on $W_i$ via multiplication by a scalar $\lambda_i\in K^\times$. Given $v_1\in W_1$, we have
\[
\lambda_1v_1=g(v_1)=(\sigma g\sigma)(v_1)=\sigma(\lambda_2\sigma(v_1))=\sigma(\lambda_2)v_1.
\]
Hence, $\lambda_2=\sigma(\lambda_1)$ and $g=\cm(\lambda_1)$.
 \end{proof}
  
Let  $V_\CC:=V^{1,0}\oplus V^{0,1}$ be the Hodge decomposition with respect to the complex structure of $X\times \hat{X}$. We call $\oplus_{p=0}^n H^{p,p}(X,\QQ)$ the {\em Hodge ring} of $X$. 

\begin{rem}
\label{remark-P-is-contained-in-the-Hodge-ring-if-I-is-in-image-of-Spin-V-P}
Note that $P$ is contained in the Hodge ring of $X$ for every complex structure $I:V_\RR\rightarrow V_\RR$ of $X$, which 
belongs to $\rho(\Spin(V_\RR)_P)$.
\end{rem}

\begin{lem}
\label{lemma-decomposition-into-4-direct-summands}
If $P$ is contained in the Hodge ring, then 
%the element $f$ in $SO_+(V_\QQ)$, given in (\ref{eq-f}) has an $n$-dimensional $\sqrt{-d}$-eigen space in 
$W_1^{1,0}:=W_{1,\CC}\cap V^{1,0}$ and $W_2^{1,0}:=W_{2,\CC}\cap V^{1,0}$  are both  $n$-dimensional. Equivalently,
$W_1^{0,1}:=W_{1,\CC}\cap V^{0,1}$ and $W_2^{0,1}:=W_{2,\CC}\cap V^{0,1}$  are both  $n$-dimensional. 
\end{lem}

\begin{proof}
The two statements are the complex conjugate of each other  and are thus equivalent.
%The integral Hodge structure of $V$ endows the tensor algebra $\oplus_{k=0}^\infty V$ with an integral Hodge structure.
%The homomorphism $\oplus_{k=0}^\infty V^{\otimes k}\rightarrow \End(S)$, induced by (\ref{eq-embedding-of-V-in-End-S}), 
%is a homomorphism of integral Hodge structures (ignoring weights
%The latter homomorphism factors through 
The isomorphism $m:C(V)\rightarrow \End(S)$, given in (\ref{eq-m}),  endows $C(V)$ with an integral Hodge structure. 
The subspace $V$ of $C(V)$  is a sub-Hodge-structure, which agrees with
the natural Hodge structure of $V$, provided we adjust the weight of the direct summand $H^1(\hat{X},\ZZ)$ of $V$ to be $-1$ as is the weight of the Hodge structure dual to $H^1(X,\ZZ)$. With this weight adjusment the homomorphism $V\rightarrow \End(S)$, given in 
(\ref{eq-embedding-of-V-in-End-S}), becomes a morphism of Hodge structures. 
Let $\lambda_i$ be a non-zero element of $\tilde{\ell}_i$, $i=1,2$.
Assume that $P$ is contained in the Hodge ring. Then $\lambda_i$ belongs to $\oplus_{p=0}^n H^{p,p}(X)$, for $i=1,2$. Hence, 
\[
m_{\lambda_i}:V_\CC\rightarrow S^-_\CC
\]
is  equivariant with respect to the $U(1)$ (circle) action defined by the Hodge structures of $V$ and $S^-$. Thus, its kernel $W_i$ is $U(1)$-invariant.
The weight adjustment does not change the weights of the $U(1)$-action\footnote{The weight of the $U(1)$-action on $V^{p,q}$ is $p-q$, and the weight adjustment changes the bidegree $(1,0)$ of $H^{1,0}(\hat{X})$ to bidegree $(0,-1)$ and the bidegree $(0,1)$ of $H^{0,1}(\hat{X})$ to bidegree $(-1,0)$.}, and so considering $V$with its original weight one Hodge structure we have
$W_{i,\CC}=W_{i,\CC}\cap V^{1,0}\oplus W_{i,\CC}\cap V^{0,1}$. 
It remains to show that the two direct summands have the same dimension, i.e., that $\wedge^{2n} W_{i,\CC}$ is the trivial $U(1)$ character. 

We recall next the isomorphism 
\[
\varphi : S\otimes_\ZZ S\rightarrow C(V)
\]
of \cite[III.3.1]{chevalley} and verify that it is an isomorphism of Hodge structure.
The product  $C(V)\otimes C(V)\rightarrow C(V)$ is a morphism of Hodge structures, as such is the product in $\End(S)$. 
The exterior algebras\footnote{
Given a free $\ZZ$-module $M$, we denote by $\wedge^kM$ the quotient of the $k$-th tensor power $M^{\otimes k}$ of $M$ over $\ZZ$ by the submodule of symmetric tensors and set $\wedge^*M:=\oplus_{k\geq 0}\wedge^kM$.
} 
$S_X:=\wedge^*H^1(X,\ZZ)$ and $S_{\hat{X}}:=\wedge^* H^1(\hat{X},\ZZ)$ both embed naturally as subalgebras and sub-Hodge-structures of $C(V)$, sending exterior products of elements of $H^1(X,\ZZ)$ to the corresponding products in $C(V)$, since $H^1(X,\ZZ)$ and $H^1(\hat{X},\ZZ)$ are isotropic subspaces. We can thus regard the class $[pt_{\hat{X}}]\in H^{2n}(\hat{X},\ZZ)$, Poincar\'{e} dual to a point, as an element of $C(V)$. Define $\varphi$ by 
\begin{equation}
\label{eq-Chevalley-varphi}
\varphi(u\otimes v):= u[pt_{\hat{X}}]\tau(v),
\end{equation}
for all $u, v\in S_X$. The element $[pt_{\hat{X}}]$ of $C(V)$ is a Hodge class (of weight $(-n,-n)$ under the above convention), and the involution $\tau$ of $S_X$ is an automorphism of Hodge structure, so $\varphi$ is an integral homomorphism of Hodge structures. Tensoring with $\QQ$ it becomes an isomorphism of $\Spin(V_\QQ)$ representations, by  \cite[III.3.1]{chevalley}.
Hence, $\varphi$ is an injective homomorphism of $\Spin(V)$-representations.
Working over $\ZZ$, the composition $m\circ \varphi:S\otimes_\ZZ S\rightarrow \End(S)$ is proved surjective\footnote{
One checks that $m\circ\varphi:S\otimes_\ZZ S\rightarrow \End(S)$ is  equal to $(-1)^n$ times the isomorphism sending $s\otimes t$ to $s\otimes (t,\bullet)_S$. Indeed, 
$
[pt_{\hat{X}}]\tau(t)x[pt_{\hat{X}}]=
%\left\{\begin{array}{ccc}
(-1)^{\frac{(2n)(2n-1)}{2}}(t,x)_S[pt_{\hat{X}}] =(-1)^n(t,x)_S[pt_{\hat{X}}],
%& \mbox{if} \ 
%\\
%0 & \mbox{otherwise}
%\end{array}\right.
$ 
for all $t,x\in S$.
Hence, 
\[
\varphi(s\otimes t)(x[pt_{\hat{X}}])=s[pt_{\hat{X}}]\tau(t)x[pt_{\hat{X}}]=(-1)^n(t,x)_Ss[pt_{\hat{X}}],
\] 
for all $s,t\in S$.
Now, the inclusion $S\subset C(V)$ described above, composed with right multiplication by $[pt_{\hat{X}}]$ yields an embedding $\zeta:S\rightarrow C(V)$ as the left ideal  $I:=C(V)[pt_{\hat{X}}]$ and the left action of $C(V)$ on $I$ corresponds to the action of $C(V)$ on $S$ via $m$. Hence, the displayed equation becomes
$\zeta(m_{\varphi(s\otimes t)}(x))=(-1)^n(t,x)_S\zeta(s).$ The equality 
$m_{\varphi(s\otimes t)}(x)=(-1)^n(t,x)_Ss$ follows.
} 
in the proof of \cite[Prop. 3.2.1(e)]{golyshev-luntz-orlov} that $m:C(V)\rightarrow\End(S)$ is surjective. Hence, $\varphi$ is surjective, as $m$ is an isomorphism. We conclude that $\varphi$ is an isomorphism of $\Spin(V)$ representations as well as of Hodge structures.

The isomorphism $\varphi$ maps the line $\tilde{\ell}_i\otimes \tilde{\ell}_i$ in $S_\CC\otimes S_\CC$ onto 
the top exterior product $\wedge^{2n}W_{i,\CC}$, considered as a one-dimensional subspace of $C(V_\CC)$, by  \cite[III.3.2]{chevalley}.
The $U(1)$ character $\wedge^{2n}W_{i,\CC}$ is isomorphic to the trivial $U(1)$ character $\tilde{\ell}_i\otimes \tilde{\ell}_i$,
as $\varphi$ is an isomorphism of Hodge structures.
%isomorphic to the line in $C(V_\CC)$ spanned by the product of elements of a basis of $W_{i,\CC}$ and the latter is the 
%image of the line $\tilde{\ell}_i\otimes \tilde{\ell}_i$ in $S_\CC\otimes S_\CC$ via the isomorphism $\varphi$ 
%given in (\ref{eq-varphi}) \cite[III.3.2]{chevalley}. The $U(1)$ character $\tilde{\ell}_i\otimes \tilde{\ell}_i$ is trivial, 
%by  assumption, and $\varphi$ is an isomorphism of Hodge structures, hence the $U(1)$ character $\wedge^n W_{i,\CC}$ 
%is trivial as well.
\end{proof}

The character group $\Hom(K^\times,K^\times)$ is isomorphic to $\ZZ\times\ZZ$, generated by $id$ and conjugation.
Given a vector space $U$ over $\QQ$ and a homomorphism $K\rightarrow \End_\QQ(U)$ 
we get the decompostion $U_K:=U\otimes_\QQ K=\oplus_{(a,b)\in \ZZ}U_{a,b}$, where $\lambda\in K^\times$ acts on the subspace $U_{a,b}$ of $U_K$ via $\lambda^a\bar{\lambda}^b$. 
The subspaces $U_{\{a,b\}}:=U_{a,b}\oplus U_{b,a}$ are defined over $\QQ$.
With this notation $W_1=V_{1,0}$ and $W_2=V_{0,1}$. Set $\wedge_{a,b}V:=(\wedge^{a+b}V)_{a,b}=\wedge^aW_1\otimes\wedge^bW_2.$

The exterior algebra $\wedge^*V_\CC$ 
inherits two commuting actions of $K^\times$ and of $\CC^\times$, where $x+iy\in \CC^\times$ acts via $x+yI_{V_\RR}$. Hence, 
$\wedge^kV_\CC$ admits the decomposition
\[
\wedge^kV_\CC = \oplus^{p+q=k}_{a+b=k} \wedge^{p,q}_{a,b}V,
\]
where the top weights correspond to the $\CC^\times$ action on $V_\RR$ and the bottom weights to the $K^\times$ 
action on $V_\QQ$. With this notation we have $W^{1,0}_1=\wedge^{1,0}_{1,0}V$, $W^{0,1}_1=\wedge^{0,1}_{1,0}V$, $W^{1,0}_2=\wedge^{1,0}_{0,1}V$, and $W^{0,1}_2=\wedge^{0,1}_{0,1}V$. 
Lemma \ref{lemma-decomposition-into-4-direct-summands} states that the latter four are all $n$ dimensional. 

Below and everywhere in the paper unless mentioned otherwise, the $\Spin(V)$-action on $\wedge^*V$ is the $\rho$-action in Diagram (\ref{eq-diagram-of-rho-and-rho-prime})
preserving the grading.

\begin{lem}
\label{lemma-Spin-V-P-invariant-classes-are-Hodge}
The $\Spin(V)_P$-invariant subspace $(\wedge^{2j}V_\QQ)^{\Spin(V)_P}$ is a subspace of $\wedge^{j,j}V_\CC$. Furthermore,
$\dim(\wedge^kV_\QQ)^{\Spin(V)_P}=
\left\{
\begin{array}{ccl}
0 & \mbox{if} & k \ \mbox{is odd,}
\\
1 &  \mbox{if} &0\leq k \leq 4n,\ k \ \mbox{is even, and}\ \ k\neq 2n,
\\
3 & \mbox{if} & k=2n.
\end{array}
\right.
$\\
The space $(\wedge^{2n}V_K)^{\Spin(V)_P}$ decomposes as a direct sum of three characters of $\Spin(V_K)_{\ell_1,\ell_2}$ consisting of $\wedge^{2n}W_1$ isomorphic to $\det_1$, $\wedge^{2n}W_2$ isomorphic to $\det_2$,  and the trivial character. For $j\neq n$, $0\leq j\leq 2n$, the space $(\wedge^{2j}V_\QQ)^{\Spin(V)_P}$ is a trivial $\Spin(V_K)_{\ell_1,\ell_2}$ character.
\end{lem}

Note: We will see that the space $(\wedge^2V_\QQ)^{\Spin(V)_P}$ is spanned by a non-degenerate $2$-form on $V_K^*$, and so the subspace $(\wedge^*V_K)^{\Spin(V_K)_{\ell_1,\ell_2}}$ consists of powers of this $2$-form (see Equation  (\ref{eq-Xi})).

\begin{proof}
%\label{remark-bi-grading}
It suffices to prove the statement for $0\leq k\leq 2n$, 
as $\wedge^kV_\QQ$ and $\wedge^{4n-k}V_\QQ$ are dual $\Spin(V)_P$ representations.
 For $0\leq k \leq 2n$, $\wedge^kW_i$, $i=1,2$, are dual irreducible $\Spin(V)_P$ representation. 
Furthermore, $\wedge^aW_i$ and $\wedge^{2n-a}W_i$ are dual representations. We have an isomorphism of $\Spin(V)_P$ representations $\wedge^kV_K\cong \oplus_{a+b=k}(\wedge^aW_1)\otimes(\wedge^bW_2)$.
Hence, $(\wedge^kV)^{\Spin(V)_P}$ is trivial, if $k$ is odd, and
\[
(\wedge^{2j}V_\QQ)^{\Spin(V)_P}= \ \mbox{rational points of} \ (\wedge^jW_1\otimes\wedge^jW_2)^{\Spin(V)_P},
%\oplus (\wedge^jW_2\otimes\wedge^jW_1)^{\Spin(V)_P},
\]
for $0<j<n$, and it is $1$-dimensional, while
\[
(\wedge^{2n}V_\QQ)^{\Spin(V)_P}=
 \mbox{rational points of} \ (\wedge^nW_1\otimes\wedge^nW_2)^{\Spin(V)_P}
 %\oplus (\wedge^nW_2\otimes\wedge^nW_1)^{\Spin(V)_P}
 \oplus \wedge^{2n}W_1\oplus \wedge^{2n}W_2,
\]
and it is $3$-dimensional, as $\wedge^{2n}W_i$ is a trivial character of $SL(W_i)$.

Note that $\wedge^{n,n}_{2n,0}V=(\wedge ^{2n}V)_{2n,0}=\wedge^{2n}W_1$ and $\wedge^{n,n}_{0,2n}V=\wedge^{2n}W_2$ are one-dimensional complex conjugate subspaces defined over $K$, and so $\wedge^{n,n}_{2n,0}V\oplus \wedge^{n,n}_{0,2n}V$ is a $2$-dimensional subspace defined over $\QQ$. The latter subspace is $\Spin(V_K)_P$-invariant.

It remains to prove that the $1$-dimensional subspace $[(\wedge^jW_1)\otimes(\wedge^jW_2)]^{\Spin(V)_P}$ is of Hodge type $(j,j)$, i.e., a subspace of $\wedge^{j,j}_{j,j}V$, for $0\leq j\leq n$. The subspace $[(\wedge^jW_1)\otimes(\wedge^jW_2)]^{\Spin(V)_P}$ is a one-dimensional $U(1)$-invariant subspace of 
$(\wedge^jW_1)\otimes(\wedge^jW_2)$ defined over $\QQ$, hence it must be the trivial $U(1)$-character.

The representations $\wedge^kW_i$, $i=1,2$, are dual irreducible $\Spin(V_K)_{\ell_1,\ell_2}$-representations. Hence, the triviality of the $\Spin(V_K)_{\ell_1,\ell_2}$-character $(\wedge^{2j}V_\QQ)^{\Spin(V)_P}$ for $j\neq n$, $0\leq j\leq 2n$.
\end{proof}

%***************************************************************************
% 
%***************************************************************************
\subsection{The isomorphism $\tilde{\varphi}:S\otimes_\ZZ S\rightarrow \wedge^*V$}
\label{subsection-the-isomorphism-tilde-varphi}
 
 Let $B_0:V\otimes_\ZZ V\rightarrow \ZZ$ be a bilinear pairing, not necessarily symmetric, satisfying $B_0(u,u)=\frac{1}{2}(u,u)_V$. Note that $(u,u)_V$ is even, for every $u\in V$.
 Our choice of $B_0$ is as follows. Write  $v_i=(w_i,\theta_i)$, $i=1,2$, with $w_i\in H^1(X,\ZZ)$ and $\theta_i\in H^1(\hat{X},\ZZ)\cong H^1(X,\ZZ)^*$. We set\footnote{Our choice yields the equality in Lemma \ref{lemma-nu-equal-tilde-varphi} below. The choice of $B_0$ used in \cite[Sec. 3.3]{chevalley} is $B_0(v_1,v_2):=\theta_1(w_2)$} 
 $B_0(v_1,v_2):=\theta_2(w_1)$.
 Given $x\in V$, define $L_x:\wedge^*V\rightarrow \wedge^*V$ by $L_x(u)=x\wedge u$. 
 Given $x\in V$, define $\delta_x:\wedge^*V\rightarrow\wedge^*V$ as contraction with $B_0(x,\bullet)$. 
 Set $L'_x:=L_x+\delta_x$. Then $(L'_x)^2=\frac{1}{2}(x,x)_V\cdot \one,$
 where $\one$ is the identity in $\End(\wedge^*V)$. Hence, the map $x\mapsto L'_x$ extends to an algebra homomorphism 
 \[
 \psi':C(V)\rightarrow \End(\wedge^*V),
 \]
 by the universal property of $C(V)$ (see \cite[II.1.1]{chevalley}). Define
 \begin{equation}
 \label{eq-psi}
 \psi:C(V)\rightarrow \wedge^*V
 \end{equation}
 by $\psi(x)=\psi'(x)\cdot \one,$ where $\one\in\wedge^0V$ is the unit  in  $\wedge^*V$. Then $\psi$ is a homomorphism of left $C(V)$-modules. 
  
 Let $C(V)_k$ be the additive subgroup of $C(V)$ generated by products of $j$ elements of $V$, for $j\leq k$. We get the increasing filtration $C(V)_0\subset C(V)_1\subset \cdots C(V)_{4n}=C(V)$. Let $F^k(\wedge^*V):=\oplus_{i\leq k}\wedge^iV.$
Then $\psi(C(V)_k)\subset F^k(\wedge^*V)$. We get the induced surjective homomorphism 
\[
\bar{\psi}_k:C(V)_k/C(V)_{k-1}\rightarrow \wedge^kV,
\]
which is injective, as it induces an isomorphism once we tensor with $\QQ$ \cite[II.1.6]{chevalley}.
 Hence, $\psi$ is an isomorphism of left $C(V)$-modules. Furthermore, the conjugation action of $\Spin(V)$ on $C(V)$ preserves the filtration $C(V)_k$ and induces an action on the associated graded group and $\bar{\psi}_k$ is $\Spin(V)$-equivariant, where the action on $\wedge^kV$ is the one induced from the representation $V$ \cite[Sec. 3.3]{chevalley}. 
 
 Consider the composite isomorphism 
 \begin{equation}
 \label{eq-tilde-varphi}
 \tilde{\varphi}:=\psi\circ\varphi:S\otimes_\ZZ S\rightarrow \wedge^*V,
 \end{equation}
 %be the composition $\tilde{\varphi}:=\psi\circ\varphi$, 
 where $\varphi$ is given in (\ref{eq-Chevalley-varphi}). Conjugating the $\Spin(V)$-action via $\tilde{\varphi}$ we get on $\wedge^*V$ the structure 
of a $\Spin(V)$-representation. The latter depends on the choice of $B_0$, but the associated graded action, with respect to the increasing filtration $F^k(\wedge^*V)$ does not (see the discussion  following the proof of III.3.1 on page 85 in \cite{chevalley}).
 
 \begin{rem}
 \label{remark-non-equivariance-of-varphi-tilde}
 We emphasize that $\psi$ depends on the choice of $B_0$. Hence so does $\tilde{\varphi}$.
 Note that the projection $\bar{B}_0$ of $B_0$ to $\wedge^2 V^*:=V^*\otimes V^*/Sym^2(V^*)$ is equal to the projection of $B_0+(\bullet,\bullet)_V$ and thus satisfies 
 \[
 \bar{B}_0((w_1,\theta_1),(w_2,\theta_2))=\frac{1}{2}(\theta_2(\omega_1)-\theta_1(\omega_2)).
 \]
 The right hand side is $-1/2$ times the alternating form of the Poincare line bundle $\P$ \cite[Th. 2.5.1]{BL} and so equal to $\frac{1}{2}c_1(\P)$,  by \cite[2.6 (2b)]{BL}. The choice of $B_0$ is done so that the construction is over the integers. Working over $\QQ$ one can replace $B_0$ by $\frac{1}{2}(\bullet,\bullet)_V$ in the above construction. In that case the resulting isomorphism $S_\QQ\otimes S_\QQ\rightarrow \wedge^*V_\QQ$ is $\Spin(V)$-equivariant with respect to the $\Spin(V)$-action on $\wedge^*V_\QQ$
 induced by that on $V_\QQ$ (and preserving the grading) (see \cite[Th. 1(i)]{trautman}, where $\tilde{\varphi}$ is denoted by $E$).
 \end{rem}

 \begin{lem}
 \label{lemma-symmetric-or-alternating-product-of-two-pure-spinors-has-weight-2}
 \begin{enumerate}
 \item
 \label{lemma-item-weight-4n-2}
 The element 
 $\tilde{\varphi}([pt_{X}]\otimes 1-(-1)^n1\otimes[pt_{X}])$ belongs to $F^{4n-2}(\wedge^*V)$, but not to 
 $F^{4n-3}(\wedge^*V)$.
 \item
 \label{lemma-item-weight-4n}
 The element 
 $\tilde{\varphi}([pt_{X}]\otimes 1+(-1)^n1\otimes[pt_{X}])$ does not belong to $F^{4n-1}(\wedge^*V)$.
 \end{enumerate}
 \end{lem}
 
 \begin{proof}
 (\ref{lemma-item-weight-4n-2})
 It suffices to prove that $\varphi([pt_{X}]\otimes 1-(-1)^n1\otimes[pt_{X}])$
 belongs to $C(V)_{4n-2}$ but not to $C(V)_{4n-3}$. We have
 \begin{eqnarray*}
 \varphi(1\otimes[pt_X])&=&1\cdot[pt_{\hat{X}}]\cdot \tau([pt_X])=(-1)^{\frac{2n(2n-1)}{2}}\cdot[pt_{\hat{X}}]\cdot [pt_X]
 =(-1)^n\cdot[pt_{\hat{X}}]\cdot [pt_X],
 \\
 \varphi([pt_{X}]\otimes 1)&=&[pt_X]\cdot [pt_{\hat{X}}].
 \end{eqnarray*}
 Hence, $\varphi([pt_{X}]\otimes 1-(-1)^n1\otimes[pt_{X}])=[pt_X]\cdot[pt_{\hat{X}}]-[pt_{\hat{X}}]\cdot [pt_X]$. It remains to prove that $[pt_X]\cdot[pt_{\hat{X}}]-[pt_{\hat{X}}]\cdot [pt_X]$ belongs to 
 $C(V)_{4n-2}$ but not to $C(V)_{4n-3}$.
% $F^{4n-2}(\wedge^*V)$, but not to  $F^{4n-3}(\wedge^*V)$.

Choose a basis $\{e_1, \dots, e_{2n}\}$ of $H^1(X,\ZZ)$ satisfying $[pt_X]=e_1\wedge \cdots\wedge e_{2n}$.
Let $\{f_1, \dots, f_{2n}\}$ be the dual basis of $H^1(\hat{X},\ZZ)$. Then $[pt_{\hat{X}}]=f_1\wedge \dots\wedge f_{2n}.$
We have
\begin{eqnarray*}
e_{2n-1}e_{2n}f_{2n}f_{2n-1}&=& e_{2n-1}f_{2n-1}-e_{2n-1}f_{2n}e_{2n}f_{2n-1}=
e_{2n-1}f_{2n-1}-f_{2n}e_{2n-1}f_{2n-1}e_{2n}
\\
&=&f_{2n-1}f_{2n}e_{2n}e_{2n-1}+[e_{2n-1}f_{2n-1}+e_{2n}f_{2n}]-1.
\end{eqnarray*}
The terms above commute in $C(V)$ with $\{e_1, \dots, e_{2n-2}\}$ and $\{f_1, \dots, f_{2n-2}\}$. Hence,
\[
\begin{array}{lc}
(e_1e_2\cdots e_{2n})(f_{2n}f_{2n-1}\cdots f_1)&=
\\
(e_1e_2\cdots e_{2n-2})(f_{2n-2}f_{2n-3}\cdots f_1)\{f_{2n-1}f_{2n}e_{2n}e_{2n-1}+[e_{2n-1}f_{2n-1}+e_{2n}f_{2n}]-1\} &=
\\
\prod_{k=1}^n \{f_{2k-1}f_{2k}e_{2k}e_{2k-1}+[e_{2k-1}f_{2k-1}+e_{2k}f_{2k}]-1\}.
\end{array}
\]
We see that the commutator $[pt_X][pt_{\hat{X}}]-[pt_{\hat{X}}][pt_X]$ belongs to $C(V)_{4n-2}$ and its projection to
$C(V)_{4n-2}/C(V)_{4n-3}$ maps to $\wedge^{4n-2}V$ as the element
\[
2\sum_{k=1}^n
\wedge_{j=1,j\neq k}^n \left(f_{2j-1}\wedge f_{2j}\wedge e_{2j}\wedge e_{2j-1}
\right)\wedge \left[
e_{2k-1}\wedge f_{2k-1}+e_{2k}\wedge f_{2k}
\right].
\]
The above is a sum of $2n$ linearly independent terms, hence it does not vanish.

(\ref{lemma-item-weight-4n})
Clear from the above computation.
 \end{proof}
 
 %\begin{rem}
 %\label{remark-two-dual-filtrations}
 The K\"{u}nneth theorem interprets (\ref{eq-tilde-varphi}) as an isomorphism 
\[
\tilde{\varphi}:H^*(X\times X,\ZZ)\rightarrow H^*(X\times\hat{X},\ZZ).
\]
%Note that $V^*:=H_1(X\times\hat{X},\ZZ)$ comes with a symmetric bilinear pairing, and $\wedge^kV$ is naturally isomorphic to %$\wedge^{4n-k}(V^*)$ under Poincar\'{e} duality. The increasing filtration $F^k(\wedge^*(V^*))$ corresponds 
%via the Poincar\'{e} duality isomorphism $\wedge^*(V^*)\cong \wedge^*V$ to a 
%decreasing filtration $F_k(\wedge^*V):=\oplus_{i\geq k}\wedge^iV$.
%We get a $C(V^*)$-module structure on $\wedge^*V$ and a $\Spin(V^*)$ representation, 
%which preserves the decreasing filtration and acts on the graded summands $\wedge^kV$ via the natural representation.
%\end{rem}
 
%***************
% Hide
%***************
\hide{
The spin representation $S$ is self dual, via the pairing (\ref{eq-Mukai-pairing}),
and so the Clifford algebra $C(V)$ is isomorphic to $S\otimes S$, as a $\Spin(V)$-representation, 
using the isomorphism (\ref{eq-m}). 
The Clifford algebra $C(V)$ is isomorphic as a $\Spin(V)$-representation also to the exterior algebra
$\wedge^*V$, 
by\footnote{
Let $\psi:C(V)\rightarrow \End(\wedge^*V)$ be the unique algebra homomorphism sending an element $v\in V\subset C(V)$ to 
$L_v+\iota_v$, where $L_v(u)=v\wedge u$ and $\iota_v$ is contraction with $B_0(v,\bullet)_V$, where $B_0$ is any bilinear pairing, not necessarily symmetric, which satisfies $B_0(x,x)=\frac{1}{2}(x,x)_V$, for all $x\in V$. 
The isomorphism $C(V)\rightarrow \wedge^*V$ of \cite[II.1.6]{chevalley} sends $x\in C(V)$ to $\psi(x)\cdot 1$ 
(the value of the group endomorphism $\psi(x)$ on the unit $1$ of the algebra $\wedge^*V$). Note that in
\cite[page 85]{chevalley} the bilinear pairing $B_0$ used 
is non-symmetric and it vanishes on the cartesian squares of  $H^1(X,\ZZ)$ and $H^1(\hat{X},\ZZ)$ as well as on 
$H^1(X,\ZZ)\times H^1(\hat{X},\ZZ)$, but on $H^1(\hat{X},\ZZ)\times H^1(X,\ZZ)$ it restricts as $(\bullet,\bullet)_V$. So $\iota_v$ vanishes for $v\in H^1(X,\ZZ)$.
Using this $B_0$ Chevalley shows that the isomorphism $\varphi:S\otimes S\rightarrow C(V)$ defined in \cite[III.3.1]{chevalley} 
conjugates elements of the Clifford group to automorphisms of $C(V)$, whose associated graded acts on $\wedge^kV$ via the natural representation induced from $V$. The latter isomorphism $\varphi$ 
 is the composition of $S\otimes S\cong \End(S)$ composed with the inverse of $m:C(V)\rightarrow \End(S)$ given in (\ref{eq-m}), but working over a field rather than $\ZZ$.
}
\cite[II.1.6]{chevalley}. 
Combining both isomorphisms we get the isomorphism 
\[
\tilde{\varphi}: S\otimes_\ZZ S \rightarrow \wedge^*V.
\]
The K\"{u}nneth theorem interprets it as an isomorphism 
%\begin{equation}
%\label{eq-varphi}
\[
\tilde{\varphi}:H^*(X\times X,\ZZ)\rightarrow H^*(X\times\hat{X},\ZZ).
\]
%\end{equation}
This is (??? for which bilinear pairing $B_0$ ???) the cohomological action of Orlov's derived equivalence 
(\ref{eq-Orlov-derived-equivalence-to-XxX-from-X-times-hat-X}), by\footnote{
The restriction of the isomorphism $\varphi$ to the image of $V$
agrees with that of the cohomological action of Orlov's derived equivalence 
(\ref{eq-Orlov-derived-equivalence-to-XxX-from-X-times-hat-X}), by \cite[Prop. 4.3.7(b)]{golyshev-luntz-orlov}. 
Now $\varphi$ is an algebra isomorphism, with respect to the product of elements of $H^*(X\times X,\ZZ)$ as correspondences
and the Clifford algebra product conjugated to a product on $\wedge^*V$ via the composition (??? no, don't use it ???)
$\wedge^*V\rightarrow \oplus_{k=0}^\infty V\rightarrow C(V)$.
Both algebras are generated by the two embeddings of $V$.
%The associated graded acts via a diffeomorphism of the complex torus, as such is the action of automorphisms associated to 
%a derived equivalence via Orlov's result and of monodromy operators. In particular, it acts via ring automorphisms.  
%Hence, the action of the associated graded is determined by its action on the first cohomology.
} 
\cite[Prop. 4.3.7(b)]{golyshev-luntz-orlov}.
%***************
% End Hide
%***************
}

%***************************************************************************
% 
%***************************************************************************
\subsection{Polarized abelian varieties of Weil type from oriented $K$-secant lines}
\label{sec-polarized-avwt-from-secants}
\begin{assumption} 
\label{assumption-on-rational-secant-plane-P}
The rational plane $P$ is non-isotropic with respect to the pairing (\ref{eq-Mukai-pairing}) and is contained in the Hodge ring of $X$. $\PP(P)$ intersects the spinor variety in two complex conjugate points defined over $K:=\QQ[\sqrt{-d}]$, where $d$ is a positive rational number.
\end{assumption}

Set 
\begin{equation}
\label{eq-f}
f:=\cm_{\sqrt{-d}}:V_\QQ\rightarrow V_\QQ, 
\end{equation}
where $\cm$ is given in (\ref{eq-action-by-imaginary-quadratic-number-field}) and we choose the square root $\sqrt{-d}:=\sqrt{d}\exp(i\pi/2)$ with argument $\pi/2$. 
%Let $f: V_\QQ\rightarrow V_\QQ$ be the isomorphism (\ref{eq-f}). 
The isomorphism $f$ depends on the rational plane $P$ and the choice of the pure spinor $\ell_1$ among $\{\ell_1,\ell_2\}$, which is equivalent to a choice of an orientation of $P$. 
Note that $f$ belongs to $\tilde{O}(V_\QQ)$, $(f(x),f(y))_V=d(x,y)_V$, and $f^2=-d$, by Lemma \ref{lemma-centralizer-of-rho-Spin-V-P}. Hence 
\[
(f(x),y)_V=\frac{1}{d}(f^2(x),f(y))_V=-(x,f(y))_V,
\]
and so $f$ is anti-self-dual.
Let $\Xi_P\in \wedge^2V_\QQ^*$ be the $2$-form
\begin{equation}
\label{eq-Xi}
\Xi_P(x,y):=(f(x),y)_V.
\end{equation}
The $2$-form $\Xi_P$ is non-degenerate, since $f$ is invertible and the pairing  (\ref{eq-pairing-on-V}) on $V$ is non-degenerate.
Note the equality
\[
\Xi_P(x,y)=\sqrt{-d}\left((x_1,y_2)_V-(x_2,y_1)_V\right),
\]
where $x=x_1+x_2$ is the decomposition with $x_i\in W_i$, $i=1,2$, and $y=y_1+y_2$ is the analogous decomposition.
An element $v\in V_\QQ$, admits a unique decomposition $v=v_1^{1,0}+v_2^{1,0}+v_1^{0,1}+v_2^{0,1}$, where
$v_i^{1,0}\in W_{i,\CC}\cap V^{1,0}$ and $v_i^{0,1}\in W_{i,\CC}\cap V^{0,1}$, by Lemma \ref{lemma-decomposition-into-4-direct-summands}. The summands satisfy
\begin{eqnarray*}
\overline{v_1^{1,0}}&=& v_2^{0,1}
\\
\overline{v_1^{0,1}}&=& v_2^{1,0}
\end{eqnarray*}

The complex structure $I:=I_{V_\RR}$ is an isometry\footnote{
Let $X=U/\Lambda$ be an $n$-dimensional compact complex torus, where $U$ is an $n$-dimensional complex vector space and $\Lambda$ is a lattice. The complex vector space $U$ is naturally the pair $(H_1(X,\RR),I_X)$, where $I_X$ is the complex structure of $X$. The dual torus  $\hat{X}$ is $\Hom_{\bar{\CC}}(U,\CC)/\hat{\Lambda}$, where $\Hom_{\bar{\CC}}(U,\CC)$ consists of $\RR$-linear homomorphisms satisfying $\ell(iu)=-i\ell(u)$ and $\hat{\Lambda}$ is the dual lattice with respect to the pairing 
$\langle \ell,u\rangle:=Im\ell(u)$ (see \cite[Sec. 2.4]{BL}). Hence, $\langle i\ell,u\rangle=\langle\ell,-i u\rangle$. 
An element of $\Hom_{\bar{\CC}}(U,\CC)$ is determined by its real part yielding an isomorphism of $H_1(X,\RR)^*$ with $\Hom_{\bar{\CC}}(U,\CC).$
Under this identification of $H_1(X,\RR)^*$ with $\Hom_{\bar{\CC}}(U,\CC)=H_1(\hat{X},\RR)$ the complex structure $I_{\hat{X}}$ acts by composing elements of $H_1(X,\RR)^*$ with $-I_X$. 
This agrees with the complex structure of the dual Hodge structure.
Hence, $(I_X,I_{\hat{X}})$ is an isometry 
of $H_1(X,\RR)\oplus H_1(\hat{X},\RR)$ with respect to the pairing $((u_1,\ell_1),(u_2,\ell_2))=\langle u_1,\ell_2\rangle+\langle u_2,\ell_1\rangle$. 
Now $I_{V_\RR}$ is the complex structure of the dual Hodge structure, which is again the direct sum of two dual Hodge structures, hence
$I_{V_\RR}$ is an isometry with respect to the pairing (\ref{eq-pairing-on-V}).
} of $V_\RR$ 
satisfying $I^{-1}=-I$, hence $I$ is anti-self-dual with respect to the pairing on $V_\RR$. 
The eigenspaces $V^{1,0}$ and $V^{0,1}$ of $I$ in $V_\CC$ are thus isotropic.
Furthermore, $I$ commutes with $f$, by Lemma \ref{lemma-decomposition-into-4-direct-summands}. Hence, $\Xi_P$ is of Hodge-type $(1,1)$ and $f\circ I$ is self-dual. 
We get the symmetric bilinear form on $V_\RR$ given by $g_P(x,y):=\Xi_P(I(x),y)=(f(I(x)),y)_V$. 
%(??? if we want $g_P$ to be a metric, when $\Xi$ is ample, then we need to change the sign and define $g_P(x,y):=-\Xi(I(x),y)=-(f(I(x)),y)_V$, 
%see \cite[1.2.13]{huybrechts-complex-geometry-book} ???).

\begin{lem}
\label{lemma-g-P-in-terms-of-4-summands}
$g_P(v,v)=2\sqrt{d}(-(v_1^{1,0},v_2^{0,1})_V+(v_1^{0,1},v_2^{1,0})_V).$
\end{lem}

\begin{proof}
Our sign convention for square roots yields $\sqrt{-1}\sqrt{-d}=-\sqrt{d}$.
\begin{eqnarray*}
(f(I(v)),v)&=&-\sqrt{d}(v_1^{1,0}-v_2^{1,0}-v_1^{0,1}+v_2^{0,1},v_1^{1,0}+v_2^{1,0}+v_1^{0,1}+v_2^{0,1})_V
\\
&=&-2\sqrt{d}\left(
(v_1^{1,0},v_2^{0,1})_V-(v_1^{0,1},v_2^{1,0})_V
\right),
\end{eqnarray*}
where in the second equality we used the fact that $W_i$, $i=1,2$, $V^{1,0}$, and $V^{0,1}$ are all isotropic with respect to the pairing (\ref{eq-pairing-on-V}).
\end{proof}

\begin{rem}
\label{rem-complex-structure-lifts-to-Spin-V-RR}
The complex structure $I$ on $V_\RR$ lifts to an element of $\Spin(V_\RR)$. This is a special case of the following more general fact. Let $F$ be a field and let $M$ and $N$ be two complementary maximal isotropic subspaces of $V_F$. Let $\varpi:F^\times\rightarrow SO(V_F)$ send $\lambda\in F^\times$ to $\varpi(\lambda)$ acting on $L$ by multiplication by $\lambda^2$ and on $M$ by multiplication by $\lambda^{-2}$. Then $\varpi$ admits a lift  $\tilde{\varpi}:F^\times \rightarrow \Spin(V_F)$ defined as follows. Let $\{e_1, \dots, e_{4n}\}$ be a basis of $V_F$, such that $e_i\in L$, for $1\leq i\leq 2n$,  $e_i\in M$, for $2n+1\leq i\leq 4n$,  $(e_i,e_j)=1$, if $j=2n+i$ or $i=2n+j$, and $(e_i,e_j)=0$ otherwise. Then 
\[
\tilde{\varpi}(\lambda):=\prod_{k=1}^{2n}(\lambda^{-1}+(\lambda-\lambda^{-1})e_ke_{k+2n})
\]
is an element of $\Spin(V_F)$, which maps to $\varpi(\lambda)\in SO(V_F)$ (apply the second displayed formula in \cite[Sec. 2]{igusa}). If $I$ is a complex structure on $V_\RR$ take $L=V^{1,0}$, $N=V^{0,1}$, and $\lambda=e^{\pi i/4}$, to get the element
$\tilde{\varpi}(e^{i\pi/4})$ of $\Spin(V_\CC)$ which must be already in $\Spin(V_\RR)$ as it maps to $I=\varpi(e^{\pi i/4})$ in $SO(V_\RR)$.
\end{rem}

Let $\Theta\in H^{1,1}(X,\ZZ)$ be an ample class. 
Let 
\begin{equation}
\label{eq-theta-is-contraction-with-Theta}
\theta:H^1(X,\QQ)^*\rightarrow H^1(X,\QQ)
\end{equation}
be contraction with $\Theta$ and denote its $K$-linear extension by $\theta:H^1(X,K)^*\rightarrow H^1(X,K)$ and similarly its $\RR$ and $\CC$-linear extensions. 
Note that $\theta$ is an isomorphism of rational Hodge structures under the identification of 
$H^1(X,\QQ)^*$ with $H^1(\hat{X},\QQ)$. 
In fact, $-\theta$ is the pullback homomorphism associated to the isogeny
$\phi_L:X\rightarrow \hat{X}$ sending $x$ to $\tau_x^*(L)\otimes L^{-1}$, where $L$ is any line bundle on $X$ with Chern class $\Theta$, by\footnote{
Let $H$ be the hermitian form on the complex vector space $U:=(H_1(X,\RR),I_X)$ with imaginary part $-\Theta$. Then the differential 
%(called analytic representation in \cite{BL}) 
$\phi_H$ of $\phi_L$ is given by $u\mapsto H(u,\bullet)$, by \cite[Lemma 2.4.5 and 3.6.4]{BL}. The pairing between 
$U$ and $H^1(\hat{X},\RR):=\Hom_{\bar{\CC}}(U,\CC)$ is given by $\langle u,\ell\rangle=Im\ell(u)$. Given $u_1, u_2\in  U$, we have
\[
\langle u_1,\phi_\Theta(u_2)\rangle=Im H(u_2,u_1)=-\Theta(u_2,u_1).
\]
Thus, $\phi_H(u_2)=-\Theta(u_2,\bullet)=-\theta(u_2).$ So $-\theta:H^1(X,\RR)^*\rightarrow H^1(X,\RR)$ gets identified with the differential $\phi_H$ of $\phi_L$ under the identification of its domain $H^1(X,\RR)^*$ with $H_1(X,\RR)$ and its codomain $H^1(X,\RR)$ with $H_1(\hat{X},\RR)$ . Now use the self-duality $\phi_L^*=\phi_L$ \cite[Cor 2.4.6]{BL}.
} 
\cite[Lemma 2.4.5]{BL}. 

Let $d$ be a positive integer. Set $K:=\QQ[\sqrt{-d}]$. 
Set $u:=\sqrt{-d}\Theta$. Cup product with $\exp(u)$ is a linear automorphism
of $H^*(X,K)$ corresponding to the spin representation image of an element $\exp(u)$ of $\Spin(V_K)$ which acts on $V_K$ by
\begin{equation}
\label{eq-action-of-exp-u-on-V}
\exp(u)\cdot (w,y)=(w-\sqrt{-d}\theta(y),y),
\end{equation}
$w\in H^1(X,K)$ and $y\in H^1(\hat{X},K)\cong H^1(X,K)^*$ (see \cite[III.1.7]{chevalley} and its proof).
Note that $\exp(u)$ leaves invariant the pure spinor $[pt]\in H^{2n}(X,K)$ corresponding to the maximal isotropic subspace
$H^1(X,K)$ and leaves invariant every element of the latter subspace. On the other hand $\exp(u)$ takes the pure spinor $1$ corresponding to $H^1(X,K)^*$ to the class 
$\exp(u)\in H^{even}(X,K)$. We set $\ell_1:=\span_K\{\exp(u)\}$, $\ell_2:=\span_K\{\exp(\bar{u})\}$, and 
\begin{equation}
\label{eq-P}
P:=\span\left\{Re(\exp(u)),\frac{Im(\exp(u))}{\sqrt{d}}\right\}
\end{equation} 
oriented via the choice of $\ell_1$.
The maximal isotropic subspaces  corresponding to $\ell_1$ and $\ell_2$ are 
%\begin{equation}
%\label{eq-W-1-exp-u-of-N}
\begin{equation}
\label{eq-W-1-and-2}
\begin{array}{lcccc}
W_1&:=&\exp(u)(H^1(X,K)^*)&=&
\{(-\sqrt{-d}\theta(y),y) \ : \
y\in H^1(X,K)^*
\},
\\
W_2&:=&\overline{W_1}&=&
\{(+\sqrt{-d}\theta(y),y) \ : \
y\in H^1(X,K)^*
\}.
\end{array}
\end{equation}
%\end{equation}
The intersection $W_1\cap W_2$ is the zero subspace, since $\theta$ is an isomorphism.

\begin{prop}
\label{prop-polarized-abelian-variety-of-Weil-type}
The symmetric bilinear pairing $g_P$ of Lemma \ref{lemma-g-P-in-terms-of-4-summands} associated with the oriented plane $P$ in (\ref{eq-P}) is negative definite.
\end{prop}

\begin{proof}
Let $x$ be a class in $V_\QQ$ and write $x=x_1+x_2$, with $x_i\in W_i$. Then 
$x_1=(-\sqrt{-d}\theta(y),y)$, for some $y\in H^1(X,K)^*$, and $x_2=\bar{x}_1=(\sqrt{-d}\theta(\bar{y}),\bar{y})$.
%We identify $H^1(\hat{X},\RR)$ with $\Hom_{\bar{\CC}}(H^1(X,\RR),\CC)$, so that $I_{X\times\hat{X}}$ acts on 
We have noted that $\theta$, given in (\ref{eq-theta-is-contraction-with-Theta}), is an isomorphism of rational Hodge structures under the identification $H^1(X,\QQ)^*$ with $H^1(\hat{X},\QQ)$. We identify $H^1(X,K)^*$  with $H^1(\hat{X},K)$ and regard it as a subset of $H^1(\hat{X},\CC)$.
Hence, $\theta(y)^{1,0}=\theta(y^{1,0})$ and $\theta(y)^{0,1}=\theta(y^{0,1})$ and so
\begin{eqnarray*}
x_1^{1,0}&=&(-\sqrt{-d}\theta(y^{1,0}),y^{1,0}),
\\
x_1^{0,1}&=&(-\sqrt{-d}\theta(y^{0,1}),y^{0,1}).
\end{eqnarray*}
The first equality below follows from Lemma \ref{lemma-g-P-in-terms-of-4-summands}, the second by definition,  the third from 
the fact that the direct summands $H^1(X,\CC)$ and $H^1(\hat{X},\CC)$ of $V_\CC$ are isotropic. The fourth, since $V^{1,0}$ and $V^{0,1}$ are isotropic.
\begin{eqnarray*}
g_P(x,x)&=&2\sqrt{d}(-(x_1^{1,0},\overline{x_1^{1,0}})_V
+(x_1^{0,1},\overline{x_1^{0,1}})_V)
\\
&=&
2\sqrt{d}\left[
-((-\sqrt{-d}\theta(y^{1,0}),y^{1,0}),(\sqrt{-d}\theta(\overline{y^{1,0}}),\overline{y^{1,0}}))_V
\right.
\\
& & 
\left.
+
((-\sqrt{-d}\theta(y^{0,1}),y^{0,1}),(\sqrt{-d}\theta(\overline{y^{0,1}}),\overline{y^{0,1}}))_V
\right]
\\
&=& 2di\left[
-(y^{1,0},\theta(\overline{y^{1,0}}))_V
+(\theta(y^{1,0}),\overline{y^{1,0}})_V
+(y^{0,1},\theta(\overline{y^{0,1}}))_V
-(\theta(y^{0,1}),\overline{y^{0,1}})_V
\right]
\\
&=& 
2di\left[
(-y^{1,0}+y^{0,1},\theta(\bar{y}))_V+
(\theta(y^{1,0}-y^{0,1}),\bar{y})_V
\right]
\\
&=&
2d
\left[
-(I(y),\theta(\bar{y}))_V
+(\theta(I(y)),\bar{y})_V
\right]
\\
&=& -4d\Theta(\bar{y}\wedge I(y)).
%=4d\Theta(I(y)\wedge \bar{y}).
\end{eqnarray*}

Write $y=a+ib$, where $a,b$ in $H^1(\hat{X},\RR)$. Then
$\bar{y}\wedge I(y)=(a-ib)\wedge(I(a)+iI(b))=a\wedge I(a)+b\wedge I(b)+i[a\wedge I(b)- b\wedge I(a)].$
The fact that $\Theta$ is of type $(1,1)$ yields
\[
\Theta(a\wedge I(b))=\Theta(I(a)\wedge I^2(b))=-\Theta(I(a)\wedge b)=\Theta(b\wedge I(a)).
\]
Hence, $\Theta(\bar{y}\wedge I(y))=\Theta(a\wedge I(a))+\Theta(b\wedge I(b))$.  The two summands are non-negative  and the sum is non-zero if $y\neq 0$, as $\Theta$ is ample (we use the sign convention of \cite[Lemma 1.2.15]{huybrechts-complex-geometry-book}).
\end{proof}

%***************************************************************************
% 
%***************************************************************************
\section{A $\Spin(V)_P$-invariant Hermitian form}
\label{sec-Hermitian-form}
In Section \ref{subsec--Hermitian-form} 
we show that an oriented $K$-secant $P$, for a quadratic imaginary number field $K$, determines a $\Spin(V)_P$-invariant hermitian form on $V$, up to a rational scalar. We show that the examples of $2n$-dimensional polarized abelian varieties of Weil type, considered in Section \ref{sec-polarized-avwt-from-secants}, all have discriminat $(-1)^n$. In Section \ref{sec-complex-structures-in-Spin-V-P}
we give a criterion for an element of $\Spin(V)_P$ to act on $V_\RR$ as a complex structure of an abelian variety of Weil type on $V_\RR/V$.
%***************************************************************************
% 
%***************************************************************************
\subsection{The Hermitian form}
\label{subsec--Hermitian-form}
Keep Assumption \ref{assumption-on-rational-secant-plane-P}. Let 
%\begin{equation}
%\label{eq-f}
$f:=\cm_{\sqrt{-d}}:V_\QQ\rightarrow V_\QQ$
%\end{equation}
be the similarity given in Equation (\ref{eq-f}).
% given in (\ref{eq-action-by-imaginary-quadratic-number-field}) and we choose the square root $\sqrt{-d}:=\sqrt{d}\exp(i\pi/2)$ with argument $\pi/2$. 
Let 
\begin{equation}
\label{eq-so-f}
SO_+(V_\QQ)_f 
\end{equation}
be the subgroup of $SO_+(V_\QQ)$ of elements $g$, which commute with $f$ and which
restriction $\restricted{g}{W_i}$ to each of the eigenspaces $W_i$ of $f$ satisfy $\det(\restricted{g}{W_i})=1$, $i=1,2$.
Let $\sigma$ be the generator of $Gal(K/\QQ)$. 
Let $u_1\in S^+_K$ be an even pure spinor corresponding to $W_1$. Then $u_2:=\sigma(u_1)$ is an even pure spinor corresponding to $W_2$.

\begin{lem}
\label{lemma-stabilizer-is-isomorphic-to-so-f}
The stabilizer $\Spin(V_\QQ)_P$ 
%in Lemma \ref{lemma-imaginary-quadratic-field-is-centralizer} 
is mapped by $\rho$ isomorphically onto $SO_+(V_\QQ)_f$. 
\end{lem}

\begin{proof}
The homomorphism $\rho$ restricts to an injective homomorphism from the stabilizer $\Spin(V_\QQ)_P$ into $SO_+(V_\QQ)_f$, 
since the kernel of $\rho$ has order two, generated by $-1\in C(V_\QQ)$, and the latter acts by $-id_{S_\QQ}$ on the spin representation. Hence, the kernel of $\rho$ intersects $\Spin(V_\QQ)_P$ trivially.
%by the argument in the first paragraph of the proof of Lemma \ref{lemma-secant-to-spinor-variety}. 
It remain to prove its surjectivity.
Let $g$ be an element of  $SO_+(V_\QQ)_f$. The homomorphism $\rho:\Spin(V_\QQ)\rightarrow SO_+(V_\QQ)$
is surjective, so we may choose an element $\tilde{g}\in\Spin(V_\QQ)$ satisfying $\rho(\tilde{g})=g$.
There exists a $\Spin(V_\QQ)$ equivariant homomorphism 
$\wedge_+^{2n}V_K\rightarrow \Sym^2(S^+_K)$, from the subspace $\wedge^{2n}_+V_K$ of $\wedge^{2n}V_K$ spanned by the top exterior powers $\wedge^{2n}W$ of maximal isotropic subspaces $W$, which maps the line $\wedge^{2n}W$ to the line spanned by the square $u^2$ of the corresponding even pure spinor $u$, by
\cite[III.3.2 and III.4.5]{chevalley}. 
The element $\tilde{g}$ acts on the line spanned by $u_i^2$ as the identity, since $\wedge^{2n}g$ acts on 
$\wedge^{2n}W_i$ as the identity, by definition of $SO_+(V_\QQ)_f$. Hence,
$\tilde{g}$ acts on the line $[u_i]$ spanned by $u_i$ via multiplication by a scalar $\lambda_i$ equal to $1$ or $-1$. Now, $u_2=\sigma(u_1)$ and $\tilde{g}=\sigma\circ \tilde{g} \circ \sigma$, since $\tilde{g}$ is defined over $\QQ$. Hence $\lambda_2=\sigma(\lambda_1)=\lambda_1$. We conclude that one of $\tilde{g}$ or $-\tilde{g}$ 
acts as the identity on the plane $P=\mbox{span}\{u_1,u_2\}$ and so belongs to $\Spin(V_\QQ)_P$.
\end{proof}

Consider the $K$-valued bilinear form on $V_\QQ$ given by
\begin{equation}
\label{eq-H}
H(x,y):=d(x,y)_V+\sqrt{-d}(f(x),y)_V.
\end{equation}

\begin{lem}
\label{lemma-su-3-3}
$H$ is an $SO_+(V_\QQ)_f $-invariant Hermitian form on $V_\QQ$ considered as a $K$-vector space, i.e., we have
$H(x,y)=\sigma(H(y,x))$ and $H(x,\cm_\lambda(y))=\lambda H(x,y)$, for $\lambda\in K$. The signature of $H$ is $(n,n)$. 
The group $SO_+(V_\QQ)_f$ is a finite index subgroup of the subgroup $SU(V_\QQ,H)$ of $SL(V_\QQ)$ leaving $H$ invariant. 
\end{lem}

\begin{proof}
The automorphism $f$ is anti-self-dual with respect to the pairing (\ref{eq-pairing-on-V}), $(f(x),y)_V=(x,-f(y))_V$, for all $x, y\in V_\QQ$,  and $f^2=-d \one_{V_\QQ}$,
where $\one_{V_\QQ}$ is the identity endomorphism of $V_\QQ$.
So, $\tilde{f}:=\frac{1}{\sqrt{d}}f$ is a complex structure and an isometry of $V_\RR$ 
and $(x,y)_V+i(\tilde{f}(x),y)$ is a $\CC$-valued Hermitian form on $V_\RR$.
Multiplying by $\sqrt{d}$ we see that  $H$, given in (\ref{eq-H}), 
is a $K$-valued Hermitian form 
on $V_\QQ$, considered as a ${2n}$-dimensional $K$-vector space.
$H$ is $SO_+(V_\QQ)_f $-invariant, since $f$ centralizes $SO_+(V_\QQ)_f $. 
%We get an embedding 
%$\rho(\Spin(V_\QQ)_w)\rightarrow SU(V_\QQ,H)$, which is an isomorphism, 
%as both groups are $\QQ$-algebraic groups of the same dimension, connected in the Zariski topology. 

The signature of $H$ is $(a,b)$ if the $2n\times 2n$ diagonal Gram matrix $G$ (necessarily with 
rational entries) of the quadratic form $H(x,x)$ with respect to some orthogonal $K$-basis $\{v_1, \dots, v_{2n}\}$
of $V_\QQ$ has $a$ positive and $b$ negative diagonal entries. 
%Note that the quadratic form $H(x,x)$ is that of the real part of the Hermitian form and 
%$Re(H(v,f(v)))=d(v,f(v))_V=0$, for all $v\in V_\QQ$, since $f$ is anti-self-dual.
The quadratic form depends only on the real part of $H$.
Hence, $\{v_1, \dots, v_{2n},f(v_1), \dots, f(v_{2n})\}$ is an orthogonal $\QQ$-basis of $V_\QQ$ with respect to the real part of $H$ and the Gram matrix of 
the quadratic form is diagonal of the form
$\left(\begin{array}{cc}G & 0\\ 0 & dG\end{array}\right)$. 
Hence, if the signature of $H$ is $(a,b)$ then that of the latter is $(2a,2b)$.
On the other hand, the latter quadratic form is that of $d$
times the bilinear pairing (\ref{eq-pairing-on-V}) on $V_\QQ$, which has signature $(2n,2n)$. Hence, $a=b=n.$

The norm character has finitely many values on $SO(V_\QQ)$, hence it suffices to show that the group $SO(V_\QQ)_f$ of isometries of $(V_\QQ,(\bullet,\bullet)_V)$ commuting with $f$ and restricting to $W_i$ with determinant $1$ is equal to $SU(V_\QQ,H)$. The inclusion $SO(V_\QQ)_f\subset SU(V_\QQ,H)$
is clear. It remains to show that the subgroup of $SU(V_\QQ,H)$ of elements, which restrict to $W_i$ with determinant $1$, $i=1,2$, is a finite index subgroup.
%The inclusion $SU(V_\QQ,H)\subset SO(V_\QQ)_f$ would follow, once we prove that elements of $SU(V_\QQ,H)$ restrict to $W_i$ with determinant $1$.
Let $\beta:=\{e_1, \dots, e_{2n}\}$ be a basis of $W_1$ and set $\sigma(\beta):=\{\sigma(e_1), \dots, \sigma(e_{2n})\}$. 
Then 
%$\delta:=\{e_1+\sigma(e_1), \dots e_{2n}+\sigma(e_{2n}\}$ is a basis of $V_\QQ$ and 
$\beta\cup\sigma(\beta)$ is a basis of
$V_K$. Let $g$ be an element of $SU(V_\QQ,H)$. Then $g$ commutes with $f$. Let $M$ be the matrix of $\restricted{g}{W_1}$ in the basis $\beta$.
Then the matrix of $\restricted{g}{W_2}$ in the basis $\sigma(\beta)$ is $\sigma(M)$. Now $1=\det(g)$ is equal to the determinant of $g$ as an element of $GL(V_K)$, which is $\det(M)\det(\sigma(M))$. Thus, $\det(M)$ is a unit in the quadratic imaginary number field $K$. The statement follows, since the number of units is finite.
\end{proof}

\begin{lem}
\label{lemma-the-discriminant-is-minus-1-to-the-n}
Assume that the similarity $f:V_\QQ\rightarrow V_\QQ$ given in Equation (\ref{eq-f}) is defined in terms of the oriented plane $P$ given in Equation (\ref{eq-P}).
Then the discriminant of the hermitian form $H$ is $(-1)^n$.
\end{lem}

\begin{proof}
Given a basis $\{y_1, \dots, y_{2n}\}$ of $H^1(\hat{X},\QQ)$ we get the $K$-basis 
$\{(0,y_1), \dots, (0,y_{2n})\}$ of $V_\QQ$.
We evaluate $\det (H((0,y_i),(0,y_j)))$.
Given $y\in H^1(\hat{X},\QQ)$, we get the element $\exp(u)\cdot (0,y)=(-\sqrt{-d}\theta(y),y)$ of $W_1$, by Equation (\ref{eq-action-of-exp-u-on-V}).
%Using $(0,y)=(1/2)[\exp(u)\cdot (0,y)+\overline{\exp(u)}\cdot(0,y)]$ 
We get
\begin{eqnarray*}
(0,y)&=&(1/2)[\exp(u)\cdot (0,y)+\overline{\exp(u)}\cdot(0,y)],
\\
2f(0,y)&=& \sqrt{-d}\exp(u)\cdot (0,y)+(-\sqrt{-d})\overline{\exp(u)}\cdot(0,y)
\\
&=& \sqrt{-d}(-\sqrt{-d}\theta(y),y)-\sqrt{-d}(\sqrt{-d}\theta(y),y)=(2d\theta(y),0),
\\
H((0,y_i),(0,y_j))&=& d((0,y_i),(0,y_j))_V+\sqrt{-d}(f(0,y_i),(0,y_j))
=d\sqrt{-d}\Theta(y_i,y_j),
\\
\det(H(((0,y_i),(0,y_j)))&=&(d\sqrt{-d})^{2n}\det(\Theta(y_i,y_j))=(-1)^nd^{3n}\det(\Theta(y_i,y_j)).
\end{eqnarray*}
%Thus, $\det(H(((0,y_i),(0,y_j)))=(d\sqrt{-d})^{2n}\det(\Theta(y_i,y_j))=(-1)^nd^{3n}\det(\Theta(y_i,y_j)).$
Now, $\det(\Theta(y_i,y_j))$ is the square of a rational number, since $\Theta$ is anti-symmetric, and $d^{3n}=Nm((\sqrt{-d})^{3n})$.
Hence, $\det(H(((0,y_i),(0,y_j)))Nm(K^\times)= (-1)^nNm(K^\times).$
\end{proof}

%*****************
% Hide
%*****************
\hide{
\begin{rem}
\label{rem-stabilizer-as-commutator}
Consider the case where $X$ is an abelian $3$-fold and $w\in S^+_\QQ$ satisfies $J(w)>0$.
Igusa states in the last paragraph of Section 3 of \cite{igusa} that $\Spin(V_\QQ)_w$ is isomorphic to $SU(3,3,K)$. By $SU(3,3,K)$ he seemed to be referring to the subgroup of $SO(V_\QQ)$ leaving $H$ invariant.
We seem to get that $\Spin(V_\QQ)_w$  is a strict subgroup of $SU(3,3,K)$.
%\begin{enumerate}
%\item
%Lemma \ref{lemma-su-3-3}
%shows that $\Spin(V_\QQ)_w$ maps isomorphically onto the subgroup $SO_+(V_\QQ)_f$
%of $SO_+(V_\QQ)$ commuting with the automorphism $f$, given in (\ref{eq-f}), as the latter 
%\item
If $h\in SO(V_\QQ)$ preserves $H$, then 
$(f(h(x)),h(y))_V=(f(x),y)_V=(h(f(x),h(y))_V$, for all $x, y\in V_\QQ$, and so $h$ commutes with $f$.
Conversely, elements of $SO(V_\QQ)$ commuting with $f$ leave $H$ invariant.
An element $g$ of $SO(V_K)$ commuting with $f$ leaves each of the eigenspaces $W_1$ and $W_2$ of $f$ invariant, and $\lambda_i:=\det((g\restricted{)}{W_i})$ are elements of $K^\times$ satisfying $\lambda_1\lambda_2=1$. If, furthermore, $g$ belongs to $SO(V_\QQ)$, then 
$\lambda_2=\sigma(\lambda_1)$.
%$\lambda_1^{2n}=\lambda_2^{2n}=1$, 
%since the composition $V_\QQ\rightarrow V_K=W_1\oplus W_2\rightarrow W_i$, 
%of the inclusion and the projection onto the $i$-th direct summand, $i=1,2$, 
%is a $g$-equivariant $\QQ$-isomorphism and $g$ acts on $V_\QQ$ with determinant $1$.
%\end{enumerate}
If $g$  is the image of an element in $\Spin(V_\QQ)_w$, then $\lambda_i=1$, for $i=1,2$, by Lemma \ref{lemma-secant-to-spinor-variety}.  This last condition does not seem to follow from the condition that $g$ leaves $H$ invariant.
\end{rem}
%*****************
% End Hide
%*****************
}

%\begin{example}
%Let $g$ be an element of $SO_+(V_\RR)_f$, such that $g^2=\one_V$ and the $-1$ eigenspace of $g$ is $2n$-dimensional. Then $I:=\sqrt{d}gf^{-1}$ is a %complex structure on $V_\RR$ and $I$ belongs to $SO(V_\RR)$.
%\end{example}

%*************************************************************************************
% 
%*************************************************************************************
\subsection{Elements of $\Spin(V_\CC)_P$ which are complex sturctures on abelian varieties of Weil type}
\label{sec-complex-structures-in-Spin-V-P}
Let $\tilde{I}$ be an element of $\Spin(V_\RR)_P$, such that $I:=\rho(\tilde{I})$ is a complex structure on $V_\RR$. 
$I$ belongs to $\rho(\Spin(V_\RR)_P)$ and so it commutes with $f:=\cm_{\sqrt{-d}}$, 
by Lemma \ref{lemma-stabilizer-is-isomorphic-to-so-f}. Hence, $I\circ f)^2=d\one_{V_\RR}.$
Let $\nu(I)$ be the multiplicity of the positive square root $\sqrt{d}$ as an eigenvalue of $I\circ f$.

\begin{lem}
\label{lemma-3-dimensional-eigenvalues}
Assume that $\nu(I)=2n$.
%Keep the notation of Lemma \ref{lemma-imaginary-quadratic-field-is-centralizer}. 
%and both the $\sqrt{-d}$ and $-\sqrt{-d}$ eigenspaces of $f\circ I$ are ${2n}$-dimensional. 
Let $V^{1,0}$ and $V^{0,1}$ be the eigenspaces of $I$ in $V_\CC$ with eigenvalues $\pm\sqrt{-1}$.
Then each of $V^{1,0}$ and $V^{0,1}$ intersects each of $W_{1,\CC}$ and $W_{2,\CC}$ along an $n$-dimensional subspace. Furthermore, we have
\begin{equation}
\label{eq-complex-structure-is-determined-by}
V^{1,0}\cap W_{2,\CC}=\left(V^{1,0}\cap W_{1,\CC}\right)^\perp\cap W_{2,\CC},
\end{equation}
where $\left(V^{1,0}\cap W_{1,\CC}\right)^\perp$ is the subspace orthogonal to $V^{1,0}\cap W_{1,\CC}$ with respect to the pairing $(\bullet,\bullet)_V$.
\end{lem}

\begin{proof}
%\ref{lemma-imaginary-quadratic-field-is-centralizer}. 
The elements $I$ and $f$ are simultaneously diagonalizable and 
\[
(V^{1,0}\cap W_{1,\CC})\oplus(V^{1,0}\cap W_{2,\CC})\oplus(V^{0,1}\cap W_{1,\CC})\oplus(V^{0,1}\cap W_{2,\CC})=V_\CC.
\]
The dimension of each direct summand above in $n$, by Remark \ref{remark-P-is-contained-in-the-Hodge-ring-if-I-is-in-image-of-Spin-V-P} and Lemma \ref{lemma-decomposition-into-4-direct-summands} in the special case when $I$ is associated to a complex structure on $X$ and $P$ is contain in the Hodge ring. We provide a proof below for the general case.

The subspace $V^{1,0}\cap W_{1,\CC}$ is the simultaneous eigenspace with eigenvalues $\sqrt{-1}$ for $I$ and $\sqrt{-d}$ for $f$, and we will abbreviate it by saying that it is the $(\sqrt{-1},\sqrt{-d})$-eigenspace.
Similarly, the other three summands are the $(\sqrt{-1},-\sqrt{-d})$, $(-\sqrt{-1},\sqrt{-d})$, and $(-\sqrt{-1},-\sqrt{-d})$-eigenspaces. Now $(I\circ f)^2=d\one_{V_\RR}$ and $I\circ f$ is defined over $\RR$. Hence, its eigenspaces 
$L_1$ and $L_2$ in $V_\CC$, with eigenvalues $-\sqrt{d}$ and $\sqrt{d}$ respectively, are defined over $\RR$. 
Clearly, $L_1=(V^{1,0}\cap W_{1,\CC})+(V^{0,1}\cap W_{2,\CC})$ and
$L_2=(V^{1,0}\cap W_{2,\CC})+(V^{0,1}\cap W_{1,\CC})$.
So $L_1\cap W_{1,\CC}=V^{1,0}\cap W_{1,\CC}$ and $L_2\cap W_{1,\CC}=V^{0,1}\cap W_{1,\CC}$ and similarly for $W_{2,\CC}$. Finally, $\sigma(L_1\cap W_{1,\CC})=L_1\cap W_{2,\CC}$ and
$\sigma(L_2\cap W_{1,\CC})=L_2\cap W_{2,\CC}$. The  equality
$\dim(L_1\cap W_{1,\CC})+\dim(L_1\cap W_{2,\CC})=\dim(L_1)=2n$ 
implies that $\dim(L_1\cap W_{1,\CC})=\dim(L_1\cap W_{2,\CC})=n$ 
and similarly for $L_2$.

The equality
\[
\left(V^{1,0}\cap W_{1,\CC}\right)^\perp=
(V^{1,0}\cap W_{1,\CC})\oplus(V^{1,0}\cap W_{2,\CC})\oplus(V^{0,1}\cap W_{1,\CC}).
\]
holds, since
both subspaces in the above equation are $3n$-dimensional and the inclusion of the 
left hand side in the right hand side follows from the fact that $V^{1,0}$ and $W_{1,\CC}$ are both isotropic.
Equation (\ref{eq-complex-structure-is-determined-by}) follows from the equality
$W_{2,\CC}=(V^{1,0}\cap W_{2,\CC})+(V^{0,1}\cap W_{2,\CC})$ and the one displayed above.
\end{proof}

\begin{cor}
\label{cor-plane-of-Hodge-Weil-classes}
The plane $\wedge^{2n}W_1+\wedge^{2n}W_2$ in $\wedge^{2n}V_K$ is defined over $\QQ$ and consists of rational classes of Hodge-type $(n,n)$ (the Hodge-Weil classes) for every complex structure $I$ in $\rho(\Spin(V_\RR)_P)$ satisfying $\nu(I)=2n$.
%, such that both eigenspaces of $f\circ I$ are $2n$-dimensional.
\end{cor}

\begin{proof}
$\wedge^{2n}W_1+\wedge^{2n}W_2$ is defined over $\QQ$, since it is defined over $K$ and is $\sigma$-invariant.
Each of $\wedge^{2n}W_i$, $i=1,2$, consists of classes of Hodge-type $(n,n)$, by Lemma \ref{lemma-3-dimensional-eigenvalues}.
\end{proof}

Given $x,y\in V_\RR$ and $\tilde{I}\in \Spin(V_\RR)_w$, such that $I:=\rho(\tilde{I})$ is a complex structure on $V_\RR$ as in Corollary \ref{cor-plane-of-Hodge-Weil-classes}, set 
\begin{eqnarray*}
%\label{eq-Theta-P}
\Xi_P(x,y)&:=&(f(x),y)_V,
\\
%\nonumber
g_I(x,y)&:=& \Xi_P(x,I(y))=(f(x),I(y))_V.
\end{eqnarray*}
The automorphisms  $I$ and $f$ of $V_\RR$ commute and both are anti-self-dual with respect to the symmetric bilinear pairing $(\bullet,\bullet)_V$ on $V_\RR$. Hence, $g_I$ is symmetric:
\[
g_I(y,x)=(f(y),I(x))_V=-(y,f(I(x)))_V=-(y,I(f(x)))_V=(I(y),f(x))_V
%=(f(x),Iy)_V
=g_I(x,y).
\]

\begin{cor}
\label{cor-abelian-variety-of-weil-type}
If the bilinear form $g_I$ is positive definite, then the rational $(1,1)$ class $\Xi_P(x,y)$ is a K\"{a}hler class, and so the complex torus $(V_\RR/V_\ZZ,I,\Xi_P)$ is a polarized abelian variety of Weil type.
\end{cor}

\begin{proof}
Set $A:=V_\RR/V_\ZZ$, endowed with the complex structure $I$.
Consider the embedding $K\hookrightarrow \End_\QQ(A)$, which sends $\sqrt{-d}$ to $f$, given in (\ref{eq-f}).
Then 
\[
f^*\Xi_P(x,y):=\Xi_P(f(x),f(y))=(f^2(x),f(y))_V=-(f^3(x),y)_V=d(f(x),y)=d\Xi_P(x,y),
\]
verifying the condition on the polarization in \cite[Def. 4.9]{van-Geemen}.
\end{proof}

%***************************************************************************
% 
%***************************************************************************
\section{An adjoint orbit in  $\Spin(V_\RR)_P$ as a period domain of abelian varieties of Weil type}
\label{sec-period-domains}
Recall that $\rho$ maps $\Spin(V_\RR)_P$ isomorphically onto the subgroup $SO_+(V_\RR)_f$, by Lemma \ref{lemma-stabilizer-is-isomorphic-to-so-f}.
Let 
\begin{equation}
\label{eq-Omega}
\Omega_P \ \subset \ SO_+(V_\RR)_f
\end{equation} be the subset of elements $I$, such that $I$ is a complex structure on $V_\RR$, the eigenspaces of $f\circ I$ are both $2n$-dimensional,  and the bilinear form $g_I$ in Corollary \ref{cor-abelian-variety-of-weil-type} is positive definite. 

Let $\iota:\Omega_P\rightarrow Gr(n,W_{1,\CC})$ be given by $\iota(I):=V^{1,0}_I\cap W_{1,\CC}$.
The map $\iota$ is well defined, by Lemma \ref{lemma-3-dimensional-eigenvalues}.
%\ref{lemma-3-dimensional-eigenvalues}.

\begin{lem}
\label{lemma-coadjoint-orbit-embedds-as-open-subset-of-Grassmannian}
The map $\iota$ is an embedding of $\Omega_P$ as a non-empty subset, open in the classical topology, of the Grassmannian $Gr(n,W_{1,\CC})$.
\end{lem}

\begin{proof}
The map $\iota$ is injective. Indeed,
if $U=\iota(I)$, then $I$ is the unique complex structure on $V_\RR$ with 
\[
V^{1,0}_I=U\oplus \left[U^\perp\cap W_{2,\CC}\right],
\]
by Equation (\ref{eq-complex-structure-is-determined-by}), where $U^\perp$ is the subspace of $V_\CC$ orthogonal to $U$ with respect to the pairing $(\bullet,\bullet)_V$. Given an $n$-dimensional subspace $U$ of
$W_{1,\CC}$, set $V^{1,0}_U:=U\oplus \left[U^\perp\cap W_{2,\CC}\right]$. 
Then $V^{1,0}_U$ is an isotropic $2n$-dimensional subspace of $V_\CC$, invariant under $f$.
The condition that $V^{1,0}_U$ and its complex conjugate $V^{0,1}_U$ are transversal is open. If $U$ is such, let $I_U$ be the complex structure on $V_\RR$ with $V^{1,0}_{I_U}=V^{1,0}_U$. The subspace $V^{1,0}_{I_U}$ is $f$-invariant and $f$ is defined over $\QQ$ and so $V^{0,1}_{I_U}$ is $f$-invariant as well. 
Hence, $I_U$ commutes with $f$. The condition that $g_I$ is positive definite is open, as is the condition 
that $I_U$ preserves the orientation of the positive cone in $V_\RR$.
We conclude that $I_U$ belongs to $SO_+(V_\RR)_f$ for $U$ in an open analytic subset of $Gr(n,W_{1,\CC})$. Furthermore, the $-\sqrt{d}$-eigenspace of $I\circ f$ is $U\oplus \bar{U}$ and is $2n$-dimensional. Hence, $\nu(I)=2n$.

It remains to prove that $\Omega_P$ is non-empty. Note that for $I$ the complex structure of $X\times \hat{X}$ the negative definite bilinear form $g_P$ of 
Proposition \ref{prop-polarized-abelian-variety-of-Weil-type} is $-g_I$. Hence, $g_I$ is positive definite. 
Assumption \ref{assumption-on-rational-secant-plane-P} and Lemma \ref{lemma-decomposition-into-4-direct-summands} imply that $I$ commutes with $f$.
Remark \ref{rem-complex-structure-lifts-to-Spin-V-RR} verifies that $I$ is in $SO_+(V_\RR)$. The restriction of $I$ to $W_i$ has determinant $1$, for $i=1,2$, 
and $\nu(I)=2n$, by Lemma \ref{lemma-decomposition-into-4-direct-summands}. Hence, $I$ belongs to $SO_+(V_\RR)_f$ and so to $\Omega_P$.
%
%The bilinear form $g_I$ is positive definite if and only if the $\sqrt{d}$-eigenspace $L_1$ of $f\circ I$ is negative definite 
%and the $-\sqrt{d}$-eigenspace $L_2$ is positive definite, since $g_I(x,y)=(f(x),I(y)_V=-(x,I(f(y))_V$.
%We constructed such an $I$ in Proposition \ref{prop-polarized-abelian-variety-of-Weil-type} 
%(Remark \ref{rem-complex-structure-lifts-to-Spin-V-RR} verifies that $I$ is in $SO_+(V_\RR)$). Hence, $\Omega_P$ is non-empty.
\end{proof}

\begin{lem}
\label{lemma-period-domain-is-an-adjoint-orbit}
The connected components of $\Omega_P$ are $SO_+(V_\RR)_f$-adjoint orbits. 
\end{lem}

\begin{proof}
Set $I_1:=hIh^{-1}$, 
for some $h\in SO_+(V_\RR)_f$ and $I\in\Omega_P$. Then $f$ commutes with $I$ and $h$ and so
\begin{eqnarray*}
g_{I_1}(x,y)&=&\Xi_P(x,I_1(y))=(f(x),hIh^{-1}(y))_V=
(h^{-1}(f(x)),I(h^{-1}(y)))_V\\
&=&(f(h^{-1}(x)),I(h^{-1}(y)))_V
=g_I(h^{-1}(x),h^{-1}(y)).
\end{eqnarray*}
Hence, $g_{I_1}$ is positive definite as well.

It remains to prove that the dimensions of the adjoint orbits of $I\in\Omega_P$ is equal to that of $\Omega_P$.
The dimension $\dim_\QQ(SO_+(V_\QQ)_f)$ is the dimension $(2n)^2-1$ of $SU(V_\QQ,H)$, by Lemma \ref{lemma-su-3-3}.
%since its elements correspond to elements $g$ of $SL(W_{1,K})$, such that $\bar{g}$ acts on $W_2$ as $(g^*)^{-1}$, 
%when we identify $W_2$ with $W_1^*$ via the pairing $(\bullet,\bullet)_{V_K}$. 
Hence, $\dim_\RR(SO_+(V_\RR)_f)=4n^2-1$.
If $g\in SO_+(V_\RR)_f$ commutes with $I\in\Omega_P$, then $g$ leaves invariant each of the direct summands $W_1^{1,0}$, $W_1^{0,1}$, $W_2^{1,0}$, $W_2^{0,1}$ in 
Lemma \ref{lemma-decomposition-into-4-direct-summands}. Then $V_\RR$ decomposes as the $H$-orthogonal direct sum of two $g$-invariant $H$-non-degenerate subspaces $[W^{1,0}_1\oplus W^{0,1}_2]\cap V_\RR$ and $[W^{0,1}_1\oplus W^{1,0}_2]\cap \RR$. The determinants of the restrictions of $g$
to the two direct summands are inverses of each other. We conclude that the real dimension of the commutator of $I$ is $2\dim(U(n))-1=2n^2-1$. 
%**********
% Hide
%**********
\hide{
Denote by $g_i^{1,0}$ the restriction of $g$ to $W_i^{1,0}$ and define $g_i^{0,1}$ analogously. 
Then $g_2^{0,1}=\overline{g_1^{1,0}}$ and $g_2^{1,0}=\overline{g_1^{0,1}}$, since $\bar{g}=g$. Furthermore, $g_2^{1,0}$ is the inverse transpose of $g_1^{1,0}$
under the identification of $W_2^{1,0}$ with the dual of $W_1^{1,0}$ via the pairing $(\bullet,\bullet)_{V_K}$. 
We see that $g$ is determined by $g_1^{1,0}$.
Furthermore,
\[
\det(g_1^{0,1})=\det(g_2^{0,1})^{-1}=\det(\overline{g_1^{1,0}})^{-1}
\]
Hence, the commutator of $I$ in $SO_+(V_\RR)_f$
is isomorphic to the subgroup of $GL(W_{1,\CC}^{1,0})$ with real determinant, since
\[
1=\det(g_1^{1,0})\det(g_1^{0,1})=\det(g_1^{1,0})\det(\overline{g_1^{1,0}})^{-1}.
\]
We conclude that the real dimension of the commutator of $I$ is $2n^2-1$. 
%**********
% End Hide
%**********
}
Thus the real dimension of the adjoint orbit is $2n^2$.
This is the dimension of each component of $\Omega_P$, by Lemma \ref{lemma-coadjoint-orbit-embedds-as-open-subset-of-Grassmannian}. 
\end{proof}

We may regard $\Omega_P$ as a union of adjoint orbits in $\Spin(V_\RR)_P$, by Lemmas \ref{lemma-stabilizer-is-isomorphic-to-so-f} and \ref{lemma-period-domain-is-an-adjoint-orbit}.
Given a complex structure $I\in\Omega_P$, denote by $\tilde{I}$ the unique element in $\Spin(V_\RR)_P$ satisfying $\rho(\tilde{I})=I$ (Lemma \ref{lemma-stabilizer-is-isomorphic-to-so-f}). Note that $m(\tilde{I})$ need not preserve the grading of $H^*(X,\RR)$.
Consider the subgroup ${\mathbb S}_I:=\{\cos(\theta)id_{V_\RR}+\sin(\theta)I \ : \ \theta\in\RR\}$ of $SO_+(V_\RR)$. The identity component of its inverse image in
 $\Spin(V_\RR)$ defines a real Hodge structure of weight $0$ on $S_\RR^+=H^{ev}(X,\RR)$. In other words, we get a decomposition $S_\CC^+:=\oplus S_\CC^{-p,p}$, with $p$ an integer, satisfying 
$\overline{S^{-p,p}_\CC}=S^{p,-p}_\CC.$ 
We will refer to a rational class in $S^{0,0}_\CC$ as a {\em semi-Hodge class}.

\begin{lem}
\label{lemma-Hodge-Weil-classes-remain-of-Hodge-type-for-all-I-in-Omega-P}
\begin{enumerate}
\item
\label{lemma-item-HW-consists-of-Hodge-classes-for-I-in-Omega-P}
The plane $\wedge^{2n}W_1\oplus\wedge^{2n}W_2$ corresponds to a rational plane in $\wedge^{2n}V_\QQ$, which is spanned by Hodge classes, for every complex structure $I$ in $\Omega_P$. 
\item
\label{lemma-item-P-consists-of-Hodge-classes-for-I-in-Omega-P}
The plane $P$ is spanned by semi-Hodge classes, for every complex structure $I$ in $\Omega_P$. 
\end{enumerate}
\end{lem}

\begin{proof}
Part (\ref{lemma-item-HW-consists-of-Hodge-classes-for-I-in-Omega-P})
is Lemma \ref{cor-plane-of-Hodge-Weil-classes}.
(\ref{lemma-item-P-consists-of-Hodge-classes-for-I-in-Omega-P})
It suffices to show that $I_\theta:=\cos(\theta)id_{V_\RR}+\sin(\theta)I$ belongs to  $SO_+(V_\RR)_f$, for all $\theta\in\RR$, since then 
the inverse image of ${\mathbb S}_I$
%$\{\cos(\theta)id_{V_\RR}+\sin(\theta)I \ : \ \theta\in\RR\}$ 
in $\Spin(V_\RR)$ is disconnected, with the identity component in $\Spin(V_\RR)_P$ and the other component is its product with $-1$ and is disjoint from $\Spin(V_\RR)_P$. 
%$\tilde{I}_\theta:=\cos(\theta)id_V+\sin(\theta)\tilde{I}$ belongs (??? no, $-1$ does not belong  ???) to $\Spin(V_\RR)_P$, for all $\theta\in\RR$, and so 
Hence, classes in $P$ are invariant by under the identity component determining the Hodge structure on $S_\CC$. We have seen that $I$ is anti-self dual with respect to $(\bullet,\bullet)_V$. Hence, $I_\theta^\dagger=\cos(\theta)id_V-\sin(\theta)I$ and
$I_\theta^\dagger I_\theta=id_V$ and so $I_\theta$ belongs to $SO(V_\RR)$, for all $\theta\in\RR$, and hence to the connected component $SO_+(V_\RR)$ of $I$. Clearly, $I_\theta$ commutes with $f$, as $I$ does. 
Now $I_\theta$ acts on $W_i^{1,0}$ by scalar multiplication by $\cos(\theta)+\sqrt{-1}\sin(\theta)$ and on $W_i^{0,1}$ by $\cos(\theta)-\sqrt{-1}\sin(\theta)$
and so the determinant of its restriction to $W_i$ is $1$. Hence, $I_\theta$ belongs to $SO_+(V_\RR)_f$.
\end{proof}

\begin{cor}
\label{corollary-Spin-V-P-invariant-classes-are-Hodge}
The classes in $\left(\wedge^*(V_\QQ)\right)^{\Spin(V)_P}$ remain of Hodge type for every complex structure in $\Omega_P$
\end{cor}

\begin{proof}
A class $\alpha$ in $\wedge^2(V_\QQ)$ is of type $(1,1)$ with respect to a complex structure $I$, if and only if $I(\alpha)=\alpha$.
Hence the statement holds for classes in $\wedge^2(V_\QQ)^{\Spin(V)_P}$. The subring $\left(\wedge^*(V_\QQ)\right)^{\Spin(V)_P}$ 
consists of the direct sum of the rational $2$-dimensional subspace of Hodge-Weil classes in $\wedge^{2n}(V_\QQ)^{\Spin(V)_P}$ and
the space of powers of  classes in the one-dimensional $\wedge^2(V_\QQ)^{\Spin(V)_P}$, by Lemma \ref{lemma-Spin-V-P-invariant-classes-are-Hodge}. 
Hence, the statement follows from Lemma \ref{lemma-Hodge-Weil-classes-remain-of-Hodge-type-for-all-I-in-Omega-P}(\ref{lemma-item-HW-consists-of-Hodge-classes-for-I-in-Omega-P}).
\end{proof}

%**************
% Hide
%**************
\hide{
\begin{example}
\label{example-with-positive-definite-g-I}
(??? this example was generalized and conceptualized  in Proposition \ref{prop-polarized-abelian-variety-of-Weil-type}
%page $3_{\aleph}$ of hand written notes 
???)
Let $\{e_1, e_2, \dots, e_6\}$ be a basis of $H^1(X,\ZZ)$ and $\{f_1, f_2, \dots, f_6\}$ a dual basis of $H^1(\hat{X},\ZZ)$. 
Let $W_1$ be the maximal isotropic subspace spanned over $K$ by 
\[
\{e_1+\sqrt{-d}e_2, -\sqrt{-d}f_1+f_2, e_3+\sqrt{-d}e_4, -\sqrt{-d}f_3+f_4, e_5+\sqrt{-d}e_6, -\sqrt{-d}f_5+f_6\}.
\] 
Then $W_1$ is odd  associated to the odd pure spinor 
\[
u_1=(e_1+\sqrt{-d}e_2)\wedge(e_3+\sqrt{-d}e_4)\wedge(e_5+\sqrt{-d}e_6).
\] 
Similarly, $W_2:=\sigma(W_1)$ is associated to the odd pure spinnor $u_2:=\sigma(u_1)$.
Clearly, $W_1+W_2=V_K.$
%Set $w:=u_1+u_2=-d(e_1\wedge e_4\wedge e_6+e_2\wedge e_3\wedge e_6+e_2\wedge e_4\wedge e_5)$.
Let $L_2$ be the subspace of $V_\RR$ spanned by
\[
\{e_1+df_1, e_2+f_2, e_3+df_3,e_4+f_4,e_5+df_5,e_6+f_6\}.
\]
Set $L_1:=L_2^\perp$. Then $L_2$ is positive definite and $L_1$ is negative definite.
The intersection $L_{2,\CC}\cap W_{1,\CC}$ is $3$-dimensional spanned by 
\[
\{
e_1+df_1+\sqrt{-d}(e_2+f_2), \ 
e_3+df_3+\sqrt{-d}(e_4+f_4), \ 
e_5+df_5+\sqrt{-d}(e_6+f_6)
\}.
\]
$L_{2,\CC}=(L_{2,\CC}\cap W_{1,\CC})+(L_{2,\CC}\cap W_{2,\CC})$ and similarly for $L_{1,\CC}$.
Set $v:=e_1+f_1$. Then $m_v$ acts on $V$ as minus the reflection by $V$ and induces an isomorphism
$m_v:S^-\rightarrow S^+$. 
The subspace $m_v(W_1)$ is an even maximal isotropic subspace with even pure spinor
\[
m_v(u_1)=(e_3+\sqrt{-d}e_4)\wedge(e_5+\sqrt{-d}e_6)+
\sqrt{-d}e_1\wedge e_2\wedge (e_3+\sqrt{-d}e_4)\wedge(e_5+\sqrt{-d}e_6).
\]
Set $w=\frac{1}{2}[m_v(u_1)+\sigma(m_v(u_1))]$. Then
\[
w=[e_3\wedge e_4-d(e_4\wedge e_6)] -d\left[
e_1\wedge e_2 \wedge [e_3\wedge e_6 + e_4\wedge e_5
\right]=[e_3\wedge e_4-d(e_4\wedge e_6)] -d\left[
e_{45}^*+e_{36}^*
\right],
\]
$Pf(X_{12})=d$, $Pf(Y_{12})=d^2$, and
$J(w)=Pf(X_{12})Pf(Y_{12})=d^3$.
The subspace $m_v(L_1)$ is positive definite. Let $h$ be the isometry of $V_\RR$ acting on $m_v(L_1)$ by $1$ and on $m_v(L_2)$ by $-1$. 
Let $f$ act on $m_v(W_1)$ by multiplication by $\sqrt{-d}$ and on $m_v(W_2)$ by multiplication by $-\sqrt{-d}$.
Set $I:=\sqrt{d}h\circ f^{-1}$.  Then $I$ is a complex structure on $V_\RR$ with $V^{1,0}=m_v\left[(W_{1,\CC}\cap L_{1,\CC})+L_{2,\CC}\cap W_{2,\CC}\right]$. The form $g_I$ is positive definite, since the $\sqrt{d}$-eigenspace $m_v(L_1)$ of $f\circ I$ is negative definite.
\end{example}

\begin{question}
Describe the image $\iota(\Omega_P)$ in $Gr(n,W_{1,\CC})$ in terms of the signature of the restriction to a $n$-dimensional subspace $U$ of $W_{1,\CC}$ of the following quadratic form on $W_{1,\CC}$. $W_1$ is a $K$-vector space which is isomorphic as a $2n$-dimensional  $\QQ$-vector space to $V_\QQ$. Pulling back the quadratic form (\ref{eq-pairing-on-V}) on $V_\QQ$ we get one on $W_1$.
\end{question}
%**************
% End Hide
%**************
}

%****************************************************************
% 
%****************************************************************
\section{Equivalences of derived categories}
\label{section-equivalences-of-derived-categories}

In Section \ref{sec-cohomological-action-of-derived-equivalences} we recall that the cohomological action $\Phi^H:H^*(X,\ZZ)\rightarrow H^*(X,\ZZ)$, 
of an auto-equivalence $\Phi$ of the derived category of an abelian variety $X$, corresponds to an element of the image of $\Spin(V)$, 
$V:=H^1(X\times\hat{X},\ZZ)$, via the spin representation, and that its image via the vector representation is
an an element of $SO^+(V)$ preserving the Hodge structure.
In section \ref{sec-Spin-eqivariance-of-convolutions} we consider two representations of $\Spin(V)$ on $S=H^*(X,\ZZ)$, the spin representation $m$, and its conjugate $m^\dagger$ via the main anti-automorphism $\tau$. We recall that the Mukai pairing is invariant with respect to the $m\otimes m$-representation on $S\otimes S$, but the Poincar\'{e} pairing is invariant with respect to $m\otimes m^\dagger.$ Hence, the latter needs to be considered in the context of composition of correspondences.

%****************************************************************
% 
%****************************************************************
\subsection{The cohomological action factors through $\Spin(V)$}
\label{sec-cohomological-action-of-derived-equivalences}
Given two smooth projective varieties $X$ and $Y$ and an object $F$ in the bounded derived category $D^b(X\times Y)$ of coherent sheaves, denote by 
\[
\Phi_F:D^b(X)\rightarrow D^b(Y)
\]
the integral functor given by $\Phi_F(\bullet):=R\pi_{Y,*}(F\otimes L\pi_X^*(\bullet))$, where the tensor product  is the left derived functor. Denote by 
\[
\Psi_F:D^b(Y)\rightarrow D^b(X)
\]
the integral functor given by $\Psi_F(\bullet):=R\pi_{X,*}(F\otimes L\pi_Y^*(\bullet))$.
Given a third smooth projective variety $Z$ and an object $G$ in $D^b(Y\times Z)$, the composition 
$\Phi_G\circ \Phi_F$ is isomorphic to the integral transform $\Phi_{F\ast G}$ given by the convolution
\[
F\ast G := R\pi_{13,*}(L\pi_{12}^*(F)\otimes L\pi_{23}^*G),
\]
where $\pi_{ij}$ is the projection onto the product of the $i$-th and $j$-th factors of $X\times Y\times Z$, numbered from left to right. 
Note that $F\ast G$ is given by the following formula as well:
\[
F\ast G := R\pi_{14,*}(L\pi_{12}^*(F)\otimes L\pi_{23}^*\StructureSheaf{\Delta_{Y}} \otimes L\pi_{34}^*G),
\]
where $\pi_{ij}$ are now projections from $X\times Y\times Y \times Z$ and 
$\Delta_{Y}$ is the diagonal in $Y\times Y$. 
Denote by $F_R$ the object $F^\vee\otimes \pi_X^*\omega_X[\dim(X)]$. Then $\Psi_{F_R}:D^b(Y)\rightarrow D^b(X)$ is the right adjoint functor of $\Phi_F$ and is an inverse,  if $\Phi_F$ is an equivalence. If $\Phi_F$ is an equivalence, then  so is $\Phi_{F_R}$ \cite[Rem. 7.7]{huybrechts-derived-categories-book}. Given an object $E$ in $D^b(X)$ Grothendieck-Verdier-Duality yields the isomorphism 
\begin{equation}
\label{eq-Phi-F-R-is-conjugate-of-Phi-F}
\Phi_{F_R}(E)\cong \left(\Phi_F(E^\vee)\right)^\vee,
\end{equation}
where $E^\vee$ is the derived dual of $E$. 

Let $X$ be an abelian variety. 
The group $\Aut(D^b(X))$, of auto-equivalences, 
fits in the exact sequence
\begin{equation}
\label{eq-short-exact-sequence-of-auto-equivalences}
0\rightarrow 2\ZZ\times X\times \hat{X}\rightarrow \Aut(D^b(X)) \RightArrowOf{ch} \Spin_{Hdg}(V_X)\rightarrow 0,
\end{equation}
where $\Spin_{Hdg}(V_X)$ is the subgroup of $\Spin(V_X)$ preserving the Hodge structure of $V_X$,
and the homomorphism $ch$ sends an auto-equivalence $\Phi_G$ with kernel $G\in D^b(X\times X)$ to the homomorphism
$\phi_G:H^*(X,\ZZ)\rightarrow H^*(X,\ZZ)$, given by
\[
\phi_G(\bullet):=\pi_{2,*}(\pi_1^*(\bullet)\cup ch(G)),
\]
\cite[Prop. 4.3.7 and Cor. 4.3.8]{golyshev-luntz-orlov}. An even integer $2k$ in the factor $2\ZZ$ of the kernel of $ch$ corresponds to an even shift functor mapping an object $F$ to $F[2k]$. The factor $X$ in the kernel of $ch$ acts via translations and the factor $\hat{X}$ via tensorization by line bundles. Equality (\ref{eq-Phi-F-R-is-conjugate-of-Phi-F}) corresponds on the cohomological level to the equality 
\begin{equation}
\label{eq-phi-G-R-is-tau-conjugate-of-phi-G}
\phi_{G_R}=\tau\circ\phi_G\circ\tau,
\end{equation}
where $\tau$ is the involution given in (\ref{eq-Mukai-pairing}). We will see below in Equation (\ref{eq-invariance-of-Poincare-pairing}) that $\phi_{G_R}$ is the adjoint of $\phi_G^{-1}$ with respect to the Poincar\'{e} pairing. In contrast, the cohomological action of $\Psi_{G_R}$ is the adjoint of $\phi_G$ with respect to the Mukai pairing, as the inverse is the adjoint of an isometry and $\phi_G$ is an isometry with respect to the Mukai pairing.

Assume next that $X$, $Y$, and $Z$ are $n$ dimensional abelian varieties. Then the canonical line bundle and the Todd classes are all trivial and so 
\[
ch(F\ast G) = \pi_{14,*}(\pi_{12}^*ch(F)\cup \pi_{23}^*ch(\StructureSheaf{\Delta_Y})\cup\pi_{34}^*ch(G)).
\]
In other words, it is the contraction of the two middle factors in the K\"{u}nneth decomposition 
of $H^*(X\times Y\times Y\times Z,\QQ)$ via the Poincar\'{e} pairing $H^*(Y,\QQ)\otimes H^*(Y,\QQ)\rightarrow \QQ$,
given by $(s,t)\mapsto \int_Ys\cup t.$
The latter pairing is not $\Spin(V_Y)$-invariant, where $V_Y:=H^1(Y,\ZZ)\oplus H^1(\hat{Y},\ZZ)$. 
Consequently, the cohomological convolution is not invariant with respect to the diagonal
$\Spin(V_Y)$ action on the two middle factors in the tensor product $S_X\otimes S_Y\otimes S_Y\otimes S_Z$
of the spin representations.
%The failure of the Poincar\'{e} pairing to be invariant also explains the fact that the cohomological adjoint  
%of the correspondence 
%\[
%\phi_G:=\pi_{Y,*}(\pi_X^*(\bullet)\cup ch(G)):H^*(X)\rightarrow H^*(Y)
%\]
% with the class $ch(G)$ is not the transpose, but is rather given by the correspondence 
%$\phi_{G^R}:H^*(Y)\rightarrow H^*(X)$ associated to the class $ch(G_R)$. 

%****************************************************************
% 
%****************************************************************
\subsection{$\Spin(V)$-equivariance of convolutions}
\label{sec-Spin-eqivariance-of-convolutions}
The $\Spin(V_Y)$-invariance is restored by changing the $\Spin(V_Y)$-action on the right K\"{u}nneth factor of
$H^*(X\times Y)\cong H^*(X)\times H^*(Y)$ or on the left K\"{u}nneth factor of $H^*(Y\times Z)$, as we will show in Corollary \ref{cor-Spin-V-Y-invariance} below.
Let $m:\Spin(V)\rightarrow GL(S)$ be the Spin representation and denote by
\begin{equation}
\label{eq-m-dagger}
m^\dagger: \Spin(V)\rightarrow GL(S)
\end{equation}
the homomorphism $m^\dagger_g:=\tau m_g\tau$, where $\tau$ is given in (\ref{eq-Mukai-pairing}). We have
\begin{equation}
\label{eq-invariance-of-Poincare-pairing}
\int_X m^\dagger_g(s)\cup m_g(t)=(m_g(\tau(s)),m_g(t))_S=(\tau(s),t)_S=\int_X s\cup t,
\end{equation}
for all $s,t\in S$ and $g\in \Spin(V)$, 
where the second equality follows from the $\Spin(V)$-invariance of the Mukai pairing (\ref{eq-Mukai-pairing}). We conclude that the Poincar\'{e} pairing is invariant\footnote{Equivalently, $m_{g^{-1}}^\dagger$ is the left adjoint of $m_g$ with respect to the super-symmetric Poincar\'{e} pairing.}
with respect to the action of $\Spin(V)$ on $S\otimes S$ by
$g\mapsto m^\dagger_g\times m_g.$
Consequently, the Spin-equivariance of the convolution is restored if we let $\Spin(V_X)\times \Spin(V_Y)$ act on $H^*(X\times Y)=S_X\otimes S_Y$ via $m\otimes m^\dagger$ and similarly 
let $\Spin(V_Y)\times \Spin(V_Z)$ act on $H^*(Y\times Z)$ via $m\times m^\dagger$. Note that the action on the product $Y\times Y$ of the second and third factors of $X\times Y \times Y \times Z$ is via
$m^\dagger\times m$ explaining the computation (\ref{eq-invariance-of-Poincare-pairing}).

Observe that the following analogue of equality (\ref{eq-invariance-of-Poincare-pairing}) holds as well.
\begin{equation}
\label{eq-second-invariance-of-Poincare-pairing}
\int_X m_g(s)\cup m^\dagger_g(t)=\int_X s\cup t,
\end{equation}
for all classes $s,t\in H^*(X)$ and all $g\in \Spin(V)$. Indeed, if $s$ and $t$ are both even, then so are $m_g(s)$ and $m^\dagger_g(t)$ and the above follows from 
(\ref{eq-invariance-of-Poincare-pairing}) and the symmetry of the Poincar\'{e} pairing on even cohomology, if $s$ and $t$ are both odd, then so are $m_g(s)$ and $m^\dagger_g(t)$ and
\[
\int_X s\cup t=-\int_X t\cup s\stackrel{(\ref{eq-invariance-of-Poincare-pairing})}{=}
-\int_X m_g^\dagger(t)\cup m_g(s)=\int_X m_g(s)\cup m^\dagger_g(t),
\]
and if one is odd and the other even, then both sides of (\ref{eq-second-invariance-of-Poincare-pairing}) vanish. 

Let $PD:H^*(X,\QQ)\rightarrow H^*(X,\QQ)^*$ be given by $PD(s):=\int_X(\bullet\cup s)$.
Equation (\ref{eq-invariance-of-Poincare-pairing}) yields 
$PD(m_h(s))=PD(s)\circ m^\dagger_{h^{-1}}$, for all $h\in \Spin(V)$.
\[
\xymatrix{
H^*(X) \ar[r]^{PD}\ar[d]_{m_h} &
H^*(X)^* \ar[d]^{(\bullet)\circ m^\dagger_{h^{-1}}}
\\
H^*(X) \ar[r]_{PD} & H^*(X)^*.
}
\]

Given a class $\gamma\in H^*(X\times Y,\QQ)$, let $\gamma_*:H^*(X)\rightarrow H^*(Y)$ be the associated correspondence homomorphism $\gamma_*(\bullet):=\pi_{Y,*}(\pi_X^*(\bullet)\cup \gamma)$. 

\begin{lem}
\label{lemma-action-of-m-h-tensor-1}
The following equalities hold for all $h\in \Spin(V_X)$ and $g\in \Spin(V_Y)$.
\begin{eqnarray}
\nonumber
[(m_h\otimes m_g^\dagger)(\gamma)]_*&=&m^\dagger_g\circ \gamma_*\circ m^\dagger_{h^{-1}},
\\
\label{eq-equivariance-of-mapping-correspondence-to-homomorphism}
{}[(m^\dagger_h\otimes m_g)(\gamma)]_*&=&m_g\circ \gamma_*\circ m_{h^{-1}}.
\end{eqnarray} 
\end{lem}

Note that Equation (\ref{eq-equivariance-of-mapping-correspondence-to-homomorphism}) establishes 
the $\Spin(V_X)\times\Spin(V_Y)$-equivariance of the map
$H^*(X\times Y,\QQ)\rightarrow \Hom(H^*(X,\QQ),H^*(Y,\QQ))$ sending $\gamma$ to $\gamma_*$
with respect to the $m^\dagger\times m$ on the domain and the usual action on the target.

\begin{proof}
We prove the first equality.
Let $\{u_i\}$ be a basis of $H^*(X,\QQ)$ and $\{v_j\}$ a basis for $H^*(Y,\QQ)$. Write
$\gamma:=\sum a_{ij}u_i\otimes v_j$. Then $\gamma_*(\bullet)=\sum a_{ij} PD(u_i)\otimes v_j$.
%where $(s,t)_{PD}=\int_X s\cup t$. 
We get
\[
[(m_h\otimes 1)(\gamma)]_*=\sum a_{ij}PD(m_h(u_i))\otimes v_j
\stackrel{(\ref{eq-invariance-of-Poincare-pairing})}{=}
\sum a_{ij}(PD(u_i)\circ m^\dagger_{h^{-1}})\otimes v_j=\gamma_*\circ m^\dagger_{h^{-1}}(\bullet).
\]
The identity $(1\otimes m^\dagger_g)(\gamma)=m^\dagger_g\circ\gamma_*$ is evident, 
for all $g\in\Spin(V_Y)$.

The proof of the second equality 
(\ref{eq-equivariance-of-mapping-correspondence-to-homomorphism})  is identical, using 
Equation (\ref{eq-second-invariance-of-Poincare-pairing}) instead of (\ref{eq-invariance-of-Poincare-pairing}).
\end{proof}

%\begin{lem}
%\label{lemma-action-of-m-h-dagger-tensor-1}
%$[(m^\dagger_h\otimes m_g)(\gamma)]_*=m_g\circ \gamma_*\circ m_{h^{-1}}$, for all $h\in \Spin(V_X)$ and %$g\in \Spin(V_Y)$.
%\end{lem}

Let $\gamma\in H^*(X\times Y,\QQ)$ and $\delta\in H^*(Y\times Z,\QQ)$.
The following $\Spin(V_Y)$-invariance of composition of correspondence homomorphisms is an immediate corollary of Lemma \ref{lemma-action-of-m-h-tensor-1}.

\begin{cor}
\label{cor-Spin-V-Y-invariance}
The following equalities hold for all $h\in \Spin(V_Y).$
\begin{eqnarray*}
[(m_h\otimes 1)(\delta)]_*\circ [(1\otimes m_h^\dagger)(\gamma)]_*&=&\delta_*\circ\gamma_*,
\\
{}[(m^\dagger_h\otimes 1)(\delta)]_*\circ [(1\otimes m_h)(\gamma)]_*&=&\delta_*\circ\gamma_*,
\end{eqnarray*}
\end{cor}

\begin{rem}
\label{rem-invariance-of-w-vee}
The action of $\tau$ on $S^+_\QQ:=H^{even}(X,\QQ)$ takes the Chern character $w:=ch(F)$ of an object $F$ in $D^b(X)$ to the Chern character $w^\vee:=ch(F^\vee)$ of the dual object 
$F^\vee:=R\SheafHom(F,\StructureSheaf{X})$. 
Hence, $w^\vee$ is $\Spin(V)_w$-invariant with respect to the $m^\dagger$-action.
\end{rem}

%****************************************************************
% 
%****************************************************************
\section{Orlov's derived equivalence $\Phi:D^b(X\times X)\rightarrow D^b(X\times\hat{X})$}
\label{section-Orlov-equivalence}

In Section \ref{sec-equivariance-properties-of-Orlov-equivalence}
we recall the definition of Orlov's equivalence $\Phi:D^b(X\times X)\rightarrow D^b(X\times\hat{X})$ and reduce the proof of its main $\Spin(V)$-equivariance property to the surface case. In Section \ref{sec-objects-with-Spin-invariant-Chern-class} we complete the proof of the $\Spin(V)$-equivariance property
of $\Phi$ in the abelian surface case. We also relate Orlov's equivalence to the construction of the $\Spin(7)$-invariant Cayley class in $H^4(X\times\hat{X},\ZZ)$,
when $X$ is an abelian surface, used in \cite{markman-generalized-kummers} to prove the main results of this paper for abelian fourfolds of Weil type.
In Section \ref{sec-Orlov-equal-Chevalley} we relate the cohomological isomorphism induced by Orlov's equivalence to an isomorphism
$S\otimes S\cong \wedge^*V$ constructed by Chevalley.
In Section \ref{sec-HW-classes-from-squares-of-pure-spinors} we consider a $K$-secant $P\subset \PP(H^*(X,\QQ))$ and use Orlov's equivalence to relate the Hodge-Weil classes in the middle cohomology of $X\times \hat{X}$, associated to the complex multiplication $\eta_P:K\rightarrow \End_\QQ(X\times\hat{X})$, to the plane in $H^*(X\times X,\QQ)$ spanned by the tensor squares of the two pure spinors in $P$.

%****************************************************************
% 
%****************************************************************
\subsection{$\Spin(V)$-Equivariance properties of Orlov's equivalence}
\label{sec-equivariance-properties-of-Orlov-equivalence}

%**********
% Hide
%**********
\hide{
\begin{example}
\label{example-F-i-in-M-2-0-n}
Let $X$ be a principally polarized abelian threefold and $\Theta$ its principal polarization. 
Choose two sheaves $F_1$, $F_2$ in the moduli space $\M(2,0,n\Theta)$ of Example 
\ref{example-gulbrandsen}, such that $F_2$ is not isomorphic to $\tau_x^*(F_1)\otimes L^{-1}$, for any $(x,L)\in X\times\hat{X}$.
Let
\[
\E^i_{F_1,F_2}:=\SheafExt^i_{\pi_{23}}\left(\pi_1^*F_1,\F_2\right)
\]
be the $i$-th relative extension sheaf over $X\times\hat{X}$, where $\F_2$ is associated to $F_2$ as in Equation  
(\ref{eq-universal-spreading-of-F}). Note that $\E^i_{F_1,F_2}$ vanishes if $i\not\in\{1,2\}$, and $\E^1_{F_1,F_2}$ is isomorphic to the dual of the pullback of $\E^2_{F_2,F_1}$ via the inversion map $(x,L)\mapsto (-x,L^{-1})$ over the Zariski open subset consisting of points $(x,L)$,
where the dimension of $\Ext^1(F_1,\tau_x^*(F_2)\otimes L^{-1})$ 
%and $\Ext^1(F_2,\tau_{-x}^*(F_1)\otimes L)$ 
is minimal.
%I can not rule out the possibility that both $\E^1_{F_1,F_2}$ and $\E^2_{F_1,F_2}$ vanish, but 
The techniques of \cite{gulbrandsen} should allow us to compute their rank, which is $0$ when $n=2$, their support, and hopefully their Chern classes, see Section \ref{subsec-n-equal-2-case}.
%since Gulbrandsen computes the dimension of $\Ext^1(F,F)$, for certain $F$
%in $\M(2,0,n\Theta^2)$.
\end{example}

%**********
% End Hide
%**********
}

%\begin{question}
%Compute $ch(R\pi_{23,*}(\pi_1^*F_1^\vee\otimes \F_2))=\pi_{23,*}\left(\pi_1^*ch(F_1^\vee)ch(\F_2)\right)$
%using the monad description of $F_i$, $i=1,2$, given in (\ref{eq-monad}).
%\end{question}

Let $X$ be an abelian variety. Let $\P$ be the normalized Poincar\'{e} line bundle over $X\times \hat{X}$.
Set $n:=\dim_\CC(X)$.
Let $\mu:X\times X\rightarrow X\times X$ be given by $(x_1,x_2)=(x_1+x_2,x_2)$. 
Let $id\times \Phi_\P: D^b(X\times\hat{X})\rightarrow D^b(X\times X)$
be the integral transform with kernel $\StructureSheaf{\Delta_X}\boxtimes \P$. Then
\begin{equation}
\label{eq-Orlov-derived-equivalence-to-XxX-from-X-times-hat-X}
\mu_*\circ (id\times \Phi_\P):D^b(X\times\hat{X})\rightarrow D^b(X\times X)
\end{equation} 
is an equivalence of categories 
which intertwines the action of $\Aut(D^b(X))$ on $D^b(X\times\hat{X})$ and $D^b(X\times X)$, where an autoequivalence $\Phi_\G$ with kernel $\G\in D^b(X\times X)$ acts on $D^b(X\times X)$ via $\Phi_\G\times \Phi_{\G^\vee[n]}$, by \cite[Cor.  9.37]{huybrechts-derived-categories-book}. 
%Let $\sigma:X\times X\rightarrow X\times X$ be the transposition of the factors. Then 
%$\sigma_*\circ\mu_*\circ (id\times \Phi_\P):D^b(X\times\hat{X})\rightarrow D^b(X\times X)$
%intertwines two actions of $\Aut(D^b(X))$ on $D^b(X\times X)$ and $D^b(X\times\hat{X})$.
%The action  on $D^b(X\times X)$ is via $\Phi_{\G^\vee[n]}\times \Phi_\G$.
The action of $\Phi_\G$ on $D^b(X\times\hat{X})$ is the composition of push-forward via the automorphism  associated to the element $\phi_\G$ of $\Spin(V_X)$ preserving the Hodge structure, followed by tensorization by a line bundle on $X\times\hat{X}$.
The inverse 
\begin{equation}
\label{eq-Orlov-derived-equivalence-from-XxX-to-X-times-hat-X}
\Phi:=(id\times \Psi_{\P^{-1}[n]})\circ \mu^* : D^b(X\times X) \rightarrow D^b(X\times\hat{X})
\end{equation}
 intertwines the two actions as well.
 The above functor is induced by an object in $D^b([X\times X]\times [X\times \hat{X}])$, whose Chern character yields a correspondence isomorphism 
 \begin{equation}
 \label{eq-Orlov-cohomological-isomorphism}
 \phi:S\otimes S \rightarrow \wedge^\bullet V
 \end{equation}
 on the level of cohomology.
 Note that $\Psi_{\G^\vee[n]}$ is the quasi-inverse of $\Phi_\G$, while $\Phi_{\G^\vee[n]}$ 
 is the conjugate of $\Phi_\G$ by the dualization functor from $D^b(X)$ to $D^b(X)^{op}$ (see Equation (\ref{eq-Phi-F-R-is-conjugate-of-Phi-F})).
Cohomologically, $\phi_{\G^\vee[n]}:S\rightarrow S$ satisfies $\phi_{\G^\vee[n]}=\tau\phi_\G\tau$, by (\ref{eq-phi-G-R-is-tau-conjugate-of-phi-G}).

 Let $\psi_{\P^{-1}[n]}:H^*(X,\ZZ)\rightarrow H^*(\hat{X},\ZZ)$ be the cohomological correspondence associated to $\Psi_{\P^{-1}[n]}$. Let $\tilde{\varphi}:H^*(X\times X,\ZZ)\rightarrow H^*(\hat{X}\times X,\ZZ)$ be the isomorphism (\ref{eq-tilde-varphi}).
 The following statement is proved in Section \ref{sec-categorification}.
 
 \begin{lem}
 \label{lemma-orlov-isomorphism-is-chevalley}
 $\phi=(\phi_\P\otimes\psi_{\P^{-1}[n]})\circ \tilde{\varphi}\circ (id\otimes \tau).$ 
% where $\tilde{\varphi}:H^*(X\times X,\ZZ)\rightarrow H^*(\hat{X}\times X,\ZZ)$ be the isomorphism (\ref{eq-tilde-varphi}).
\end{lem}

 Given $g\in\Spin(V)$, the composition 
 \begin{equation}
 \label{rho-prime-g}
 \rho'_g:=\phi(m_g\times m^\dagger_g)\phi^{-1}:\wedge^\bullet V\rightarrow \wedge^\bullet V
 \end{equation}
 does not preserve the grading, but it preserves the decreasing filtration by the subspaces
 \begin{equation}
 \label{eq-decreasing-weight-filtration}
 F_k(\wedge^\bullet V):=\oplus_{i\geq k} \wedge^iV, 
 \end{equation}
 and the induced graded action is the one induced from $\rho:\Spin(V)\rightarrow SO(V)$, by 
 \cite[Prop. 4.3.7 and Cor. 4.3.8]{golyshev-luntz-orlov}, and
 more directly by Lemma \ref{lemma-orlov-isomorphism-is-chevalley}.
 We thus get two $\Spin(V)$ actions on
 $\wedge^*V$
 \begin{eqnarray}
 \label{eq-rho-extended-to-exterior-algebra}
 \rho : \Spin(V) & \rightarrow & GL(\wedge^*V),
 \\
\label{eq-rho-prime}
\rho' : \Spin(V) & \rightarrow & GL(\wedge^*V),
\end{eqnarray}
where $\rho'_g$ is given in (\ref{rho-prime-g}) and 
$\rho_g$ acts on $\wedge^kV$ by $\wedge^k\rho_g$ and the latter $\rho_g$ is the image of $g$ via $\rho:\Spin(V)\rightarrow SO(V)$.
 
Given $g\in \Spin(V)$, there exists a topological complex line-bundle $N_g$ on $X\times \hat{X}$, such that 
\begin{equation}
\label{eq-N-g}
\rho'_g=ch(N_g)\cup\rho_g,
\end{equation} 
by \cite[Theorem 2.10]{orlov-abelian-varieties} (see also \cite[Prop. 9.39]{huybrechts-derived-categories-book}).
The equality

\begin{equation}
\label{eq-1-cocycle-identity}
c_1(N_{g_1g_2})=c_1(N_{g_1})+\rho_{g_1}(c_1(N_{g_2})),
\end{equation}
holds for all $g\in \Spin(V)$, by \cite[Exercise 9.41]{huybrechts-derived-categories-book}. 
The above equality means that the function $c_1(N_{(\bullet)}):\Spin(V)\rightarrow H^2(X\times\hat{X},\ZZ)$
is a $1$-cocycle determining a class in the first group cohomology $H^1(\Spin(V),H^2(X\times\hat{X},\ZZ))$. The latter is\footnote{Recall that given a group $G$ and a $\ZZ G$-module $M$ the group cohomology $H^k(G,M)$ is defined to be $\Ext^k_{\ZZ G}(\ZZ,M)$. Given a $1$-cocycle $\nu:G\rightarrow M$, the extention of $\ZZ$ by $M$ corresponding to $\nu$ is the abelian group $\ZZ\oplus M$, on which $g\in G$ acts via $g(z,m)=(z,g(m)+z\nu_g)$.} 
the extension class of the short exact sequence of $\Spin(V)$-modules
\[
0\rightarrow H^2(X\times\hat{X},\ZZ)\rightarrow H^{ev}(X\times\hat{X},\ZZ)/F_4(X\times\hat{X},\ZZ)\rightarrow H^0(X\times\hat{X},\ZZ)\rightarrow 0,
\]
where $F_k(X\times\hat{X},\ZZ)$, $k\geq 0$,  is the decreasing weight filtration 
(\ref{eq-decreasing-weight-filtration})
of cohomology preserved by the $\rho'$-representation (\ref{eq-rho-prime}) of $\Spin(V)$.

\begin{prop} 
\label{prop-extension-class-of-decreasing-filtration-of-spin-V-representations}
The following equality holds 
\begin{equation}
\label{eq-relating-rho-and-rho-prime}
\rho'_g=\exp\left(\frac{1}{2}[c_1(\P)-\rho_g(c_1(\P))]\right)\cup \rho_g.
\end{equation}
\end{prop}

Notice that the integral alternating bilinear form 
$c_1(\P)$ and the symmetric bilinear pairing $(\bullet,\bullet)_V$ agree modulo $2$, and so the class 
$\frac{1}{2}[c_1(\P)-\rho_g(c_1(\P))]$ is indeed integral, for every element $g\in\Spin(V)$.
The statement of the proposition is motivated by  
Remark \ref{remark-non-equivariance-of-varphi-tilde} (its relevance to the isomorphism $\phi$ is explained by 
Lemma \ref{lemma-orlov-isomorphism-is-chevalley}). 

\begin{proof} 
The special case of the statement, when $X$ is an abelian surface, is proved in 
Lemma \ref{example-conjecture-holds-for-abelian-surfaces} below. We prove that the general case follows from the surface case.

\underline{Step 1:}
%See also 
%\cite[Example 9.38]{huybrechts-derived-categories-book} for additional examples in agreement with the conjecture.
If $X_1$ and $X_2$ are abelian varieties and $X=X_1\times X_2$ and the conjecture is known for $X$, then it follows for $X_i$,
since it holds for the image of $\Spin(V_{X_1})\times \Spin(V_{X_2})$ in $\Spin(V_X)$ via the identification of $V_X$ with $V_{X_1}\oplus V_{X_2}$. Indeed, the Poincar\'{e} line bundle $\P_X$ is identified with $\P_{X_1}\boxtimes\P_{X_2}$ under the natural isomorphism $X\times \hat{X}\cong [X_1\times\hat{X}_1]\times[X_2\times\hat{X}_2]$. It suffices to prove the conjecture for one abelian variety in each dimension, hence for powers $E^n$ of an elliptic curve $E$, as it is topological in nature. 
The statement for abelian varieties of dimension $n$ thus implies the statement for abelian varieties of lower dimension. 

\underline{Step 2:}
The case $n=2$ is assumed. 
We prove next that if the statement holds  for abelian varieties of dimension $2k$, $k\geq 1$, then it holds  for abelian varieties of dimension $3k$.
Write $V=V_{E^{3k}}$. Then $V=V_1\oplus V_2\oplus V_3$, where  $V_i=V_{E^k}$, for $1\leq i\leq 3$. Both representations $\rho$ and $\rho'$ of $\Spin(V)$ factor through $SO^+(V)$.
Set
\begin{eqnarray*}
G_1&:=&SO^+(V_1)\times SO^+(V_2\oplus V_3),
\\
G_2&:=&SO^+(V_2)\times SO^+(V_1\oplus V_3),
\\
G_3&:=&SO^+(V_3)\times SO^+(V_1\oplus V_2).
\end{eqnarray*} 
Let $f:G_1\times G_2\times G_3\rightarrow SO^+(V)$ send $(g_1,g_2,g_3)$ to $g_1g_2g_3$.
The subset of $SO^+(V)$, consisting of elements $g$ for which equality (\ref{eq-relating-rho-and-rho-prime}) holds, is a subgroup. Hence,
it suffices to show that the Zariski closure in $SO^+(V_\QQ)$ of the subgroup generated by the image of $f$ is the whole of 
$SO^+(V_\QQ)$. The dimension of $SO^+(V_\QQ)$ is $72k^2-6k$ and any closed subgroup of 
$SO^+(V_\QQ)$ of dimension larger than $(12k-1)(12k-2)/2$ is equal to $SO^+(V_\QQ)$, by
\cite[Theorem B]{obata}. Hence, it suffices to prove that the dimension of the Zariski closure of the image of $f$ is 
$\dim(SO^+(V_\QQ))-2$.

We regard each $G_i$ as a subgroup of $SO^+(V_\QQ)$.
% and let $\overline{G}_i$ be its Zariski closure.
Assume that $f(g_1,g_2,g_3)=1$, where $g_i\in G_i$. Then $g_1g_2$ belongs to 
$G_3$ and so it maps $V_3$ to itself. We claim that $g_2$ maps each $V_i$ to itself, $1\leq i\leq 3$.
It suffices to prove that it maps $V_3$ to itself. Assume otherwise and let $x\in V_3$, such that $g_2(x)=y_1+y_2$, with $y_1\in V_1$ and $y_2\in V_3$ and $y_1\neq 0$. Then $g_1g_2(x)=g_1(y_1)+g_1(y_2)$ and $g_1(y_2)$ belongs to $V_2\oplus V_3$ and $g_1(y_1)$ is a non-zero element of $V_1$. Hence, $g_1g_2(x)$ does not belong to $V_3$. A contradiction. We conclude that $g_2$ maps $V_3$ to itself as well, and so it maps each $V_i$ to itself. Hence, $g_1$ maps $V_3$ to itself, and hence it maps each $V_i$ to itself.
The same must thus follow also for $g_3$. We conclude that the fiber of $f:{G}_1\times{G}_2\times{G}_3\rightarrow SO^+(V_\QQ)$ over $1$ 
has dimension $2[3\dim SO^+(V_{i,\QQ})+1]=12k(4k-1)+2$. Now $\dim({G}_i)=2k[20k-3]$. Hence, the dimension of the image of $f$ is
\[
6k[20k-3]-[12k(4k-1)+2]=72k^2-6k-2=\dim SO^+(V_\QQ)-2.
\]
%The conjecture clearly holds for the monodromy
%subgroup $SL[H^1(X,\ZZ)]$ of $\Spin(V_X)$. The conjecture thus follows for the subgroup of $\Spin(V_{E^d})$
%generated by $SL[H^1(E^d,\ZZ)]$ and the $d$ embeddings of $\Spin(V_{E^{d-1}})\times\Spin(V_E)$ in $\Spin(V_{E^d})$, 
%provided it is known in dimension $d-1$.

\underline{Step 3:} We  prove next the case $n=4$.
We decompose $V_{E^4}$ as the direct sum $V_1\oplus V_2\oplus V_3$, where $V_1=V_2=V_{E^1}$ and $V_3=V_{E^2}$.
We set $G_i$, $1\leq i\leq 3$, and $f:G_1\times G_2\times G_3\rightarrow SO^+(V)$ as in Step 2. Then
$\dim({G}_1\times {G}_2 \times {G}_3)=200,$
$\dim SO^+(V_\QQ)=120$, and the same argument as in Step 2 shows that the 
fiber $f^{-1}(1)$ consists of triples $(g_1,g_2,g_3)$, such that each $g_i$ maps each $V_j$ to itself. Hence, the 
dimension of the fiber $f^{-1}(1)$ is $82$. Again we get that the closure of the image of $f$ has codimension $2$. Hence, 
the Zariski closure of the subgroup generated by the image of $f$ is equal to $SO^+(V_\QQ)$. 

\underline{Step 4:} We complete the proof by induction on $n$. The case $n=2$ is assumed and implies the case $n=1$, by Step 1.
The case $n=3$ follows from Step 2 and the case $n=4$ from Step 3. Assume that $n\geq 4$ and the statement holds for $E^k$,
for $k\leq n$. We prove it for $E^{n+1}$. If $n$ is even, then the statement holds for $E^{3n/2}$, by Step 2 and the induction hypothesis, and so for $n+1$, by Step 1, since $n+1<3n/2$. If $n$ is odd, then $n\geq 5$, the statement holds for $E^{3(n-1)/2}$
by Step 2 and the induction hypothesis, and so for $n+1$, by Step 1, since $n+1\leq 3(n-1)/2$.
\end{proof}

%****************************************************************
% 
%****************************************************************
\subsection{Objects in $D^b(X\times\hat{X})$ with $\Spin(V)_P$-invariant Chern classes}
\label{sec-objects-with-Spin-invariant-Chern-class}
Given a point $x\in X$, let $\tau_x:X\rightarrow X$ be the translation by $x$, given by $\tau_x(y)=x+y$.
%Denote by $\pi_X:X\times \hat{X}\rightarrow X$ the projection.
Let $a:X\times X\rightarrow X$ be the addition morphism, given by $a(x,y)=x+y$.
Let $\pi_{ij}$ be the projection from $X\times X\times \hat{X}$ onto the product of the $i$-th and $j$-th factors. 
Given a coherent sheaf $F$ over $X$ we get the sheaf 
\begin{equation}
\label{eq-universal-spreading-of-F}
\F:= (\pi_{12}^*(a^*(F))\otimes\pi_{13}^*\P^{-1}
\end{equation}
over $X\times [X\times\hat{X}]$. The sheaf $\F$ restricts to $X\times\{(x,L)\}$ as
$\tau_x^*(F)\otimes L^{-1}$. If $F$ is stable in some sense and corresponds to a point $[F]$ in some fine 
moduli space $\M$ of stable coherent sheaves over $X$ with a universal sheaf $\U$ over $X\times \M$, then 
$\F$ is the tensor product of the pullback via $\pi_{23}$ of some line bundle with the
pullback of $\U$ via the morphism $id_X\times\iota_F$, where
\begin{equation}
\label{eq-iota-F}
\iota_F:X\times \hat{X}\rightarrow \M
\end{equation} 
sends $(x,L)$ to
$\tau_x^*(F)\otimes L^{-1}$. The morphism $\iota_F$ is associated to the point $[F]$ via the natural action of $X\times\hat{X}$ on $\M$.
The following statement relates the construction in \cite[Theorem 1.5 and 13.4]{markman-generalized-kummers} for proving the algebraicity of the Weil classes on abelian fourfolds of Weil-type of discriminant $1$, to our strategy in the current paper. In \cite{markman-generalized-kummers} the right hand side of Equation (\ref{eq-object-is-image-of-F-1-dual-times-F-2}) was used to construct a coherent sheaf over $X\times\hat{X}$ from a pair of sheaves on $X$, while in the current paper we use the left hand side.

\begin{lem} The isomorphism
\begin{equation}
\label{eq-object-is-image-of-F-1-dual-times-F-2}
\Phi(F_2\boxtimes F_1^\vee)\cong
R\pi_{23,*}(\pi_1^*F_1^\vee\otimes \F_2)[n],
%((id\times \Psi_{\P^{-1}[n]})\circ \mu^*)(F_2\boxtimes F_1^\vee).
\end{equation} 
holds for all objects $F_i$, $i=1,2$, in $D^b(X\times X)$.
\end{lem}

\begin{proof}
Let $p_i$ be the projection from $X\times X\times X\times \hat{X}$ onto the $i$-th factor. 
\begin{eqnarray*}
(id\times\Psi_{\P^{-1}[n]})(\mu^*(F_2\boxtimes F_1^\vee))) & \cong &
\\
Rp_{34,*}\left\{
p_2^*(F_1^\vee)\otimes p_{12}^*(a^*(F_2))\otimes p_{13}^*(\StructureSheaf{\Delta_X})\otimes p_{24}^*(\P^{-1}[n])
\right\} &\cong&
\\
R\pi_{23,*}\left\{
\pi_1^*(F_1^\vee)\otimes \pi_{12}^*(a^*(F_2))\otimes \pi_{13}^*\P^{-1}
\right\}[n] &\cong& 
R\pi_{23,*}(\pi_1^*F_1^\vee\otimes \F_2)[n].
\end{eqnarray*}
%where we replaced in Equation (\ref{eq-universal-spreading-of-F}) $\P$ by $\P^{-1}$.
\end{proof}

\begin{defi}
\label{def-secant-square-object}
When $ch(F_i)$ belongs to the $K$-secant $P$, for $i=1,2$, 
we will refer to the image (\ref{eq-object-is-image-of-F-1-dual-times-F-2}) of $F_2\boxtimes F_1^\vee$ via Orlov's equivalence as a {\em secant$^{\boxtimes 2}$-object} in $D^b(X\times\hat{X})$.
\end{defi}

Let $F_1$, $F_2$ be objects in $D^b(X)$, such that $ch(F_i)$ belongs to the $K$-secant $P$, for $i=1,2$, such that the $2$-form $\Xi_P$ in Corollary \ref{cor-abelian-variety-of-weil-type} is ample.
%such that $\Spin(V)_{ch(F_i)}=\Spin(V)_w$, $i=1,2$, 
%$w$ as in Lemma \ref{lemma-imaginary-quadratic-field-is-centralizer} 
%(??? when $\dim(X)$ is even then $\Spin(V_K)_w$ has two connected components, 
%by Remark \ref{remark-stabilizer-of-w-may-have-two-connected-components}. 
%The identity component is $\Spin(V_K)_{P_w}$. 
%Better to require that $ch(F_i)$ belongs to $P_w$, for $i=1,2$. In that case the identity component of $\Spin(V)_{ch(F_i)}$ 
%would be $\Spin(V)_{P_w}$, but the other component may be different than that of $\Spin(V)_w$. ???). 
Denote by $H^2(X\times\hat{X},\QQ)_P$
the direct sum of all non-trivial irreducible $\Spin(V)_P$-subrepresentations of $H^2(X\times\hat{X},\QQ).$ Then
$H^2(X\times\hat{X},\QQ)=H^2(X\times\hat{X},\QQ)_P+\QQ\Xi_P$, 
by Lemma \ref{lemma-Spin-V-P-invariant-classes-are-Hodge}.
%where $\Theta_w$ is the $2$-form in Corollary \ref{cor-abelian-variety-of-weil-type}.
Let $k$ be the minimal non-negative integer, such that $ch_k(\Phi(F_2\boxtimes F_1^\vee))\neq 0.$ 

\begin{lem}
\label{ch-3-alpha-is-second-partial-of-J}
%\begin{enumerate}
%\item
%\label{lemma-item-existence-of-ell}
Assume that $k<\dim_\CC(X)$. 
There exists a unique class $\ell$ of type $(1,1)$ in $H^2(X\times\hat{X},\QQ)_P$, such that 
the class $\alpha:=\exp(\ell)ch(\Phi(F_2\boxtimes F_1^\vee))$ is $\Spin(V)_P$-invariant. 
The class $\ell$ depends on the secant line $P$, but not on the choice of $F_i$, $i=1,2$.
%\item
%\label{lemma-item-F-i-in-M-2-0-n}
%When $F_i$ are as in Example \ref{example-F-i-in-M-2-0-n}, then the graded summands $\alpha_i\in H^{2i}(X\times\hat{X},\QQ)$ 
%satisfy $\alpha_i=0,$ for %$i=0,2$, 
%and $\alpha_3=s\hat{J}_{ww}+t\Xi_P^3$, for some $s,t\in\QQ$, with $s\neq 0$. 
%\end{enumerate}
\end{lem}

\begin{proof}
%(\ref{lemma-item-existence-of-ell})
Set $w_i:=ch(F_i)$. 
The class $w_2\otimes w_1^\vee=ch(F_2\boxtimes F_1^\vee)$ is $\Spin(V)_P$-invariant with respect to the diagonal action via $m\times m^\dagger$ on $S_X\otimes S_X$, by Remark \ref{rem-invariance-of-w-vee}.
%\footnote{
%(???) It would have been nicer to choose the $m^\dagger\times m$ action instead as it provides the $\Spin(V)_P$-equivariance 
%with respect to the natural action when correspondences are regarded as homomorphisms, as established in Equation
%(\ref{eq-equivariance-of-mapping-correspondence-to-homomorphism}). However, we chose to keep the conventions 
%of \cite[Cor.  9.37]{huybrechts-derived-categories-book}. 
%}
Let $\beta$ be the class $ch(\Phi(F_2\boxtimes F_1^\vee))$ in $H^*(X\times\hat{X},\QQ)$. Then $\beta$ is $\Spin(V)_P$-invariant with respect to the $\rho'$ representation.
Set $\beta_i:=ch_i(\Phi(F_2\boxtimes F_1^\vee))$.  
%The isomorphism (\ref{eq-object-is-image-of-F-1-dual-times-F-2}) implies that 
Given $g\in \Spin(V)_P$ there exists a topological complex line-bundle $N_g$ on $X\times \hat{X}$, such that 
\begin{equation}
\label{eq-N-g-beta}
\beta=ch(N_g)\rho_g(\beta),
\end{equation}
by Equation (\ref{eq-N-g}).
%\cite[Theorem 2.10]{orlov-abelian-varieties} (see also \cite[Prop. 9.39]{huybrechts-derived-categories-book}).
Note that $\beta_k$ is $\Spin(V)_P$-invariant with respect to  $\rho$, by the minimality of $k$,  and is hence a non-zero scalar multiple of $\Xi_P^k$, by the assumption that $k<\dim_\CC(X)$. In particular, the homomorphism  $\beta_k\cup(\bullet):H^2(X\times\hat{X},\QQ)\rightarrow H^{2k+2}(X\times\hat{X},\QQ)$ is injective, as $\Xi_P$ is ample.
%We conclude that $\beta_1$ is $\Spin(V)_w$-invariant, since the rank of  $R\pi_{23,*}(\pi_1^*F_1^\vee\otimes \F_2)$ is zero and %so $\beta_0=0$. Hence $\beta_1$ is a scalar multiple of $\Theta_w$. 

The coset
$\beta_{k+1}+\beta_k\cup H^2(X\times\hat{X},\QQ)$ is $\Spin(V)_P$-invariant.
Hence, $\beta_{k+1}$  belongs to the sum of $\beta_k\cup H^2(X\times\hat{X},\QQ)$ and the $\Spin(V)_P$-invariant subspace $H^{2k+2}(X\times\hat{X},\QQ)^{\Spin(V)_P}$.
%Now, $\wedge^4V_K\cong \oplus_{i=0}^4[\wedge^iW_1\otimes \wedge^{4-i}W_1^*]$, and the 
%$\Spin(V)_w$-invariant subspace is a one-dimensional subspace of $\wedge^2W_1\otimes \wedge^2 W_1^*$ spanned by 
%$\Theta_w^2$. It follows that $\beta_{k+1}$ belongs to $\Theta_w\cup H^2(X\times\hat{X},\QQ)$. 
%The homomorphism$\Theta_w\cup(\bullet): H^2(X\times\hat{X},\QQ)\rightarrow H^4(X\times\hat{X},\QQ)$ is injective, 
%since $\Theta_w$ is ample.
Thus, there exists a uniques class $\ell\in H^2(X\times\hat{X},\QQ)_P$, such that 
$\gamma:=\beta_{k+1}+\beta_k \ell$ is $\Spin(V)_P$-invariant. Given $g\in\Spin(V)_P$, 
Equation (\ref{eq-N-g-beta}) yields the first equality below. The invariance of $\beta_k$ and $\gamma$ yields  the  second equality.
\begin{eqnarray*}
\rho_g(\beta_{k+1})&=&-c_1(N_g)\beta_k+\beta_{k+1}=\gamma-[c_1(N_g)+\ell]\beta_k
\\
\rho_g(\beta_{k+1})&=&\rho_g(\gamma-\beta_k\ell)=\gamma-\beta_k\rho_g(\ell).
\end{eqnarray*}
We conclude that 
\[
c_1(N_g)=\rho_g(\ell)-\ell,
\] 
for all $g\in \Spin(V)_P$, by the injectivity of the cup product with $\beta_k$. In particular, $\ell=0$, if and only if $c_1(N_g)=0$,
for all $g\in\Spin(V)_P$. Furthermore, $\ell$ is determined by $P$ and is independent of the choices of $F_i$ with $ch(F_i)\in P$, $i=1,2$. Indeed, if $c_1(N_g)=\rho_g(\ell')-\ell'$, for all $g\in \Spin(V)_P$, for some $\ell'\in H^2(X\times\hat{X},\QQ)_P$,
then $\ell-\ell'$ is a $\Spin(V)_P$-invariant class in $H^2(X\times\hat{X},\QQ)_P$, hence $\ell'=\ell$.

Set $\alpha:=\exp(\ell)\beta$. Then $\alpha_{k+1}=\gamma$ and $\alpha$ is $\Spin(V)_P$-invariant, since
$
\rho_g(\alpha)=\rho_g(\exp(\ell))\rho_g(\beta)=\exp(\rho_g(\ell))ch(N_g^{-1})\beta=
\exp(\rho_g(\ell))\exp(\ell-\rho_g(\ell))\beta=\exp(\ell)\beta=\alpha.
$
%*************
% Hide
%*************
\hide{
Case 1: If $\beta_1$ vanishes, then $\beta_2$
is $\Spin(V)_w$-invariant as well. It follows that $\beta_2$ is a scalar multiple of $\Theta_w^2$.

Case 2: If $\beta_1$ does not vanish, then the coset
$\beta_2+\Theta_w\cup H^2(X\times\hat{X},\QQ)$ is $\Spin(V)_w$-invariant.
Hence, $\beta_2$  belongs to the sum of $\Theta_w\cup H^2(X\times\hat{X},\QQ)$ and the $\Spin(V)_w$-invariant subspace $(\wedge^4V)^{\Spin(V)_w}$.
Now, $\wedge^4V_K\cong \oplus_{i=0}^4[\wedge^iW_1\otimes \wedge^{4-i}W_1^*]$, and the 
$\Spin(V)_w$-invariant subspace is a one-dimensional subspace of $\wedge^2W_1\otimes \wedge^2 W_1^*$ spanned by 
$\Theta_w^2$. It follows that 
$\beta_2$ belongs to $\Theta_w\cup H^2(X\times\hat{X},\QQ)$. The homomorphism
$\Theta_w\cup(\bullet): H^2(X\times\hat{X},\QQ)\rightarrow H^4(X\times\hat{X},\QQ)$ is injective, since $\Theta_w$ is ample.
Hence, there exists a uniques class $\ell\in H^{1,1}(X\times\hat{X},\QQ)$, such that 
$\beta_2=-\beta_1\cup \ell.$ Equation (\ref{eq-N-g}) yields $c_1(N_g)=\ell-\rho_g(\ell)$, for all $g\in \Spin(V)_w$.
Set $\alpha:=\exp(\ell)\cup \beta$. Then $\alpha_2$ vanishes and $\alpha$ is $\Spin(V)_w$-invariant, since
$
\rho_g(\alpha)=\rho_g(\exp(\ell))\rho_g(\beta)=\exp(\rho_g(\ell))ch(N_g)\beta=
\exp(\rho_g(\ell))\exp(\ell-\rho_g(\ell))\beta=\exp(\ell)\beta=\alpha.
$
%*************
% End Hide
%*************
}
%(\ref{lemma-item-F-i-in-M-2-0-n}) We claim that the restriction of $\phi$  to $\Sym^2(S^+)$ induces a scalar multiple of the isomorphism 
%(\ref{eq-decomposition-of-sym-2-S-plus-into-spin-12-irreducible-reps}). 
%(??? Equation  (\ref{eq-decomposition-of-sym-2-S-plus-into-spin-12-irreducible-reps}) corresponds to a $\CC^*\times\CC^*$-orbit of isomorphisms).
\end{proof}

\begin{rem}
\label{remark-k-equal-0-case}
Assume $k=0$, so that the rank $r$ of $\Phi(F_2\boxtimes F_1^\vee)$ is non-zero. 
Set $\beta_1:=c_1(\Phi(F_2\boxtimes F_1^\vee))$. The class 
\[
\kappa(\Phi(F_2\boxtimes F_1^\vee)):=\exp(-\beta_1/r)ch(\Phi(F_2\boxtimes F_1^\vee))
\] 
is then $\Spin(V)_P$-invariant, by Lemma \ref{ch-3-alpha-is-second-partial-of-J}. In this case $\ell=(t\Xi_P-\beta_1)/r$, where $t$ is the unique scalar, such that $\ell$ belongs to $H^2(X\times\hat{X},\QQ)_P$.
\end{rem}

\begin{lem}
\label{example-conjecture-holds-for-abelian-surfaces}
Proposition \ref{prop-extension-class-of-decreasing-filtration-of-spin-V-representations} holds in  case $X$ is an abelian surface. 
\end{lem}

\begin{proof}
%$2$-dimensional compact complex torus with trivial Neron-severi group. 
Choose $F_1$ and $F_2$ to be ideal sheaves of length $n$ subschemes with Chern character $w_n=(1,0,-n)$, $n\geq 1$. 
Set $E:=R\pi_{23,*}(\pi_1^*F_1^\vee\otimes \F_2)$.
The first three graded summands of the Chern character of $E$ are
\[
ch(E)=
-2n-nc_1(\P)+
\left[-\frac{n}{2}c_1(\P)^2+n^2\pi_X^*[pt_X]+\pi_{\hat{X}}^*[pt_{\hat{X}}]\right] + \cdots,
\]
by the proof of \cite[Prop. 11.2]{markman-generalized-kummers} (the definition of $\iota_F$ there is ours, given in (\ref{eq-iota-F}),  composed with the automorphism of $X\times\hat{X}$ of multiplication by $-1$ and that automorphism acts as the identity on the even cohomology, so that the result there applies here as well).
The first Chern class of the object 
$E:=R\pi_{23,*}(\pi_1^*F_1^\vee\otimes \F_2)$ in Lemma \ref{ch-3-alpha-is-second-partial-of-J} is $\beta_1=-nc_1(\P)$.
%must then be a scalar multiple of the first Chern class of the Poincar\'{e} line bundle $\P$, 
%as the latter generates the cyclic Neron-Severi group of $X\times \hat{X}$. 
%The same thus must be true also when $X$ is projective, as the construction of $E$ works in families.
On the other hand, the $K$-secant $P$ can be chosen to be $\span_\QQ\{w_n,h\}$, where $h\in w_n^\perp$ and $(h,h)_{S^+}<0$.  
The class $\ell$ is the projection\footnote{The $\Spin(V)_P$-invariant class $\Xi_P$ clearly depends on $P$ (under the identification of $\wedge^2V_\QQ$ and $\wedge^2S^+_\QQ$ the class $\Xi_P$ is a scalar multiple of $w_n\wedge h$, 
by \cite[(12.5)]{markman-generalized-kummers}).
So $\beta_1$ need not be a scalar multiple of $\Xi_P$.
} 
of $-\frac{c_1(\P)}{2}$ to $H^2(X\times\hat{X},\QQ)_P,$ by Remark \ref{remark-k-equal-0-case}. 
We conclude the equality
\begin{equation}
\label{eq-c-1-N-g}
c_1(N_g)=\frac{1}{2}\left[c_1(\P)-\rho_g(c_1(\P))\right],
\end{equation}
for all $g\in\Spin(V)_P$, for every negative definite rational plane $P$ containing $w_n$, for some $n\geq 1$. 
%Hence, the above equality holds for all $g\in \Spin(V)_{w_n}$, for all $n\geq 1$.
The cocycle identity (\ref{eq-1-cocycle-identity}) implies that 
Equation (\ref{eq-c-1-N-g}) holds for every element $g$ of the subgroup $\Gamma$ of $\Spin(V)$ generated by the union 
\[
\cup \{\Spin(V)_P \ : \ P \ \mbox{is a negative definite rational plane containing} \ w_n, \ \mbox{for some} \ n\geq 1\}.
\]
%Let $G$ be the subgroup of $\Spin(V)$ consisting of $g$ for which Equation (\ref{eq-c-1-N-g}) holds. Then 
Equation (\ref{eq-relating-rho-and-rho-prime}) holds for all $g\in \Gamma$. 
Passing to $\QQ$ coefficients, Equation (\ref{eq-relating-rho-and-rho-prime}) holds for a Zariski closed subgroup of $\Spin(V_\QQ)$ and in particular for $g$ in the Zariski closure of $\Gamma$.
The latter is the whole of $\Spin(V_\QQ)$ (see for example \cite[Th. 2.1]{verbitsky}).
\end{proof}

%****************
% Hide
%****************
\hide{
Let $\Sigma_w\subset \Omega_w$ be the complex $6$-dimensional subvariety consisting of complex structures $I$ on $V_\RR=H^1(X,\RR)\oplus H^1(\hat{X},\RR)$ obtained by varying the complex structure on $X$ keeping the class of the principal polarization $\Theta$ on $X$ of Hodge type. Given $t\in \Omega_w$, let $I_t$ be the complex structure on $V_\RR$ and denote by $A_t$ the abelian variety $(V_\RR/V_\ZZ,I_t)$.  
We get a natural graded isomorphism $\eta_t:H^*(A_t,\ZZ)\rightarrow \wedge^*V^*_\ZZ$. 
\begin{cor}
\label{cor-hat-J-ww-is-algebraic}
Given $g\in \Spin(V)_w$ and $t\in g(\Sigma_w)$, the class $\eta_t^{-1}(\hat{J}_{ww})$ in $H^{3,3}(A_t,\QQ)$ is algebraic.
\end{cor}

\begin{proof}
The class $\eta_{g^{-1}(t)}^{-1}(\hat{J}_{ww})$ is algebraic, by Lemma \ref{ch-3-alpha-is-second-partial-of-J}.
The element $g$ induces a Hodge isometry $\tilde{g}:=\eta_{g^{-1}(t)}\circ\eta_t^{-1}:H^1(A_t,\ZZ)\rightarrow H^1(A_{g^{-1}(t)},\ZZ)$, hence an isomorphism $G:A_t\rightarrow A_{g^{-1}(t)}$, such that 
$G_*=\tilde{g}$, and so $G^*(\eta_{g^{-1}(t)}^{-1}(\hat{J}_{ww}))=\eta_t^{-1}(\hat{J}_{ww})$. Hence, the latter is algebraic as well.
\end{proof}

Note that Corollary \ref{cor-hat-J-ww-is-algebraic} holds also if we replace $g\in \Spin(V)_w$ by $g\in \Spin(V_\QQ)_w$, as the proof goes through with $G$ an isogeny instead of an isomorphism.

\begin{question}
What is the closure of ${\displaystyle \bigcup_{g\in \Spin(V)_w}g(\Sigma_w)}$ in $\Omega_w$ in the  classical topology?
\end{question}

Set $\G_{F_1,F_2}:=\E^2_{F_1,F_2}$. We have a surjective homomorphism from the fiber of $\G_{F_1,F_2}$ at $(x,L)$ onto
$\Ext^2(F_1,\tau_x^*(F_2)\otimes L^{-1})$, by cohomology and base change,  as 
$\Ext^i(F_1,\tau_x^*(F_2)\otimes L)$ vanishes for $i>2$.
Set $w:=ch(F_1)$. Let $\Theta_w\in \wedge^2V^*$ be the $\Spin(V)_w$-invariant class in 
Corollary \ref{cor-abelian-variety-of-weil-type}. Let $\hat{J}_{ww}\in \wedge^6_+V^*$ be the $\Spin(V)_w$-invariant class in Lemma \ref{lemma-hat-J}.

\begin{conj}
\begin{enumerate}
\item
$\G_{F_1,F_2}$ is supported on a subscheme of codimension $1$ in $X\times \hat{X}$, the cohomology classes
$ch_i\left(\G_{F_1,F_2}\right)$ are $\Spin(V)_w$-invariant, for all $i$, 
and
%does not vanish and the classes  $\kappa_i\left(\G_{F_1,F_2}\right)$ are $\Spin(V)_w$-invariant.
%Furthermore, 
$ch_3\left(\G_{F_1,F_2}\right)$ is a linear combination 
\begin{equation}
\label{eq-ch-3}
ch_3\left(\G_{F_1,F_2}\right)=s\hat{J}_{ww}+t\Theta_w^3, 
\end{equation}
with $s,t\in\QQ$ and $s\neq 0$.
\item
The pair $(X\times\hat{X},\G_{F_1,F_2})$ deforms with $X\times\hat{X}$
to a pair $(A,\G)$, for every abelian $6$-fold $A$ in the period domain $\Omega_w$, given in
(\ref{eq-Omega}).
\end{enumerate}
\end{conj}
%****************
% End Hide
%****************
}

Let $\Spin(V_K)_{\ell_1,\ell_2}$ be the group appearing in Lemma \ref{lemma-Spin-V-P-invariant-classes-are-Hodge}.

\begin{lem}
\label{lemma-Spin-V-ell-1-ell-2-invariance-conditions}
Let $E$ be an object of $D^b(X\times \hat{X})$ of non-zero rank $r$ and let $G$ be a subgroup of $\Spin(V)$. 
\begin{enumerate}
\item
\label{lemma-item-G-invariance-conditions}
The class $ch(E)$ is $G$-$\rho'$-invariant, if and only if both $\kappa(E)$ and $c_1(E)-\frac{r}{2}c_1(\P)$ are $G$-$\rho$-invariant.  
\item
\label{lemma-item-Spin-V-ell-1-ell-2-invariance-conditions} 
Assume that $ch(E)$ is $\Spin(V)_P$-$\rho'$-invariant. Then $ch(E)$ is $\Spin(V_K)_{\ell_1,\ell_2}$-$\rho'$-invariant, if and only if $\kappa(E)$ is 
$\Spin(V_K)_{\ell_1,\ell_2}$-$\rho$-invariant.
\end{enumerate}
\end{lem}

\begin{proof}
(\ref{lemma-item-G-invariance-conditions}) Proposition \ref{prop-extension-class-of-decreasing-filtration-of-spin-V-representations} implies 
that for all $g\in \Spin(V)$, $ch(E)$ is $\rho'_g$-invariant, if and only if $ch(E)\cup\exp(-\frac{1}{2}c_1(\P))=r+[c_1(E)-\frac{r}{2}c_1(\P)]+\cdots$ is $\rho_g$-invariant. The latter is $\rho_g$-invariant, if and only if 
$c_1(E)-\frac{r}{2}c_1(\P)$ and $\kappa(E)$ are both $\rho_g$-invariant.

(\ref{lemma-item-Spin-V-ell-1-ell-2-invariance-conditions}) 
Lemma \ref{lemma-Spin-V-P-invariant-classes-are-Hodge} implies that the $\Spin(V)_P$-$\rho$-invariant  subspace $H^2(X\times\hat{X},\ZZ)^{\Spin(V)_P}$
is also $\Spin(V_K)_{\ell_1,\ell_2}$-$\rho$-invariant. Hence, the statement follows from part (\ref{lemma-item-G-invariance-conditions}).
\end{proof}

%****************************************************************
% 
%****************************************************************
\subsection{Orlov's equivalence induces Chevalley's isomorphism $S\otimes S\cong \wedge^*V$}
\label{sec-Orlov-equal-Chevalley}
\label{sec-categorification}
Let $\beta:=\{e_1, \dots, e_{2n}\}$ be a basis of $H^1(X,\ZZ)$ and let $\{f_1, \dots, f_{2n}\}$ be the dual basis of $H^1(\hat{X},\ZZ)$.
Define $\mu:X\times X\rightarrow X\times X$ by $\mu(x,y)=(x+y,y)$. 
Then
\begin{eqnarray*}
\mu^*(\pi_1^*(e_i))&=&\pi_1^*(e_i)+\pi_2^*(e_i),
\\
\mu^*(\pi_2^*(e_i))&=&\pi_2^*(e_i).
\end{eqnarray*}
Let $\nu:H^*(X\times X,\ZZ)\rightarrow H^*(\hat{X}\times X,\ZZ)$ be the isomorphism induced by the equivalence
$(\Psi_{\P^{-1}[n]}\otimes 1)\circ\mu^*:D^b(X\times X)\rightarrow D^b(\hat{X}\times X)$.
Let $\tilde{\varphi}:H^*(X\times X,\ZZ)\rightarrow H^*(\hat{X}\times X,\ZZ)$ be the isomorphism (\ref{eq-tilde-varphi}).
\begin{lem}
\label{lemma-nu-equal-tilde-varphi}
The equality $\nu\circ(id\otimes \tau)=\tilde{\varphi}$ holds.
\end{lem}

\begin{proof}
Given a subset $K:=\{i_1,i_2,\dots, i_k\}$ of $[1,2n]:=\{1, 2, \dots, 2n\}$, ordered by the induced ordering of $[1,2n]$ so that  $i_t<i_{t+1}$,
set $e_K:=e_{i_1}\wedge e_{i_2}\wedge\dots\wedge e_{i_k}$ and $f_K:=f_{i_1}\wedge f_{i_2}\wedge\dots\wedge f_{i_k}$. 
%We will always order the subset in increasing order 
Let $L=\{j_1, \dots, j_\ell\}$ be another subset, with $j_t<j_{t+1}$, and define $e_L$ similarly.
Given a subset $I\subset K$, we denote by $I'$ its complement in $K$ and by $I^c$ its complement in $[1,2n]$.
Define $\epsilon_{K,L}\in\{-1,0,1\}$ by the equality $e_K\wedge e_L=\epsilon_{K, L}e_{K\cup L}$. So $\epsilon_{K,L}=0$, if $K\cap L\neq\emptyset$. Note that $\epsilon_{K,K^c}=(-1)^{\sum(K)-k(k+1)/2}$, where $\sum(K):=\sum_{t=1}^k i_t$.
Furthermore, $\epsilon_{K^c,K}=(-1)^k\epsilon_{K,K^c}=(-1)^{\sum(K)-k(k-1)/2}$.
We have $e_K\wedge\epsilon_{K,K^c}e_{K^c}=(\epsilon_{K,K^c})^2e_{K\cup K^c}=e_1\wedge\cdots \wedge e_{2n}$. So Poincar\'{e} duality $PD_X$ sends $e_K$ to $\int_Xe_K\wedge(\bullet)=\epsilon_{K,K^c}f_{K^c}$. 
In this notation we have
\begin{eqnarray*}
\mu^*(\pi_1^*e_K\wedge \pi_2^*e_L)&=&
[\pi_1^*e_{i_1}+\pi_2^*e_{i_1}]\wedge\cdots\wedge
[\pi_1^*e_{i_k}+\pi_2^*e_{i_k}]\wedge\pi_2^*e_L
\\
&=&\sum_{I\subset K}\epsilon_{I,I'}\pi_1^*e_I\wedge\pi_2^*(e_{I'}\wedge e_L)
\\
&=&\sum_{I\subset K}\epsilon_{I,I'}\epsilon_{I',L}\pi_1^*e_I\wedge\pi_2^*(e_{I'\cup L}).
\end{eqnarray*}
The cohomological action of the functor $\Psi_{\P^{-1}[n]}:D^b(X)\rightarrow D^b(\hat{X})$ is given by 
\[
\Psi_{\P^{-1}[n]}(e_K)=\sigma_K PD(e_K)=\sigma_K\epsilon_{K,K^c}f_{K^c},
\]
where\footnote{
$\phi_\P:H^k(\hat{X},\ZZ)\rightarrow H^{2n-k}(X,\ZZ)$ is equal to $(-1)^{k(k+1)/2+n} PD_k$, where $PD_k:H^k(\hat{X},\ZZ)\rightarrow H^{2n-k}(X,\ZZ)$ is the Poincar\'{e} duality isomorphism \cite[Lemma 9.23]{huybrechts-derived-categories-book}. Furthermore, the composition
\[
H^k(\hat{X},\ZZ)\RightArrowOf{\phi_\P}H^{2n-k}(X,\ZZ)\RightArrowOf{\phi_\P} H^k(\hat{X},\ZZ)
\]
is $(-1)^{k+n}$ \cite[Cor. 9.24]{huybrechts-derived-categories-book}. Now, $\Psi_{\P^{-1}}[n]$ is the inverse of $\Phi_\P$ and so the the cohomological action of $\Psi_{\P^{-1}}[n]$
restricts to $H^k(X,\ZZ)$ as $(-1)^{k+n}\phi_\P=(-1)^{k(k+3)/2}PD_k$.
} 
$\sigma_K=(-1)^{k(k+3)/2}$. 
We get
\[
\nu(\pi_1^*e_K\wedge \pi_2^*e_L)=
\sum_{I\subset K}\epsilon_{I,I'}\epsilon_{I',L}\sigma_I
\epsilon_{I,I^c}
%\epsilon_{(I'\cup L)^c,I'\cup L}
\pi_{\hat{X}}^*f_{I^c}
\wedge\pi_{X}^*(e_{I'\cup L}).
\]
Interchanging $I$ and $I'$, with respect to which the sum is symmetric, we get
\begin{eqnarray}
\label{eq-phi-of-e-K-e-L}
\nu(\pi_1^*e_K\wedge \pi_2^*e_L)&=&
\sum_{I\subset K}\epsilon_{I',I}\epsilon_{I,L}\sigma_{I'}
\epsilon_{I',(I')^c}
\pi_{\hat{X}}^*f_{(I')^c}
\wedge\pi_X^*(e_{I\cup L})
\\
\nonumber
&=& \sum_{I\subset K}\epsilon_{I',I}\epsilon_{I,L}
(-1)^{\sum(I')-|I'|}
\pi_{\hat{X}}^*f_{(I')^c}
\wedge\pi_X^*(e_{I\cup L}).
\end{eqnarray}
%Note that $PD_{X\times\hat{X}}$ restricts to $H^k(X)\otimes H^\ell(\hat{X})$ as $(-1)^{k\ell}PD_{X}\otimes PD_{\hat{X}}$.
% Hence,
%\[
%PD_{X\times \hat{X}}(\nu(\pi_1^*e_{K}\wedge \pi_2^*e_{L}))=
%\sum_{I\subset K}\epsilon_{I',I}\epsilon_{I,L}\sigma_{I\cup L}
%\epsilon_{I',(I')^c}
%(-1)^{|I\cup L|}(-1)^{|I'|(|I|+\ell)}
%\pi_{\hat{X}}^*f_{(I')^c}
%\wedge\pi_X^*(e_{I\cup L}).
%\]

Consider next the isomorphism $\tilde{\varphi}:S\otimes_\ZZ S\rightarrow \wedge^*V$ of section 
\ref{subsection-the-isomorphism-tilde-varphi}. We chose 
 the bilinear pairing $B_0(\bullet,\bullet)$ on $V$ with respect to which $H^1(X,\ZZ)$ and $H^1(\hat{X},\ZZ)$ are isotropic, and satisfing $B_0(e_i,f_j)=\delta_{i,j}$ and $B_0(f_i,e_j)=0$, for all $1\leq i,j\leq 2n$.
Let $\psi':C(V)\rightarrow \End(\wedge^*V)$ send $v$ to $L_v+\delta_{B_0(v,\bullet)}$, where $\delta_x$, $x\in V^*$, is contraction with $x$. 
We denote $\delta_{B_0(e_i,\bullet)}$ by $\delta_{e_i}$ for short using the equality $B_0(e_i,\bullet)=(e_i,\bullet)_V$ and the identifiation of $V$ with $V^*$ via $(\bullet,\bullet)_V$. 
So $\psi'(e_i)=
L_{e_i}+\delta_{e_i}$ and $\psi'(f_i)=L_{f_i}$. We set $\delta_K:=\delta_{e_{i_1}}\delta_{e_{i_2}}\cdots\delta_{e_{i_k}}$.
Let $\psi:C(V)\rightarrow \wedge^*V$ be $\psi'$ composed with evaluation at the unit $\one$.
Set $\tilde{\varphi}(s\otimes t)=\psi(s[pt_{\hat{X}}]\tau(t))$. Then
\begin{eqnarray*}
\tilde{\varphi}(e_K\otimes e_L)&=&
[L_{e_{i_1}}+\delta_{e_{i_1}}] \circ \cdots \circ [L_{e_{i_k}}+\delta_{e_{i_k}}] f_1\wedge \cdots \wedge f_{2n}\wedge \tau(e_L)
\\
&=& \sum_{I\subset K} \epsilon_{I,I'}e_I\wedge\delta_{I'}( f_1\wedge \cdots \wedge f_{2n})\wedge (-1)^{\ell(\ell-1)/2}e_L
\\
&=&
\sum_{I\subset K} \epsilon_{I,I'}(-1)^{\ell(\ell-1)/2} (-1)^{|I|(2n-|I'|)}\delta_{I'}(f_1\wedge\cdots\wedge f_{2n})\wedge e_I\wedge e_L
\\
&=&
\sum_{I\subset K} \epsilon_{I,I'}(-1)^{\ell(\ell-1)/2} (-1)^{|I||I'|}(-1)^{\sum(I')-|I'|}\epsilon_{I,L}f_{(I')^c}\wedge e_{I\cup L}
\\
&=&
\sum_{I\subset K} \epsilon_{I',I}(-1)^{\ell(\ell-1)/2} (-1)^{\sum(I')-|I'|}\epsilon_{I,L}f_{(I')^c}\wedge e_{I\cup L},
\end{eqnarray*}
where in the fourth equality we used the equality $\delta_{I'}(f_1\wedge\cdots\wedge f_{2n})=(-1)^{\sum(I')-|I'|}f_{(I')^c}$.
Comparing with  (\ref{eq-phi-of-e-K-e-L}) we get
\begin{equation}
\label{eq-orlov-isomorphism-categorifies-chevalley}
\nu(\pi_1^*\left(e_K)\wedge\pi_2^*\tau(e_L)\right)=\tilde{\varphi}(e_K\otimes e_L).
\end{equation}
This completes the proof of Lemma \ref{lemma-nu-equal-tilde-varphi}.
\end{proof}

\begin{proof}[Proof of 
Lemma \ref{lemma-orlov-isomorphism-is-chevalley}]
We have 
\[
\phi:=(id\otimes \psi_{\P^{-1}[n]})\circ \mu^*=(\phi_\P\otimes\psi_{\P^{-1}[n]})\circ \nu=(\phi_\P\otimes\psi_{\P^{-1}[n]})\circ \tilde{\varphi}\circ (id\otimes\tau),
\]
where the first equality is the definition of $\phi$, the second is the definition of $\nu$, and the third is Lemma \ref{lemma-nu-equal-tilde-varphi}.
\end{proof}
%************
% Hide
%************
\hide{
The proof  follows from Lemma \ref{lemma-nu-equal-tilde-varphi}, since
Orlov's equivalence $\Phi:D^b(X\times X)\rightarrow D^b(X\times\hat{X})$, 
given in (\ref{eq-Orlov-derived-equivalence-from-XxX-to-X-times-hat-X}), is the composition 
$(\Phi_\P\boxtimes \Psi_{\P^{-1}[n]})\circ(\Psi_{\P^{-1}[n]}\boxtimes 1)\circ\mu^*$.
The cohomological action 
$\phi_\P\otimes\psi_{\P^{-1}[n]}$ of $\Phi_\P\boxtimes \Psi_{\P^{-1}[n]}$
is degree reversing, sending $H^k(\hat{X}\times X,\ZZ)$ to $H^{4n-k}(X\times\hat{X},\ZZ)$. 
Furthermore, $\phi_\P\otimes\psi_{\P^{-1}[n]}$ is $\Spin(V)$-equivariant (with respect to the graded representation $\rho$)
by Lemma \ref{lemma-phi-P-psi-P-inverse-is-PD-up-to-sign}.
The isomorphism $\nu$ preserves the increasing filtration $F^k(\wedge^*V):=\oplus_{i\leq k}\wedge^iV$, 
by Equation (\ref{eq-orlov-isomorphism-categorifies-chevalley}) and the fact that $\tilde{\varphi}$ does as shown in \cite[Sec. 3.3 page 85]{chevalley}. Furthermore, given $g\in \Spin(V)$, the automorphism $\nu(g\otimes g)\nu^{-1}$
of $H^*(\hat{X}\times X,\ZZ)$ acts on the graded summands $\wedge^kV$ of the filtration $F^k(\wedge^*V)$ via the usual action of $g\in \Spin(V)$, for the same reason. 
Hence, the cohomological action 
$\phi=(\phi_\P\otimes\psi_{\P^{-1}[n]})\circ \nu$ of $\Phi$ preserves the decreasing filtration 
(\ref{eq-decreasing-weight-filtration})
and the automorphism $\phi(g\otimes g)\phi^{-1}$ acts on the graded summands $\wedge^kV$ via the usual action of $g\in \Spin(V)$ (compare with \cite[Prop. 4.3.7 and Cor. 4.3.8]{golyshev-luntz-orlov}).
\end{proof}
%************
% Hide
%************
}

\begin{lem}
\label{lemma-phi-P-psi-P-inverse-is-PD-up-to-sign}
$(\phi_\P\otimes \psi_{\P^{-1}}):H^d(\hat{X}\times X)\rightarrow H^{4n-d}(X\times\hat{X})$ is equal to $(-1)^{d(d+1)/2}PD_{\hat{X}\times X}$.
\end{lem}

\begin{proof} Keep the notation of the proof of Lemma \ref{lemma-nu-equal-tilde-varphi}.
In particular, $f_L=f_{j_1}\wedge \cdots\wedge f_{j_\ell}$ and
$e_K=e_{i_1}\wedge\cdots\wedge e_{i_k}$ are classes of degrees $\ell$ and $k$ respectively.
We have the equality $\phi_\P^{-1}=\psi_{\P^{-1}}$. By definition,
$(PD_{\hat{X}\times X}(f_L\wedge e_K),\bullet)=\int_{\hat{X}\times X}(f_L\wedge e_K)\wedge\bullet$.
\begin{eqnarray*}
(\phi_\P\otimes\phi_\P^{-1})(f_L\wedge e_K) & = & (-1)^{\ell(\ell+1)/2+n}(-1)^{(2n-k)(2n-k+1)/2+n}PD(f_L)\wedge PD^{-1}(e_K)
\\
&=& (-1)^{[k(k+1)+\ell(\ell+1)]/2}(-1)^k\epsilon_{L,L^c}\epsilon_{K^c,K}e_{L^c}\wedge f_{K^c}
\\
&=&
(-1)^{[k(k+1)+\ell(\ell+1)]/2}\epsilon_{L,L^c}\epsilon_{K,K^c}e_{L^c}\wedge f_{K^c}
\end{eqnarray*}
%So $(e_L\wedge f_K)\wedge \left[(\phi_\P\otimes\phi_\P^{-1})(f_L\wedge e_K) \right]
%=(-1)^{(k+\ell)(k+\ell+1)/2}e_1\wedge\cdots\wedge e_{2n}\wedge f_1\wedge\cdots\wedge f_{2n}.$
%On the other hand, $((e_L\wedge f_K),(f_L\wedge e_K))=(e_L,f_L)(f_K,e_K)=1.$
So 
\begin{eqnarray*}
((\phi_\P\otimes\phi_\P^{-1})(f_L\wedge e_K) ,f_{L^c}\wedge e_{K^c})&=&(-1)^{[k(k+1)+\ell(\ell+1)]/2}\epsilon_{L,L^c}\epsilon_{K,K^c}
\\
&=& (-1)^{(k+\ell)(k+\ell+1)/2}(-1)^{kl}\epsilon_{L,L^c}\epsilon_{K,K^c}.
\end{eqnarray*}
On the other hand, 
$\int_{\hat{X}\times X}(f_L\wedge e_K)\wedge (f_{L^c}\wedge e_{K^c})=(-1)^{kl}\int_{\hat{X}\times X}f_L\wedge f_{L^c}\wedge e_K\wedge e_{K^c}=
(-1)^{kl}\epsilon_{L,L^c}\epsilon_{K,K^c}$.
\end{proof}

\begin{rem}
Let $\varsigma:X\times \hat{X}\rightarrow \hat{X}\times X$ be the transposition of the factors. The isomorphism $\varsigma_*:V=H^1(X\times\hat{X},\ZZ)\rightarrow H^1(\hat{X}\times X,\ZZ)=V^*$ is equal to the one induced by the pairing $(\bullet,\bullet)_V$. Hence, the composition $(\phi_\P\otimes \psi_{\P^{-1}})\circ \varsigma_*:\wedge^*V\rightarrow\wedge^*V$ 
is an analogue of the Hodge $*$ operator \cite[Sec. 1.2]{huybrechts-complex-geometry-book}.
\end{rem}

%***************************************************************************
% 
%***************************************************************************
\subsection{Hodge-Weil classes on  $X\times\hat{X}$ from tensor squares of even pure spinors}
\label{sec-HW-classes-from-squares-of-pure-spinors}

Let $P\subset S^+_\QQ$ be a rational plane satisfying Assumption \ref{assumption-on-rational-secant-plane-P}. We keep the notation of Section \ref{sec-pure-spinors}. In particular $\ell_i\in\PP(P)$, $i=1,2$, are the two complex conjugate pure spinors, $\tilde{\ell}_i\subset P_K$ is the  $1$-dimensional subspace corresponding to $\ell_i$, and $W_i\subset V_K$ is the corresponding maximal isotropic subspace.
The subspace $P\otimes P$ of $S^+_\QQ\otimes S^+_\QQ$ is a trivial $\Spin(V_\QQ)_P$ sub-representation, but it decomposes as a direct sum of one-dimensional $\Spin(V_K)_{\ell_1,\ell_2}$-representations, the two non-trivial distinct complex conjugate characters $\tilde{\ell}_1^{\otimes 2}$, $\tilde{\ell}_2^{\otimes 2}$, and the two trivial characters $\tilde{\ell}_1\wedge\tilde{\ell}_2$ in $\wedge^2 S^+_\QQ$ and $\tilde{\ell}_1\tilde{\ell}_2$ in $\Sym^2(S^+_\QQ)$. Note that the $2$-dimensional subspace 
\[
HW_{P_K}:=\tilde{\ell}_1^{\otimes 2}\oplus \tilde{\ell}_2^{\otimes 2}
\] 
of $S^+_K\otimes_K S^+_K$ is defined over $\QQ$. 
Denote by $HW_P$ the corresponding $2$-dimensional subspace of $S^+_\QQ\otimes S^+_\QQ$.

Let $\rho,\rho':\Spin(V)\rightarrow GL(\wedge^* V)\cong GL(H^*(X\times\hat{X},\ZZ))$ be the two representations given in (\ref{eq-rho-extended-to-exterior-algebra}) and (\ref{eq-rho-prime}). Both factor though the image $SO^+(V)$ of $\Spin(V)$. 
Recall that $\rho$ is the graded extension of the natural homomorphism $\rho:\Spin(V)\rightarrow SO^+(V)$. 
Given $g\in \Spin(V),$ the automorphism $\rho'(g)$ preserves the decreasing filtration 
\[
F^kH^*(X\times\hat{X},\ZZ):=\oplus_{i\geq k} H^i(X\times\hat{X},\ZZ),
\] 
$0\leq k\leq 4n$. The induced action of $\rho'(g)$ on the graded summands agrees with that of $\rho(g)$, by Proposition \ref{prop-extension-class-of-decreasing-filtration-of-spin-V-representations}.

%***************
% Hide
%***************
\hide{
Let $\tilde{\gamma}:DMon(X)\rightarrow GL(H^*(X\times\hat{X},\ZZ))$ be the action of the derived monodromy group of $X$ as in the proof of Lemma \ref{ch-3-alpha-is-second-partial-of-J}. Given $g\in DMon(X),$ the automorphism $\tilde{\gamma}(g)$ preserves the decreasing filtration $F^kH^*(X\times\hat{X},\ZZ):=\oplus_{i\geq k} H^i(X\times\hat{X},\ZZ)$, $0\leq k\leq 4n$.
We get an associated graded automorphism $\gamma(g)\in GL(H^*(X\times\hat{X},\ZZ))$ and the representation
\[
\gamma:DMon(X)\rightarrow GL(H^*(X\times\hat{X},\ZZ)).
\]
Denote by $\gamma_i:DMon(X)\rightarrow GL(H^i(X\times\hat{X},\ZZ))$ the restriction of the representation $\gamma$ to the $i$-th graded summand. Then $\gamma_1$ maps $DMon(X)/\{\pm 1\}$ isomorphically onto the image $SO^+(V)$ of $\Spin(V)$ in the orthogonal group of $V$ and $\gamma_i$ is $\wedge^i\gamma_1$. We denote by 
\[
\rho:\Spin(V)\rightarrow GL(\wedge^* V)\cong GL(H^*(X\times\hat{X},\ZZ))
\]
the graded extension of the natural homomorphism $\rho:\Spin(V)\rightarrow SO^+(V)$. Then $\gamma$ conjugates to $\rho$ via the isomorphism $DMon(X)\cong \Spin(V)$. The latter isomorphism is due to the fact that each is isomorphic to its image in $GL(S)$, $S=H^*(X,\ZZ)$, and the two images coincide.
Let
\[
\rho':\Spin(V)\rightarrow  GL(H^*(X\times\hat{X},\ZZ))
\]
be the conjugate of $\tilde{\gamma}$ via the above mentioned isomorphism $DMon(X)\cong \Spin(V)$.
The isomorphism $\phi:S\otimes S\rightarrow \wedge^*V,$ given in (\ref{eq-Orlov-cohomological-isomorphism}),  conjugates the $\Spin(V)$ action on $H^*(X\times X,\ZZ)$ to the $\rho'$ action. 
%***************
% End Hide
%***************
}

Let $U\subset H^*(X\times \hat{X},K)$ be an irreducible representation of the restriction of the $\rho'$ action of $\Spin(V_K)$ to a subgroup $G\subset \Spin(V_K)$.
Let $k(U)$ be the maximal integer, such that $U$ is contained in $F^kH^*(X\times \hat{X},K)$. We call $k(U)$ the {\em weight} of $U$. The intersection  of $U$ with the subspace $F^{k(U)+1}H^*(X\times \hat{X},K)$ vanishes, as it  is a proper subrepresentation of the irreducible representation $U$. Hence, $U$ projects injectively 
and $G$-equivariantly into an irreducible representation $\hat{U}\subset H^{k(U)}(X\times \hat{X},K)$ of $G$ via the $\rho$ action of $\Spin(V_K)$. If $U$ is reducible, but each irreducible $G$-subrepresentation of $U$ has the same weight $k_0$, we set $k(U)=k_0$ and let $\hat{U}$ be the projection of $U$ in $H^{k(U)}(X\times \hat{X},K)$. Again we call $k(U)$ the {\em weight} of $U$. If, furthermore, $U$ decomposes as a direct sum of pairwise distinct irreducible subrepresentations of the same weight $k(U)$,  then again the projection homomorphism $U\rightarrow \hat{U}$ is an isomorphism of $G$ representations.
Set $\phi':=\phi\circ(id\otimes\tau):S\otimes S\rightarrow \wedge^*V$, where $\phi$ given in (\ref{eq-Orlov-cohomological-isomorphism}). By definition (\ref{rho-prime-g}),
$\rho'_g=\phi'(m_g\otimes m_g)\phi'^{-1}$.

\begin{prop}
\label{prop-the-orlov-image-of-HW-P-projects-into-the-3-dimensional-space-of-HW-classes}
\begin{enumerate}
\item
\label{prop-item-HW-P-maps-to-a-weight-2n-subrepresentation}
The isomorphism $\phi'$ maps $HW_P:=\tilde{\ell}_1^{\otimes 2}\oplus \tilde{\ell}_2^{\otimes 2}$
%each of $\tilde{\ell}_1^{\otimes 2}$ and $\tilde{\ell}_2^{\otimes 2}$
into a weight $2n$ $\Spin(V)_P$-subrepresentation of $H^*(X\times\hat{X},\QQ)$ via $\rho'$ and the projection 
$\hat{HW}_P$ of $\phi'(HW_P)$ is the rational subspace of $H^{2n}(X\times\hat{X},\QQ)$ corresponding to the subspace
$\wedge^{2n}W_1\oplus\wedge^{2n}W_2$ of $H^{2n}(X\times\hat{X},K)$.
%The corresponding $2$-dimensional subspace
%$\hat{HW}_{P_K}\subset H^{2n}(X\times\hat{X},K)$ is the direct sum of two distinct  
%complex conjugate $1$-dimensional subrepresentations of $\Spin(V_K)_{\ell_1,\ell_2}$ isomorphic to the 
%characters $\det_i$, $i=1,2$.
\item
\label{prop-item-weight-2-subrepresentation}
The weights of $\phi'(\ell_1\wedge\ell_2)$ and $\phi'(\ell_1\cdot\ell_2)$, as $\Spin(V)_P$-subrepresentations via $\rho'$, are as follows.
If $n$ is even, then the weight of $\phi'(\ell_1\wedge\ell_2)$ is $2$  
and the weight of $\phi'(\ell_1\cdot\ell_2)$ is $0$.
If $n$ is odd, then the weight of $\phi'(\ell_1\cdot\ell_2)$ is $2$ 
and the weight of $\phi'(\ell_1\wedge\ell_2)$ is $0$.
\end{enumerate}
\end{prop}

\begin{proof} 
(\ref{prop-item-HW-P-maps-to-a-weight-2n-subrepresentation})
The homomorphism $\phi'$ is defined over $\QQ$. Thus, the two distinct subrepresentations $\phi'(\tilde{\ell}_i^2)$, $i=1,2$, are complex conjugates and hence have the same weight $k:=k(\tilde{\ell}_i^2)$. The weight is even, as $\phi'$ is $\ZZ/2\ZZ$ graded and $P\otimes P$ is contained in $H^{even}(X\times X,\QQ)$. 
It follows that $\hat{HW}_{P_K}:=\widehat{\phi'(\tilde{\ell}_1^2)}\oplus \widehat{\phi'(\tilde{\ell}_1^2)}$
is a $2$-dimensional subrepresentation of $H^k(X\times\hat{X},K)$, which is defined over $\QQ$.

Lemma \ref{lemma-Spin-V-P-invariant-classes-are-Hodge} shows that the $\Spin(V)_P$-invariant subspace 
$H^{2i}(X\times \hat{X},\QQ)^{\Spin(V)_P}$ of $H^{2i}(X\times \hat{X},\QQ)$
is one-dimensional and it is a trivial $\Spin(V_K)_{\ell_1,\ell_2}$ character, for $i\neq n$ and $0\leq i\leq 4n$. 
The space $H^{2n}(X\times \hat{X},\QQ)^{\Spin(V)_P}$ is three-dimensional and it is the direct sum of the characters $\det_1$, $\det_2$ and the trivial character of $\Spin(V_K)_{\ell_1,\ell_2}$. Hence, the weight of 
$\phi'(\tilde{\ell}_1^2)\oplus\phi'(\tilde{\ell}_1^2)$ is $2n$. 
We have already seen in the proof of Lemma \ref{lemma-decomposition-into-4-direct-summands} that $\tilde{\ell}_i^2$ is the character $\wedge^{2n}W_i\cong\det_i$ of $\Spin(V_K)_{\ell_1,\ell_2}$.
It follows that $\widehat{\phi'(\tilde{\ell}_i^2)}$ is equal to $\wedge^{2n}W_i$.

(\ref{prop-item-weight-2-subrepresentation}) 
%********
% Hide
%********
\hide{
Let $\wedge^{2n}_+V_\QQ$ be the subspace of $\wedge^{2n}V_\QQ$ spanned by the top exterior powers of even maximan isotropic subspaces and let $\wedge^{2n}_-V_\QQ$ be the odd analogue.
We have the isomorphisms of $\Spin(V_\QQ)$ representations
\begin{eqnarray*}
S^+_\QQ\wedge S^+_\QQ&\cong&
\oplus_{h=0, \ h\equiv 2n+2 ({\rm mod}\ 4)}^{2n-1} \wedge^hV_\QQ
\\
\Sym^2(S^+_\QQ)
&\cong& \wedge^{2n}_+V_\QQ \oplus
\oplus_{h=0, \ h\equiv 2n ({\rm mod}\ 4)}^{2n-1} \wedge^hV_\QQ
\end{eqnarray*}
(see \cite[Sec. 3.4 page 96]{chevalley}). The weight of every irreducible $\Spin(V)_P$ subrepresentation of $S^+_\QQ\wedge S^+_\QQ$ is thus congruent to $2$ modulo $4$, if $n$ is even, by the first isomorphism above. The analogous statement for odd $n$ follows from the second isomorphism above.
%********
% End Hide
%********
}
The group $\Spin(V_K)$ acts transitively on the set of ordered pairs of 
complementary maximally isotropic subspaces of $V_K$, by \cite[Sec. 3.3, Lemma 1]{chevalley}.
The weight is invariant under the $\Spin(V_K)$-action. 
Hence, it suffices to calculate the weights of the one dimensional representations spanned by  
of $\phi'(1\wedge [pt_X])$ and $\phi'(1\cdot [pt_X])$. We related the isomorphism $\tilde{\varphi}$ given in (\ref{eq-tilde-varphi}) to the isomorphism $\phi$ 
given in (\ref{eq-Orlov-cohomological-isomorphism}) in Lemma \ref{lemma-orlov-isomorphism-is-chevalley}.
%section \ref{sec-categorification}.
We have $\phi'=(\phi_\P\otimes\psi_{\P^{-1}[n]})\circ \tilde{\varphi}$. The image 
of $([pt_X]\otimes1)-(-1)^n (1\otimes[pt_X])$
via $\tilde{\varphi}$ is contained in $F^{4n-2}(\wedge^*V)$ and not in
$F^{4n-3}(\wedge^*V)$, by
Lemma \ref{lemma-symmetric-or-alternating-product-of-two-pure-spinors-has-weight-2}(\ref{lemma-item-weight-4n-2}). Now $(\phi_\P\otimes\psi_{\P^{-1}[n]})$
sends $F^k(\wedge^*V)$ to $F_{4n-k}(\wedge^*V)$, by Lemma \ref{lemma-phi-P-psi-P-inverse-is-PD-up-to-sign}, and so the weight is $2$.
Similarly, the weight of the other line is zero, by Lemma \ref{lemma-symmetric-or-alternating-product-of-two-pure-spinors-has-weight-2}(\ref{lemma-item-weight-4n}), via the same argument.
\end{proof}

%****************************************************************
% 
%****************************************************************
\section{Semiregular twisted sheaves}
\label{section-semiregular-twisted-sheaves}
In Section \ref{sec-semiregular-sheaves} we recall the definition of the semi-regularity map of a coherent sheaf $F$ in terms of its Atiyah class $at_F$.
A sheaf is semiregular, when its semi-regularity map is injective.
In Section \ref{sec-semiregular-projective-bundles} we extend the definition of the semi-regularity map for projective bundles. In Section \ref{sec-semiregular-twisted-sheaves}
we extend the definition for coherent sheaves twisted by a \v{C}ech cocycle with coefficients in the local system $\mu_r$ of $r$-th roots of unity.
We state the analogue of the Buchweitz-Flenner Semiregularity Theorem for twisted sheaves as Conjecture \ref{conjecture-semiregular-twisted-sheaves-deform}.
It states that a semiregular twisted sheaf on a special fiber of a family deforms to a twisted sheaf over an open neighborhood of the fiber, provided its Chern character remains of Hodge type.
An analogue of the Semiregularity Theorem for twisted sheaves follows from a very general theorem of Pridham \cite[Remark 2.6]{pridham}. The author's ignorance prevents him from comparing the two statements, though Conjecture \ref{conjecture-semiregular-twisted-sheaves-deform} should follow from Pridham's result. In Section \ref{sec-the-semiregularity-theorem-for-twisted-sheaves-on-abelian-varieties} we reduce Conjecture \ref{conjecture-semiregular-twisted-sheaves-deform} to the Semiregularity Theorem for untwisted sheaves in the case of families of abelian varieties (and more generally, whenever the Brauer class of the twisted sheaf on the special fiber is the restriction of that of a projective bundle over the family).
%****************************************************************
% 
%****************************************************************
\subsection{Semiregular coherent sheaves}
\label{sec-semiregular-sheaves}
Let $E$ be a coherent sheaf on a $d$-dimensional complex manifold $M$. Denote the Atiyah class of $E$ by  $at_E\in \Ext^1(E,E\otimes\Omega^1_M)$. 
The $q$-th component $\sigma_q$ of the semiregularity map 
\begin{equation}
\label{eq-semiregularity-map}
\sigma:=(\sigma_0, \dots,\sigma_{d-2}):\Ext^2(E,E)\rightarrow \prod_{q= 0}^{d-2}H^{q+2}(M,\Omega^q_M)
\end{equation}
is the composition
\[
\Ext^2(E,E)\LongRightArrowOf{(at_E)^q/q!}\Ext^{q+2}(E,E\otimes\Omega^q_M)\LongRightArrowOf{Tr}H^{q+2}(M,\Omega^q_M).
\]
The sheaf $E$ is said to be {\em semiregular}, if $\sigma$ is injective.
We have the commutative diagram
\begin{equation}
\label{eq-diagram-semi-regularity}
\xymatrix{
H^1(M,TM) \ar[rr]^{at_E} \ar[dr]_{\Contract ch(E)}&&\Ext^2(E,E)\ar[dl]^{\sigma}
\\
&
\prod_{q= 0}^{d-2}H^{q+2}(M,\Omega^q_M)
}
\end{equation}
by \cite[Cor. 4.3]{buchweitz-flenner}.\footnote{We follow the convention of \cite{toda} and \cite{huybrechts-lehn} for the sign of the Atiyah class, so that $ch(E)=Tr(at_E)$, while in \cite{buchweitz-flenner} $ch(E)=Tr(-at_E)$. So for a vector bundle we choose the Atiyah class to be the extension class of the first jet bundle, rather than the extension class of the bundle of first order differential operators with scalar symbol. See \cite[Theorem 5]{atiyah}. \label{footnote-atiyah-class}}

%****************************************************************
% 
%****************************************************************
\subsection{Semiregular projective bundles}
\label{sec-semiregular-projective-bundles}
Let  $B$ be a projective $\PP^{r-1}$-bundle over $M$. 
Denote by $\A$ the Azumaya algebra of $B$. Then $\A$ is a coherent sheaf of associative algebras with a unit over $M$, locally isomorphic to the sheaf of endomorphisms of a locally free sheaf. If $B=\PP(E)$, for a locally free sheaf $E$ over $M$, then  
$\A$ is naturally isomorphic to $\SheafEnd(E)$. Let $\A_0$ be the kernel of the trace homomorphism $tr:\A\rightarrow \StructureSheaf{M}$.
We have the direct sum decomposition $\A:=\A_0\oplus \StructureSheaf{M}$, where the second direct summand is generated by the unit. 
%Set $r:=\rank(E)$. 
The sheaf $\A_0$ is isomorphic to the adjoint Lie algebra bundle of the principal $PGL(r)$-bundle $Pr(B)$ associated to $B$. Denote by $Q$ the Atiyah bundle of $Pr(B)$ \cite{atiyah}. It fits in the short exact sequence 
\[
0\rightarrow \A_0\rightarrow Q\rightarrow TM\rightarrow 0.
\]
Let $-at_B\in H^1(M,\A_0\otimes\Omega^1_M)$ be the extension class of the above sequence. 
%Then $q^*at_B=at_{\PP(E)}$, by \cite[page 189 property (2)]{atiyah}, and the latter 
If $B=\PP(E)$, then $at_B$
is the traceless direct summand of $at_E$, by \cite[page 189 property (1)]{atiyah}. Consequently, 
\begin{equation}
\label{eq-at-B-versus-at-E}
at_B=at_E-\frac{c_1(E)}{r}\cdot id_E
\end{equation} 
(keeping the convention of footnote \ref{footnote-atiyah-class}). 

We view the class $at_B$ as a class in $H^1(M,\A\otimes\Omega^1_M)$ via the inclusion of $\A_0$ as a direct summand in $\A$. We get the exponential Atiyah class $\exp(at_B)$ with graded summands $at_k(B)$ in 
$H^k(M,\A\otimes\Omega^k_M)$. Note that $at_0(B)$ is $r$ times the unit section and $at_1(B)=at_B$. Denote by $\kappa(B)$ the trace of $\exp(at_B)$. If $B=\PP(E)$, then 
\[
\kappa(B)=ch(E)\exp(-c_1(E)/r), 
\]
by equation (\ref{eq-at-B-versus-at-E}).
We get the commutative diagram
\begin{equation}
\label{commutative-diagram-of-semiregularity-map-of-projective-bundle}
\xymatrix{
H^1(M,TM) \ar[rr]^{at_B} \ar[dr]^{\Contract \kappa(B)} & & H^2(M,\A_0) \ar[dl]_{\sigma_B}
\\
& 
\prod_{q=1}^{d-2}H^{q+2}(M,\Omega_M^q),
}
\end{equation}
where the {\em semiregularity map} $\sigma_B$ is defined as in equation (\ref{eq-semiregularity-map}) with $at_E$ replaced by $at_B$. The commutativity of Diagram (\ref{commutative-diagram-of-semiregularity-map-of-projective-bundle}) is proved by the same argument as in \cite[Cor. 4.3]{buchweitz-flenner} establishing the commutativity of Diagram (\ref{eq-diagram-semi-regularity}). 
See Remark \ref{remark-commutativity-of-diagram-of-semiregularity-map-of-projective-bundle} for another proof.

%****************************************************************
% 
%****************************************************************
\subsection{Semiregular $\mu_r$-twisted sheaves}
\label{sec-semiregular-twisted-sheaves}
%Assume that $M$ is a smooth projective variety or a compact K\"{a}hler manifold. 
We refer to \cite{caldararu-thesis} for basic facts about twisted sheaves.
Let $\U:=\{U_i\}_{i\in I}$ be an open covering of $M$.
Let $\B$ be a coherent sheaf of non-zero rank $r$ over $M$ twisted by a \v{C}ech cocycle $\tilde{\theta}$ in $\Z^2(\U,\mu_r)$,
where $\mu_r$ is the local system of $r$-th roots of unity.
The $\tilde{\theta}$-twisted sheaf $\B$ has a well defined untwisted determinant line-bundle $\wedge^r\B$, since $(\tilde{\theta}_{ijk})^r=1$.
%Assume that $\wedge^r\B$ is trivial.
% in the image of $H^2(M,\mu_r)$ in $H^2(M,\StructureSheaf{M}^*)$. 

\begin{example}
\label{example-a-lift-of-a-projective-bundle-to-a-mu-r-twisted-sheaf}
Every projective bundle $B$ on $M$ admits a lift to such a $\tilde{\theta}$-twisted locally free sheaf $\B$ with trivial determinant line bundle, where $\tilde{\theta}$ is a cocycle representing  
the  characteristic class $\theta\in H^2(M,\mu_r)$ of the bundle. The characteristic class  $\theta$ is the image of the class of the bundle in $H^1(M,PGL(r,\StructureSheaf{M}))$ via the connecting homomorphism of the short exact sequence
\begin{equation}
\label{eq-short-exact-sequence-mu-r-SL-PGL}
1\rightarrow \mu_r \rightarrow SL(r,\StructureSheaf{M})\rightarrow PGL(r,\StructureSheaf{M})\rightarrow 1.
\end{equation}
Note that the group of line bundles of order $r$ in $\Pic^0(M)$ acts transitively on the set of isomorphism classes of choices of lifts $\B$. 
Note that if $\B$ happens to be untwisted, so that $B$ admits a lift to a vector bundle $\B$ with trivial determinant, then the Atiyah class of the principal $SL(r)$-bundle associated to $\B$ and the Atiyah class of the projective bundle $B$ are equal, by
\cite[page 189 property (1)]{atiyah}.
\end{example}

\begin{rem}
\label{remark-twisting-cocycle-can-be-chosen-in-mu-order-of-Brauer-class}
Keep the notation of Example \ref{example-a-lift-of-a-projective-bundle-to-a-mu-r-twisted-sheaf}.
%Assume that $H^3(M,\ZZ)$ is torsion free. Then $H^2(M,\ZZ)\rightarrow H^2(M,\ZZ/r\ZZ)$ is surjective, by the universal coefficients theorem. Consequently, 
If the characteristic class $\theta$ of the projective bundle $B$ 
%as in Example \ref{example-a-lift-of-a-projective-bundle-to-a-mu-r-twisted-sheaf} 
has order $\rho$ in $H^2(M,\mu_r)$, then $\rho$ divides $r$. Set $k=r/\rho$. The short exact sequence
$0\rightarrow \mu_\rho\rightarrow \mu_r\RightArrowOf{(\bullet)^\rho}\mu_k\rightarrow 0$ yields exactness of
\[
H^2(M,\mu_\rho)\rightarrow H^2(M,\mu_r)\RightArrowOf{(\bullet)^\rho} H^2(M,\mu_k).
\]
Hence, $\theta$ is the image of a class $\theta'\in H^2(M,\mu_\rho).$
Let $\tilde{\theta}'$ be a \v{C}ech $2$-cocycle with coefficients in $\mu_\rho$, which is cohomologous in $Z^2(\U,\mu_r)$ to the cocycle $\tilde{\theta}$ in Example 
\ref{example-a-lift-of-a-projective-bundle-to-a-mu-r-twisted-sheaf}, $\tilde{\theta}'=\tilde{\theta}\delta(\alpha)$, for a \v{C}ech $1$-co-chain $\alpha$ with coefficients in $\mu_r$. Multiplying the gluing transformations of $\B$ by the co-chain $\alpha$ we get the $\tilde{\theta}'$-twisted sheaf $\B'$, still with trivial determinant, with
$\PP(\B')\cong\PP(\B)\cong B$.
We can thus choose the \v{C}ech $2$-cocycle $\tilde{\theta}$ in Example \ref{example-a-lift-of-a-projective-bundle-to-a-mu-r-twisted-sheaf}
to have coefficients in $\mu_\rho$. 
\end{rem}

\begin{construction}
\label{construction-twisted-sheaf-with-trivial-determinant-associated-to-a-coherent-sheaf}
Let $\E$ be a  coherent sheaf of rank $r>0$ over $M$. Let $\U:=\{U_i\}_{i\in I}$ be an open covering, in the analytic topology,  with each $U_i$ biholomorphic to a polydisc and with simply connected finite intersections. Denote by $\E_i$ the restriction of $\E$ to $U_i$ and let
$\psi_i:\wedge^r\E_i\IsomRightArrow \StructureSheaf{U_i}$ be a trivialization of the determinant line bundle. 
The line bundle $\wedge^r\E$ is represented by the \v{C}ech cocycle $\eta_{ij}:=(\psi_i\restricted{)}{U_{ij}}\circ(\psi_j\restricted{)}{U_{ij}}^{-1}$.
Choose an $r$-th root $\tilde{\eta}_{ij}$ of the invertible holomorphic function
$\eta_{ij}$. 
Let $\varphi_{ij}:(\E_j\restricted{)}{U_{ij}}=\restricted{\E}{U_{ij}}\rightarrow \restricted{\E}{U_{ij}}=(\E_i\restricted{)}{U_{ij}}$ be multiplication by $\tilde{\eta}_{ij}^{-1}$. 
Then $\theta_{ijk}:=(\tilde{\eta}_{ij}\tilde{\eta}_{jk}\tilde{\eta}_{ki})^{-1}$ is a \v{C}ech cocycle 
$\theta$ in $\Z^2(\U,\mu_r)$ and 
$\B:=(\{\E_i\},\{\varphi_{ij}\})$ is a $\theta$-twisted sheaf. 
The line bundle $\wedge^r\B$ is trivial, as its transition functions
are
$\tilde{\eta}_{ij}^{-r}\psi_i\circ \psi_j^{-1}=(\psi_j\circ\psi_i^{-1})\circ\psi_i\circ\psi_j^{-1}=1.$
Note that the $2$-cocycle $\theta$ represents the image of the class of $\det(\E)^{-1}$ under the connecting homomorphism $H^1(M,\StructureSheaf{M}^\times)\rightarrow H^2(M,\mu_r)$ associated to the short exact sequence
\[
0\rightarrow \mu_r\rightarrow \StructureSheaf{M}^\times\RightArrowOf{(\bullet)^r}\StructureSheaf{M}^\times\rightarrow 0.
\]
If the class $[\theta]\in H^2(X,\mu_r)$ has order $\rho$, then $\rho$ divides $r$ and we may assume that the cocycle $\theta$ has coefficients in the local system
$\mu_\rho\subset \mu_r$, by the construction in Remark \ref{remark-twisting-cocycle-can-be-chosen-in-mu-order-of-Brauer-class}.
%can replace the cocycle $\theta$ by a cohomologous cocycle
The above construction goes through more generally for a twisted sheaf $\E=(\{\E_i\},\phi_{ij})$, by a cocycle $\theta'\in Z^2(\U,\mu_r)$, in which case the resulting sheaf $\B:=(\{\E_i\},\phi_{ij}\tilde{\eta}_{ij}^{-1})$ with a trivial determinant line bundle is $\theta'\theta$-twisted.
\end{construction}

\begin{rem}
Note that for the twisted sheaf $\B$ in Construction \ref{construction-twisted-sheaf-with-trivial-determinant-associated-to-a-coherent-sheaf}
we have the natural isomorphism $\Ext^1(\B,\B\otimes \Omega_M)\cong \Ext^1(\E,\E\otimes \Omega_M)$, 
%and $\Ext^1(\iota^*\B,\iota^*\B\otimes \Omega_{M^0})\cong \Ext^1(\iota^*\E,\iota^*\E\otimes \Omega_{M^0})$, 
since $R\SheafHom(\B,\B)$ and $R\SheafHom(\E,\E)$ are naturally isomorphic.
\end{rem}

Let $M$ be a smooth variety, $\U:=\{U_i\}_{i\in I}$ an open covering, and $\theta:=(\theta_{ijk})_{i,j,k\in I}$ a \v{C}ech $2$-cocycle with coefficients in the local system $\mu_r$. Let $\Delta_M\subset M\times M$ be the diagonal and let $\M\subset M\times M$ the first order infinitesimal neighborhood of $\Delta_M$, i.e., $\M$ is the subscheme of $M\times M$ with ideal sheaf $\Ideal{\Delta_M}^2$. Let $\pi_i:\M\rightarrow M$, $i=1,2$, be the two projections and let $\delta:M\rightarrow \M$ be the inclusion. We have the short exact sequence
\begin{equation}
\label{eq-short-exact-sequence-of-structure-sheaf-of-first-order-neighborhood}
0\rightarrow \delta_*\Omega_M\rightarrow \StructureSheaf{\M}\RightArrowOf{\delta^*}\delta_*\StructureSheaf{M}\rightarrow 0.
\end{equation}
Its extension class is called the {\em universal Atiyah class} and is 
a morphism 
\begin{equation}
\label{eq-universal-atiyah-class}
at:\delta_*\StructureSheaf{M}\rightarrow \delta_*\StructureSheaf{\Omega_M}[1]
\end{equation} 
in $D^b(M\times M)$, which we regard as a natural transformation $at:id\rightarrow \Omega_M[1]\otimes(\bullet)$ of endofunctors of $D^b(M)$.
Note that the open coverings
$\{\pi_1^{-1}(U_i)\}_{i\in I}$ and $\{\pi_2^{-1}(U_i)\}_{i\in I}$ of $\M$ coincide. The $2$-cocycles $(\pi_1^*\theta_{ijk})_{i,j,k\in I}$ and 
$(\pi_2^*\theta_{ijk})_{i,j,k\in I}$ coincide as well, since the functions $\theta_{ijk}$ are locally constant. Hence, 
the categories of twisted coherent sheaves $Coh(\M,\pi_1^*\theta)$ and $Coh(\M,\pi_2^*\theta)$ coincide, and so do 
$D^b(\M,\pi_1^*\theta)$ and $D^b(\M,\pi_2^*\theta)$. The morphism $\pi_2$ is affine. Consequently, the functor
\[
\pi_{2,*}:Coh(\M,\pi_1^*\theta)\rightarrow Coh(M,\theta)
\]
is well defined as is its derived analogue $R\pi_{2,*}:D^b(\M,\pi_1^*\theta)\rightarrow D^b(M,\theta)$. Any object $E$ in $D^b(\M)$ is thus a Fourier-Mukai kernel for an endofunctor $R\pi_{2,*}(L\pi_1^*(\bullet)\otimes E)$ of $D^b(M,\theta)$. The universal Atiyah class (\ref{eq-universal-atiyah-class}) is thus also a natural transformation of endofunctors of $D^b(M,\theta)$.

\begin{defi}
\label{def-atiyah-class-of-twisted-sheaf}
Let $\B$ be a $\theta$-twisted sheaf on $M$. 
\label{def-jet-bundle-of-mu-r-twisted-sheaf}
%\begin{enumerate}
%\item
%\item
The {\em Atiyah class}\footnote{The Atiyah class for more general twisted sheaves is defined in \cite[Sec. 6.5.1]{lieblich}.}  
\[
at_\B^{}\in\Ext^1(\B,\B\otimes \Omega_M):= \Hom(\B,\B\otimes \Omega_M[1])
\] 
of $\B$ is 
the evaluation of the natural transformation (\ref{eq-universal-atiyah-class}) on the object $\B$.\footnote{This agrees with our sign convention in footnote \ref{footnote-atiyah-class}, see \cite[Theorem 5]{atiyah}.} 
%\end{enumerate}
\end{defi}

The {\em first jet sheaf} of $\B$ is the $\theta$-twisted sheaf 
$j^1(\B):=\pi_{2,*}(\pi_1^*\B)$.
Note that if $\B$ is locally free, then $at_\B^{}$ is the extension class
of the short exact sequence of $\theta$-twisted sheaves
\[
0\rightarrow \B\otimes \Omega_M\rightarrow j^1(\B)\rightarrow \B\rightarrow 0
\]
associated to the short exact sequence (\ref{eq-short-exact-sequence-of-structure-sheaf-of-first-order-neighborhood}) of Fourier-Mukai kernels in $D^b(\M)$.
The Atiyah class of a complex of locally free sheaves is defined in \cite[Sec 10.1]{huybrechts-lehn} and the definition goes through for complexes $E^\bullet$ 
of $\theta$-twisted locally free sheaves, hence for objects in $D^b(M,\theta)$. One has the identity 
\begin{equation}
\label{eq-atiyah-class-of-tensor-product}
at^{}_{E\otimes F}=at^{}_{E}\otimes id_{F}+id_{E}\otimes at^{}_{F}, 
\end{equation}
for objects $E, F$ in $D^b(M,\theta)$ (see \cite[Sec 10.1]{huybrechts-lehn}).

\begin{defi}
\label{def-kappa-B}
Let $\B$ be a $\mu_\rho$-twisted sheaf. Define $c_1(\B)\in H^1(M,\StructureSheaf{M})$ as $tr(at^{}_\B)$ and let $ch(\B)\in\oplus_q H^q(M,\Omega^q_M)$ be the trace of the exponential Atiyah class $\exp(at^{}_\B)$. If $r:=\rank(\B)$ is non-zero, set $\kappa(\B):=ch(\B)\exp(-c_1(\B)/r)$.
\end{defi}

Note that $c_1(\B)$ belongs to the image in $H^1(M,\Omega^1_M)$ of the Neron-Severi group tensored with $\QQ$. It suffices to show it for $\B$ of non zero rank $r$. The object $\B^{\otimes \rho}$ is untwisted, and $c_1(\B^{\otimes \rho})=\rho r^{\rho-1}c_1(\B)$, by Equation (\ref{eq-atiyah-class-of-tensor-product}). Similarly, 
$ch(\B)$ is a rational class, since
$ch(\B^{\otimes \rho})=ch(\B)^\rho$ and so $ch(\B)$ is the $\rho$-th root with constant term $r$ of the Chern character $ch(\B^{\otimes \rho})$.

\begin{lem}
Let $\B$ be a coherent sheaf of positive rank $r$ twisted by a cocycle $\theta$ with coefficients in $\mu_r$ and with a trivial determinant line bundle. 
Assume that the image of $\theta$ in $H^2(M,\StructureSheaf{M}^\times)$ is trivial, so that $\B\cong\E\otimes \LB$, where $\E$ is an untwisted rank $r$ coherent sheaf and $\LB$ is a rank $1$ locally free $\theta$ twisted sheaf. Then the following equality holds
\begin{equation}
\label{eq-atiya-class-of-twisted-B-versus-that-of-E}
at^{}_\B=at_\E-\frac{c_1(\E)}{r}\otimes id_\E,
\end{equation}
where $\frac{c_1(\E)}{r}$ is regarded as a class in $H^1(M,\Omega^1_M)$.
\end{lem}

\begin{proof}
Equation (\ref{eq-atiyah-class-of-tensor-product}) yields the equality $at_\B^{}=at_\E\otimes id_\LB+id_\E\otimes at^{}_\LB$ and $at_\LB=c_1(\LB)$, which is $-\frac{c_1(\E)}{r}$, since $\det(\B)$ is trivial.
%***********
% Hide
%***********
\hide{
Keep the notation $\iota:M_0\rightarrow M$ of Lemma \ref{lemma-atiyah-class-of-a-reflexive-twisted-sheaf}.
Then $\iota^*at^{}_\B=at_{\PP(\iota^*\E)}$, by Lemma \ref{lemma-atiyah-class-of-a-reflexive-twisted-sheaf} and \cite[page 189 property (2)]{atiyah}, and the latter is the traceless direct summand of $at_{\iota^*\E}=\iota^*at_\E$, by \cite[page 189 property (1)]{atiyah}. Equation
(\ref{eq-atiya-class-of-twisted-B-versus-that-of-E}) follows,
by the injectivity Lemma \ref{lemma-restriction-of-extension-classes-is-injective}
(keeping the convention of footnote \ref{footnote-atiyah-class}). 
%***********
% End Hide
%***********
}
\end{proof}

Let $\B$ be a $\mu_r$-twisted sheaf of positive rank $r$ of trivial determinant.
%as in Lemma \ref{lemma-atiyah-class-of-a-reflexive-twisted-sheaf}. 
In that case $\kappa(\B)=ch(\B)$. 
We will see in 
Remark \ref{remark-commutativity-of-diagram-of-semiregularity-map-of-projective-bundle} that
Diagram (\ref{commutative-diagram-of-semiregularity-map-of-projective-bundle}) remains commutative when $at_B$ is replaced with $at^{}_\B$ to define $\sigma^{}_\B$,  and the class $\kappa(B)$ is replaced with the characteristic class $\kappa(\B)$. 
%defined as the $r$-th root with constant term $r=\rank(\B)$ of the Chern character $ch(\B^{\otimes r}\otimes\det(\B)^{-1})$ 
%of the untwisted object $\B^{\otimes r}\otimes\det(\B)^{-1}$ (see  \cite[Sec. 2.2]{markman-BBF}). 

\begin{defi}
A projective bundle $B$ (or a twisted sheaf $\B$, twisted by a \v{C}ech $2$-cocycle with coefficients in $\mu_r$) 
%and with trivial determinant) 
is said to be {\em semiregular}, if the semiregularity map $\sigma_B$ (resp. $\sigma^{}_\B$) is injective.
\end{defi}

The following is the analogue of \cite[Th. 5.1]{buchweitz-flenner} for projective bundles and for twisted sheaves. It should follow from \cite[Remark 2.26]{pridham}, though one needs to reconcile the different terminology. In \cite{pridham} the semiregularity map takes values in Deligne cohomology. See also \cite[Theorem 14.4]{perry-hotchkiss} for an application of Pridham's semiregularity theorem for perfect twisted objects.

\begin{conj}
\label{conjecture-semiregular-twisted-sheaves-deform}
Let $\pi:\M\rightarrow S$ be a deformation of a smooth complex projective variety $M_0$ over a smooth germ $(S,0)$ and set $M_s:=\pi^{-1}(s)$ for $s\in S$. Assume that $\B$ is a semiregular  rank $r$ coherent sheaf over $M_0$, twisted by a cocycle with coefficients in $\mu_r$, 
%and with a trivial determinant line bundle,
%as in Lemma \ref{lemma-atiyah-class-of-a-reflexive-twisted-sheaf}, 
such that for all $p$ the class $ch_p(\B)$
extends to a horizontal section of $R^{2p}\pi_*\QQ$, which belongs to the direct summand $R^p\pi_*\Omega^p_\pi$ under the Hodge decomposition. Then $\B$ extends
%\footnote{The ``extending'' sheaf $\tilde{\B}$ over $\pi^{-1}(U)$ may be twisted by a \v{C}ech $2$-cocycle $\theta'$ with coefficients in $\mu_r$, 
%whose restriction to the special fiber $M_0$ is cohomologous to the cocycle $\theta$ twisting $\B$, so that $\theta=\restricted{\theta'}{M_0}\delta(\alpha)$ 
%for some \v{C}ech $1$-cochain $\alpha$ with coefficients in $\mu_r$, and $\B$ is the image of $\restricted{\tilde{\B}}{M_0}$ 
%via the equivalence induced by $\alpha$ between the category of $\restricted{\theta'}{M_0}$-twisted sheaves and $\theta$-twisted sheaves 
%over $M_0$ (see \cite{caldararu-thesis}).
%} 
to a twisted coherent sheaf over  $\pi^{-1}(U)$ for some open analytic neighborhood 
$U$ of $0$ in $S$. 
\end{conj}

Note that $ch(\B)$ remains of Hodge type, if and only if both $c_1(\B)$ and $\kappa(\B)$ remain of Hodge type. The locus where $\kappa(\B)$ remains of Hodge type contains the one where $ch(\B)$ does. If $\kappa(\B)$ remains of Hodge type over $S$, but $ch(\B)$ does not, 
we can use Construction \ref{construction-twisted-sheaf-with-trivial-determinant-associated-to-a-coherent-sheaf} to replace $\B$ with a twisted sheaf $\B'$ with trivial determinant, which is the tensor product of $\B$ with a twisted line bundle. The sheaf $\B'$ satisfies $\kappa(\B)=\kappa(\B')$ and $\B'$ satisfies the hypotheses of Conjecture \ref{conjecture-semiregular-twisted-sheaves-deform}, since $\kappa(\B')=ch(\B')$.

%************
% Hide
%************
\hide{
%****************************************************************
% 
%****************************************************************
\subsection{Reflexive $\mu_r$-twisted sheaves} 
Let $\B$ be a reflexive coherent sheaf of rank $r$ over $Y$ twisted by a \v{C}ech cocycle $\tilde{\theta}$ in $\Z^2(\U,\mu_r)$,
where $\mu_r$ is the local system of $r$-th roots of unity. Assume that $\wedge^r\B$ is trivial.
The singular locus $Y_s$ of $\B$ has codimension $\geq 3$ in $Y$. 
Set $Y^0:=Y\setminus Y_s$ and let $B$ be the projectivization of the restriction of $\B$ to $Y^0$. 

\begin{lem}
\label{lemma-restriction-of-extension-classes-is-injective}
The restriction homomorphism
\begin{equation}
\label{eq-restriction-isomorphism}
\Ext^1(\B,\B\otimes \Omega_Y)\rightarrow \Ext^1(\restricted{\B}{Y^0},\restricted{\B}{Y^0}\otimes \Omega_{Y^0})
\end{equation}
is injective.
\end{lem}

\begin{proof}
Let $\iota:Y^0\rightarrow Y$ be the inclusion morphism. 
The sheaf $\iota^*(\B\otimes \Omega_{Y})$ is locally free and the sheaf $\iota_*(\iota^*(\B\otimes \Omega_{Y}))$ is reflexive, by the Main Theorem of \cite{siu}.
A homomorphism between two reflexive sheaves, which is an isomorphism away from a subvariety of codimension larger than $1$, is necessarily an isomorphism.\footnote{Let $E$ and $F$ be reflexive sheaves and let $f:E\rightarrow F$ be such a homomorphism. The cokernel $\mbox{coker}(f)$ of $f$ has support of codimension $\geq 2$, and so the sheaf $\SheafExt^1(\mbox{coker}(f),\StructureSheaf{M})$ vanishes, which implies that $f^*:F^*\rightarrow E^*$ is an isomorphism, and so $f^{**}:E^{**}\rightarrow F^{**}$ is an isomorphism.}
The adjunction isomorphism
$\Hom(\iota^*(\B\otimes \Omega_{Y}),\iota^*(\B\otimes \Omega_{Y}))\cong \Hom(\B\otimes \Omega_{Y},\iota_*\iota^*(\B\otimes \Omega_{Y}))$
must thus map the identity automorphism to an isomorphism. The same holds for $\B$ and for any extension $\F$ 
\[
0\rightarrow \B\otimes \Omega_{Y}\rightarrow \F\rightarrow\B\rightarrow 0
\]
of $\B$ by $\B\otimes \Omega_Y$. It follows that $\iota_*\iota^*$ sends the above displayed extension to itself. Hence, if $\iota^*\epsilon_1=\iota^*\epsilon_2$,
then $\epsilon_1=\epsilon_2$. 
\end{proof}

%**************
% Hide
%**************
\hide{
\begin{lem}
\label{lemma-restriction-of-extension-classes-is-bijective}
The restriction homomorphism
\begin{equation}
\label{eq-restriction-isomorphism}
\Ext^1(\B,\B\otimes \Omega_Y)\rightarrow \Ext^1(\restricted{\B}{Y^0},\restricted{\B}{Y^0}\otimes \Omega_{Y^0})
\end{equation}
is an isomorphism 
\end{lem}

\begin{proof}
Let $\iota:Y^0\rightarrow Y$ be the inclusion morphism. 
The sheaf $\iota^*(\B\otimes \Omega_{Y})$ is locally free and the sheaf $\iota_*(\iota^*(\B\otimes \Omega_{Y}))$ is reflexive, by the Main Theorem of \cite{siu}.
A homomorphism between two reflexive sheaves, which is an isomorphism away from a subvariety of codimension larger than $1$, is necessarily an isomorphism.\footnote{Let $E$ and $F$ be reflexive sheaves and let $f:E\rightarrow F$ be such a homomorphism. The cokernel $\mbox{coker}(f)$ of $f$ has support of codimension $\geq 2$, and so the sheaf $\SheafExt^1(\mbox{coker}(f),\StructureSheaf{M})$ vanishes, which implies that $f^*:F^*\rightarrow E^*$ is an isomorphism, and so $f^{**}:E^{**}\rightarrow F^{**}$ is an isomorphism.}
The adjunction isomorphism
$\Hom(\iota^*(\B\otimes \Omega_{Y}),\iota^*(\B\otimes \Omega_{Y}))\cong \Hom(\B\otimes \Omega_{Y},\iota_*\iota^*(\B\otimes \Omega_{Y}))$
must thus map the identity automorphism to an isomorphism. 

Note that the functors $\iota^!$ and $L\iota^*$ from
$D^b(Y,\tilde{\theta})$ to $D^b(Y^0,\iota^*\tilde{\theta})$ coincide. Given a sheaf $E$ on $Y$, $\iota^*E$ is isomorphic to $L\iota^*E$. Similarly, given an $\iota^*\tilde{\theta}$-twisted sheaf $E$ over $Y_0$, $\iota_*E$ is isomorphic to $R\iota_*E$. Now, Grothendieck-Verdier duality implies that $\iota^!$ is the right adjoint of $R\iota_*$ (??? is this functor well defined for perfect complexes ???)
(see for example \cite[Theorem 3.34]{huybrechts-derived-categories-book} ???). We get\\
$
\Hom(\iota^*\B,\iota^*(\B\otimes \Omega_{Y}[1]))\cong
\Hom(\B,\iota_*\iota^*(\B\otimes \Omega_{Y})[1])\cong
\Hom(\B,\B\otimes\Omega_Y[1]).
$
\end{proof}

\begin{rem}
Given and extension $\epsilon \in \Ext^1(\iota^*\B,\iota^*\B\otimes \Omega_{Y^0})$ we get a short exact sequence
\[
0\rightarrow \iota^*\B\otimes \Omega_{Y^0}\rightarrow \F\RightArrowOf{j} \iota^*\B\rightarrow 0.
\]
The sheaf $\F$ is locally free and so $\iota_*\F$ is a coherent sheaf, by \cite{siu}, which is an extension of $\B$ (??? why is the homomorphism $\iota_*(j):\iota_*\F\rightarrow \iota_*\iota^*\B$ surjective ???) by $\B\otimes \Omega_Y$. Denote the extension class of the above sequence by $\iota_*\epsilon$. It is easy to see that $\iota_*\iota^*\epsilon=\epsilon$ and $\iota^*\iota_*\epsilon=\epsilon$ providing a more elementary proof of Lemma \ref{lemma-restriction-of-extension-classes-is-bijective}.
\end{rem}
}
%**************
% End Hide
%**************

\begin{lem}
\label{lemma-atiyah-class-of-a-reflexive-twisted-sheaf}  
The Atiyah class
$at^{}_\B$ of a reflexive coherent sheaf $\B$, twisted by a \v{C}ech $2$-cocycle with coefficients in $\mu_r$ and with a trivial determinant line bundle,
is traceless and it is mapped 
via the injective homomorphism (\ref{eq-restriction-isomorphism}) and the isomorphism
$\Ext^1(\iota^*\B,\iota^*\B\otimes \Omega_{Y^0})\cong H^1(Y^0,\A\otimes \Omega_{Y^0})$ to the Atiyah class $at_B$ of the projective bundle $B$.
\end{lem}

\begin{proof}
The trace of the Atiyah class $at_\B^{}$ is the Atiyah class of the determinant line bundle, which vanishes by our assumption that the latter is trivial.
Let $\U=\{U_i\}_{i\in I}$ be an open covering of $Y_0$, which is a refinement of the restriction to $Y_0$ of an open covering of $Y$ describing $\B$ as a twisted sheaf, so that the restriction $\B_i$ of $\iota^*\B$ to each $U_i$ admits a trivialization $u_i:\oplus_{i=1}^r\StructureSheaf{U_i}\rightarrow \B_i$. Let the locally free $\iota^*\theta$-twisted sheaf $\iota^*\B$ be given by a the \v{C}ech cochain $\{g_{ij}\}_{i,j\in I}$ with coefficients in $GL(r,\StructureSheaf{Y_0})$,
where\footnote{If the gluing transformations of $\B$ are given by 
$f_{ij}:(\B_j\restricted{)}{U_j}\rightarrow (\B_i\restricted{)}{U_i}$, satisfying $f_{ij}f_{jk}f_{ki}=\theta_{ijk}$, then $g_{ij}=u_i^{-1}f_{ij}u_j$
and $g_{ij}g_{jk}g_{ki}=u_i^{-1}f_{ij}f_{jk}f_{ki}u_i=u_i^{-1}\theta_{ijk}u_i=\theta_{ijk}$.
} 
%$g_{ij}=u_i^{-1}u_j$ and 
$g_{ij}g_{jk}g_{ki}=\theta_{ijk}$. The cocycle $\{b_{ij}\}$ in $Z^1(\U,\End(\iota^*\B)\otimes\Omega_M)$ representing the Atiyah class satisfies
\begin{equation}
\label{eq-GL-cocycle-of-Atiyah-class}
u_i^{-1}b_{ij}u_i=-(dg_{ij})g_{ij}^{-1}
\end{equation}
(see \cite[Eq. (9)]{atiyah}). The following equation verifies that $\{(dg_{ij})g_{ij}^{-1}\}$ is indeed a cocycle in our twisted setting 
\[
0=d(\theta_{ijk})\theta_{ijk}^{-1}=
d(g_{ij}g_{jk}g_{ki})(g_{ij}g_{jk}g_{ki})^{-1}=
(dg_{ij})g_{ij}^{-1}+g_{ij}[(dg_{jk})g_{jk}^{-1}]g_{ij}^{-1}+g_{ik}[(dg_{ki})g_{ki}^{-1}]g_{ik}^{-1},
\]
where the first equality is due to $\theta_{ijk}$ being locally constant and in the last equality the third term takes the above form due to the equalities $g_{ij}g_{jk}=\theta_{ijk}g_{ki}^{-1}=\theta_{ijk}g_{ik}$.

The conceptual reason for the equality (\ref{eq-GL-cocycle-of-Atiyah-class}) between the cocycle representing the extension class of the jet bundle (the left hand side) and minus the cocycle representing the extension class of the bundle $Q$ of first order differential operators with scalar symbol (the right hand side) is that the extension class of $Q$ is the pullback of that of 
\[
0\rightarrow \SheafEnd(\iota^*\B) \rightarrow \SheafHom(j^1(\iota^*\B),\iota^*\B)\rightarrow \SheafEnd(\iota^*\B)\otimes T_M\rightarrow 0
\]
via the homomorphism $T_M\rightarrow \SheafEnd(\iota^*\B)\otimes T_M$ of tensoring by the identity endomorphism. 
This identification of $Q$ as a subbundle of $\SheafHom(j^1(\iota^*\B),\iota^*\B)$ is local and so holds as well when $\iota^*\B$ is twisted. 
Hence, the equality (\ref{eq-GL-cocycle-of-Atiyah-class}) holds also in our setting in which $\iota^*\B$ is a twisted sheaf.

Let $\bar{g}_{ij}:U_{ij}\rightarrow \mbox{PGL}(r,\CC)$ be the composition of $g_{ij}:U_{ij}\rightarrow \mbox{GL}(r,\CC)$ with the natural homomorphism
$\mbox{GL}(r,\CC)\rightarrow\mbox{PGL}(r,\CC)$. The term $(dg_{ij})g_{ij}^{-1}$ on the right hand side of (\ref{eq-GL-cocycle-of-Atiyah-class})
is the homomorphism $TU_{ij}\rightarrow U_{ij}\times\LieAlg{gl}(r,\CC)$ given by the differential of $g_{ij}$ composed with right translation to $T_1\mbox{GL}(r,\CC)=\LieAlg{gl}(r,\CC)$. Composing with the projection $\mathfrak{gl}(r,\CC)\rightarrow \mathfrak{pgl}(r,\CC)$ we get the $1$-cocycle $\rho_{ij}$ with values in $\mathfrak{pgl}(r,\CC)$
representing the Atiyah class of $B$ in terms of the trivializations of $B$ over $U_i$ induced by the trivializations $u_i$ of $\iota^*\B$ (see \cite[Eq. (2)]{atiyah}).
Hence, $at_B$ is the traceless part of $\iota^*at_\B^{}$ (which is itself traceless).
\end{proof}

The Atiyah class of a line bundle of order $r$ vanishes, as it admits a holomorphic connection, a fact which agrees with the above statement 
according to which the Atiyah classes of the various choices of lifts $\B$ of $B$ are all the same. 

%************
% End Hide
%************
}

%****************************************************************
% 
%****************************************************************
\subsection{The Semiregularity Theorem for $\mu_r$-twisted sheaves on abelian varieties}
\label{sec-the-semiregularity-theorem-for-twisted-sheaves-on-abelian-varieties}
In this section we prove Conjecture \ref{conjecture-semiregular-twisted-sheaves-deform} in case $\pi$ is a family of abelian varieties. 
The proof establishes Conjecture \ref{conjecture-semiregular-twisted-sheaves-deform}, more generally, whenever the Brauer class $[\theta]\in H^2(M_0,\mu_r)$ of $\B$ is the restriction of the Brauer class of a projective bundle over $\pi^{-1}(U)$, for some open neighborhood  $U$ of $0$ in $S$.
%on $\pi^{-1}(U)$, for some open neighborhood  $U$ of $0$ in $S$, and (ii) there exists over $\pi^{-1}(U)$ a projective bundle with Brauer class $[\theta]$. 
Note\footnote{I thank Nick Addington for this comment.} 
that if $\M$ is quasi projective and $[\theta]$ is the restriction of a Brauer class on $\M\times_S\tilde{S}$, for some finite \'{e}tale cover $\tilde{S}\rightarrow S$,  then such a projective bundle exists, by Gabber's theorem \cite{de-Jong}.

%****************************************************************
% 
%****************************************************************
\subsubsection{Construction of a projective bundle}
Let $X$ be an abelian variety. 
The composition
\begin{equation}
\label{eq-embedding-of-H-2-mu-r-into-H-2-Q-mod-Z}
H^2(X,\mu_r)\IsomRightArrow H^2(X,\frac{1}{r}\ZZ/\ZZ)\rightarrow H^2(X,\QQ/\ZZ)
\end{equation}
is injective, where the left arrow is the isomorphism induced by the inverse of the sheaf isomorphism
$\exp(2\pi i(\bullet)):\frac{1}{r}\ZZ/\ZZ\rightarrow \mu_r$ and the right arrow is induced by the inclusion $\frac{1}{r}\ZZ\subset \QQ$.
The injectivity of the right arrow follows from the snake lemma applied to 
\[
\xymatrix{
0 \ar[r] & H^2(X,\ZZ) \ar[r] \ar[d]_{=} & H^2(X,\frac{1}{r}\ZZ) \ar[r] \ar[d] & H^2(X,\frac{1}{r}\ZZ/\ZZ) \ar[r]\ar[d] & 0
\\
0\ar[r] & H^2(X,\ZZ) \ar[r] & H^2(X,\QQ) \ar[r]& H^2(X,\QQ/\ZZ) \ar[r] &0,
}
\]
where the right exactness of the horizontal rows follows from the Universal Coefficient Theorem and the torsion freeness of $H^*(X,\ZZ)$, which implies also the injectivity of the middle vertical arrow.
Given a $\PP^{r-1}$ bundle $B$ over $X$ we denote by $[B]\in H^1(X,PGL(r,\StructureSheaf{X}))$ its isomorphism class and by 
$\bar{c}_1(B)\in H^2(X,\mu_r)$ the image of $[B]$ via the connecting homomorphism of the short exact sequence (\ref{eq-short-exact-sequence-mu-r-SL-PGL}).
Denote by 
\[
\delta(B)
\]
the image of $\bar{c}_1(B)$ in $H^2(X,\QQ/\ZZ)$ via (\ref{eq-embedding-of-H-2-mu-r-into-H-2-Q-mod-Z}).

\begin{lem}
\label{lemma-deforming-semihomogeneous-vb}
Let $\pi:\X\rightarrow S$ be a smooth and proper morphism of smooth analytic spaces all of which fibers are connected abelian varieties. Let $X$ be the fiber of $\pi$ over $0\in S$. 
Given a class $\delta_0\in H^2(X,\QQ/\ZZ)$, there exists an open neighborhood $U$ of $0$ in $S$ and a holomorphic projective bundle $p:\PP\rightarrow \pi^{-1}(U)$,
such that $\delta(\restricted{\PP}{X})=\delta_0$.
\end{lem}

\begin{proof}
\underline{Step 1:}
Assume first that $\delta_0$ is the image of a class $\frac{1}{r}\tilde{\delta}_0$ via the natural homomorphism
$H^2(X,\QQ)\rightarrow H^2(X,\QQ/\ZZ)$, 
where $\tilde{\delta}_0\in H^2(X,\ZZ)$ is a primitive class of Hodge type $(1,1)$. Then there exists a simple semi-homogeneous vector bundle $Q$ over $X$ with $\delta(\PP(Q))=\delta_0$,
by \cite[Theorem 7.11(1)]{mukai-semihomogeneous}. The rank $r(Q)$ of $Q$ is divisible by $r$ and divides $r^n$, where $n$ is the dimension of $X$.
Let $\Sigma(Q)\subset \hat{X}$ be the subgroup $\{L \ : \ Q\otimes L\cong Q\}$. 
Choose an origin for $X$ and denote by $X_r$ the subgroup of $X$ of $r$-torsion points. We have the short exact sequence
\[
0\rightarrow X_r\cap K(\tilde{\delta}_0)\rightarrow X_r\rightarrow \Sigma(Q)\rightarrow 0,
\]
where $K(\tilde{\delta}_0)$ is the kernel of the homomorphism $\phi_D:X\rightarrow \hat{X}$ associated to a line-bundle $D$ with $c_1(D)=\tilde{\delta_0}$,
by \cite[Theorem 7.11(4)]{mukai-semihomogeneous}.
Furthermore, $\SheafEnd(Q)\cong \oplus_{L\in\Sigma(Q)}L$, by \cite[Proposition 7.1]{mukai-semihomogeneous}.

The subgroup $\Sigma(Q)$ consists of isomorphism classes of line bundles on $X$. However,
we may and do regard $\Sigma(Q)$ as subgroup of autoequivalences of $D^b(X)$ endowed with a linearization, i.e., a choice of a line-bundle representing each isomorphism class in $\Sigma(Q)$ together with isomorphisms between
the composite functor $L_1\otimes (L_2\otimes(\bullet))$ and $L_3\otimes (\bullet)$, for all $L_1, L_2, L_3\in \Sigma(Q)$ such that $L_3$ is isomorphic to $L_1\otimes L_2$, and these isomorphisms of functors satisfy the natural associativity axioms. Such a linearization is provided by conjugating the action on $D^b(\hat{X})$ of the group $\Sigma(Q)$ of translation automorphisms of $\hat{X}$ via the equivalence $\Phi_\P:D^b(X)\rightarrow D^b(\hat{X})$ with the Poincare line bundle $\P$ as a Fourier-Mukai kernel. Alternatively, it is provided by the Appell-Humbert theorem \cite[Theorem 2.2.3]{BL}.

Set $\A:=\oplus_{L\in\Sigma(Q)}L$.
Let $a:\A\otimes\A\rightarrow \A$ be a homomorphism of $\StructureSheaf{X}$-modules. Denote by
\[
\lambda, \rho : \A \rightarrow \SheafEnd(\A,\A)
\]
the homomorphism $\lambda(s)=a(s,(\bullet))$ and $\rho(s)=a((\bullet)\otimes s)$. Then $a$ is the multiplication for the structure of an Azumaya algebra on $\A$, if and only if $a$ endows $\A$ with the structure of a sheaf of associative algebras with a unit, and  
$\lambda\otimes \rho:\A\otimes \A^{\circ}\rightarrow \SheafEnd(\A)$ is an isomorphism of sheaves of algebras (see \cite[Proposition IV.2.1]{milne}).
Here $\A^{\circ}$ is $\A$ endowed with the multiplication $a^\circ$ obtained by composing $a$ with the transposition of the tensor factors of $\A\otimes\A$. 
The condition that $\lambda\otimes\rho$ is an isomorphism of sheaves of algebras means that it is an isomorphism of $\StructureSheaf{X}$-modules satisfying 
\[
(\lambda\otimes \rho)(s_2\otimes t_2)\circ (\lambda\otimes \rho)(s_1\otimes t_1)=(\lambda\otimes \rho)(a(s_2\otimes s_1)\otimes a(t_1\otimes t_2))
\]
for local sections $s_1, s_2$ of $\A$ and $t_1, t_2$ of $\A^\circ$.

Given $L_1, L_2\in \Sigma(Q)$, let the sheaf homomorphism $e_{L_2,L_1}:L_2\otimes L_1^{-1}\rightarrow \SheafEnd(\A,\A)$ send a section $s$ of
$L_2\otimes L_1^{-1}$ to the homomorphism mapping the direct summand $L_1$ of $\A$  to the direct summand $L_2$ by tensorization by $s$ and annihilating all other direct summands of $\A$. 
There exist complex numbers $a_{L_1,L_2}$, such that 
\[
\lambda=
\sum_{L_1, L_2\in \Sigma(Q)} a_{L_1,L_2}e_{L_2,L_1}.
\]
Here we use the linearization of the action of $\Sigma(Q)$ of $D^b(A)$, which provides an isomorphism between $L_2\otimes L_1^{-1}$ and a direct summand of $\A$. Consider the composition 
%$e_{L_3,L_2}\circ e_{L_2,L_1}$ given by
\[
L_3\otimes L_1^{-1}\IsomRightArrow (L_3\otimes L_2^{-1})\otimes (L_2\otimes L_1^{-1}) \LongRightArrowOf{e_{L_3,L_2}\otimes e_{L_2,L_1}} \SheafEnd(\A)\otimes\SheafEnd(A)\rightarrow \SheafEnd(\A),
\]
where 
%the left isomorphism is provided by the linearization of the action of $\Sigma(Q)$ on $D^b(X)$ and 
the right arrow is multiplication in $\SheafEnd(\A)$. The above composition is equal to $e_{L_3,L_1}$ and is independent of $a$.
Thus, the conditions that $\lambda$ and $\rho$ are homomorphisms of sheaves of algebras and that $\lambda\otimes\rho$ is an isomorphism of sheaves of algebras are all expressed in terms of universal equations among the constants $a_{L_1,L_2}$. These equations depend on the group $\Sigma(Q)$ as an abstract abelian group, but are independent of the complex structure of $X$ and of the embedding of $\Sigma(Q)$ as a subgroup of $\hat{X}$. 

The subgroup $\Sigma(Q)$ of $\hat{X}$ deforms with $X$ canonically as a fiber of a local system $\Sigma$, which is trivial 
over a simply connected open neighborhood $U$ of $0$ in $S$. Hence, 
The locally free sheaf $\A$ deforms with $X$ to a locally free sheaf over $\pi^{-1}(U)$. 
Indeed, it is simply the pushforward to the first factor of the restriction to $\X\times_U\Sigma$  of a relative Poincar\'{e} line bundle over $\X\times_U \Pic^0(\X/U)$.
The constants $a_{L_i,L_j}$ vary in the local system $\CC[\Sigma^2]$
over $U$ of functions from the fibers of $\Sigma\times_U\Sigma$ to $\CC$.
We get a locally trivial fibration $Azu\subset \CC[\Sigma^2]$ by the space of solutions to the equations defining the structure of an Azumaya algebra. The fibration  $Azu$ is trivial over every subset where the local system $\Sigma$ is trivial, in particular over the simply connected neighborhood $U$.
Choosing a flat section of $\CC[\Sigma^2]$ with value\footnote{Note that the projective bundle $\PP(Q)$ is infinitesimally rigid over $X$ and its automorphism group is trivial, by \cite[Proposition 5.9 and Lemma 6.7]{mukai-semihomogeneous}. It follows that the connected component of $a$ in the fiber of $Azu$ over $0$ is a single point, by the bijection between isomorphism classes of projective bundles and of Azumaya algebras \cite[Proposition IV.2.3]{milne}.} 
$a$ over $0\in U$ equal to the multiplication in $\SheafEnd(Q)$, we get a section of $Azu$  and hence also the desired deformation of $(\A,a)$ via an Azumaya algebra over $\pi^{-1}(U)$. We get the holomorphic projective bundle $\PP$ over $\pi^{-1}(U)$, by the bijection between isomorphism classes of Azumaya algebras and projective bundles \cite[Proposition IV.2.3]{milne}.

\underline{Step 2:}
If the class $\delta_0$ does not admit a lift to $\frac{1}{r}\tilde{\delta}_0\in H^{1,1}(X,\ZZ),$ then we can choose an auxiliary family of abelian varieties $\X'\rightarrow S'$ over a simply connected analytic base $S'$ with one fiber $X$ and another $X'$, such that the parallel transport $\delta_1$ of $\delta_0$ to the fiber $X'$ does admit such a lift. The construction in Step 1 gives rise to a projective bundle over $\X'$ whose restriction $\PP_0$ to $X$ has class $\delta_0$. Hence, we get an Azumaya algebra $\A$ over $X$, which is again isomorphic as a locally free sheaf to the direct sum of line bundles in a subgroup $\Sigma_0$ of $\hat{X}$. Repeating the construction in Step 1 with $\Sigma_0$ instead of $\Sigma(Q)$ we get  the desired projective bundle $\PP$ over $\pi^{-1}(U)$, for any simply connected open neighborhood $U$ of $0$ in $S$.
\end{proof}

Let $p:\PP\rightarrow \pi^{-1}(U)$ be as in Lemma \ref{lemma-deforming-semihomogeneous-vb}.
Let $\D\subset \PP\times_{\pi^{-1}(U)}\PP^*$ be the incidence divisor. Given a point $x\in\X$, the line bundle $\StructureSheaf{\PP\times_{\pi^{-1}(U)}\PP^*}(\D)$
restricts to the fiber $\PP_x\times\PP^*_x$ as $\StructureSheaf{\PP_x\times\PP^*_x}(1,1).$
Let $\Q_\PP$ be the direct image of $\StructureSheaf{\PP\times_{\pi^{-1}(U)}\PP^*}(\D)$ via the projection from $\PP\times_{\pi^{-1}(U)}\PP^*$ to the first factor $\PP$.
Then $\Q_\PP$ is a locally free sheaf over $\PP$, which restriction to the fiber $\PP_x$ is a direct sum of copies of $\StructureSheaf{\PP_x}(1)$ and such that $p_*\Q_\PP$
is isomorphic to the Azumaya algebra of $\PP$.

%****************************************************************
% 
%****************************************************************
\subsubsection{Proof of Conjecture \ref{conjecture-semiregular-twisted-sheaves-deform} when $\pi:\M\rightarrow S$ is a family of connected abelian varieties}
\label{sec-proof-of-the-semiregularity-thm-twisted-sheaves-case-for-families-of-abelian-varieties}
\begin{proof}
\underline{Step 1:}
Let $\B$ be a rank $r$ semiregular coherent sheaf over $M_0$ twisted by a \v{C}ech $2$-cocycle $\theta$ with coefficients in $\mu_\rho$, for some $\rho$ dividing $r$, and with a trivial determinant line bundle. The latter assumption will be dropped in Step 6.
Note that we may and do assume that the order of the class $[\theta]$ in $H^2(M_0,\mu_\rho)$ is $\rho$, by Remark \ref{remark-twisting-cocycle-can-be-chosen-in-mu-order-of-Brauer-class}.
There exists an open neighborhood $U$ of $0$ in $S$ and a projective bundle $p:\PP\rightarrow \pi^{-1}(U)$, which restricts to $M_0$ as a projective bundle $\PP_0$ with $\bar{c}_1(\PP_0)=[\theta]\in H^2(M_0,\mu_\rho)$, by Lemma \ref{lemma-deforming-semihomogeneous-vb}. Let $Q_0$ be a lift of $\PP_0$ to a
$\theta$-twisted locally free sheaf over $M_0$, such that $\PP_0$ is isomorphic to $\PP(Q_0)$ (here we use Remark \ref{remark-twisting-cocycle-can-be-chosen-in-mu-order-of-Brauer-class} to keep the coefficients in $\mu_\rho$, rather than in $\mu_{\rank(Q_0)}$). We may and do assume that $\det(Q_0)$ is trivial, by Example \ref{example-a-lift-of-a-projective-bundle-to-a-mu-r-twisted-sheaf}. 
The pullback $p_0^*Q_0^*$ admits a tautological quotient rank one $p_0^*\theta^{-1}$-twisted locally free sheaf $\StructureSheaf{\PP(Q_0)}(1)$.
Note that $(\StructureSheaf{\PP(Q_0)}(1))^{\otimes \rho}$ is untwisted, as $\theta$ has coefficients in $\mu_\rho$. 
We denote the line bundle $(\StructureSheaf{\PP(Q_0)}(1))^{\otimes \rho k}$ by $\StructureSheaf{\PP(Q_0)}(\rho k)$, for any integer $k$.
Let $\tilde{E}_0$ 
be\footnote{In the special case where $\B$ is obtained from an untwisted reflexive sheaf $E$, as in Construction \ref{construction-twisted-sheaf-with-trivial-determinant-associated-to-a-coherent-sheaf}, we can define $\tilde{E}_0$ more directly as follows. Let $Q$ be a simple semihomogeneous bundle with $c_1(Q)/\rank(Q)=c_1(E)/\rank(E)$. We can further choose $Q$, so that $\det(E\otimes Q^*)$ is trivial.
Let $\PP$ be the extension of $\PP(Q)$ over $\pi^{-1}(U)$ provided by Lemma \ref{lemma-deforming-semihomogeneous-vb} and set $\tilde{E}_0:=(p_0^*E)\otimes\StructureSheaf{\PP(Q)}(1)$.
} 
the corresponding (untwisted) tautological quotient sheaf $(p_0^*\B)\otimes\StructureSheaf{\PP(Q_0)}(1)$
of $p_0^*(\B\otimes Q_0^*).$ Then $p_{0,*}\tilde{E}_0$ is isomorphic to the untwisted sheaf $\B\otimes Q_0^*$ with trivial determinant.

\underline{Step 2:} We prove next that $c_1(\tilde{E}_0)\in H^2(\PP_0,\ZZ)$ remains of Hodge type $(1,1)$ over each fiber of $\pi\circ p:\PP\rightarrow U$.
The line bundle $\det(\tilde{E}_0)$ is isomorphic to $\StructureSheaf{\PP(Q_0)}(r)$, since $\det(\B)$ is trivial. 
The locally free sheaf $p_0^*Q_0\otimes\StructureSheaf{\PP(Q_0)}(1)$ is untwisted.
The line bundle $\det((p_0^*Q_0)\otimes \StructureSheaf{\PP(Q_0)}(1))$ is isomorphic to 
$\StructureSheaf{\PP(Q_0)}(\rank(Q_0))$, since $\det(Q_0)$ is trivial.
%We conclude that the tensor power of the rank $1$ locally free $p_0^*\theta^{-1}$-twisted sheaf $\StructureSheaf{\PP(Q_0)}(1)$ to the 
%power $\rho:=\gcd(r,\rank(Q_0))$ is untwisted. We denote $(\StructureSheaf{\PP(Q_0)}(1))^{\otimes \rho k}$ by $\StructureSheaf{\PP(Q_0)}(\rho k)$, 
%for any integer $k$.
The line bundle $\StructureSheaf{\PP(Q_0)}(\rank(Q_0))\cong \det((p_0^*Q_0)\otimes \StructureSheaf{\PP(Q_0)}(1))$ is isomorphic 
to the restriction of $\det(\Q_\PP)$ to $\PP_0$.  Hence, $c_1(\StructureSheaf{\PP(Q_0)}(\rho))$ remains of Hodge type $(1,1)$ over every fiber of $p\circ \pi:\PP\rightarrow U$. Consequently, so does $c_1(\tilde{E}_0)$.

\underline{Step 3:} The Chern character $ch(\tilde{E}_0)$ is equal to $\kappa(\tilde{E}_0)\exp(c_1(\tilde{E}_0)/r)$, which is equal to $p^*\kappa(\B)\exp(c_1(\tilde{E}_0)/r)$. It remains of Hodge type over every fiber of $\pi\circ p:\PP\rightarrow U$, since $\kappa(\B)$ does, by assumption, and 
$c_1(\tilde{E}_0)$ does, by Step 2.

\underline{Step 4:}
We prove next that the sheaf $\tilde{E}_0$ is semiregular.
The homomorphism $p_0^*:\Ext^2(\B,\B)\rightarrow \Ext^2(p_0^*\B,p_0^*\B)$ is an isomorphism, since $Rp_{0,*}\StructureSheaf{\PP_0}$ is isomorphic to 
$\StructureSheaf{M_0}$ and so 
\[
\Ext^2(p_0^*\B,p_0^*\B)\cong H^2(M_0,R\SheafHom(\B,\B)\otimes Rp_{0,*}\StructureSheaf{\PP_0})\cong H^2(M_0,R\SheafHom(\B,\B))\cong \Ext^2(\B,\B).
\]
Set $\lambda:=c_1(\tilde{E}_0)$.
We have the commutative diagram
\[\xymatrixcolsep{5pc}
\xymatrix{
\Ext^2(\B,\B) \ar[r]^{p_0^*} \ar[d]^{\sigma_\B^{}} &
\Ext^2(p_0^*\B,p_0^*\B) \ar[r]^{\cong} \ar[d]^{\sigma_{p_0^*\B}^{}} &
\Ext^2(\tilde{E}_0,\tilde{E}_0) \ar[d]_{\sigma_{\tilde{E}_0}}
\\
\oplus_{q=0}^{2n-2} H^{q+2}(\Omega^q_{M_0}) \ar[r]_{p_0^*} &
\oplus_{q=0}^{2n+2r-4} H^{q+2}(\Omega^q_{\PP_0}) \ar[r]_{\cup(\exp(\lambda/r))}&
\oplus_{q=0}^{2n+2r-4} H^{q+2}(\Omega^q_{\PP_0}) 
}
\]
The commutativity of the right square follows from Equation (\ref{eq-atiya-class-of-twisted-B-versus-that-of-E}).
The top right isomorphism is induced by the natural isomorphism $R\SheafHom(p_0^*\B,p_0^*\B)\cong R\SheafHom(\tilde{E}_0,\tilde{E}_0)$.
The top horizontal homomorphisms are isomorphisms. The bottom horizontal homomorphisms are injective, and the left vertical homomorphism is the semiregularity map $\sigma_\B^{}$, which is assumed to be injective. Hence, the semiregularity map $\sigma_{\tilde{E}_0}$ is injective as well.

\underline{Step 5:}
The Semiregularity Theorem \cite[Th. 5.1]{buchweitz-flenner} implies that the sheaf $\tilde{E}_0$ extends to a coherent sheaf $\tilde{E}$ over $\PP$, possibly after replacing the open neighborhood $U$ of $0$ in $S$ by a smaller open neighborhood. We may further shrink $U$, so that the fibration $\pi^{-1}(U)\rightarrow U$
is topologically trivial and thus extend the cocycle $\theta$ in $Z^2(\U,\mu_\rho)$ to a \v{C}ech $2$-cocycle $\theta'$ over $\pi^{-1}(U)$, which restricts to $M_0$ as $\theta$.
Then there exists a $\theta'$-twisted sheaf $\Q'$ over $\pi^{-1}(U)$, such that the restriction of $\Q'$ to $M_0$ is isomorphic to $Q_0$
and $\PP(\Q')$ is isomorphic to $\PP$. We get the $p^*\theta'$-twisted tautological line subbundle $\StructureSheaf{\PP(\Q')}(-1)$ of $p^*\Q'$, which restricts to $\PP_0$ as $\StructureSheaf{\PP(Q_0)}(-1)$. We have the isomorphism
$p_{0,*}\left(\tilde{E}_0\otimes \StructureSheaf{\PP(Q_0)}(-1)\right)\cong \B$ and the vanishing 
$Rp_{0,*}^i\left(\tilde{E}_0\otimes \StructureSheaf{\PP(Q_0)}(-1)\right)=0$, for $i>0$. Hence, we may assume that 
$Rp_*^i\left(\tilde{E}\otimes \StructureSheaf{\PP(\Q')}(-1)\right)$ vanishes, for $i>0$, by upper-semi-continuity, possibly after shrinking the open neighborhood $U$ further. Consequently, 
$p_*\left(\tilde{E}\otimes \StructureSheaf{\PP(\Q')}(-1)\right)$ is a coherent $\theta'$-twisted sheaf over $\pi^{-1}(U)$, flat over $U$, extending
$\B$.

\underline{Step 6:}
Finally, we drop the assumption that $\det(\B)$ is trivial. Let $\F$ be a coherent sheaf with trivial determinant, twisted by a  cocycle $\alpha$ with coefficients in $\mu_r$,
and $L$ a line bundle twisted by a  cocycle $\beta$ with coefficients in $\mu_r$, so that $\theta=\alpha\beta$ and $\B$ is isomorphic to $\F\otimes L$.
Such $\F$ and $L$ are constructed in Construction \ref{construction-twisted-sheaf-with-trivial-determinant-associated-to-a-coherent-sheaf}.
There exists an open subset $U$ of $S$, a cocycle $\alpha'$ with coefficients in $\mu_r$, and an $\alpha'$-twisted coherent sheaf 
$\F'$ over $\pi^{-1}(U)$ extending $\F$, by Step 5. It thus remains to extend $L$. The untwisted line bundle $L^r$ is isomorphic to $\det(\B)$, and so it extends to a line bundle $N$ over $\pi^{-1}(U)$, by the assumption that $c_1(\B)$ remains of Hodge type, possibly after shrinking $U$. Let $L'$ be  a $\beta'$-twisted $r$-th root of $N$, where $\beta'$ is a $2$-cocycle with coefficients in $\mu_r$, as in 
Construction \ref{construction-twisted-sheaf-with-trivial-determinant-associated-to-a-coherent-sheaf}. The restriction $\bar{L}'$ of $L'$
to $M_0$ satisfies $(\bar{L}')^r\cong L^r$. Hence, the transition functions of $(\bar{L}')^{-1}\otimes L$ in the open covering in Construction 
\ref{construction-twisted-sheaf-with-trivial-determinant-associated-to-a-coherent-sheaf} form a $1$-cochain $\gamma$ with coefficients in $\mu_r$. Let $\gamma'$ be an extension of $\gamma$ to a $1$-cochain with coefficients in $\mu_r$ over $\pi^{-1}(U)$ using the topological triviality of $\pi$, as in the previous step. Multiplying the transition functions of $L'$ by $\gamma'$ we obtain a twisted line bundle extending $L$.
%a twisted sheaf $\beta'$ over $\pi^{-1}(U)$ extending $\B$.
\end{proof}

%****************************************************************
% 
%****************************************************************
\section{Secant sheaves on abelian threefolds}
\label{section-secant-sheaves-on-abelian-threefolds}
In Section \ref{sec-K-secants-on-the-generic-ppav} we enumerate the $K$-secants spanned by rational Hodge classes for a generic principally polarized abelian variety.
In Section \ref{sec-examples-of-secant-sheaves} we consider examples of ideal sheaves twisted by a line-bundle, which are secant sheaves over abelian threefolds and fourfolds.
In Section \ref{sec-rank-6-obstruction-map} we consider the ideal sheaf $F:=\Ideal{\cup_{i=1}^{d+1}C_i}$ of $d+1$ translates $C_i$ of the Abel-Jacobi 
image of a non-hyperelliptic curve $C$ of genus $3$ in its Jacobian $X=\Pic^2(C)$. We show that the obstruction map
$ob_F:HT^2(X)\rightarrow \Ext^2(F,F)$ has rank $6$, and so its kernel is a $9$-dimensional subspace of unobstructed first order deformations of the pair $(X,F)$.

In Section \ref{sec-ob-E} we set $F_1:=\Ideal{\cup_{i=1}^{d+1}C_i}(\Theta)$ and $F_2:=\Ideal{\cup_{i=1}^{d+1}\Sigma_i}(\Theta)$,
where $\Sigma_i$ are images under the natural involution $L\mapsto \omega_C\otimes L^{-1}$ of $\Pic^2(C)$ of translates of the Abel-Jacobi image of $C$.
Let the object $E:=\Phi(F_1\boxtimes F_2)$ be the image of the outer tensor product of $F_1$ and $F_2$ via Orlov's
equivalence $\Phi:D^b(X\times X)\rightarrow D^b(X\times\hat{X})$. We show that the kernel of
$ob_E:HT^2(X\times\hat{X})\rightarrow \Ext^2(E,E)$ is equal to the kernel of the homomorphism $ch(E):HT^2(E)\rightarrow H^*(X\times\hat{X},\CC)$
via the $HT^*(X\times\hat{X})$-module structure of $H^*(X\times\hat{X},\CC)$ and the action of $HT^*(X\times\hat{X})$ on $ch(F)\in HH^*(X\times\hat{X},\CC)$.

In Section \ref{sec-diaginal-deformations} we show that the isomorphism $\Phi^{HT}:HT^2(X\times X)\rightarrow HT^2(X\times\hat{X})$
maps the ``diagonally'' embedded $\ker(ob_{F_1})\subset HT^2(X)$ to a $9$-dimensional subspace of first order commutative and gerby deformations of $X\times\hat{X}$. 

%****************************************************************
% 
%****************************************************************
\subsection{$K$-secants on a generic ppav}
\label{sec-K-secants-on-the-generic-ppav}
Let $(X,\Theta)$ be a principally polazixed abelian $n$-fold with a cyclic Neron-Severi group generated by $\Theta$. Let $\phi_\Theta:X\rightarrow \hat{X}$ be the isomorphism induced by $\Theta$ and set $f:=(\phi_{\Theta}^*)^{-1}:H^1(X,\ZZ)\rightarrow H^1(\hat{X},\ZZ)$. Given $w\in H^1(X,\ZZ)$ and $\theta\in H^1(\hat{X},\ZZ)$ and identifying $H^1(\hat{X},\ZZ)$ with $H^1(Z,\ZZ)^*$, we have
\begin{equation}
\label{eq-identity-expressing-anti-symmetry-of-theta}
f(w)(f^{-1}(\theta))=-\theta(w).
\end{equation}
We have 
\begin{equation}
\label{eq-two-by-two-matrix}
\End_\QQ(X\times\hat{X})\cong \End_{Hdg}(V_\QQ)=
\left\{\left(
\begin{array}{cc}
a_{11}I_X & a_{12}f^{-1}
\\
a_{21}f & a_{22}I_{\hat{X}}
\end{array}
\right) \ : \ a_{ij}\in\QQ
\right\},
\end{equation}
where $I_X$ is the identity endomorphism of $H^1(X,\QQ)$  and $I_{\hat{X}}$ is defined analogously. 
The identity (\ref{eq-identity-expressing-anti-symmetry-of-theta}) implies that 
the group $SO^+_{Hdg}(V_\QQ)$ is the subgroup of invertible elements of $\End_{Hdg}(V_\QQ)$
with $a_{11}a_{22}-a_{21}a_{12}=1.$ 
Allowing the coefficients to belong to a quadratic imaginary number field $K$ and imposing $a_{11}a_{22}-a_{21}a_{12}=1$ 
we get the group $SO^+_{Hdg}(V_K)$.

\begin{lem}
\label{lemma-orbit-of-pure-spinors}
The $SO^+_{Hdg}(V_K)$-orbit of the pure spinor $\span_K\{1\}\in \PP(S^+_K)$ is \\
$\{\span_K\{\Theta^n\}\}\cup \{\span_K\{\exp(k\Theta)\} \ : \ k\in K
\}$.
\end{lem} 

\begin{proof}
Denote the displayed matrix in (\ref{eq-two-by-two-matrix}) by $A$. If  $a_{22}=0$, then $A(H^1(\hat{X},K))=H^1(X,K)$.
Hence, in that case $A$ maps the pure spinor $H^0(X,K)=\span_K\{1\}$ of the maximal isotropic subspace $H^1(\hat{X},K)$ to the pure spinor $H^{2n}(X,K)=\span_K\{\Theta^n\}$
of $H^1(X,K)$. 
The subgroup of $SO^+_{Hdg}(V_K)$ leaving $H^1(\hat{X},K)$ invariant is the lower triangular subgroup with $a_{12}=0$. Hence, the $SO^+_{Hdg}(V_K)$-orbit
of $\span_K\{1\}$ is isomorphic to $\PP^1_K$. It follows that the latter orbit is the
rational normal curve of degree $n$ in $\PP(\span_K\{\Theta^j\! : \! 0\!\leq\! j \!\leq\! n\})$, which is 
the Zariski closure
of $\{\span_K\{\exp(k\Theta)\} \! : \! k\in K\}$.
\end{proof}

The involution $\iota$ interchanging the coefficients of $\frac{\Theta^j}{j!}$ and $\frac{\Theta^{n-j}}{(n-j)!}$ acts on the above set of pure spinors and corresponds to 
the equality $k^n\exp(k^{-1}\Theta)=\iota(\exp(k\Theta))$. The action of the element $\left(\begin{array}{cc}0&-f^{-1}\\f&0\end{array}\right)$ on $\PP(S_K)$ extends the action of $\iota$.

If $K=\QQ(\sqrt{-d})$, $d$ a positive integer, and we write $k=(\rho+\tau\sqrt{-d})/q$, with $\rho,\tau, q\in\ZZ$, $\gcd(\rho,\tau,q)=1$, $q>0$, and $\tau\neq 0$, then the non-rational pure spinors defined over $K$ in Lemma \ref{lemma-orbit-of-pure-spinors} are enumerated by $(\alpha+\tau\sqrt{-d}\beta)/q^n$, where
\begin{eqnarray*}
\alpha & = &
q^n+\rho q^{n-1}\Theta+q^{n-2}(\rho^2-\tau^2d)\frac{\Theta^2}{2} + 
q^{n-3}\rho(\rho^2-3\tau^2d)\frac{\Theta^3}{3!}
\\
&&+q^{n-4}(\rho^4-6\rho^2\tau^2d+\tau^4d^2)\frac{\Theta^4}{4!}+\dots
\\
&=&\exp\left(\frac{\rho}{q}\Theta\right)\sum_{j=0}^{\lfloor\frac{n}{2}\rfloor}(-1)^jq^{n-2j}(\tau^2d)^j\frac{\Theta^{2j}}{(2j)!}
\\
\beta& = & q^{n-1}\Theta +\rho q^{n-2}\Theta^2+q^{n-3}(3\rho^2-d\tau^2)\frac{\Theta^3}{3!}+q^{n-4}\rho(\rho^2-\tau^2d)\frac{\Theta^4}{4!}+\dots
\\
&=&\exp\left(\frac{\rho}{q}\Theta\right)\sum_{j=0}^{\lfloor\frac{n-1}{2}\rfloor}(-1)^jq^{n-1-2j}(\tau^2d)^j\frac{\Theta^{2j+1}}{(2j+1)!}
\end{eqnarray*}
and $\alpha,\beta\in H^{ev}(X,\ZZ)$. Furthermore, $\span_\ZZ\{\alpha,\beta\}$ is saturated in $H^{ev}(X,\ZZ)$ when $q=1$.

The rational $K$-secant $P:=\span_\QQ\{\alpha,\beta\}$ is spanned by Hodge classes, and so a choice of one of the two pure spinors in $\PP(P)$ determines 
the structure of an abelian variety of Weil type for $X\times\hat{X}$, by Lemma \ref{lemma-decomposition-into-4-direct-summands}. If $a\alpha+b\beta$ is the Chern character of an object $F\in D^b(X)$, then
\[
\chi(F^\vee\otimes F)=\int_X(a\alpha+b\beta)^\vee(a\alpha+b\beta)=\left\{
\begin{array}{ccl}
0 & \mbox{if} & n \ \mbox{is odd},
\\
 -2q^2(a^2\tau^2d+b^2) & \mbox{if} & n=2,
 \\
 8dq^4\tau^2(a^2\tau^2d+b^2) & \mbox{if} & n=4,
 \\
(-1)^{\frac{n}{2}}2^{n-1}d^{\frac{n}{2}-1}q^n\tau^{n-2}(a^2\tau^2d+b^2)
& \mbox{if} & n \ \mbox{is even.}
\end{array}
\right.
\]
If $n$ is even, the minimum value of $|\chi(F^\vee\otimes F)|$ above for a given imaginary quadratic number field $K$ 
is $2^{n-1}d^{\frac{n}{2}-1}$, where $d$ is the product of distinct positive primes satisfying $K=\QQ(\sqrt{-d})$, and it is obtained for $q=1$, $\tau=\pm 1$, $\rho\in \ZZ$, $a=0$, and $b=1$, provided the class $\beta$ is the Chern character of an object $F$. In that case, applying a lift of $\iota$ in $\Aut(D^b(X))$, we can get such a representative object with a non-zero rank. When $n=4$ and $\dim\Hom(F,F)=1,$
then $\dim\Ext^2(F,F)=\chi(F^\vee\otimes F)+2\dim\Ext^1(F,F)-2$, and so $\chi(F^\vee\otimes F)$ is relevant for the semi-regularity property of $F$ in the absolute or equivariant sense. 
%****************************************************************
% 
%****************************************************************
\subsection{Examples of secant sheaves}
\label{sec-examples-of-secant-sheaves}

Let $C$ be a non-hyperelliptic curve of genus $3$. 
Set $X:=\Pic^2(C)$ and let $\Theta\subset X$ be the canonical divisor. The natural morphism $C^{(2)}\rightarrow \Theta$ is an isomorphism. Indeed, the morphism is injective, $\Theta$ is smooth, by Riemann's Singularity Theorem \cite[Theorem 11.2.5]{BL}, and so the morphism is an isomorphism by Zariski's Main Theorem. 
Let $AJ:C\rightarrow \Pic^1(X)$ be the Abel-Jacobi morphism.
Given a point $t\in \Pic^1(X)$ denote by $C_t\subset X$ the translate of $AJ(C)$ by $t$. Denote by $[pt]\in H^6(X,\ZZ)$ the class Poincar\'{e}-dual to a point. 
Given a subvariety $Z$ of $X$, denote by $[Z]$ the class in $H^*(X,\ZZ)$ Poncar\'{e}-dual to $Z$.
Then $[\Theta]^3/6=[pt]$ and 
\[
[C_t]=[\Theta]^2/2,
\]
by Poncar\'{e}'s formula \cite[Sec. 11.2]{BL}. We denote $[\Theta]$ by $\Theta$ as well.

Let $d$ be a positive integer. Set $\alpha:=1-\frac{d}{2}\Theta^2$ and $\beta:=\Theta-d[pt]$. We have
\[
\exp(\sqrt{-d}\Theta)=1+\sqrt{-d}\Theta-\frac{d}{2}\Theta^2-d\sqrt{-d}[pt]=
\left(1-\frac{d}{2}\Theta^2\right)+\sqrt{-d}\left(\Theta-d[pt]\right)=\alpha+\sqrt{-d}\beta.
\]
Cup product with $\exp(\sqrt{-d}\Theta)$ is an automorphism of $S_K:=H^*(X,K)$, which belongs to the image of $m:\Spin(V_K)\rightarrow GL(S_K)$,
where $m$ is given in (\ref{eq-m}). Now $1\in H^0(X,\ZZ)$ is an even pure spinor, hence, so is $\exp(\sqrt{-d}\Theta)$.
Let $t_i$, $1\leq i\leq d+1$, be distinct points of $\Pic^1(C)$ and set $C_i:=C_{t_i}$. 

\begin{lem}
\label{lemma-ch-F-i-is-on-secant-to-spinor-variety}
Assume that $C_i$, $1\leq i\leq d+1$, are pairwise disjoint. Then the following equality holds
\[
ch\left(\Ideal{\cup_{i=1}^{d+1}C_i}\otimes\StructureSheaf{X}(\Theta)\right)=
1+\Theta-\frac{d}{2}\Theta^2-d[pt]=\alpha+\beta.
\]
Consequently, $ch\left(\Ideal{\cup_{i=1}^{d+1}C_i}\otimes\StructureSheaf{X}(\Theta)\right)$ 
and $ch\left(R\SheafHom(\Ideal{\cup_{i=1}^{d+1}C_i}\otimes\StructureSheaf{X}(\Theta),\StructureSheaf{X})\right)$ 
both belong to the secant line to the spinor variety through the pure spinor $\exp(\sqrt{-d}\Theta)$ and its complex conjugate $\exp(-\sqrt{-d}\Theta)$.
\end{lem}

\begin{proof}
The equality $ch(\StructureSheaf{C_i})=\Theta^2/2-2[pt]$ holds, since $\chi(\StructureSheaf{C_i})=-2$.
Hence, given $n$ disjoint translates $C_i$ of $AJ(C)$, we get
\[
ch(\Ideal{\cup_{i=1}^nC_i})=1-n\left(\Theta^2/2-2[pt]\right)=1-\frac{n}{2}\Theta^2+2n[pt].
\]
Then 
\begin{eqnarray*}
ch(\Ideal{\cup_{i=1}^nC_i}\otimes\StructureSheaf{X}(k\Theta))&=&
(1-\frac{n}{2}\Theta^2+2n[pt])(1+k\Theta+\frac{k^2}{2}\Theta^2+k^3[pt])
\\
&=&1+k\Theta+\frac{(k^2-n)}{2}\Theta^2+(k^3-3kn+2n)[pt]).
\end{eqnarray*}
Taking $k=1$ and $n=d+1$ we get the desired equality.
\end{proof}

\begin{example}
Let $X$ be an abelian surface, let $F$ be a coherent sheaf with $w:=ch(F)$ satisfying$(w,w)_S<0$, and let
$h\in H^{ev}(X,\ZZ)$ be an algebraic class, such that $(h,w)_S=0$ and  $(h,h)_S<0$. Then $\span_\QQ\{w,h\}$ is a secant to the spinor variety inducing complex multiplication by the imaginary quadratic number field $\QQ(\sqrt{-d})$, where $d=\frac{(w,w)(h,h)}{4}$, by \cite[Prop. 1.7]{markman-generalized-kummers}. In particular, $F$ is a $\QQ(\sqrt{-d})$-secant sheaf.
\end{example}

\begin{example}
\label{example-secant-sheaf-on-jacobian-of-genus-4-curve}
When $X$ is an abelian $n$-fold, $n\geq 4$, Lemma \ref{lemma-ch-F-i-is-on-secant-to-spinor-variety} can be generalized to produce secant sheaves with a secant line inducing complex multiplication by $\QQ(\sqrt{-d})$ as follows. Let $C$ be a Brill-Noether generic curve of genus $n$ and set $X:=\Pic^{n-1}(C)$. Denote by $W_k$ the Abel-Jacobi image of $C^{(k)}$ in $\Pic^k(C)$ and by $[W_k]$ 
 the class of any translate of $W_k$ in $X$. Then $[W_k]=\Theta^{n-k}/(n-k)!$, by Poincar\'{e}'s formula. Let $t_{jk}$ be generic points in $\Pic^{n-k-1}(C)$.
Let the subscheme $Z$ of $X$ be the union 
\begin{equation}
\label{eq-Z-is-a-union-of-Brill-Noether-loci}
Z:=\cup_{k=0}^{n-2} \cup_{j=1}^{a_k}\tau_{t_{jk}}(W_k).
\end{equation}
Let $F:=\Ideal{Z}(a_{n-1}\Theta)$ be the tensor product of the ideal sheaf of $Z$ with $\StructureSheaf{X}(a_{n-1}\Theta)$. 
Let $\alpha$ be the real part of $\exp(\sqrt{-d}\Theta)$ and $\sqrt{d}\beta$ its imaginary part.
\begin{eqnarray*}
\alpha &:=& 1-d[W_{n-2}]+d^2[W_{n-4}] + \cdots + (-d)^{n/2}[pt]
\\
\beta & := & \Theta -d[W_{n-3}]+\cdots
\end{eqnarray*}
The integers $a_k$, $0\leq k\leq n-1$, in (\ref{eq-Z-is-a-union-of-Brill-Noether-loci}) should be chosen to satisfy the equation 
\begin{equation}
\label{eq-coefficients-of-a-secant-sheaf}
ch(F)=\alpha+a_{n-1}\beta.
\end{equation}
%So $a_{n-2}=a_{n-1}^2+d$.
For example, when $n=4$ we may assume that only the irreducible components of $Z$ in (\ref{eq-Z-is-a-union-of-Brill-Noether-loci}), which are translates of $W_2$, intersect and every such pair interests at $\Theta^4/4=6$ points. Note that $\chi(\StructureSheaf{W_1})=\chi(\StructureSheaf{C})=-3$ and 
$\chi(\StructureSheaf{W_2})=3$. For the latter equality use that $C$ is non-hyperelliptic to conclude that $W_2$ is isomorphic to $C^{(2)}$, as well as the equalities $h^{0,1}(C^{(2)})=4$ and $h^{0,2}(C^{(2)})=6$. We get  
\begin{eqnarray*}
ch(\StructureSheaf{W_1})&=& \Theta^3/6-3[pt],
\\
ch(\StructureSheaf{W_2})&=&\Theta^2/2 -\Theta^3/3+3[pt] \ \ \mbox{(see Lemma \ref{lemma-ch-O-W-2})},
\\
ch(\StructureSheaf{Z})&=&\sum_{k=0}^{2}a_k ch(\StructureSheaf{W_k})-6\Choose{a_2}{2}[pt]
\\
ch(\Ideal{Z})&=& 1-\frac{a_2}{2}\Theta^2+\left[\frac{2a_2-a_1}{6}\right]\Theta^3+
\left[
6\Choose{a_2}{2}-3a_2+3a_1-a_0
\right][pt],
\\
ch(\Ideal{Z}(a_3\Theta)&=&
1+a_3\Theta+
\left[
\frac{a_3^2-a_2}{2}
\right]\Theta^2+
\left[
\frac{a_3^3-3a_3a_2+2a_2-a_1}{6}
\right]\Theta^3+
\\
&&
\left[
a_3^4-6a_2a_3^2+8a_2a_3-4a_1a_3+6\Choose{a_2}{2}-3a_2+3a_1-a_0
\right][pt]
%\frac{a_k}{(n-k)!}\Theta^{n-k}
\\
\alpha+a_3\beta&=&
1+a_3\Theta-\frac{d}{2}\Theta^2-\frac{da_3}{6}\Theta^3+d^2[pt]
\end{eqnarray*}
Comparing coefficients in Equation \ref{eq-coefficients-of-a-secant-sheaf}
we get $a_2=d+a_3^2$, $a_1=2(d+a_3^2)(1-a_3)$, 
\[
a_0=6a_3^4-6a_3^3+8da_3^2-6da_3+2d^2=(a_3^2+d)(6a_3^2-6a_3+2d).
\]
We get a secant ideal sheaf tensored with $\StructureSheaf{X}(a_3\Theta)$ for every choice of an integer $a_3\leq 1$.
If we choose $a_3=1$, then $a_2=d+1$, $a_1=0$, and $a_0=2d(d+1)$. 
Once again we can choose $Z$ to be invariant with respect to a cyclic subgroup $G$ of $X$ of order $d+1$. 
However, $F_d:=\Ideal{Z}(\Theta)$ is unlikely to be semiregular, and 
the semiregularity map could restrict to $\Ext^2(F_d,F_d)^{G}$ as an injective map for at most finitely many values of $d$, 
since $\dim\Ext^2(F_d,F_d)=\int_X ch(F_d)ch(F_d^\vee)+2\dim\Ext^1(F_d,F_d)-2=8d(d+1)-2+2\dim\Ext^1(F_d,F_d)$ 
grows quadratically
%\footnote{If $ch(F)=c(\alpha+t\beta)$, for some integer $c$, then 
%$\chi(F\otimes F^\vee)=c^2\int_X\alpha^2-t^2\beta^2$
%} 
with $d$, while the order of $G$ is linear in $d$.
%$\Ext^1(\Ideal{Z},\Ideal{Z})^{G}$ will contain a copy of $H^0(TX)$ as a direct summand for each of the $2d$ $G$-orbits of the zero dimensional 
%connected components of $Z$ as well as one for the $G$-orbit of the $2$-dimensional irreducible components and a copy of $H^2(\StructureSheaf{X})$. 
Another choice of a $K$-secant object with the same Chern character is the following. 
Let $\tau_{t_0}(W_2)$ be a translate of $W_2$ in $X$ and set $Z:=\cup_{t\in G}\tau_{t_0+t}(W_2)$. 
Let $g$ be a generator of $G$ and set $t_i:=t_0+ig$, $0\leq i\leq d$.
We choose $G$ so that 
for each pair $(i,j)$, with $0\leq i<j\leq d$, the intersection $\tau_{t_i}(W_2)\cap\tau_{t_j}(W_2)$ consists of $6$ distinct points, and this collection of $\Choose{d+1}{2}$ sets of $6$ points are pairwise disjoint. For each such pair $(i,j)$ 
choose $4$ of the $6$ points of intersection $\tau_{t_i}(W_2)\cap \tau_{t_j}(W_2)$ and let $\nu:\tilde{Z}\rightarrow Z\subset X$ be the partial normalization of $Z$ resolving the self intersection at each of the $4$ chosen points for each pair $(i,j)$. Let $F'$ be the object
\begin{equation}
\label{eq-F-1-over-jacobian-of-genus-4-curve}
\StructureSheaf{X}\RightArrowOf{\nu^*} \nu_*\StructureSheaf{\tilde{Z}}.
\end{equation}
We choose the collection of $\Choose{d+1}{2}$ sets of $4$ points to be $G$-invariant. Note that $ch(F'(\Theta))=ch(F_d))$. 
Let $q:X\rightarrow Y:=X/G$ be the quotient map and set $\bar{Z}:=q(\tau_{t_0}(W_2))$. 
It admits the partial normalization $\bar{\nu}:\tilde{Z}/G\rightarrow \bar{Z}$.
Then $Z=q^{-1}(\bar{Z}))$ is an etal\'{e} $G$-cover of $\bar{Z}$ and $F'$ is the pullback via $q$ of the object $\bar{F}':=\left[\StructureSheaf{Y}\rightarrow \bar{\nu}_*(\StructureSheaf{\tilde{Z}/G)})\right]$. 
\end{example}

%****************
% Hide
%****************
\hide{
\begin{question}
Is the restriction of the semiregularity map to $\Ext^2(F',F')^{G_1}$ injective? Equivalently, is the object $\bar{F}'$ semiregular?
\end{question}

The analogous question in the genus $3$ case for the sheaf $\Ideal{\cup_{i=1}^{d+1}C_i}$ of Lemma \ref{lemma-ch-F-i-is-on-secant-to-spinor-variety} has an affirmative answer (see Lemma \ref{lemma-Yoneda-algebra-is-generated-in-degree-1} and Corollary \ref{cor-kernel-of-obstruction-is-annihilator-of-ch}, which imply the semiregularity via Lemma \ref{lemma-criterion-for-semiregularity}).
%****************
% End Hide
%****************
}

\begin{example}
Keep the notation of Example \ref{example-secant-sheaf-on-jacobian-of-genus-4-curve}, let $C$ be a non-hyperelliptic curve of genus $4$ with two $g^1_3$'s, $X=\Pic^3(C)$
let $e:\Theta\subset X$ be the theta divisor $W_3$, and let $W_{2,p}=\{p+p_1+p_2 \ ; \ p_1, p_2\in C\}$ be the translate of $W_2\subset \Pic^2(C)$ by the point $p$ into $X$. Then $W_{2,p}$ is contained in $\Theta$ and passes through the singularities of $\Theta$ (the two $g^1_3$'s) and is thus a Weil divisor\footnote{
Let $\iota:X\rightarrow X$ be the involution sending $L$ to $\omega_C\otimes L^{-1}$. Given two distinct points $p,p'\in C$, we have the scheme theoretic equality $W_{2,p}\cup \iota(W_{2,p'})=\Theta\cap \tau_{p'-p}(\Theta)$.
So $W_{2,p}+ \iota(W_{2,p'})$ is a cartier divisor in the linear system of $e^*\StructureSheaf{X}(\tau_{p'-p}(\Theta))$.
} 
in $\Theta$. Let $\Ideal{W_{2,p}}$ be the ideal sheaf of $W_{2,p}$ as a subvariety of $\Theta$. Choose $d$ unordered pairs of points $\{p_i,q_i\}_{i=1}^d$ in $C$, with $\{p_i,q_i\}\cap\{p_j,q_j\}=\emptyset$, for $i\neq j$. Assume that the line secant to the canonical embedding of $C$ through $p_i$ and $q_i$ does not lie on the unique quadric containing the canonical curve. 
Let $C_i=\{p+p_i+q_i\}\subset \Theta$ be 
the translate $W_1:=AJ(C)\subset \Pic^1(C)$ into $\Theta$. Then $C_i$ is contained in the smooth locus of $\Theta$. Assume that $p\not\in\{p_i,q_i\}$, for all $1\leq i\leq d$. Then $C_i\cap W_{2,p}=\{p+p_i+q_j\}$ with multiplicity\footnote{
Indeed, the scheme theoretic intersection of $C_i$ with $W_{2,p}\cup \iota(W_{2,p'})$ is equal to the scheme theoretic intersection of $C_i$ with $\tau_{p'-p}(\Theta)$, which has degree $4$, and $C_i\cap \iota(W_{2,p'})$ is a divisor in the linear system of $\omega_C(-p_i-q_i-p')$ and so of degree $3$.
}
$1$.
Set $Z_d:=W_{2,p}\cup \cup_{i=1}^d C_i$, considered as a subscheme of $\Theta$. Then $e_*\Ideal{Z_d}(\Theta)$ is a $\QQ(\sqrt{-d})$-secant sheaf. Indeed, we have:
\begin{eqnarray*}
ch(e_*(\StructureSheaf{\Theta}))&=& 1-ch(\StructureSheaf{X}(\Theta))=\Theta-\frac{\Theta^2}{2}+\frac{\Theta^3}{6}-\frac{\Theta^4}{24},
\\
ch(e_*(\StructureSheaf{\Theta}(\Theta)))&=& \Theta+\frac{\Theta^2}{2}+\frac{\Theta^3}{6}+\frac{\Theta^4}{24},
\\
ch(\StructureSheaf{W_2})&=& \frac{\Theta^2}{2}-\frac{\Theta^3}{3}+\frac{\Theta^4}{8} \ \mbox{(see Example \ref{example-secant-sheaf-on-jacobian-of-genus-4-curve})},
\\
ch(\StructureSheaf{W_2}(\Theta))&=& \frac{\Theta^2}{2}+\frac{\Theta^3}{6}+\frac{\Theta^4}{24},
\\
ch(e_*(\Ideal{W_{2,p}}(\Theta)))&=&\Theta,
\\
ch(\StructureSheaf{W_1})&=&\frac{\Theta^3}{6}-\frac{\Theta^4}{8}  \ \mbox{(see Example \ref{example-secant-sheaf-on-jacobian-of-genus-4-curve})},
\\
ch(\StructureSheaf{W_1}(\Theta))&=&\frac{\Theta^3}{6}+\frac{\Theta^4}{24}=[W_1]+[pt],
\\
ch(e_*\Ideal{Z_d}(\Theta))&=& \Theta-d\frac{\Theta^3}{6}=\beta.
\end{eqnarray*}
Set $F_d:=e_*\Ideal{Z_d}(\Theta)$.
Note that $\dim\Ext^2(F_d,F_d)=-\int_X\beta^2+2\dim\Ext^1(F_d,F_d)-2=8d+2\dim\Ext^1(F_d,F_d)-2$ grows linearly with $d$, and hence it can be semi-regular for at most finitely many values of $d$.
\end{example}

\begin{lem}
\label{lemma-ch-O-W-2}
We have 
$ch(\StructureSheaf{W_2})=\Theta^2/2 -\Theta^3/3+3[pt]$ in  the notation of Example \ref{example-secant-sheaf-on-jacobian-of-genus-4-curve}. 
\end{lem}

\begin{proof}
We have $ch(\StructureSheaf{W_2})=\Theta^2/2 +\lambda\Theta^3+3[pt]$, for some rational number $\lambda$ satisfying
\begin{equation}
\label{eq-Euler-characteristic-of-structure-sheaf-of-W-2-minus-Theta}
\chi(\StructureSheaf{W_2}(-\Theta))=\int_X\left(\Theta^2/2 +\lambda\Theta^3+3[pt]\right)
\left(1-\Theta+\Theta^2/2-\Theta^3/6+\Theta^4/24\right)=9-24\lambda,
\end{equation}
since $\dim H^{3,3}(X,\QQ)$ is the rank of the Neron-Severi group of $X$, which is $1$ for generic $C$.
Assume that $C$ is the intersection of a smooth quadric $Q$ and a cubic in $\PP^3$. So $C$ has two $g^1_3$'s, associated to the two rulings of $Q$.
Let $q_1+q_2+q_3$ be a reduced divisor in one $g^1_3$, denote by $\LB:=\omega_C(-q_1-q_2-q_3)$ the other $g^1_3$, and let $p\in C\setminus\{q_1,q_2,q_3\}$ be a point, such that $\linsys{\LB(-p)}=\{a+b\}$ for distinct points $a, b \in C\setminus\{p\}$. 
We calculate the intersection $W_2\cap \tau_{p-q_1-q_2}(\Theta)$. 
The intersection has cohomology class $\Theta^3/2$. It consists of the union of three irreducible components of class $\Theta^3/6$ each.
As a subset of $X=\Pic^2(C)$, the intersection consists of classes of effective divisors $D$ on $C$ of degree $2$, such that $D+q_1+q_2-p$ is effective. 
One irreducible component is $C_1:=\tau_p(AJ(C))$.

The complement of $C_1$ in the intersection consist of divisors $D$, such that $h^0(\StructureSheaf{C}(D+q_1+q_2)=2$ (the inequality $h^0(\StructureSheaf{C}(D+q_1+q_2)\leq2$ follows from the assumption that $C$ is not hyperelliptic). Such divisors $D$ satisfy 
\[
\omega_C(-D')\cong \StructureSheaf{C}(D+q_1+q_2),
\]
for some effective divisor $D'$. So $D+D'$ belongs to $\linsys{\omega_C(-q_1-q_2)}$. 
If $q_3$ is in the support of $D$, then $D=q_3+q$, where $q$ is any point of $C$, since $h^0(\StructureSheaf{C}(q_1+q_2+q_3))=2$.
Hence, a second irreducible component is $C_2:=\tau_{q_3}(AJ(C))$.
The third irreducible component $C_3$ consists of $D$, such that $D+q\in \linsys{\LB}$, for some $q\in C$. So $C_3=\tau_{\LB}(-AJ(C))$,
where $-AJ(C)$ is a curve in $\Pic^{-1}(C)$. 

The curves $C_1$ and $C_2$ intersect at the point corresponding to the divisor $p+q_3$. The curves $C_1$ and $C_3$ intersect at the two points
$\{p+a, p+b\}$. 
Let $P$ be the plane tangent to $Q$ at $q_3$. Then $P\cap Q$ consists of two lines through $q_3$ and $P\cap C=q_1+q_2+2q_3+a'+b'$, for some points $a', b'$ of $C$, which we may assume distinct, possibly by changing the choice of the divisor $q_1+q_2+q_3$. 
The curves $C_2$ and $C_3$ intersect at the two points $\{q_3+a', q_3+b'\}$.
We conclude that
\[
\chi(\StructureSheaf{W_2\cap \tau_{p-q_1-q_2}(\Theta)})=\chi(\StructureSheaf{C_1})+\chi(\StructureSheaf{C_2})+\chi(\StructureSheaf{C_3})-5=-14.
\]
\[
\chi(\StructureSheaf{W_2}(-\tau_{p-q_1-q_2}(\Theta)))=\chi(\StructureSheaf{W_2})-\chi(\StructureSheaf{W_2\cap \tau_{p-q_1-q_2}(\Theta)})=17.
\]
Comparing with (\ref{eq-Euler-characteristic-of-structure-sheaf-of-W-2-minus-Theta}) we get $9-24\lambda=17$, so $\lambda=-1/3$.
\end{proof}
%****************************************************************
% 
%****************************************************************
\subsection{Secant sheaves on abelian threefolds  with a rank $6$ obstruction map}
\label{sec-rank-6-obstruction-map}
Keep the notation of Lemma \ref{lemma-ch-F-i-is-on-secant-to-spinor-variety}.
Set $F_1:=\Ideal{\cup_{i=1}^{d+1}C_i}\otimes\StructureSheaf{X}(\Theta)$. Let $P:=\span\{\alpha,\beta\}$ be the rational $\QQ(\sqrt{-d})$-secant plane.
Assumption \ref{assumption-on-rational-secant-plane-P} is satisfied, since $(\alpha,\beta)_S=\int_X\tau(\alpha)\cup\beta=\int_X\alpha\cup\beta=-4d\neq 0.$
Let $h$ be an ample class in the rank $1$ subgroup $H^2(X\times\hat{X},\ZZ)^{\Spin(V)_P}$. Such an ample class $h$ exists, by Proposition \ref{prop-polarized-abelian-variety-of-Weil-type}.

\begin{lem}
\label{lemma-kappa-3-is-linearly-independent-from-h-cube}
The rank of $\Phi(F_1\boxtimes F_1)$ is non-zero.
The $\Spin(V)_P$-invariant classes $h^3$ and $\kappa_3(\Phi(F_1\boxtimes F_1))$ are linearly independent.
\end{lem}

\begin{proof}
We use the notation of Proposition \ref{prop-the-orlov-image-of-HW-P-projects-into-the-3-dimensional-space-of-HW-classes}.
We have 
\[
ch(\Phi(F_1\boxtimes F_1))=
\phi(ch(F_1)\boxtimes ch(F_1))=\phi'(ch(F_1)\boxtimes \tau(ch(F_1))).
\]
Set $\lambda_1:=\exp(\sqrt{-d}\Theta)$ and $\lambda_2=\bar{\lambda}_1$, so that 
\[
ch(F_1)=\frac{1}{2}\left[
(\lambda_1+\lambda_2)+\frac{1}{\sqrt{-d}}(\lambda_1-\lambda_2)
\right]=
\frac{1}{2\sqrt{-d}}\left[(1+\sqrt{-d})\lambda_1+(-1+\sqrt{-d})\lambda_2\right]
\]
Now, $\tau$ interchanges $\lambda_1$ and $\lambda_2$. Hence,
\begin{eqnarray*}
ch(F_1)\!\boxtimes\! \tau(ch(F_1))
\!&=&\!
\frac{-1}{4d}\!\left[
[
(1\!\!+\!\!\sqrt{-d})\lambda_1\!+\!(-1\!\!+\!\!\sqrt{-d})\lambda_2
]
\boxtimes
[
(-1\!\!+\!\!\sqrt{-d})\lambda_1\!+\!(1\!\!+\!\!\sqrt{-d})\lambda_2
]
\right]
\\
\!&=&\!
\frac{d\!+\!1}{4d}[\lambda_1\!\!\boxtimes\!\!\lambda_1\!+\!\lambda_2\!\!\boxtimes\!\!\lambda_2]
\!+\!
\frac{d\!-\!1}{4d}[\lambda_1\!\!\boxtimes\!\!\lambda_2\!+\!\lambda_2\!\!\boxtimes\!\!\lambda_1]
\!+\!
\frac{\sqrt{-d}}{2d}[\lambda_2\!\!\boxtimes\!\!\lambda_1\!-\!\lambda_1\!\!\boxtimes\!\!\lambda_2]
\end{eqnarray*}

The rank of $\Phi(F_1\boxtimes F_1)$ is non-zero, by Proposition \ref{prop-the-orlov-image-of-HW-P-projects-into-the-3-dimensional-space-of-HW-classes}(\ref{prop-item-weight-2-subrepresentation}),
since the coefficient of $[\lambda_2\!\!\boxtimes\!\!\lambda_1\!-\!\lambda_1\!\!\boxtimes\!\!\lambda_2]$ is non-zero. 
We prove that the classes 
$\kappa(\Phi(F_1\boxtimes F_1))$ and $h^3$ are linearly independent, by contradiction. Assume otherwise. Then $\kappa(\Phi(F_1\boxtimes F_1))$
belongs to the subring generated by powers of $h$ and is thus $\Spin(V_K)_{\ell_1,\ell_2}$-$\rho$-invariant, by Lemma \ref{lemma-Spin-V-P-invariant-classes-are-Hodge}. Hence, $ch(\Phi(F_1\boxtimes F_1))$ is $\Spin(V_K)_{\ell_1,\ell_2}$-$\rho'$-invariant, by Lemma \ref{lemma-Spin-V-ell-1-ell-2-invariance-conditions}.
It follows that $ch(F_1)\boxtimes\tau(ch(F_1))$ is invariant under $\Spin(V_K)_{\ell_1,\ell_2}$, by the $\Spin(V)$-$\rho'$-equivariance of $\phi\circ(id\otimes\tau).$
But the last two summands 
$\lambda_1\!\!\boxtimes\!\!\lambda_2\!+\!\lambda_2\!\!\boxtimes\!\!\lambda_1$ and 
$\lambda_2\!\!\boxtimes\!\!\lambda_1\!-\!\lambda_1\!\!\boxtimes\!\!\lambda_2$
in the displayed formula above are $\Spin(V_K)_{\ell_1,\ell_2}$-invariant, while the first summand is a scalar multiple of $\lambda_1\!\!\boxtimes\!\!\lambda_1\!+\!\lambda_2\!\!\boxtimes\!\!\lambda_2$ and is a not $\Spin(V_K)_{\ell_1,\ell_2}$-invariant, by Lemma \ref{lemma-Spin-V-P-invariant-classes-are-Hodge} and Proposition \ref{prop-the-orlov-image-of-HW-P-projects-into-the-3-dimensional-space-of-HW-classes}(\ref{prop-item-HW-P-maps-to-a-weight-2n-subrepresentation}).
Hence, $ch(F_1)\boxtimes\tau(ch(F_1))$ is not $\Spin(V_K)_{\ell_1,\ell_2}$-invariant, 
since the coefficient of $\lambda_1\!\!\boxtimes\!\!\lambda_1\!+\!\lambda_2\!\!\boxtimes\!\!\lambda_2$ is non zero. A contradiction.
\end{proof}

Set $F:=\Ideal{\cup_{i=1}^nC_i}$. 
Let $U\subset (\Pic^0(C))^n$ be the Zariski open subset of points $(\ell_1, \dots, \ell_n)$, such that translating each $C_i$ by $\ell_i$ results in $n$ disjoint curves. We have a natural morphism from $U$ to the Hilbert scheme of $X$.
Explicitly, let $Z_i$ be the product $X^{i-1}\times C_i\times X^{n-i}$, for $1\leq i\leq n$. Then $F$ is the pullback of the ideal of the union $\cup_{i=1}^nZ_i$ via trhe diagonal embedding of $X$ in $X^n$. 
$\Pic^0(C)^n$ acts on $X^n$ introducing the desired map from the set $U$ to the Hilbert scheme of $X$.
The symmetric group $\fS_n$ acts freely on $U$ as follows. If $(t_1, t_2, \dots, t_n)\in \Pic^1(C)^n$ translates $C\times \cdots \times C$ to $C_1\times C_2 \times \cdots \times C_n$, then $U$ is invariant with respect to the action on $\Pic^0(X)^n$ which is conjugated by translation by $(t_1, t_2, \dots, t_n)$ to the permutation action on $\Pic^1(X)^n$ and the action on $U$ is fixed point free.
The connected component of $F$ in the moduli space of simple sheaves on $X$ contains a smooth subscheme isomorphic to $(U/\fS_n)\times\Pic^0(X)$ obtained by translating the $C_i$'s and tensoring $F$ by line bundles. 
Hence, we have a canonical injective homomorphism
\begin{equation}
\label{eq-injective-homomorphism-to-Ext-1}
H^0(X,TX)^n\oplus H^1(X,\StructureSheaf{X})\rightarrow \Ext^1(F,F)
\end{equation}
and 
$\dim \Ext^1(F,F)\geq 3n+3$. 

\begin{lem}
\label{lemma-dim-Ext-1-F-F-is-3n+3}
The homomorphism (\ref{eq-injective-homomorphism-to-Ext-1}) is an isomorphism. Consequently, $\dim Ext^1(F,F)=3n+3$.
\end{lem}

\begin{proof}
It suffices to prove the inequality $\dim \Ext^1(F,F)\leq 3n+3$.
Set $\I_k:=\Ideal{\cup_{i=1}^kC_i}$. We have the short exact sequence
\begin{equation}
\label{eq-short-exact-sequence-of-F-n}
0\rightarrow \I_n\rightarrow \I_{n-1}\rightarrow \StructureSheaf{C_n}\rightarrow 0
\end{equation}
and the long exact
\[
0\rightarrow \Hom(\I_n,\I_n) \IsomRightArrow \Hom(\I_n,\I_{n-1})\RightArrowOf{0} \Hom(\I_n,\StructureSheaf{C_n})\rightarrow 
\Ext^1(\I_n,\I_n)\rightarrow \Ext^1(\I_n,\I_{n-1})
\]
We have the isomorphism $\Hom(\I_n,\StructureSheaf{C_n})\cong H^0(C_n,N_{C_n/X})$. 
The quotient of $H^0(C_n,N_{C_n/X})$ by $H^0(C_n,\restricted{TX}{C_n})$ is the kernel of the differential  of  the Torelli map 
$H^1(C_n,TC_n)\rightarrow  H^1(C_n,\restricted{TX}{C_n})\cong H^{0,1}(C_n)\otimes H^{0,1}(C_n)$, and the latter is injective for our non-hyperelliptic curve $C_n$, by  Noether's theorem. Hence, $H^0(C_n,N_{C_n/X})$ is  $3$-dimensional. It remains to prove the inequality
$\dim \Ext^1(\I_n,\I_{n-1})\leq 3n$. We will prove by induction on $n$ the equality
$\dim \Ext^1(\I_n,\I_{n-1})= 3n$.

When $n=1$, $\Ext^1(\I_n,\I_{n-1})\cong H^2(\I_1)^*\cong H^2(X,\Ideal{C_1})^*$ is $3$-dimensional. Indeed, we have more generally the short exact sequence
\[
0\rightarrow \frac{\oplus_{i=1}^n H^1(\StructureSheaf{C_i})}{H^1(\StructureSheaf{X})}
\rightarrow H^2(\I_n)\rightarrow H^2(\StructureSheaf{X})\rightarrow 0,
\]
obtained from the long exact sequence of cohomology associated to the short exact $0\rightarrow \I_n\rightarrow \StructureSheaf{X}\rightarrow \oplus_{i=1}^n\StructureSheaf{C_i}\rightarrow 0$
using the injectivity of $H^1(\StructureSheaf{X})\rightarrow \oplus_{i=1}^nH^1(C_i,\StructureSheaf{C_i})$.

Assume that $n\geq 2$  and  $\dim \Ext^1(\I_{n-1},\I_{n-2})= 3(n-1)$. Then $\dim  \Ext^1(\I_{n-1},\I_{n-1})=3n$, as shown above. 
Consider long exact sequence obtained from the short exact sequence (\ref{eq-short-exact-sequence-of-F-n})
\begin{eqnarray*}
0 \rightarrow & \Hom(\I_{n-1},\I_{n-1})& \IsomRightArrow \Hom(\I_{n-1},\StructureSheaf{C_n})
\\
\RightArrowOf{0} \Ext^1(\I_{n-1},\I_n)\rightarrow & \Ext^1(\I_{n-1},\I_{n-1})& \rightarrow \Ext^1(\I_{n-1},\StructureSheaf{C_n})
\\
\RightArrowOf{\xi} \Ext^2(\I_{n-1},\I_n)\rightarrow & \Ext^2(\I_{n-1},\I_{n-1})& \rightarrow 0. 
\end{eqnarray*}
The sheaves $\SheafExt^i(\I_{n-1},\StructureSheaf{C_n})$ vanish, for $i>0$. Hence, the local to global spectral sequence computing $\Ext^i(\I_{n-1},\StructureSheaf{C_n})$ degenerates at the $E_2$ term. We get the 
vanishing of $\Ext^i(\I_{n-1},\StructureSheaf{C_n})$, for $i>1$ and the 
isomorphism 
$\Ext^1(\I_{n-1},\StructureSheaf{C_n})\cong H^1(\StructureSheaf{C_n})$ and the latter is $3$-dimensional. 
Furthermore, the composition 
\[
H^1(\StructureSheaf{X})\cong H^1(\SheafEnd(\I_{n-1},\I_{n-1}))\rightarrow \Ext^1(\I_{n-1},\I_{n-1})\rightarrow \Ext^1(\I_{n-1},\StructureSheaf{C_n})\cong 
H^1(\StructureSheaf{C_n})
\]
is surjective. Hence, the connecting homomorphism $\xi$ vanishes and
\[
\dim\Ext^1(\I_n,\I_{n-1})=\dim\Ext^2(\I_{n-1},\I_n)=\dim\Ext^2(\I_{n-1},\I_{n-1})=3n,
\]
where the first equality is by Serre's duality, the second by the vanishing of $\xi$, and the third was established above via the induction hypothesis.
%Let $r=\rank(\xi)$.
%We have
%\begin{eqnarray*}
%\dim\Ext^2(\I_{n-1},\I_n)&=&\dim\Ext^2(\I_{n-1},\I_{n-1}))+r=3n+r,
%\\
%\dim\Ext^1(\I_{n-1},\I_n)&=&\dim\Ext^1(\I_{n-1},\I_{n-1}))+r-3=3n+r-3
%\end{eqnarray*}
%Hence, $\sum_{i=0}^3\dim\Ext^i(\I_{n-1},\I_n)=$.
\end{proof}

Let $\Delta_X\subset X\times X$ be the diagonal. 
A class in the Hochschild cohomology $HH^i(X):=\Hom(\StructureSheaf{\Delta_X},\StructureSheaf{\Delta_X}[i])$ corresponds to a natural transformation from the identity endofunctor of $D^b(X)$ to its shift by $i$. 
The evaluation homomorphism
\begin{equation}
\label{eq-ev-F}
ev_F:HH^*(X)\rightarrow \Ext^*(F,F)
\end{equation}
is a graded algebra homomorphism. Denote by $ev_F^i:HH^i(X)\rightarrow \Ext^i(F,F)$ the restriction of $ev_F$ to $HH^i(X)$.

The second Hochschild cohomology $HH^2(X)$ parametrizes first order deformations of $D^b(X)$.
Let 
\[
ob_F: HH^2(X)\rightarrow \Ext^2(F,F)
\]
be the homomorphism $ev_F^2$. 
%evaluating on $F$ a natural transformation $\alpha:id\rightarrow id[2]$ from the identity endofunctor $id$ of $D^b(X)$ to its shift.  
The kernel of $ob_F$  parametrizes those deformations along which $F$ deforms to first order, by \cite{toda}. As above, we set $F:=\Ideal{\cup_{i=1}^n C_i}$, where $n\geq 1$. Denote by $ob_F:HT^2(X)\rightarrow \Ext^2(F,F)$ also the composition of $ob_F$ above with the HKR isomorphism $HT^2(X):=H^2(\StructureSheaf{X})\oplus H^1(TX)\oplus H^0(\wedge^2TX)\cong HH^2(X)$. Then $ob_F$ is given by contraction with the exponential Atiyah class $\exp(at_E)$, by \cite[Th. A]{Huang}.

\begin{lem}
\label{lemma-rank-of-ob-F-is-at-least-6}
$\rank(ob_F)\geq 6.$
\end{lem}

\begin{proof}
Consider the contraction homomorphism
\[
H^2(\StructureSheaf{X})\oplus H^1(TX)\oplus H^0(\wedge^2TX)
\LongRightArrowOf{\Contract ch(F)}
H^2(\StructureSheaf{X})\oplus H^3(\Omega^1_X).
\]
It restricts to the first summand as an isomorphism onto the first summand of the co-domain, as $ch_0(F)=1$, and to 
the second summand as an isomorphism onto the second summand of the co-domain, as $ch_2(F)=\frac{-n}{2}\Theta^2$.
Hence, the above homomorphism is sujective and so its kernel has co-dimension $6$. The kernel of $ob_F$ is contained in the kernel of the above homomorphism, by \cite[Theorem B]{Huang}, hence $ob_F$ has rank $\geq 6$.
\end{proof}

Diagram (\ref{eq-diagram-semi-regularity}) extends to the commutative diagram
\begin{equation}
\label{eq-extended-diagram-semi-regularity}
\xymatrix{
HT^2(M) \ar[rr]^{ob_E} \ar[dr]_{\Contract ch(E)}&&\Ext^2(E,E)\ar[dl]^{\sigma}
\\
&
\prod_{q= 0}^{d-2}H^{q+2}(M,\Omega^q_M)
}
\end{equation}
by \cite[Prop. 6.2.1 and Cor. 6.3.2]{buchweitz-flenner-HH}.
%***************
% Hide
%***************
\hide{
The commutativity of (\ref{eq-extended-diagram-semi-regularity}) is proved by restricting to each of the three direct summands of $HT^2(M)=H^2(\StructureSheaf{M})\oplus H^1(TM)\oplus H^0(\wedge^2TM)$.
The commutativity of the restrictions to $H^2(\StructureSheaf{M})$ follows from the equality $ch(E)=Tr(at_E)$ and the equality
$Tr(at_E)\cup\alpha=Tr(at_E\circ (\alpha\cdot id_E))$, for all $\alpha\in H^2(\StructureSheaf{M})$. The restriction to $H^1(TM)$
is the commutative diagram (\ref{eq-diagram-semi-regularity}). The commutativity of the restriction to $H^0(\wedge^2TM)$
is equivalent to the equality
\[
(k+2)(k+1)Tr\left(\langle\xi,at_E^2\rangle\cdot at_E^k\right)=
2\langle\xi,Tr\left(at_E^{k+2}\right)\rangle,
\]
for all $\xi\in H^0(\wedge^2TM)$. When $M$ is an abelian variety, then it suffices to prove the latter equality for 
$\xi=\xi_1\wedge\xi_2$, with $\xi_i\in H^0(TM)$. In that case the above identity follows from a repeated use of the equality
\[
(j+1)Tr\left(\langle\xi,at_E\rangle\cdot at_E^j\right)=
\langle\xi,Tr\left(at_E^{j+1}\right)\rangle,
\]
where the latter is proved by the same argument as in \cite[Prop. 4.2]{buchweitz-flenner}. The general case follows from 
a computation using a \v{C}ech cocycle representing $at_E$ as in \cite[Sec. 10.1.5]{huybrechts-lehn} in order to carry the above argument locally. 
%***************
% End Hide
%***************
}

\begin{lem}
\label{lemma-criterion-for-semiregularity}
If the kernels of $ob_E$ and $\Contract ch(E)$ in $HT^2(M)$ are equal  and $ob_E$ is surjective, then $E$ is semiregular.
\end{lem}

\begin{proof}
The hypotheses imply that there exists a unique injective map $\sigma':\Ext^2(E,E)\rightarrow \prod_{q= 0}^{d-2}H^{q+2}(M,\Omega^q_M)$, such that 
$\Contract ch(E)=\sigma'\circ ob_E.$
The equality $\sigma=\sigma'$ follows from the commutativity of Diagram (\ref{eq-extended-diagram-semi-regularity}). 
\end{proof}

\begin{rem}
\label{remark-semiregularity-map-restricts-as-an-injective-map-to-the-image-of-ob}
If we drop the  assumption that $ob_E$ is surjective in Lemma \ref{lemma-criterion-for-semiregularity} we still conclude that the semiregularity map restricts to the image of the obstruction map as an injective map.
\end{rem}

\begin{rem}
\label{rem-derived-equivariance-of-semiregularity}
Note that the hypotheses of Lemma \ref{lemma-criterion-for-semiregularity}
are invariant under derived equivalences. If $\Phi:D^b(M)\rightarrow D^b(M')$ is an equivalence of derived categories, $E$ satisfies the hypotheses of Lemma \ref{lemma-criterion-for-semiregularity}, and $\Phi(E)$ is represented by a coherent sheaf $E'$, then $E'$ satisfies the hypotheses of Lemma \ref{lemma-criterion-for-semiregularity} and is thus semiregular as well. 
The space $\prod_{q= 0}^{d-2}H^{q+2}(M,\Omega^q_M)$ in the above diagrams is the graded summand $H\Omega_{-2}(M)$ of the Hochschild homology $HH_*(M)$.
The equivalence $\Phi$ induces isomorphisms $\Phi:\Ext^2(E,E)\rightarrow \Ext^2(E',E')$ and 
$\Phi_*:H\Omega_{-2}(M)\rightarrow H\Omega_{-2}(M')$. The hypotheses of Lemma \ref{lemma-criterion-for-semiregularity} imply 
that the semiregularity map $\sigma'$ of $E'$ is the conjugate of the semiregularity map $\sigma$ of $E$,
$\sigma'\circ\Phi=\Phi_*\circ\sigma$. It is natural to expect that the latter equality holds, more generally, without the hypotheses of Lemma \ref{lemma-criterion-for-semiregularity}.
\end{rem}

\begin{lem}
\label{lemma-Ext-algebra-generated-by-Ext-1-when-n=1}
\begin{enumerate}
\item
\label{lemma-item-Ext-algebra-generated-by-Ext-1-when-n=1}
The algebra $\Ext^*(\Ideal{C_j},\Ideal{C_j})$ is generated by $\Ext^1(\Ideal{C_j},\Ideal{C_j})$.
\item
\label{lemma-item-Yoneda-algebra-is-quotient-of-Hochschild}
The homomorphism $ev_{\Ideal{C_j}}:HT^*(X)\rightarrow \Ext^*(\Ideal{C_j},\Ideal{C_j})$ is surjective and its kernel is the annihilator $\ann(ch(\Ideal{C_j}))$ of the Chern character $1-\frac{1}{2}\Theta^2+2[pt]$ of $\Ideal{C_j}$.
\[
0\rightarrow \ann(ch(\Ideal{C_j}))\cap HT^2(X)\rightarrow 
\begin{array}{c}
H^2(\StructureSheaf{X})\\
\oplus\\
H^1(TX)\\
\oplus\\
H^0(\wedge^2TX)
\end{array}
\LongRightArrowOf{\left(
\begin{array}{ccc}
1&0&-\Theta^2/2
\\
0 & -\Theta^2/2&2[pt]
\end{array}
\right)}
\begin{array}{c}
H^2(\StructureSheaf{X}) \\ \oplus \\ H^3(\Omega^1_X)
\end{array}
\rightarrow 0.
\]
\item
\label{lemma-item-semiregular}
The sheaf $\Ideal{C_j}$ is semiregular.
\end{enumerate}
\end{lem}

\begin{proof}
(\ref{lemma-item-Ext-algebra-generated-by-Ext-1-when-n=1})
This is the case $n=1$. In this case $\Ext^2(\Ideal{C_j},\Ideal{C_j})$ is $6$-dimensional,
by Lemma \ref{lemma-dim-Ext-1-F-F-is-3n+3}, and so $ob_{\Ideal{C_j}}$ is surjective, 
by Lemma \ref{lemma-rank-of-ob-F-is-at-least-6}. Now $HT^*(X)$ is generated by $HT^1(X)$ and $ev_{\Ideal{C_j}}$
is an algebra homomorphism.  Hence, the surjectivity of $ev^2_{\Ideal{C_j}}=ob_{\Ideal{C_j}}$ implies that 
the Yoneda product $\Ext^1(\Ideal{C_j},\Ideal{C_j})\otimes \Ext^1(\Ideal{C_j},\Ideal{C_j})\rightarrow \Ext^2(\Ideal{C_j},\Ideal{C_j})$
is surjective. The surjectivity of $\Ext^1(\Ideal{C_j},\Ideal{C_j})\otimes \Ext^2(\Ideal{C_j},\Ideal{C_j})\rightarrow \Ext^3(\Ideal{C_j},\Ideal{C_j})$ follows from Serre's duality.

(\ref{lemma-item-Yoneda-algebra-is-quotient-of-Hochschild})
The homomorphism $ev_{\Ideal{C_j}}^1:HT^1(X)\rightarrow \Ext^1(\Ideal{C_j},\Ideal{C_j})$ is an isomorphism, by Lemma
 \ref{lemma-dim-Ext-1-F-F-is-3n+3}. Hence, the homomorphism $ev_{\Ideal{C_j}}$ is surjective, 
 by part (\ref{lemma-item-Ext-algebra-generated-by-Ext-1-when-n=1}). The inclusion 
 $\ker(ev_{\Ideal{C_j}})\subset \ann(ch(\Ideal{C_j}))$ 
 follows from \cite[Theorem B]{Huang}. Both ideals are graded (homogeneous). Indeed, the homomorphism $ev_{\Ideal{C_j}}$ is graded, by definition, and contraction with $ch(\Ideal{C_j})$ maps $HT^k(X)$ to $\oplus_{q-p=k}H^{p,q}(X)$.
 The graded summands of $\ann(ch(\Ideal{C_j}))$ in $HT^0(X)$ and $HT^1(X)$ vanish.
 The equality of $\ker(ev_{\Ideal{C_j}}^2)$ and the graded summand of $\ann(ch(\Ideal{C_j}))$ in $HT^2(X)$ follows from the proof of Lemma \ref{lemma-rank-of-ob-F-is-at-least-6}. The graded summands of both ideals in $HT^3(X)$ have co-dimension $1$, since $\Ext^3(\Ideal{C_j},\Ideal{C_j})$ and $H^3(\StructureSheaf{X})$ are both one-dimensional, and the summand $H^3(\StructureSheaf{X})$ of $HT^3(X)$ surjects onto both. 
 Hence the inclusion $\ker(ev_{\Ideal{C_j}})\subset \ann(ch(\Ideal{C_j}))$ 
implies the equality of the graded summands in $HT^3(X)$ of both ideals. $HT^k(X)$ is contained in both ideals for $k>3$ for degree reasons.

(\ref{lemma-item-semiregular}) 
Follows immediately from Part (\ref{lemma-item-Yoneda-algebra-is-quotient-of-Hochschild})
and Lemma \ref{lemma-criterion-for-semiregularity}.
%*******************
% Hide
%*******************
\hide{
Contraction with $ch(\Ideal{C_j})$ maps the direct summand $H^2(\StructureSheaf{X})$ of $HT^2(X)$ isomorphically onto the direct summand $H^2(\StructureSheaf{X})$ of $H^*(X,\CC)$ and it maps the $9$-dimensional direct summand $H^1(TX)$ of $HT^2(X)$ onto the $3$-dimensional direct summand $H^3(\Omega^1_X)$, since $ch_1(\Ideal{C_j})=0$. Hence, the kernel of $ob_{\Ideal{C_j}}$ intersects $H^2(\StructureSheaf{X})\oplus H^1(TX)$ in a $6$-dimensional subspace, by the equality of the kernels in part (\ref{lemma-item-Yoneda-algebra-is-quotient-of-Hochschild}). It follows that  $ob_{\Ideal{C_j}}$ maps $H^2(\StructureSheaf{X})\oplus H^1(TX)$ onto the $6$-dimensional 
$\Ext^2(\Ideal{C_j},\Ideal{C_j})$. The injectivity of the semi-regularity map $\sigma$ would follow once 
we prove that the intersections of the kernels of $ob_{\Ideal{C_j}}$ and of $\sigma\circ ob_{\Ideal{C_j}}$ with $H^2(\StructureSheaf{X})\oplus H^1(TX)$ are equal.

The composition $\sigma\circ ob_{\Ideal{C_j}}$ restricts to $H^1(TX)$ as the contraction with $ch(\Ideal{C_j})$, by the commutativity of Diagram (\ref{eq-diagram-semi-regularity}), and both 
$\sigma\circ ob_{\Ideal{C_j}}$ and the contraction with $ch(\Ideal{C_j})$
restrict to $H^2(\StructureSheaf{X})$ as multiplication by the rank $1$ of $\Ideal{C_j}$, since $c_1(\Ideal{C_j})=0$. 
Consequently, the intersections of the kernels of $\sigma\circ ob_{\Ideal{C_j}}$ and of the contraction with $ch(\Ideal{C_j})$
 with $H^2(\StructureSheaf{X})\oplus H^1(TX)$ are equal. The latter intersection is equal to the intersection of the kernel of 
$ob_{\Ideal{C_j}}$ with $H^2(\StructureSheaf{X})\oplus H^1(TX)$, by part (\ref{lemma-item-Yoneda-algebra-is-quotient-of-Hochschild}). Hence, indeed, the intersections of the kernels of $ob_{\Ideal{C_j}}$ and of $\sigma\circ ob_{\Ideal{C_j}}$ with $H^2(\StructureSheaf{X})\oplus H^1(TX)$ are equal.
%*******************
% End Hide
%*******************
}
\end{proof}

The proof of Lemma \ref{lemma-dim-Ext-1-F-F-is-3n+3} identifies the $i$-th direct summand in the domain of (\ref{eq-injective-homomorphism-to-Ext-1})
with $H^0(C_i,N_{C_i/X})$, $1\leq i\leq n$. Denote this direct summand by $\tilde{E}^1_i$ and set $\tilde{E}_0^1:=H^1(X,\StructureSheaf{X})$, so that the domain of (\ref{eq-injective-homomorphism-to-Ext-1})
is $\oplus_{i=0}^n \tilde{E}^1_i$. 
Note that each $\tilde{E}^1_i$ is $3$-dimensional. 
We denote by $E^1_i$ the image of $\tilde{E}^1_i$ via the isomorphism (\ref{eq-injective-homomorphism-to-Ext-1}).
The Yoneda product $\Ext^1(F,F)\otimes \Ext^2(F,F)\rightarrow \Ext^3(F,F)$ is a perfect pairing. 
Let $E^2_i$, $1\leq i\leq n$, be the subspace of $\Ext^2(F,F)$ 
annihilating $E^1_0\oplus \oplus_{j=1, j\neq i}^nE^1_i$. Let $E^2_0$ be the image of $H^2(\StructureSheaf{X})$ in $\Ext^2(F,F)$. 
Then $E^2_i$ is $3$-dimensional, for $0\leq i\leq n$. The Yoneda product restricts to $E^1_0\otimes E^2_0\rightarrow \Ext^3(F,F)$ as a perfect pairing. Indeed, the algebra homomorphism 
\[
\Ext^*(\StructureSheaf{X},\StructureSheaf{X})\rightarrow \Ext^*(F,F)
\] 
is injective, since it composes with the trace linear homomorphism 
$tr:\Ext^*(F,F)\rightarrow \Ext^*(\StructureSheaf{X},\StructureSheaf{X})$  to the identity of $\Ext^*(\StructureSheaf{X},\StructureSheaf{X})$, by \cite[10.1.3]{huybrechts-lehn} and the equality $\rank(F)=1$.
We get the direct sum decomposition
\[
\Ext^2(F,F)=\oplus_{i=0}^n E^2_i.
\]
When considering different ideal sheaves $\Ideal{\cup_{i=1}^nC_i}$ we will denote $E^i_j$  by $E^i_j(\Ideal{\cup_{i=1}^nC_i})$.
We have the isomorphism $F:=\Ideal{\cup_{i=1}^n C_i}\cong \otimes_{i=1}^n\Ideal{C_i}$, hence the functor of tensoring with $\otimes_{i=1,i\neq j}^n\Ideal{C_i}$ induces an algebra homomorphism
\[
e_j:\Ext^*(\Ideal{C_j},\Ideal{C_j})\rightarrow \Ext^*(F,F).
\]

The ring structure of the Yoneda algebra $\Ext^*(F,F)$ is determined by the following Lemma and Lemma \ref{lemma-Ext-algebra-generated-by-Ext-1-when-n=1}.

\begin{lem}
\label{lemma-Yoneda-algebra-is-generated-in-degree-1}
\begin{enumerate}
\item
\label{lemma-item-vanishing-of-Yoneda-product}
The Yoneda product maps $E^1_i\otimes E^1_j$ to zero in $\Ext^2(F,F)$, if $i\neq j$ and $1\leq i,j\leq n$.
\item
\label{lemma-item-yoneda-product-is-anti-symmetric}
The Yoneda product $\Ext^1(F,F)\otimes \Ext^1(F,F)\rightarrow \Ext^2(F,F)$ is anti-symmetric.
\item
\label{lemma-item-Yoneda-algebra-generated-by-Ext-1}
The Yoneda product maps  $(E^1_0\oplus E^1_i)\otimes (E^1_0\oplus E^1_i)$ surjectively onto $E^2_0\oplus E^2_i$, for all $1\leq i\leq n$. In particular, the algebra $\Ext^*(F,F)$ is generated by $\Ext^1(F,F)$.
\item
\label{lemma-item-image-of-e-j}
The image of $e_j$ is $\Hom(F,F)\oplus (E^1_0\oplus E^1_j) \oplus (E^2_0\oplus E^2_j) \oplus \Ext^3(F,F).$
\item
\label{lemma-item-e-j-maps-E^2-j-to-same}
The  homomorphism $e_j$ maps  $E^d_1(\Ideal{C_j})$ isomorphically onto $E^d_j(F)$, for $d=1,2$.
\end{enumerate}
\end{lem}

\begin{proof}
(\ref{lemma-item-vanishing-of-Yoneda-product})
Let $\xi_i$ be a section of $H^0(N_{C_i/X})$ and $\tilde{\xi}_i$ the corresponding class in $E^1_i$. 
Set $R:=\CC[\epsilon_1,\epsilon_2]/\langle \epsilon_1^2,\epsilon_2^2,\epsilon_1\epsilon_2\rangle$. 
The product $\tilde{\xi}_i\circ\tilde{\xi}_j$ vanishes, if and only if there exists a deformation of  $\Ideal{\cup_{k=1}^nC_k}$ 
by an ideal over $X\times \Spec(R)$, which restricts to 
the first order deformation along $\xi_i$ over $\Spec(\CC[\epsilon_1,\epsilon_2]/\langle \epsilon_1^2,\epsilon_2,\rangle)$ and to the first order deformation along $\xi_j$ over
$\Spec(\CC[\epsilon_1,\epsilon_2]/\langle \epsilon_1,\epsilon_2^2\rangle)$,
by \cite[Sec. 2]{artamkin}. 
Assume that $i\neq j$.
Let $\F$ be the ideal sheaf over $X\times \Spec(R)$ consisting of elements locally of the form
$f_0+f_1\epsilon_1+f_2\epsilon_2$, where 
$
f_0\in F:=\Ideal{\cup_{k=1}^n C_k}, 
$
$f_1$ belongs to $\Ideal{\cup_{k=1,k\neq i}^nC_k}$,
$f_2$ belongs to $\Ideal{\cup_{k=1,k\neq j}^nC_k}$,
\[
(f_1\restricted{)}{C_i}+df_0(\xi_i)=0, \  \mbox{and} \ (f_2\restricted{)}{C_j}+df_0(\xi_j)=0.
\]
One easily checks that $\F$ is indeed an ideal and $\F$ clearly restricts to the ideals of the two desired first order deformations.

(\ref{lemma-item-yoneda-product-is-anti-symmetric}) The anti-symmetry would follow from that of $(E^1_0\oplus E^1_i)\otimes (E^1_0\oplus E^1_i)\rightarrow \Ext^2(F,F)$, by
Part (\ref{lemma-item-vanishing-of-Yoneda-product}). Now, $E^1_0\oplus E^1_i$ is equal to $e_i(\Ext^1(\Ideal{C_i},\Ideal{C_i}))$, $e_i$ is an algebra homomorphism, and $\Ext^1(\Ideal{C_i},\Ideal{C_i})\otimes \Ext^1(\Ideal{C_i},\Ideal{C_i})\rightarrow \Ext^2(\Ideal{C_i},\Ideal{C_i})$ is anti-symmetric, 
by the anti-symmetry of the product $HT^1(X)\otimes HT^1(X)\rightarrow HT^2(X)$ and the surjectivity of $ev_{\Ideal{C_i}}$ in Lemma \ref{lemma-Ext-algebra-generated-by-Ext-1-when-n=1}.

(\ref{lemma-item-Yoneda-algebra-generated-by-Ext-1}) 
We prove first that the image of $(E^1_0\oplus E^1_i)\otimes (E^1_0\oplus E^1_i)$ is contained in $E^2_0\oplus E^2_i$. It suffices to prove that the image of $E^1_i\otimes \Ext^1(F,F)$ is contained in $E^2_0\oplus E^2_i$. This follows from part 
(\ref{lemma-item-vanishing-of-Yoneda-product}) of the lemma as $E^1_i\otimes \Ext^1(F,F)$ annihilates $E^1_j$, for all $j\neq i$ (here we use also the anti-symmetry property in part (\ref{lemma-item-yoneda-product-is-anti-symmetric})).
Surjectivity would follow from the proof of part (\ref{lemma-item-image-of-e-j}) and 
Lemma \ref{lemma-Ext-algebra-generated-by-Ext-1-when-n=1}.

(\ref{lemma-item-image-of-e-j}) We know that $e_j$ maps $\Ext^k(\Ideal{C_j},\Ideal{C_j})$  onto $\Ext^k(F,F)$, for $k=0$ and $k=3$. We also know that $e_j(\Ext^1(\Ideal{C_j},\Ideal{C_j}))=E^1_0\oplus E^1_j.$
It remains to show that $e_j(\Ext^2(\Ideal{C_j},\Ideal{C_j}))=E^2_0\oplus E^2_j.$ Clearly, $e_j(\Ext^2(\Ideal{C_j},\Ideal{C_j}))$
is contained in the image of $(E^1_0\oplus E^1_j)\otimes (E^1_0\oplus E^1_j)$ as $e_j$ is an algebra homomorphism. Hence,
$e_j(\Ext^2(\Ideal{C_j},\Ideal{C_j}))$ is contained in $E^2_0\oplus E^2_j,$ by part (\ref{lemma-item-Yoneda-algebra-generated-by-Ext-1}). Now the fact that the restriction of $e_j$ is injective on $\Ext^1(\Ideal{C_j},\Ideal{C_j})$ implies that it is also injective on $\Ext^2(\Ideal{C_j},\Ideal{C_j})$,
since the pairing induced by the Yoneda product of their images in $\Ext^*(F,F)$ pulls back to the perfect pairing of the Yoneda product in $\Ext^*(\Ideal{C_j},\Ideal{C_j})$, as $e_j$ induces an isomorphism of $\Ext^3(\Ideal{C_j},\Ideal{C_j})$ with 
$\Ext^3(F,F)$. The equality $e_j(\Ext^2(\Ideal{C_j},\Ideal{C_j}))=E^2_0\oplus E^2_j$ follows for dimension reasons.

(\ref{lemma-item-e-j-maps-E^2-j-to-same}) The statement is clear for $d=1$. For $d=2$ it follows from part (\ref{lemma-item-image-of-e-j}) and the fact that $e_j$ is an $H^*(\StructureSheaf{X})$-algebra homomorphism.
\end{proof}

\begin{prop}
\label{prop-obstruction-map-has-rank-6}
$\rank(ob_F)=6.$
\end{prop}

\begin{proof}
We already know that $\rank(ob_F)\geq 6$, by Lemma \ref{lemma-rank-of-ob-F-is-at-least-6}.
It remains to prove that $\rank(ob_F)\leq 6$.
If $n=1$, then $\dim\Ext^2(F,F)=6$, by Lemma \ref{lemma-dim-Ext-1-F-F-is-3n+3}, and  so $\rank(ob_F)= 6$.
%In this case $ev_F^1$ is surjective and so $ev_F$ is surjective, as both algebras $HH^*(X)$ and $\Ext^*(F,F)$ are 
%generated in degree $1$ (for the latter use Lemma \ref{lemma-Yoneda-algebra-is-generated-in-degree-1}). 

Denote by $ev_F:HT^*(X)\rightarrow \Ext^*(F,F)$ also the composition of (\ref{eq-ev-F})
with the HKR isomorphism $HT^*(X)\cong HH^*(X)$. The HKR isomorphism is an $H^*(\StructureSheaf{X})$-algebra isomorphism,  since the Todd class of $X$ vanishes \cite{calaque-et-al}. Hence, the latter $ev_F$ is an $H^*(\StructureSheaf{X})$-algebra homomorphism.
An element $\xi\in HT^1(X)$ decomposes uniquely as the sum $\xi'+\xi''$ with $\xi'\in H^1(\StructureSheaf{X})$ and $\xi''\in H^0(TX).$
We have the following equalities:
\begin{eqnarray}
\label{eq-relations-between-ev-F-and-ev-I-C-i}
ev_F(\xi') &=& e_j(ev_{\Ideal{C_j}}(\xi')),
\\
\nonumber
ev_F(\xi'')&=&\sum_{j=1}^n e_j(ev_{\Ideal{C_j}}(\xi'')),
\end{eqnarray}
by Lemma \ref{lemma-dim-Ext-1-F-F-is-3n+3}.
Note that  $e_j\circ ev_{\Ideal{C_j}}$ is a  composition of $H^*(\StructureSheaf{X})$-algebra homomorphisms.
If $j\neq k$, then $e_j(ev_{\Ideal{C_j}}(\xi''_1))e_k(ev_{\Ideal{C_k}}(\xi_2''))=0$, for every two elements $\xi_1, \xi_2\in HT^1(X)$, by Lemma \ref{lemma-Yoneda-algebra-is-generated-in-degree-1}. We get
\begin{eqnarray*}
ev_F(\xi_1\xi_2) &=&
\left[ev_F(\xi'_1)+\sum_{j=1}^n  e_j(ev_{\Ideal{C_j}}(\xi''_1))\right]
\left[ev_F(\xi'_2)+\sum_{j=1}^n  e_j(ev_{\Ideal{C_j}}(\xi''_2))\right]
\\ 
&=&
ev_F(\xi'_1\xi'_2)+
\sum_{j=1}^n e_j(ev_{\Ideal{C_j}}(\xi'_1\xi''_2+\xi''_1\xi'_2+\xi''_1\xi''_2))
\\ 
&=& \sum_{j=1}^n e_j\left(ev_{\Ideal{C_j}}\left(\frac{1}{n}\xi'_1\xi'_2+\xi'_1\xi''_2+\xi''_1\xi'_2+\xi''_1\xi''_2\right)\right).
\end{eqnarray*}

The algebra $HT^*(X)$ is generated by $HT^1(X)$. The element $\xi'_1\xi'_2$ belongs to $H^2(\StructureSheaf{X})$,
while the element $\xi'_1\xi''_2+\xi''_1\xi'_2+\xi''_1\xi''_2$ belongs to $H^1(TX)\oplus H^0(\wedge^2TX)$. 
%Given an element $\eta\in HT^*(X)$,
%let $\eta_{i,-j}$ be its direct summand in $H^i(\wedge^jTX)$. 
Thus, the two equations (\ref{eq-relations-between-ev-F-and-ev-I-C-i}) hold also for $\xi\in HT^2(X)$ under the decomposition $\xi=\xi'+\xi''$, with $\xi'\in H^2(\StructureSheaf{X})$ and $\xi''\in H^1(TX)\oplus H^0(\wedge^2TX)$. 

Let $\tau_{ij}:D^b(X)\rightarrow D^b(X)$ be the autoequivalence induced by the translation automorphism mapping $C_j$ to $C_i$. This autoequivalence 
acts trivially on $HT^*(X)$ and $ev_{\Ideal{C_i}}=\tau_{ij}\circ ev_{\Ideal{C_j}}$. Furthermore, $e_j$ is injective, by Lemma \ref{lemma-Yoneda-algebra-is-generated-in-degree-1}.
Hence, the kernel of the composition $e_j\circ ev_{\Ideal{C_j}}:HT^*(X)\rightarrow \Ext^*(F,F)$
is independent of $j$. 
Let
\[
\gamma_n: HT^2(X)\rightarrow HT^2(X)
\]
be the automorphism multiplying the direct summand $H^2(\StructureSheaf{X})$ by $n$ and acting as the identity on 
$H^1(TX)\oplus H^0(\wedge^2TX)$. We see that $\gamma_n$ maps the kernel of $e_j\circ ob_{\Ideal{C_j}}$
 into that of $ob_F$. Now $e_j$ is injective, by Lemma \ref{lemma-Yoneda-algebra-is-generated-in-degree-1}, and $ev_{\Ideal{C_j}}^2=ob_{\Ideal{C_j}}$ has rank $6$, by the case $n=1$. Hence, $\rank(ob_F)\leq\rank(ob_{\Ideal{C_j}})=6$.
%**********
% Hide
%**********
\hide{
We already know that $\rank(ob_F)\geq 6$, by Lemma \ref{lemma-rank-of-ob-F-is-at-least-6}.
It remains to prove that $\rank(ob_F)\leq 6$.

If $n=1$, then $\dim\Ext^2(F,F)=6$, by Lemma \ref{lemma-dim-Ext-1-F-F-is-3n+3}, and  so $\rank(ob_F)= 6$.
In this case $ev_F^1$ is surjective and so $ev_F$ is surjective, as both algebras $HH^*(X)$ and $\Ext^*(F,F)$ are generated in degree $1$ (for the latter use Lemma \ref{lemma-Yoneda-algebra-is-generated-in-degree-1}). 

Consider the composition
\[
H^0(TX)\rightarrow HT^1(X)\RightArrowOf{HKR} HH^1(X)\RightArrowOf{ev_F^1} \Ext^1(E,E)\rightarrow H^0(\SheafExt^1(F,F))%\rightarrow
%\oplus_{i=1}^n H^0(N_{\cup_{i=1}^n C_i/X}),
\]
where the first arrow is the natural inclusion, the second is the HKR-isomorphism, the third is $ev_F^1$, and the fourth is the the natural global to local homomorphism. We have the isomorphism 
$\SheafExt^1(F,F)\cong N_{\cup_{i=1}^n C_i/X}$. 
The image of the above composition 
is induced by the sheaf homomorphism $TX\rightarrow N_{\cup_{i=1}^n C_i/X}$. It follows that the image of 

\[
HH^1(X)\RightArrowOf{ev_F^1} \Ext^1(E,E)\rightarrow H^0(\SheafExt^1(F,F))\cong \oplus_{i=1}^n H^0(N_{C_i/X})
\]
is equal to the image of $H^0(X,TX)$ via the diagonal evaluation. 

The stalks at closed points of the ideal  sheaf $F$ of the smooth curve $\cup_{i=1}^n C_i$ of codimension $2$ in $X$ have projective dimension $1$. Hence, $\SheafExt^2(F,G)$ vanishes, for every coherent $\StructureSheaf{X}$-module $G$. In particular, $\SheafExt^2(F,F)$ vanishes. We get the left  exact sequence
\[
0\rightarrow H^0(\SheafEnd(F,F))\rightarrow \Ext^2(F,F)\RightArrowOf{j} H^1(\SheafExt^1(F,F)).
\]
We have the isomorphisms $\SheafEnd(F,F)\cong\StructureSheaf{X}$ and $\SheafExt^1(F,F)\cong \oplus_{i=1}^n N_{C_i/X}$.
We have seen in the proof of Lemma \ref{lemma-dim-Ext-1-F-F-is-3n+3}  that $H^0(N_{C_i/X}))$ is $3$-dimensional. So is
$H^0(\StructureSheaf{X})$. Hence, the above sequence is also right exact, by Lemma \ref{lemma-dim-Ext-1-F-F-is-3n+3}.
The ring $HH^*(X)$ is generated by $HH^1(X)$. If $\alpha$ and $\beta$ are classes of $\Ext^1(F,F)$, whose image in $H^0(\SheafExt^1(F,F))$ belong to the 
%**********
% End Hide
%**********
}
\end{proof}

\begin{lem}
\label{lemma-invariance-under-derived-equivalence-of-equality-of-ker-ob-and-ker-ch}
Let $\Phi:D^b(A)\rightarrow D^b(B)$ be an equivalence of derived categories of two abelian varieties and $F$ an object of $D^b(A)$. Assume that the kernel of $ob_F:HT^2(A)\rightarrow \Hom(F,F[2])$ is equal to the subspace annihilating $ch(F)$.
Then the kernel of $ob_{\Phi(F)}$ is equal to the subspace annihilating $ch(\Phi(F))$.
\end{lem}
\begin{proof}
We have the commutative diagram
\[
\xymatrix{
H^*(A,\CC) \ar[d]_{\Phi^H}
& HT^2(A) \ar[d]^{\Phi^{HT}} \ar[l]_{ch(F)}\ar[r]^-{\exp(at_{F})}
& \Hom(F,F[2]) \ar[d]^\Phi
\\
H^*(B,\CC) 
& HT^2(B) \ar[l]^{ch(E)}\ar[r]_-{\exp(at_{E})}
& \Hom(E,E[2]),
}
\]
%************
% Hide
%************
\hide{
We have the commutative diagram
\[
\xymatrix{
H^*(X\times X,\CC) \ar[d]_{\varphi}
& HT^2(X\times X) \ar[d]^{\Phi^{HT}} \ar[l]_{ch(\pi_1^*F_1\otimes\pi_2^*F_2)}\ar[r]^{\exp(at_{\pi_1^*F_1\otimes\pi_2^*F_2})}
& \Ext^2(\pi_1^*F_1\otimes\pi_2^*F_2) \ar[d]^\Phi
\\
H^*(X\times\hat{X},\CC) 
& HT^2(X\times\hat{X}) \ar[l]^{ch(E)}\ar[r]_{\exp(at_{E})}
& \Ext^2(E,E)
}
\]
%************
% End Hide
%************
}
the left square 
by \cite{calaque-et-al}, and the right square by \cite[Theorem A]{Huang}.
The vertical arrows are isomorphisms and the kernels of the two horizontal arrows in the top row are equal.
Hence, the same is true for the bottom row.
\end{proof}

\begin{rem}
The above Lemma holds for more general projective varieties, replacing $ch(F)$ by the Mukai vector $v(F):=ch(F)td(X)^{\frac{1}{2}}$ and factoring the action of $HT^*(X)$ on its module $H^*(X,\CC)$ by the Duflo operator $D:HT^*(X)\rightarrow HT^*(X)$, given by $D(\alpha)=td(X)^{\frac{1}{2}}\Contract\alpha$. For abelian varieties $td(X)^{\frac{1}{2}}=1.$
\end{rem}

\begin{cor}
\label{cor-kernel-of-obstruction-is-annihilator-of-ch} Set $F:=\Ideal{\cup_{i=1}^n C_i}(\Theta).$
The kernel of $ob_F$ is equal to the kernel of the homomorphism $(\bullet)\Contract ch(F):HT^2(X)\rightarrow H^*(X,\CC)$ of contraction with the Chern character of $F$.
\end{cor}

\begin{proof}
%************
% Hide
%************
\hide{
The $9$-dimensional kernel of $ob_F$ is contained in the kernel $K$ of $\Contract ch(F)$, by \cite[Theorem B]{Huang}. 
It suffices to show that $\dim(K)\leq 9$. 
Assume that $\dim(K)> 9$. Then the intersection
$K\cap [H^2(\StructureSheaf{X})\oplus H^0(\wedge^2TX)]$  would be non-trivial. 
We claim that $K\cap [H^2(\StructureSheaf{X})\oplus H^0(\wedge^2TX)]=(0).$
Indeed, $ch(F)=1+\Theta-\frac{d}{2}\Theta^2-d[pt]$ and the composition of $\Contract ch(F)$ with the projection onto the two summands in the Hodge decomposition
\[
HT^2(X)\RightArrowOf{\Contract ch(F)} H^*(X,\CC)\rightarrow H^{3,1}(X)\oplus H^{1,3}(X)
\]
maps $H^2(\StructureSheaf{X}$ isomorphically into 
%**********
% End Hide
%**********
}
It suffices to prove the statement for $F':=\Ideal{\cup_{i=1}^n C_i}$, by Lemma \ref{lemma-invariance-under-derived-equivalence-of-equality-of-ker-ob-and-ker-ch}, as $F$ is the image of $F'$ by the autoequivalence of tensorization by $\Theta$. Now, $ch(\Ideal{\cup_{i=1}^n C_i})=1-\frac{n}{2}\Theta^2+2n[pt]$. The kernel of $ob_{F'}$ is contained in the kernel $K$ of $\Contract ch(F')$, by \cite[Theorem B]{Huang}. 
The kernel of $ob_{F'}$ is $9$ dimensional, by Proposition \ref{prop-obstruction-map-has-rank-6}.
It suffices to show that $\dim(K)\leq 9$. Assume that $\dim(K)> 9$. Then the intersection
$K\cap [H^2(\StructureSheaf{X})\oplus H^0(\wedge^2TX)]$  would be non-trivial. 
We claim that $K\cap [H^2(\StructureSheaf{X})\oplus H^0(\wedge^2TX)]=(0).$
Indeed, contraction with $ch(F')$ induces the homomorphism
\[
H^2(\StructureSheaf{X})\oplus H^0(\wedge^2TX)\rightarrow H^2(\StructureSheaf{X})\oplus H^3(\Omega^1_X)
\]
with upper triangular matrix $\left(
\begin{array}{cc}
1 & -\frac{n}{2}\Theta^2
\\
0 & 2n[pt]
\end{array}
\right),$ which is invertible
\end{proof}

%*************
% Hide
%*************
\hide{
Let $\iota:HT^1(X)\rightarrow HT^1(X^n)$ be given by $\iota(\xi'+\xi'')=\sum_{i=1}^n\pi_i^*(\xi')/n+\sum_{i=1}^n\pi_i^*(\xi'')$, where
$\xi'\in H^1(\StructureSheaf{X})$ and $\xi''\in H^0(TX)$ and we regard $\pi_i^*TX$ as a direct summand of $T(X^n)$. 
Let $\Delta:X\rightarrow X^n$ be the diagonal embedding.
Lemma \ref{lemma-dim-Ext-1-F-F-is-3n+3}
%\ref{lemma-Yoneda-algebra-is-generated-in-degree-1} 
yields the following commutative diagram of algebras homomorphisms. 
\[
\xymatrix{
\otimes_{i=1}^nHT^*(X) \ar[r]^{\cong} \ar[d]_{\otimes_{i=1}^n ev_{\Ideal{C_i}}} &
HT^*(X^n) \ar[d]_{ev_{\boxtimes_{i=1}^n\Ideal{C_i}}} &
HT^*(X) \ar[l]_{\wedge^*\iota} \ar[d]_{ev_F}
\\
\otimes_{i=1}^n \Ext^*(\Ideal{C_i},\Ideal{C_i}) \ar[r]_{\cong} &
\Ext^*(\boxtimes_{i=1}^n\Ideal{C_i},\boxtimes_{i=1}^n\Ideal{C_i})\ar[r]_{{\hspace{5ex}}\Delta^*} &
\Ext^*(F,F)
}
\]
All arrows except $\wedge^*\iota$ and $ev_F$ are surjective, since $\Ext^*(\Ideal{C_i},\Ideal{C_i})$ and $\Ext^*(F,F)$ are generated in degree $1$, by Lemmas  \ref{lemma-Ext-algebra-generated-by-Ext-1-when-n=1} and \ref{lemma-Yoneda-algebra-is-generated-in-degree-1}, and $\Delta^*\circ ev_{\boxtimes_{i=1}^n\Ideal{C_i}}$ is surjective, by Lemma \ref{lemma-dim-Ext-1-F-F-is-3n+3}.
The left square commutativity is just the K\"{u}nneth decomposition. The right square commutes in degree $1$, by Lemma \ref{lemma-dim-Ext-1-F-F-is-3n+3},  and so the commutativity follows since $HT^*(X)$ is generated in degree $1$.
%(using Lemma \ref{lemma-Yoneda-algebra-is-generated-in-degree-1} for $\Ext^*(F,F)$ and 
%Lemma \ref{lemma-Ext-algebra-generated-by-Ext-1-when-n=1} for $\Ext^*(\Ideal{C_i},\Ideal{C_i})$).

The auto-equivalences $\tau_{ij}:D^b(X)\rightarrow D^b(X)$, $1\leq i,j\leq n$, induces isomorphisms 
$\tau_{ij}:\Ext^*(\Ideal{C_j},\Ideal{C_j})\rightarrow \Ext^*(\Ideal{C_i},\Ideal{C_i})$. 
The object $\boxtimes_{i=1}^n\Ideal{C_i}$ of $X^n$ is the image of the $\fS_n$-equivariant object $\boxtimes_{i=1}^n\Ideal{C_1}$
via the (non-equivariant) auto-equivalence $(id,\tau_{2,1},\dots,\tau_{n,1})$ of $D^b(X^n)$.
We get an $\fS_n$-action on $\Ext^*(\boxtimes_{i=1}^n\Ideal{C_i},\boxtimes_{i=1}^n\Ideal{C_i})$,  and the kernel of $\Delta^*$ in the diagram above is $\fS_n$-invariant, by Lemma \ref{lemma-Yoneda-algebra-is-generated-in-degree-1}. 
Hence, we get an $\fS_n$-action on $\Ext^*(F,F)$.
The left two vertical homomorphisms above are $\fS_n$-equivariant and $\wedge^*\iota$ is $\fS_n$-invariant. As a corollary, we get the following.
\begin{cor}
The homomorphism $ev_F:HT^*(X)\rightarrow \Ext^*(F,F)$ is 
$\fS_n$-invariant. In particular, the image of $ob_F$ is equal to the invariant subspace $\Ext^2(F,F)^{\fS_n}$.
\end{cor} 

\begin{proof}
The homomorphism $ev_F$ is $\fS_n$-invariant, by the commutativity of the right square in the diagram above.
The image of $ev^2_F$ is contained in $\Ext^2(F,F)^{\fS_n}$, since $ev^2_F$ is $\fS_n$-invariant. 
Lemma \ref{lemma-Yoneda-algebra-is-generated-in-degree-1} implies that $\Ext^2(F,F)^{\fS_n}$ is $6$-dimensional. Hence, equality follows from Proposition \ref{prop-obstruction-map-has-rank-6}.
\end{proof}

%*************
% End Hide
%*************
}

%****************************************************************
% 
%****************************************************************
\subsection{Secant$^{\boxtimes 2}$-sheaves over $X\times\hat{X}$ with 
a $9$-dimensional space of unobstructed commutative-gerby deformations}
\label{sec-ob-E}
Let $\pi_i$, $i=1,2$, be the projections from $X\times X$ to $X$. Denote by $at_F$ the Atiyah class of $F$.
The Atiyah class $at_{\pi_1^*F}\in \Ext^1(\pi_1^*F,(\pi_1^*F)\otimes \Omega^1_{X\times X})$ of $\pi_1^*F$ is equal to the pushforward of $\pi_1^*at_F\in \Ext^1(\pi_1^*F,\pi_1^*(F\otimes\Omega^1_X))$ via the inclusion of $\pi_1^*\Omega^1_X$ as a direct summand in $\Omega^1_{X\times X}$ \cite[Prop. 3.14]{buchweitz-flenner}. Let $F_1$ and $F_2$ be the sheaves in Theorem \ref{main-theorem-introduction}.
The Atiyah class of $\pi_1^*F_1\otimes\pi_2^*F_2$ satisfied 
\begin{equation}
\label{eq-decomposition-of-atyah-class-of-secant-square-sheaf}
at_{\pi_1^*F_1\otimes\pi_2^*F_2}=at_{\pi_1^*F_1}\otimes 1+1\otimes at_{\pi_2^*F_2}.
\end{equation}

%(\ref{eq-F-1-and-F-2}). 
The K\"{u}nneth  decomposition of $\Ext^2(\pi_1^*F_1\otimes \pi_2^*F_2,\pi_1^*F_1\otimes \pi_2^*F_2)$ is the direct  sum
\[
%\Ext^2(\pi_1^*F_1\otimes \pi_2^*F_2)=
[\Ext^2(F_1,F_1)\otimes\Ext^0(F_2,F_2)]\oplus [\Ext^0(F_1,F_1)\otimes\Ext^2(F_2,F_2)]\oplus [\Ext^1(F_1,F_1)\otimes\Ext^1(F_2,F_2)].
\]
We have the direct sum decomposition of $HT^2(X\times X)$
\begin{eqnarray}
\label{eq-decomposition-of-HT-2}
HT^2(X\times X)&=& 
\pi_1^*HT^2(X)
%\left[H^2(\StructureSheaf{X})\oplus H^0(\wedge^2TX)\oplus H^1(TX)\right]
\otimes\pi_2^*HT^0(X)\oplus
\\
\nonumber
&&
\pi_1^*HT^0(X)\otimes \pi_2^*HT^2(X)
%\left[H^2(\StructureSheaf{X})\oplus H^0(\wedge^2TX)\oplus H^1(TX)\right]
\oplus
\\
\nonumber
& &
\pi_1^*HT^1(X)
%\left[H^1(\StructureSheaf{X})\oplus  H^0(TX)\right]
\otimes \pi_2^*HT^1(X)
%\left[H^1(\StructureSheaf{X})\oplus  H^0(TX)\right]
\end{eqnarray}
Lemma \ref{lemma-dim-Ext-1-F-F-is-3n+3} implies that $\exp(at_F): H^1(\StructureSheaf{X})\oplus H^0(TX)\rightarrow \Ext^1(F,F)$ is injective. 
The obstruction map $ob_F:HT^2(X)\rightarrow \Ext^2(F,F)$ is the restriction to $HT^2(X)$ of the algebra homomorphism 
\[
(\bullet)\Contract \exp(at_F):HT^*(X)\rightarrow \Ext^*(F,F)
\]
(see \cite[Theorem A]{Huang}).
We see that the obstruction map $ob_{\pi_1^*F_1\otimes\pi_2^*F_2}$ maps the  summand in the $i$-th row above into the $i$-th summand of $\Ext^2(\pi_1^*F_1\otimes \pi_2^*F_2,\pi_1^*F_1\otimes \pi_2^*F_2)$ in the decomposition displayed above and 
$ob_{\pi_1^*F_1\otimes\pi_2^*F_2}$ restricts  as an injective homomorphism to the third summand. We get
\begin{equation}
\label{eq-kernel-of-obstruction-of-box-product}
\ker(ob_{\pi_1^*F_1\otimes\pi_2^*F_2})=[\pi_1^*\ker(ob_{F_1})\otimes \pi_2^*H^0(\StructureSheaf{X})]
\oplus [\pi_1^*H^0(\StructureSheaf{X})\otimes\pi_2^*\ker(ob_{F_2})].
\end{equation}

%************
% Hide
%************
\hide{
The derived monodromy group $\Spin(V)$ of $X$ acts on both spaces 
(??? does it act on $HT^*(X)$ ???) and the action
\[
HT^*(X)\otimes H^*(X,\CC)\rightarrow HT^*(X)
\]
is $\Spin(V)$-equivariant. Hence, the stabilizer $\Spin(V)_P$ of $ch(F_i)$ leaves $\ker(ob_{F_i})$-invariant. 
Hence, the diagonal $\Spin(V)$ action leaves $\ker(ob_{\pi_1^*F_1\otimes\pi_2^*F_2})$ invariant. 
Let $\varphi:H^*(X\times X,\ZZ)\rightarrow H^*(X\times\hat{X},\ZZ)$ be the isomorphism induced by $\Phi$.
The $\varphi\circ \Spin(V)\varphi^{-1}$-action on $H^*(X\times\hat{X},\ZZ)$ is a subgroup of the derived monodromy grooup of $X\times\hat{X}$. Hence it acts also on $HT^*(X\times\hat{X})$. Again, $ch(\Phi(\pi_1^*F_1\otimes\pi_2^*F_2))$ is 
$\varphi\circ \Spin(V)_P\varphi^{-1}$-invariant. Set $E:=\Phi(\pi_1^*F_1\otimes\pi_2^*F_2)$. Then $\ker(ob_E)$ is
$\varphi\circ \Spin(V)_P\varphi^{-1}$-invariant. So the tangent to the $\varphi\circ \Spin(V)_P\varphi^{-1}$-orbit of every element of 
$\ker(ob_E)$ is contained in $\ker(ob_E)$.
%************
% End Hide
%************
}
Let $\Phi:D^b(X\times X)\rightarrow D^b(X\times\hat{X})$ be Orlov's derived equivalence (\ref{eq-Orlov-derived-equivalence-from-XxX-to-X-times-hat-X}) and set
\begin{equation}
\label{eq-E}
E:=\Phi(\pi_1^*F_1\otimes\pi_2^*F_2).
\end{equation} 

\begin{lem}
\label{lemma-kernel-of-ob-E-is-annihilator-of-ch_E}
\begin{enumerate}
\item
\label{lemma-item-kernel-of-ob-F-1-boxtimes-F-2-is-annihilator-of-ch}
The kernel of $ob_{\pi_1^*F_1\otimes\pi_2^*F_2}$ is equal to the subspace of $HT^2(X\times X)$ annihilating 
$ch(\pi_1^*F_1\otimes\pi_2^*F_2)$.
\item
\label{lemma-item-kernel-of-ob-E-is-annihilator-of-ch_E}
The kernel of $ob_{E}$ 
is equal to the subspace of $HT^2(X\times \hat{X})$ annihilating 
$ch(E)$.
\end{enumerate}
\end{lem}

\begin{proof}
(\ref{lemma-item-kernel-of-ob-F-1-boxtimes-F-2-is-annihilator-of-ch})
Let $Z$ be the subspace of $HT^2(X\times X)$ annihilating $ch(\pi_1^*F_1\otimes\pi_2^*F_2)$.
The inclusion $\ker(ob_{\pi_1^*F_1\otimes\pi_2^*F_2})\subset Z$ follows from \cite[Theorem B]{Huang}.
Equation (\ref{eq-kernel-of-obstruction-of-box-product}) implies that $\ker(ob_{\pi_1^*F_1\otimes\pi_2^*F_2})$ is contained in 
\begin{equation}
\label{eq-part-of-HT2}
[\pi_1^*HT^2(X)\otimes \pi_2^*HT^0(X)]\oplus [\pi_1^*HT^0(X)\otimes \pi_2^*HT^2(X)].
\end{equation}
We have seen that the kernel of $ob_{F_i}$ is equal to the subspace of $HT^2(X)$ annihilating $ch(F_i)$ under the action of $HT^*(X)$ on its module $H^*(X,\CC)$ (Corollary \ref{cor-kernel-of-obstruction-is-annihilator-of-ch}). 
In order to prove the inclusion $Z\subset\ker(ob_{\pi_1^*F_1\otimes\pi_2^*F_2})$
it suffices to prove that $Z$ is contained in (\ref{eq-part-of-HT2}), by the decomposition (\ref{eq-decomposition-of-atyah-class-of-secant-square-sheaf}) of the Atiyah class of $\pi_1^*F_1\otimes\pi_2^*F_2$.
% by equation  (\ref{eq-kernel-of-obstruction-of-box-product}).
The homomorphism
\[
(\pi_1^*ch(F_1)\cup\pi_2^*ch(F_2)) \Contract (\bullet):HT^2(X\times X)\rightarrow H^*(X\times X,\CC)
\]
maps the third summand $\pi_1^*HT^1(X)\otimes\pi_2^*HT^1(X)$ in the decomposition (\ref{eq-decomposition-of-HT-2}) to a subspace of $H^*(X\times X,\CC)$ intersecting trivially the sum of the images of the other two summands.
Indeed, the first summand in (\ref{eq-decomposition-of-HT-2}) is mapped into $\pi_1^*H^*(X,\CC)\otimes\pi_2^*ch(F_2)$,
the second into $\pi_1^*ch(F_1)\otimes \pi_2^*H^*(X,\CC)$ and every element in the direct sum of the latter two 
is the sum of classes of the form $\pi_1^*\alpha\cup\pi_2^*\beta$, where either $\alpha$ or $\beta$ is a  Hodge class. 
On the other hand, the image of  $\pi_1^*HT^1(X)\otimes\pi_2^*HT^1(X)$ is contained in the subspace
\[
\pi_1^*[\oplus_{q-p=1}H^{p,q}(X)]\otimes \pi_2^*[\oplus_{q-p=1}H^{p,q}(X)].
\]
In order to prove that $Z$ is contained in (\ref{eq-part-of-HT2}) it suffices to prove that  the homomorphism 
$ch(F_i)\Contract (\bullet):HT^1(X)\rightarrow H^*(X,\CC)$ is injective. 
%by the argument  analogous to the one we used to prove equation (\ref{eq-kernel-of-obstruction-of-box-product}).
Indeed, multiplication by $1=\rank(F_i)$ induces an injective homomorphism from the subspace $H^1(\StructureSheaf{X})$ of $HT^1(X)$ to the subspace $H^1(\StructureSheaf{X})$ of  $H^1(X,\CC)$ and contraction with $ch_2(F_i)=-\frac{n}{2}\Theta^2$ induces an injective homomorphism from $H^0(TX)$ to $H^{1,2}(X)$.

(\ref{lemma-item-kernel-of-ob-E-is-annihilator-of-ch_E})
Apply Lemma \ref{lemma-invariance-under-derived-equivalence-of-equality-of-ker-ob-and-ker-ch} 
with $A=X\times X$, $B=X\times\hat{X}$, and $F=\pi_1^*F_1\otimes\pi_2^*F_2$.
\end{proof}

%****************************************************************
% 
%****************************************************************
\subsection{Orlov's isomorphism $\Phi^{HT}:HT^2(X\times X)\rightarrow HT^2(X\times\hat{X})$ maps diagonal deformations to commutative-gergy ones}
\label{sec-diaginal-deformations}

The algebra $HT^*(X)$ acts on its module $H^*(X,\CC):=\oplus H^{p,q}(X)$ and embeds in $\End(H^*(X,\CC))$.
Given $\alpha\in HT^*(X)$ denote by $e_\alpha\in \End(H^*(X,\CC))$ the corresponding endomorphism.
Let $\tau$ be the involution in (\ref{eq-Mukai-pairing}).
If $\alpha$ is an element of $H^i(\wedge^jTX)$ and $x$ is a class in $H^k(X,\CC)$, then $e_\alpha(x)$ belongs to  $H^{k+i-j}(X,\CC)$. Set $t:=i-j$. We have
\[
(\tau\circ e_\alpha\circ \tau)(x) = (-1)^{\frac{(k+t)(k+t-1)}{2}}(-1)^{\frac{k(k-1)}{2}}e_\alpha(x)=(-1)^{kt+\frac{t(t-1)}{2}}e_\alpha(x).
\]
In particular, for $k$ even, we have $(\tau\circ e_\alpha\circ \tau)(x) =(-1)^{\frac{t(t-1)}{2}}e_\alpha(x)$. In particular,
\[
(\tau\circ e_\alpha\circ \tau)(ch(F)) =(-1)^{\frac{t(t-1)}{2}}e_\alpha(ch(F))
\]
Let $(\bullet)^*:HT^*(X)\rightarrow HT^*(X)$ act on $H^i(\wedge^jTX)$ by multiplication by $(-1)^{\frac{(i-j)(i-j-1)}{2}}$.
We get 
\begin{equation}
\label{eq-conjugation-of-HT-algebra-by-tau}
(\tau\circ e_\alpha\circ \tau)(ch(F)) = e_{\alpha^*}(ch(F)).
\end{equation}
Recall that $\tau(ch(F))=ch(F^\vee)$.
In particular, $e_{\alpha^*}$ annihilates $ch(F^\vee)$, if and only if $e_\alpha$ annihilates $ch(F)$.
For $(\alpha,\beta,\gamma)\in HT^2(X)=H^2(\StructureSheaf{X})\oplus H^1(TX)\oplus H^0(\wedge^2TX)$ we have
$(\alpha,\beta,\gamma)^*=(-\alpha,\beta,-\gamma).$

The {\em diagonal embedding} of $HT^2(X)$ in $HT^2(X\times X)$ is given by $\alpha\mapsto \pi_1^*(\alpha)+\pi_2^*(\alpha)$. 
We let the involution $id\otimes (\bullet)^*$ act on $HT^*(X\times X)$ via the K\"{u}nneth decomposition of the latter.

%**************
% Hide
%**************
\hide{
\begin{lem}
The intersection
\begin{eqnarray*}
&&
\left[\pi_1^*HT^2(X)\oplus\pi_2^*HT^2(X)
\right]\cap 
(\Phi^{HT})^{-1}\left(
H^2(\StructureSheaf{X\times \hat{X}})\oplus H^1(T[X\times \hat{X}])
\right) 
\end{eqnarray*}
is equal to the diagonal embedding of $H^1(TX)$ and the anti-diagonal embedding of 
$H^2(\StructureSheaf{X})\oplus H^0(\wedge^2TX)$.
\end{lem}
%**************
% End Hide
%**************
}

Given an equivalence $F:D^b(X)\rightarrow D^b(Y)$ of the derived categories of two smooth projective varieties $X$ and $Y$ we get the graded ring isomorphism 
$F^{HT}:HT^*(X)\rightarrow HT^*(Y)$ (see \cite[Cor. 8.3]{caldararu-I} and \cite[Theorem 1.4]{calaque-et-al}). The summand $HT^2(X)$ parametrizes first order deformations of $D^b(X)$ associated to first order deformations of the abelian category of coherent sheaves on $X$ \cite{toda}.
The summand $HT^1(X)$ is the Lie algebra of the identity component of $\Aut(D^b(X))$, and $F^{HT}$ restricts to the differential of the isomorphism induced by conjugation by $F$. If $F=f_*$, for an isomorphism $f:X\rightarrow Y$, then $f_*^{HT}$ restricts to the summands $H^1(X,\StructureSheaf{X})$ and $H^0(X,TX)$
of $HT^1(X)$ as the homomorphism induced by the direct image functor composed with the isomorphism induced by the natural sheaf isomorphisms $f_*\StructureSheaf{X}\rightarrow\StructureSheaf{Y}$ and $df:f_*TX\rightarrow TY$. When $X$ is an abelian variety, the equivalence $\Phi_\P:D^b(\hat{X})\rightarrow D^b(X)$ with Fourier-Mukai kernel the Poincar\'{e} line bundle $\P$ conjugates autoequivalences associated to translation automorphisms to autoequivalences associated with tensorization by line bundles in $\Pic^0$. Hence, $\Phi_\P^{HT}:HT^1(\hat{X})\rightarrow HT^1(X)$ maps the Lie subalgebra $H^0(T\hat{X})$ of the subgroup $\hat{X}$ of translations of $\hat{X}$ to the Lie subalgebra $H^1(\StructureSheaf{X})$ of the subgroup $\Pic^0(X)$ and it maps $H^1(\StructureSheaf{\hat{X}})$ to $H^0(TX)$.

\begin{lem}
\label{lemma-Orlov's-equivalence-maps-diagonal-deformations-to-commutative-gerby-ones}
The composition 
$\Phi^{HT}\circ (id\otimes (\bullet)^*):HT^2(X\times X)\rightarrow HT^2(X\times\hat{X})$
maps the diagonal embedding of $HT^2(X)$ into $H^1(T[X\times\hat{X}])\oplus H^2(\StructureSheaf{X\times\hat{X}})$.
The image is the direct sum of the graphs of the following three homomorphisms: 
\begin{enumerate}
\item The graph in $\pi_1^*H^1(TX)\oplus \pi_2^*H^1(T\hat{X})$ of the isomorphism 
$\Psi_{\P^{-1}[n]}^{HT}:H^1(TX)\rightarrow H^1(T\hat{X})$.
\item
The graph in $\pi_1^*H^2(\StructureSheaf{X}) \oplus H^1(T[X\times\hat{X}])$
of the homomorphism 
\begin{eqnarray*}
H^2(\StructureSheaf{X})&\rightarrow& H^1(T[X\times\hat{X}])
\\
\eta_1\wedge\eta_2&\mapsto&\pi_1^*(\eta_1)\wedge\pi_2^*(\Psi_{\P^{-1}[n]}^{HT}(\eta_2))-\pi_1^*(\eta_2)\wedge\pi_2^*(\Psi_{\P^{-1}[n]}^{HT}(\eta_1)),
\end{eqnarray*}
where $\eta_i\in H^1(\StructureSheaf{X})$, $i=1,2$.
\item
The graph in $\pi_2^*H^2(\StructureSheaf{\hat{X}})\oplus H^1(T[X\times\hat{X}])$
of the homomorphism
\begin{eqnarray*}
H^2(\StructureSheaf{\hat{X}})&\rightarrow &H^1(T[X\times\hat{X}])
\\
\eta_1\wedge\eta_2&\mapsto& 
-\pi_1^*(\Phi_\P^{HT}(\eta_1))\wedge\pi_2^*(\eta_2)+\pi_1^*(\Phi_\P^{HT}(\eta_2))\wedge\pi_2^*(\eta_1),
\end{eqnarray*}
where $\eta_i\in H^1(\StructureSheaf{\hat{X}})$, $i=1,2$.
\end{enumerate}
In particular, the image of $HT^2(X)$ in $HT^2(X\times\hat{X})$ projects injectively into the direct summand $H^1(T[X\times\hat{X}])$.
\end{lem}

\begin{proof}
$\Phi^{HT}:HT^*(X\times X)\rightarrow HT^*(X\times \hat{X})$ is a graded ring isomorphism.
$HT^*(X\times X)$ is generated by $HT^1(X\times X)$.
The isomorphism $\Phi^{HT}_{\P}:D^b(\hat{X})\rightarrow D^b(X)$, associated to the Poincar\'{e} line bundle, maps 
$H^1(\StructureSheaf{\hat{X}})$ to $H^0(TX)$  and $H^0(T\hat{X})$ to $H^1(\StructureSheaf{X})$. Hence, its inverse $\Psi^{HT}_{\P^{-1}[n]}$
maps $H^0(TX)$ to $H^1(\StructureSheaf{\hat{X}})$ and $H^1(\StructureSheaf{X})$ to $H^0(T\hat{X})$. 
 %**************
% Hide
%**************
\hide{
 It follows that $(1\boxtimes \Psi_{\P^{-1}[n]})^{HT}:HT^2(X\times X)\rightarrow HT^2(X\times\hat{X})$ maps 
$\pi_1^*H^0(TX)\otimes \pi_2^*H^1(\StructureSheaf{X})$ to the subspace $\pi_1^*H^0(TX)\otimes \pi_2^*H^0(T\hat{X}))$
of $H^0(\wedge^2T(X\times\hat{X}))$ and it maps the subspace 
$\pi_1^*H^0(TX)\otimes \pi_2^*H^0(TX)$ of $H^0(\wedge^2T(X\times X))$
to the subspace $\pi_1^*H^0(TX)\otimes \pi_2^*H^1(\StructureSheaf{X})$ 
of $H^1(T(X\times\hat{X}))$ of the commutative first order deformations.
%**************
% End Hide
%**************
}

The isomorphism $(\mu^{-1})_*=\mu^*:D^b(X\times X)\rightarrow D^b(X\times\hat{X})$ induces the isomorphism
\[
(\mu^{-1})_*^{HT}:HT^1(X\times X)\rightarrow HT^1(X\times X),
\]
where $\mu^{-1}(x,y)=(x-y,y)$. 
Given $\xi_2, \xi_2\in H^0(TX)$, we have
%**************
% Hide
%**************
\hide{
\begin{eqnarray*}
(\mu^{-1}_*)^{HT}(\pi_1^*\xi_1\wedge\pi_2^*\xi_2)&=&\pi_1^*(\xi_1)\wedge[-\pi_1^*\xi_2+\pi_2^*\xi_2]
\\ &=&
-\pi_1^*(\xi_1\wedge\xi_2)+\pi_1^*\xi_1\wedge\pi_2^*\xi_2,
\\
((1\boxtimes \Psi_{\P^{-1}[n]})\circ \mu^{-1}_*)^{HT}(\pi_1^*\xi_1\wedge\pi_2^*\xi_2)&=&
-\pi_1^*(\xi_1\wedge\xi_2)+\pi_1^*\xi_1\wedge\pi_2^*\eta_2.
\end{eqnarray*}
%When $\xi_1=\xi_2=\xi$ we get
%\[
%\Phi^{HT}(\pi_1^*\xi\wedge\pi_2^*\xi)=
%((1\boxtimes \Psi_{\P^{-1}[n]})\circ \mu^{-1}_*)^{HT}(\pi_1^*\xi\wedge\pi_2^*\xi)=\pi_1^*\xi\wedge\pi_2^*\eta.
%\]
Also
%**************
% End Hide
%**************
}
\begin{eqnarray*}
(\mu^{-1}_*)^{HT}(\pi_1^*(\xi_1\wedge\xi_2))&=&\pi_1^*(\xi_1\wedge\xi_2).
\\
(\mu^{-1}_*)^{HT}(\pi_2^*(\xi_1\wedge\xi_2))&=&
[-\pi_1^*\xi_1+\pi_2^*\xi_1]\wedge[-\pi_1^*\xi_2+\pi_2^*\xi_2]
\\
&=& \pi_1^*(\xi_1\wedge\xi_2)-\pi_1^*\xi_1\wedge\pi_2^*\xi_2+\pi_1^*\xi_2\wedge\pi_2^*\xi_1+\pi_2^*(\xi_1\wedge\xi_2).
\end{eqnarray*}
On the other hand, given $\eta\in H^1(X,\StructureSheaf{X})$, 
\begin{eqnarray*}
\mu^*(\pi_1^*(\eta)) &=& (\pi_1\circ\mu)^*(\eta)=\pi_1^*\eta+\pi_2^*\eta,
\\
\mu^*(\pi_2^*(\eta))&=& \pi_2^*\eta.
\end{eqnarray*}
Hence, given $\xi\in H^0(X,TX)$ and $\eta_1,\eta_2\in H^1(X,\StructureSheaf{X})$,
\begin{eqnarray*}
\mu^*(\pi_1^*(\eta_1\wedge\eta_2))&=& \pi_1^*(\eta_1\wedge\eta_2)+(\pi_1^*\eta_1\wedge\pi_2^*\eta_2)-(\pi_1^*\eta_2\wedge\pi_2^*\eta_1)+\pi_2^*(\eta_1\wedge\eta_2),
\\
\mu^*(\pi_1^*(\xi\wedge\eta))&=& \pi_1^*\xi\wedge(\pi_1^*\eta+\pi_2^*\eta).
\\
\mu^*(\pi_2^*(\xi\wedge\eta))&=& [-\pi_1^*\xi+\pi_2^*\xi]\wedge \pi_2^*\eta=-(\pi_1^*\xi\wedge\pi_2^*\eta)+\pi_2^*(\xi\wedge\eta).
\end{eqnarray*}

Given $\xi_i\in H^0(TX)$,  set $\eta_i:=\Psi_{\P^{-1}[n]}^{HT}(\xi_i)\in H^1(\StructureSheaf{\hat{X}})$, $i=1,2$.
$(\mu^{-1}_*)^{HT}(\pi_2^*(\xi_1\wedge\xi_2))$ is sent via $(1\boxtimes \Psi_{\P^{-1}[n]})^{HT}$ to
\[
\Phi^{HT}(\pi_2^*(\xi_1\wedge\xi_2))=
\pi_1^*(\xi_1\wedge\xi_2)-\pi_1^*\xi_1\wedge\pi_2^*\eta_2+\pi_1^*\xi_2\wedge\pi_2^*\eta_1+\pi_2^*(\eta_1\wedge\eta_2).
\]
On the other hand, $\Phi^{HT}(\pi_1^*(\xi_1\wedge\xi_2))=\pi_1^*(\xi_1\wedge\xi_2)$. Hence,
\begin{equation}
\label{eq-Phi-HT-takes-non-commutative-anti-diagonal-deformations-to-commutative}
\Phi^{HT}(-\pi_1^*(\xi_1\wedge\xi_2)+\pi_2^*(\xi_1\wedge\xi_2))=
-\pi_1^*\xi_1\wedge\pi_2^*\eta_2+\pi_1^*\xi_2\wedge\pi_2^*\eta_1+\pi_2^*(\eta_1\wedge\eta_2).
\end{equation}
The non-commutative first order deformation $-\pi_1^*(\xi_1\wedge\xi_2)+\pi_2^*(\xi_1\wedge\xi_2)$ of $X\times X$ in $H^0(\wedge^2T(X\times X))$ is mapped to a commutative-gerby deformation of $X\times\hat{X}$.

Let $\eta'_i\in H^1(\StructureSheaf{X})$ and $\xi'_i\in H^0(T\hat{X})$ satisfy $\Psi_{\P^{-1}[n]}^{HT}(\eta'_i)=\xi'_i$, $i=1,2$. Then
\begin{eqnarray*}
\nonumber
\Phi^{HT}(\pi_1^*(\eta'_1\wedge\eta'_2))&=&
\pi_1^*(\eta'_1\wedge\eta'_2)+(\pi_1^*\eta'_1\wedge\pi_2^*\xi'_2)-(\pi_1^*\eta'_2\wedge\pi_2^*\xi'_1)+\pi_2^*(\xi'_1\wedge\xi'_2).
\\
\nonumber
\Phi^{HT}(\pi_2^*(\eta'_1\wedge\eta'_2))&=&\pi_2^*(\xi'_1\wedge\xi'_2).
\end{eqnarray*}
Hence,
\begin{equation}
\label{eq-Phi-HT-takes-gerby-anti-diagonal-deformations-to-gerby-commutative}
\Phi^{HT}(\pi_1^*(\eta'_1\wedge\eta'_2)-\pi_2^*(\eta'_1\wedge\eta'_2))=
\pi_1^*(\eta'_1\wedge\eta'_2)+(\pi_1^*\eta'_1\wedge\pi_2^*\xi'_2)-(\pi_1^*\eta'_2\wedge\pi_2^*\xi'_1).
\end{equation}

Let $\eta'\in H^1(\StructureSheaf{X})$ and $\xi'\in H^0(T\hat{X})$ satisfy $\Psi_{\P^{-1}[n]}^{HT}(\eta')=\xi'$. 
Let $\xi\in H^0(TX)$ and $\eta\in H^1(\StructureSheaf{\hat{X}})$ satisfy $\Psi_{\P^{-1}[n]}^{HT}(\xi)=\eta$.
Then
\begin{eqnarray*}
\Phi^{HT}(\pi_1^*(\xi\wedge\eta'))&=& \pi_1^*(\xi\wedge\eta')+(\pi_1^*\xi\wedge\pi_2^*\xi')).
\\
\Phi^{HT}(\pi_2^*(\xi\wedge\eta'))&=&-(\pi_1^*\xi\wedge\pi_2^*\xi')+\pi_2^*(\eta\wedge\xi').
\\
\Phi^{HT}(\pi_1^*(\xi\wedge\eta'+\pi_2^*(\xi\wedge\eta'))&=&\pi_1^*(\xi\wedge\eta')+\pi_2^*(\eta\wedge\xi').
\end{eqnarray*}
So $\Phi^{HT}$ maps the diagonal commutative first order deformations of $X\times X$ to commutative first order deformations of $X\times\hat{X}$. Furthermore, it maps the anti-diagonal non-commutative and gerby deformations of $X\times X$ to
commutative and gerby deformations of $X\times\hat{X}$, by Equations 
(\ref{eq-Phi-HT-takes-non-commutative-anti-diagonal-deformations-to-commutative}) and 
(\ref{eq-Phi-HT-takes-gerby-anti-diagonal-deformations-to-gerby-commutative}).
%\begin{eqnarray*}
%\Phi^{HT}(-\pi_1^*(\eta'_1\wedge\eta'_2)+\pi_2^*(\eta'_1\wedge\eta'_2))&=&
%-\pi_1^*\eta'_1\wedge\pi_2^*\xi_2+\pi_1^*\eta'_2\wedge\pi_2^*\xi_1+\pi_2^*(\xi_1\wedge\xi_2).
%\\
%\Phi^{HT}(-\pi_1^*(\xi_1\wedge\eta'_2)+\pi_2^*(\xi_1\wedge\eta'_2))&=&
%-\pi_1^*\xi_1\wedge\pi_2^*\xi_2+\pi_1^*\eta'_2\wedge\pi_2^*\xi_1+\pi_2^*(\eta_1\wedge\xi_2).
%\end{eqnarray*}
\end{proof}

%**************
% Hide
%**************
\hide{
\begin{rem}
\label{rem-Orlov's-equivalence-maps-diagonal-deformations-to-commutative-gerby-ones}
The appearance of the anti-diagonal embedding in Lemma \ref{lemma-Orlov's-equivalence-maps-diagonal-deformations-to-commutative-gerby-ones} is explained as follows.
Let $(\alpha,\beta,\gamma)$ be an element of $H^2(\StructureSheaf{X}\oplus H^1(TX)\oplus H^0(\wedge^2TX)$.
Then $(\alpha,\beta,\gamma)$ annihilates $ch(F)$, if and only if $(-\alpha,\beta,-\gamma)$ annihilates $ch(F^\vee)$,
since $ch_{k}(F^\vee)=(-1)^kch_{k}(F)$ and
\begin{eqnarray*}
(\alpha,\beta,\gamma)\Contract ch(F) &=&
\alpha\Contract ch_0(F)+\beta\Contract ch_1(F)+\gamma\Contract ch_2(F) +
\\ 
&&
\alpha\Contract ch_1(F)+\beta\Contract ch_2(F)+\gamma\Contract ch_3(F),
\end{eqnarray*}
where the sum in the first row is in $H^{0,2}(X)$ and the sum in the second row is in $H^{1,3}(X)$. 
\end{rem}
%**************
% End Hide
%**************
}

Let $F_1$, $F_2$, and $E$ be as in Equation (\ref{eq-E}). Note that $ch(F_1)=ch(F_2)$. Hence, $\ker(ob_{F_1})=\ker(ob_{F_2})$, by Corollary \ref{cor-kernel-of-obstruction-is-annihilator-of-ch}.

\begin{cor}
\label{cor-Orlov's-equivalence-maps-diagonal-deformations-to-commutative-gerby-ones}
The isomorphism $\Phi^{HT}\circ (id\otimes(\bullet)^*):HT^2(X\times X)\rightarrow HT^2(X\times\hat{X})$ maps the diagonal embedding of the $9$-dimensional subspace $\ker(ob_{F_1})$ of $HT^2(X)$ to 
a $9$-dimensional subspace of $H^1(T[X\times\hat{X}])\oplus H^2(\StructureSheaf{X\times\hat{X}})$ of commutative and gerby deformations in the kernel of $ob_E$.
\end{cor}

\begin{proof}
\underline{Step 1:}
%Define the object $E$ to be $\Phi(F_2\boxtimes F_1^\vee)$ as in (\ref{eq-object-is-image-of-F-1-dual-times-F-2}).
%There $ch(F_1)=ch(F_2)$. 
We claim that $\ker(ob_{F_i})=\ker(ob_{F_i^\vee})$. The kernel of $of_{F_i}$ is the annihilator of $ch(F_i)$, which is the kernel of the following homomorphism.
\[
\begin{array}{c}
H^2(\StructureSheaf{X})\\
\oplus\\
H^1(TX)\\
\oplus\\
H^0(\wedge^2TX)
\end{array}
\LongRightArrowOf{\left(
\begin{array}{ccc}
1&\Theta&-(d/2)\Theta^2
\\
\Theta & -(d/2)\Theta^2&-(d/6)\Theta^3
\end{array}
\right)}
\begin{array}{c}
H^2(\StructureSheaf{X}) \\ \oplus \\ H^3(\Omega^1_X)
\end{array}
\]
The subspace annihilating $ch(F_i^\vee)$ is the kernel of the homomorphism obtained by replacing the above matrix by
\[
\left(
\begin{array}{cc}
-1 & 0\\
0&1
\end{array}
\right)
\left(
\begin{array}{ccc}
1&\Theta&-(d/2)\Theta^2
\\
\Theta & -(d/2)\Theta^2&-(d/6)\Theta^3
\end{array}
\right)
\left(
\begin{array}{ccc}
-1&0&0
\\
0 & 1 &0
\\
0&0&-1
\end{array}
\right)
\]
Hence, it suffices to show that the kernel is the direct sum of the subspace of $H^1(TX)$ annihilating $\Theta$ and the subspace of $H^2(\StructureSheaf{X})\oplus H^0(\wedge^2TX)$ in the kernel of $(1,-\frac{d}{2}\Theta^2)$. Indeed, both are $9$ dimensional, and so it suffices to prove the inclusion of this direct sum in the subspace annihilating $ch(F_i)$. This follows from the fact that (i) the subspace of $H^1(TX)$ annihilating $\Theta$ is equal to the subspace annihilating $\Theta^2$, and (ii)  the subspace of $H^2(\StructureSheaf{X})\oplus H^0(\wedge^2TX)$ 
in the kernel of $(1,-\frac{d}{2}\Theta^2)$ is equal to the subspace in the kernel of $(\Theta,-\frac{d}{6}\Theta^3)$. Fact (i) is easy to verify. Fact (ii) 
follows from the identity $\xi\Contract c\Theta^n=(-1)^inc(\xi\Contract \Theta)\Theta^{n-1}$, for $\xi\in H^0(TX)$ and $c\in H^i(\StructureSheaf{X})$, $i\geq 0$, and the observation that $c(\xi\Contract \Theta)$ belongs to $H^{i+1}(\StructureSheaf{X})$. Indeed, the identity implies that the homomorphism $(\Theta,-\frac{d}{6}\Theta^3)$ is
the composition
\[
\begin{array}{c}
H^2(\StructureSheaf{X})\\
\oplus\\
H^0(\wedge^2TX)
\end{array}
\LongRightArrowOf{\left(
\begin{array}{cc}
1&-(d/2)\Theta^2
\end{array}
\right)}
H^2(\StructureSheaf{X})\LongRightArrowOf{\Theta\cup}H^3(\Omega^1_X)
\]
and cup product with $\Theta$ is an injective homomorphism.

\underline{Step 2:}
If $e_\alpha$ annihilates $ch(F_1)$, then it annihilates $ch(F_2)$ and also $ch(F_2^\vee)$, by Step 1, and so 
 $e_{\alpha^*}$ annihilates $ch(F_2).$
Hence, $e_{\pi_1^*(\alpha)+\pi_2^*(\alpha^*)}$ annihilates
$ch(F_1\boxtimes F_2)$. It follows that $e_{\Phi^{HT}(\pi_1^*(\alpha)+\pi_2^*(\alpha^*))}$ annihilates $ch(E)$, by \cite[Theorem 1.4]{calaque-et-al}. 
%where $E:=\Phi(F_2\boxtimes F_1^\vee)$. 
The statement thus follows from Lemmas \ref{lemma-Orlov's-equivalence-maps-diagonal-deformations-to-commutative-gerby-ones} 
and \ref{lemma-kernel-of-ob-E-is-annihilator-of-ch_E}(\ref{lemma-item-kernel-of-ob-E-is-annihilator-of-ch_E})
\end{proof}

\begin{rem}
The image of the diagonal embedding in $HT^2(X\times X)$, of the kernel in $HT^2(X)$  of $ob_{F_1}$, is mapped via $\Phi^{HT}\circ(id\otimes(\bullet)^*)$ into a subspace of $H^2(\StructureSheaf{X\times\hat{X}})\oplus H^1(T(X\times\hat{X}))$, which projects onto the tangent space in $H^1(T(X\times\hat{X}))$ of the moduli space of abelian varieties of Weil type. This will follow from Lemma \ref{lemma-huge-diagram-is-commutative}, which shows that contraction with $\exp(-c_1(E)/\rank(E))$ induces an automorphism of $HT^2(X\times\hat{X})$ mapping the kernel of $ob_E$ into the subspace of $H^1(T(X\times\hat{X}))$ tangent to the moduli space of abelian varieties of Weil type. The above automorphism is unipotent, leaving each of $H^2(\StructureSheaf{X\times\hat{X}})\oplus H^1(T(X\times\hat{X}))$ and $H^2(\StructureSheaf{X\times\hat{X}})$ invariant and inducing the identity on the graded summand. Hence, the automorphism restricts to the kernel of $ob_E$ as the projection to $H^1(T(X\times\hat{X}))$.
\end{rem}

\begin{rem}
\label{rem-interchanging-F-1-and-F-2}
All the results of Section \ref{section-secant-sheaves-on-abelian-threefolds} remain valid after interchanging the sheaves $F_1:=\Ideal{\cup_{i=1}^{d+1}C_i}(\Theta)$ and $F_2:=\Ideal{\cup_{i=1}^{d+1}\Sigma_i}(\Theta).$ The two are interchanged by the involution $\iota:X\rightarrow X$ sending $L\in X=\Pic^2(C)$ to $\omega_C\otimes L^{-1}$. The cohomological action of $\iota$ corresponds to the element in the center of $\Spin(V)$ acting as the isentity on $S^+$ and by multiplication by $-1$ on $V$ and $S^-$. It acts on $HT^{ev}(X)$, and so on $HT^2(X)$,  as the identity and on $HT^{odd}(X)$ by multiplication by $-1$.  
\end{rem}

\begin{rem}
Note that when $n=2$ and $F$ is a simple sheaf on an abelian surface $X$, then the kernel of $ob_F$ is $5$-dimensional and Orlov's derived equivalence maps it to a subspace of $H^2(\StructureSheaf{X\times\hat{X}})\oplus H^1(T(X\times\hat{X}))$, which projects onto the tangent space in $H^1(T(X\times\hat{X}))$ of the moduli space of intermediate jacobians associated to the $3$-rd cohomology of generalized kummers \cite{markman-generalized-kummers}.
\end{rem}

%****************************************************************
% 
%****************************************************************
\section{A reflexive sheaf over $X\times\hat{X}$ with $\Spin(V)_P$-invariant characteristic classes}
\label{section-Jacobians-of-genus-3-curves}
Let $C$ be a non-hyperelliptic curve of genus $3$. Let  $C_{\ell_i}=AJ(C)+\ell_i$, $1\leq i\leq d+1$, be the translate in $X=\Pic^2(C)$ of the Abel-Jacobi image  $AJ(C)\subset \Pic^1(C)$ of $C$ by $\ell_i\in\Pic^1(C)$.
In Section \ref{sec-general-position} we show that $H^0(X,\Ideal{\cup_{i=1}^{d+1}C_i}(2\Theta)\otimes L)$ vanishes for all $L\in \Pic^0(X)$, provided the intersection  $\bigcap_{1\leq i< j\leq d+1}\tau_{-\ell_i-\ell_j}(\Theta)$ is empty. Furthermore, the latter condition holds for generic $C$, if the set $\{\ell_i\}_{i=1}^{d+1}$ is an orbit under translations by elements of a cyclic subgroup of $\Pic^0(C)$ of order $d+1$. Set $F_1:=\Ideal{\cup_{i=1}^{d+1}C_i}(\Theta)$ and
$F_2:=\Ideal{\cup_{i=1}^{d+1}\Sigma_i}(\Theta)$. Set $\G:=\Phi(F_2\boxtimes F_1)[-3]$, where $\Phi$ is Orlov's derived equivalence.
In Section \ref{sec-reflexive-secant-square-sheaf} we show that  $\E:=\G^\vee[-1]$ is a reflexive sheaf over $X\times\hat{X}$ of rank $8d$, for a generic choice of the curves $C_i$ and $\Sigma_j$. Furthermore, $\E$ is locally free away from $(d+1)^2$ smooth surfaces in $X\times\hat{X}$.
In Section \ref{subsection-a-semiregular-secant-square-sheaf} we choose the set $\{C_i\}_{i=1}^{d+1}$ to consist of translates of $C_1$ by a cyclic subgroup $G_1$ of $\Pic^0(C)$ and choose the set $\{\Sigma_i\}_{i=1}^{d+1}$ similarly, for a cyclic subgroup $G_2$ of $\Pic^0(C)$. We show that a tensor product of the reflexive sheaf $\E$ with a suitable twisted line bundle descends to a semiregular reflexive twisted sheaf $\B$ over a quotient $Y$ of $X\times \hat{X}$ by a group $\bar{G}$ of translations, where $\bar{G}\cong G_1\times G_2$. We then prove Theorem \ref{thm-algebraicity}, stating the algebraicity of the Hodge-Weil classes over abelian sixfolds of Weil type of discriminant $-1$, using the version of the Semiregularity Theorem for twisted sheaves proved in Section \ref{sec-proof-of-the-semiregularity-thm-twisted-sheaves-case-for-families-of-abelian-varieties}.

%****************************************************************
% 
%****************************************************************
\subsection{A general position assumption}
\label{sec-general-position}
Keep the notation of Section \ref{section-secant-sheaves-on-abelian-threefolds}.
Let $\Sigma_t\subset X$ be the image of $C_t$ under the involution $\iota$ of $X=\Pic^2(C)$ sending
a line bundle $L$ to $\omega_C\otimes L^{-1}$. Then $\Sigma_t$ and $C_t$ are not algebraically equivalent for a generic non-hyperelliptic $C$. Furthermore, the cycle $C_t-\Sigma_t$ is non-torsion in the group of algebraic cycles modulo algebraic equivalence, by \cite{ceresa}.
We will not need these facts and assume only that $C$ is non-hyperelliptic. 

Choose a point $p\in C$ and use it to identify  $X$ and $\Pic^0(C)$ via tensorization with $\StructureSheaf{C}(2p)$ endowing $X$ with a group structure. 
Set 
\[
C_p:=AJ(C)+p=\{\StructureSheaf{C}(p+q)\ : \ q\in C\}.
\]
Set  $C_i=\tau_{s_i}(C_p)$, for a point $s_i\in X$, $1\leq i\leq d+1$.
Set $\Sigma_p:=\iota(C_p)$ and $\Sigma_i=\tau_{t_i}(\Sigma_p)$, for a point $t_i\in X$, $1\leq i\leq d+1$.
%Choose $d+1$ distinct points $s_i$, $1\leq i\leq d+1$, in $\Pic^1(C)$ and set $\Sigma_i:=\Sigma_{s_i}$. 
Assume that the curves $C_i$ are pairwise disjoint and so are the $\Sigma_i$. 
%See Lemma \ref{lemma-ideal-of-two-translates-tensor-two-theta-translates} for the case $d=1$.

\begin{assumption}
\label{assumption-on-C-i-s}
Assume that $H^0(X,\Ideal{\cup_{i=1}^{d+1}C_i}(2\Theta)\otimes L)$ vanishes, for all $L\in \Pic^0(X)$.
\end{assumption}

Given a line bundle $L$ of degree $k$, set $C_L:=\{L(p) \ : \ p\in C\}\subset \Pic^{k+1}(C).$

\begin{lem}
Assumption \ref{assumption-on-C-i-s} holds for $d\geq 3$ for a generic choice of $C_i$'s.
Given a subset $\{\ell_i\}_{i=1}^{d+1}$ of $\Pic^1(C)$ the assumption  
holds for $\{C_{\ell_i}\}_{i=1}^{d+1}$, provided 
\begin{equation}
\label{eq-intersection-of-translates-of-theta-is-empty}
\bigcap_{1\leq i<j\leq d+1}\tau_{-\ell_i-\ell_j}(\Theta)=\emptyset.
\end{equation}
\end{lem}

\begin{proof}
\underline{Step 1:} Note first that for a generic choice of $\ell_i$'s the intersection (\ref{eq-intersection-of-translates-of-theta-is-empty}) is empty.
It suffices to prove it for $d=3$. Now the intersection of generic $4$ translates of $\Theta$ is empty. So it suffices to observe that the morphism
 $t:\Pic^1(C)\rightarrow\Pic^2(C)$ given by
\[
t(\ell_1,\ell_2,\ell_3,\ell_4)= (\ell_1+\ell_2,\ell_1+\ell_3,\ell_1+\ell_4,\ell_2+\ell_3)
\]
is surjective. Indeed, set $T=(t_1,t_2,t_3,t_4)$ and define $f:\Pic^2(C)^4\rightarrow \Pic^2(C)$ by $f(T)=t_1+t_2-t_4$ and  $u:\Pic^2(C)^4\rightarrow\Pic^2(C)^4$ by
\[
u(T)=(f(T),2t_1-f(T),2t_2-f(T),2t_3-f(T)).
\]
Then $f(t(\ell_1,\ell_2,\ell_3,\ell_4))=2\ell_1$
and $u(t(\ell_1,\ell_2,\ell_3,\ell_4))=(2\ell_1,2\ell_2,2\ell_3,2\ell_4)$. It follows that $u\circ t$ is surjective, and hence so is $t$, as the image of $t$ must be $4$ dimensional.

\underline{Step 2:} We show next that a divisor $D$ in $\linsys{2\Theta}$ contains the translate $C_L$, $L\in \Pic^1(C)$, if and only if $D$ is the divisor $D_E$ associated to some semistable rank $2$ vector bundle $E$ on $C$ with trivial determinant  and $L^{-1}$ is a subsheaf of $E$. 
Recall that for a semistable vector bundle $E$ on $C$ with trivial determinant the set 
\[
\{M\in\Pic^2(C)  \ : \ H^0(C,E\otimes M)\neq 0\}
\]
is the set theoretic support of a divisor $D_E$ in $\linsys{2\Theta}$, which depends only on the $S$-equivalence class of $E$ 
(see \cite{NR}).
The inclusion $C_L\subset D_E$ is clear, if $L^{-1}$ is a subsheaf of $E$. Conversely, assume that $C_L$ is contained in $D$, $D\in\linsys{2\Theta}$. Then $D$ belongs to the projective subspace $\PP[H^0(\Ideal{C_L}(2\Theta))]$ of $\linsys{2\Theta}$. Now $\PP[H^0(\Ideal{C_L}(2\Theta))]$ is $3$-dimensional (we postpone the proof to Lemma \ref{lemma-global-sections-of-ideal-I-C-2Theta}). So it suffices to prove the claim that the subset of the moduli space of $S$-equivalence classes of semistable vector bundles of rank $2$ and trivial determinant, which has a representative containing $L^{-1}$ as a subsheaf, is isomorphic to $\PP^3$. Indeed, every such equivalence class $[E]$ is represented by an extension class in 
$\PP\Ext^1(L,L^{-1})\cong \PP H^0(L^2\otimes\omega_C)^*\cong\PP^3$. This is clear if $[E]$ is represented by a vector bundle containing $L^{-1}$ as a saturated subsheaf.
The semi-stable but unstable $S$-equivalence classes in $\PP\Ext^1(L,L^{-1})$ are represented by the set
$L^{-1}(q)\oplus L(-q)$, $q\in C$, by \cite[Lemma 3.6(b)]{beauville}. This is precisely the set of $S$-equivalence classes of rank $2$ semistable vector bundles of trivial determinant containing $L^{-1}$ as an unsaturated subsheaf.
Finally, the map $E\mapsto D_E$ is a linear embedding of $\PP\Ext^1(L,L^{-1})$ in $\linsys{2\Theta}$,
by \cite[Lemma 3.7]{beauville}, proving the claim.

\underline{Step 3:} We observe next that if $E$ is a semistable vector bundle with trivial determinant on $C$, $L_1$ and $L_2$ are two line bundles of degree $1$ admitting embeddings $\iota_k:L_k^{-1}\rightarrow E$ as subsheaves, and the curves $C_{L_1}$ and $C_{L_2}$ are disjoint in $\Pic^2(C)$, then the line bundle $L_1\otimes L_2$ is effective. We have two cases. Case 1: If $L_1^{-1}$ is a subbundle of $E$ we get the short exact sequence
\[
0\rightarrow L_1^{-1}\RightArrowOf{\iota_1} E\RightArrowOf{j_1} L_1\rightarrow 0
\]
and $j_1\circ \iota_2:L_2^{-1}\rightarrow L_1$ does not vanish, hence $L_1$ is isomorphic to $L_2^{-1}(p+q)$, for some points $p,q\in C$. 

Case 2: If $L_1^{-1}$ is only a subsheaf, then $L_1^{-1}(p)$ is a subbundle, for some point $p\in C$. 
We get the short exact sequence
\[
0\rightarrow L_1^{-1}(p)\rightarrow E \RightArrowOf{j_1} L_1(-p)\rightarrow 0.
\]
If $L_2^{-1}$ is a subsheaf of $L_1^{-1}(p)$, then $L_2$ is isomorphic to $L_1(q-p)$, for some point $q\in C$, and the curves $C_{L_1}$ and $C_{L_2}$ both contain $L_1(q)\cong L_2(p)$ and so are not disjoint. Hence, $j_1 \circ \iota_2$ does not vanish and $L_2^{-1}$ is isomorphic to $L_1(-p-q)$, for some point $q\in C$.

\underline{Step 4:} Assume that the intersection (\ref{eq-intersection-of-translates-of-theta-is-empty}) is empty. Then there does not exist a point $t\in \Pic^0(C)$, such that $2t+\ell_i+\ell_j\in\Theta$, for all $1\leq i<j\leq d+1$. Thus there does not exist a semi-stable vector bundle $E$ of trivial determinant, such that $D_E$ contains the union $\bigcup_{i=1}^{d+1}C_{\ell_i+t}$, for some $t\in \Pic^0(C)$, by Step 3. Hence, $H^0(X,\Ideal{\cup_{i=1}^{d+1}C_{\ell_i+t}}(2\Theta))$
vanishes, for all $t\in \Pic^0(C)$, by Step 2. It follows that $H^0(X,\Ideal{\cup_{i=1}^{d+1}C_{\ell_i}}(\tau_t^*(2\Theta)))$ vanishes,  for all $t\in \Pic^0(C)$. Finally, the morphism $\Pic^0(C)\rightarrow \Pic(X)$, given by $t\mapsto \tau_t^*(2\Theta)$, is an isogeny onto the connected component of $2\Theta$.
\end{proof}

\begin{rem}
Consider the rank $4$ vector bundle $\U$ over $\Pic^1(C)$ with fiber $\U_L:=H^0(X,\Ideal{C_L}(2\Theta))$ over $L\in\Pic^1(C)$. We have seen in the proof above that $\PP(\U_L)$ is naturally isomorphic to $\PP H^0(C,L^2\otimes\omega_C)$. Hence, one gets a morphism $\epsilon:\PP(\U)\rightarrow \linsys{2\Theta}$, by 
\cite[Lemma 3.7]{beauville}. The morphism is generically finite of degree $8$ onto its image, which is the Coble quartic and is isomorphic to the moduli space of equivalence classes of semistable rank $2$ vector bundles of trivial determinant over $C$, by  \cite{NR} and \cite[Sec. 4.1]{pauly}. The Zariski closed subset $\epsilon^{-1}(\epsilon(\PP(\U_L)))$ is thus reducible of dimension $\geq 3$. Let $\pi:\PP(\U)\rightarrow \Pic^1(C)$ be the natural projection. The proof above shows that the components of $\epsilon^{-1}(\epsilon(\PP(\U_L)))$ other that $\PP(\U_L)$ all lie in $\pi^{-1}(\tau_{L^{-1}}(\Theta))$ and so do not surject onto $\Pic^1(C)$. Pauly shows that the generic rank $2$ stable vector bundle $E$ of trivial determinant contains the inverses of $8$ distinct line bundles $L_i$, $1\leq i\leq 8$, in $\Pic^1(C)$ and $\otimes_{i=1}^8 L_i\cong \omega_C^2$ \cite[Lemma 4.2]{pauly}. The divisor $D_E$, for such $E$,  contains precisely $8$ translates $C_{L_i}$, $1\leq i\leq 8$,  of $C$.
\end{rem}

\begin{lem}
\label{lemma-emptiness-condition-holds-when-the-curves-are-an-orbit}
Let $G$ be a cyclic subgroup of $\Pic^0(C)$ of order $d+1$, $d\geq 3$. 
The emptiness condition (\ref{eq-intersection-of-translates-of-theta-is-empty}) holds and the curves $C_{L_i}$ are pairwise disjoint
for every $G$-orbit $\{L_i\}_{i=1}^{d+1}$ in $\Pic^1(C)$, for a generic $C$.
\end{lem}

\begin{proof}
The emptiness condition (\ref{eq-intersection-of-translates-of-theta-is-empty}) would follow 
from the emptiness of
$
\bigcap_{g\in G}\tau_g(\Theta).
$
The latter emptiness condition is open in the moduli space of pairs $(C,G)$ and so it suffices to verify it for a boundary pair, where $C$ is a chain $E_1\cup E_2\cup E_3$ of three elliptic curves $E_i$, $1\leq i\leq 3$. Let $G_i$ be a cyclic group of order $d+1$ of $\Pic^0(E_i)$, and let $G$ be a diagonal embedding of $\ZZ/(d+1)\ZZ$ in $G_1\times G_2\times G_3\subset  \Pic^0(E_1)\times \Pic^0(E_2)\times \Pic^0(E_3)\cong \Pic^0(C)$. 
%We choose a point $p_0\in E$ and identify $\Pic^k(E)$ with $E$ via translation by $\StructureSheaf{E}(k p_0)$. 
Assume that $E_2$ intersects $E_1$ at the point $p_0$ of $E_2$ and $E_2$ intersect $E_3$ and the point $p_1$ of $E_2$, and $p_1- p_0$ does not belong to $G_2$. We will explain below that 
in this case $\Theta=\cup_{i=1}^3 D_i$, where
\begin{eqnarray}
\label{eq-theta-divisor-of-a-chain-of-elliptic-curves}
D_1&:=&\{\StructureSheaf{E_1}\}\times\Pic^2(E_2)\times\Pic^0(E_3), 
\\
\nonumber
D_2&:=&\Pic^0(E_1)\times\{\StructureSheaf{E_2}(p_0+p_1)\}\times \Pic^0(E_3), 
\\ 
\nonumber
D_3&:=&\Pic^0(E_1)\times\Pic^2(E_2)\times\{\StructureSheaf{E_3}\}.
%D_1=\{p_0\}\times E_2\times E_3, \ D_2=E_1\times\{p_0\}\times E_3, \ D_3=E_1\times E_2\times \{p_1\}.
\end{eqnarray}
Choose a generator $g$ of $G$.
\[
\bigcap_{k=1}^{d+1}\tau_{kg}(\Theta)= 
\bigcup_{(i_0,\dots,i_d)\in \{1,2,3\}^{d+1}} \bigcap_{k=1}^{d+1} \tau_{kg}(D_{i_k}).
\]
Each intersection $\bigcap_{k=1}^{d+1} \tau_{kg}(D_{i_k})$ is empty, as $i_j=i_k$ for some $j<k$, since $d+1>3$, and $D_i\cap\tau_{kg}(D_i)=\emptyset$, for $0<k< d+1$, for all $i$.

We show next that the $G$-orbit of the Abel-Jacobi image of $C$ consists of pairwise distinct curves.
Consider the following embedded copy $C'$ of $C$ in $\Pic^{(0,1,0)}(C)$.
\begin{eqnarray*}
\left[\Pic^0(E_1)\times\{\StructureSheaf{E_2}(p_0)\}\times\{\StructureSheaf{E_3}\}\right]
& \cup &
\left[
\{\StructureSheaf{E_1}\}\times\Pic^1(E_2)\times \{\StructureSheaf{E_3}\}\right]
\\
&\cup &
\left[
\{\StructureSheaf{E_1}\}\times\StructureSheaf{E_2}(p_1)\}\times \Pic^0(E_3)
\right].
\end{eqnarray*}
Denote the $i$-th component above by $C'_i$. 
%The intersection of the first and second components above is 
Then $C'_1\cap C'_2=(\StructureSheaf{E_1},\StructureSheaf{E_2}(p_0),\StructureSheaf{E_3})$,
%the intersection of the second and third components is  
$C'_2\cap C'_3=(\StructureSheaf{E_1},\StructureSheaf{E_2}(p_1),\StructureSheaf{E_3})$,
and $C'_1$ and $C'_3$ are disjoint. The embedding $AJ:C\rightarrow C'\subset \Pic^{(0,1,0)}(C)$ is given by 
\begin{equation}
\label{eq-AJ}
p\mapsto\left\{
\begin{array}{ccl}
(\StructureSheaf{E_1}(p-p_0),\StructureSheaf{E_2}(p_0),\StructureSheaf{E_3}) & \mbox{if} & p\in E_1,
\\
(\StructureSheaf{E_1},\StructureSheaf{E_2}(p),\StructureSheaf{E_3}) & \mbox{if} & p\in E_2,
\\
(\StructureSheaf{E_1},\StructureSheaf{E_2}(p_1),\StructureSheaf{E_3}(p-p_1)) & \mbox{if} & p\in E_3.
\end{array}
\right.
\end{equation}

For $0<k<d+1$, $C'_i\cap\tau_{kg}(C'_i)=\emptyset$, for all $i$, and  $C'_i\cap\tau_{kg}(C'_j)=\emptyset$, for $\{i,j\}=\{1,2\}$ and $\{2,3\}$. The intersection is empty also for $\{i,j\}=\{1,3\}$, since $p_1-p_0\not\in G_2$. Hence, the translates $\tau_{kg}(C')$, $0\leq k \leq d$, are pairwise disjoint.

We prove next that $C'$ is the limit of Abel-Jacobi images of smooth genus $3$ curves in a flat family that degenerates to $C$. 
Observe that the curve $C'$ is precisely the Brill-Noether locus of $L\in\Pic^{(0,1,0)}(C)$ with $h^0(L)\neq 0$. Indeed, if $L=(L_1,L_2,L_3)$ and $s$ is a non-zero global section of $L$, then $s$ is not identically zero on $E_2$. The line bundle $L_2$ is isomorphic to $\StructureSheaf{E_2}(p)$, for a unique point $p\in E_2$. If $p\not\in\{p_0,p_1\}$, then $s$ does not vanish at $p_0$ and $p_1$ and so $L_1$ and $L_2$ must both be trivial, so that $L$ is in $C'_2$. If $p=p_0$, then $s$ vanishes at $p_0$ and so it must be identically zero on $E_1$, but non-zero on $E_3$. Hence, $L_1$ is arbitrary, but $L_3$ is trivial, and so $L$ is in $C'_1$. Similarly, if $p=p_1$, then $L$ is in $C'_3$. Choose a family $\pi:\C\rightarrow S$ over a smooth one-dimensional analytic base $S$ with special fiber $C$ over $0\in S$
and generic fiber a smooth genus $3$ curve. Let $q:S\rightarrow \C$ be a section with $q(0)$ a point of $E_2\setminus\{p_0,p_1\}$. The section $q$ determines a section 
$Q:S\rightarrow \Pic(\C/S)$ with value $\StructureSheaf{C_s}(q(s))$ over $s\in S$, hence a family of abelian varieties $\Pi:\J\rightarrow S$ with connected fibers, whose generic fiber is $\Pic^1(C_s)$ and its 
special fiber is $\Pic^{(0,1,0)}(C)$. Over $\C\times_S\Pic^0(\C/S)$ we have the relative (normalized) Poincar\'{e} line bundle $\P_0$.
Let $\tau_Q: \Pic^0(\C/S)\rightarrow \J$ be the isomorphism of translation by $Q$.
Let $f_1 : \C\times_S\J\rightarrow \J$ and $f_2: \C\times_S\J\rightarrow \C$ be the projections.
Translating $\P_0$ by the section $Q$ and 
tensoring by the pullback of the line bundle $\StructureSheaf{\C}(q(S))$ over $\C$ and we get
a relative Poincar\'{e} line bundle $\P:=(id\times\tau_Q)_*(\P_0)\otimes f_2^*\StructureSheaf{\C}(q(S))$ 
over $\C\times_S\J$.  The coherent torsion sheaf $R^1f_*(\omega_{f_1}\otimes\P^{-1})$ is  supported as a line bundle on a relative curve $\C'$ over $S$, by Cohomology and Base Change. The generic fiber of $\C'$ is the Abel-Jacobi image of $C_s$ and the special fiber is $C'$.
%This is done in Lemma \ref{lemma-C-prime-is-the-limit-of-translates-of-AJ-images}.

Finally observe that $\Theta$ is the image of $\Sym^2(C')$ via the addition morphism 
\[
\Sym^2(\Pic^{(0,1,0)}(C))\rightarrow \Pic^{(0,2,0)}(C)
\] 
(the symmetric square of each irreducible component of $C'$ contracts to a curve in $\Theta$).
\end{proof}

%*************
% Hide
%*************
\hide{
%\begin{rem}
Set $t=
%(\StructureSheaf{E}(p_1-p_0),\StructureSheaf{E}(p_1-p_0),\StructureSheaf{E}(p_1-p_0))
(L_1,L_2,L_3)
\in\Pic^{(0,0,0)}(C)$, with $L_i\not\cong\StructureSheaf{E_i}$, and consider the following principal polarization of $\Pic^{(0,1,0)}(C)$.
\begin{eqnarray*}
\Theta&:=& [\{\StructureSheaf{E_1}\}\times\Pic^1(E_2)\times \Pic^0(E_3)] \ \ \cup \ \ 
[\Pic^0(E_1)\times \{\StructureSheaf{E_2}(p_0)\}\times \Pic^0(E_3)] \ \ \cup
\\
& & 
[\Pic^0(E_1)\times\Pic^1(E_2)\times\{\StructureSheaf{E_3}\}].
\end{eqnarray*}
We get that $\tau_t(\Theta)\cap \Theta$ is a cycle of six elliptic curves
\begin{eqnarray*}
\Pic^0(E_1)\times\{\StructureSheaf{E_2}(p_0)\}\times  \{L_3\} 
&\cup &\{\StructureSheaf{E_1}\}\times\Pic^1(E_2)\times \{L_3\} \ \ \cup
\\
\{\StructureSheaf{E_1}\}\times \{L_2(p_0)\}\times\Pic^0(E_3)
&\cup & \Pic^0(E_1)\times\{L_2(p_0)\}\times\{\StructureSheaf{E_3}\} \ \ \cup
\\
\{L_1\}\times \Pic^1(E_2)\times \{\StructureSheaf{E_3}\} 
& \cup & \{L_1\}\times \StructureSheaf{E_2}(p_0)\}\times \Pic^0(E_3),
%
%\{\StructureSheaf{E}\}\times \{\StructureSheaf{E}(p_1)\}\times\Pic^0(E)
%&\cup &\{\StructureSheaf{E}\}\times\Pic^1(E)\times \{\StructureSheaf{E}(p_1-p_0)\} \ \ \cup
%\\
%\Pic^0(E)\times\{\StructureSheaf{E}(p_0)\}\times  \{\StructureSheaf{E}(p_1-p_0)\} 
%& \cup &
%\{\StructureSheaf{E}(p_1-p_0)\}\times \StructureSheaf{E}(p_0)\}\times \Pic^0(E)
%\ \ \cup
%\\
% \{\StructureSheaf{E}(p_1-p_0)\}\times \Pic^1(E)\times \{\StructureSheaf{E}\} & \cup &
%\Pic^0(E)\times\{\StructureSheaf{E}(p_1)\}\times\{\StructureSheaf{E}\},
\end{eqnarray*}
each intersecting only the two adjacent in the cyclic order. Set $L_2:=\StructureSheaf{E_2}(p_1-p_0)$. Then the first three components are 
$\tau_{(\StructureSheaf{E},\StructureSheaf{E},L_3)}(C')$, and the last three components are the image of the first three under the involution 
\[
\iota'(M_1,M_2,M_3)= (L_1\otimes M_1^{-1},\StructureSheaf{E}(p_0+p_1)\otimes M_2^{-1},L_3\otimes M_3^{-1}.
\] 
The canonical line bundle of $C$ is $\omega_C=(\StructureSheaf{E_1}(p_0),\StructureSheaf{E_2}(p_0+p_1),\StructureSheaf{E_3}(p_1))$ in $\Pic^{(1,2,1)}(C)$.
The involution $\iota'$ satisfies $\iota'(L)= \omega_C\otimes L^{-1}\otimes (L_1(-p_0),\StructureSheaf{E},L_3(-p_1))$.
The involution $L\mapsto \omega_C\otimes L^{-1}$ maps each component of $\Pic(C)$ to a different component. Considering instead the involution $\iota'$, we get the analogue of the fact that when $C$ is smooth of genus 3 and $C'$ is a translate of $AJ(C)$, then the complete intersection of two
distinct translates of $\Theta$ containing $C'$ is  
the union $C'\cup\Sigma'$ of $C'$ and a translate $\Sigma'$ of $\Sigma$
% is the complete intersection of two translates of $\Theta$, provided 
%the union $C'\cup \Sigma'$ is contained in more than one translate of $\Theta$ 
(see Lemma \ref{lemma-intersection-of-C-and-Sigma-has-length-at-most-2}(\ref{case-a-of-intersection-lemma})). Note that $\tau_t(\Theta)=\iota'(\Theta)$.

Assume that $E_1$ is not isomorphic to $E_3$.
The isomorphism class of $C$ is determined by the data of (1) the isomorphism classes of $E_i$, $1\leq i\leq 3$, and (2) the pair of isomorphism classes $\{\StructureSheaf{E_2}(p_1-p_0),\StructureSheaf{E_2}(p_0-p_1)\}$, interchanged by the inversion automorphism of $E_2$. Let $E'_j$ be the $j$-th component of $\tau_t(\Theta)\cap \Theta$ in the order listed above. Then 
\begin{eqnarray*}
E'_1\cap E'_2&=&(\StructureSheaf{E_1},\StructureSheaf{E_2}(p_0),L_3)
\\
E'_2\cap E'_3&=&(\StructureSheaf{E_1},\StructureSheaf{E_2}(p_1),L_3)
\\
E'_3\cap E'_4&=&(\StructureSheaf{E_1},\StructureSheaf{E_2}(p_1),\StructureSheaf{E_3})
\\
E'_4\cap E'_5&=&(L_1,\StructureSheaf{E_2}(p_1),\StructureSheaf{E_3})
\\
E'_5\cap E'_6&=&(L_1,\StructureSheaf{E_2}(p_0),\StructureSheaf{E_3})
\\
E'_6\cap E'_1&=&(L_1,\StructureSheaf{E_2}(p_0),L_3)
\end{eqnarray*}
%\end{rem}
The difference of the middle entry of $E'_3\cap E'_2$ and $E'_2\cap E'_1$ is $\StructureSheaf{E_2}(p_1-p_0)$ and 
the difference of the middle entry of $E'_6\cap E'_5$ and $E'_5\cap E'_4$ is $\StructureSheaf{E_2}(p_0-p_1)$. 
Hence $E'_1\cup E'_2\cup E'_3$ and $E'_4\cup E'_5\cup E'_6$  are both isomorphic to $C$. 
Only the former satisfies the following property characterizing translates of the Abel-Jacobi embedding (note that 
only $E'_1\cup E'_2\cup E'_3$ and $E'_4\cup E'_5\cup E'_6$
have a middle component corresponding to varying the line bundle in the middle component of $C$). 

A characterization of translates of $AJ$: 
Consider the isomorphism $\phi_\Theta:\Pic^0(C)\rightarrow \Pic^0(\Pic^{(0,1,0)}(C))$, given by $\phi_\Theta(L)=\tau_L(\Theta)-\Theta$.
Then 
%the combinatorics mentioned in the previous paragraph is captured by the statement that 
$AJ:C\rightarrow C'\subset \Pic^{(0,1,0)}(C)$ has the property that 
\[
AJ\times\phi_L:C\times\Pic^0(C)\rightarrow \Pic^{(0,1,0)}\times \Pic^0(\Pic^{(0,1,0)}(C))
\]
pulls back the Poincar\'{e} line bundle $\P$ to a universal line bundle. In other words, the restriction of $(AJ\times\phi_\Theta)^*(\P)$ to $C\times\{L\}$
is isomorphic to $L$ (see the proof of \cite[Prop. 11.3.2]{BL}). 
This property is preserved if we compose $AJ$ by a translation. It is not satisfies by  $AJ$ composed with $\iota'$.
No other union $C''$ of three consecutive components of $\tau_t(\Theta)\cap \Theta$
admits an isomorphism $AJ':C\rightarrow C''\subset \Pic^{(0,1,0)}(C)$ with the above property.

\begin{lem}
\label{lemma-C-prime-is-the-limit-of-translates-of-AJ-images}
The image $C'$ of the Abel-Jacobi morphism $AJ$ given in (\ref{eq-AJ}) is the limit of translates of the Abel-Jabobi images of smooth genus $3$ curves in a flat family.
\end{lem}

\begin{proof}
%{\bf A degeneration of the Abel-Jacobi morphism:}
For a smooth genus $3$ curve $C$ with a choice $\Theta$ of a translate of its theta divisor in $\Pic^1(C)$ there is a divisor $D_C=\{\StructureSheaf{C}(p-q) \ : \ p,q\in C\}$ in $\Pic^0(C)$ consisting of $t$, such that $\tau_t(\Theta)\cap \Theta$ is a union of a translate $C'$ of $AJ(C)$ and a translate $\Sigma'$ of $\Sigma$ for $t\in D_C\setminus\{0\}$. 
When $C$ is a chain of three elliptic curves as above, the locus of $t\in\Pic^{(0,0,0)}(C)$ such that 
$\tau_t(\Theta)\cap\Theta$ is the union of a translate of a chain of elliptic curve and a translate of its image under $-1$ is an open subset of $\Pic^{(0,0,0)}(C)$,
but only points in the open subset $D^0$ of the divisor $D=\{(L_1,L_2,L_3) \ : \ L_2\cong \StructureSheaf{E}(p_1-p_0)\}$ with $L_1\not\cong\StructureSheaf{E}$ and $L_3\not\cong\StructureSheaf{E}$, yield the correct isomorphism class of $C$.
Choosing a flat family of triples $(J_s,\Theta_s,t_s)$, $s\in S$, of Jacobians $J_s$ of curves of arithmetic genus $3$ with a choice of a translate $\Theta_s\subset \Pic^1(C_s)$ of the $\Theta$ divisor  and a point $t_s\in D_{C_s}\subset \Pic^0(C_s)$, with $C_s$ smooth for $s\neq 0\in S$, degenerating at $s=0$ to $(\Pic^{(0,1,0)}(C),\Theta,t_0)$, with $t_0\in D^0$ (??? why does $D^0$ intersect the closure of $\cup_{s\in S\setminus\{0\}}D_{C_s}$ in $\Pic^0(\C/S)$ ???), 
%with $L_2\cong\StructureSheaf{E}(p_1-p_0)$ and $L_1$, $L_3$ both non-trivial, 
we get that $\Theta_s\cap\tau_{t_s}(\Theta)$ is a family $C'_s\cup\Sigma'_s$,
where $C'_s$ is a translate of $AJ(C_s)$, and a translate $\Sigma'_s$ of $\Sigma_s$, for all $s\in S$. This shows that the Abel-Jacobi image in the proof of Lemma \ref{lemma-emptiness-condition-holds-when-the-curves-are-an-orbit} is indeed the limit of the a family of standard Abel-Jacobi images of smooth genus $3$ curves.
\end{proof}

%*************
% End Hide
%*************
}

%****************************************************************
% 
%****************************************************************
\subsection{A reflexive secant$^{\boxtimes 2}$-sheaf over $X\times\hat{X}$}
\label{sec-reflexive-secant-square-sheaf}
We continue to assume Assumption \ref{assumption-on-C-i-s}.
Set
\begin{equation}
\label{eq-F-1-and-F-2}
F_1:=\Ideal{\cup_{i=1}^{d+1}C_i}(\Theta) \ \ \mbox{and} \ \ F_2:=\Ideal{\cup_{i=1}^{d+1}\Sigma_i}(\Theta).
\end{equation}
Note that $ch(F_1)=ch(F_2)$.  
Let $\F_2:=\pi_{12}^*(a^*(F_2))\otimes\pi_{13}^*\P^{-1}$ be the sheaf over $X\times X\times\hat{X}$  using the notation of (\ref{eq-universal-spreading-of-F}). 

\begin{assumption}
\label{assumption-intersections-of-surfaces-are-generic}
\begin{enumerate}
\item
\label{assumption-item-intersections-of-surfaces-are-generic}
We choose the curves $C_i$ and $\Sigma_j$, so that the intersection of any four of the surfaces $\Theta_{i,j}:=\Sigma_j-C_i$ in $X$
is empty and the triple intersections are zero dimensional. 
\item
\label{assumption-3s-ij-pairwise-distinct}
The 
%$\Choose{d+1}{2}$ points $3s_{ij}$, $i\neq j$, are pairwise 
$(d+1)^2$ points $t_j+s_i$, $1\leq i,j\leq d+1$, are distinct.\footnote{
Recall that $C_i=\tau_{s_i}(C_p)$ and $\Sigma_j=\tau_{t_j}(\Sigma_p)$.
}
\end{enumerate}
\end{assumption}

Let $f_{i,j}:C_i\times \Sigma_j\rightarrow X\times\hat{X}$ be the morphism given by 
$f_{i,j}(x,y)=(y-x,L_{i,j}(x-y))$, where $L_{i,j}(x-y)\in \Pic^0(X)$ is the line bundle
\begin{equation}
\label{eq-L-i-j}
L_{i,j}(x-y):=\StructureSheaf{X}(\tau_{2(x-y)-t_j-s_i}(\Theta)-\Theta).
\end{equation}
Denote by $\tilde{\Theta}_{i,j}$ the image of $f_{i,j}$ and observe that it is isomorphic to $\Theta_{i,j}$. Note that the surface $\pi_X(\tilde{\Theta}_{i,j})=\Theta_{i,j}$ is a translate\footnote{
Let $q_1, q_2$ be points of $C$. Consider the special case where $\tau_{s_i}$ and $\tau_{t_j}$ are both the identity.
The morphism $\pi_X\circ f_{i,j}:C_p\times \Sigma_p\rightarrow X$ sends the pair of degree $2$ line bundles $(\StructureSheaf{C}(p+q_1),\omega_C(-p-q_2))$ to $\omega_C(-2p-q_1-q_2)$, hence $\pi_X\circ f_{i,j}$ is a branched double cover onto its image $\tau_{-2p}(\Theta)$.
} 
of $\Theta$ and is thus smooth. Assumption \ref{assumption-intersections-of-surfaces-are-generic}(\ref{assumption-3s-ij-pairwise-distinct}) implies that the surfaces $\tilde{\Theta}_{i,j}$ are pairwise disjoint. 
In particular, the union 
\[
\tilde{\Theta}:=\cup_{1\leq i,j\leq d+1}\tilde{\Theta}_{i,j}
\]
is smooth.

\begin{prop}
\label{prop-local-freeness}
\begin{enumerate}
\item
\label{prop-local-freeness-E}
The sheaf cohomology $R^i\pi_{23,*}(\pi_1^*F_1\otimes \F_2)$ of the object 
\[
\G:=\Phi(F_2\boxtimes F_1)[-3]\stackrel{(\ref{eq-object-is-image-of-F-1-dual-times-F-2})}{=}R\pi_{23,*}(\pi_1^*F_1\otimes \F_2)
\]
vanishes, for $k=0$ and for $k> 2$. The sheaf $\G_1:=R^1\pi_{23,*}(\pi_1^*F_1\otimes \F_2)$ over $X\times\hat{X}$ is reflexive of rank $8d$ and it is locally free away from $\tilde{\Theta}$ and $\tilde{\Theta}$ 
%$\cup_{1\leq i,j\leq d+1}(S_{i,j})$. 
is the set theoretic support of 
the sheaf $\G_2:=R^2\pi_{23,*}(\pi_1^*F_1\otimes \F_2)$.
\item
\label{prop-local-freeness-G}
The object $\G^\vee[-1]$ 
%is isomorphic to $\Phi(F_2^\vee\boxtimes F_1^\vee)[-1]$ and 
is represented by the reflexive coherent sheaf $\E$ of rank $8d$, which is isomorphic to $\G_1^*$ and is hence locally free away from $\tilde{\Theta}$. The sheaf $\G_2$ is isomorphic to $\SheafExt^1(\E,\StructureSheaf{X\times\hat{X}})$.
\end{enumerate}
\end{prop}

\begin{rem}
\label{remark-all-the-results-for-E-hold-for-G}
Note that the object $\G$ is related to the object $E$ in (\ref{eq-E}) by interchanging $F_1$ and $F_2$. All the results of section \ref{section-secant-sheaves-on-abelian-threefolds} for $E$ holds for $\G$ as well, by Remark \ref{rem-interchanging-F-1-and-F-2}.
\end{rem}

The proof of Proposition \ref{prop-local-freeness} requires  the following lemmas.

\begin{lem}
\label{lemma-intersection-of-C-and-Sigma-has-length-at-most-2}
\begin{enumerate}
\item
\label{lemma-item-length-at-most-2}
Given $t\in \Pic^0(X)$, the intersection $C_i\cap \tau_t(\Sigma_j)$ is either empty, or a subscheme of length $2$.
\item
\label{lemma-item-length-is-2-iff}
The intersection subscheme $C_i\cap \tau_t(\Sigma_j)$ has length $2$, if and only if 
there exists $u\in\Pic^0(C)$, such that $C_i\cup \tau_t(\Sigma_j)$ is contained in $\tau_u(\Theta)$. 
\item
\label{lemma-item-one-or-two-theta-translates}
If  $C_i\cup \tau_t(\Sigma_j)$ is contained in $\tau_u(\Theta)$, then 
one of the following holds.
\begin{enumerate}
\item
\label{case-a-of-intersection-lemma}
The union $C_i\cup \tau_t(\Sigma_j)$ is the complete intersection $\tau_u(\Theta)\cap\tau_{u'}(\Theta)$ 
of a unique pair of translates of $\Theta$. 
Furthermore, $(\tau_u\iota\tau_{-u})(C_i)\neq\tau_t(\Sigma_j)$.
\item
\label{case-b-of-intersection-lemma}
The divisor $\tau_u(\Theta)$ is the unique translate of $\Theta$ containing $C_i\cup \tau_t(\Sigma_j)$. 
Furthermore, $(\tau_u\iota\tau_{-u})(C_i)=\tau_t(\Sigma_j)$ and the canonical line bundle of $C_i\cup \tau_t(\Sigma_j)$ is the restriction of $\StructureSheaf{X}(2\tau_u(\Theta))$.
\end{enumerate}
\end{enumerate}
\end{lem}

\begin{proof}
Given $p\in C$, set $C_p:=C_{AJ(p)}$.  We may assume, without loss of generality, that $C_i=C_p$, for some $p\in C$. 
Let $s$ be a point of $\Pic^0(C)$, such that $\tau_t(\Sigma_j)=\tau_s(\Sigma_p)$.
%Let $\iota:X\rightarrow X$ be the involution given by $\iota(L)=\omega_C\otimes L^{-1}$. 

\underline{Step 0:} We prove part (\ref{lemma-item-one-or-two-theta-translates})  with the exception of the uniqueness in part (\ref{case-a-of-intersection-lemma}) which is postponed to Step 4 below.
Observe  the equivalence 
\begin{equation}
\label{equivalence-of-inclusions}
C_p\cup \tau_s(\Sigma_p)\subset \tau_u(\Theta)
\Leftrightarrow
C_p\cup \tau_s(\Sigma_p)\subset \tau_{s-u}(\Theta).
\end{equation}
Indeed, the right inclusion is obtained by applying $\tau_s\circ\iota$ to both sides of the left inclusion and using the equalities $\iota(\Theta)=\Theta$ and $\tau_s\circ\iota=\iota\circ\tau_{-s}$.

If $s\neq 2u$, then the union $C_p\cup \tau_s(\Sigma_p)$ is contained in the two distinct translates $\tau_u(\Theta)$ and
$\tau_{s-u}(\Theta)$ and so the union must be their complete intersection, as the cohomology classes of both are equal. If $\tau_u(\Theta)$ is the unique translate of $\Theta$ containing $C_p\cup \tau_s(\Sigma_p)$, then
%\footnote{We will see in Step 4 that in fact $s=0$ and $u=0$.} 
$s=2u$ and so $(\tau_u\iota\tau_{-u})(C_p)=\tau_{2u}(\iota(C_p))=\tau_s(\Sigma_p)$. 
If there exist two distinct translates of $\Theta$ containing $C_p\cup \tau_s(\Sigma_p)$, then the latter is their complete intersection, by the equality of the cohomology classes. 

Assume that $(\tau_u\iota\tau_{-u})(C_p)=\tau_s(\Sigma_p)$. Then $\tau_{2u}(\Sigma_p)=\tau_s(\Sigma_p)$. Hence, $s=2u$.
We will see 
%in Step 4(b) that this implies $s=u=0$ and 
in Step 3 that $\tau_u(\Theta)$ is the unique translate of $\Theta$ containing both $C_p$ and $\tau_{2u}(\Sigma_p)$. This would prove the inequality in part \ref{case-a-of-intersection-lemma}.

In both cases $Z':=C_i\cup\tau_t(\Sigma_j)$ is a divisor on a smooth surface $\tau_u(\Theta)$, and so its canonical sheaf is a line bundle.
In case (\ref{case-a-of-intersection-lemma}) the canonical line bundle $\omega_{Z'}$  is the restriction of
$\StructureSheaf{X}(\tau_u(\Theta)+\tau_{u'}(\Theta))$, as the normal bundle is the restriction of 
$\StructureSheaf{X}(\tau_u(\Theta))\oplus \StructureSheaf{X}(\tau_{u'}(\Theta))$.
Hence, $H^0(Z',\omega_{Z'}(-\tau_u(\Theta)-\tau_{u'}(\Theta)))$ is one dimensional. Case (\ref{case-b-of-intersection-lemma})
is a limit case and we conclude that $H^0(Z',\omega_{Z'}(-2\tau_u(\Theta)))$ does not vanish, by semi-continuity. The line bundle 
$\omega_{Z'}(-2\tau_u(\Theta))$ restricts to each irreducible component of $Z'$ as a line bundle of degree $0$, hence the existence of a non-zero global section implies that it is the trivial line-bundle.

\underline{Step 1:}
Let $q\in C$ be a point not equal to $p$ and set $y=p-q$. We prove first the equalities
\begin{eqnarray}
\label{eq-intersection-is-union-of-two-curves}
\Theta\cap \tau_y(\Theta)&=&C_p\cup \Sigma_q,
\\
\label{eq-C-p-and-Sigma-q-intersect-at-two-points}
C_p\cap\Sigma_q&=&\{\StructureSheaf{C}(p+r),\StructureSheaf{C}(p+t)\},
\end{eqnarray}
such that $\omega_C\cong\StructureSheaf{C}(p+q+r+t)$. The points $r$ and $t$ are the two other points on the intersection of the line through $\varphi_{\omega_C}(p)$ and $\varphi_{\omega_C}(q)$ with the canonical curve $\varphi_{\omega_C}(C)$. The curves $C_p$ and $\Sigma_q$ are tangent at $\StructureSheaf{C}(p+r)$, if $r=t$. 

Proof of (\ref{eq-intersection-is-union-of-two-curves}):
It suffices to prove the inclusion $(C_p\cup \Sigma_q)\subset (\Theta\cap \tau_y(\Theta))$, as the cohomology classes of both sides are equal. 
The curves $C_p$ and $C_q$ are contained in $\Theta$ and  $\iota(\Theta)=\Theta$, hence $\Sigma_p$ and $\Sigma_q$ are  contained in $\Theta$.
Now $C_q$ is contained in $\Theta$ and so $C_p=\tau_y(C_q)$ is contained in $\tau_y(\Theta)$. 
%Hence, $C_p$ is contained in $\Theta\cap \tau_y(\Theta)$.
In addition, $\iota(\tau_{-y}(\Theta))=\tau_y(\iota(\Theta))=\tau_y(\Theta)$, and so 
\begin{equation}
\label{eq-Sigma-q-is-translate-of-Sigma-p-by-p-q}
\Sigma_q=\iota(C_q)=\iota(\tau_{-y}(C_p))=\tau_y(\iota(C_p))=\tau_y(\Sigma_p)
\end{equation}
is contained in  $\tau_y(\Theta).$

%Observe next that $C_p\cap\Sigma_q$ consists of the two points $\StructureSheaf{C}(p+r)$ and $\StructureSheaf{C}(p+t)$,
%Indeed, 
Proof of (\ref{eq-C-p-and-Sigma-q-intersect-at-two-points}):
\[
p+r\in\Sigma_q\Leftrightarrow \exists t\in C, \ \omega_C(-q-t)\cong\StructureSheaf{C}(p+r)
\Leftrightarrow 
\exists t\in C, \  \omega_C\cong\StructureSheaf{C}(p+q+r+t).
\]
If $r=t$ the curves $C_p$ and $\Sigma_q$ are tangent at $\StructureSheaf{C}(p+r)$, since their intersection number in $\Theta$ is $2$.

\underline{Step 2:}
We show next that given $y\in\Pic^0(C)$, the curve $C_p$ is contained in $\tau_y(\Theta)$, if and only if $y=\StructureSheaf{C}(p-r)$, for some $r\in C$. 
\begin{equation}
\label{eq-criterion-for-inclusion}
C_p\subset \tau_y(\Theta) \Leftrightarrow \exists r\in C, \ y=\StructureSheaf{C}(p-r).
\end{equation}
The ``if'' implication is clear. We prove the ``only if'' implication. Assume that $C_p$ is contained in $\tau_y(\Theta)$.
Let $L$ be the line bundle represented by $-y$. Then $\tau_{-y}(C_p)$ is contained in $\Theta$, if and only if the line bundle $\StructureSheaf{C}(p+q)\otimes L$ is effective, for all $q\in C$.  In this case the point $p$ will belong to
the support of the effective divisor in the linear system $|\StructureSheaf{C}(p+q)\otimes L|$, for some $q$ as we vary $q$ in $C$.
Hence, if $C_p$ is contained in $\tau_y(\Theta)$, then $L\cong \StructureSheaf{C}(r-q)$, for some points $q, r\in C$.
Furthermore, $\StructureSheaf{C}(r-q+p+q')$ is effective, for all $q'\in C$. So $q$ is a base point of
$\StructureSheaf{C}(r+p+q')$, for the generic $q'\in C$ for which $h^0(\StructureSheaf{C}(r-q+p+q'))=1$ (i.e., for $q'\in C$, such that the images of $r, p, q'$ on the canonical curve are not colinear, or if $p=r$, then for $q'\in C$, such that $q'$ is not on the line tangent to the canonical curve at $p$). 
For such generic $q'$ the base locus is $\{r, p, q'\}$. 
Hence $q$ is equal to one of $r$ or $p$. If $q=p$ we are done. If $q=r$, then $L$ is trivial and we are done. 

Applying $\iota$ to (\ref{eq-criterion-for-inclusion}) and replacing $y$ by $-y$ we get
\begin{equation}
\label{eq-Sigma-p-is-contained-in-theta-translate-iff}
\Sigma_p\subset \tau_{y}(\Theta) \Leftrightarrow \exists r\in C, \ y=r-p.
\end{equation}

Combing (\ref{eq-criterion-for-inclusion}) and (\ref{eq-Sigma-p-is-contained-in-theta-translate-iff}) we see that $\Theta$ is the only translate of $\Theta$ containing both $C_p$ and $\Sigma_p$. Indeed, if $p-r=r'-p$, for points $r,r'\in C$, then $2p=r+r'$ and so $p=r=r'$, since $C$ is non-hyperelliptic.

\underline{Step 3:}
Combining the results of the above two steps we conclude that if $\Theta\cap\tau_t(\Theta)=C_p\cup C'$, for some $t\in \Pic^0(C)$ and some curve $C'$, then $t=p-q$, for some $q\in C$, and $C'=\Sigma_q$. Furthermore, the scheme $C_p\cap\Sigma_q$
has length $2$. 

We prove next that if $C_p\subset \tau_u(\Theta)$, then $\tau_u(\Theta)$ is the unique translate of $\Theta$ containing $C_p\cup \tau_{2u}(\Sigma_p)$. Indeed, the assumed inclusion implies that $u=p-r$, for some $r\in C$, by (\ref{eq-criterion-for-inclusion}). 
Hence, $\tau_u(\Sigma_p)=\Sigma_r\subset\Theta$ and so $\tau_{2u}(\Sigma_p)\subset\tau_u(\Theta)$.
Assume that 
$
C_p\cup \tau_{2u}(\Sigma_p)\subset \tau_u(\Theta)\cap\tau_t(\Theta).
$
Applying $\tau_{-u}$ to both sides we get the inclusion
$
C_r\cup\Sigma_r\subset \Theta\cap \tau_{t-u}(\Theta).
$
Thus, $t-u=r-r=0$, by the previous paragraph. So $t=u$.

\underline{Step 4:} We show next that 
if $C_p\cup\tau_s(\Sigma_p)=\tau_{t_1}(\Theta)\cap\tau_{t_2}(\Theta)$, for some $t_1, t_2\in\Pic^0(C)$, then 
\begin{enumerate}
\item
\label{step-4-item-1}
$t_1=p-q_1$ and $t_2=p-q_2$, for some $q_1, q_2\in C$, 
\item
\label{step-4-item-2}
$s=2p-q_1-q_2$ and $s\neq 0$,
\item
\label{step-4-item-3}
$C_p\cap\tau_s(\Sigma_p)=\left\{p+r, p+t
\right\}$, where $\StructureSheaf{C}(q_1+q_2+r+t)\cong\omega_C$. Furthermore, $C_p$ is tangent to $\tau_s(\Sigma_p)$ at
$p+r$, if $r=t$.
\end{enumerate}
This proves the uniqueness of the pair of translates of $\Theta$ whose complete intersection is $C_p\cup\tau_s(\Sigma_p)$ in Part (\ref{case-a-of-intersection-lemma}).

(\ref{step-4-item-1}) follows from Equation (\ref{eq-criterion-for-inclusion}).

(\ref{step-4-item-2}) 
Applying $\tau_{-t_1}$ we get that 
$C_{q_1}\cup \tau_{s-t_1}(\Sigma_p)=\Theta\cap \tau_{q_1-q_2}(\Theta)$. So $\tau_{s-t_1}(\Sigma_p)=\Sigma_{q_2}$, 
by (\ref{eq-intersection-is-union-of-two-curves}) and $s-t_1=p-q_2$, by (\ref{eq-Sigma-q-is-translate-of-Sigma-p-by-p-q}). This proves the equality in (\ref{step-4-item-2}). If $s=0$, then $p=q_1=q_2$ and so $t_1=t_2$, which contradicts the assumed equality 
$C_p\cup\tau_s(\Sigma_p)=\tau_{t_1}(\Theta)\cap\tau_{t_2}(\Theta)$.

(\ref{step-4-item-3}) We have
$C_{q_1}\cap \tau_{s-t_1}(\Sigma_p)=C_{q_1}\cap \tau_{p-q_2}(\Sigma_p)=C_{q_1}\cap \Sigma_{q_2}=\{q_1+r, q_1+t\}$, 
where $\StructureSheaf{C}(q_1+q_2+r+t)\cong\omega_C$. The first equality follows from (\ref{step-4-item-1})  and (\ref{step-4-item-2}), the second equality from (\ref{eq-Sigma-q-is-translate-of-Sigma-p-by-p-q}), and the third equality from Step 1. If $r=t$ the two curves $C_{q_1}$ and $\Sigma_{q_2}$ are tangent at $q_1+r$, by Step 1. Hence, $C_p$ is tangent to $\tau_s(\Sigma_p)$ at $p+r$.
%and $C_{q_1}\cap \Sigma_{q_2}$ has length $2$. Hence, $C_p\cap\tau_s(\Sigma_p)$ has length $2$. 

%(\ref{step-4-item-4}) 
%We know that $p+r$ and $p+t$ belong to $\tau_s(\Sigma_p)=\tau_s(\iota(C_p))=\iota(\tau_{-s}(C_p)),$ by (\ref{step-4-item-3}).
%Hence, $\tau_s(\iota(p+r))$ belong to $C_p$. Set $L:=\StructureSheaf{C}(s)$. 
%Then $\omega_C(-p-r)\otimes L\cong \StructureSheaf{C}(p+u_1)$, for some $u_1\in C$. 
%Similarly,  $\omega_C(-p-t)\otimes L\cong \StructureSheaf{C}(p+u_2)$, for some $u_2\in C$.
%Thus, 
%\[
%\StructureSheaf{C}(u_1+r)\cong\omega_C\otimes L(-2p)\cong \StructureSheaf{C}(u_2+t).
%\]
%We conclude that $u_1+r=u_2+t$.

%\underline{Step 4(b):} 
%Assume that $C_p\cup \tau_s(\Sigma_p)$ is contained in a unique translate $\tau_u(\Theta)$ of $\Theta$. 
%We have seen in Step 0 that $s=2u$.
%We show that 
%$s=0$ and $u=0$.
%Step 2 implies that $u=p-q$, for some $q\in C$. The inclusion $\tau_{s-u}(\Sigma_p)\subset\Theta$ is equivalent 
%to $\tau_{u-s}(C_p)\subset\Theta$, which implies that $s-u=p-q'$, for some $q'\in C$. 
%So $s=q'-q$ (NO!).  Combined with the equality $s=2u$ we get that $q+q'=2p$. It follows that $p=q=q'$, as $C$ is not hyperelliptic. Thus $s=u=0$.
%**********
% Hide
%**********
\hide{
The following inclusions are equivalent:
\begin{eqnarray*}
C_p\cup \tau_{q'-q}(\Sigma_p) & \subset & \tau_{p-q}(\Theta)=\tau_{p-q}(\iota(\Theta))=\iota(\tau_{q-p}(\Theta))
\\
\Sigma_p\cup \iota(\tau_{q'-q}(\Sigma_p))& \subset &\tau_{q-p}(\Theta)
\\
\Sigma_p\cup \tau_{q-q'}(C_p)& \subset &\tau_{q-p}(\Theta)
\\
\tau_{q'-q}(\Sigma_p)\cup C_p& \subset &\tau_{q'-p}(\Theta)=\tau_{q'+q-2p}(\tau_u(\Theta)).
\end{eqnarray*}
The assumed uniqueness of $u$ implies that $q'+q=2p$. Hence, $p=q'=q$. So $s=0$ and $u=0$. 
The equality $(\tau_u\iota\tau_{-u})(C_p)=\tau_s(\Sigma_p)$ follows.

and $(\tau_u\iota\tau_{-u})(C_p)=\tau_s(\Sigma_p)$. As above we get that $u=p-q$, for some $q\in C$. The inclusion $\tau_{s-u}(\Sigma_p)\subset\Theta$ is equivalent to $\tau_{u-s}(C_p)\subset\Theta$, which implies that $u-s=p-q'$, for some $q'\in C$. Applying $\iota\tau_{-u}$ to the equality 
$(\tau_u\iota\tau_{-u})(C_p)=\tau_s(\Sigma_p)$ and using $\tau_{-u}(C_p)=C_q$ we get
\[
C_q=\iota(\tau_{s-u}(\Sigma_p))=\tau_{u-s}(C_p).
\]
So $p-q'=u-s=q-p$. Thus, $p=q=q'$. So $s=u=0$. This proves the uniqueness of $u$ in 
the ``only if'' direction of part (\ref{lemma-item-length-is-2-iff}) of the lemma if (\ref{case-b-of-intersection-lemma}) holds. 
%**********
% End Hide
%**********
}

%************
% Hide:
%************
\hide{
\underline{Step 5:} Suppose that $C_p\cap\tau_s(\Sigma_p)$ contains a length $2$ subscheme $Z$. So either $Z$ consists of  
two distinct points $\StructureSheaf{C}(p+r)$ and
$\StructureSheaf{C}(p+t)$, or a tangency point $\StructureSheaf{C}(p+r)=\StructureSheaf{C}(p+t)$. 
%If $s=0$ then (\ref{case-b-of-intersection-lemma}) holds for $u=0$. Assume that $s\neq 0$. 
We prove that $C_p\cup\tau_s(\Sigma_p)$ is 
%the intersection of two translates 
contained in a translate
of $\Theta$.
Let $s'\in\Pic^0(C)$ be the point corresponding to $\omega_C^{-1}\otimes  \StructureSheaf{C}(r+t+2p)$.
So $s'=2p-q_1-q_2$, where $q_1+q_2$ is the effective divisor satisfying $\omega_C\cong \StructureSheaf{C}(r+t+q_1+q_2)$.
Then $Z$ in contained also in $C_p\cap \tau_{s'}(\Sigma_p)$, by Step 4(\ref{step-4-item-3}). Furthermore, 
\[
C_p\cup\tau_{s'}(\Sigma_p)\subset \tau_{p-q_1}(\Theta)\cap \tau_{p-q_2}(\Theta).
\]
Indeed, $\tau_{q_1-p}(\tau_{s'}(\Sigma_p))=\tau_{p-q_2}(\Sigma_p)=\Sigma_{q_2}$, by (\ref{eq-Sigma-q-is-translate-of-Sigma-p-by-p-q}), and $\Sigma_{q_2}$  is contained in $\Theta$. Similarly, $\tau_{q_2-p}(\tau_{s'}(\Sigma_p))$   is contained in $\Theta$.
%is the intersection of two translates of $\Theta$, by Step 4. 
We claim that $s=s'$. 
Observe that $\tau_s(\Sigma_p)\cap \tau_{s'}(\Sigma_p)$ contains the subscheme $Z$. Applying $\iota$ we get that 
$\tau_{-s}(C_p)\cap \tau_{-s'}(C_p)$ contains $\iota(Z)$. Hence, $C_p\cap \tau_{s-s'}(C_p)$ contains a length $2$ subscheme. 
We conclude that $s=s'$, by Lemma \ref{lemma-two-translates-of-C}. 
%Now, $s'\neq 0$, by the assumption that $s\neq 0$. 
%We show next that $q_1\neq q_2$, so that $C_p\cup\tau_s(\Sigma_p)$ is the intersection of two translates of $\Theta$. 
%If $q_1=q_2$, then $s=2u$, where $u=p-q_1$. The argument in Step 4(b) shows that $s=0$, which contradicts our assumption.
%************
% End Hide:
%************
}

%\underline{Step 6:} 
%Suppose that $C_p$ and $\tau_s(\Sigma_p)$ are tangent at $\StructureSheaf{C}(p+r).$ The argument in Step 5 still applies, 
%using the case $r=t$ in Step 4, to conclude that $s$ corresponds to the line bundle 
%$\omega_C^{-1}\otimes  \StructureSheaf{C}(2r+2p)$ and to conclude that 
%$C_p\cup \tau_s(\Sigma_p)$ is the intersection of two translates to $\Theta$.

\underline{Step 5:} 
If $C_p\cap \tau_s(\Sigma_p)$ is non-empty, then $C_p\cup \tau_s(\Sigma_p)$ 
is contained in a translate of $\Theta$, by Lemma \ref{lemma-intersection-is-either-empty-or-of-length-2}.
%We have seen that if the subscheme $C_p\cap \tau_s(\Sigma_p)$ has length $\geq 2$, then $C_p\cup \tau_s(\Sigma_p)$ 
%is contained in a translate of $\Theta$. 
It remains to prove that in this case the length of the subscheme $C_p\cap \tau_s(\Sigma_p)$ is $2$. As in the proof of Step 4(3) we can translate so that the subscheme in question is $C_{q_1}\cap \Sigma_{q_2}$ for two points $q_1, q_2\in C$. Now the intersection number of these two curves in $\Theta$ is $2$. This completes the proof of Parts (\ref{lemma-item-length-at-most-2}) and (\ref{lemma-item-length-is-2-iff}).
\end{proof}

%**************
% Hide
%**************
\hide{
\begin{lem}
\label{lemma-two-translates-of-C}
Two distinct translates of $AJ(C)$ in $\Pic^1(C)$ intersect along a subscheme of length at most $1$.
\end{lem}

\begin{proof}
Let $t$ be a point of $\Pic^0(C)$ corresponding to a line bundle $L$. If there exist two points $p_i\in C$ such that $\StructureSheaf{C}(p_i)$ belongs to $AJ(C)\cap \tau_t(AJ(C)),$ for $i=1,2$, 
then $\StructureSheaf{C}(p_i)\cong \StructureSheaf{C}(q_i)\otimes L$, for some $q_i\in C$. Hence, $L\cong \StructureSheaf{C}(p_i-q_i)$, for $i=1,2$.
Thus $\StructureSheaf{C}(p_1+q_2)$ is isomorphic to $\StructureSheaf{C}(p_2+q_1)$. 
It follows that the divisors $p_1+q_2$ and $p_2+q_1$ are equal, since $C$ is non-hyperelliptic.
If $p_1\neq p_2$, then $q_1=p_1$ and $q_2=p_2$ and so $L$
 must be the trivial line bundle. It remains to show that 
if $p_1=p_2$ the two translates are not tangent at $\StructureSheaf{C}(p_1)$. Translating by $q_1$ into $\Pic^2(C)$ we need to show that $C_{p_1}$ and $C_{q_1}$ are not tangent at $\StructureSheaf{C}(p_1+q_1)$. Indeed, both are contained in $\Theta$, which is isomorphic to $C^{(2)}$, and the intersection number of $C_{p_1}$ and $C_{q_1}$ in $\Theta$ is $1$, being half the intersection of $[C\times \{p_1\}\cup\{p_1\}\times C]$ with $[C\times \{q_1\}\cup\{q_1\}\times C]$ in $C\times C$.
\end{proof}
%**************
% End Hide
%**************
}

\begin{lem}
\label{lemma-intersection-is-either-empty-or-of-length-2}
Let $t$ be a point of $\Pic^0(C)$. If $C_p\cap \tau_t(\Sigma_p)$ is non-empty, then 
$2p-t\sim a+b$, for a unique effective divisor $a+b$, $a,b\in C$. In that case 
the union $Z_{p,t}:=C_p\cup\tau_t(\Sigma_p)$ is contained in each of 
$\tau_{p-a}(\Theta)$ and $\tau_{p-b}(\Theta)$. 
\end{lem}

\begin{proof}
\underline{Step 1:}
Assume that $C_p\cap \tau_t(\Sigma_p)\neq \emptyset$. Then $t$ belongs to the surface $C_p-\Sigma_p$ in $\Pic^0(C)$, i.e., 
$t=p+q-r-s$, where $r+s$ belongs to $\Sigma_p$. The fact that the divisor $r+s$ belongs to $\Sigma_p$ means that there exists a point $u\in C$, such that $\StructureSheaf{C}(p+u+r+s)\cong \omega_C$.

\underline{Step 2:} We show next that there exists a unique effective divisor $a+b$, $a,b\in C$, such that $a+b+q\sim p+r+s$. 
The dimension of $H^0(C,\StructureSheaf{C}(p+r+s))$ is $2$, by Step 1. 
So $H^0(C,\StructureSheaf{C}(p+r+s-q))$ is $1$-dimensional, since $C$ is not hyperelliptic. Hence, 
$p+r+s-q\sim a+b$, for a unique effective divisor $a+b$, $a,b\in C$. 

\underline{Step 3:}
We have
\[
\tau_{-t}(\tau_{p-a}(\Sigma_b))=\tau_{r+s-q-a}(\Sigma_b)\stackrel{(\ref{eq-Sigma-q-is-translate-of-Sigma-p-by-p-q})}{=}
\tau_{r+s-q-a}(\tau_{p-b}(\Sigma_p))=\tau_{r+s+p-q-a-b}(\Sigma_p)=\Sigma_p,
\]
where the last equality follows from Step 2. We conclude the equality $\tau_t(\Sigma_p)=\tau_{p-a}(\Sigma_b).$
Hence,
\[
C_p\cup\tau_t(\Sigma_p)=\tau_{p-a}(C_a)\cup \tau_{p-a}(\Sigma_b) \subset \tau_{p-a}(\Theta).
\]
Interchanging the roles of $a$ and $b$ we get also that $C_p\cup\tau_t(\Sigma_p)$ is contained in $\tau_{p-b}(\Theta)$.
Finally we have
\[
a+b\sim p-q+r+s\sim (p-q)+(p+q-t)=2p-t.
\]
Uniqueness of $a+b$ follows, since $C$ is assumed non-hyperelliptic.
%Lemma \ref{lemma-intersection-of-C-and-Sigma-has-length-at-most-2}(\ref{lemma-item-one-or-two-theta-translates}) states that 
%the canonical line bundle $\omega_{C_p\cup\tau_t(\Sigma_p)}$ is the restriction of $\StructureSheaf{X}(\tau_{p-a}(\Theta)+\tau_{p-b}(\Theta))$, which is %isomorphic to $\StructureSheaf{X}(3\Theta-\tau_{-t}(\Theta))$, by the Theorem of the Square.
\end{proof}

\begin{rem}
\label{rem-intersection-is-either-empty-or-of-length-2}
Keep the notation of Lemma \ref{lemma-intersection-is-either-empty-or-of-length-2}. If $C_p\cap\tau_t(\Sigma_p)$ is non-empty,
then the canonical line bundle  $\omega_{Z_{p,t}}$ is the restriction of $\StructureSheaf{X}(\tau_{p-a}(\Theta)+\tau_{p-b}(\Theta))$, 
by Lemma \ref{lemma-intersection-of-C-and-Sigma-has-length-at-most-2}(\ref{lemma-item-one-or-two-theta-translates}). Hence 
$\omega_{Z_{p,t}}$ is isomorphic to $\StructureSheaf{X}(3\Theta-\tau_{-t}(\Theta))$, by the Theorem of the Square.
\end{rem}

As a corollary, we get the following.

\begin{lem}
\label{lemma-LB-s-t}
Let $s$ and $t$ be points of $\Pic^0(C)$. If $\tau_s(C_p)\cap \tau_t(\Sigma_p)$ is non-empty, then 
the following statements hold.
\begin{enumerate}
\item
\label{lemma-item-s-t+2p=a+b}
$s-t+2p\sim a+b$, $a,b\in C$, for a unique effective divisor $a+b$. 
\item
\label{lemma-item-formula-for-LB}
The union $Z_{p,s,t}:=\tau_s(C_p)\cup\tau_t(\Sigma_p)$ is contained in each of 
$\tau_{p+s-a}(\Theta)$ and $\tau_{p+s-b}(\Theta)$ and its canonical line bundle  $\omega_{Z_{p,s,t}}$ is the restriction of 
$\LB_{s,t}:=\StructureSheaf{X}(3\Theta-\tau_{-t-s}(\Theta))$. 
\item
\label{lemma-item-restriction-of-L-s-t-to-C-p}
Let $D_{p,s,t}$ be the degree $2$ divisor on $\tau_s(C_p)$ corresponding to the length $2$ subscheme $\tau_s(C_p)\cap \tau_t(\Sigma_p)$. Denote by $D'_{p,s,t}$ the analogous degree $2$ divisor on $\tau_t(\Sigma_p)$. 
The line bundle $\LB_{s,t}$ is the unique line bundle in its connected component of $\Pic(X)$, which restriction to $\tau_s(C_p)$ is $\omega_{\tau_s(C_p)}(D_{p,s,t})$. The line bundle $\LB_{s,t}$ is also the unique line bundle in its connected component of $\Pic(X)$, which restriction to $\tau_t(\Sigma_p)$ is $\omega_{\tau_t(\Sigma_p)}(D'_{p,s,t})$. 
\end{enumerate}
\end{lem}

\begin{proof}
Part (\ref{lemma-item-s-t+2p=a+b}) follows from Lemma \ref{lemma-intersection-is-either-empty-or-of-length-2}.

Part (\ref{lemma-item-formula-for-LB}): The statement is a translate by $\tau_s$ of that of 
Lemma \ref{lemma-intersection-is-either-empty-or-of-length-2} 
and Remark \ref{rem-intersection-is-either-empty-or-of-length-2}. The canonical line bundle $\omega_{Z_{p,s,t}}$ is thus  the restriction of 
$
\StructureSheaf{X}(\tau_{p+s-a}(\Theta)+\tau_{p+s-b}(\Theta)).
$
%$\tau_s(3\Theta-\tau_{s-t}(\Theta))$. 
The latter is linearly equivalent to $3\Theta-\tau_{-t-s}(\Theta)$, by the Theorem of the Square.

Part (\ref{lemma-item-restriction-of-L-s-t-to-C-p}): 
$\LB_{s,t}$ restricts to $Z_{p,s,t}$ as the canonical line bundle $\omega_{Z_{p,s,t}}$, by part (\ref{lemma-item-formula-for-LB}),
%Lemma \ref{lemma-intersection-is-either-empty-or-of-length-2}, 
and the latter restricts to $\tau_s(C_p)$ as $\omega_{\tau_s(C_p)}(D_{p,s,t})$. If $\LB'$ is another translate of $\StructureSheaf{X}(2\Theta)$, which restricts to $\tau_s(C_p)$ as $\omega_{\tau_s(C_p)}(D_{p,s,t})$, then
$\LB'$ is isomorphic to $\LB_{s,t}$, since the restriction homomorphism $\Pic(X)\rightarrow \Pic(\tau_s(C_p))$ induces an isomorphism from each connected component of $\Pic(X)$ onto the corresponding connected component of $\Pic(\tau_s(C_p))$.
The same argument proves the analogous statement for $\tau_t(\Sigma_p)$.
\end{proof}

\begin{proof}[Proof of Proposition \ref{prop-local-freeness}]
Reduction of part (\ref{prop-local-freeness-G}) to part (\ref{prop-local-freeness-E}). The following tensor products are in the derived category.
The object $\pi_1^*F_1\otimes\F_2$ in $D^b(X\times X\times\hat{X})$ is the tensor product of the line bundle $\pi_1^*(\StructureSheaf{X}(\Theta))\otimes \pi_{12}^*(a^*(\StructureSheaf{X}(\Theta)))\otimes\pi_{13}^*\P^{-1}$ and the object $\pi_1^*(\Ideal{\cup_{i=1}^{d+1}C_i})\otimes \pi_{12}^*(a^*(\Ideal{\cup_{j=1}^{d+1}\Sigma_j}))$. 
Each of the factors in the latter tensor product is the ideal sheaf of a subscheme flat over $X\times\hat{X}$ with respect to $\pi_{23}$.
The object $\pi_1^*(\Ideal{\cup_{i=1}^{d+1}C_i})\otimes \pi_{12}^*(a^*(\Ideal{\cup_{j=1}^{d+1}\Sigma_j}))$
restricts to the fiber of $\pi_{23}$ over $(x_2,L)\in X\times\hat{X}$ as the object $\Ideal{\cup_{i=1}^{d+1}C_i}\otimes \Ideal{\cup_{j=1}^{d+1}\tau_{-x_2}(\Sigma_j)}$, which is isomorphic to the ideal sheaf of a subscheme of $X$, by Lemma \ref{lemma-vanishing-of-tor-sheaves} (the torsion sheaves 
$\SheafTor_k\left(\Ideal{\cup_{i=1}^{d+1}C_i}, \Ideal{\cup_{j=1}^{d+1}\tau_{-x_2}(\Sigma_j)}\right)$ vanish, for $k\neq 0$). Hence, the object 
$\pi_1^*(\Ideal{\cup_{i=1}^{d+1}C_i})\otimes \pi_{12}^*(a^*(\Ideal{\cup_{j=1}^{d+1}\Sigma_j}))$ is isomorphic to a coherent sheaf over $X\times X\times\hat{X}$, which is flat over $X\times \hat{X}$ with respect to $\pi_{23}$. Consequently,  so is $\pi_1^*F_1\otimes\F_2$. 

There exists over $X\times\hat{X}$ a complex of locally free sheaves of finite rank 
\[
K^\bullet :
K^0\RightArrowOf{d_0}  \cdots \RightArrowOf{d_{p-1}} K^p \RightArrowOf{d_p} K^{p+1} \RightArrowOf{d_{p+1}} \cdots \RightArrowOf{d_{n-1}} K^n
\]
representing the object $\G:=R\pi_{23,*}(\pi_1^*F_1\otimes\F_2)$ in $D^b(X\times\hat{X})$. Furthermore, for every subscheme $B$ of $X\times\hat{X}$ we have
\[
R\pi_{23,*}(\pi_1^*F_1\otimes\F_2\otimes \pi_{23}^*\StructureSheaf{B})\cong (K^\bullet)\otimes \StructureSheaf{B},
\]
by cohomology and base change (see the theorem in section 5 of \cite{mumford-abelian-varieties}).

The restriction of $\G$ to the fiber $\pi_{23}^{-1}((x,L))$ is $F_1\otimes \tau_x^*(F_2)\otimes L^{-1}$.
We will prove part (\ref{prop-local-freeness-E}) by showing that 
\begin{enumerate}
\item[(i)]
$H^i(X,F_1\otimes\tau_x^*(F_2)\otimes L^{-1})$ vanishes for $i\neq 1$ for all $(x,L)\in [X\times \Pic^0(X)]\setminus\tilde{\Theta}$.
\item[(ii)]
%away from $\tilde{\Theta}$ 
For $(x,L)\in \tilde{\Theta}$ the cohomology 
$H^i(X,F_1\otimes\tau_x^*(F_2)\otimes L^{-1})$ vanishes if and only if  $i\not\in \{1,2\}$. 
%and $H^2(X,F_1\otimes\tau_x^*(F_2)\otimes L)$ does not vanish. 
\end{enumerate}
It follows that $d_0$ is fiberwise injective and thus $K^1/Im(d_0)$ is locally free. We may thus assume that $K^0=0.$ Furthermore, if $n>2$, then the rightmost homomorphism $d_{n-1}:K^{n-1}\rightarrow K^n$ is fiberwise surjective, by the vanishing of 
$H^n(X,F_1\otimes\tau_x^*(F_2)\otimes L^{-1})$. Hence, $\ker(d_{n-1})$ is locally free. We may thus assume that $K^p=0$, for $p>2$.  The complex is thus $K^1\RightArrowOf{d_1} K^2$ and the cokernel of $d_1$ is supported, set theoretically, on $\tilde{\Theta}$. Hence, the object $\G^\vee[-1]$ is represented by the complex
$(K^\bullet)^*[-1]:=(K^2)^*\RightArrowOf{d_1^*} (K^1)^*$, where $(K^1)^*$ is in degree $0$ and $d_1^*$ is an injective sheaf homomorphism, whose cokernel $\E$ has rank equal to the rank $8d$ of the kernel of $d_1$. 

Applying $\R\SheafHom(\bullet,\StructureSheaf{X\times\hat{X}})$ to the short exact sequence
\[
0\rightarrow (K^2)^*\RightArrowOf{d_1^*} (K^1)^* \rightarrow \E\rightarrow 0
\]
we get the exact sequence
\[
0\rightarrow \G_1 \rightarrow K^1 \RightArrowOf{d_1} K^2 \rightarrow \SheafExt^1(\E,\StructureSheaf{X\times\hat{X}}) \rightarrow 0.
\]
We conclude that $\G_1\cong \E^*$,
$\G_2\cong \SheafExt^1(\E,\StructureSheaf{X\times\hat{X}})$, and $\SheafExt^i(\E,\StructureSheaf{X\times\hat{X}})=0,$ for $i>2$. 
The sheaf $\G_1$ is a saturated subsheaf of $K^1$, and is thus reflexive. The codimension of the support of 
$\SheafExt^1(\E,\StructureSheaf{X\times\hat{X}})$ is $4$. 
We see that if $\SheafExt^i(\E,\StructureSheaf{X\times\hat{X}})$ does not vanish, then the codimension of its support 
 is larger than $i+1$, for all $i>0$. Hence $\E$ is reflexive, by \cite[Prop. 1.1.10(3')]{huybrechts-lehn}.

Proof of part (\ref{prop-local-freeness-E}).
It suffices to prove (i) and (ii) above.
%show that $H^i(X,F_1\otimes\tau_x^*(F_2)\otimes L)$ vanishes for $i\neq 1$ for all $(x,L)\in X\times \Pic^0(X)$ 
%away from $\tilde{\Theta}$ and for $(x,L)$ in the latter locus 
%$H^i(X,F_1\otimes\tau_x^*(F_2)\otimes L)$ vanishes for $i\neq \{1,2\}$ 
%and $H^2(X,F_1\otimes\tau_x^*(F_2)\otimes L)$ is one-dimensional. 
The tensor product $F_1\otimes \tau_x^*F_2$ is isomorphic to $\Ideal{Z_x}(\Theta+\tau_{-x}(\Theta))$ for a subscheme $Z_x$ of $X$ supported on the union of the curves $C_i$ and $\tau_{-x}(\Sigma_j)$, $1\leq i,j\leq d+1$, possibly with embedded points at points of intersections of $C_i$ and $\Sigma_j':=\tau_{-x}(\Sigma_j)$, by Lemma \ref{lemma-vanishing-of-tor-sheaves}. 
Set $L':=L^{-1}(\tau_{-x}(\Theta)-\Theta)$.
We have the short exact sequence
\begin{equation}
\label{eq-short-exact-sequence-of-O-Z-x-twisted-by-line-bundle}
0\rightarrow F_1\otimes \tau_x^*(F_2)\otimes L^{-1}
\rightarrow \StructureSheaf{X}(2\Theta)\otimes L'\RightArrowOf{} \StructureSheaf{Z_x}(2\Theta)\otimes L'\rightarrow 0.
\end{equation}
$H^i(\StructureSheaf{X}(2\Theta)\otimes L')$ vanishes, for $i>0$, and $H^0(\StructureSheaf{X}(2\Theta)\otimes L')$ is $8$-dimensional. We have $\chi(\StructureSheaf{Z_x}(2\Theta)\otimes L')=8(d+1)$. The later statement is clear when the curves $C_i$ and $\Sigma_j'$ are disjoint for all $1\leq i,j \leq d+1$. In this case $Z_x$ is the disjoint union of these $2d+2$ curves,
\[
\StructureSheaf{Z_x}(2\Theta)\otimes L'\cong 
\left[
\left(\oplus_{i=1}^{d+1}\StructureSheaf{C_i}\right)
\oplus
\left(\oplus_{i=1}^{d+1}\StructureSheaf{\Sigma'_i}\right)\right]\otimes \StructureSheaf{X}(2\Theta)\otimes L',
\]
The restriction of $\StructureSheaf{X}(2\Theta)\otimes L'$ to each of $C_i$ and $\Sigma_j'$ has degree $6$ and thus Euler characteristic $4$. Hence, $\chi(\StructureSheaf{Z}(2\Theta)\otimes L')=8(d+1)$ as claimed.

It suffices to prove the two vanishing and one non-vanishing statements below. 
%for all $L\in \Pic^0(X)$, and the second vanishing for all $(x,L)\in X\times \hat{X}\setminus \tilde{\Theta}$.
\begin{eqnarray}
\label{eq-vanishing-of-global-sections-of-tensor-product}
H^0(F_1\otimes \tau_x^*(F_2)\otimes L^{-1})&=&0, \ \ \forall L\in \Pic^0(X),
\\
\label{eq-vanishing-of-H-1}
H^1(\StructureSheaf{Z_x}(2\Theta)\otimes L')&=&0, \ \ \forall (x,L)\in X\times \hat{X}\setminus \tilde{\Theta}.
\\
\label{eq-non-vanishing-of-H-1-for-x-L-in-tilde-Theta}
h^2(F_1\otimes \tau_x^*(F_2)\otimes L)&\neq&0, \ \ \forall (x,L)\in \tilde{\Theta}.
\end{eqnarray}
%for $(x,L)\in \tilde{\Theta}$.
Statements (i) and (ii) above would follow from the above three statements.
The vanishing of $H^i(F_1\otimes \tau_x^*(F_2)\otimes L)$, for $i\geq 2$ and $(x,L)\not\in \tilde{\Theta}$, would then follow from the vanishing (\ref{eq-vanishing-of-H-1}) and the long exact sheaf cohomology sequence associated to the short exact sequence (\ref{eq-short-exact-sequence-of-O-Z-x-twisted-by-line-bundle}).
The vanishing for $i=3$ and $(x,L)\in \tilde{\Theta}$ follows from the vanishing of $H^2(\StructureSheaf{Z_x}(2\Theta)\otimes L')$ and 
the latter long exact sheaf cohomology sequence.

The vanishing (\ref{eq-vanishing-of-global-sections-of-tensor-product}) follows from Assumption \ref{assumption-on-C-i-s}, as $\cup_{i=1}^{d+1}C_i$ is a subscheme of $Z_x$.

The vanishing (\ref{eq-vanishing-of-H-1}) is clear, when the curves $C_i$ and $\Sigma_j'$ are disjoint, for $1\leq i,j\leq d+1,$
by Riemann-Roch.

Let $Z_{red}:=(\cup_{1\leq i\leq d+1}C_i)\cup (\cup_{1\leq j\leq d+1}\Sigma_j')$ be the reduced induced subscheme of $Z_x$.
We have the short exact sequence
\[
0\rightarrow \StructureSheaf{Z_{tor}}(2\Theta)\otimes L'\rightarrow \StructureSheaf{Z_x}(2\Theta)\otimes L'\rightarrow \StructureSheaf{Z_{red}}(2\Theta)\otimes L'\rightarrow 0.
\]
The sheaf $\StructureSheaf{Z_{tor}}(2\Theta)\otimes L$ has zero-dimensional support, hence
$H^1(\StructureSheaf{Z_{tor}}(2\Theta)\otimes L)$ vanishes. The vanishing (\ref{eq-vanishing-of-H-1}) would thus follow from the vanishing of $H^1(\StructureSheaf{Z_{red}}(2\Theta)\otimes L)$. 

The locus in $X$ of points $x$, such that $\tau_{-x}(\Sigma_j)$ and $C_i$ meet is the surface $\Theta_{i,j}=\Sigma_j-C_i$.
The connected components of $Z_{red}$ are thus all of the following type, 
by Lemma \ref{lemma-intersection-of-C-and-Sigma-has-length-at-most-2} and Assumption \ref{assumption-intersections-of-surfaces-are-generic}(\ref{assumption-item-intersections-of-surfaces-are-generic}).
\begin{enumerate}
\item
\label{case-of-a-smooth-curve}
A smooth curve of genus $3$.
\item
\label{case-of-two-components-meeting-along-a-length-2-subscheme}
Case $x\in\Theta_{i,j}$ and $x\not\in\Theta_{i',j}$ if $i\neq i'$ and $x\not\in\Theta_{i,j'}$ if $j\neq j'$. In that case one connected component is 
the union of $C_i$ and $\tau_{-x}(\Sigma_j)$ meeting along a length $2$ subscheme.
\item
\label{case-of-a-chain1}
Case (a)  $x\in \Theta_{i_1,j}\cap \Theta_{i_2,j}\cap \cdots \cap \Theta_{i_k,j}$, where $k\in \{2,3\}$ and $i_1, \dots, i_k$ are pairwise distinct, or (b)
$x\in \Theta_{i,j_1}\cap \Theta_{i,j_2}\cap \cdots\cap \Theta_{i,j_k}$, where $k\in \{2,3\}$ and $i_1, \dots, i_k$ are pairwise distinct.
\item
\label{case-of-a-chain2}
Case $x\in \Theta_{i_1,j_1}\cap \Theta_{i_2,j_1}\cap \Theta_{i_2,j_2}$, where $i_1\neq i_2$ and $j_1\neq j_2$.
%\item
%The union of $C_i$ and $\tau_{-x}(\Sigma_j)$ meeting at one point along a length $2$ subscheme.
\end{enumerate}

\begin{figure}[hb]%                 use [hb] only if necceccary!
  \centering
  \includegraphics[width=15cm]{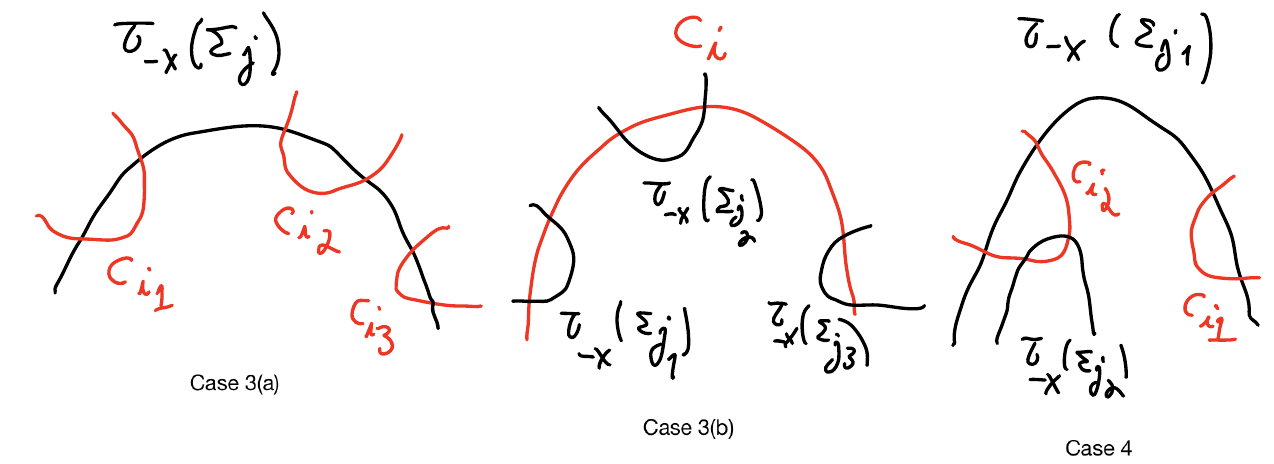}
%  \includegraphics[width=15cm]{image250131.png}
%  \caption{test figure}
  \label{fig:test}
\end{figure}

The vanishing (\ref{eq-vanishing-of-H-1}) in case (\ref{case-of-a-smooth-curve}) is clear.
%In case (\ref{case-of-a-chain}) the connected component is a nodal curve of arithmetic genus $3k$, $1\leq k\leq 4$,
%and $\StructureSheaf{X}(2\Theta)\otimes L$ restricts to each of its irreducible components as a line bundle of degree $6$, 
%while the canonical line bundle restricts to the first and last irreducible components as a line bundle of degree $5$, 
%and both restrict to the intermediate components with degree $6$. Hence, the vanishing (\ref{eq-vanishing-of-H-1}) 
%is a consequence of Serre's duality.
In case (\ref{case-of-two-components-meeting-along-a-length-2-subscheme}) and $(x,L)\in \tilde{\Theta}_{i,j}$ 
the line bundle $\StructureSheaf{Z_{red}}(2\Theta)\otimes L'$ restricts to the canonical line bundle of the connected component $C_i\cup\tau_{-x}(\Sigma_j)$ of $Z_{red}$,
%$\omega_{Z_{red}}$,
by Lemma \ref{lemma-LB-s-t}.
%\ref{lemma-intersection-of-C-and-Sigma-has-length-at-most-2} part (\ref{lemma-item-one-or-two-theta-translates}). 
Hence, $H^1(\StructureSheaf{Z_x}(2\Theta)\otimes L')$ does not vanish in this case. The nonvanishing
(\ref{eq-non-vanishing-of-H-1-for-x-L-in-tilde-Theta}) follows from the long exact sheaf cohomology sequence associated to the short exact sequence (\ref{eq-short-exact-sequence-of-O-Z-x-twisted-by-line-bundle}). The nonvanishing (\ref{eq-non-vanishing-of-H-1-for-x-L-in-tilde-Theta}) follows for all $(x,L)\in\tilde{\Theta}$, by semi-continuity.

We prove next the vanishing (\ref{eq-vanishing-of-H-1}) in cases (\ref{case-of-a-chain1}) and (\ref{case-of-a-chain2}) when $(x,L)$ does not belong to $\tilde{\Theta}$.
Let $T$ be a connected component of $Z_{red}$.
In Case (\ref{case-of-a-chain1})(b) if we remove $C_i$ from $T$ we get the disjoint union of $\tau_{-x}(\Sigma_{j_\ell})$, $1\leq \ell\leq k$. In Case (\ref{case-of-a-chain1})(a) if we remove $\tau_{-x}(\Sigma_j)$ from $T$ 
we get the disjoint union of $C_{i_\ell}$, $1\leq \ell\leq k$.
In case (\ref{case-of-a-chain2})  $C_{i_2}\cup\tau_{-x}(\Sigma_{j_1})$
is a subscheme of the connected component $T$ of $Z_{red}$,  $C_{i_2}\cap\tau_{-x}(\Sigma_{j_1})\neq\emptyset$,
and if we remove $C_{i_2}$ and $\tau_{-x}(\Sigma_{j_1})$ from $T$ 
we get a disconnected union $T''$ of $C_{i_1}$ and $\tau_{-x}(\Sigma_{j_2})$. Denote by $T'$ the union of the components removed (in cases 3(a), 3(b), and (4)). 
We have the short  exact sequence
\[
0\rightarrow \StructureSheaf{T''}(-D)\rightarrow \StructureSheaf{T}\rightarrow \StructureSheaf{T'}\rightarrow 0,
\]
where the divisor $D$ is associated to the intersection subscheme of the connected components of $T'$ with $T''$.
If  $(x,L)$ does not belong to $\tilde{\Theta}$, then $H^1(\StructureSheaf{T'}(2\Theta)\otimes L')$ vanishes. This is clear if $T'$ is a smooth curve, as the line bundle has degree $6$. If $T'= C_{i_2}\cup\tau_{-x}(\Sigma_{j_1})$ the vanishing follows from the fact that $(x,L)$ does not belong to $\tilde{\Theta}_{i_2,j_1}$ and Lemma \ref{lemma-LB-s-t}(\ref{lemma-item-formula-for-LB}).
Indeed, Lemma \ref{lemma-LB-s-t}(\ref{lemma-item-formula-for-LB}) yields that
$\omega_{C_{i_2}\cup\tau_{-x}(\Sigma_{j_1})}=\omega_{\tau_{s_{i_2}}(C_p)\cup\tau_{t_{j_1}-x}(\Sigma_p)}$ is the restriction of 
$\StructureSheaf{X}(3\Theta-\tau_{x-s_{i_2}-t_{j_1}}(\Theta))$, while $(x,L)\not\in \tilde{\Theta}_{i_2,j_1}$ yields the difference of the line bundles in the second step below.
\begin{eqnarray*}
\StructureSheaf{C_{i_2}\cup\tau_{-x}(\Sigma_{j_1})}(2\Theta)\otimes L' &\cong&
\StructureSheaf{C_{i_2}\cup\tau_{-x}(\Sigma_{j_1})}(\Theta+\tau_{-x}(\Theta))\otimes L^{-1}
\\
&\not\cong &
\StructureSheaf{C_{i_2}\cup\tau_{-x}(\Sigma_{j_1})}(2\Theta+\tau_{-x}(\Theta)-\tau_{2x-t_{j_1}-s_{i_2}}(\Theta))
\\
&\cong &\StructureSheaf{C_{i_2}\cup\tau_{-x}(\Sigma_{j_1})}(3\Theta-\tau_{x-t_{j_1}-s_{i_2}}(\Theta)).
\end{eqnarray*}

The vanishing of $H^1(\StructureSheaf{T''}(-D)\otimes\StructureSheaf{X}(2\Theta)\otimes L')$ is seen as follows.
%follows from Lemma \ref{lemma-LB-s-t}(\ref{lemma-item-restriction-of-L-s-t-to-C-p}). 
Each connected component of $T''$ is a smooth genus $3$ curve meeting precisely one connected component of $T'$, 
the divisor $D$ has degree $2$ on each connected component of $T''$, and 
Lemma \ref{lemma-LB-s-t}(\ref{lemma-item-restriction-of-L-s-t-to-C-p}) implies that $\StructureSheaf{X}(3\Theta-\tau_{x-s_{i_2}-t_{j_1}}(\Theta))$
restricts to each component $C''$ of $T''$ as $\omega_{C''}(D'')$, where $D''$ is the part of the divisor $D$ supported on $C''$.
On the other hand, the assumption that $(x,L)\not\in \tilde{\Theta}$ implies that $\StructureSheaf{X}(2\Theta)\otimes L'$ is not isomorphic to 
$\StructureSheaf{X}(3\Theta-\tau_{x-s_{i_2}-t_{j_1}}(\Theta))$. Hence, the line bundle $\StructureSheaf{C''}(-D'')\otimes\StructureSheaf{X}(2\Theta)\otimes L'$
has degree $4$ but is not isomorphic to $\omega_{C''}$.

The vanishing of $H^1(\StructureSheaf{T}(2\Theta)\otimes L')$ follows from the vanishings of $H^1(\StructureSheaf{T'}(2\Theta)\otimes L')$ and $H^1(\StructureSheaf{T''}(-D)\otimes\StructureSheaf{X}(2\Theta)\otimes L')$. We conclude that $H^1(\StructureSheaf{Z_{red}}(2\Theta)\otimes L')$ vanishes.
As observed above, it implies the vanishing (\ref{eq-vanishing-of-H-1}).

Cases (\ref{case-of-a-chain1}) and (\ref{case-of-a-chain2}) lie in the closure of 
Case (\ref{case-of-two-components-meeting-along-a-length-2-subscheme}) and so $H^1(\StructureSheaf{Z_x}(2\Theta)\otimes L')$ does not vanish in these cases for $(x,L)\in\tilde{\Theta}$ 
by semi-continuity. 
\end{proof}

Let $p$ be a point of $C$, let $C_p\subset X=\Pic^2(C)$ be the translate of $AJ(C)$ by $p$, let $t\in\Pic^0(C)$, and set $C_t\:=\tau_t(C_p)$.  Set $\Theta_t:=\tau_t(\Theta)$.

\begin{lem}
\label{lemma-global-sections-of-ideal-I-C-2Theta}
The equality $\dim H^0(X,\Ideal{C_t}(2\Theta)\otimes L)=4$ holds, for all line bundles $L\in\Pic^0(X)$.
\end{lem}

\begin{proof}
Let $s\in\Pic^0(C)$ be such that $L(2\Theta)$ is isomorphic to $\StructureSheaf{X}(\Theta_t+\Theta_s)$. 
Set $D:=\Theta_s+\Theta_t$. Note the inclusion $C_t\subset \Theta_t$.
Denote by $\restricted{D}{C_t}$ and
$\restricted{D}{\Theta_t}$ the restriction of the divisor class.
Consider the following diagram with short exact rows and columns. 
\[
\xymatrix{\StructureSheaf{X}(\Theta_s)\ar[r]^= \ar[d] & \StructureSheaf{X}(\Theta_s)\ar[r] \ar[d] & 0 \ar[d]
\\
\Ideal{C_t}(D) \ar[r] \ar[d]& \StructureSheaf{X}(D) \ar[r] \ar[d]&\StructureSheaf{C_t}(\restricted{D}{C_t})\ar[d]_{=}
\\
\StructureSheaf{\Theta_t}(\restricted{D}{\Theta_t}-C_t) \ar[r] & \StructureSheaf{\Theta_t}(\restricted{D}{\Theta_t})\ar[r] &\StructureSheaf{C_t}(\restricted{D}{C_t})
}
\]
The degree of $\restricted{D}{C_t}$ is $6$ and so $h^0(\StructureSheaf{C_t}(\restricted{D}{C_t}))=4$.
The equality $h^0(\StructureSheaf{X}(D))=8$ implies that $h^0(\Ideal{C_t}(D))$ is at least $4$, by the middle horizontal 
short exact sequence. Consider the vertical left short exact sequence. The dimensions $h^i(\StructureSheaf{X}(\Theta_s))$ are $1$ for $i=0$ and $0$ for $i>0$. Hence, it suffices to prove that $H^0(\StructureSheaf{\Theta_t}(\restricted{D}{\Theta_t}-C_t))$ is $3$-dimensional, as it would follow that $h^0(\Ideal{C_t}(D))\leq 4$.

The middle vertical short exact sequence implies that $H^0(\StructureSheaf{\Theta_t}(\restricted{D}{\Theta_t}))$ is $7$-dimensional, since $h^1(\StructureSheaf{X}(\Theta_s))=0$. Hence, the equality
$h^0(\StructureSheaf{\Theta_t}(\restricted{D}{\Theta_t}-C_t))=3$ 
would follow from the exactness of the bottom horizontal sequence once we prove the vanishing of $h^1(\StructureSheaf{\Theta_t}(\restricted{D}{\Theta_t}-C_t))$. 
It suffices to prove that  $\StructureSheaf{\Theta_t}(\restricted{D}{\Theta_t}-C_t)\otimes\omega_{\Theta_t}^{-1}$ is ample, by Kodaira's vanishing. Now, $\omega_{\Theta_t}\cong \StructureSheaf{X}(\Theta_t\restricted{)}{\Theta_t}$. It remains to prove the ampleness of $(\Theta_s\restricted{)}{\Theta_t}-C_t$. Ampleness depends only on the numerical class and so we may  consider the case where $C_t=C_p$, $\Theta_t=\Theta$, $\Theta_s\cap\Theta=C_p+\Sigma_p.$
%Write $D$ as the sum
%$\Theta_{t_1}+\Theta_{t_2}$, where $\Theta_t\cap\Theta_{t_1}=C_t+\Sigma$. 
%Then $\restricted{D}{\Theta_t}-C_t\sim\restricted{(\Theta_{t_2})}{\Theta_t}+\Sigma$, 
%which is a sum of two ample divisors. 
The divisor $\Theta$ is $\iota$-symmetric and $\Sigma_p=\iota(C_p)$, so ampleness of $\Sigma_p$ follows from that of $C_p$.
Ampleness of $C_p$ is seen via 
the isomorphism of
$\Theta$ with $C^{(2)}$ and the ampleness of the pullback of $C_p$ to the cartesian square $C^2$. Indeed, 
%ampleness depends only on the numerical class and so we may  
%consider the case $\Theta_t=\Theta$ and $\Sigma=\iota(C_p)$.
 the pullback of $C_p$ to $C^2$ is $C\times\{p\}\cup\{p\}\times C$, which is ample.
\end{proof}

\begin{rem}
\begin{enumerate}
\item
The linear systems in the above Lemma all have base loci larger than $C_t$. 
Let $r\in\Pic^0(C)$ be such that $L(2\Theta)$ is isomorphic to $\StructureSheaf{X}(2\Theta_r)$. Then every divisor in 
$|L(2\Theta)|$ is $\iota'$-invariant, where $\iota':=\tau_r\iota\tau_{-r}$. Hence, the dimension of
$H^0(X,\Ideal{C_t\cup \iota'(C_t)}(2\Theta)\otimes L)$ is $4$ as well. 
%Note, however, that $\iota'$ depends on $L$.
\item
Reducible divisors in the linear system $|\Ideal{C_t}(2\Theta)\otimes L|$ are described in \cite[Prop. 11.9.1(a)]{BL}, when $C_t\cup \iota'(C_t)$ is the complete intersection of two translates of $\Theta$. They are related to trisecants of $\varphi_{L(2\Theta)}(X)$ in \cite[Prop. 11.9.3]{BL}.
\end{enumerate}
\end{rem}

%************
% Hide
%************
\hide{
\begin{lem}
\label{lemma-union-of-two-translates-of-Theta-contains-union-of-translates-of-AJ-curve}
Let $\tau_{t_1}(C_p)$ and $\tau_{t_2}(C_p)$ be two disjoint translates with $t_1, t_2\in\Pic^0(C)$.
The following statements are equivalent.
\begin{enumerate}
\item
\label{lemma-item-two-disjoint-translates-of-C-are-contained-in-two-translates-of-Theta}
The union $\tau_{t_1}(C_p)\cup \tau_{t_2}(C_p)$ is contained in a divisor of the form 
$\tau_s(\Theta)+\tau_{-s}(\Theta)\in \linsys{2\Theta}$, for some $s\in \Pic^0(C)$. 
\item
\label{lemma-item-2p-minus-t1-t2-is-effective}
The divisor $2p+t_1+t_2$ is linearly equivalent to an effective divisor $q_1+q_2$, $q_1,q_2\in C$. 
\item
\label{lemma-item-translate-of-first-intersects-the-reflection-of-translate-of-second}
The curves $\tau_{t_1}(C_p)$ and $\iota(\tau_{t_2}(C_p)$ intersect.
\end{enumerate}
In this case
the pair $\{s,-s\}$ is uniquely determined by $(t_1,t_2)$ 
\begin{equation}
\label{eq-s-minus-s}
\{s,-s\}=\{t_1-q_1+p,t_2-q_2+p\}.
\end{equation}
If, furthermore, $q_1\neq q_2$, then $\dim H^0(X,\Ideal{\tau_{t_1}(C_p)\cup \tau_{t_2}(C_p)}(2\Theta))=1.$
\end{lem}

\begin{proof}
The equivalence of (\ref{lemma-item-2p-minus-t1-t2-is-effective})and  (\ref{lemma-item-translate-of-first-intersects-the-reflection-of-translate-of-second}) follows from Lemma \ref{lemma-LB-s-t}. We prove the equivalence of   (\ref{lemma-item-two-disjoint-translates-of-C-are-contained-in-two-translates-of-Theta}) and (\ref{lemma-item-2p-minus-t1-t2-is-effective}).
The curves $\tau_{t_1}(C_p)$ and $\tau_{t_2}(C_p)$ can not be both contained in one translate of $\Theta$, as they are disjoint.
Equation (\ref{eq-criterion-for-inclusion}) implies that $\tau_t(C_p)$ is contained in $\Theta$, if and only if $t\sim q-p$, for some $q\in C$.
We see that $\tau_t(C_p)\subset \tau_s(\Theta)\Leftrightarrow t-s\sim q-p$, for some $q\in C$.
Hence, $\tau_{t_1}(C_p)\subset \tau_s(\Theta)$ and $\tau_{t_2}(C_p)\subset \tau_{-s}(\Theta)$, if and only if 
$t_1-q_1+p\sim s\sim -p+q_2-t_2$, for some $q_1,q_2\in C$. The right hand side is equivalent to $2p+t_1+t_2\sim q_1+q_2$, for some $q_1,q_2\in C$ and $(s,-s)=(t_1-q_1+p,t_2-q_2+p)$. 
%Note also that $q_1\neq q_2$, as otherwise $s\sim -s$ and $\tau_{t_1}(C_p)$ and $\tau_{t_2}(C_p)$ 
%would both be contained in $\tau_s(\Theta)$.

Observe next that $\tau_{-t_2}(\Sigma_p)=\tau_{-t_2}(\iota(C_p))=\iota(\tau_{t_2}(C_p))\subset \iota(\tau_{-s}(\Theta))=\tau_s(\iota(\Theta))=\tau_s(\Theta).$ Hence, $\tau_s(\Theta)$ contains both $\tau_{t_1}(C_p)$ and $\tau_{-t_2}(\Sigma_p)$. It follows that 
\[
\tau_{t_1}(C_p)\cup \tau_{-t_2}(\Sigma_p)\subset\tau_s(\Theta)\cap\tau_{p+t_1-q_2}(\Theta), 
\]
by Lemma \ref{lemma-LB-s-t}(\ref{lemma-item-formula-for-LB}). 

Assume that $q_1\neq q_2$ and set $u:=p+t_1-q_2$. Then $\tau_{t_1}(C_p)\cup \tau_{-t_2}(\Sigma_p)$ is the complete intersection
$\tau_s(\Theta)\cap\tau_u(\Theta)$, by Lemma \ref{lemma-intersection-of-C-and-Sigma-has-length-at-most-2}.
Note that $u\not\sim -s$, as otherwise the two translates $\tau_{t_i}(C_p)$, $i=1,2$ would be contained in $\tau_{-u}(\Theta)$ and would intersect. 

Assume that there exists another divisor $D$ in $\linsys{2\Theta}$ containing $\tau_{t_1}(C_p)$ and $\tau_{t_2}(C_p)$ and
$D\neq \tau_s(\Theta)+\tau_{-s}(\Theta)$.
Every divisor in the linear system is symmetric, and so $D$ contains $\tau_{t_1}(C_p)\cup \tau_{-t_2}(\Sigma_p)$.
Thus $D\cap \tau_s(\Theta)=\tau_{t_1}(C_p)\cup \tau_{-t_2}(\Sigma_p)\cup B$,
where $B$ is a curve in $\tau_s(\Theta)$ in the linear system
$\linsys{\StructureSheaf{\tau_s(\Theta)}(2\Theta-\tau_u(\Theta))}$, which is $\linsys{\StructureSheaf{\tau_s(\Theta)}(\tau_{-u}(\Theta))}$
and consists of the single divisor $\tau_s(\Theta)\cap\tau_{-u}(\Theta)$. Indeed, we have the short exact sequence
\[
0\rightarrow \StructureSheaf{X}(\tau_{-u}(\Theta)-\tau_s(\Theta))\rightarrow 
\StructureSheaf{X}(\tau_{-u}(\Theta))\rightarrow
\StructureSheaf{\tau_s(\Theta)}(\tau_{-u}(\Theta)\cap \tau_s(\Theta))\rightarrow 0,
\]
and $H^i(\StructureSheaf{X}(\tau_{-u}(\Theta)-\tau_s(\Theta)))=0$, for all $i$.
We get the equalities
\[
D\cap\tau_s(\Theta)=\tau_{t_1}(C_p)\cup \tau_{-t_2}(\Sigma_p)\cup B=\tau_{s}(\Theta)\cap[\tau_u(\Theta)\cup\tau_{-u}(\Theta)].
\]
Hence, $D$ is the zero divisor of a section in $\span\{\sigma_s,\sigma_u\},$
where $\sigma_s$ is a section of $\StructureSheaf{X}(2\Theta)$ with zero divisor $\tau_s(\Theta)+\tau_{-s}(\Theta)$
and $\sigma_u$ is a section with zero divisor $\tau_u(\Theta)+\tau_{-u}(\Theta)$. We may thus assume that $D=\tau_u(\Theta)+\tau_{-u}(\Theta).$ But this divisor is again the sum of two translates of $\Theta$, which contradicts the uniqueness of such a divisor containing
$\tau_{t_1}(C_p)\cup \tau_{t_2}(C_p)$ in
Equation (\ref{eq-s-minus-s}). Hence, such a divisor $D$ does not exist.
\end{proof}

\begin{lem}
If the disjoint union $\tau_t\left(\cup_{i=1}^{d+1}\tau_{t_i}(C_p)\right)$ is contained in the union $\tau_s(\Theta)+\tau_{-s}(\Theta)$, for some $s,t\in X$,
then 
\[
2t\in\bigcap_{1\leq i<j\leq d+1} \tau_{-t_i-t_j-2p}(\Theta).
\]
In particular, if the intersection above is empty, then such a pair $(s,t)$ does not exist.
\end{lem}

\begin{proof}
Lemma \ref{lemma-union-of-two-translates-of-Theta-contains-union-of-translates-of-AJ-curve} implies that $2p+t_1+t_2+2t$ is effective, for all $1\leq i<j\leq d+1$.
\end{proof}

\begin{rem}
Assumption \ref{assumption-on-C-i-s} would follow for $d\geq 3$ from the above lemma, provided we prove that if the disjoint union $\tau_{t_1}(C_p)\cup\tau_{t_2}(C_p)$ is contained in a divisor in $\linsys{2\Theta}$, then this divisor is of the form  $\tau_s(\Theta)+\tau_{-s}(\Theta)$, for some $s\in X$, and if we exhibit $t_1$, $t_2$, $t_3$, $t_4$, such that the intersection
\[
\bigcap_{1\leq i<j\leq 4}\tau_{t_i+t_j-2p}(\Theta)
\]
is empty.
\end{rem}

%************
%End  Hide
%************
}

%*****************
% Hide
%*****************
\hide{
\begin{lem}
\label{lemma-ideal-of-two-translates-tensor-two-theta-translates}
Let $C_1$ and $C_2$ be two  disjoint curves in $X$, which are translates of $AJ(C)$. There exists a line bundle $L\in \Pic^0(X)$, such that
$H^0(X,\Ideal{C_1\cup C_2}(2\Theta)\otimes L)$ is $2$-dimensional. 
\end{lem}
\begin{proof}
We may assume, without loss of generality, that $C_1=C_p$ and $C_2=\tau_t(C_p)$. Choose two distinct points $q_1, q_2\in C$. Then $C_p$ is contained in 
$\tau_{p-q_i}(\Theta)$ and $C_2$ in $\tau_{p-q_i+t}(\Theta)$. So $C_1\cup C_2$ is contained in both
$\tau_{p-q_1}(\Theta)+\tau_{p-q_2+t}(\Theta)$ and $\tau_{p-q_2}(\Theta)+\tau_{p-q_1+t}(\Theta)$. Thus, 
$h^0(\Ideal{C_1\cup C_2}(2\Theta)\otimes L)\geq 2,$ for $L:=\StructureSheaf{X}(\tau_{2p+t-q_1-q_2}(\Theta)-\Theta)$.

It remains to show that $h^0(\Ideal{C_1\cup C_2}(2\Theta)\otimes L)\leq 2$. 
The space $H^0(\Ideal{C_1\cup C_2}(2\Theta)\otimes L)$ is the intersection
\[
H^0(\Ideal{C_1}(2\Theta)\otimes L)\cap H^0(\Ideal{C_2}(2\Theta)\otimes L)
\]
in $H^0(L(2\Theta))$. The dimension of the latter intersection  is equal to the co-dimension of
$H^0(\Ideal{C_1}(2\Theta)\otimes L)+ H^0(\Ideal{C_2}(2\Theta)\otimes L)$ in $H^0(L(2\Theta))$. The latter co-dimension is equal to the dimension of the intersection 
\[
W:=[H^0(L(2\Theta))/H^0(\Ideal{C_1}(2\Theta))]^*\cap [H^0(L(2\Theta))/H^0(\Ideal{C_2}(2\Theta))]^*
\]
in $H^0(L(2\Theta))^*$. 
Set $L_i:=L(2\Theta\restricted{)}{C_i}$, $i=1,2$. 
Note that $[H^0(L(2\Theta))/H^0(\Ideal{C_i}(2\Theta))]^*$ is naturally isomorphic to $H^0(C_i,L_i)^*$.
The inequality $\dim(W)\leq 2$ would follow once we prove the (???) emptiness of $\PP(W)\cap\varphi_{L_i}(C_i)$.
We denote by $L_i$ also the pullback of $L_i$ to $C$ via the translation isomorphism.
Choose line bundles $M_i$ in $\Pic^1(C)$, such that $L_i\cong \omega_C\otimes M_i^2$. Then 
$\PP{H^0(C,L_i)}^*$ is isomorphic to $\PP\Ext^1(M_i,M_{i}^{-1})$ and it parametrizes rank $2$ semistable bundles $E$ over $C$, which are extensions of $M_i$ by $M_i^{-1}$, by \cite{beauville}. 
The construction yields an embedding of $\PP{H^0(C,L_i)}^*$ in the $6$-dimensional moduli space $\M_C(2,0)$ of rank $2$ semi-stable vector bundles on $C$ of trivial determinant, where $\M_C(2,0)$ is embedded in $|L(2\Theta)|^*$ as the Coble quartic hypersurface \cite[Sec. 2.3]{pauly}. 
Furthermore, $E$ is stable, if and only if the extension class does not belong to $\varphi_{L_i}(C)$. The semi-stable locus in
$\M_C(2,0)$ is the Kummer, namely the image of the morphism $\varphi_{L(2\Theta)}:X\rightarrow |L(2\Theta)|^*$.
Hence, the set theoretic intersection in $|L(2\Theta)|^*$ of each of $\PP{H^0(C_i,L_i)}^*$ with the Kummer is equal to $\varphi_{L(2\Theta)}(C_i)$, for $i=1,2$. We get the set theoretic equality
\[
\PP(W)\cap \varphi_{L(2\Theta)}(C_1)=\PP(W)\cap \varphi_{L(2\Theta)}(X)=\PP(W)\cap \varphi_{L(2\Theta)}(C_2).
\]
(???) If the scheme theoretic  intersection in $|L(2\Theta)|^*$ of $\PP{H^0(C_2,L_2)}^*$ with the Kummer is equal\footnote{
Equivalently, $H^0(\Ideal{C_2}(2\Theta)\otimes L)$ generates the subsheaf $\Ideal{C_2\cup \iota'(C_2)}(2\Theta)\otimes L$,
where $\iota'$ is the Galois involution of $\varphi_{L(2\Theta)}:X\rightarrow |L(2\Theta)|^*$.
} 
to $\varphi_{L(2\Theta)}(C_2)$, then the above right displayed equality is valid scheme theoretically. 
We get that $\PP(W)\cap \varphi_{L(2\Theta)}(C_1)$ is a subscheme of $\PP(W)\cap \varphi_{L(2\Theta)}(C_2)$ and hence also of
$\varphi_{L(2\Theta)}(C_1)\cap \varphi_{L(2\Theta)}(C_2)$.

The morphism $\varphi_{L(2\Theta)}:X\rightarrow |L(2\Theta)|^*$ is of degree $2$ onto the Kummer variety. 
Denote by $\iota'$ the Galois involution. The intersection scheme $\varphi_{L_1}(C)\cap \varphi_{L_2}(C)$
is isomorphic to $C_1\cap \iota'(C_2)$, which has length at most $2$, by Lemma \ref{lemma-intersection-of-C-and-Sigma-has-length-at-most-2}. 
If $\dim(W)=3$, then the plane $\PP(W)$ will intersect $\varphi_{L_1}(C)$ along a length $6$ subscheme, contradicting the scheme theoretic inclusion of $\PP(W)\cap \varphi_{L(2\Theta)}(C_1)$ in $\varphi_{L(2\Theta)}(C_1)\cap \varphi_{L(2\Theta)}(C_2)$.
%hence at some point in $|L(2\Theta)|^*$ which does not lie on $\varphi_{L_2}(C)$, 
\end{proof}
%*****************
% End Hide
%*****************
}

%*****************
% Hide
%*****************
\hide{
\begin{rem}
\begin{enumerate}
\item
{\bf Higher dimensional components of the Hilbert scheme of curves:}
If $C_1$ and $C_2$ are distinct translates of $C$ but intersect at a point $\ell$, then $\SheafTor_i(\Ideal{C_1},\Ideal{C_2})$ vanishes, for $i\neq 0$ and 
$\Ideal{C_1}\otimes\Ideal{C_2}$ 
is a subsheaf of $\Ideal{C_1\cup C_2}$ with quotient supported at the point $\ell$ of intersection. This means that 
$\Ideal{C_1}\otimes\Ideal{C_2}$ is the ideal sheaf of a subscheme with an embedded point. Thus, 
$\Ideal{C_1}\otimes\Ideal{C_2}$ deforms to the ideal sheaf of the union of $\Ideal{C_1\cup C_2}$ and a point away from $C_1\cup C_2$. Thus, 
$\Ideal{C_1}\otimes\Ideal{C_2}$ belongs to the intersection of two components of the Hilbert scheme, one of which of dimension $8$, while if $C_1$ and $C_2$ are disjoint, $\Ideal{C_1\cup C_2}$ belongs to a 6 dimensional component of the Hilbert scheme. 
\item
{\bf Comparison of two components of the Hilbert scheme:}
Let ${\mathcal H}$  be the component of the Hilbert scheme of $X$, whose generic point is the ideal sheaf of two disjoint translates of $AJ(C)$. Clearly, ${\mathcal H}$ is birational to the irreducible component $X^{[2]}$ of the Hilbert scheme whose generic point is the ideal sheaf of a length $2$ subscheme of $X$ (here we abuse notation, more canonically the birational isomorphism is with $\Pic^1(C)^{[2]}$, while $X=\Pic^2(C)$). Let $C_t\subset X$, $t\in \Pic^1(C)$, be such a translate.
We have seen in the proof of Lemma \ref{lemma-dim-Ext-1-F-F-is-3n+3} that $H^0(X,TX)$ maps isomorphically onto $H^0(C_t,N_{C_t/X})$. Let $\tilde{\xi}$ be a non-zero section of $H^0(X,TX)$ mapping to the section $\xi$ of $H^0(C_t,N_{C_t/X})$. 
The spaces $H^0(X,TX)$ and $H^0(C,\omega_C)^*$ are naturally isomorphic. Let $[\tilde{\xi}]\in \PP{H^0(C,\omega_C)^*}$ be the point corresponding to $\tilde{\xi}$.  Let $\tilde{C}_\xi\subset X$ be the subscheme with ideal sheaf the kernel of the composition
\[
\Ideal{C_t}\rightarrow \Ideal{C_t}/\Ideal{C_t}^2\IsomRightArrow N^*_{C_t/X}\RightArrowOf{\xi}\StructureSheaf{C_t}.
\]
The canonical morphism $\varphi_{\omega_C}$ maps a point $p\in C$ to the line $T_pC$ in $H^0(X,TX)\cong H^0(C,\omega_C)^*$.
If $[\tilde{\xi}]$ does not belong to the canonical curve $\varphi_{\omega_C}(C)$, then
$\xi$ does not vanish at any point and the above composition is surjective.
We get the short exact sequence
\[
0\rightarrow \Ideal{\tilde{C}_\xi}\rightarrow \Ideal{C_t}\rightarrow \StructureSheaf{C_t}\rightarrow 0.
\]
In particular, the Euler characteristic of $\Ideal{\tilde{C}_\xi}$ is equal to that of $\Ideal{C_1\cup C_2}$ for two disjoint translates of $C$. This suggests that the birational map from $X^{[2]}$ to ${\mathcal H}$ is regular at the point $(t,[\tilde{\xi}])$ on the diagonal divisor $X\times \PP{H^0(X,TX)}$ of $X^{[2]}$. 
If, however, $[\tilde{\xi}]$ belongs to $\varphi_{\omega_C}(C)=\varphi_{\omega_{C_t}}(C_t)$ and $[\tilde{\xi}]=\varphi_{\omega_{C_t}}(p)$, then the above composition is not surjective at $p$. The component ${\mathcal H}$ contains as limits also ideals of subschemes containing $\tilde{C}_\xi$ as well as an embedded point. 
Is the limit unique, as in the previous part of this remark, and the latter scheme structure at $p$ is determined by $\xi$, or does the scheme structure involves additional choices? 
One limit, with an embedded point at $[\tilde{\xi}]=q+t\in C_t$, is obtained as the limit of $\Ideal{C_t}\otimes \Ideal{C_{t+q-q'}}$ as we let the point $q'$ approach $q$ in $C$. This limit is well defined, by the properness of ${\mathcal H}$ and smoothness of $C$. It has  an embedded point at $q+t$, since $\Ideal{C_t}\otimes \Ideal{C_{t+q-q'}}$ is the ideal of a subscheme with an embedded point at $C_t\cap C_{t+q-q'}=\{q+t\}$.
%This question is related to the comparison of the components ${\mathcal H}$ and $X^{[2]}$ of the Hilbert scheme. 
It seems plausible that the birational morphism from $\H$ to $X^{[2]}$ is regular, and the above question asks for a description of its fibers over the $4$-dimensional subscheme $X\times \varphi_{\omega_C}(C)$ of the diagonal divisor $X\times \PP{H^0(X,TX)}$ of $X^{[2]}$. 
\item
One may hope that the component ${\mathcal H}$ deforms as a projective variety along generalized deformations of $X$ to a component of objects in the deformed derived category, so that the deformation of $D^b(X)$ is recovered by a monad construction as in \cite{markman-mehrotra}. One expects to get a $9$-dimensional family of projective varieties deforming ${\mathcal H}$, by the first order calculation in Proposition \ref{prop-obstruction-map-has-rank-6}. If an additional structure appears in the previous part of this remark, then its deformations may account for the $3$ additional parameters in completing to a $9$-dimensional family the $6$-dimensional family of projective deformations of $X$ as a Jacobian of a genus $3$ curve. On the other hand, it is possible that the variety ${\mathcal H}$ varies only in a $6$-dimensional family, and the $3$ additional parameters are encoded in the monad data.
\end{enumerate}
\end{rem}

%*****************
% End Hide
%*****************
}

%****************************************************************
% 
%****************************************************************
\subsection{A semiregular reflexive secant$^{\boxtimes 2}$-sheaf}
\label{subsection-a-semiregular-secant-square-sheaf}
Set $n:=d+1$.
Choose $\Ideal{\cup_{i=1}^n C_i}$ and 
$\Ideal{\cup_{i=1}^n\Sigma_i}$ to each be equivariant with respect to a subgroup $G_i$, $i=1,2$, of $X$ of order $n$, where $G_1$ permutes the connected components $\{\Sigma_i\}_{i=1}^n$ transitively and $G_2$ 
permutes the connected components $\{C_i\}_{i=1}^n$ transitively. 

\begin{lem}
\label{lemma-G_1-and-G_2-can-be-chosen-to-satisfy-assumption-intersections-of-surfaces-are-generic}
A generic $C$ admits subgroups $G_1$ and $G_2$ of $\Pic^0(C)$, such that Assumption \ref{assumption-intersections-of-surfaces-are-generic} holds for a $G_1$ orbit $\{C_i\}_{i=1}^n$ of translates of $C_p$ and a $G_2$ orbit $\{\Sigma_i\}_{i=1}^n$ of translates of $\Sigma_p$.
\end{lem}

\begin{proof}
Assumption \ref{assumption-intersections-of-surfaces-are-generic}(\ref{assumption-3s-ij-pairwise-distinct}) is equivalent to the condition that the intersection $G_1\cap G_2$ is trivial, which we assume.
We have already observed that the surface $\Theta_{i,j}:=\Sigma_j-C_i$ in $\Pic^0(C)$ is a translate of the $\Theta$ divisor. Hence, it suffices to prove the existence of $G_1$ and $G_2$, such that the intersection of any four translates
$\tau_{g_1+g_2}(\Theta)$, $(g_1,g_2)\in G_1\times G_2$, is empty and any three translates have finite intersection. The property is open in moduli, and so it suffices to prove it for the degenerate case, where $C$ is a chain of elliptic curves $E_1$, $E_2$, and $E_3$, as in the proof of Lemma \ref{lemma-emptiness-condition-holds-when-the-curves-are-an-orbit}. The curves $E_1$ and $E_2$ intersect at a point $\{p_0\}$, the curves $E_2$ and $E_3$ intersect at a point $p_1$, and the divisor $\Theta$ is  given in (\ref{eq-theta-divisor-of-a-chain-of-elliptic-curves}). Choose $G_1$ and $G_2$ so that $G_1\cap G_2=\{0\}$
and for each $i$ the restriction homomorphism $\Pic^0(C)\rightarrow \Pic^0(E_i)$ restricts to the subgroup $\langle G_1, G_2\rangle$ generated by $G_1$ and $G_2$ as an injective homomorphism.
Each of the three irreducible components $D_i$ of $\Theta$ is disjoint from its translate $\tau_{g_1+g_2}(D_i)$, if $(g_1,g_2)\neq(0,0)$. 
Hence, the intersection of any four translates is empty and the intersection $\tau_{g_1+g_2}(\Theta)\cap \tau_{g'_1+g'_2}(\Theta)\cap \tau_{g''_1+g''_2}(\Theta)$ of three translates is the union of
$\tau_{g_1+g_2}(D_i)\cap \tau_{g'_1+g'_2}(D_j)\cap \tau_{g''_1+g''_2}(D_k)$, with pairwise distinct $i, j, k$. Each of the latter intersections consists of precisely one point. Hence, all triple intersections 
$\tau_{g_1+g_2}(\Theta)\cap \tau_{g'_1+g'_2}(\Theta)\cap \tau_{g''_1+g''_2}(\Theta)$
are finite.
\end{proof}

Set $E':=\Ideal{\cup_{i=1}^n\Sigma_i}\boxtimes \Ideal{\cup_{i=1}^n C_i}$.
The image of the obstruction map
\[
ob_{E'}:HT^2(X\times X)\rightarrow \Ext^2(E',E')
\]
is $\Ext^2(E',E')^{G_1\times G_2}$. The inclusion $Im(ob_{E'})\subset \Ext^2(E',E')^{G_1\times G_2}$
follows from the $G_i$-equivariance of $ob_{\Ideal{\cup_{i=1}^n C_i}}$ and $ob_{\Ideal{\cup_{i=1}^n \Sigma_i}}$ and the fact that both groups act trivially on $HT^2(X\times X)$.
The inclusion $\Ext^2(E',E')^{G_1\times G_2}\subset Im(ob_{E'})$ follows from the surjectivity of 
$HT^j(X)\rightarrow \Ext^j(\Ideal{\cup_{i=1}^n C_i},\Ideal{\cup_{i=1}^n C_i})^{G_1}$, for $j\leq 2$, and the analogous surjectivity for 
$\Ideal{\cup_{i=1}^n \Sigma_i}$, which in turn follows from Lemmas \ref{lemma-Ext-algebra-generated-by-Ext-1-when-n=1} and  \ref{lemma-Yoneda-algebra-is-generated-in-degree-1}.
The composition 
\[
\tilde{\Phi}:=\Phi\circ ([\Theta\boxtimes\Theta]\otimes)=(id\times \Psi_{\P^{-1}[3]})\circ\mu^*\circ([\Theta\boxtimes\Theta]\otimes),
\] 
of 
Orlov's derived equivalence $\Phi$ with tensorization by the line bundle $\Theta\boxtimes\Theta$, conjugates the subgroup $G_1\times G_2$ of $X\times X$ to a subgroup $G$ of the identity component 
$X\times\hat{X}\times \Pic^0(X\times\hat{X})$
of the group of autoequivalences of the derived category of $X\times\hat{X}$. The group $G$ is calculated below in Equation (\ref{eq-elements-of-G}).

The sheaf $E'$ admits a natural $G_1\times G_2$ linearization $\lambda'$ yielding the $G_1\times G_2$-equivariant sheaf $(E',\lambda')$. Hence, 
the image $\G=\Phi(\Ideal{\cup_{i=1}^n\Sigma_i}(\Theta)\boxtimes \Ideal{\cup_{i=1}^n C_i}(\Theta))[-3]$ in $D^b(X\times \hat{X})$ admits a $G$-linearization $\lambda:=\tilde{\Phi}(\lambda')$ yielding a $G$-equivariant object $(\G,\lambda)$ with respect to the action of $G$ on $D^b(X\times \hat{X})$. 

\begin{lem}
\label{lemma-image-of-ob-E-is-the-G-invariant-subspace}
The image of the obstruction homomorphism $ob_\G:HH^2(X\times X)\rightarrow \Hom(\G,\G[2])$ is 
$\Hom(\G,\G[2])^G:=\Hom((\G,\lambda),(\G,\lambda)[2])$. 
\end{lem}

\begin{proof}
The statement follows from the equality, observed above, of the image of $ob_{E'}$ and $\Hom((E',\lambda'),(E',\lambda')[2])^{G_1\times G_2}$.
\end{proof}

The action of $G$ on $HH^*(X\times\hat{X})$ is trivial, so the latter is also the 
Hochschild cohomology of $D^b_G(X\times\hat{X}).$
The semi-regularity map 
\[
\sigma: \Hom(\G,\G[2])\rightarrow \prod_{q=0}^4 H^{q+2}(\Omega^q_{X\times\hat{X}})
\]
restricts to an injective homomorphism from 
$\Hom(\G,\G[2])^G$, by Remark \ref{remark-semiregularity-map-restricts-as-an-injective-map-to-the-image-of-ob}, Lemma \ref{lemma-kernel-of-ob-E-is-annihilator-of-ch_E}, and Remark \ref{remark-all-the-results-for-E-hold-for-G}.
%the $G$-equivariant analogue of Lemma \ref{lemma-criterion-for-semiregularity}. 
In that sense $(\G,\lambda)$ is semi-regular.
The group $G$ is a finite subgroup of $X\times\hat{X} \times \Pic^0(X\times\hat{X})$ and so it deforms with $D^b(X\times\hat{X})$ to an action on every polarized abelian variety in the same connected component. Below we choose, instead, to pass to an action on the derived category induced only by automorphisms of $X\times\hat{X}$.
%***********
% Hide
%***********
\hide{
So an equivariant version of the unobstructedness result of \cite[Th. 5.1]{buchweitz-flenner}
should yield the deformability of $[X\times\hat{X},(\G,\lambda)]$ 
to a pair $[A,(\G',\lambda')],$
for every polarized abelian variety $A$ of Hodge-Weil type in  the connected component of $X\times\hat{X}$, at least in  some open neighborhood of the latter. 
We need, however, to deform $\E:=\G^\vee[-1]$ as a twisted sheaf allowing also gerby deformations of $X\times\hat{X}$.
% and so we formulate below the conjectural analogue of \cite[Th. 5.1]{buchweitz-flenner} in the setting of twisted sheaves.
%***********
% End Hide
%***********
}

Let $\bar{G}$ be the projection of $G$ to $X\times \hat{X}$, considered as the group of translation automorphisms of $X\times \hat{X}$. 

\begin{lem} 
\label{lemma-the-projection-of-G-to-bar-G-is-an-isomorphism} 
If the intersection $G_1\cap G_2$ does not contain\footnote{This condition is satisfied for $G_1$ and $G_2$ as in 
Lemma \ref{lemma-G_1-and-G_2-can-be-chosen-to-satisfy-assumption-intersections-of-surfaces-are-generic}, as then $G_1\cap G_2=\{0\}$.} 
an element of order $2$, then 
the projection $p:G\rightarrow \bar{G}$ is an isomorphism. 
\end{lem}

\begin{proof}
Note that elements of the subgroup $G$ of the identity component $X\times\hat{X}\times\Pic^0(X\times\hat{X})$ of $\Aut(D^b(X\times\hat{X}))$ are determined by their action on 
sky-scraper sheaves and line bundles.
Given $x\in X$, denote by $\tau_x$ the translation automorphism of $X$. 
Set $L_x:=\Theta\otimes\tau_{x,*}(\Theta)^{-1}$. The third isomorphism below is due to the Theorem of the Square. 
\[
L_x\cong\Theta\otimes\tau_{x,*}(\Theta)^{-1}\cong
[\tau_{x,*}(\Theta)\otimes\Theta^{-1}]^{-1}\cong
\tau_{-x,*}(\Theta)\otimes\Theta^{-1}\cong 
\tau_x^*(\Theta)\otimes\Theta^{-1}=:\phi_\Theta(x).
\]
The autoequivalence $\tau_x^\Theta:=(\Theta\otimes)\circ\tau_{x,*}\circ(\Theta^{-1}\otimes)$ of $D^b(X)$
corresponds to the element $(\tau_{x,*},L_x)$ of $X\times \Pic^0(X)$.

Let $(x_1,x_2)\in G_1\times G_2$. 
The compositions $\mu^*\circ (\tau_{x_1}^\Theta\boxtimes\tau_{x_2}^\Theta)  \circ\mu_*$ and
$\mu^{-1}_*\circ (\tau_{x_1}^\Theta\boxtimes\tau_{x_2}^\Theta)  \circ\mu_*$ are equal. We have:
\[
\mu^{-1}((\tau_{x_1},\tau_{x_2})(\mu(x,y)))=\mu^{-1}(x+y+x_1,y+x_2)=(x+x_1-x_2,y+x_2).
\]
Hence,  $\mu^*\circ (\tau_{x_1},\tau_{x_2})_*  \circ\mu_*=(\tau_{x_1-x_2},\tau_{x_2})_*.$

Given a line bundle $M$ over $X\times X$, we have 
$\mu^*\circ (M\otimes)\circ \mu_*\cong (\mu^*M\otimes)\circ \mu^*\circ\mu_*\cong (\mu^*M\otimes)$.
Now, $\mu^*(\pi_1^*L_{x_1}\otimes \pi_2^*L_{x_2})\cong a^*L_{x_1}\otimes \pi_2^*L_{x_2}$, where $a:X\times X\rightarrow X$ is the addition. Now $a^*L_{x_1}\cong \pi_1^*L_{x_1}\otimes \pi_2^*L_{x_1}$ and $L_{x_1}\otimes L_{x_2}\cong L_{x_1+x_2},$ by the theorem of the square.
We get
\[
\mu^*\circ (\tau_{x_1}^\Theta\boxtimes\tau_{x_1}^\Theta)  \circ\mu_*\cong
((\pi_1^*L_{x_1}\otimes \pi_2^*(L_{x_1+x_2}))\otimes)\circ (\tau_{x_1-x_2},\tau_{x_2})_*.
\]

Let $\P_x$ be the restriction of the Poincar\'{e} line bundle to $\{x\}\times \hat{X}$. 
The autoequivalence $\Psi_{\P^{-1}[3]}\circ \tau_{x,*}\circ\Phi_\P$ of $D^b(\hat{X})$ is isomorphic to tensorization by the line bundle $\P_{-x}$, as we have
\begin{eqnarray*}
\Psi_{\P^{-1}[3]}(\tau_{x,*}(\Phi_\P(\P^{-1}_y[3])))
&\cong&\Psi_{\P^{-1}[3]}(\tau_{x,*}(\Phi_\P(\Psi_{\P^{-1}[3]}(\CC_y))))
\cong\Psi_{\P^{-1}[3]}(\tau_{x,*}(\CC_y))
\\
&\cong&\Psi_{\P^{-1}[3]}(\CC_{x+y})\cong\P^{-1}_{x+y}[3]\cong
\P^{-1}_y[3]\otimes\P^{-1}_x.
\end{eqnarray*}
The autoequivalence $\Psi_{\P^{-1}[3]}\circ (L_x\otimes) \circ\Phi_\P$ is translation $\tau_{L_x,*}$ 
by the point of $\hat{X}$ corresponding to the isomorphism class of $L_x$, as we have
\[
\Psi_{\P^{-1}[3]}((L_x\otimes\Phi_\P(\CC_{L_y})))\cong \Psi_{\P^{-1}[3]}((L_x\otimes L_y))\cong
\Psi_{\P^{-1}[3]}(\Phi_\P(\CC_{L_{x+y}}))\cong \CC_{L_{x+y}}.
\] 
We get the isomorphism
\begin{equation}
\label{eq-elements-of-G}
\tilde{\Phi}\circ (\tau_{x_1},\tau_{x_2})_*\circ \tilde{\Phi}^{-1}\cong
((\pi_1^*L_{x_1}\otimes \pi_2^*\P_{-x_2})\otimes)\circ (\tau_{x_1-x_2},\tau_{L_{x_1+x_2}})_*
\end{equation}
The element  $(\tau_{x_1-x_2},\tau_{L_{x_1+x_2}})$ of $\bar{G}$ is the identity, if and only if $x_1=x_2$ and $x_1+x_2=0$,
so that $x_1$ is a point of order $2$ of $G_1\cap G_2$.
\end{proof} 

Assume that $G_1$ and $G_2$ are chosen as in Lemma \ref{lemma-G_1-and-G_2-can-be-chosen-to-satisfy-assumption-intersections-of-surfaces-are-generic}.
In particular, $G_1\cap G_2=\{0\}$. 

\begin{lem}
If $d$ is even, then 
the divisibility $\mbox{div}(\det(\G))$ is relatively prime to 
the order $(d+1)^2$ of $G$. 
\end{lem}

\begin{proof}
Let $q:G\rightarrow \hat{G}\subset \Pic^0(X\times\hat{X})$ be the projection.
Both $p:G\rightarrow \bar{G}$ and $q:G\rightarrow \hat{G}$ are isomorphism.
The order $(d+1)^2$ of $\hat{G}$ is relatively prime to the rank $8d$ of $\G$ and so the map
$G\rightarrow\Pic(X\times\hat{X})$, given by $g\mapsto \det(\G\otimes q(g))=\det(\G)\otimes q(g)^{\rank(\G)}$ is injective. The $G$-equivariance of $\G$ implies 
that the map $G\rightarrow \Pic(X\times\hat{X})$, given by $g\mapsto \det(\tau_{p(g),*}(\G))$ must be injective as well, since for $g\in G$ we have
\begin{eqnarray*}
\G&\cong& g(\G)\cong \tau_{p(g),*}(\G)\otimes q(g),
\\
\det(\G)&\cong & \tau_{p(g),*}(\det(\G))\otimes q(g)^{8d}.
\end{eqnarray*}
Let $\phi_{\det(\G)}:X\times\hat{X}\rightarrow \Pic^0(X\times\hat{X})$ be the homomorphism sending a point $y$ to $\tau_y^*\det(\G)\otimes\det(\G)^{-1}$.
We have
\begin{eqnarray*}
\phi_{\det(\G)}(p(g))&\cong& \phi_{\det(\G)}(-p(g))^{-1}:= [\tau_{p(g),*}(\det(\G))\otimes\det(\G)^{-1}]^{-1}\cong q(g)^{8d}. 
\end{eqnarray*}
It follows that $\bar{G}$ intersects trivially the kernel of $\phi_{\det(\G)}$.
Hence, the order of $\bar{G}$ is relatively prime to the divisibility of $\det(\G)$.
\end{proof}

Regardless of the parity of $d$, the intersection $\bar{G}\cap \ker(\phi_{\det(\G)})$ consists of the elements of $\bar{G}$ of order dividing $\gcd(d+1,8)$, by proof of the above lemma.

%\begin{rem}
%(???) Use (\ref{eq-elements-of-G}) and Proposition \ref{prop-extension-class-of-decreasing-filtration-of-spin-V-representations} 
%to show that the divisibility of $c_1(\E)$ is relatively prime to $8d$. It would follow that the generic deformation of $\E$ is very twisted, hence slope stable.
%\end{rem}

Given $g\in G$, let $(x_g,L_g)\in (X\times\hat{X})\times \Pic^0(X\times\hat{X})$ be the point, such that the autoequivalence $g$ of $D^b(X\times\hat{X})$
is isomorphic to $L_g\otimes \tau_{x_g,*}$. The group $G$ has exponent $n=d+1$. 
We identify $G$ with the corresponding subgroup of $X\times\hat{X}\times \Pic^0(X\times\hat{X})$. 
This identification provides a choice of an isomorphism $L_g^n\cong \StructureSheaf{X\times\hat{X}}$, for every $g\in G$, since $g^n=L_g^n\otimes\tau_{nx_g,*}$ is the identity endofunctor and $nx_g$ is the identity element of the group $X\times\hat{X}$, so that $\tau_{x_g,*}^n$ is the identity endofunctor. 
The rank $r$ of $E$ is $8d$. 
%$(D,\wedge^r\lambda):=\det(E,\lambda)$. The latter is an object 
If $d$ is even, then $\gcd(r,n)=\gcd(8d,d+1)=1$ and there exists a positive integer $a$, such that $ar\equiv -1$ modulo $n$.
Set $D:=\det(\G)$.

\begin{lem}
\label{lemma-if-d-is-even-E-descends}
If $d$ is even,\footnote{Note that $\QQ(\sqrt{-d})=\QQ(\sqrt{-4d})$, so the assumption that $d$ is even does not restrict the compex multiplications we can treat.} then the object $\G\otimes D^a$ admits a $\bar{G}$-linearization.
%$\lambda\otimes (\wedge^r\lambda)^a$. 
\end{lem}

\begin{proof}
The linearization isomorphisms $\lambda_g:\G\rightarrow \tau_{x_g,*}(\G)\otimes L_g$, $g\in G$,
induce the isomorphisms
\begin{eqnarray*}
\wedge^r\lambda_g:D&\rightarrow & \tau_{x_g,*}(D)\otimes L_g^r,
\\
\lambda_g\otimes (\wedge^r \lambda_g)^a : \G\otimes D^a
&\rightarrow & \tau_{x_g,*}(\G\otimes D^a)\otimes L_g^{ar+1}.
\end{eqnarray*}
Now, tensorization with $L_g^{ar+1}$ is a power of the identity endofunctor,
%canonically isomorphic to $\StructureSheaf{X\times\hat{X}}$, 
as noted above. Hence, $\lambda\otimes (\wedge^r\lambda)^a$ is a $\bar{G}$-linearization for $\G\otimes D^a$.
\end{proof}

We proceed without any assumption on the parity of $d$. Set $U:=X\times\hat{X}\setminus \tilde{\Theta}$ and let $\tilde{\iota}:U\rightarrow X\times\hat{X}$ be the inclusion. 
Clearly, $U$ is $\bar{G}$-invariant. The reflexive sheaf $\E$ of Proposition \ref{prop-local-freeness}
is locally free over $U$ and we denote by $\PP(\tilde{\iota}^*\E)$ the projectivization of the restriction of $\E$ to $U$. 
Let $\pi:\PP(\tilde{\iota}^*\E)\rightarrow U$ be the natural projection. 

Set $G^\vee:=\{(x_g,L_g^{-1}) \ : \ g\in G\}$, so that the functor $L_g^{-1}\otimes\tau_{x_g,*}$ belongs to $G^\vee$, if and only if $L_g\otimes\tau_{x_g,*}$ belongs to $G$. 
The $G$-linearization $\lambda$ of $\G$ induces a $G$-linearization on the first sheaf cohomology $\G^1$ of $\G$, which we denote by $\lambda$ as well.
The isomorphisms $\lambda_g:L_g\otimes\tau_{x_g,*}(\G^1)\rightarrow \G^1$ yield the isomorphisms
$(\lambda_g^\vee)^{-1}:L_g^{-1}\otimes\tau_{x_g,*}(\E)\rightarrow \E$, as $\E$ is defined to be $(\G^1)^*$. In particular, we get a $G^\vee$-linearization $\nu:=(\lambda^\vee)^{-1}$ of $\E$.
The linearization of
$\E$ induces a $\bar{G}$-action on $\PP(\tilde{\iota}^*\E)$, so that $\pi$ is $\bar{G}$-equivariant. 
Set $Y:=[X\times\hat{X}]/\bar{G}$ and let $q:X\times \hat{X}\rightarrow Y$ be the quotient morphism. 
Set $Y^0:=U/\bar{G}$ and let $\iota:Y^0\rightarrow Y$ be the inclusion.
Then $\PP(\tilde{\iota}^*\E)$ descends to a projective bundle $B$ over $Y^0$. Denote by $\A$ the Azumaya algebra of $B$. 
It extends to a reflexive sheaf over $Y$, by the Main Theorem of \cite{siu}. We denote the extension to the whole of $Y$ by $\A$ as well.
Then $\A$ is a coherent sheaf of associative algebras with a unit over $Y$, locally isomorphic to the sheaf of endomorphisms of a reflexive sheaf,  and $q^*\A$ is naturally isomorphic to $\SheafEnd(\E)$. Let $\A_0$ be the kernel of the trace homomorphism $tr:\A\rightarrow \StructureSheaf{Y}$.
We have the direct sum decomposition $\A:=\A_0\oplus \StructureSheaf{Y}$, where the second direct summand is generated by the unit. Set $r:=\rank(\E)$. 

\begin{lem}
\label{lemma-E-descends-to-a-reflexive-twisted-sheaf-B-over-Y}
There exists a reflexive sheaf $\B$ over $Y$ twisted by a $2$-cocycle with coefficients in $\mu_r$ and with trivial determinant, which is locally free over $Y_0$ and such that 
$\PP(\iota^*\B)\cong B$.
\end{lem}

\begin{proof}
\underline{Step 1:} We first construct a twisted reflexive sheaf $\B'$ which is locally free over $Y_0$ and such that 
$\PP(\iota^*\B')\cong B$. Choose a covering $\bar{\U}:=\{\bar{U}_j\}_{j\in \bar{I}}$ of $Y$, open in the analytic topology and consisting of subsets biholomorphic to open polydiscs, such that $q^{-1}(\bar{U}_j)$
is a disjoint union of open subsets to each of which $q$ restricts as an injective map, and such that if $U_i$ is a connected component of $q^{-1}(\bar{U}_{\bar{i}})$ and
$U_j$ is a connected component of $q^{-1}(\bar{U}_{\bar{j}})$ and $\bar{U}_{\bar{i}}\cap \bar{U}_{\bar{j}}\neq \emptyset$, then there exists a unique element $g_{ij}\in\bar{G}$, such that $\tau_{g_{ij}}(U_j)$ intersects $U_i$.
%If such $g\in\bar{G}$ exists, we denote it by $g_{ij}$. 
Let $L_{ij}$ be the unique line bundle, such that $(g_{ij},L_{ij})$ is an element of $G^\vee$ corresponding to the autoequivalence $L_{ij}\otimes\tau_{g_{ij},*}$.
Denote by $\U:=\{U_i\}_{i\in I}$ the open covering of $X\times X$ obtained by the connected components of $q^{-1}(U_j)$, $j\in\bar{I}$. Given $i\in I$, denote by $\bar{i}\in\bar{I}$ the index of $q(U_i)$. Let $\V:=\{V_i\}_{i\in I}$ be the open covering of $Y$ with $V_i=\bar{U}_{\bar{i}}$. Note that the restriction $q_i:U_i\rightarrow V_i$ of $q$ is biholomorphic. If non-empty, the intersection  $V_{ij}:=V_i\cap V_j$ is the isomorphic image of $U_i\cap \tau_{g_{ij}}(U_j)$ via $q$.

Let $\E_i$ be the restriction of $\E$ to $U_i$. Set $\B'_i:=q_{i,*}(\E_i)$. If $V_{ij}\neq\emptyset$, let
$
\rho_{ij}:\B'_{\restricted{j}{V_{ij}}}\rightarrow \B'_{\restricted{i}{V_{ij}}}
$
be the composition
\[
\B'_{\restricted{j}{V_{ij}}}=
q_{j,*}(\E_j{\restricted{)}{V_{ij}}}\RightArrowOf{q_*(\bar{\lambda}_{ij})}
%q_*(\tau_{g_{ij}}^*(\E_i)\restricted{)}{V_{ij}}=
q_{j,*}(\tau_{g_{ji,*}}(\E_i)\restricted{)}{V_{ij}}=
q_{i,*}(\E_i\restricted{)}{V_{ij}}=
\B'_{\restricted{i}{V_{ij}}}, 
\]
where 
$\bar{\lambda}_{ij}:\restricted{\E}{U_j\cap\tau_{g_{ji}}(U_i)}\rightarrow (\tau_{g_{ji},*}\E\restricted{)}{U_j\cap\tau_{g_{ji}}(U_i)}$ is induced by the 
isomorphism $\nu_{g_{ji}}:\E\rightarrow (\tau_{g_{ji},*}\E)\otimes L_{ji}$ of the 
$G^\vee$-linearization 
of $\E$ and by a choice of a trivialization of $L_{ji}$ over $U_j$. The sheaves $\B'_i$ are clearly locally free over $V_i\cap Y_0$. If $V_{ijk}:=V_i\cap V_j\cap V_k$ is non-empty, then 
the composition $\rho_{ij}\rho_{jk}\rho_{ki}$ is an invertible holomorphic function times the identity automorphism, since $g_{ij}g_{jk}g_{ki}$ is the identity in $\bar{G}$ and so $\bar{\lambda}_{ij}\circ\tau_{g_{ij}}^*(\bar{\lambda}_{jk})\circ\tau_{g_{ik}}^*(\bar{\lambda}_{ki})$ is an invertible holomorphic function times the identity automorphism. 
The twisted sheaf $\B':=\{\B'_i,\rho_{ij}\}_{i,j\in I}$ is locally free over $Y_0$ and
the projectivisation of $\iota^*\B'$ is naturally isomorphic to $B$.

\underline{Step 2:} We change the gluing transformations of the twisted sheaf $\B'$ by a $1$-cochain with coefficients in $\StructureSheaf{Y}^\times$ to obtain a twisted sheaf $\B$ twisted by a $2$-cocycle with coefficients in $\mu_r$. Choose trivializations $\psi_i:\det(\B'_i)\rightarrow \StructureSheaf{U_i}$.
Note that the determinant line bundle is defined for every complex of coherent sheaves, and for our reflexive sheaves $\det(\B'_i)$ is simply the pushforward of 
$\wedge^r(\B'_{\restricted{i}{V_i\cap Y_0}})$ from $V_i\cap Y_0$ to $V_i$. We get the holomorphic functions
$\eta_{ij}:=\psi_i\circ \det(\rho_{ij})\circ \psi_j^{-1}$ over $V_{ij}$. 
We have
\begin{eqnarray*}
\eta_{ij}\eta_{jk}\eta_{ki}&=&
\psi_i\circ \det(\rho_{ij})\psi_j^{-1}
\psi_j\circ \det(\rho_{jk})\psi_k^{-1}
\psi_k\circ \det(\rho_{ki})\psi_i^{-1}
\\
&=&\psi_i\det(\rho_{ij}\rho_{jk}\rho_{ki})\psi_i^{-1}=\det(\rho_{ij}\rho_{jk}\rho_{ki}).
\end{eqnarray*}
Choose $r$-th roots $\tilde{\eta}_{ij}$ of $\eta_{ij}$ (here we further assume that the open covering $\bar{\U}$ was chosen so that $\bar{U}_{ij}$ are simply connected or empty). Set $\rho''_{ij}:=\tilde{\eta}_{ij}^{-1}\rho_{ij}$. Then
$\rho''_{ij}\rho''_{jk}\rho''_{ki}=(\tilde{\eta}_{ij}\tilde{\eta}_{jk}\tilde{\eta}_{ki})^{-1}\rho_{ij}\rho_{jk}\rho_{ki}$ is again a holomorphic invertible function times the identity automorphism, and 
\begin{equation}
\label{eq-cocycle-condition-for-det-B'}
\det(\rho''_{ij})\det(\rho''_{jk})\det(\rho''_{ki})=(\tilde{\eta}_{ij}\tilde{\eta}_{jk}\tilde{\eta}_{ki})^{-r}\det(\rho_{ij}\rho_{jk}\rho_{ki})=
(\eta_{ij}\eta_{jk}\eta_{ki})^{-1}\det(\rho_{ij}\rho_{jk}\rho_{ki})
=1.
\end{equation}
Hence, $\rho''_{ij}\rho''_{jk}\rho''_{ki}=\theta_{ijk}\one$, where $\one$ is the identity endomorphism of $\B'_{\restricted{i}{V_{ijk}}}$ and 
the $2$-cocycle $\{\theta_{ijk}\}$ has coefficients in $\mu_r$. We set 
$\B:=\{\B'_i,\rho_{ij}''\}_{i,j\in I}$.

\underline{Step 3:} 
The gluing transformations of  $\det(\B)=\{\wedge^r\B'_i,\det(\rho_{ij}'')\}$ satisfy the co-cycle condition, by equation (\ref{eq-cocycle-condition-for-det-B'}).
%are given by $\det(\rho_{ij}'')=\eta_{ij}^{-1}\det(\rho_{ij})$. 
Hence, $\det(\B)$ is an untwisted line bundle.
We have the equality 
$\psi_i\circ \eta_{ij}^{-1}\det(\rho_{ij}))\circ \psi_j^{-1}=1$, by definition of $\eta_{ij}$. Hence, the trivializations $\psi_i$ from Step 2 of the restrictions $\det(\B'_i)$ of $\det(\B)$ to $V_i$ glue to an isomorphism $\psi:\det(\B)\rightarrow \StructureSheaf{Y}$.
%We change the gluing transformations of the twisted sheaf $\B"$ by a $1$-cochain with coefficients in $\StructureSheaf{Y}^\times$ 
%to obtain a twisted sheaf $\B$ twisted by a $2$-cocycle with coefficients in $\mu_r$ with a trivial determinant line bundle. 
\end{proof}

%Choose a locally free sheaf $\B$ over $Y^0$ twisted by a Brauer class in the image of $H^2(Y^0,\mu_r)$ in $H^2(Y_0,\StructureSheaf{Y_0}^*)$ 
%with $\PP(\B)=B$. 
%Note that the restriction homomorphism $H^2(Y,\mu_r)\rightarrow H^2(Y_0,\mu_r)$ is an  isomorphism, 
%since the complex codimension of the complement of $Y_0$ is $4$, by Proposition \ref{prop-local-freeness}.
%The sheaf $\B$ extends to a unique reflexive sheaf over $Y$, which we denote again by $\B$, by the Main Theorem of \cite{siu}.

Let $\B$ be a twisted sheaf as in Lemma \ref{lemma-E-descends-to-a-reflexive-twisted-sheaf-B-over-Y}.
We have the equality
%Then $\tilde{\iota}^*q^*at^{}_\B=at_{\PP(\tilde{\iota}^*\E)}$, by Lemma \ref{lemma-atiyah-class-of-a-reflexive-twisted-sheaf} 
%and \cite[page 189 property (2)]{atiyah}, and the latter is the traceless direct summand of $at_{\tilde{\iota}^*\E}=\tilde{\iota}^*at_\E$, 
%by \cite[page 189 property (1)]{atiyah}. Consequently, 
\begin{equation}
\label{eq-at-tw-B-versus-at-E}
q^*at^{}_\B=at_\E-\frac{c_1(\E)}{r}\cdot id_\E,
\end{equation} 
%by the injectivity Lemma \ref{lemma-restriction-of-extension-classes-is-injective}
%(keeping the convention of footnote \ref{footnote-atiyah-class} in Section \ref{section-semiregular-twisted-sheaves}). 
by Equation (\ref{eq-atiya-class-of-twisted-B-versus-that-of-E}).

\begin{rem}
If $d$ is even, then the Brauer class of $\B$ in $H^2(Y,\StructureSheaf{Y}^*)$ is trivial, since $(\G\otimes D^a,\lambda\otimes(\wedge^r\lambda)^a)$ descends
to an object in $D^b(Y)$, whose derived dual $\bar{\E}$ is represented by an untwisted reflexive sheaf, by Lemma \ref{lemma-if-d-is-even-E-descends}.
In this case $\B$ can be chosen to be associated to $\bar{\E}$ via Construction \ref{construction-twisted-sheaf-with-trivial-determinant-associated-to-a-coherent-sheaf}.
%However, $at_{\bar{\E}}$ is different from our $at_\B^{}$, as $tr(at_{\bar{\E}})$ need not vanish, as $\det(\E)$ need not be trivial. 
 %The Atiyah class of a twisted sheaf is defined in \cite[Sec. 6.5.1]{lieblich}. Different lifts of the projective bundle $B$ over $Y_0$ to a 
 %twisted sheaf have different Atiyah classes, and our $\B$ comes from a $\mu_r$-twisted sheaf, which yields the vanishing of $tr(at_\B^{})$.
\end{rem}

%We view the class $at^{}_\B$ as a class in $H^1(Y,\A\otimes\Omega^1_Y)$ via the inclusion of $\A_0$ as a direct summand in $\A$. 
Consider the exponential Atiyah class $\exp(at^{}_\B)$ with graded summands $at_k(\B)$ in 
$\Hom(\B,\B\otimes\Omega_Y^k[k])$.
%$H^k(Y,\A\otimes\Omega^k_Y)$. 
Note that $at^{}_0(\B)$ is $r$ times the unit section and $at^{}_1(\B)=at^{}_\B$. Note that $\kappa(\B)$ is the trace of $\exp(at^{}_\B)$ and $q^*\kappa(\B)=ch(\E)\exp(-c_1(\E)/r)$, by equation (\ref{eq-at-tw-B-versus-at-E}).
%We get the commutative diagram
%\begin{equation}
%\label{commutative-diagram-of-semiregularity-map-of-projective-bundle}
%\xymatrix{
%H^1(Y,TY) \ar[rr]^{at^{}_\B} \ar[dr]^{\Contract \kappa(B)} & & H^2(Y,\A_0) \ar[dl]_{\sigma_B}
%\\
%& 
%\prod_{q=1}^4H^{q+2}(Y,\Omega_Y^q),
%}
%\end{equation}
%where the {\em semiregularity map} $\sigma_B$ is defined as in equation (\ref{eq-semiregularity-map}) with $at_E$ 
%replaced by $at_B$. 

Set $M:=X\times\hat{X}$ and $\lambda:=c_1(\E)/r$.
Let $\Contract\exp(\lambda):H^1(TM)\rightarrow HT^2(M)$ be the composition 
\begin{equation}
\label{eq-action-of-tensorization-by-a-line-bundle-on-HT}
H^1(TM)\RightArrowOf{\subset} 
\begin{array}{c}
H^2(\StructureSheaf{M})
\\
H^1(TM)
\\
H^0(\wedge^2TM)
\end{array}
\LongRightArrowOf{\Contract\left(
\begin{array}{ccc}
1 & \lambda & \lambda^2/2
\\
0 & 1 & \lambda
\\
0 & 0 & 1
\end{array}
\right)}
\begin{array}{c}
H^2(\StructureSheaf{M})
\\
H^1(TM)
\\
H^0(\wedge^2TM)
\end{array}
\end{equation}
which has image in $H^2(\StructureSheaf{M})\oplus H^1(TM)$.

\begin{lem}
\label{lemma-action-of-tensorization-by-a-lb-on-HT}
The following diagram is commutative
\[
\xymatrix{
HT^2(M) \ar[r] \ar[d]_{\Contract \exp(\lambda)} & \End(H\Omega^*(M))\ar[d]^{Ad_{\exp(-\lambda)}} &\zeta\ar@{|->}[d]
\\
HT^2(M)\ar[r] & \End(H\Omega^*(M)) & 
\exp(-\lambda)\cup[\zeta\circ(\exp(\lambda)\cup(\bullet))]
%\zeta\circ(\exp(\lambda)\cup(\bullet))-\exp(\lambda)\cup(\zeta(\bullet)),
}
\]
where the horizontal arrows are defined via the $HT^*(M)$-module structure of $H\Omega^*(M)$ and the right vertical arrow is the conjugation by cup product with $\exp(-\lambda)$.
\end{lem}

\begin{proof}
An element $\zeta\in H^2(\StructureSheaf{M})\subset HT^2(M)$ is mapped to itself by $\Contract\lambda$ and similarly, its action on $H\Omega^*(M)$  commutes with cup product with $\exp(-\lambda)$ and so it is mapped to itself by $Ad_{\exp(-\lambda)}$. Given $\zeta\in H^1(TM)$ and $\gamma\in H\Omega^*(M)$ we have
\begin{eqnarray*}
(Ad_{\exp(-\lambda)}(\zeta))(\gamma)&=&\exp(-\lambda)\cup\left\{\zeta\Contract(\exp(\lambda)\cup\gamma)\right\}
\\
&=& \exp(-\lambda)\cup\left\{
\zeta\Contract \left[
\gamma+\lambda\cup\gamma+\frac{\lambda^2}{2}\cup\gamma+\frac{\lambda^3}{3!}\cup\gamma\cdots
\right]
\right\}
\\
&=&
\exp(-\lambda)\cup\left\{
\exp(\lambda)\cup(\zeta\Contract\gamma)+\exp(\lambda)\cup(\zeta\Contract \lambda)\cup\gamma
\right\}
\\
&=& (\zeta\Contract \lambda)\cup\gamma +(\zeta\Contract\gamma).
\end{eqnarray*}
The proof for $\zeta\in H^0(\wedge^2TM)$ is similar.
\end{proof}

\begin{lem}
\label{lemma-huge-diagram-is-commutative}
The following diagram is commutative.
\begin{equation}
\label{diagram-huge}
\xymatrix{
HT^2(M) \ar@/_8pc/[dddddrr]_{\Contract ch(\E)} \ar[rrrr]^{ob_\E:=\exp(at_\E)} &&&& \Ext^2(\E,\E)^{G^\vee} \ar@/^8pc/[dddddll]^{\sigma_\E}
\\
&H^1(TM) \ar[ul]^{\Contract\exp(-\lambda)}_{(\ref{eq-action-of-tensorization-by-a-line-bundle-on-HT})} \ar[rr]^{at_\E-\lambda\cdot id_\E} \ar@/_10pc/[dddr]^{\Contract\kappa(\E)} &&\Ext^2_0(\E,\E)^{G^\vee} \ar[ur]_{\subset} \ar@/^10pc/[dddl]_{q^*\sigma^{}_\B}
\\
& H^1(TY) \ar[u]_{q^*} \ar[rr]^{at^{}_\B} \ar[dr]_{\Contract\kappa(\B)} && H^2(Y,\A_0\!) \ar[u]^{q^*} \ar[dl]^{\sigma^{}_\B}
\\
&& \prod_{q=1}^4H^{q+2}(\Omega_Y^q) \ar[d]_{q^*}
\\
&& \prod_{q=1}^4H^{q+2}(\Omega_M^q) \ar[d]_{\cup\exp(\lambda)}
\\
&& \prod_{q=0}^4H^{q+2}(\Omega_M^q)
}
\end{equation}
\end{lem}

\begin{proof}
The image of $ob_\E$ in $\Ext^2(\E,\E)$ is $\Ext^2(\E,\E)^{G^\vee}$, by Lemma \ref{lemma-image-of-ob-E-is-the-G-invariant-subspace}.
The commutativity of the outer triangle of the diagram follows from the commutativity of diagram (\ref{eq-extended-diagram-semi-regularity}).
 If the class $\lambda$ is the first Chern class of a line bundle $L$, then the right arrow in (\ref{eq-action-of-tensorization-by-a-line-bundle-on-HT}) is the action
 %\footnote{Indeed, one easily checks that the following diagram is commutative
%\[
%\xymatrix{
%HT^2(M) \ar[r] \ar[d]_{\Contract \exp(\lambda)} & \End(H\Omega^*(M))\ar[d]^{Ad_{\exp(-\lambda)}} &\eta\ar@{|->}[d]
%\\
%HT^2(M)\ar[r] & \End(H\Omega^*(M)) & 
%\exp(-\lambda)\cup[\eta\circ(\exp(\lambda)\cup(\bullet))]
%\eta\circ(\exp(\lambda)\cup(\bullet))-\exp(\lambda)\cup(\eta(\bullet)),
%}
%\]
%where the horizontal arrows are defined via the $HT^*(M)$-module structure of $H\Omega^*(M)$ and the right vertical arrow 
%is the conjugation by cup product with $\exp(-\lambda)$.
%} 
of the derived equivalence of tensorization by $L$ on $HT^2(M)$, by Lemma \ref{lemma-action-of-tensorization-by-a-lb-on-HT}. In that case the functor $L^{-1}\otimes(\bullet)$ acts on $HT^2(M)$ as in Lemma \ref{lemma-action-of-tensorization-by-a-lb-on-HT}, by \cite{calaque-et-al},  and trivially on $\Ext^2(\E,\E)$, as we identify the latter with $\Ext^2(\E\otimes L^{-1},\E\otimes L^{-1})$, and the functor $L\otimes (\bullet)$ conjugates the Atiyah class $at_{\E\otimes L^{-1}}=at_\E-\lambda\cdot id_\E$ to $at_\E$. 
The commutativity of the upper trapezoid of Diagram (\ref{diagram-huge}) follows in this case. If $\lambda$ is not integral, then the commutativity of the upper trapezoid follows from the integral case, as the vanishing of the difference of the two linear homomorphisms
in $\Hom(H^1(TM),\Ext^2(\E,\E))$ for integral values of $\lambda$ implies the vanishing of the differences over the Zariski closure of the $\Pic(M)$-orbits in $\Hom(H^1(TM),\Ext^2(\E,\E))$. 

 The cohomological action of $L\otimes(\bullet)$ on $H^*(M,\CC)$ is multiplication by $\exp(\lambda)$.
 When $\lambda=c_1(L)$, for some line bundle $L$, and 
$\zeta\in HT^2(M)$  Lemma \ref{lemma-action-of-tensorization-by-a-lb-on-HT} and the equality $ch(\E)=\exp(\lambda)\cup\kappa(\E)$ yield the equality
 \[
 \exp(-\lambda)\cup(\zeta\Contract ch(\E))=(\zeta\Contract \exp(\lambda))\Contract \kappa(\E).
 \]
 Substitute $\zeta\Contract \exp(-\lambda)$ for $\zeta$ above and multiply on the left by $\exp(\lambda)$  to get the equivalent equality
 \[
(\zeta\Contract \exp(-\lambda))\Contract ch(\E)=\exp(\lambda)\cup( \zeta\Contract \kappa(\E)).
 \]
 The commutativity of the leftmost square follows in this case. It follows in general, by the argument used above taking the  Zariski closure of the $\Pic(M)$-orbit. 
 
 The commutativity of the rightmost square follows from Equation (\ref{eq-at-tw-B-versus-at-E}).
The commutativity of the three inner squares is clear. 
\end{proof}

\begin{rem}
\label{remark-commutativity-of-diagram-of-semiregularity-map-of-projective-bundle}
Note that Diagram (\ref{commutative-diagram-of-semiregularity-map-of-projective-bundle}) appears as the inner triangle in the above diagram. 
Lemma \ref{lemma-huge-diagram-is-commutative} provides an alternative proof for the commutativity of Diagram (\ref{commutative-diagram-of-semiregularity-map-of-projective-bundle}), as it
%is proved by the same argument as in \cite[Cor. 4.3]{buchweitz-flenner} establishing the commutativity of 
%Diagram (\ref{eq-diagram-semi-regularity}).
follows from the commutativity of the outer triangle and the six squares in Diagram (\ref{diagram-huge}), which implies the commutativity of the intermediate and inner triangles (as all arrows labeled $q^*$ are isomorphisms and the bottom vertical arrow is injective). 
\end{rem}

\begin{lem}
\label{lemma-semiregularity-of-twisted-sheaf-B}
The twisted reflexive sheaf $\B$ is semiregular.
%semiregularrity map $\sigma_B$ is injective.
\end{lem}

\begin{proof}
The image of $ob_\E$ in $\Ext^2(\E,\E)$ is $\Ext^2(\E,\E)^{G^\vee}$, by Lemma \ref{lemma-image-of-ob-E-is-the-G-invariant-subspace}.
%***************
% Hide
%***************
\hide{ 
It suffices to prove that the image of $ob_{\pi_1^*\Ideal{\cup_{i=1}^n C_i}\otimes\pi_2^*\Ideal{\cup_{j=1}^n\Sigma_j}}$ is
$\Ext^2(\pi_1^*\Ideal{\cup_{i=1}^n C_i}\otimes\pi_2^*\Ideal{\cup_{j=1}^n\Sigma_j})^{G_1\times G_2}$.
The latter reduces to showing that (i) the image of $ob_{\Ideal{\cup_{i=1}^n C_i}}$ is equal to 
$\Ext^2(\Ideal{\cup_{i=1}^n C_i},\Ideal{\cup_{i=1}^n C_i})^{G_1}$, (ii) the analogous statement for $ob_{\Ideal{\cup_{i=1}^n \Sigma_j}}$, 
%since we have seen that $ob_{\pi_1^*\Ideal{\cup_{i=1}^n C_i}\otimes\pi_2^*\Ideal{\cup_{j=1}^n\Sigma_j}}$ restricts as an 
%injective map to $HT^1(X)\otimes HT^1(X)$ in the proof of Lemma \ref{lemma-kernel-of-ob-E-is-annihilator-of-ch_E}. 
(iii) that 
\[
\exp(at_{\Ideal{\cup_{i=1}^n C_i}}):HT^0(X)\rightarrow \Ext^1(\Ideal{\cup_{i=1}^n C_i},\Ideal{\cup_{i=1}^n C_i})^{G_1}
\] 
is surjective, and (iv) the analogous statement for $\Ideal{\cup_{j=1}^n \Sigma_j}$.
Now the equality of the image of $ob_{\Ideal{\cup_{i=1}^n C_i}}$ with $\Ext^2(\Ideal{\cup_{i=1}^n C_i},\Ideal{\cup_{i=1}^n C_i})^{G_1}$ follows from Lemma \ref{lemma-Yoneda-algebra-is-generated-in-degree-1} and 
Proposition \ref{prop-obstruction-map-has-rank-6}, establishing (i).
These two results also show that the natural map from the direct summand $H^1(\StructureSheaf{X})$ of $HT^1(X)$ into 
$\Ext^1(\Ideal{\cup_{i=1}^n C_i},\Ideal{\cup_{i=1}^n C_i})$ together with the image of $at_{\Ideal{\cup_{i=1}^n C_i}}:H^1(TX)\rightarrow \Ext^1(\Ideal{\cup_{i=1}^n C_i},\Ideal{\cup_{i=1}^n C_i})$ span $\Ext^1(\Ideal{\cup_{i=1}^n C_i},\Ideal{\cup_{i=1}^n C_i})^{G_1}$, establishing (iii). The proofs of (ii) and (iv) are analogous.
%***************
% End Hide
%***************
}
Consequently, the restriction of $\sigma_\E$ to $\Ext^2(\E,\E)^{G^\vee}$  is injective, by Remark \ref{remark-semiregularity-map-restricts-as-an-injective-map-to-the-image-of-ob}, Lemma \ref{lemma-kernel-of-ob-E-is-annihilator-of-ch_E}(\ref{lemma-item-kernel-of-ob-E-is-annihilator-of-ch_E}), and Remark \ref{remark-all-the-results-for-E-hold-for-G}. The injectivity of the restriction of $q^*\sigma^{}_{\B}$ to $\Ext^2(\E,\E)^{G^\vee}$ follows by the commutativity of the rightmost square of diagram (\ref{diagram-huge}). The injectivity of $\sigma^{}_\B$ follows by the commutativity of the square with edges $\sigma^{}_\B$ and $q^*\sigma^{}_{\B}$, as the arrows labeled $q^*$ are all isomorphisms.
\end{proof}

\begin{proof}[Proof of Theorem \ref{thm-algebraicity}]
The discriminant of the Hermitian form is $-1$, by Lemma \ref{lemma-the-discriminant-is-minus-1-to-the-n}.
The algebraicity of $\kappa_3(\B)$ follows in a non-empty open subset in moduli from the semiregularity of $\B$, proved in Lemma \ref{lemma-semiregularity-of-twisted-sheaf-B}, and Conjecture \ref{conjecture-semiregular-twisted-sheaves-deform}, as explained in the paragraph preceding the statement of Theorem \ref{thm-algebraicity}. The conjecture was verified in our case of families of abelian varieties in Section 
\ref{sec-proof-of-the-semiregularity-thm-twisted-sheaves-case-for-families-of-abelian-varieties}.
The algebraicity of the Hodge-Weil classes follows from that of $\kappa_3(\B)$ and Theorem \ref{main-theorem-introduction}(\ref{thm-item-K-translates-of-kappa-3-and-h-cube-span-HW}), since given a polarized abelian variety of Weil type $(A,\eta,h)$ the rational endomorphisms $\eta(K)\subset \End_\QQ(A)$
act on $H^*(A,\QQ)$ via algebraic correspondences. 
The locus in moduli where the Hodge-Weil classes are algebraic is a countable union of closed algebraic susbsets \cite[Sec. 4.2]{voisin}. Hence, the locus contains the whole irreducible component of moduli of deformations of $(X\times\hat{X},\eta,h)$.
%, as it contains the Zariski closure of the open analytic 
\end{proof}

%**********
% Hide
%**********
\hide{
%****************************************************************
% 
%****************************************************************
\section{A period domain of non-commutative deformations of $n$-dimensional  compact complex tori}

\begin{question}
\begin{enumerate}
\item
Lemma 
\ref{lemma-Orlov's-equivalence-maps-diagonal-deformations-to-commutative-gerby-ones} suggests that there exists a $15$-dimensional period domain of $6$-dimensional complex tori $Y$ corresponding to commutative-gerby deformations of $X\times\hat{X}$, which in turn correspond to all generalized deformations of $3$-dimensional compact complex tori, at least in the formal neighborhood of each $Y$ of the form $Y=X\times\hat{X}$. This should generalize to $n$-dimensional complex tori $X$ and a $n(2n-1)$-dimensional (over $\CC$) period domain $\D$. When $n=2$ we would get a $6$-dimensional period domain containing the $5$-dimensional period domains of generalized Kummers as codimension $1$ loci where a Mukai vector on $X$ remains algebraic. The $n(2n-1)$-dimensional period domain $\D$ should parametrize $2n$-dimensional compact complex tori $Y$ for which the symmetric bilinear pairing $(\bullet,\bullet)_V$ on $V:=H^1(Y,\ZZ)$ remains of Hodge type and induces an isomorphism $Y\cong \hat{Y}$. The period domain $\D$ should be the adjoint orbit of the complex structure $I$ of $X\times\hat{X}$ under the full real orthogonal group $SO(V_\RR)$. This adjoint orbit is isomorphic to an open subset of the grassmanian of $2n$-dimensional isotropic subspaces of the $4n$-dimensional $V_\RR$ transversal to their complex conjugate. The stabilizer of $I$ is $GL(2n,\RR)$ and so the dimension of $\D$ over $\RR$  is  indeed $\dim(SO(V_\RR))-\dim(GL(2n,\RR))=\frac{4n(4n-1)}{2}-(2n)^2=4n^2-2n=2n(2n-1)$.

Algebraically, the fact that the pairing $(\bullet,\bullet)_V$ remains of Hodge type is encoded by the line bundle $\LB$ over $Y\times Y$, which is the pullback of the Poincar\'{e} line bundle over $Y\times \hat{Y}$ via the isomorphism $Y\cong \hat{Y}$ induced by $(\bullet,\bullet)_V.$
\item
The universal complex torus $\T\rightarrow \D$, of relative dimension $6$,  should be thought of as the universal family of the identity components of the group of autoequivalences of NC-deformations of $3$-dimensional compact complex tori. Here NC means not necessarily commutative.
\item
The period domains of $2n$-dimensioonal abelian varieties of Weil type naturally embed in $\D$ as subloci, where the Chern character of the sheaf $E=\Phi(F^\vee\boxtimes F)$ over $X\times \hat{X}$ remains of Hodge type, for some secant sheaf $F$ on $X$.
\item
A monad on $Y\times Y$, which recovers the derived category of the $3$-dimensional $X'$, whenever $Y$ is $X'\times\hat{X'}$ for some complex torus $X'$, should exist over the moduli space of abelian varieties of Weil type, as the class $ch(F_i)$ remains algebraic. 
The middle characteristic class of the twisted sheaf  in the monad data should restrict to $Y\times\{pt\}$ as an analogue of the Cayley class in the case $n=2$. 
%******
% Hide
%*****
\hide{
\item
For a general (non-commutative) formal deformation $X'$ of $X$ there should be an object $\StructureSheaf{\Delta_{X'}}$
in $D^b(X'\times X')$ which is the Fourier-Mukai kernel of the identity. So over a general gerby complex torus $Y$ in the period domain $\D$, there should be an object which is the image of $\StructureSheaf{\Delta_{X'}}$ via the equivalence $\Phi':D^b(X'\times X')\rightarrow D^b(Y)$ deforming Orlov's equivalence $\Phi:D^b(X\times X)\rightarrow D^b(X\times\hat{X})$.
ndeed, $\mu^*(\StructureSheaf{\Delta_{X}})\cong \mu^{-1}_*(\StructureSheaf{\Delta_{X}})$. Now $\mu^{-1}(\Delta_X)=\{0\}\times X$. So $\Phi(\StructureSheaf{\Delta_{X}})\cong \pi_1^*(\StructureSheaf{0})\boxtimes\Phi_{\P}^{-1}(\StructureSheaf{X})\cong
\StructureSheaf{0\times\hat{0}}[\pm n]$ is a shift of a sky-scraper sheaf of the origin of $X\times\hat{X}$.
%*******
% End Hide
%******
}
\item
Every $Y$ in the period domain $\D$ should come with an involutive autoequivalence corresponding to the deformed transposition of the factors $\sigma(x,y)=(y,x)$ of $X\times X$. 
Describe this autoequivalence for $X\times\hat{X}$. 
Setting $\Phi=(id\boxtimes\Psi_{\P^{-1}[n]})\circ\mu^*:D^b(X\times X)\rightarrow D^b(X\times\hat{X})$ as above,
\[
\xymatrix{
D^b(X\times X) \ar[d]^{\sigma_*}\ar[r]^{\Phi} & D^b(X\times\hat{X})\ar[d]_{\Phi\circ\sigma_*\circ\Phi^{-1}}
\\
D^b(X\times X) \ar[r]_{\Phi}&D^b(X\times\hat{X}).
}
\]
Now $\Phi\circ\sigma_*\circ\Phi^{-1}=[(id\boxtimes\Psi_{\P^{-1}[n]})\mu^*]\sigma_*[\mu_*(id\boxtimes \Phi_\P)]$. Note the equality $\mu^*\sigma_*\mu_*=(\mu^{-1}\sigma\mu)_*$ and $(\mu^{-1}\sigma\mu)(x,y)=(-x,x+y)$.
Calculate first the cohomological involution associated to $\Phi\circ\sigma_*\circ\Phi^{-1}$ using \cite[Page 91]{chevalley}.
\[
\tau(uf\tau(v))=(-1)^{2n(2n-1)/2}vf\tau(u).
\]
where $u\otimes v\mapsto uf\tau(v)$ is $\tilde{\varphi}$ in Lemma \ref{lemma-nu-equal-tilde-varphi}. So
\[
\tilde{\varphi}\sigma_*\tilde{\varphi}^{-1}=(-1)^{2n(2n-1)/2}\tau.
\]
\end{enumerate}
\end{question}

Let $\Delta:X\times X\rightarrow (X\times X)^n$ and $\delta:X\rightarrow X^n$ be the diagonal embeddings.
The ideal $\Ideal{\cup_{i=1}^n C_i}$ is the pullback via $\Delta^*$ of the exterior product $\Ideal{C_1}\boxtimes\cdots\boxtimes\Ideal{C_n}$. 
Consider the commutative diagram
\[
\xymatrix{
D^b((X\times X)^n) \ar[r]^{\Phi^{\boxtimes n}} \ar[d]_{\Delta^*} & D^b((X\times\hat{X})^n) 
\ar[d]^{\Phi\circ \Delta^*\circ(\Phi^{\boxtimes n})^{-1}}
\\
D^b(X\times X) \ar[r]_{\Phi} & D^b(X\times\hat{X})
}
\]
Orlov's derived equivalence $\Phi$ should deform in a formal neighborhood of $X\times X$ over the germ of the moduli space of generalized deformations of $X$. Hence, we expect the functor $\Phi\circ \Delta^*\circ(\Phi^{\boxtimes n})^{-1}$
to deform over the period domain $\D$. 
\begin{question}
Calculate the Fourier-Mukai kernel of  $\Phi\circ \Delta^*\circ(\Phi^{\boxtimes n})^{-1}$.
\end{question}

Let $d=\dim(X)$.
$\Phi\circ \Delta^*\circ(\Phi^{\boxtimes n})^{-1}=[(id\times\Psi_{\P^{-1}[d]})\circ\mu^*]\circ \Delta^* \circ[\mu_*\circ (id\times\Phi_\P)]^{\boxtimes n}$. Now $\mu^*\circ\Delta^*\circ (\mu_*)^{\boxtimes n}=\Delta^*,$
as $(\mu^{-1})^n\circ \Delta\circ\mu=\Delta$. Hence,
\[
\Phi\circ \Delta^*\circ(\Phi^{\boxtimes n})^{-1}\cong \delta^*\boxtimes (\Phi_\P^{-1}\circ \delta^*\circ \Phi_\P^{\boxtimes n})
\]
Now, the Fourier-Mukai kernel of $\Phi_\P^{-1}\circ \delta^*\circ \Phi_\P^{\boxtimes n}:D^b(\hat{X}^n)\rightarrow D^b(\hat{X})$ 
is supported as a line bundle on the graph of the summation morphism $m:\hat{X}^n\rightarrow\hat{X}$. So the Fourier-Mukai kernel of $\Phi\circ \Delta^*\circ(\Phi^{\boxtimes n})^{-1}$ in $(X\times\hat{X})^{n+1}\cong X^{n+1}\times \hat{X}^{n+1}$
is supported as a line bundle on the product $\Gamma(\delta)\times \Gamma(m)$ of the graphs of $\delta$ and $m$. Note that they have dimensions $d$ and $nd$ respectively.

\begin{rem}
The subspaces of $HT^2(X)$ of first order deformations along which $\Ideal{C_i}$ and $\Ideal{\cup_{i=1}^nC_i}$
deform project to the same subspace of $H^1(TX)\oplus H^0(\wedge^2TX)$. This suggests that the formal germ of the moduli space of generalized deformations of $X$ along which these sheaves can be deformed are the same, but the gerby structures are different. The same should hold for the diagonal embedding into deformations of $X\times X$. However, the derived equivalence $\Phi:D^b(X\times X)\rightarrow D^b(X\times\hat{X})$  maps the diagonal embedding of the subspace $H^2(\StructureSheaf{X})$
of $HT^2(X)$ into a subspace of $HT^2(X\times\hat{X})$, which projects non-trivially into $H^1(X\times\hat{X},T[X\times\hat{X}])$.
So we can not expect to deform the images of both $\Ideal{C_i}$ and $\Ideal{\cup_{i=1}^nC_i}$ simultaneously. This should correspond to the fact that the object represented by the structure sheaf of the subvariety $\Gamma(\delta)\times \Gamma(m)$ would not deform globally.
\end{rem}

%**********
% End Hide
%**********
}

%**********
% Hide
%**********
\hide{
%****************************************************************
% 
%****************************************************************
\section{Abelian varieties with complex multiplication}
Let $K$ be a CM-field, i.e., a totally complex quadratic extension of a totally real field $F$. We do not assume that $K$ is a Galois extension  of $\QQ$. In this section 
we generalize to the case of abelian varieties with complex multiplication by $K$ the construction in sections \ref{section-abelian-2n-folds-of-Weil-type-from-rational-secants} to \ref{section-Orlov-equivalence}. 
We associate to a pair $F_1$, $F_2$ of secant sheaves on an abelian $n$-fold $X$ with multiplication by the totally real field $F$ an object $E$ in the derived category 
of the abelian $2n$-folds $A:=X\times\hat{X}$ with complex multiplication $\eta:K\rightarrow\End_\QQ(A)$, such that a Hodge-Weil class on $X\times\hat{X}$ is a characteristic class of $E$, and such that $\kappa(E)$ remains of Hodge type, under deformations of $(A,\eta)$ (Proposition \ref{prop-kappa-class-of-image-of-secant-class-yields-a-HW-class}).

%We relate the Chern character of secant sheaves on abelian $n$-folds with multiplication by the totally real field $F$ to Hodge-Weil classes on abelian $2n$-folds %with complex multiplication by  $K$ (Proposition \ref{prop-kappa-class-of-image-of-secant-class-yields-a-HW-class}).

%****************************************************************
% 
%****************************************************************
\subsection{A summary of the construction}
The Galois involution $\iota$  in $Gal(K/F)$ is a central element of $Gal(K/\QQ)$ and for every embedding $\sigma:K\rightarrow \CC$ we have $\sigma(\iota(t))=\overline{\sigma(t)}$, for all $t\in K$, where the right hand side is complex conjugation \cite[Sec. 4]{deligne-milne}. 
Set 
\[
e:=\dim_\QQ(K).
\] 
Let $\Sigma:=\Hom(K,\CC)$ be the set of complex embeddings. Note that the cardinality of $\Sigma$ is $e$.
The group $Gal(K/\QQ)$ acts on $\Sigma$ by $g(\sigma)=\sigma\circ g^{-1}$, for all $g\in Gal(K/\QQ)$ and $\sigma\in\Sigma$. The latter action is transitive. It is free if $K$ is a Galois extension of $\QQ$.

Given an abelian variety $A$ and an embedding
$\eta:K\rightarrow \End_\QQ(A)$, we get a decomposition
\[
H^1(A,\CC)=\oplus_{\sigma\in\Sigma}H^1_\sigma(A),
\]
where $\eta(K)$ acts on $H^1_\sigma(A)$ via the embedding $\sigma$. The complex multiplication $\eta$ corresponds to an embedding $\eta:K\rightarrow \End(H^1(A,\QQ))$ preserving the Hodge structure.
Set 
\[
d:=\dim_K(H^1(A,\QQ)).
\] 
The dimension of $H^1_\sigma(A)$ is $d$, for all $\sigma\in \Sigma$. 
Set $H^{1,0}_\sigma(A):=H^1_\sigma(A)\cap H^{1,0}(A)$ and
$H^{0,1}_\sigma(A):=H^1_\sigma(A)\cap H^{0,1}(A)$. Assume that for every $\sigma\in\Sigma$, we have 
\begin{equation}
\label{eq-condition-for-HW-to-consist-of-Hodge-classes}
\dim(H^{1,0}_\sigma(A))=\dim(H^{0,1}_\sigma(A))=\frac{d}{2}.
\end{equation}
Then $\wedge_K^dH^1(A,\QQ)$ is an $e$-dimensional subspace $HW$ of $H^{\frac{d}{2},\frac{d}{2}}(A,\QQ)$, which is a one-dimensional $K$ vector space
\cite[Prop. 4.4]{deligne-milne}.
The challenge is then to prove the algebraicity of the classes in $HW$. Classes in $HW$ are called {\em Weil classes} in the literature (see \cite{Moonen-Zarhin-Weil-classes}).
%$\wedge_K^dH^1(A,\QQ)$. 
Again, it suffices to prove the algebraicity of one non-zero class in $HW$,
%$\wedge_K^dH^1(A,\QQ)$, 
as $K$ acts via algebraic correspondences. If $d=2$, then $HW$ consists of rational $(1,1)$ classes, which are algebraic.
Hence, we may assume that $d>2$.

Let $X$ be an abelian variety of dimension 
\[
n:=de/4
\] 
and set $A=X\times\hat{X}$. 
Let $\hat{\Sigma}:=\Sigma/\iota$ be the set of pairs of complex conjugate embeddings of $K$. 
We will construct the embedding $\eta:K\rightarrow \End_\QQ(A)$, 
%depending on a choice of a CM-type $T$,  
starting with a given real multiplication $\hat{\eta}:F\rightarrow \End_\QQ(X)$.

Let $W$  be the image in 
\[
V_{\hat{\eta},K}:=[H^1(X,\QQ)\oplus H^1(X,\QQ)^*]\otimes_FK
\]
of $H^1(\hat{X},\QQ)\otimes_FK$ under an element 
\begin{equation}
\label{eq-g}
g_0
\end{equation} 
of a $K$-linear spin group, defined below in (\ref{eq-Spin-V-K}), so that $W$ is maximal isotropic.
% and consider its orbit under $Gal(K/\QQ)$. As the subspace $H^1(X,\QQ)$ is $Gal(F/\QQ)$ invariant, the orbit will consist of two subspace $W_T$ and 
Let $\bar{W}$ be the image of $W$ via the non-trivial element $\iota$ in $Gal(K/F)$. 
Then $\bar{W}$ is the complex conjugate of $W$, under every embedding of $K$ in $\CC$. Assume that $W\cap\bar{W}=(0)$. Then $W$ is naturally isomorphic to 
$H^1(A,\QQ)$, as a vector space over $\QQ$, and so determines an action of $K$ on $H^1(A,\QQ)$ (section \ref{sec-complex-multiplication--by-CM-field-K}).

A {\em complex multiplication type} (CM-type) consists of a choice of an element $\sigma$, for each  $\hat{\sigma}\in\hat{\Sigma}$ (see \cite[Sec. 5]{deligne-milne}). 
For each CM-type $T$ we get an even maximal isotropic subspace $W_T$ of $H^1(A,\CC)$ and so a pure spinor $\ell_T\in \PP(S^+_\CC)$.
The set $\{\ell_T\}$, as $T$ varies over all CM-types, spans a rational subspace $B$ of $S^+_\QQ$ of dimension $2^{e/2}$ (Lemma \ref{lemma-B-is-rational}).
A {\em $B$-secant sheaf} is a sheaf $F$ on $X$ with $ch(F)$ in $B$. Given two secant sheaves $F_1$ and $F_2$ we get the object $E:=\Phi(F_1\boxtimes F_2^\vee)$ in $D^b(X\times\hat{X})$ via Orlov's derived equivalence $\Phi$. Assume that $\rank(E)\neq 0$ \dots(???)
%\begin{rem}
%\label{rem-notation-T}
%A word about the notation. The subspace $W_T$ of $H^1(A,\QQ)\otimes_FK$ is independent of the CM-type $T$, but $T$ determines an embedding 
%(\ref{eq-embedding-of-V-otimes-FK-in-V-CC})
%of 
%$H^1(A,\QQ)\otimes_FK$ in $H^1(A,\CC)$, and so the subspace $W_{T,\CC}$ of $H^1(A,\CC)$. 
%The subspace $W_T$ will facilitate the construction of the complex multiplication $\eta$, which is independent of $T$.
%Once constructed, we will get a maximal isotropic subspace $W_{T'}$ of $H^1(A,\CC)$, for every CM-type $T'\in \T_K$, 
%so that $W_{T,\CC}$ is the one associated to $T$. 
%\end{rem}

%****************************************************************
% 
%****************************************************************
\subsection{Linear algebra over $F$}
%If $\sigma:K\rightarrow \CC$ is an embedding and we denote by $\hat{\sigma}:F\rightarrow \RR$ its restriction, then 
We have the decomposition 
\[
H^1(X,\RR)=H^1(X,\QQ)\otimes_\QQ\RR\cong \oplus_{\hat{\sigma}\in\hat{\Sigma}} H^1_{\hat{\sigma}}(X),
\]
since
$
F\otimes_\QQ\RR \cong \prod_{\hat{\sigma}\in\hat{\Sigma}}\RR.
$
%Let 
%\[
%\hat{\eta}:F\rightarrow \End(H^1(X,\QQ))
%\] 
%be the embedding. 
The dimensions $\dim_\RR H^1_{\hat{\sigma}}(X)$ are the same for all $\hat{\sigma}\in \hat{\Sigma}$ and are equal to $d=\dim_FH^1(X,\QQ)$.
We have the equality
\[
\dim(H^{1,0}_{\hat{\sigma}}(X))=\dim(H^{0,1}_{\hat{\sigma}}(X)),
\]
for all $\hat{\sigma}\in\hat{\Sigma}$, since the $\hat{\eta}(F)$-action commutes with the complex structure of $X$, by assumption, and with complex conjugation, as $F$ is totally real.

%****************************************************************
% 
%****************************************************************
\subsubsection{Trace forms}
We have the isomorphism of vector spaces over $\QQ$
\[
\Hom_F(H^1(X,\QQ),F) \rightarrow \Hom_\QQ(H^1(X,\QQ),\QQ)
\]
sending $h\in \Hom_F(H^1(X,\QQ),F)$ to $tr_{F/\QQ}\circ h$.
We denote by 
\[
\hat{\eta}:F\rightarrow \End(V_\QQ)
\]
the scalar multiplication action of $F$ on $H^1(X,\QQ)\oplus \Hom_F(H^1(X,\QQ)_{\hat{\eta}},F)$.
Explicitly, $\hat{\eta}(a)$ acts on $\Hom_F(H^1(X,\QQ)_{\hat{\eta}},F)$ by 
$\hat{\eta}_a(\theta)=\theta\circ a$.
We defined $V_{\hat{\eta}}$ as $V_\QQ$ with its $F$-vector space structure. 
Define the $\hat{\eta}(F)$-bilinear pairing 
\[
(\bullet,\bullet)_{V_{\hat{\eta}}} : V_\QQ\otimes_{\hat{\eta}(F)} V_\QQ\rightarrow F,
\]
by $((w_1,\theta_1),(w_2,\theta_2))_{V_{\hat{\eta}}}=\theta_1(w_2)+\theta_2(w_1)$, where $\theta_i$ is regarded as an element of $\Hom_F(V_{\hat{\eta}},F)$.
The $\QQ$-valued pairing 
$(\bullet,\bullet)_{V_\QQ}$ 
is, by definition, 
the composition $tr_{F/\QQ}\circ (\bullet,\bullet)_{V_{\hat{\eta}}}.$ The $\hat{\eta}(F)$-bilinearity of $(\bullet,\bullet)_{V_{\hat{\eta}}}$ implies that $\hat{\eta}(a)$ is 
self-dual with respect to the pairing
$(\bullet,\bullet)_{V_\QQ}$.

Consider the $\RR$-bilinear $\RR$-valued pairing on $F\otimes_\QQ\RR$ given by
\[
(f_1\otimes r_1,f_2\otimes r_2)=r_1r_2tr_{F/\QQ}(f_1f_2),
\]
for $f_i\in F$ and $r_i\in \RR$, $i=1,2$.
The decomposition $F\otimes_\QQ\RR=\oplus_{\hat{\sigma}\in\hat{\Sigma}}\hat{\sigma}$ consists of pairwise orthogonal lines with respect to the above pairing. Set 
\begin{eqnarray*}
V_{\hat{\sigma},\RR}&:=&H^1_{\hat{\sigma}}(X)\oplus H^1_{\hat{\sigma}}(\hat{X}).
%\\
%V_{\hat{\sigma}}&:=&[H^1_{\hat{\sigma}}(X)\oplus H^1_{\hat{\sigma}}(\hat{X})]\cap H^1(X\times\hat{X},\hat{\sigma}(F)).
\end{eqnarray*}
%The composition $V_\QQ\rightarrow V_\RR\rightarrow V_{\hat{\sigma},\RR}$ maps $V_\QQ$ isomorphically onto $V_{\hat{\sigma}}$. 

\begin{lem}
\label{lemma-orthogonal-direct-sum-decomposition-V-hat-sigma}
The decomposition $V\otimes_\ZZ\RR=\oplus_{\hat{\sigma}\in\hat{\Sigma}} V_{\hat{\sigma},\RR}$ consists of pairwise orthogonal subspaces with respect to $(\bullet,\bullet)_{V_\RR}.$
\end{lem}

\begin{proof}
This follows from the self-duality of
 $\hat{\eta}(a)$, for $a\in F$. Explicitly, the pairing $(\bullet,\bullet)_{V_\RR}:V_\RR\otimes V_\RR\rightarrow \RR$  factors as an $\hat{\eta}(F)$-bilinear pairing 
$(\bullet,\bullet)_{V_{\hat{\eta}}}$ with values in $F\otimes_\QQ\RR$ composed with $tr_{F/\QQ}\otimes id_\RR:F\otimes_\QQ\RR\rightarrow \RR$. Now, given $x\in V_{\hat{\sigma}_1,\RR}$, $y\in V_{\hat{\sigma}_2,\RR}$, and $q\in F$ a primitive element,  we have
\begin{eqnarray*}
q(x,y)_{V_{\hat{\eta}}}&=&(\hat{\eta}_q x,y)_{V_{\hat{\eta}}}=(\hat{\sigma}_1(q)x,y)_{V_{\hat{\eta}}}=
\hat{\sigma}_1(q)(x,y)_{V_{\hat{\eta}}}
\\
q(x,y)_{V_{\hat{\eta}}}&=&(x,\hat{\eta}_q y)_{V_{\hat{\eta}}}= (x,\hat{\sigma}_2(q)y)_{V_{\hat{\eta}}}=\hat{\sigma}_2(q)(x,y)_{V_{\hat{\eta}}},
\end{eqnarray*}
where multiplication by $q$ on the left above corresponds to multiplication in the left tensor factor of $F\otimes_\QQ\RR$ and multiplication by 
$\hat{\sigma}_i(q)$ on the right above corresponds to multiplication in the right tensor factor of $F\otimes_\QQ\RR$. Hence, if $\hat{\sigma}_1\neq\hat{\sigma}_2$, then
$(x,y)_{V_{\hat{\eta}}}=0.$ Consequently, its trace $(x,y)_{V_\RR}$ vanishes as well.
\end{proof}

The $j$-th cartesian power $(F^\times)^j$ of the multiplicative group $F^\times$ acts on the $j$-th tensor power over $\QQ$ of $H^1(X,\QQ)$ and hence
on $\wedge^j_\QQ H^1(X,\QQ)$ and on $\wedge^j_\QQ H^1(X,\QQ)\otimes_\QQ\RR\cong\wedge^j_\RR H^1(X,\RR)$.
The latter decomposes as a direct sum of characters subspaces 
\[
\wedge^j_\RR H^1(X,\RR)=
\oplus_{(\hat{\sigma}_1, \dots, \hat{\sigma}_j)\in (\hat{\Sigma})^j}\left(\wedge^j_\RR H^1(X,\RR)\right)_{(\hat{\sigma}_1, \dots, \hat{\sigma}_j)}
\]
The vector space $\wedge^j_\QQ H^1(X,\QQ)$ decomposes into subrepresentations of $(F^\times)^j$ defined over $\QQ$. The subrepresentation $\wedge^j_FH^1(X,\QQ)$ of $\wedge^j_\QQ H^1(X,\QQ)$  corresponds to the subrepresentation
\[
\oplus_{\hat{\sigma}\in\hat{\Sigma}}\wedge^j_\RR H^1_{\hat{\sigma}}(X)= 
\oplus_{\hat{\sigma}\in\hat{\Sigma}}\left(\wedge^j_\RR H^1(X,\RR)\right)_{(\hat{\sigma}, \dots, \hat{\sigma})}.
\]
of $\wedge^j_\RR H^1(X,\RR)$. Given $\alpha_i\in \wedge^{j_i}_F H^1(X,\QQ)$, $i=1,2$, then $\alpha_1\wedge_F\alpha_2$ 
is the projection of $\alpha_1\wedge_\QQ\alpha_2$ to the direct summand $\wedge_F^{j_1+j_2}H^1(X,\QQ)$.

%*************************************************************
%
%*************************************************************
\subsection{Spin groups over $F$ and real embeddings of $F$}
We define 
\[
\Spin(V_{\hat{\eta}})
\] 
as the Spin group associated to the $F$-valued $F$-bilinear pairing $(\bullet,\bullet)_{V_{\hat{\eta}}}$.
%We define 
%\[
%\Spin(V)_{\hat{\eta}}
%\] 
%as the subgroup of $\Spin(V_{\hat{\eta}})$ preserving the lattice $V$. 

%*************************************************************
%
%*************************************************************
\subsubsection{A homomorphism from $\Spin(V_{\hat{\eta}})$ to $\Spin(V_\RR)$}
\label{sec-homomorphism-from-Spein-V-hat-eta-to-Spin-V-RR}

%\begin{rem}
%\label{rem-factorization-of-Clifford-algebra}
%\begin{enumerate}
%\item 
The tensor product $V_{\hat{\eta}}\otimes_\QQ\RR$, with the induced bilinear pairing, decomposes as an orthogonal direct sum $\oplus_{\hat{\sigma}\in \hat{\Sigma}}V_{\hat{\sigma},\RR}$, by Lemma \ref{lemma-orthogonal-direct-sum-decomposition-V-hat-sigma}. Hence, $C(V_\QQ)\otimes_\QQ\RR$ is isomorphic to the tensor product 
$\otimes_{\hat{\sigma}\in\hat{\Sigma}}C(V_{\hat{\sigma},\RR})$, where $V_{\hat{\sigma},\RR}$ is endowed with the restriction of the bilinear pairing of $V_\RR$ and
 the tensor product is over $\RR$ in the category of $\ZZ_2$-graded algebras \cite[Lemma V.1.7]{lam}.
 In the current section \ref{sec-homomorphism-from-Spein-V-hat-eta-to-Spin-V-RR} we will denote $V_{\hat{\sigma},\RR}$ by $V_{\hat{\sigma}}$.
The composition
\[
C(V_\QQ)\rightarrow C(V_\RR)\IsomRightArrow \otimes_{\hat{\sigma}\in\hat{\Sigma}}C(V_{\hat{\sigma}})
\]
is a $\QQ$-algebra homomorphism. 
%The field $F$ acts on the $\hat{\sigma}$-factor via the embedding $\hat{\sigma}:F\rightarrow \RR$.

%\item
%\label{rem-item-C-V-hat-eta}
Consider the Clifford algebra 
\[
C(V_{\hat{\eta}}):=\oplus_{j\geq 0} V_{\hat{\eta}}^{\otimes_F j}/\langle u\otimes_F v+ v\otimes_F u - (u,v)_{V_{\hat{\eta}}}
\rangle,
\]
where the scalar $(u,v)_{V_{\hat{\eta}}}$ is in $F$. 
The embedding
\[
(\hat{\sigma})_{\hat{\sigma}\in\hat{\Sigma}}:
V_{\hat{\eta}}\rightarrow V_\RR=\oplus_{\hat{\sigma}\in\hat{\Sigma}}V_{\hat{\sigma}}
\]
is isometric, if the pairing on the right hand side is given by $(u,v)_{V_{\hat{\eta}}}=((u_{\hat{\sigma}},v_{\hat{\sigma}})_{V_{\hat{\sigma}}})_{\hat{\sigma}\in\Hat{\Sigma}}$
as an element in $F\otimes_\QQ\RR\cong\oplus_{\hat{\sigma}\in\Hat{\Sigma}}\RR$. 
For each $\hat{\sigma}\in\hat{\Sigma}$ we have an embedding $\hat{\sigma}:C(V_{\hat{\eta}})\rightarrow C(V_\RR)\IsomRightArrow \otimes_{\hat{\sigma}\in\hat{\Sigma}}C(V_{\hat{\sigma}})$, sending $g\in C(V_{\hat{\eta}})$ to $g$ tensored with $1_{\hat{\sigma}_1}\in C(V_{\hat{\sigma}_1})$, for all $\hat{\sigma}_1\neq\hat{\sigma}$. The diagonal embedding of $V_\QQ$ in $V_\RR$ sends $v\in V_\QQ\subset C(V_\QQ)$ to $\sum_{\hat{\sigma}\in\hat{\Sigma}}\hat{\sigma}(v)\in \otimes_{\hat{\sigma}\in\hat{\Sigma}}C(V_{\hat{\sigma}})$ (see the proof of \cite[Lemma V.1.7]{lam}). It does not extends to an algebra homomorphism
$
C(V_{\hat{\eta}})\rightarrow C(V_\RR).
$
The vector space $V_\RR$ is also a module over $F\otimes_\QQ\RR\cong\prod_{\hat{\sigma}\in\hat{\Sigma}}\RR$ and we can form the Clifford algebra of $V_\RR$ over the latter ring, in which $C(V_{\hat{\eta}})$ embeds. Nevertheless, we do get a diagonal multiplicative group homomorphism from the group 
$C(V_{\hat{\eta}})^{even,\times}$, of invertible even\footnote{Note that even tensor factors  commute in the tensor product $\otimes_{\hat{\sigma}\in\hat{\Sigma}}C(V_{\hat{\sigma}})$ in the category of $\ZZ_2$-graded algebras.} elements in $C(V_{\hat{\eta}})$, into $C(V_\RR)^{even,\times}.$
\begin{equation}
\label{eq-diagram-of-spin-groups}
\xymatrix{
C(V_{\hat{\eta}})^{even,\times} \ar[r] &
\prod_{\hat{\sigma}\in\hat{\Sigma}}C(V_{\hat{\sigma}})^{even,\times}
\ar[r] &
\left[\otimes_{\hat{\sigma}\in\hat{\Sigma}}C(V_{\hat{\sigma}})\right]^{even,\times} \ar[r]^-\cong
&
C(V_\RR)^{even,\times}
\\
\Spin(V_{\hat{\eta}}) \ar[r] \ar[u]^{\cup} &
\prod_{\hat{\sigma}\in\hat{\Sigma}}\Spin(V_{\hat{\sigma}})
\ar[rr] \ar[u]^{\cup} & &
\Spin(V_\RR). \ar[u]^{\cup}
}
\end{equation}
The composition 
\begin{equation}
\label{eq-homomorphism-of-even-miltiplicative-Clifford-groups}
C(V_{\hat{\eta}})^{even,\times}\rightarrow C(V_\RR)^{even,\times},
\end{equation}
of the top horizontal homomorphisms above, restricts to 
$Nm_{F/\QQ}:F^\times\rightarrow \QQ^\times\subset\RR^\times$. In particular, it is not injective.
%\item

%The subspace $\wedge^*H^1_{\hat{\sigma}}(X)$ of $S_\RR:=\wedge^*H^1(X,\RR)$ is $C(V_{\hat{\eta}})^{even,\times}$-invariant, 
%as the tensor factorization of $C(V_\RR)$ is compatible with the factorization 
%$\wedge^*H^1(X,\RR)\cong \otimes_{\hat{\sigma}\in\hat{\Sigma}}\wedge^*H^1_{\hat{\sigma}}(X)$, 
%where again the tensor product is taken in the category of $\ZZ_2$-graded algebras.
%\end{enumerate}
%\end{rem}

%***********
% Hide
%***********
\hide{
The direct sum
$\oplus_{\hat{\sigma}\in\hat{\Sigma}}V_{\hat{\sigma}}$, considered as a subspace of $V_\QQ\otimes_\QQ(F\otimes_\QQ\RR)$, is equal to the image of
$V_\QQ\otimes_\QQ F$ via
\begin{equation}
\label{eq-embedding-of-V-tensor-F}
V_\QQ\otimes_\QQ F\LongRightArrowOf{(id_{V_\QQ}\otimes \hat{\sigma})_{\hat{\sigma}\in\hat{\Sigma}}} V_\QQ\otimes_\QQ \left(\prod_{\hat{\sigma}\in\hat{\Sigma}}\RR\right)
\cong V_\QQ\otimes_\QQ (F\otimes_\QQ \RR).
%\rightarrow V_\QQ\otimes_F (F\otimes_\QQ \RR).
\end{equation}
%***********
% End  Hide
%***********
}

%*************************************************************
%
%*************************************************************
\subsubsection{A homomorphism from $\Spin(V_{\hat{\eta}})$ to $\Spin(V_\QQ)$}
Let $\{e_{\hat{\sigma},1}, \dots, e_{\hat{\sigma},d}\}$ be an orthogonal basis of $V_{\hat{\sigma},\RR}$ satisfying $(e_{\hat{\sigma},i},e_{\hat{\sigma},i})\in\{2,-2\}$. Then $m_{e_{\hat{\sigma},i}}\in C(V_\RR)$ belongs to the Clifford group and acts on $V_\RR$ as minus the reflection in the hyperplane orthogonal to $e_{\hat{\sigma},i}$. Note that $m_{e_{\hat{\sigma},i}}m_{e_{\hat{\sigma},j}}=-m_{e_{\hat{\sigma},i}}m_{e_{\hat{\sigma},j}}$, if $i\neq j$, but their images in $O(V_\RR)$ commute.
Set 
\begin{equation}
\label{eq-g-hat-sigma}
g_{\hat{\sigma}}:=\prod_{i=1}^dm_{e_{\hat{\sigma},i}}.
\end{equation} 
Choose an embedding $\hat{\sigma}_0:F\rightarrow \RR$ and let $\tilde{F}\subset \RR$ be the Galois closure of $\hat{\sigma}_0(F)$ over $\QQ$. Diagram (\ref{eq-diagram-of-spin-groups}) holds, with $V_\RR$ replaced by $V_{\tilde{F}}$.
The element $g_{\hat{\sigma}}$ belongs to the Clifford subgroup of $C(V_{\tilde{F}})^{even,\times}$, which belongs to $\Spin(V_{\tilde{F}})$ if and only if $d/2$ is even, since the signature of $V_{\hat{\sigma}}$ is $(d/2,d/2)$. Note that $g_{\hat{\sigma}}^2=(-1)^{d/2}$.
The element 
$g_{\hat{\sigma}}$ maps to $SO(V_{\tilde{F}})$ acting on $V_{\hat{\sigma}}$ via multiplication by  $-1$ and it acts as the identity on $V_{\hat{\sigma}'}$, 
for $\hat{\sigma}'\neq \hat{\sigma}$. 

Let $\Spin(V_\QQ)_{\{g_{\hat{\sigma}}\}}$ be the subgroup 
of $\Spin(V_\QQ)$ consisting of elements, which commute in $C(V_{\tilde{F}})^{even,\times}$ with $g_{\hat{\sigma}}$, for all $\hat{\sigma}\in\hat{\Sigma}$.
Let
%\footnote{A better notation would be $\Spin(V_\QQ)_{\hat{\eta}}$, but only after we change the notation in section \ref{sec-isotypic-decomposition}.} 
\begin{equation}
\label{eq-Spin-V-QQ-hat-eta}
\Spin(V_\QQ)_{\hat{\eta}}
\end{equation}
be the commutator subgroup of the commutator subgroup of $\Spin(V_\QQ)_{\{g_{\hat{\sigma}}\}}$ (the second derived subgroup).
%the subgroup of $\Spin(V_\QQ)$ consisting of elements $g$, which commute in $C(V_{\tilde{F}})^{even,\times}$ with $g_{\hat{\sigma}}$, 
%for all $\hat{\sigma}\in\hat{\Sigma}$, and such that $\rho_g$ restricts to the subspace $V_{\hat{\sigma}}$ of $V_{\tilde{F}}$ 
%as an element of $SO_+(V_{\hat{\sigma}})$.
The group $\Spin(V_\QQ)_{\hat{\eta}}$ is independent of the choice of the bases $\{e_{\hat{\sigma},i}\}$,
since $g_{\hat{\sigma}}$ is determined, up to sign,  by its image $\rho_{g_{\hat{\sigma}}}$ in $SO(V_{\tilde{F}})$ and $\rho_{g_{\hat{\sigma}}}$ is independent of the choice of the bases. 

Let $SO(V_{\hat{\eta}})$ be the subgroup of $SO(V_\QQ)$ preserving the $F$-valued pairing $(\bullet,\bullet)_{V_{\hat{\eta}}}.$ 
%Elements of $SO(V_{\hat{\eta}})$ are precisely the elements of $SO(V_\QQ)$ leaving invariant $V_{\hat{\sigma}}$, for all $\hat{\sigma}\in\hat{\Sigma}$. 
Let  $SO_+(V_{\hat{\eta}})$  be the image of $\Spin(V_{\hat{\eta}})$ in $SO(V_{\hat{\eta}})$.
Let $SO_+(V_{\hat{\sigma}})$ be the image of $\Spin(V_{\hat{\sigma}})$ in $SO(V_{\hat{\sigma}})$. We have the short exact sequences
\begin{eqnarray*}
0\rightarrow SO_+(V_\QQ)\rightarrow &SO(V_\QQ)&\rightarrow \QQ^\times/\QQ^{\times,2}\rightarrow 0,
\\
0\rightarrow SO_+(V_{\hat{\eta}})\rightarrow &SO(V_{\hat{\eta}})&\rightarrow F^\times/F^{\times,2}\rightarrow 0,
\\
0\rightarrow SO_+(V_{\hat{\sigma}})\rightarrow &SO(V_{\hat{\sigma}})&\rightarrow \tilde{F}^\times/\tilde{F}^{\times,2}\rightarrow 0,
\end{eqnarray*}
where $\QQ^{\times,2}$, $F^{\times,2}$, and $\tilde{F}^{\times,2}$ are the groups of squares of non-zero elements 
(see \cite[II.3.7]{chevalley}). Furthermore, $SO_+(\bullet)$ is the commutator subgroup of $SO(\bullet)$, for $\bullet=V_\QQ$, $V_{\hat{\eta}}$, and $V_{\hat{\sigma}}$ \cite[II.3.9]{chevalley}. 
Hence, for every $g\in \Spin(V_\QQ)_{\hat{\eta}}$, $\rho_g$ restricts to the subspace $V_{\hat{\sigma}}$ of $V_{\tilde{F}}$ 
as an element of $SO_+(V_{\hat{\sigma}})$.

%***********
% Hide
%***********
\hide{
Following is another characterization of $\Spin(V_\QQ)_{\hat{\eta}}$. The Galois group $Gal(\tilde{F}/\QQ)$ acts on the Clifford algebra $C(V_{\tilde{F}})\cong C(V_\QQ)\otimes_\QQ\tilde{F}$ via its action on the second tensor factor. The group $\Spin(V_{\tilde{F}})$ is $Gal(\tilde{F}/\QQ)$-invariant.
The group $\Spin(V_\QQ)_{\hat{\eta}}$ is the subgroup of $\Spin(V_{\tilde{F}})$ consisting of $Gal(\tilde{F}/\QQ)$-invariant elements $g$, such that $\rho_g$
leaves invariant each summand in the decomposition $V_{\tilde{F}}=\oplus_{\hat{\sigma}\in\hat{\Sigma}}V_{\hat{\sigma}}$ and its restriction to each $V_{\hat{\sigma}}$ belongs to $SO_+(V_{\hat{\sigma}})$. 
Indeed, 
let $g$ be an element of $\Spin(V_\QQ)_{\hat{\eta}}$. For each $\hat{\sigma}\in\hat{\Sigma}$ choose an element $h_{\hat{\sigma}}\in\Spin(V_{\hat{\sigma}})$ mapping to the restriction of $\rho_{g}$ to $V_{\hat{\sigma}}$. Set $g':=\otimes_{\hat{\sigma}\in\hat{\Sigma}}h_{\hat{\sigma}}\in C(V_{\tilde{F}})$. Then $g'$ belongs to $\Spin(V_{\tilde{F}})$ and $\rho_g=\rho_{g'}$ in $SO_+(V_{\tilde{F}})$. Thus, $g'=\pm g$ and 
so $g'$ belongs to $\Spin(V_\QQ)_{\hat{\eta}}$.
%$\rho_{g'}$ belongs to $SO_+(V_\QQ)$, since $\rho_g$ belongs to $SO_+(V_\QQ)$. 
It follows that $g$ is a $Gal(\tilde{F}/\QQ)$-invariant element of $\Spin(V_{\tilde{F}})$ such that $\rho_g$ preserves the decomposition $V_{\tilde{F}}=\oplus_{\hat{\sigma}\in\hat{\Sigma}}V_{\hat{\sigma}}$. Conversely, let $g$ be a $Gal(\tilde{F}/\QQ)$-invariant element of $\Spin(V_{\tilde{F}})$, which restriction to each $V_{\hat{\sigma}}$ belongs to $SO_+(V_{\hat{\sigma}})$. Choose an element $h_{\hat{\sigma}}\in\Spin(V_{\hat{\sigma}})$ mapping to the restriction of $\rho_{g}$ to $V_{\hat{\sigma}}$. Set $g':=\otimes_{\hat{\sigma}\in\hat{\Sigma}}h_{\hat{\sigma}}\in C(V_{\tilde{F}})$. Then $g'$ belongs to $\Spin(V_{\tilde{F}})$ and $\rho_g=\rho_{g'}$ in $SO_+(V_{\tilde{F}})$. Thus, $g'=\pm g$. Now $h_{\hat{\sigma}}$ commutes with $g_{\hat{\sigma}}$, for all $\hat{\sigma}\in\hat{\Sigma}$.
Hence, so does $g$. We conclude that $g$ belongs to $\Spin(V_\QQ)_{\hat{\eta}}$. 
%If $\dim_\QQ(F)$ is odd, then $-1$ belongs to the image of $\Spin(V_{\hat{\eta}})$ and so 
%************
% End Hide
%************
}

\begin{lem}
\label{lemma-Spin-V-hat-eta}
%\begin{enumerate}
%\item
%\label{lemma-item-Spin-V-hat-eta-maps-into-Spin-V-QQ}
The  homomorphism (\ref{eq-homomorphism-of-even-miltiplicative-Clifford-groups}) maps $\Spin(V_{\hat{\eta}})$ onto 
$\Spin(V_\QQ)_{\hat{\eta}}$.
%the intersection in $C(V_\RR)^{even,\times}$ of $\Spin(V_\QQ)$. 
%with the image of the subgroup $C(V_{\hat{\eta}})^{even,\times}$ of the group of even invertible elements.  
The resulting homomorphism 
\begin{equation}
\label{eq-homomorphism-from-Spin-V-hat-eta-to-commutator-in-Spin-V-QQ}
\aleph:\Spin(V_{\hat{\eta}})\rightarrow \Spin(V_\QQ)_{\hat{\eta}}
\end{equation}
%$\Spin(V_{\hat{\eta}})\rightarrow \Spin(V_\QQ)$ 
is injective, if $\dim_\QQ(F)$ is odd, and its kernel is $-1$, if $\dim_\QQ(F)$ is even.
The diagram
\[
\xymatrix{
\Spin(V_{\hat{\eta}})\ar[r]^{\aleph} \ar[d]_\rho&
\Spin(V_\QQ)_{\hat{\eta}}\ar[r]^{\subset} & \Spin(V_\QQ) \ar[d]_\rho
\\
SO_+(V_{\hat{\eta}}) \ar[rr]_-\subset & & SO_+(V_\QQ)
}
\]
is commutative.
%\item
%\label{lemma-item-Spin-V-hat-eta-commutes-with-g-hat-sigma}
%The group $\Spin(V_{\hat{\eta}})$ is mapped via (\ref{eq-diagram-of-spin-groups}) into 
%$\Spin(V_\QQ)_{\hat{\eta}}$.
%The resulting homomorphism 
%\begin{equation}
%\label{eq-homomorphism-from-Spin-V-hat-eta-to-commutator-in-Spin-V-QQ}
%\Spin(V_{\hat{\eta}})\rightarrow \Spin(V_\QQ)_{\hat{\eta}}
%\end{equation}
%is an isomorphism, if $\dim_\QQ(F)$ is odd, and its image is an index $2$ subgroup if $\dim_\QQ(F)$ is even.
%\end{enumerate}
\end{lem}

\begin{proof}
%(\ref{lemma-item-Spin-V-hat-eta-maps-into-Spin-V-QQ}) 
The homomorphism  (\ref{eq-homomorphism-of-even-miltiplicative-Clifford-groups})  factors through $C(V_{\tilde{F}})^{even,\times}$.
The group $Gal(\tilde{F})$ acts on $C(V_{\tilde{F}})=C(V_\QQ)\otimes_\QQ\tilde{F}$ via its action on the second tensor factor and the subset $C(V_{\tilde{F}})^{even,\times}$ is $Gal(\tilde{F}/\QQ)$-invariant. The homomorphism (\ref{eq-homomorphism-of-even-miltiplicative-Clifford-groups}) is $Gal(\tilde{F})$-invariant, as the top left homomorphism in the diagram is the diagonal homomorphism. Hence, the image of $C(V_{\hat{\eta}})^{even,\times}$   in $C(V_{\tilde{F}})^{even,\times}$, via the homomorphism (\ref{eq-homomorphism-of-even-miltiplicative-Clifford-groups}),
is contained in $C(V_\QQ)^{even,\times}$. The image of $Spin(V_{\hat{\eta}})$ is thus contained in $\Spin(V_\QQ)$, by the commutativity of the diagram (\ref{eq-diagram-of-spin-groups}).
The last statement about the kernel follows from the equality $Nm_{F/\QQ}(-1)=(-1)^{\dim_\QQ(F)}$. 

%(\ref{lemma-item-Spin-V-hat-eta-commutes-with-g-hat-sigma}) 
The group $\Spin(V_{\hat{\eta}})$ is equal to its commutator subgroup, being a simple group \cite[21.50]{milne-algebraic-groups}. 
The element $g_{\hat{\sigma}}$ belongs to the center of 
$\left[\otimes_{\hat{\sigma}\in\hat{\Sigma}}C(V_{\hat{\sigma}})\right]^{even,\times}\cap\left[\otimes_{\hat{\sigma}\in\hat{\Sigma}}C(V_{\hat{\sigma}})^{even}\right]$, by \cite[II.2.4]{chevalley}.
Hence, the group $\Spin(V_{\hat{\eta}})$ is mapped via (\ref{eq-homomorphism-of-even-miltiplicative-Clifford-groups}) into 
$\Spin(V_\QQ)_{\hat{\eta}}$.
%************
% Hide
%************
\hide{
Let $SO_+(V_\QQ,(\bullet,\bullet)_{V_{\hat{\eta}}})$ be the subgroup of $SO_+(V_\QQ)$ preserving the pairing $(\bullet,\bullet)_{V_{\hat{\eta}}}.$ Elements 
of $SO_+(V_\QQ,(\bullet,\bullet)_{V_{\hat{\eta}}})$ are precisely the elements of $SO_+(V_\QQ)$ leaving invariant $V_{\hat{\sigma}}$, for all $\hat{\sigma}\in\hat{\Sigma}$. The image of the homomorphism $\Spin(V_{\hat{\eta}})\rightarrow SO_+(V_\QQ)$ is\footnote{The norm homomorphism from
the Clifford group of $(\bullet,\bullet)_{V_{\hat{\eta}}}$
has values in $F^\times/F^{\times,2}$, where  $F^{\times,2}$ is the group of squares of invertible elements, 
$\Spin(V_{\hat{\eta}})$ is the intersection of the kernel of the norm homomorphism with the even Clifford group,
and  $SO_+(V_\QQ,(\bullet,\bullet)_{V_{\hat{\eta}}})$  is the image of  $\Spin(V_{\hat{\eta}})\rightarrow SO(V_\QQ,(\bullet,\bullet)_{V_{\hat{\eta}}})$.
} 
$SO_+(V_\QQ,(\bullet,\bullet)_{V_{\hat{\eta}}})$, by definition (??? the group $SO_+(V_\QQ,(\bullet,\bullet)_{V_{\hat{\eta}}})$ is defined above differently ???)). Note that $SO(V_\QQ,(\bullet,\bullet)_{V_{\hat{\eta}}})/SO_+(V_\QQ,(\bullet,\bullet)_{V_{\hat{\eta}}})$
is isomorphic to $F^\times/F^{\times,2}$, 
by \cite[II.3.7]{chevalley}, and its kernel is generated by $-1$. The homomorphism $\Spin(V_\QQ)_{\hat{\eta}}\rightarrow SO_+(V_\QQ)$ has the same image $SO_+(V_\QQ,(\bullet,\bullet)_{V_{\hat{\eta}}})$ and its kernel is generated by $-1$. The homomorphism $\Spin(V_{\hat{\eta}})\rightarrow SO_+(V_\QQ,(\bullet,\bullet)_{V_{\hat{\eta}}})$
factors through $\Spin(V_\QQ)_{\hat{\eta}}$. Hence, the homomorphism (\ref{eq-homomorphism-from-Spin-V-hat-eta-to-commutator-in-Spin-V-QQ}) is an isomorphism, if $\dim_\QQ(F)$ is odd and its image is an index $2$ subgroup if $\dim_\QQ(F)$ is even.
%************
% End Hide
%************
}

It remains to prove the surjectivity of (\ref{eq-homomorphism-from-Spin-V-hat-eta-to-commutator-in-Spin-V-QQ}).
The pairing $(\bullet,\bullet)_{\hat{\eta}}$ is the unique $F$-valued symmetric $F$-bilinear pairing on $V_{\hat{\eta}}$, such that $tr_{F/\QQ}\circ (\bullet,\bullet)_{\hat{\eta}}=(\bullet,\bullet)_V$. Indeed, $f\mapsto tr_{F/\QQ}\circ f :\Hom_F(V_{\hat{\eta}},F)\rightarrow \Hom_\QQ(V_\QQ,\QQ)$ is an isomorphism of $\QQ$-vector spaces and the pairing $(\bullet,\bullet)_{\hat{\eta}}$ corresponds to an element of the subspace 
$\Hom_F(V_{\hat{\eta}},\Hom_F(V_{\hat{\eta}},F))$ of $\Hom_\QQ(V_{\hat{\eta}},\Hom_F(V_{\hat{\eta}},F))$. Hence, if $g$ is an isometry of $(V_\QQ,(\bullet,\bullet)_V)$, which commutes with $\hat{\eta}(F)$, then $g$ is an $F$-linear isometry of $(V_{\hat{\eta}},(\bullet,\bullet)_{V_{\hat{\eta}}})$. Indeed, in that case $tr_{F/\QQ}(g(x),g(y))_{V_{\hat{\eta}}}=(g(x),g(y))_V=(x,y)_V$, and 
$(g(\bullet),g(\bullet))_{V_{\hat{\eta}}}$ is an $F$-valued symmetric $F$-bilinear pairing on $V_{\hat{\eta}}$.
Hence, for $g\in \Spin(V_\QQ)_{\{g_{\hat{\sigma}}\}}$, the isometry $\rho_g$ of $(V_\QQ,(\bullet,\bullet)_V)$ is also an isometry of $(\bullet,\bullet)_{V_{\hat{\eta}}}$. Thus, if $g$ belongs to $\Spin(V)_{\hat{\eta}}$, then $\rho_g$ belongs to the commutator subgroup $SO_+(V_{\hat{\eta}})$ of $O(V_{\hat{\eta}})$,
where $O(V_{\hat{\eta}})$ is the subgroup of $O(V_\QQ)$ of $F$-linear automorphisms preserving the $F$-valued pairing $(\bullet,\bullet)_{V_{\hat{\eta}}}.$ 
Consequently, $\rho_g=\rho_{\tilde{g}}$, for some  $\tilde{g}\in\Spin(V_{\hat{\eta}})$, by \cite[II.3.8]{chevalley}. So $g=\pm\aleph(\tilde{g})$. 
If $\dim_\QQ(F)$ is odd, then $\aleph(-1)=-1$, and so $g$ belongs to the image of $\aleph$. 
If  $\dim_\QQ(F)$ is even, then the above argument shows that the image of $\aleph$ has index at most $2$ in the first derived subgroup of $\Spin(V_\QQ)_{\{g_{\hat{\sigma}}\}}$. If the index is $2$, then the image of $\aleph$ is contained in the second derived subgroup and is a normal subgroup of the first derived subgroup with an abelian quotient, and so the image of $\aleph$ is equal to the second derived subgroup, which is $\Spin(V)_{\hat{\eta}}$.
\end{proof}

\hide{
In other words, $\Spin(V)_{\hat{\eta}}$ is isomorphic to the subgroup of 
$\Spin(H^1_{\hat{\sigma}}(X)\oplus H^1_{\hat{\sigma}}(\hat{X}))$ preserving the lattice $V$, where $V_{\hat{\sigma}}:=H^1_{\hat{\sigma}}(X)\oplus H^1_{\hat{\sigma}}(\hat{X})$ is endowed with the obvious $\RR$-bilinear symmetric pairing. The composition 
\[
V_\QQ\rightarrow V_\QQ\otimes_\QQ F \rightarrow V_{\hat{\sigma}},
\]
of the inclusion with the projection onto the eigenspace,
should be (??? define the righthand side as a $\QQ$-vector space ???) a $\QQ$-linear isomorphism of  vector spaces 
such that the bilinear pairing on $V_\QQ$ is the composition of the $F$-bilinear (???) pairing on the right hand side with $tr_{F/\QQ}:F\rightarrow \QQ$. 
%********
% End Hide
%********
}

%****************************************************************
% 
%****************************************************************
\subsubsection{Tensor product factorizations of pure spinors}
Set $S_{\hat{\sigma}}:=\wedge^*H^1_{\hat{\sigma}}(X)$. 
Note the isomorphisms $S_\RR:=\wedge^*H^1(X,\RR)\cong\otimes_{\hat{\sigma}\in \hat{\Sigma}}S_{\hat{\sigma}}$,
\[
S_\RR\otimes S_\RR\cong \otimes_{\hat{\sigma}\in \hat{\Sigma}}(S_{\hat{\sigma}}\otimes_\RR S_{\hat{\sigma}}),
\]
and $C(V_\RR)\cong \otimes_{\hat{\sigma}\in\hat{\Sigma}}C(V_{\hat{\sigma}})$. 
The isomorphism $\varphi:S_\RR\otimes S_\RR\rightarrow C(V_\RR)$, given in (\ref{eq-Chevalley-varphi}),  factors as the tensor product of 
\[
\varphi_{\hat{\sigma}}:S_{\hat{\sigma}}\otimes S_{\hat{\sigma}}\rightarrow %\otimes_{\hat{\sigma}\in\hat{\Sigma}}
C(V_{\hat{\sigma}}).
\] 

Given maximal isotropic subspaces $W_{\hat{\sigma}}$ of $V_{\hat{\sigma}}$, for all $\hat{\sigma}\in \hat{\Sigma}$, set $W:\oplus_{\hat{\sigma}\in \hat{\Sigma}}W_{\hat{\sigma}}$. Let $\ell_{\hat{\sigma}}\subset S^+_{\hat{\sigma}}$ be the line spanned by a pure spinor of $W_{\hat{\sigma}}$ and by $\ell\subset S^+_\RR$ 
the line spanned by a pure spinor of $W$. Then 
\begin{equation}
\label{eq-tensor-product-decomposition-of-pure-spinor}
\ell=\bigotimes_{\hat{\sigma}\in \hat{\Sigma}}\ell_{\hat{\sigma}}.
\end{equation}
Given another set of maximal isotropic subspaces $W'_{\hat{\sigma}}$ of $V_{\hat{\sigma}}$, for all $\hat{\sigma}\in \hat{\Sigma}$, with pure spinors $\ell'_{\hat{\sigma}}$, we have
$\varphi(\ell\otimes\ell')=\otimes_{\hat{\sigma}\in \hat{\Sigma}}\varphi_{\hat{\sigma}}(\ell_{\hat{\sigma}}\otimes\ell'_{\hat{\sigma}}).$

Similarly, the isomorphism $\psi:C(V_\RR)\rightarrow \wedge^*V_\RR$, given in (\ref{eq-psi}), decomposes as a tensor product of
$\psi_{\hat{\sigma}}:C(V_{\hat{\sigma}})\rightarrow \wedge^*V_{\hat{\sigma}}$ and so $\tilde{\varphi}:=\psi\circ\varphi$ decomposes as the tensor product of 
\begin{equation}
\label{eq-tilde-varphi-hat-sigma}
\tilde{\varphi}_{\hat{\sigma}}:=\psi_{\hat{\sigma}}\circ \varphi_{\hat{\sigma}}:S_{\hat{\sigma}}\otimes S_{\hat{\sigma}}\rightarrow \wedge^*V_{\hat{\sigma}}.
\end{equation}

%**********
% Hide
%**********
\hide{
maps $S_{\hat{\sigma}_0}\otimes S_{\hat{\sigma}_0}$ to 
$C(V_{\hat{\sigma}_0})\otimes \otimes_{\{\hat{\sigma} \ : \ \hat{\sigma}\neq\hat{\sigma}_0\}}\wedge^{top}H^1_{\hat{\sigma}}(\hat{X})$,
where we regard $\wedge^{top}H^1_{\hat{\sigma}}(\hat{X})$ as a line in $C(V_{\hat{\sigma}})$.
Given a pure spinor $\ell$ in $S^+_{\hat{\sigma}_0}$ of a maximal isotropic subspace $W$ in $V_{\hat{\sigma}_0}$, the isomorphism $\varphi$ maps 
$\ell\otimes\ell$ to 
the tensor product of $\wedge^{top}W$ (??? $\varphi_{\hat{\sigma}_0}(\ell\otimes\ell)$ ???) with $\otimes_{\{\hat{\sigma} \ : \ \hat{\sigma}\neq\hat{\sigma}_0\}}\wedge^{top}H^1_{\hat{\sigma}}(\hat{X})$, which projects, in the appropriate graded summand, onto
the top wedge product of the maximal isotropic subspace $W\oplus \oplus_{\{\hat{\sigma} \ : \ \hat{\sigma}\neq\hat{\sigma}_0\}} H^1_{\hat{\sigma}}(\hat{X})$.
The pure spinor $\ell$, as a line on the left hand side, is mapped to its tensor product with the pure spinors $1_{\hat{\sigma}}$, $\hat{\sigma}\neq\hat{\sigma}_0$, on the right hand side.
%**********
% End Hide
%**********
}

\begin{rem}
%\label{rem-image-via-varphi-of-top-wedge-over-F-of-a-maximal-isotropic-F-subspace}
Let $W$ be a maximal isotropic $F$-subspace of $V_{\hat{\eta}}$. Then $W_\RR=\oplus_{\hat{\sigma}\in\hat{\Sigma}}W_{\hat{\sigma}}$, and 
$W_{\hat{\sigma}}$ is maximal isotropic in $V_{\hat{\sigma}}$. The $e/2$-dimensional $\QQ$-subspace $\wedge^d_FW$ of $\wedge^d_\QQ V_\QQ$ spans over $\RR$ the subspace
spanned by the lines $\wedge^dW_{\hat{\sigma}}$. Hence, the subspace $(\wedge^d_FW)_\RR$ of $C(V_\RR)$ is 
%the projection, into the appropriate graded summand, of the image via $\varphi$ of 
the subspace spanned by $\varphi_{\hat{\sigma}}(\ell_{\hat{\sigma}}\otimes \ell_{\hat{\sigma}})$, $\hat{\sigma}\in\hat{\Sigma}$, where $\ell_{\hat{\sigma}}$ is the pure spinor of $W_{\hat{\sigma}}$ in $S_{\hat{\sigma}}$. In particular, considering $W=H^1(\hat{X},\QQ)$,
$\varphi_{\hat{\sigma}}(1\otimes 1)$ spans $\wedge^{top}H^1_{\hat{\sigma}}(\hat{X})$.
\end{rem}

The isomorphism $\tilde{\varphi}:S_\RR\otimes S_\RR\rightarrow \wedge^*V_\RR$ maps $S_{\hat{\sigma}_0}\otimes S_{\hat{\sigma}_0}$ to 
$\wedge^*V_{\hat{\sigma}_0}\otimes\otimes_{\{ \hat{\sigma} \ : \ \hat{\sigma}\neq\hat{\sigma}_0\}}\wedge^{top}H^1_{\hat{\sigma}}(\hat{X})$, where we now regard $\wedge^{top}H^1_{\hat{\sigma}}(\hat{X})$ as a line in $\wedge^*V_{\hat{\sigma}}$. The isomorphism $\phi:S_\RR\otimes S_\RR\rightarrow \wedge^*V_\RR$ maps $S_{\hat{\sigma}_0}\otimes S_{\hat{\sigma}_0}$ to 
$\wedge^*V_{\hat{\sigma}_0}\otimes\otimes_{\{ \hat{\sigma} \ : \ \hat{\sigma}\neq\hat{\sigma}_0\}}\wedge^{top}H^1_{\hat{\sigma}}(X)$.

Set $S_{\hat{\eta}}:=S_\QQ\cap[\sum_{\hat{\sigma}\in\hat{\Sigma}}S_{\hat{\sigma}}]$, $S^+_{\hat{\eta}}:=S^+_\QQ\cap[\sum_{\hat{\sigma}\in\hat{\Sigma}}S^+_{\hat{\sigma}}]$, and define $S^-_{\hat{\eta}}$ analogously. These are subspaces of $S_\QQ:=\wedge^*_\QQ H^1(X,\QQ)$. The latter does not have a natural structure as an $F$-vector space, but it is naturally a representation of the  group $F^\times$. 
If we let $f\in F^\times$ act on $\theta\in H^1(X,\QQ)^*$ by  $(f,\theta(\bullet))\mapsto \theta(f^{-1}(\bullet))$, then $\wedge^{top}H^1(X,\RR)^*$ has the character $\prod_{\hat{\sigma}\in\hat{\Sigma}}\hat{\theta}^{-d/2}$. Identify $C(V_\RR)$ with $\End(S_\RR)$ via the isomorphism $m$.
The isomorphism $\varphi:S_\RR\otimes S_\RR\rightarrow C(V_\RR)$ has weight $\prod_{\hat{\sigma}\in\hat{\Sigma}}\hat{\theta}^{-d/2}$, which is the lowest weight in $C(V_\RR)$. The weights in $S_{\hat{\sigma}}\otimes_\RR S_{\hat{\sigma}}$ are 
$\{\hat{\theta}^j \ : \ 0\leq j \leq d\}$ and it is characterized as the subspace of $S_\RR\otimes S_\RR$ containing all characters with these weights.

%*************
% Hide
%*************
\hide{
Set $(S\otimes S)_{\hat{\eta}}:=(S_\QQ\otimes_\QQ S_\QQ)\cap [\sum_{\hat{\sigma}\in\hat{\Sigma}}S_{\hat{\sigma}}\otimes_\RR S_{\hat{\sigma}}]$.

\begin{lem}
\label{lemma-Chevalley-isomorphism-over-F} (??? needs to be corrected using Remark \ref{rem-image-via-varphi-of-top-wedge-over-F-of-a-maximal-isotropic-F-subspace} ???)
The isomorphism 
\[
\phi\circ (id\otimes\tau):S_\QQ\otimes_\QQ S_\QQ\rightarrow \wedge^*V_\QQ,
\] 
where $\phi$ is
given in (\ref{eq-phi-introduction}), 
%is also an isomorphism of $F$-vector spaces and it 
restricts to an isomorphism of $(S\otimes S)_{\hat{\eta}}$ onto (???).
\end{lem}

\begin{proof}
The isomorphism $\phi\circ (id\otimes\tau)$ is $\Spin(V)$-equivariant, by Proposition \ref{prop-the-orlov-image-of-HW-P-projects-into-the-3-dimensional-space-of-HW-classes}, and is hence also $\Spin(V_{\hat{\eta}})$-equivariant. Hence, it maps the irreducible $\Spin(V_{\hat{\eta}})$-sub-representation
$S_{\hat{\eta}}\otimes_F S_{\hat{\eta}}$ to a $\Spin(V_{\hat{\eta}})$-sub-representation of $\wedge^*_FV_{\hat{\eta}}$. (??? show that the restriction is an isomorphism of $F$-vector spaces and that the multiplicity of $\wedge^*_FV_{\hat{\eta}}$ in $\wedge^*_\QQ V_\QQ$ is $1$ ???)
Note that $\wedge^*_FV_{\hat{\eta}}\cong \wedge^*_FH^1(X,\QQ)\otimes_F \wedge^*_FH^1(\hat{X},\QQ)$
\end{proof}

%*************
% End Hide
%*************
}
%****************************************************************
% 
%****************************************************************
\subsection{Complex multiplication by $K$ from a maximal isotropic subspace of $H^1(X\times\hat{X},\QQ)\otimes_F K$}
\label{sec-complex-multiplication--by-CM-field-K}

We have the isomorphism
\[
\Hom_K(H^1(X,\QQ)\otimes_FK,K)\IsomRightArrow \Hom_\QQ(H^1(X,\QQ)\otimes_FK,\QQ), 
\]
sending $h\in\Hom_K(H^1(X,\QQ)\otimes_FK,K)$ to $tr_{K/\QQ}\circ h$. We get the $K$ bilinear pairing $(\bullet,\bullet)_{V_{\hat{\eta},K}}:V_{\hat{\eta},K}\otimes_K V_{\hat{\eta},K}\rightarrow K$ analogous to $(\bullet,\bullet)_{V_{\hat{\eta}}}$. 
We define 
\begin{equation}
\label{eq-Spin-V-K}
\Spin(V_{\hat{\eta},K})
\end{equation}
as the associated spin group. Let
\begin{equation}
\label{eq-Spin-V-Q-hat-eta-K}
\Spin(V_{\hat{\eta},K})_\QQ
\end{equation}
be the subgroup of $\Spin(V_{\hat{\eta},K})$ leaving the subset $V_\QQ$ of $V_{\hat{\eta},K}$  invariant.
%We keep the subscript $\hat{\eta}$ to indicate that $V_K$ is $V_\QQ\otimes_{\hat{\eta}{(F)}}K$ and not $V_\QQ\otimes_\QQ K$.

\begin{rem}
The group $\Spin(V_{\hat{\eta}})$ is a subgroup of $\Spin(V_{\hat{\eta},K})_\QQ$. 
%The group $\Spin(V_{\hat{\eta},K})_\QQ$ is the $Gal(K/F)$-invariant subgroup of $\Spin(V_{\hat{\eta}}\otimes_F K)$, since .
The difference is explained by the following diagram. The Galois involution $\iota$ acts on $C(V_{\hat{\eta}}\otimes_FK)\cong C(V_{\hat{\eta}})\otimes_FK$
via its action on the second tensor factor $K$. The subset $\Spin(V_{\hat{\eta},K})$ of $C(V_{\hat{\eta}}\otimes_FK)$ is $\iota$-invariant and is thus endowed with a $Gal(K/F)$-action.
\[
\xymatrix{
0 \ar[r] & \{\pm 1\} \ar[r] \ar[d]^\cong & \Spin(V_{\hat{\eta}}) \ar[r] \ar[d]^\cong & SO(V_{\hat{\eta}}) \ar[r] \ar[d]^\cong &F^\times/F^{\times, 2} \ar[r] \ar[d] & 0
\\
0 \ar[r] & \{\pm 1\} \ar[r] \ar[d]^\cong & \Spin(V_{\hat{\eta},K})^\iota \ar[r] \ar[d]^{\cap} & SO(V_{\hat{\eta}}\otimes_F K)^\iota \ar[r] \ar[d]^= &
[K^\times/K^{\times,2}]^\iota \ar[d]^=
\\
0 \ar[r] & \{\pm 1\} \ar[r] \ar[d]^\cong &\Spin(V_{\hat{\eta},K})_\QQ \ar[r] \ar[d]^{\cap} & SO(V_{\hat{\eta}}\otimes_F K)^\iota \ar[r] \ar[d]^\cap &
[K^\times/K^{\times,2}]^\iota \ar[d]^\cap
\\
0 \ar[r] & \{\pm 1\} \ar[r] &\Spin(V_{\hat{\eta},K}) \ar[r] & SO(V_{\hat{\eta}}\otimes_F K) \ar[r] & K^\times/K^{\times,2} \ar[r] & 0
}
\]
The top and bottom rows are exact, by \cite[II.3.7]{chevalley}. The third row is exact, by definition of $\Spin(V_{\hat{\eta},K})_\QQ$. 
The second row is not exact at $SO(V_{\hat{\eta}}\otimes_F K)^\iota$, since the homomorphism $F^\times/F^{\times, 2}\rightarrow K^\times/K^{\times,2}$ is not injective. Hence, $\Spin(V_{\hat{\eta}})$ is a proper subgroup of $\Spin(V_{\hat{\eta},K})_\QQ$ of index equal to the cardinality of the kernel of $F^\times/F^{\times, 2}\rightarrow K^\times/K^{\times,2}$.
\end{rem}

We choose the element $g_0$ in (\ref{eq-g}) to be an element of $\Spin(V_{\hat{\eta},K})$. 
%So we would get a set of isometries 
%\[
%\{g_\sigma \ : \ \sigma\in\Sigma\}.
%\]
Set 
\[
W:=\rho_{g_0}(H^1(\hat{X},\QQ)\otimes_FK).
\]
Its pure spinor in $\wedge^*_K [H^1(X,\QQ)\otimes_FK]\cong [\wedge^*_F [H^1(X,\QQ)]\otimes_FK]$ is $m_{g_0}(1)$. 
Define $\eta:K\rightarrow \End(V_{\hat{\eta},K})$, so that $\eta(t)$
acts on $W$ via its $K$-subspace structure and $\eta(t)$ acts on $\bar{W}$ via scalar multiplication by $\iota(t)$, where $\iota$ is the generator of $Gal(K/F)$. The vector space 
$V_{\hat{\eta}}=H^1(X,\QQ)\oplus H^1(X,\QQ)^*$ is an $F$-subspace of $V_{\hat{\eta},K}$ of half the dimension. 
It is equal to\footnote{Here we abuse notation and write $\bar{w}$ for $(id_{V_\QQ}\otimes \iota)(w)$, where $\iota$ is the generator of $Gal(K/F)$. 
} 
\begin{equation}
\label{eq-V-QQ-is-isomorphic-to-W-T}
V_{\hat{\eta}}=\{w+\bar{w} \ : \ w\in W\},
\end{equation}
provided 
\begin{equation}
\label{eq-W-T-and-its-conjugate-are-complementary}
W\cap \bar{W}=(0),
\end{equation} 
which we assume. 
Hence, the subspace $V_{\hat{\eta}}$ is invariant under $\eta(t)$ and we get an embedding
\[
\eta:K\rightarrow \End_\QQ\left[H^1(X,\QQ)\oplus H^1(X,\QQ)^*\right].
\]
Extending coefficient linearly from $\QQ$ to $\CC$ we get the embedding
\[
\eta:K\rightarrow \End_\CC\left[H^1(X,\CC)\oplus H^1(\hat{X},\CC)\right].
\]

\begin{lem}
\label{lemma-units-are-isometries}
The endomorphism $\eta(\iota(t))$ of $V_\QQ$ is the adjoint of $\eta(t)$ with respect to both $(\bullet,\bullet)_{V_{\hat{\eta},K}}$ and $(\bullet,\bullet)_{V_\QQ}$, for all $t\in K$. Consequently, $\eta(t)$ belongs to the image of $\Spin(V_{\hat{\eta},K})_\QQ$ in $SO(V_\QQ)$, if and only if $t\iota(t)=1$.
\end{lem}

\begin{proof}
It suffices to prove the statement for $\eta:K\rightarrow V_{\hat{\eta},K}$ and the $K$ valued bilinear pairing $(\bullet,\bullet)_{V_{\hat{\eta},K}}$, since
$(x,y)_{V_\QQ}=tr_{K/\QQ}((x,y)_{V_{\hat{\eta},K}})$, for all $x,y\in V_\QQ$. Let $x\in W$ and $y\in\bar{W}$. We have
\begin{eqnarray*}
(\eta(t)x,y)_{V_{\hat{\eta},K}}&=&(tx,y)_{V_{\hat{\eta},K}}=t(x,y)_{V_{\hat{\eta},K}}=(x,ty)_{V_{\hat{\eta},K}}=(x,\iota^2(t)y)_{V_{\hat{\eta},K}}
\\
&=& (x,\eta(\iota(t)y)_{V_{\hat{\eta},K}},
\end{eqnarray*}
where the first and last equality are by definition of $\eta$. The statement follows for all $x,y\in V_{\hat{\eta},K}$, since $W$ and $\bar{W}$ are complementary and isotropic.

If $\eta(t)$ belongs to the image of $\Spin(V_{\hat{\eta},K})$, then it is an isometry and so $\eta(t)\eta(\iota(t))=id$. Hence, $t\iota(t)=1$. Conversely, if 
$t\iota(t)=1$, then $\eta(t)$ act on the complementary isotropic subspaces $W$ and $\bar{W}$ via scalar multiplication by $t$ and $t^{-1}$ and so belongs to the image of 
$\Spin(V_{\hat{\eta},K})$. As $\eta(t)$ leaves $V_\QQ$ invariant, it belongs to the image of $\Spin(V_{\hat{\eta},K})_\QQ$.
\end{proof}

\begin{rem}
\label{rem-isometry-as-primitive-element}
Note that there exists an element $t\in K$ satisfying $K=\QQ(t)$ and $t\iota(t)=1$, by \cite[Theorem 3.7]{huybrechts-k3-book}. The statement of \cite[Theorem 3.7]{huybrechts-k3-book} assumes that the CM-field is associated to the transcendental sublattice of a lattice of $K3$-type, but the proof of the existence of such a primitive element is valid for any CM-field.
\end{rem}

Let $K_-$ be the $-1$-eigenspace of $\iota:K\rightarrow K$. Given $t\in K_-$, set
\begin{equation}
\label{eq-Xi-t}
\Xi_t(x,y):= (\eta(t)x,y).
\end{equation}

\begin{cor}
\label{cor-Neron-Severi-group}
$\Xi_t$ is anti-symmetric and for $s\in K$ we have
\[
\Xi_t(\eta(s)x,\eta(s)y)=\Xi_t(\hat{\eta}(s\iota(s))x,y).
\]
We get the injective
homomorphism $\Xi:K_-\rightarrow \wedge^2_FV^*_\QQ$, given by $t\mapsto \Xi_t.$
\end{cor}

\begin{proof} We have
\[
\Xi_t(x,y)=(\eta(t)x,y)_V\stackrel{\mbox{Lemma \ref{lemma-units-are-isometries}}}{=}
(x,\eta(\iota(t))y)_V=-(\eta(t)y,x)=-\Xi_t(y,x).
\]
Hence $\Xi_t$ is anti-symmetric. 
We have $\Xi_t(\eta(f)x,y)=\Xi_t(x,\eta(f)y)$, for all $f\in F$, by Lemma \ref{lemma-units-are-isometries} again.
Hence $\Xi_t$ belongs to $\wedge^2_FV^*$. 
\end{proof}

Composing the homomorphism in Corollary \ref{cor-Neron-Severi-group} 
with the isomorphism $V\cong V^*$ we get an injective homomorphism from $K_-$ to $\wedge^2_FH^1(X\times \hat{X},\QQ)$.
Note that the Neron-Severi group of a simple abelian variety with complex multiplication by $K$ has rank $\dim_\QQ(F)=\dim_\QQ(K_-)$, by \cite[Prop. 5.5.7]{BL}.

We have two isomorphisms $V_{\hat{\sigma},\RR}\otimes_FK\rightarrow V_{\hat{\sigma},\RR}\otimes_\RR\CC$, one via $id_{V_{\hat{\sigma},\RR}}\otimes\sigma$, the other via $id_{V_{\hat{\sigma}},\RR}\otimes\bar{\sigma}$, where $\sigma$ is an embedding restricting to $F$ as $\hat{\sigma}$. 
A CM-type $T$ thus provides an embedding of $V_\RR\otimes_FK=\oplus_{\{\hat{\sigma}\in\hat{\Sigma}\}}V_{\hat{\sigma},\RR}\otimes_FK$ into $V_\RR\otimes_\RR\CC$. Precomposing with the obvious embedding of $V_\QQ\otimes_FK$ into $V_\RR\otimes_FK$ and post composing with the isomorphism $V_\RR\otimes_\RR\CC\cong V_\QQ\otimes_\QQ\CC$ we get the embedding
\begin{equation}
\label{eq-embedding-of-V-otimes-FK-in-V-CC}
e_T:V_\QQ\otimes_F K\rightarrow V_\QQ\otimes_\QQ \CC.
\end{equation}
Set
\begin{equation}
\label{eq-W-T}
W_{T,\CC}:=\span_\RR(e_T(W)).
\end{equation}
Note that $W_{T,\CC}$ is a complex subspace, as $\span_\RR(e_T(W))=\span_\CC(e_T(W))$.

The subspace 
$V_{\hat{\sigma},\RR}\otimes_FK$ of $(V_{\hat{\eta}}\otimes_\QQ\RR)\otimes_F K$ decomposes as the direct sum of its intersections 
$W_{\hat{\sigma}}:=[V_{\hat{\sigma},\RR}\otimes_FK]\cap[W\otimes_\QQ\RR]$ and 
$\bar{W}_{\hat{\sigma}}:=[V_{\hat{\sigma},\RR}\otimes_FK]\cap[\bar{W}\otimes_\QQ\RR]$, since $W$ and $\bar{W}$ are complementary $\hat{\eta}(F)$-subspaces of $V_{\hat{\eta}}\otimes_F K$.
%Consider the isomorphism
%\begin{equation}
%\label{eq-embedding-of-V-tensor-K-into-V-tensor-C}
%\sigma_{W_T}:(V_{\hat{\eta}}\otimes_\QQ\RR)\otimes_F K\rightarrow V_{\hat{\eta}}\otimes_\QQ\CC
%\end{equation}
%given by mapping 
Set
\begin{eqnarray}
\label{W-T-hat-sigma}
W_{T,\hat{\sigma}}&:=&[id_{V_{\hat{\sigma}},\RR}\otimes T(\hat{\sigma})](W_{\hat{\sigma}}),
%=[V_{\hat{\sigma},\RR}\otimes_FK]\cap [W\otimes_\QQ\RR]
\\
\nonumber
\bar{W}_{T,\hat{\sigma}}&:=&[id_{V_{\hat{\sigma}},\RR}\otimes T(\hat{\sigma})](\bar{W}_{\hat{\sigma}}).
%[V_{\hat{\sigma},\RR}\otimes_FK]\cap [\bar{W}\otimes_\QQ\RR]
\end{eqnarray}
% injectively into $V_{\hat{\sigma},\RR}\otimes_\RR\CC$ via $id_{V_{\hat{\sigma}},\RR}\otimes\sigma$, and mapping 
% \[
% \bar{W}_{T,\hat{\sigma}}&:=&[V_{\hat{\sigma},\RR}\otimes_FK]\cap [\bar{W}\otimes_\QQ\RR]
% \] 
% injectively into $V_{\hat{\sigma},\RR}\otimes_\RR\CC$ via $id_{V_{\hat{\sigma}},\RR}\otimes\bar{\sigma}$.
% We have two $K$ actions on the domain 
 %$(V_{\hat{\eta}}\otimes_\QQ\RR)\otimes_F K$ 
 %of $\sigma_{W_T}$, the left action on the left tensor factor $(V_{\hat{\eta}}\otimes_\QQ\RR)$ via $\eta$ and the right action on the right tensor factor $K$.
 %The isomorphism $\sigma_{W_T}$ conjugates the right action of $K$ on its domain to the action 
 %$\eta:K\rightarrow \End_\CC(V_{\hat{\eta}}\otimes_\QQ\CC)$, given above.
 We have $W_{T,\CC}=\oplus_{\hat{\sigma}\in\hat{\Sigma}}W_{T,\hat{\sigma}}$ and
 $\bar{W}_{T,\CC}=\oplus_{\hat{\sigma}\in\hat{\Sigma}}\bar{W}_{T,\hat{\sigma}}$.
 
%*************
% Hide
%*************
\hide{
We get the commutative diagram
\[
\xymatrix{
K\ar[r]^-\eta \ar[dr]_{\eta}& \End_\QQ\left(V_{\hat{\eta}}\right) \ar[d]
\\
& \End_\CC(V_{\hat{\eta}}\otimes_\QQ\CC),
}
\]
where the vertical homomorphism sends $a$ to $a\otimes id_\CC$ 
%We need to compare $V_{\hat{\eta}}\otimes_FK$ and $V_{\hat{\eta}}\otimes_\QQ K\cong V_{\hat{\eta}}\otimes_\QQ F\otimes_F K$. 
%We have the natural surjective and injective homomorphisms
%\begin{eqnarray*}
%(id_{V_{\hat{\eta}}}\otimes tr_{F/\QQ})\otimes id_K:(V_{\hat{\eta}}\otimes_\QQ F)\otimes_F K&\rightarrow& V_{\hat{\eta}}\otimes_F K,
%\\
%(id_{V_{\hat{\eta}}}\otimes 1)\otimes id_K:V_{\hat{\eta}}\otimes_F K &\rightarrow & (V_{\hat{\eta}}\otimes_\QQ F)\otimes_F K
%\end{eqnarray*}
Hence, 
%if we define $\eta_T:K\rightarrow \End(V_{\hat{\eta},K})$, so that $\eta_T(t)$
%acts on $W_T$ via  via the CM-type $T$, i.e., $\eta_T(t)$ acts on $g(H^1_{\hat{\sigma}}(X)\otimes_FK)$ 
%via the the representative $\sigma$ in $T$ restricting to $F$ as $\hat{\sigma}$, and let it act on $\bar{W}_T$ via the complex conjugate action, then 
$V_{\hat{\eta}}$ is invariant under $\eta(K)$. 
%and we get an embedding
%\[
%\eta_T:K\rightarrow \End_\QQ\left[H^1(X,\QQ)\oplus H^1(X,\QQ)^*\right].
%\]
In other words, $H^1(X,\QQ)\oplus H^1(\hat{X},\QQ)$ is invariant under the $\eta(K)$ action via
\[
\eta:K\rightarrow \End_\CC\left[H^1(X,\CC)\oplus H^1(\hat{X},\CC)\right].
\]
%*************
% End Hide
%*************
}

%****************************************************************
% 
%****************************************************************
\subsection{The Galois group action on the set $\T_K$ of CM-types}
\label{sec-Galois-action-on-set-of-CM-types}
Let $\sigma_0:K\rightarrow\CC$ be an embedding. Let $\tilde{K}$ be the Galois closure of $\sigma_0(K)$ in $\CC$. The subfield $\tilde{K}$ of $\CC$ is independent of $\sigma_0$. The set of embedding of $K$ in $\tilde{K}$ are in bijection to the set $\Sigma$ of its embeddings in $\CC$, as any embedding of $K$ in $\CC$ factors through $\tilde{K}$. Define $\tilde{F}\subset \RR$ similarly. Let $\T'$ be the set of unital homomorphisms of $F$-algebras
\[
T:K\rightarrow F\otimes_\QQ\tilde{K}.
\]
Above, inclusion of $F$ in $F\otimes_\QQ\tilde{K}$ is via $f\mapsto f\otimes 1$. 
We have the decomposition $F\otimes_\QQ\tilde{F}\cong \oplus_{\hat{\sigma}\in\hat{\Sigma}}\hat{\sigma}$, where we denote by $\hat{\sigma}$ both a field embedding $\hat{\sigma}:F\rightarrow \tilde{F}$ and a $1$-dimensional  vector space over $\tilde{F}$ endowed with an $F$ vector space structure, such that for all $f\in F$ and $v\in \hat{\sigma}$ the equality $fv=\hat{\sigma}(f)v$ holds.
We have the isomorphisms
\[
F\otimes_\QQ\tilde{K}\cong
(F\otimes_\QQ\tilde{F})\otimes_{\tilde{F}}\tilde{K}\cong
\oplus_{\hat{\sigma}\in\hat{\Sigma}}\hat{\sigma}\otimes_{\tilde{F}}\tilde{K}.
\]
Note that the direct summand $\hat{\sigma}\otimes_{\tilde{F}}\tilde{K}$ is naturally isomorphic to $\tilde{K}$.
Composing an element $T:K\rightarrow F\otimes_\QQ\tilde{K}$ of $\T'$ with the projection onto the direct summand $\hat{\sigma}\otimes_{\tilde{F}}\tilde{K}$ we get an embedding $\sigma:K\rightarrow \tilde{K}$, which restricts to $F$ as $\hat{\sigma}$. Hence, the set $\T'$ is in natural bijection with the set $\T_K$ of CM-types.

The group $Gal(\tilde{K}/\QQ)$ acts on $F\otimes_\QQ\tilde{K}$ via its action on the right tensor factor. Hence $Gal(\tilde{K}/\QQ)$ acts on $\T'$ via 
$g(T):=(id_F\otimes g)\circ T$, for all $g\in Gal(\tilde{K}/\QQ)$ and $T\in \T'$. Let us describe the corresponding action on $\T_K$.
Let $r:\Sigma\rightarrow \hat{\Sigma}$ send $\sigma:K\rightarrow \tilde{K}$ to its restriction to $F$.
Then $\T_K=\{f:\hat{\Sigma}\rightarrow \Sigma \ : \ r\circ f=id_{\hat{\Sigma}}\}$. Given $g\in Gal(\tilde{K}/\QQ)$, define $m_g:\Sigma\rightarrow\Sigma$ 
by $m_g(\sigma):=g\circ\sigma$. Denote by $\hat{g}\in Gal(\tilde{F}/\QQ)$ the restriction of $g$ to $\tilde{F}$. Define
$m_{\hat{g}}:\hat{\Sigma}\rightarrow\hat{\Sigma}$ by $m_{\hat{g}}(\hat{\sigma})=\hat{g}\circ\hat{\sigma}$. Then $Gal(\tilde{K}/\QQ)$ acts on $\T_K$ sending
$f\in \T_K$ to
\[
g(f) := m_g\circ f\circ m_{\hat{g}^{-1}}.
\]

Given $T\in \T'$,
we get the embedding
\[
e_T:V_\QQ\otimes_F K\RightArrowOf{id_V\otimes T} V_\QQ\otimes_F(F\otimes_\QQ\tilde{K})\cong V_\QQ\otimes_\QQ\tilde{K}.
\]
The embedding in (\ref{eq-embedding-of-V-otimes-FK-in-V-CC}) naturally factors through $e_T$ above, hence the use of the same notation. 
Given an $n$-dimensional subspace $W$ of $V_\QQ\otimes_F K$, 
let 
\[
W_T
\] 
be the subspace of $V_\QQ\otimes_\QQ\tilde{K}$ spanned over $\tilde{K}$ by $e_T(W)$. 
The set 
\[
\{
(id_V\otimes g)(W_T) \ : \ g \in Gal(\tilde{K}/\QQ)
\}
\]
determines a reduced subscheme of $Gr(n,V_\QQ\otimes_\QQ\tilde{K})$ defined over $\QQ$. 
The set
\[
\{W_T \ : \ T\in\T'\}
\]
is the union of $Gal(\tilde{K}/\QQ)$-orbits, hence it corresponds to a reduced subscheme of $Gr(n,V_\QQ\otimes_\QQ\tilde{K})$ defined over $\QQ$. 
Assume that $W_T$ is an even maximal isotropic subspace of $V_\QQ\otimes_\QQ\tilde{K}$, for all $T\in \T'$.
Let $\ell_T\in \PP(S^+_{\tilde{K}})$ be the pure spinor of $W_T$. 
The isomorphism between the spinor variety in $\PP(S^+_{\tilde{K}})$ and the even component of the grassmannian of maximal isotropic subspaces of $V_\QQ\otimes_\QQ\tilde{K}$ is defined over $\QQ$. Hence, the set
$\{\ell_T \ : \ T\in\T'\}$ corresponds to a reduced subscheme defined over $\QQ$ and its span in $S^+_{\tilde{K}}$ is of the form $B\otimes_\QQ\tilde{K}$,
for some subspace $B$ of $S^+_\QQ$.

%****************************************************************
% 
%****************************************************************
\subsection{A linear space $B$ secant to the spinor variety}
We have two $K$-actions on $V_\QQ\otimes_FK$. Each of $W$ and $\bar{W}$ is invariant with respect to both. The two actions coincide on $W$ and are $\iota$-conjugate on $\bar{W}$.
The $2$-dimensional  vector $K$-subspace $\wedge^d_KW\oplus\wedge^d_K\bar{W}$ of $\wedge^d_K(V_\QQ\otimes_FK)\cong(\wedge^d_F V_\QQ)\otimes_FK$
is $\iota$-invariant and is hence of the form $HW\otimes_FK$ for a $2$-dimensional $F$-subspace $HW$ of
$\wedge^d_FV_\QQ\subset H^d(X\times\hat{X},\QQ)$. Now $HW\otimes_FK$ is also a subspace of $\wedge^d_{\eta(K)}[V_\QQ\otimes_FK]$, and so 
$HW$ is a $1$-dimensional $\eta(K)$-subspace of $\wedge^d_{\eta(K)}V_\QQ$. The latter is one-dimensional over $K$ and so we get the equality  
\[
HW=\wedge^d_{\eta(K)}H^1(X\times\hat{X},\QQ).
\] 
%The $\hat{\eta}(F)$-subspace $HW$ is independent of the CM-type $T$, when viewed as a subspace of $\wedge^d_{\eta(K)}V_{\hat{\eta},K}$.
% (see Remark \ref{rem-notation-T}). 
%but $T$ determines its embedding in $\wedge^dH^1(X\times\hat{X},\CC)$.
%the choice of the $\left(\prod_{\sigma\in T}\sigma\right)$-eigenspace $\wedge^d_FW_T$ in $HW$.
We have the equality
\begin{eqnarray*}
%HW\otimes_\QQ\CC&=&\oplus_{\hat{\sigma}\in\hat{\Sigma}} \left(
%\wedge^d_\CC[W\otimes_\QQ\RR]_{\hat{\sigma}} \oplus \wedge^d_\CC[\bar{W}\otimes_\QQ\RR]_{\hat{\sigma}} 
%\right),
%\\
HW\otimes_\QQ\CC&=&\oplus_{\sigma\in T} \left(
\wedge^d_\CC[W_{T,\CC}]_\sigma \oplus \wedge^d_\CC[\bar{W}_{T,\CC}]_{\bar{\sigma}}
\right),
\end{eqnarray*}
for every CM-type $T$, where 
%in the second equality 
the characters are of the $\eta$ action of $K$
and $W_{T,\CC}$ is given in (\ref{eq-W-T}).
%The first equality depends on the CM-type $T$ as well as written, but 
%the wedge product could be taken with respect to $\RR\otimes_{\hat{\sigma}(F)}K$ instead of $\CC$ in order to be independent of the CM-type $T$, 
%in which case the left hand side should be replaced by $HW\otimes_\QQ\RR\otimes_{F}K$. 

%Compute the intersection points of $\PP(P_{T,\CC})$ with the even spinorial variety. It contains 
%$\ell_{W_T}$ and $\ell_{\bar{W}_T}$ 
%but should contain a point for other CM-types for $K$.

%*************
% Hide
%*************
\hide{
%****************************************************************
% 
%****************************************************************
\subsection{New strategy - pure $F$-spinors}
Let $P\subset S^+_\QQ$ be the $e$-dimensional $\QQ$-subspace, such that $P_\RR$ is spanned by the pure spinors $\ell_{\hat{\sigma}}$ of 
$W_{\hat{\sigma}}$, $\hat{\sigma}\in\hat{\Sigma}$, and their conjugates $\bar{\ell}_{\hat{\sigma}}$. Note that $\ell_{\hat{\sigma}_0}$ is the pure spinor of the maximal isotropic subspace $W_{\ell_{\hat{\sigma}_0}}\oplus_{\ell_{\hat{\sigma}}\neq \ell_{\hat{\sigma}_0}}H^1_{\ell_{\hat{\sigma}}}(\hat{X})$.
Denote $\ell_{\hat{\sigma}}$ by $\ell_\sigma$, if $\sigma\in T$, and set $\ell_\sigma:=\bar{\ell}_{\hat{\sigma}}$, if $\sigma\in \iota(T)$.
%Let $P\subset S^+_{\hat{\eta}}$ be the $2$-dimensional $F$-subspace 
%corresponding to the $K$-subspace spanned by the two pure spinors $m_{g_0}(1)$ and its conjugate $\iota(m_{g_0}(1))$
%in $S^+_{\hat{\eta}}\otimes_F K$.
%Set $\ell_{W_T}:=\span_F\{m_{g_0}(1)\}$, considered as a point in the projective space $\PP(S^+_{\hat{\eta}}\otimes_F K)$ over $K$ and 
%let $\ell_{\bar{W}_T}$ be its $\iota$-conjugate.

Use Remark \ref{rem-image-via-varphi-of-top-wedge-over-F-of-a-maximal-isotropic-F-subspace} to relate the image via Orlov's equivalence of the subspace $P\otimes_\QQ P$ of $S^+_\QQ \otimes_\QQ S^+_\QQ$ (namely, of the subspace spanned by $\ell_{\sigma_1}\otimes\ell_{\sigma_2}$, $\sigma_1,\sigma_2\in\Sigma$) to 
the subspace $HW$ of $\wedge^d_FH^1(X\times\hat{X},\QQ)$. Now choose $k$-secant sheaves $F_i$ with $ch(F_i)$ in $P$.
(??? not helpful ???)

%{\bf Problem:} $P$ is contained in the subalgebra $S^+_{\hat{\eta}}$, which is contained in the subspace $\oplus_{k=0}^{2d}\wedge^k V_\QQ$, so 
%$ch_k(F_i)$ vanishes for $2d<k\leq 4n=de$. We get a constraint, when $e>2$.
%*************
% End Hide
%*************
}

%****************************************************************
% 
%****************************************************************
\subsubsection{Pure spinors}
Given an embedding $\sigma\in\Sigma$ of $K$, denote by $V_{\sigma,\CC}$ the subspace of $V_\QQ\otimes_\QQ\CC$ on which the $\eta$ action of $K$ acts via the character $\sigma:K\rightarrow\CC$. 
%These subspaces do not depend on the initial choice of CM-type $T$. 
%Define $V_{\hat{\sigma}}\subset V_\QQ\otimes_\QQ \RR$ similarly via $\hat{\eta}$.
Let
\begin{equation}
\label{eq-V-sigma}
V_\sigma
\end{equation}
be the analogous direct summand of $V_\QQ\otimes_\QQ\tilde{K}$, where the algebraic closure $\tilde{K}$ of $K$ in $\CC$ is introduced in section \ref{sec-Galois-action-on-set-of-CM-types}.
If $\sigma$ belongs to a CM-type $T$, then 
\[V_{\sigma,\CC}=W_{T,\hat{\sigma}}
%W_{T,\CC}\cap (V_{\hat{\sigma}}\otimes_\RR\CC) 
\ \  \mbox{and} \ \ V_{\bar{\sigma},\CC}=\bar{W}_{T,\hat{\sigma}},
%\bar{W}_{T,\CC}\cap (V_{\hat{\sigma}}\otimes_\RR\CC).
\]
where $W_{T,\hat{\sigma}}$ is given in (\ref{W-T-hat-sigma}).
We have the direct sum decomposition $V_\CC=\oplus_{\sigma\in\Sigma}V_{\sigma,\CC}$
and
\begin{equation}
\label{eq-decomposition-of-W-T-CC}
W_{T,\CC}=\oplus_{\{\hat{\sigma}\in\hat{\Sigma}\}}V_{T(\hat{\sigma}),\CC}.
\end{equation} 
%Set
%\begin{equation}
%\label{eq-V-sigma-CC}
%V_\sigma:=V_{\sigma,\CC}\cap [V_\QQ\otimes_\QQ \sigma(K)]\subset V_\QQ\otimes_\QQ\CC.
%\end{equation}
%*************
% Hide
%*************
\hide{
The composition $V_\QQ\rightarrow V_\CC\rightarrow V_{\sigma,\CC}$, of the natural embedding with the projection on the direct summand, 
maps $V_\QQ$ isomorphically onto $V_\sigma$. The direct sum $\oplus_{\sigma\in\Sigma}V_\sigma$, considered as a $\QQ$-subspace of $V_\QQ\otimes_\QQ(K\otimes_\QQ\CC)$,
is isomorphic to the image of $V_\QQ\otimes_\QQ K$ via the analogue of (\ref{eq-embedding-of-V-tensor-F}).
%*************
% End Hide
%*************
}

\begin{rem}
If $K$ is a Galois extension of $\QQ$, then $Gal(K/\QQ)$ acts transitively on the set of characters via $\sigma\mapsto \sigma\circ g^{-1}$, for all $g\in Gal(K/\QQ)$.
In that case the subfield $\sigma(K)\subset \CC$ is independent of $\sigma$ and we may identify $K$ with this subfield by choosing an embedding $\sigma_0$.
The subspace $V_\sigma$ is then the subspace of $V_\QQ\otimes_\QQ \sigma_0(K)$, where $\eta(a)$ acts via $id_{V_\QQ}\otimes \sigma(a)$, for all $a\in K$. Similarly, the subspaces $V_{\hat{\sigma}}$ are the analogous subspaces of $V_\QQ\otimes_\QQ \hat{\sigma}_0(F)$.
\end{rem}

%*************
% Hide
%*************
\hide{
For each CM-type $T'$ we define the subspace $\tilde{W}_{T'}$ of $V_\QQ\otimes_\QQ K$ (not $V_{\hat{\eta},K}$) as the
%the image via $g$  in (\ref{eq-g})  of the 
direct sum  
\[
\tilde{W}_{T'}:= \oplus_{\sigma\in T'}V_\sigma.
\] 
We define the subspace $W_{T'}$ of $V_{\hat{\eta},K}$ as the image of $\tilde{W}_{T'}$ via the natural homomorphism 
\begin{equation}
\label{eq-mapping-tilde-W-T-to-W-T}
V_\QQ\otimes_\QQ K
%\RightArrowOf{id_{V_\QQ}\otimes tr_{K/\QQ}} 
\rightarrow V_\QQ\otimes_F K=:V_{\hat{\eta},K}
\end{equation}
mapping $v\otimes(f_1+f_2\sqrt{-q})$ to $\hat{\eta}_{f_1}(v)+\hat{\eta}_{f_2}(v)\otimes\sqrt{-q}$, for all $v\in V_\QQ$ and $f_1, f_2\in F$.
(??? Check that for the originally chosen CM-type $T$ we do get the previously defined $W_T$. Note that $\dim_\QQ(V_\sigma)=\dim_\QQ(V_\QQ)=\dim_\QQ(W_T)$. We may need to change the definition and consider instead the image of $V_\QQ$ via its diagonal embedding into $\tilde{W}_{T'}$ and map it to $V_{\hat{\eta},K}$ via (\ref{eq-mapping-tilde-W-T-to-W-T})???)
(??? 
%When $T'=T$ we do not get $W_T$ so the notation is confusing. Use $\tilde{W}_{T'}$ instead. 
%The subspace $W_T$ is the image of $\tilde{W}_T$ via
%\begin{equation}
%\label{eq-tr-mapping-tilde-W-T-to-W-T}
%V_\QQ\otimes_\QQ K\cong V_\QQ\otimes_\QQ F\otimes_F K\LongRightArrowOf{(id_{V_\QQ}\otimes tr_{F/\QQ})\otimes id_K} V_\QQ\otimes_F K.
%\end{equation}
%What is the relationship between $W_T$ and $\tilde{W}_T$? 
Is $\tilde{W}_T$ isotropic? If indeed it is, what is its pure spinor and how is it related to that of $W_T$?   ???)
%What is the analogue $W_{T'}$ of $W_T$ for another CM-type $T'$? It can 
%*************
% End Hide
%*************
}

For each CM-type $T$ 
consider the subspace $W_{T,\CC}$ of
%Can $W_{T}$ be defined in 
$V_\QQ\otimes_\QQ\CC$.
%as

\begin{lem}
\label{lemma-W-T-is-maximal-isotropic}
The subspace $W_{T,\CC}$ of $V_\CC$ is maximal isotropic, for every CM-type $T$.
\end{lem}

\begin{proof}
If $\hat{\sigma}_1\neq \hat{\sigma}_2$, then $V_{\sigma_1,\CC}$ and $V_{\sigma_2,\CC}$ are orthogonal, since $V_{\hat{\sigma}_1}$ and $V_{\hat{\sigma}_2}$ are.
Ditto for $V_{\bar{\sigma}_1,\CC}$ and $V_{\sigma_2,\CC}$ and for  $V_{\bar{\sigma}_1,\CC}$ and $V_{\bar{\sigma}_2,\CC}$. The subspace 
$V_{\sigma,\CC}$ is isotropic, since $W_{T,\CC}$ is, and $V_{\bar{\sigma},\CC}$ is isotropic, since $\bar{W}_{T,\CC}$ is.
Hence, $W_{T,\CC}$ is  a maximal isotropic subspace of 
$V_\QQ\otimes_\QQ\CC$
%$V_{\hat{\eta},K}:=V_\QQ\otimes_F K$, 
for every CM-type $T$.  
\end{proof}

%**********
% Hide
%**********
\hide{
We set (??? these lines are not needed ???)
\[
P_{T'}:=\ell_{W_{T',\CC}}+\ell_{\bar{W}_{T',\CC}}.
\]
If $T'\neq T$, then $P_{T'}$ is a $2$-dimensional $\RR$-subspace of $H^{ev}(X,\RR)$. For our initially chosen CM-type $T$, $P_T$ is a $2$-dimensional $\QQ$-subspace (??? NO, it need not be defined over $\QQ$ ???) of $H^{ev}(X,\QQ)$. If $W_{T,\CC}$ is the image of $H^1(\hat{X},\QQ)\otimes_F K$ via the isometry $g_0$ in Example \ref{example-isometry-g-0}, then 
\[
P_T=\span_\QQ\left\{
\left(1-\hat{\eta}_q(\Theta\wedge_F\Theta)+\dots\right),
\left(\Theta-\hat{\eta}_q(\Theta\wedge_F\Theta\wedge_F\Theta)+\dots\right)
%\left(
%1-\hat{\eta}_q(\Theta^2/2)+\hat{\eta}_q^2(\Theta^4/4!)+\dots
%\right),
%\left(
%\Theta-\hat{\eta}_q(\Theta^3/3!)+\hat{\eta}_q^2(\Theta^5/5!)+\dots
%\right)
\right\}
\]
(see Equation (\ref{eq-pure-spinor-rho-g-0-of-1})).

%**********
% End Hide
%**********
}

Denote by $\T_K$ the set of $2^{e/2}$ complex multiplication types. 
The element $\iota$ of $Gal(K/\QQ)$ induces an involution of $\T_K$. Let $\ell_T\subset H^{ev}(X,\CC)$ be the one-dimensional subspace spanned by a pure spinor corresponding to the subspace $W_{T,\CC}$, $T\in\T_K$.

\begin{lem}
\label{lemma-B-is-rational}
The linear subspace of $H^{ev}(X,\CC)$ spanned by $\ell_T$, for all $T\in\T_K$, is of the form $B\otimes_\QQ\CC$, where $B$ is a rational subspace 
\[
B\subset H^{ev}(X,\QQ).
\] 
\end{lem}

\begin{proof}
The statement follows from Lemma \ref{lemma-W-T-is-maximal-isotropic} and the discussion in Section \ref{sec-Galois-action-on-set-of-CM-types}.
\end{proof}

Let $\ell_\sigma\in S^+_{\hat{\sigma}}\otimes_\RR\CC$ be the pure spinor of the maximal isotropic subspace $V_\sigma$ of $V_{\hat{\sigma}}$.
Under the tensor decomposition (\ref{eq-tensor-product-decomposition-of-pure-spinor}) of pure spinors of $F$-invariant maximal isotropic subspaces of $V_\CC$,
we have
\begin{equation}
\label{eq-tensor-factorization-of-ell-T}
\ell_T=\otimes_{\hat{\sigma}\in\hat{\Sigma}}\ell_{T(\hat{\sigma})},
\end{equation}
by the equality (\ref{eq-decomposition-of-W-T-CC}).
Denote by $P_{\hat{\sigma}}\subset S^+_{\hat{\sigma}}$ the secant plane spanned by $\ell_{\sigma}$ and $\ell_{\bar{\sigma}}$.
Then 
\begin{equation}
\label{eq-tensor-factorization-of-B}
B_\RR=\otimes_{\hat{\sigma}\in\hat{\Sigma}}P_{\hat{\sigma}}.
\end{equation}

%**********
% Hide
%**********
\hide{
\begin{rem}
Let $\tilde{K}$ be the (???)  Galois closure of $K$ over $\QQ$ and $H\subset Gal(\tilde{K}/\QQ)$ the subgroup leaving all elements of $K$ invariant. 
Let $\tilde{\Sigma}$ be the set of complex embeddings of $\tilde{K}$, so that $\Sigma=\tilde{\sigma}/H$. Then $Gal(\tilde{K}/\QQ)$ acts transitively on $\Sigma$, by
$\sigma\mapsto \sigma\circ g^{-1}$, $g\in Gal(\tilde{K}/\QQ)$. The stabilizer of each element of $\Sigma$ is $H$. 
The subfield $\tilde{K}':=\tilde{\sigma}(\tilde{K})$ of $\CC$ is independent of $\tilde{\sigma}\in\tilde{\Sigma}$, since $Gal(\tilde{K}/\QQ)$ acts transitively on $\tilde{\Sigma}$.
%Let $\tilde{\sigma}\in \tilde{\Sigma}$ and $g\in Gal(\tilde{K}/\QQ)$. 
The subspace $V_{\sigma,\CC}$ is defined over $\sigma(K)$ (and hence over $\tilde{K}'$), i.e., is associated to the eigenspace $[V_\QQ\otimes_\QQ \sigma(K)]_{\sigma(t)}$ of $\eta(t)$.
Hence, the subspace $W_{T',\CC}$ is defined over $\tilde{K}'$, for all $T'\in \T_K$. It follows that each line
$\wedge^{2n}_\CC W_{T',\CC}$ is defined over $\tilde{K}'$. The group $Gal(\tilde{K}/\QQ)$ acts on $V_\QQ\otimes_\QQ \tilde{K}'$ and
on $(\wedge^{2n}_\QQ V_\QQ)\otimes \tilde{K}'\cong \wedge^{2n}_{\tilde{K}'}[V_\QQ\otimes_\QQ\tilde{K}']$. 

The subspace $\wedge^{2n}W_{T',\CC}$ of $\wedge^{2n}V_\CC$ is characterized by the property that for $a\in K$, $\wedge^{2n}\eta(a)$ acts on $\wedge^{2n}W_{T',\CC}$ via $\prod_{\sigma\in T'}(\sigma(a))^d$. The set of characters
$\{\prod_{\sigma\in T'}(\sigma)^d \ : \ T'\in \T_K\}$ is $Gal(\tilde{K}/\QQ)$-invariant. Hence, the subspace
\[
\oplus_{T'\in\T_K}\wedge^{2n}W_{T',\CC}
\]
is defined over $\QQ$, as is the $0$-subscheme of $\PP(\wedge^{2n}V_\CC)$ of points
\[
\{\PP(\wedge^{2n}W_{T',\CC}) \ : \ T'\in\T_K\}.
\]
It follows that the $0$-dimensional subscheme $\{\ell_{W_{T',\CC}}\otimes\ell_{W_{T',\CC}}  \ : \ T'\in\T_K\}$
of $H^{ev}(X,\CC)\otimes H^{ev}(X,\CC)$, consisting of tensor squares of pure spinors, is defined over $\QQ$. 

We need to work with $K$-valued characters of the subgroup $\Spin(V_{\hat{\eta},K})_{\ell_\bullet}$ of $\Spin(V_{\hat{\eta},K})$ fixing every point the subset $\{\ell_{W_{T',\CC}}  \ : \ T'\in\T_K\}$. (???)
%It suffices to show that $\ell_{W_{T',\CC}}$ is the character $\prod_{\sigma\in T'}(\sigma)^{d/2}$. (???)

Let $t\in K$ be a primitive element satisfying $t\iota(t)=1$, as in Remark \ref{rem-isometry-as-primitive-element}.
There exists an element $\tilde{t}\in \Spin(V_{\hat{\eta},K})_\QQ$, such that $\rho_{\tilde{t}}=\eta(t)$, by Lemma \ref{lemma-units-are-isometries}.
The subspaces $V_{\sigma,\CC}$, $\sigma\in\Sigma$, of $V_\CC$, given in (\ref{eq-V-sigma-CC}), are the the eigenspaces of $\rho_{\tilde{t}}$ with eigenvalues 
$\sigma(t)$, since $t$ is a primitive element of $K$.
(???)
\EndProof
\end{rem}
%**********
% End Hide
%**********
}

%*******************************************************************************************************************
%
%*******************************************************************************************************************
\subsubsection{An isotypic decomposition of the $\Spin(V_\QQ)_\eta$-invariant space $B\otimes B$}
\label{sec-isotypic-decomposition}
Let 
\[
\Spin(V_\QQ)_\eta
\]
be the subgroup of the group $\Spin(V_\QQ)_{\hat{\eta}}$, given in (\ref{eq-Spin-V-QQ-hat-eta}),
leaving invariant the subspace $V_\sigma$, for all $\sigma\in\Sigma$.
%, commuting with all elements of the subspace $\eta(K)$ of $\End_\QQ(V_\QQ)$. 
Let $\Spin(V)_\eta$ be the analogous subgroup of $\Spin(V)_{\hat{\eta}}$. We will see that $\Spin(V)_\eta/\{\pm 1\}$ is an arithmetic subgroup of a product of $e/2$ copies of the unitary group $U(d,d)$ (see Equation 
(\ref{eq-factorization-of-Spin-V-tilde-K-eta-mod-pm1}) and Lemma \ref{lemma-H-t-is-hermitian} below).

Set $V_{\tilde{K}}:=V_\QQ\otimes_\QQ\tilde{K}$. Let  $\Spin(V_{\tilde{K}})_{\{g_{\hat{\sigma}}\}}$ be the subgroup of $\Spin(V_{\tilde{K}})$ consisting of elements $g$ commuting with $g_{\hat{\sigma}}$, given in (\ref{eq-g-hat-sigma}), for all $\hat{\sigma}\in\hat{\Sigma}$. Let 
$\Spin(V_{\tilde{K}})_{\hat{\eta}}$ be the commutator subgroup of $\Spin(V_{\tilde{K}})_{\{g_{\hat{\sigma}}\}}$.
Let  
\[
\Spin(V_{\tilde{K}})_\eta
\] 
be the subgroup of $\Spin(V_{\tilde{K}})_{\hat{\eta}}$ consisting of elements $g$ 
%be the subgroup of $\Spin(V_{\tilde{K}})_{\hat{\eta}}$ consisting of elements $g$ 
%commuting with $g_{\hat{\sigma}}$, given in (\ref{eq-g-hat-sigma}), for all $\hat{\sigma}\in\hat{\Sigma}$, such that $\rho_g$ 
%restricts to $V_{\sigma}\oplus V_{\bar{\sigma}}$ as an element of $SO_+(V_{\sigma}\oplus V_{\bar{\sigma}})$, for all $\sigma\in\Sigma$, and 
such that $V_\sigma$ is $\rho_g$-invariant, for all $\sigma\in\Sigma$. Note that $\Spin(V_\QQ)_\eta$ is contained in the $Gal(\tilde{K}/\QQ)$-invariant subgroup of $\Spin(V_{\tilde{K}})_\eta$, where the $Gal(\tilde{K}/\QQ)$-action is induced by that on $C(V_{\tilde{K}})\cong C(V_\QQ)\otimes_\QQ\tilde{K}$ via the action on the second tensor factor.

Let 
\[
\Spin(V)_{\eta,B}
\]
be the subgroup of $\Spin(V)_{\eta}$ fixing every point of the subspace $B$ of $H^{ev}(X,\QQ)$. Denote by
$\Spin(V_{\tilde{K}})_{\eta,B}$ the analogous subgroup of $\Spin(V_{\tilde{K}})_\eta$.
When $F=\QQ$, the group $\Spin(V_{\tilde{K}})_\eta$ specializes to the group $\Spin(V_K)_{\ell_1,\ell_2}$ in (\ref{eq-spin-V-K-ell-1-ell-2})
and the group $\Spin(V)_{\eta,B}$ specializes to $\Spin(V)_P$, given in (\ref{eq-Spin(V)_P}).

The group $\Spin(V_{\tilde{K}})_\eta$ leaves invariant each of the subspaces $W_T$, $T\in \T_K$, by Equation (\ref{eq-decomposition-of-W-T-CC}).
Hence, each of the lines $\ell_T$ in $H^{ev}(X,\tilde{K})$ is $\Spin(V_{\tilde{K}})_\eta$-invariant. Thus, the commutator subgroup of $\Spin(V_{\tilde{K}})_\eta$
is contained in $\Spin(V_{\tilde{K}})_{\eta,B}$. 

\[
\xymatrix{
\Spin(V)_{\hat{\eta}} \ar[r]^\subset & \Spin(V_\QQ)_{\hat{\eta}} &
\Spin(V_{\hat{\eta}})\ar[l]_{(\ref{eq-homomorphism-from-Spin-V-hat-eta-to-commutator-in-Spin-V-QQ})} \ar[r]^-\subset&
\Spin(V_{\hat{\eta},K})_\QQ \ar[r]^-\subset & \Spin(V_{\hat{\eta}}\otimes_FK)
\\
\Spin(V)_\eta \ar[r]^\subset \ar[u]^\cup & \Spin(V_\QQ)_\eta \ar[r]^\subset \ar[u]^\cup & \Spin(V_{\tilde{K}})_\eta 
\\
\Spin(V)_{\eta,B}\ar[r]_\subset \ar[u]^\cup & 
\Spin(V_\QQ)_{\eta,B}\ar[r]_\subset\ar[u]^\cup &
\Spin(V_{\tilde{K}})_{\eta,B}\ar[u]^\cup.
}
\]
%\begin{question}
%Formulate the analogue of Lemma \ref{lemma-Spin-V-hat-eta} relating $\Spin(V_{\hat{\eta},K})_\QQ$, given in (\ref{eq-Spin-V-Q-hat-eta-K}), 
%to $\Spin(V_\QQ)_{\eta,B}$. The latter should preserve also a $K$-valued hermitian form. We know that $\Spin(V_{\hat{\eta},K})_\QQ$ 
%commutes with the scalar multiplication action of $K$ on $V_\QQ\otimes_F K$, but this action is different from the $\eta(K)$ action. 
%The $\QQ$-isomorphism from $W$ to $V_\QQ$, sending $w\in W$ to $w+\iota(w)$, is equivariant with respect to the $\eta(K)$-action on both.
%\end{question}
Note that $W_T\cap W_{T'}=V_\sigma$, if $\sigma$ is the only common value of $T$ and $T'$. Hence,
$\Spin(V)_{\eta,B}$ is equal to the subgroup of $\Spin(V)_{\hat{\eta}}$ fixing every point of $B$.

For each $\sigma\in\Sigma$ let $\Spin(V_\sigma\oplus V_{\bar{\sigma}})_\eta$ be the subgroup of $\Spin(V_\sigma\oplus V_{\bar{\sigma}})$
leaving each of $V_\sigma$ and $V_{\bar{\sigma}}$ invariant.
We have a natural homomorphism 
%from the product
%$
%\prod_{\hat{\sigma}\in\hat{\Sigma}}\Spin(V_\sigma\oplus V_{\bar{\sigma}})_{\eta,P_{\hat{\sigma}}}
%$
%to $\Spin(V_{\tilde{K}})_{\eta,B}$. 
\begin{equation}
\label{eq-factorization-of-Spin-V-tilde-K-B}
\prod_{\hat{\sigma}\in\hat{\Sigma}}\Spin(V_\sigma\oplus V_{\bar{\sigma}})_{\eta,P_{\hat{\sigma}}}
\rightarrow \Spin(V_{\tilde{K}})_{\eta,B}
\end{equation}
induced by the isomorphism
\[
C(V_{\tilde{K}})\cong\otimes_{\hat{\sigma}\in\hat{\Sigma}}C(V_\sigma\oplus V_{\bar{\sigma}}),
\] 
where the tensor product is in the category of $\ZZ_2$-graded algebras, as in Diagram (\ref{eq-diagram-of-spin-groups}). Above $\sigma$ restricts to $F$ as $\hat{\sigma}$. 
Now, $\Spin(V_\sigma\oplus V_{\bar{\sigma}})_{\eta,P_{\hat{\sigma}}}$ is isomorphic to the group $SL(V_\sigma)$ of linear transformations of $V_\sigma$, as a $d$-dimensional vector space over $\tilde{K}$, of determinant $1$. An elements $g$ of $SL(V_\sigma)$ acts on $V_{\bar{\sigma}}$ by $(g^*)^{-1}$, under the identification of $V_{\bar{\sigma}}$ with $V_\sigma^*$ via the bilinear pairing $(\bullet,\bullet)_V$. The proof is identical to that of Lemma 
\ref{lemma-Spin-V-K-is-SL-n-K}.
%\ref{lemma-stabilizer-is-isomorphic-to-so-f}.

Similarly, we have the homomorphism  
\[
\prod_{\hat{\sigma}\in\hat{\Sigma}}\Spin(V_\sigma\oplus V_{\bar{\sigma}})_\eta
\rightarrow
\Spin(V_{\tilde{K}})_\eta ,
\]
where $\Spin(V_\sigma\oplus V_{\bar{\sigma}})_\eta/\{\pm 1\}$ is isomorphic to $GL(V_\sigma)$, by the proof of \cite[Lemma 1]{igusa}. An elements $g$ of $GL(V_\sigma)$ acts on $V_{\bar{\sigma}}$ by $(g^*)^{-1}$, under the identification of $V_{\bar{\sigma}}$ with $V_\sigma^*$ via the bilinear pairing $(\bullet,\bullet)_V$.

\begin{equation}
\label{eq-diagram-of-Spin-groups-over-K-tilde}
\xymatrix{
\prod_{\hat{\sigma}\in\hat{\Sigma}}C(V_\sigma\oplus V_{\bar{\sigma}})^{even,\times} \ar[r] &
\left[\otimes_{\hat{\sigma}\in\hat{\Sigma}}(V_\sigma\oplus V_{\bar{\sigma}})\right]^{even,\times}
\ar[r]^-\cong &
C(V_{\tilde{K}})^{even,\times}
\\
\prod_{\hat{\sigma}\in\hat{\Sigma}}\Spin(V_\sigma\oplus V_{\bar{\sigma}})_\eta 
\ar[rr] \ar[u]^-\cup
& &
\Spin(V_{\tilde{K}})_\eta
\ar[u]^{\cup}.
}
\end{equation}

%************
% Hide
%************
\hide{
We have the left exact sequence
\[
1\rightarrow \Spin(V_{\tilde{K}})_\eta \rightarrow \prod_{\hat{\sigma}\in\hat{\Sigma}}G(V_\sigma\oplus V_{\bar{\sigma}})_\eta\RightArrowOf{j} \tilde{K}^\times\times \ZZ/2\ZZ,
\]
where 
\[
j((g_{\hat{\sigma}})_{\hat{\sigma}\in\hat{\Sigma}})=\left(\prod_{\hat{\sigma}\in\hat{\Sigma}}N(g_{\hat{\sigma}}),\sum_{\hat{\sigma}\in\hat{\Sigma}}p(g_{\hat{\sigma}})\right),
\]
$N:G(V_\sigma\oplus V_{\bar{\sigma}})\rightarrow\tilde{K}^\times$ is the norm\footnote{
The norm character $N:G(V_{\tilde{K}})\rightarrow \tilde{K}$ is defined by 
$N(g)=g\tau(g)$ as an element of the center $\tilde{K}$ of the Clifford algebra $C(V_{\tilde{K}})$. 
If $g=v_1\cdots v_k$, for $v_i\in V_{\tilde{K}}$, then $N(g)=\prod_{i=1}^k\frac{(v_i,v_i)}{2}$. See \cite[Sec. II.2.3]{chevalley}.
}
character,
and $p:G(V_\sigma\oplus V_{\bar{\sigma}})\rightarrow \ZZ/2\ZZ$ is the parity character. Denote by 
\[
N_{\hat{\sigma}}:\Spin(V_{\tilde{K}})_\eta\rightarrow \tilde{K}^\times \ \mbox{and} \ 
p_{\hat{\sigma}}:\Spin(V_{\tilde{K}})_\eta\rightarrow \ZZ/2\ZZ
\] 
the restriction of the norm and parity characters of the factor $G(V_\sigma\oplus V_{\bar{\sigma}})_\eta$. 
Let the character $\det_\sigma$ be the composition  
\[
\Spin(V_{\tilde{K}})_\eta \rightarrow G(V_\sigma\oplus V_{\bar{\sigma}})_\eta\rightarrow O(V_\sigma\oplus V_{\bar{\sigma}})_\eta\rightarrow GL(V_\sigma)\RightArrowOf{\det} \tilde{K}^\times.
\]
%************
% End Hide
%************
}
The bottom horizontal homomorphism in (\ref{eq-diagram-of-Spin-groups-over-K-tilde}) induces the isomorphism
\[
%\begin{equation}
%\label{eq-factorization-of-Spin-V-tilde-K-eta-mod-pm1}
\prod_{\hat{\sigma}\in\hat{\Sigma}}\left[\Spin(V_\sigma\oplus V_{\bar{\sigma}})_\eta /\{\pm 1\}\right]
\rightarrow
\Spin(V_{\tilde{K}})_\eta/\{\pm 1\}.
\]
%\end{equation}
We get the isomorphism
\begin{equation}
\label{eq-factorization-of-Spin-V-tilde-K-eta-mod-pm1}
\Spin(V_{\tilde{K}})_\eta/\{\pm 1\}\cong \prod_{\hat{\sigma}\in\hat{\Sigma}}GL(V_\sigma),
\end{equation}
where each $\sigma$ is chosen to restrict to $F$ as $\hat{\sigma}$.

Given $\sigma_0\in\Sigma$, 
let the character $\det_{\sigma_0}$ be the composition  
\[
\Spin(V_{\tilde{K}})_\eta \rightarrow \Spin(V_{\tilde{K}})_\eta/\{\pm 1\} 
%\RightArrowOf{(\ref{eq-factorization-of-Spin-V-tilde-K-eta-mod-pm1})^{-1}} 
%\prod_{\hat{\sigma}\in\hat{\Sigma}}\left[\Spin(V_\sigma\oplus V_{\bar{\sigma}})_\eta /\{\pm 1\}\right]
\IsomRightArrow
\prod_{\hat{\sigma}\in\hat{\Sigma}}GL(V_\sigma)
\rightarrow GL(V_{\sigma_0}) \RightArrowOf{\det} \tilde{K}^\times,
\]
where 
%the two direct products above involve a choice of $\sigma\in\Sigma$ restricting to $\hat{\sigma}$ and 
the third arrow is the projection onto the factor associated to $\sigma_0$.
The character $\ell_T$ of $\Spin(V_{\tilde{K}})_\eta$ satisfies 
\[
\ell_T\otimes\ell_T\cong \bigotimes_{\hat{\sigma}\in\hat{\Sigma}}\mbox{det}_{T(\hat{\sigma})}
\]
More generally, we have

\begin{lem}
\label{lemma-tensor-product-of-two-ells}
${\displaystyle \ell_{T_1}\otimes\ell_{T_2}\cong \bigotimes_{\{\hat{\sigma} \ : \ T_1(\hat{\sigma})=T_2(\hat{\sigma})\}}\mbox{det}_{T_1(\hat{\sigma})}.
}$
\end{lem}

\begin{proof}
The statement follows from the factorization (\ref{eq-tensor-factorization-of-ell-T}) and the fact that $\det_{\bar{\sigma}}\otimes\det_\sigma$ is the trivial character.
\end{proof}

\begin{cor}
\label{cor-ell-T-are-linearly-independent}
The lines $\ell_T$, $T\in \T_K$, in $S^+_{\tilde{K}}$ are linearly independent over $\tilde{K}$.
\end{cor}

\begin{proof}
The group $\Spin(V_{\tilde{K}})_\eta$ leaves $B\otimes_\QQ\tilde{K}$ invariant and acts on the lines $\{\ell_T \ : \ T\in\T_K\}$ as characters. These characters are distinct, by 
Lemma \ref{lemma-tensor-product-of-two-ells}. 
\end{proof}

%\begin{question}
%Relate $N_{\hat{\sigma}}$ and $p_{\hat{\sigma}}$ to the pullback of the determinant character $\det_\sigma$.
%\end{question}

The image $m_g$, of $g\in \Spin(V)_{\eta,B}$ via the even half spin representation, leaves invariant each pure spinor $\ell_{T}$.
Given an element $g$ of $\Spin(V)_{\eta,B}$, its image $\rho_g$ via the vector representation
%even half spin representation 
leaves each of $W_T$, $T\in\T_K$,  invariant and acts on each with trivial determinant. Furthermore, $\rho_g$ commutes with $\hat{\eta}(F)$ and thus leaves invariant each of the subspaces $V_{\hat{\sigma},\RR}$ of $V_\RR$, for all $\hat{\sigma}\in\hat{\Sigma}$. Hence, $\rho_g$ leaves invariant each of the subspaces $V_{\sigma,\CC}$, for all $\sigma\in \Sigma$. 
%Consequently, $\rho_g$ leaves invariant each $W_{T'}$, $T'\in\T_K$. 
Denote by $\rho_{g,\sigma}$ the restriction of $\rho_g$ to $V_{\sigma,\CC}$.  Note that $\det_\sigma(g)=\det(\rho_{g,\sigma})$.

\begin{lem}
For $g\in\Spin(V_\QQ)_{\eta,B}$ we have $\det(\rho_{g,\sigma})=1,$ for all $\sigma\in\Sigma$.
\end{lem}

\begin{proof}
Fix $\sigma_0\in\Sigma$ restricting to $F$ as $\hat{\sigma}_0\in\hat{\Sigma}$. 
Choose CM-types $T_1$ and $T_2$, such that $T_1(\hat{\sigma}_0)=T_2(\hat{\sigma}_0)$ and $T_1(\hat{\sigma})\neq T_2(\hat{\sigma})$, for $\hat{\sigma}\neq\hat{\sigma}_0$. 
Then $\ell_{T_1}\otimes\ell_{T_2}$ is the character $\det_{\sigma_0}$, by Lemma \ref{lemma-tensor-product-of-two-ells}. Thus, $\det_{\sigma_0}(g)=1$, for $g\in \Spin(V_\QQ)_{\eta,B}$, 
since $m_g$ fixes every point in $B$ and so $m_g\otimes m^\dagger_g$ (???) fixes every point in $B\otimes B$.
\end{proof}

%The following is an analogue of Lemma \ref{lemma-centralizer-of-rho-Spin-V-P}. 

%\begin{lem}
%The centralizer of $\rho(\Spin(V_\QQ)_{\eta,B})$ in $\tilde{O}(V_\QQ)$ is $\eta(K^\times)$.
%\end{lem}

%*******************************************************************************************************************
%
%*******************************************************************************************************************
\subsubsection{A $\Spin(V)_{\eta,B}$-invariant hermitian form}
Let $t$ be a non-zero element of the $(-1)$-eigenspace $K_-$ of $\iota:K\rightarrow K$. 
%************
% Hide
%************
\hide{
Choose a CM-type $T$. 
Consider the $F\otimes_\QQ\tilde{K}$-valued form $H_{t,T}:V_\QQ\times V_\QQ\rightarrow F\otimes_\QQ\tilde{K}\cong \prod_{\hat{\sigma}\in\hat{\Sigma}}\tilde{K}$ whose $\hat{\sigma}$ component is given by
\[
H_{t,T}(x_{\hat{\sigma}},y_{\hat{\sigma}}):= \hat{\sigma}(t\iota(t))(x_{\hat{\sigma}},y_{\hat{\sigma}})_V+\sigma(t)((\eta(t)x)_{\hat{\sigma}},y_{\hat{\sigma}})_V,
\]
where $\sigma=T(\hat{\sigma}):K\rightarrow \tilde{K}$, 
%restricts to $F$ as $\hat{\sigma}$, 
$x=\sum_{\hat{\sigma}\in\hat{\Sigma}}x_{\hat{\sigma}}$ with $x_{\hat{\sigma}}\in V_{\hat{\sigma}}$, and $y_{\hat{\sigma}}$ is defined similarly.
We have $H_{t,T}(x,y)=\iota(H_{t,T}(y,x))$, by Corollary \ref{cor-Neron-Severi-group}. 
%The choices of $\sigma$ above amount to a choice of a CM-type on which $H_t$ depends.

\begin{lem}
The form $H_{t,T}$ is hermitian,  i.e., 
\begin{equation}
\label{eq-K-linearity-in-second-variable}
H_{t,T}(x_{\hat{\sigma}},\eta(\lambda)y_{\hat{\sigma}})=\sigma(\lambda)H_{t,T}(x_{\hat{\sigma}},y_{\hat{\sigma}}), 
\end{equation}
for $\lambda\in K$, where $\sigma=T(\hat{\sigma})$. Furthermore, $H_{t,T}$ is $\Spin(V_\QQ)_{\eta,B}$-invariant.
\end{lem}

\begin{proof}
The form $H_{t,T}$ is $\Spin(V_\QQ)_{\eta,B}$-invariant, since $(\bullet,\bullet)_V$ is and $\eta(t)\rho_g(x)=\rho_g(\eta(t)x)$, by definition of $\Spin(V_\QQ)_{\eta,B}$.
We prove the $K$-linearity (\ref{eq-K-linearity-in-second-variable}) next.
Write $\lambda=a+bt$ with $a, b\in F$. The form $H_{t,T}$ on $V_{\hat{\sigma}}$ is $F$-bilinear, i.e., Equation (\ref{eq-K-linearity-in-second-variable}) holds for $\lambda\in F$. Hence, it suffices to prove the case $\lambda=t$. We have
\begin{eqnarray*}
H_{t,T}(x_{\hat{\sigma}},\eta(t)y_{\hat{\sigma}})&=&-\hat{\sigma}(t^2)(x_{\hat{\sigma}},\eta(t)y_{\hat{\sigma}})+
\sigma(t)(\eta(t)x_{\hat{\sigma}},\eta(t)y_{\hat{\sigma}})
\\
&\stackrel{\mbox{Cor. \ref{cor-Neron-Severi-group}}}{=}&
-\sigma(t)^2(x_{\hat{\sigma}},\eta(t)y_{\hat{\sigma}})+\sigma(t)(-\hat{\sigma}(t^2)x_{\hat{\sigma}},y_{\hat{\sigma}})
\\
&=&\sigma(t)\left[
\hat{\sigma}(-t^2)(x_{\hat{\sigma}},y_{\hat{\sigma}})-\sigma(t)(x_{\hat{\sigma}},\eta(t)y_{\hat{\sigma}})
\right]
\\
&\stackrel{\mbox{Cor. \ref{cor-Neron-Severi-group}}}{=}& \sigma(t)\left[
\hat{\sigma}(-t^2)(x_{\hat{\sigma}},y_{\hat{\sigma}})+\sigma(t)(\eta(t)x_{\hat{\sigma}},y_{\hat{\sigma}})
\right]
=\sigma(t)H_{t,T}(x_{\hat{\sigma}},y_{\hat{\sigma}}).
\end{eqnarray*}
%The image $\rho_g$, of  $g\in \Spin(V)_{\eta,B}$, restricts to $V_\sigma\oplus V_{\bar{\sigma}}$ as the image of an $\iota$-invariant element of 
%$\Spin(V_\sigma\oplus V_{\bar{\sigma}})_{\eta,P_{\hat{\sigma}}}.$ The latter preserves the restriction of $H_t$ to $V_\sigma\oplus V_{\bar{\sigma}}$,
%by the proof of Lemma \ref{lemma-su-3-3}.
\end{proof}

We saw in Section \ref{sec-Galois-action-on-set-of-CM-types} that the CM-type $T$ determines and embedding $T:K\rightarrow F\otimes_\QQ\tilde{K}$.
Equation (\ref{eq-K-linearity-in-second-variable}) becomes
\[
H_{t,T}(x,\eta(\lambda)y)_V=T(\lambda)H_{t,T}(x,y)_V.
\]
%************
% End Hide
%************
}

Let $H_t:V_\QQ\times V_\QQ\rightarrow K$ be given by
\begin{equation}
\label{eq-H-t}
H_t(x,y) :=
(-t^2)(x,y)_{V_{\hat{\eta}}}+t(\eta(t)x,y)_{V_{\hat{\eta}}}.
\end{equation}

\begin{lem}
\label{lemma-H-t-is-hermitian}
The form $H_t$ is hermitian, i.e., it satisfies $H_t(x,y)=\iota(H_t(y,x))$ and $H_t(x,\eta(\lambda)y)=\lambda H_t(x,y)$. Furthermore, $H_t$ is $\Spin(V_\QQ)_\eta$-invariant.
%There exists a unique $K$-valued hermitian form $H_t:V_\QQ\times V_\QQ\rightarrow K$, such that $H_{t,T}=T\circ H_t$.
%Furthermore, $H_t(x,y)=\iota(H_t(y,x))$ and $H_t(x,\eta(\lambda)y)=\lambda H_t(x,y)$.
\end{lem}

\begin{proof}
The equality $H_t(x,y)=\iota(H_t(y,x))$ follows from Lemma \ref{lemma-units-are-isometries}. The form $H_t$ is $F$-bilinear. Hence, the equality 
$H_t(x,\eta(\lambda)y)=\lambda H_t(x,y)$, for all $\lambda\in K$, would follow from the case $\lambda=t$. We have
\begin{eqnarray*}
H_t(x,\eta(t)y)&=& -t^2(x,\eta(t)y)_{V_{\hat{\eta}}}+t(\eta(t)x,\eta(t)y)_{V_{\hat{\eta}}}
\\
&\stackrel{\mbox{Lemma \ref{lemma-units-are-isometries}}}{=}&
-t^2(x,\eta(t)y)_{V_{\hat{\eta}}}+t(\hat{\eta}(-t^2)x,y)_{V_{\hat{\eta}}}
\\
&=&-t^2(x,\eta(t)y)_{V_{\hat{\eta}}}+-t^3(x,y)_{V_{\hat{\eta}}}
\\
&\stackrel{\mbox{Lemma \ref{lemma-units-are-isometries}}}{=}&
t\left[
-t^2(x,y)_{V_{\hat{\eta}}}+t(\eta(t)x,y)_{V_{\hat{\eta}}}
\right].
\end{eqnarray*}

The pairing $(\bullet,\bullet)_{V_{\hat{\eta}}}$ is $\Spin(V_{\hat{\eta}})$-invariant. It follows that it is also $\Spin(V_\QQ)_{\hat{\eta}}$-invariant, by Lemma \ref{lemma-Spin-V-hat-eta}, which implies that the two groups have the same Zariski closure in $GL(V_\QQ)$. Now $\Spin(V_\QQ)_\eta$ is a subgroup of $\Spin(V_\QQ)_{\hat{\eta}}$, by definition.
The form $H_t$ is $\Spin(V_\QQ)_\eta$-invariant, since $(\bullet,\bullet)_{V_{\hat{\eta}}}$ is and $\eta(t)\rho_g(x)=\rho_g(\eta(t)x)$, by definition of $\Spin(V_\QQ)_\eta$.
\end{proof}

%************
% Hide
%************
\hide{
\begin{proof}
It suffices to prove the equality $H_{t,T}=T\circ H_t$.
The embedding $T:K\rightarrow F\otimes_\QQ\tilde{K}$ restricts to $F$ as the standard embedding of $F$ in $F\otimes_\QQ\tilde{F}$. 
%In particular, given $x, y\in V_\QQ$, 
%$(x,y)_{V_{\hat{\eta}}}$ belongs to $T(F)$ and so does $(\eta(t)x,y)_{V_{\hat{\eta}}}$. 
Now,
\begin{eqnarray*}
H_{t,T}(x,y)&=&T(-t^2)\left[T\left((x,y)_{V_{\hat{\eta}}}\right)+T(t)T\left((\eta(t)x,y)_{V_{\hat{\eta}}}\right)
\right]
\\
&=&T\left[
(-t^2)(x,y)_{V_{\hat{\eta}}}+t(\eta(t)x,y)_{V_{\hat{\eta}}}
\right]=T(H_t(x,y)).
\end{eqnarray*}
%and so it has values in $T(K)$.
\end{proof}
%************
% End Hide
%************
}

%*******************************************************************************************************************
%
%*******************************************************************************************************************
\subsubsection{If $B$ is spanned by Hodge classes then $HW$ is as well}
%The special Mumford-Tate group of the generic deformation of $(X\times\hat{X},\eta)$  contains the Zariski closure in $SO(V_\QQ)$
%of the image via $\rho$ of subgroup 
%$\Spin(V_\QQ)_{\eta,B}$ 
%of $\Spin(V_\QQ)$. This follows from the following Lemma.
In this section we show that if the secant space $B$ is spanned by Hodge classes, then condition 
(\ref{eq-condition-for-HW-to-consist-of-Hodge-classes}) is satisfied for the $K$-action $\eta$ on 
$H^1(X\times\hat{X},\QQ)$, and so the subspace $HW$ of $H^d(X\times\hat{X},\QQ)$ consists of Hodge classes.

\begin{lem}
Assume that $B$ is spanned by Hodge classes. Then the complex structure $I_{X\times\hat{X}}:H^1(X\times\hat{X},\RR)\rightarrow H^1(X\times\hat{X},\RR)$
of $X\times\hat{X}$ commutes with the action of $\eta(K)$. 
Furthermore, 
%the subspace $B$ is spanned by Hodge classes, then 
\begin{equation}
\label{eq-dim-V-sigma-0-1=dim-V-sigma-1-0}
\dim(V_{\sigma,\CC}^{0,1})=\dim(V_{\sigma,\CC}^{1,0}),
\end{equation}
for all $\sigma\in\Sigma$, and consequently $HW$ is spanned by Hodge classes.
\end{lem}

\begin{proof}
Set $I:=I_{X\times\hat{X}}$. Let $T\in\T_K$ be a CM-type. We get the two complex conjugate lines $\ell_T$ and $\ell_{\iota(T)}$ in $B$, each spanned by a Hodge class. 
The proof of Lemma \ref{lemma-decomposition-into-4-direct-summands} establishes that each of $W_{T,\CC}$ and $\bar{W}_{T,\CC}=W_{\iota(T),\CC}$ is $I$ invariant and the dimensions of the $I$-eigenspaces in each are equal,
\begin{eqnarray}
\label{eq-W-1-0=W-0-1}
\dim(W^{1,0}_{T,\CC})&=&\dim(W^{0,1}_{T,\CC})=2n
\\
\dim(\bar{W}^{1,0}_{T,\CC})&=&\dim(\bar{W}^{0,1}_{T,\CC})=2n.
\nonumber
\end{eqnarray} 
The endomorphisms in $\eta(F)$ commute with $I$, by assumption. Hence, the subspace $V_{\hat{\sigma},\RR}:=H^1(X\times\hat{X},\RR)_{\hat{\sigma}}$ is $I$-invariant, for all $\hat{\sigma}$ in $\hat{\Sigma}$. We get the equality
\begin{equation}
\label{eq-dim-V-hat-sigma-1-0-is-d}
\dim(V_{\hat{\sigma},\CC}^{1,0})=\dim(V_{\hat{\sigma},\CC}^{0,1})=d,
\end{equation}
for all $\hat{\sigma}$ in $\hat{\Sigma}$.

The following equalities hold for all $\sigma\in T$
\begin{eqnarray*}
V_{\sigma,\CC} & = & [V_{\hat{\sigma},\RR}\otimes_\RR\CC]\cap W_{T,\CC},
\\
V_{\bar{\sigma},\CC} & = & [V_{\hat{\sigma},\RR}\otimes_\RR\CC]\cap \bar{W}_{T,\CC},
\end{eqnarray*}
by construction. Hence, the $\eta(K)$ eigenspaces $V_{\sigma,\CC}$ are $I$-invariant, for all $\sigma\in\Sigma$.

We know the equality
\begin{equation}
\label{eq-sum-of-dimensions}
\sum_{\sigma\in T}\dim(V_{\sigma,\CC}^{0,1})=\sum_{\sigma\in T}\dim(V_{\sigma,\CC}^{1,0})=2n,
\end{equation}
by equality (\ref{eq-W-1-0=W-0-1}).
%If $B$ is spanned by Hodge classes, then the subspace $P_{T'}$ of $H^{ev}(X,\RR)$ is contained in $H^{1,1}(X,\RR)$. 
%The proof of Lemma \ref{lemma-decomposition-into-4-direct-summands} implies equation (\ref{eq-W-1-0=W-0-1}) and hence 
%equation (\ref{eq-sum-of-dimensions}) for all $T'\in\T_K$. 
We prove Equation (\ref{eq-dim-V-sigma-0-1=dim-V-sigma-1-0}) by contradiction. 
Assume that $\dim(V^{1,0}_{\sigma_0,\CC})>\dim(V^{0,1}_{\sigma_0,\CC})$, for some $\sigma_0\in\Sigma$.  Then $\dim(V^{1,0}_{\bar{\sigma}_0,\CC})<\dim(V^{0,1}_{\bar{\sigma}_0,\CC})$, 
and so $\dim(V^{1,0}_{\sigma_0,\CC})>d/2>\dim(V^{1,0}_{\bar{\sigma}_0,\CC})$,
by (\ref{eq-dim-V-hat-sigma-1-0-is-d}). Thus, changing the value of a CM-type $T$ at $\sigma_0$ changes the sum in (\ref{eq-sum-of-dimensions}). This contradicts the fact that the sum is $2n$, for all CM-types $T$.
Equation (\ref{eq-dim-V-sigma-0-1=dim-V-sigma-1-0}) thus hold for all $\sigma\in\Sigma$.
The equalities (\ref{eq-dim-V-sigma-0-1=dim-V-sigma-1-0}) hold for all $\sigma\in\Sigma$, if and only if $HW$ consists of rational Hodge classes, by \cite[Prop. 4.4]{deligne-milne}.
\end{proof}

%*******************************************************************************************************************
%
%*******************************************************************************************************************
\subsection{The $\Spin(V)_{\eta,B}$-invariant subalgebra}
Let $\A:=(\wedge^*V_\QQ)^{\Spin(V)_{\eta,B}}$ be the subalgebra of $(\wedge^*V_\QQ)$
of $\Spin(V)_{\eta,B}$-invariant classes with respect to the $\rho$-action given in (\ref{eq-rho-extended-to-exterior-algebra}). 
Note the inclusion $\tilde{\phi}(B\otimes B)\subset \A$, where $\tilde{\phi}$ is given in (\ref{eq-tilde-phi}).
By $\wedge^*V_{\tilde{F}}$ we will refer to $\wedge^*_{\tilde{F}}V_{\tilde{F}}$, where $V_{\tilde{F}}=V_\QQ\otimes_\QQ\tilde{F}$.

\begin{prop}
\label{prop-generator-for-the-invariant-subalgebra}
\begin{enumerate}
\item
\label{lemma-item-Spin-V-eta-B-invariant-2-forms}
The subspace $(\wedge^2V_\QQ)^{\Spin(V)_{\eta,B}}$ is equal to the image of the injective homomorphism $\Xi:K_-\rightarrow  \wedge^2V_\QQ$ of Corollary (\ref{cor-Neron-Severi-group}). This subspace is thus $(e/2)$-dimensional over $\QQ$ and is naturally a $1$-dimensional  vector space over $F$, and it
consists of $\Spin(V)_\eta$-invariant classes. 
\item
\label{lemma-item-generator-for-the-invariant-subalgebra}
The subalgebra $\A$ is abelian and is generated by the $(e/2)$-dimensional $\QQ$-vector space 
$(\wedge^2V_\QQ)^{\Spin(V)_{\eta,B}}$
and the $e$-dimensional $\QQ$-vector space 
$
HW=\wedge^d_{\eta(K)}V_\QQ.
$
\item
\label{lemma-item-decomposition-of-HW-into-characters}
The subspace  $HW\otimes_\QQ\tilde{K}$ decomposes as a direct sum of the characters $\det_\sigma$, $\sigma\in\Sigma$, of $\Spin(V_{\tilde{K}})_\eta$, each with multiplicity $1$.
\item
\label{lemma-item-generatyors-for-the-invariant-subalgebra-R}
The subalgebra $\R:=(\wedge^*V_\QQ)^{\Spin(V)_\eta}$ is generated by $(\wedge^2V_\QQ)^{\Spin(V)_{\eta,B}}$.
\end{enumerate}
\end{prop}

%Note that the algebra $(\wedge^*V_\QQ)^{\Spin(V)_{\eta,B}}$ need not be generated by $(\wedge^2V_\QQ)^{\Spin(V)_{\eta,B}}$ and $HW$, 
%since the $Gal(\tilde{F}/\QQ)$-invariant subalgebra of $(\wedge^*V_{\tilde{F}})^{\Spin(V)_{\eta,B}}$ may be larger than the subalgebra 
%generated by the $Gal(\tilde{F}/\QQ)$-invariant subspaces of the generators 
%in part (\ref{lemma-item-generator-for-the-invariant-subalgebra}) above.

\begin{proof}
(\ref{lemma-item-generator-for-the-invariant-subalgebra}) \underline{Step 1:}
We prove first that the subalgebra $(\wedge^*V_{\tilde{F}})^{\Spin(V)_{\eta,B}}$ is abelian and is generated by the $(e/2)$-dimensional $\tilde{F}$-vector space 
$(\wedge^2V_\QQ)^{\Spin(V)_{\eta,B}}\otimes_\QQ{\tilde{F}}$
and the $e$-dimensional $\tilde{F}$-vector space 
$
HW\otimes_\QQ{\tilde{F}}.
%=(\wedge^dV_{\tilde{F}})^{\Spin(V)_{\eta,B}}.
$

We have the isomorphism
\[
\wedge^*V_{\tilde{F}}\cong \bigotimes_{\hat{\sigma}\in\hat{\Sigma}}
\wedge^*V_{\hat{\sigma}}.
\]
The $\Spin(V)_{\eta,B}$-invariant subalgebra is equal to the $\Spin(V_{\tilde{F}})_{\eta,B}$-invariant subalgebra. Using the factorization 
(\ref{eq-factorization-of-Spin-V-tilde-K-B}) of $\Spin(V_{\tilde{F}})_{\eta,B}$ we get the isomorphism
\begin{equation}
\label{eq-factorization-of-Spin-V-B-invariant-subalgebra-over-tilde-K}
(\wedge^*V_{\tilde{F}})^{\Spin(V)_{\eta,B}}\cong
\bigotimes_{\hat{\sigma}\in\hat{\Sigma}}
(\wedge^*V_{\hat{\sigma}})^{\Spin(V_{\hat{\sigma}})_{P_{\hat{\sigma}}}}.
\end{equation}
The argument of Lemma \ref{lemma-Spin-V-P-invariant-classes-are-Hodge} applies, replacing $\QQ$ by $\tilde{F}$ and $K$ by $\tilde{K}$, and yields that 
\[
\dim_{\tilde{F}}\left((\wedge^kV_{\hat{\sigma}})^{\Spin(V_{\hat{\sigma}})_{P_{\hat{\sigma}}}}\right)=
\left\{
\begin{array}{ccl}
0&\mbox{if}&k \ \mbox{is odd}
\\
1&\mbox{if} & k \ \mbox{is even and} \ k\neq d
\\
3&\mbox{if}& k=d.
\end{array}
\right.
\]
Furthermore, $(\wedge^dV_{\hat{\sigma}})^{\Spin(V_{\hat{\sigma}})_{P_{\hat{\sigma}}}}$ decomposes as a direct sum of three one-dimensional   
$\Spin(V_{\hat{\sigma}}\otimes_{\tilde{F}}\tilde{K})_\eta$-invariant subspaces corresponding to the trivial character and the characters $\det_\sigma$ and $\det_{\bar{\sigma}}$, where $\sigma$ restricts to $\hat{\sigma}$. The direct sum of the latter two is of the form $HW_{\hat{\sigma}}\otimes_{\tilde{F}}\tilde{K}$
of a $2$-dimensional $\tilde{F}$ subspace $HW_{\hat{\sigma}}$. 
Then $HW\otimes_\QQ\tilde{K}=\oplus_{\hat{\sigma}\in\hat{\Sigma}}HW_{\hat{\sigma}}$.
Choose an element $\Xi_{\hat{\sigma}}$ spanning $(\wedge^2V_{\hat{\sigma}})^{\Spin(V_{\hat{\sigma}})_{P_{\hat{\sigma}}}}$. Then $\Xi_{\hat{\sigma}}^{d/2}$ spans that trivial character in $(\wedge^dV_{\hat{\sigma}})^{\Spin(V_{\hat{\sigma}})_{P_{\hat{\sigma}}}}$ and
the symmetric square of the subspace $HW_{\hat{\sigma}}$ is mapped onto $\wedge^{2d}V_{\hat{\sigma}}$, which is spanned by $\Xi_{\hat{\sigma}}^d$.
The tensor factors on the righthand side of (\ref{eq-factorization-of-Spin-V-B-invariant-subalgebra-over-tilde-K}) are all even, hence they commute with respect to wedge-product. We conclude that $\A_{\tilde{F}}:=(\wedge^*V_{\tilde{F}})^{\Spin(V)_{\eta,B}}$ is an abelian graded subalgebra of $\wedge^*V_{\tilde{F}}$ and it is generated by $HW\otimes_\QQ\tilde{F}$ and its graded summand $\A_{\tilde{F}}^2$ in degree $2$.

We conclude that $(\wedge^2V_{\tilde{F}})^{\Spin(V)_{\eta,B}}$ is an $(e/2)$-dimensional $\tilde{F}$-vector space with basis 
$\{\Xi_{\hat{\sigma}} \ : \ \hat{\sigma}\in\hat{\Sigma}\}$. Taking $Gal(\tilde{F}/\QQ)$ invariants we get that
$(\wedge^2V_\QQ)^{\Spin(V)_{\eta,B}}$ is an $(e/2)$-dimensional $\QQ$-vector space. Furthermore,
\[
(\wedge^2V_\QQ)^{\Spin(V)_{\eta,B}}=\{\Xi_t \ : \ t\in K_-\},
\]
 where $\Xi_t$ is given in (\ref{eq-Xi-t}).

\underline{Step 2:} Choose a basis $\{\beta_1, \dots, \beta_{e/2}\}$ of $\A^2_\QQ:=(\wedge^2V_\QQ)^{\Spin(V)_{\eta,B}}$ and a basis 
$\{\delta_1, \dots, \delta_e\}$ of $HW$. Then every element of $(\wedge^*V_{\tilde{F}})^{\Spin(V)_{\eta,B}}$ is the sum 
$\gamma:=\sum_i c_if_i$, where $c_i\in\tilde{F}$ and $f_i$ is a polynomial in the basis elements $\beta_j$'s and $\delta_k$'s. 
These two bases consist of rational classes. 
We may  further assume that the $f_i$'s are linearly independent over $\QQ$. Then they are also linearly independent over $\tilde{F}$,
since the homomorphism $\wedge^*_\QQ V_\QQ\rightarrow (\wedge^*_\QQ V_\QQ)\otimes_\QQ\tilde{F}$, given by $f\mapsto f\otimes 1$, sends linearly independent subsets over $\QQ$ to linearly independent subsets over $\tilde{F}$.
Hence, if $\gamma$ is $Gal(\tilde{F}/\QQ)$-invariant, then so is each of the coefficients $c_i$. 
Consequently, the algebra $(\wedge^*V_\QQ)^{\Spin(V)_{\eta,B}}$ is generated by $\A^2_\QQ$ and $HW$.

(\ref{lemma-item-Spin-V-eta-B-invariant-2-forms})
Let us verify the $\Spin(V)_\eta$-invariance of $(\wedge^2V_\QQ)^{\Spin(V)_{\eta,B}}$.  For $g\in\Spin(V)_\eta$ and $x,y\in V_\QQ$ we have
\[
\Xi_t(\rho_g(x),\rho_g(y)):=(\eta(t)\rho_g(x),\rho_g(y))_V=(\rho_g(\eta(t)x),\rho_g(y))_V=(\eta(t)x,y)_V=\Xi_t(x,y).
\]

(\ref{lemma-item-decomposition-of-HW-into-characters}) Clear.

(\ref{lemma-item-generatyors-for-the-invariant-subalgebra-R}) 
Assume that $x\in\wedge^*_\QQ V_\QQ$ is $\Spin(V)_\eta$-invariant. Then each of its factors $x_{\hat{\sigma}}$ in $(\wedge^*V_{\hat{\sigma}})^{\Spin(V)_{\hat{\sigma}}}$, with respect to the factorization (\ref{eq-factorization-of-Spin-V-B-invariant-subalgebra-over-tilde-K}), is $\Spin(V)_\eta$-invariant, as different factors contribute linearly independent characters over $\ZZ$. Hence, $x_{\hat{\sigma}}$ belongs to the subalgebra generated by the one-dimensional subspace $\Xi_{\hat{\sigma}}$ of $\wedge^2V_{\hat{\sigma}}$. Thus, $x$ is a rational class of the subalgebra generated by 
$(\wedge^2V_{\tilde{F}})^{\Spin(V)_{\eta,B}}$. Arguing as in Step 2 of the proof of part (\ref{lemma-item-generator-for-the-invariant-subalgebra}) we conclude that $x$ belongs to $\R$.
\end{proof}

Let $\Sigma'\subset\Sigma$ be a subset, such that $\Sigma'\cap\iota(\Sigma')=\emptyset$. 
Set $\det_{\Sigma'}:=\otimes_{\{\sigma\in\Sigma'\}}\det_\sigma.$
Let $k$ be the cardinality of $\Sigma'$.
Denote by $(\wedge^{dj}V_{\tilde{K}})_{\Sigma'}$ the isotypic subspace of $\wedge^{dj}V_{\tilde{K}}$ corresponding to the character $\det_{\Sigma'}$ of
$\Spin(V_{\tilde{K}})_\eta$. Note that $(\wedge^{dj}V_{\tilde{K}})_{\Sigma'}$ vanishes, if $j<k$ or $j>e-k$, and it is one-dimensional if $j=k$ or if $j=e-k$.
Let $(\wedge^{dj}V_{\tilde{K}})_k$ be the direct sum of all $(\wedge^{dj}V_{\tilde{K}})_{\Sigma'}$, for subsets $\Sigma'\subset\Sigma$ as above of cardinality $k$.
The dimension of $(\wedge^{dk}V_{\tilde{K}})_k$ is $\Choose{e/2}{k}2^k$.
Note that $(\wedge^{dk}V_{\tilde{K}})_k$ is contained in the subalgebra generated by $HW$ and 
$(\wedge^{d}V_{\tilde{K}})_d=HW\otimes_\QQ\tilde{K}.$

\begin{lem}
If $d>2$, then $HW$ is contained in the primitive cohomology with respect to every $\Spin(V)_{\eta,B}$-invariant ample class. 
%the cup product induces the zero homomorphism $(\wedge^2V_\QQ)^{\Spin(V)_{\eta,B}}\otimes HW\rightarrow \wedge^{d+2}V_\QQ)^{\Spin(V)_{\eta,B}}$.
More generally, the same is true for $(\wedge^{dk}V_{\tilde{K}})_k$.
\end{lem}

\begin{proof}
Let $h$ be a class in $(\wedge^2V_\QQ)^{\Spin(V)_{\eta,B}}$. Then $h$ is $\Spin(V_\QQ)_\eta$-invariant, 
by Proposition \ref{prop-generator-for-the-invariant-subalgebra}(\ref{lemma-item-Spin-V-eta-B-invariant-2-forms}).
%Note the equality $n-d=d(e-1)$ and so the classes in $(\wedge^{d(e-1)+2}V_\QQ)^{\Spin(V)_{\eta,B}}$ are all $\Spin(V_\QQ)_\eta$-invariant,
%by Lemma \ref{prop-generator-for-the-invariant-subalgebra}, as so are the classes in $(\wedge^{d-2}V_\QQ)^{\Spin(V)_{\eta,B}}$.
%Then $h^{d(e-2)/2+1}HW$ consists of $\Spin(V_\QQ)_\eta$-invariant classes. On the other hand, as $HW\otimes_\QQ\CC$ is a 
%direct sum of non-trivial $\Spin(V_\QQ)_\eta$-characters, and $h$ is $\Spin(V_\QQ)_\eta$-invariant, 
%then if non-zero, $h^{1+d(e-2)/2}HW$ is a direct sum of non-trivial characters.
%
%The case of $(\wedge^{dk}V_{\tilde{K}})_k$ is similar. 
Thus $h^{1+d(e-2k)/2}(\wedge^{dk}V_{\tilde{K}})_k$ is contained in 
$(\wedge^{d(e-k)+2}V_{\tilde{K}})_k$, which vanishes.
\end{proof}

%Let $\R$ be the subalgebra of $\wedge^*V_\QQ$ generated by $(\wedge^2V_\QQ)^{\Spin(V)_{\eta,B}}$.
%Let $\R'$ be the $\left((\wedge^2V_{\tilde{F}})^{\Spin(V)_{\eta,B}}\right)^{Gal(\tilde{F}/\QQ)}$. Note that $\R$ is a subalgebra of $\R'$ and the two 
%coincide if the $Gal(\tilde{F}/\QQ)$-action on $\hat{\Sigma}$ induces a surjection of $Gal(\tilde{F}/\QQ)$ onto the full permutation group of $\hat{\Sigma}$.

%\begin{lem}
%\label{lemma-generators-of-the-module-of-Spin-V-B-invariant-classes}
%The $\R$-subalgebra $(\wedge^*V_\QQ)^{\Spin(V)_{\eta,B}}$ of $\wedge^*V_\QQ$ is generated, as an $\R$-module,
%by the subspaces $(\wedge^{dk}V_{\tilde{K}})_k$, $0\leq k\leq e/2$.
%\end{lem}

%*******************************************************************************************************************
%
%*******************************************************************************************************************
\subsection{An adjoint orbit in $\Spin(V_\RR)_{\eta,B}$ as a period domain of abelian varieties of Weil type}
\label{sec-adjoint-orbit-n-Spin-V-RR-eta-B}
(??? generalize Section \ref{sec-period-domains} ???)

The adjoint orbit $\Omega_B$ of the complex structure $I_{X\times\hat{X}}$ in 
\[
\Spin(V_\RR)_{\eta,B}/\{\pm 1\}\cong \prod_{\hat{\sigma}\in\hat{\Sigma}}SO(V_{\hat{\sigma},\RR})_{P_{\hat{\sigma}}},
\]
is isomorphic to an open subset of 
$\prod_{\hat{\sigma}\in\hat{\Sigma}}Gr(d/2,V_{\sigma,\CC})$
in the analytic topology, where $\sigma\in\Sigma$ is a choice of one of the two characters which restricts to $F$ as $\hat{\sigma}$. The adjoint orbit $\Omega_B$ parametrizes polarized abelian  varieties $(A,\eta,h)$ with complex multiplication by $K$.

%***************************************************************************
% 
%***************************************************************************
\subsection{Hodge-Weil classes, pure spinors, and Orlov's equivalence}
%***************************************************************************
% 
%***************************************************************************
\subsubsection{A grading of the tensor square of the secant space $B$}
We have the decomposition $B\otimes_\QQ\tilde{K}=\oplus_{T\in\T_K}\ell_T$, by Corollary \ref{cor-ell-T-are-linearly-independent}.
Given $T,T'\in\T_K$, set $T\cap T':=\{\hat{\sigma}\in\hat{\Sigma} \ : \ T(\hat{\sigma})=T'(\hat{\sigma})\}$.
Let $|T\cap T'|$ denote the cardinality of $T\cap T'$. 
Then $\dim_{\tilde{K}}(W_T\cap W_{T'})=\sum_{\hat{\sigma}\in T\cap T'}\dim_{\tilde{K}}(V_{T(\hat{\sigma})})=d|T\cap T'|$.

The subspace
\[
BB_{i,\tilde{K}} := \bigoplus_{\{(T,T')\in \T_K\times\T_K \ : \ |T\cap T'|=i\}}\ell_T\otimes\ell_{T'}
\]
of $S^+_\QQ\otimes_\QQ S^+_\QQ\otimes_\QQ \tilde{K}$ is $Gal(\tilde{K}/\QQ)$-invariant, hence of the form $BB_i\otimes_\QQ\tilde{K}$, for a subspace $BB_i$ of $S^+_\QQ\otimes_\QQ S^+_\QQ$.
The vector space $B\otimes_\QQ B$ decomposes as the direct sum
\[
B\otimes_\QQ B = \oplus_{i=0}^{e/2} BB_i.
\]
Note that the characters $\ell_{T_1}\otimes\ell_{T'_1}$ and $\ell_{T_2}\otimes\ell_{T'_2}$ of $\Spin(V_\QQ)_\eta$ are isomorphic, if and only if 
$T_1\cap T_1'=T_2\cap T_2'$, by Lemma \ref{lemma-tensor-product-of-two-ells}. Consequently, the characters $\ell_{T_1}\otimes\ell_{T'_1}$ and $\ell_{T_2}\otimes\ell_{T'_2}$ are different, if $|T_1\cap T'_1|\neq |T_2\cap T'_2|.$

\begin{lem}
\label{lemma-on-lines-tensor-ell-T-ell-T-prime}
If $|T\cap T'|=1$, then the line $\tilde{\varphi}(\ell_T\otimes\ell_{T'})$ belongs to $F^{d(e-1)}(\wedge^*V_\CC)$ and it projects onto the line
$\wedge^{d(e-1)}[W_T+W_{T'}]$ in $\wedge^{d(e-1)}V_\CC$. The image of the composition 
\begin{equation}
\label{eq-composition-from-BB-1-to-wedge-de-d}
BB_1\RightArrowOf{\tilde{\varphi}}F^{d(e-1)}(\wedge^*V_\QQ)\rightarrow \wedge^{d(e-1)}V_\QQ
\end{equation}
is an $e$-dimensional subspace, which is the graded summand of degree $d(e-1)$ of the subalgebra $\langle HW\rangle$ of $\wedge^*V_\QQ$ generated by $HW$.
\end{lem}

Note that $\dim_\QQ(B)=2^{\dim_\QQ(F)}=2^{e/2}$ and so $\dim_\QQ(B\otimes_\QQ B)=2^e$. 
The character $\det_\sigma$ of $\Spin(V_{\tilde{K}})_\eta$ appears in $BB_k\otimes_\QQ\tilde{K}$, 
if and only if $k=1$ and its multiplicity in $BB_1\otimes_\QQ\tilde{K}$ is $2^{(\frac{e}{2}-1)}$. 
We see that $\dim_\QQ(BB_1)=e2^{(\frac{e}{2}-1)}$
%We have $\dim_\QQ(BB_1)=\frac{e}{2}2^{e/2}$ 
and so the composition (\ref{eq-composition-from-BB-1-to-wedge-de-d}) is injective, if and only if $e=2$.

\begin{proof}
If $d/2$ is even, then the line $\ell_\sigma\ell_{\bar{\sigma}}$ in the subspace $\Sym^2P_{\hat{\sigma}}$ of $S^+_{\hat{\sigma}}\otimes S^+_{\hat{\sigma}}$ is mapped via $\tilde{\varphi}_{\hat{\sigma}}$, given in (\ref{eq-tilde-varphi-hat-sigma}), to a line in $\wedge^*V_{\hat{\sigma}}$ which 
projects onto the top graded summand $\wedge^{2d}V_{\hat{\sigma}}$, by Lemma \ref{lemma-symmetric-or-alternating-product-of-two-pure-spinors-has-weight-2}. 
In that case the image $\tilde{\varphi}_{\hat{\sigma}}(\ell_\sigma\wedge \ell_{\bar{\sigma}})$, of the line $\ell_\sigma\wedge \ell_{\bar{\sigma}}$ in the subspace $\wedge^2P_{\hat{\sigma}}$ of $S^+_{\hat{\sigma}}\otimes S^+_{\hat{\sigma}}$, 
is contained in $F^{2d-2}(\wedge^*V_{\hat{\sigma}})$, by Lemma \ref{lemma-symmetric-or-alternating-product-of-two-pure-spinors-has-weight-2}. 

If $d/2$ is odd, then the line $\ell_\sigma\wedge \ell_{\bar{\sigma}}$ in the subspace $\wedge^2P_{\hat{\sigma}}$ of $S^+_{\hat{\sigma}}\otimes S^+_{\hat{\sigma}}$ is mapped via $\tilde{\varphi}_{\hat{\sigma}}$ to a line in $\wedge^*V_{\hat{\sigma}}$ which 
projects onto the top graded summand $\wedge^{2d}V_{\hat{\sigma}}$, by Lemma \ref{lemma-symmetric-or-alternating-product-of-two-pure-spinors-has-weight-2}. 
In that case the line $\tilde{\varphi}_{\hat{\sigma}}(\ell_\sigma \ell_{\bar{\sigma}})$ is contained in $F^{2d-2}(\wedge^*V_{\hat{\sigma}})$.

The line $\ell_\sigma\otimes\ell_\sigma$  is mapped via $\tilde{\varphi}_{\hat{\sigma}}$ to a line in $\wedge^*V_{\hat{\sigma}}$ which belongs to $F^d(\wedge^*V_{\hat{\sigma}})$, but not to $F^{d-1}(\wedge^*V_{\hat{\sigma}})$, and 
projects onto the line $\wedge^dV_\sigma$ in $\wedge^dV_{\hat{\sigma}}$. Assume that $T,T'\in \T_K$ satisfy $|T\cap T'|=1$. Say $T(\hat{\sigma}_0)=T'(\hat{\sigma}_0)=\sigma_0$. Then
\[
\ell_T\otimes\ell_{T'}=(\ell_{\sigma_0}\otimes\ell_{\sigma_0})\otimes\bigotimes_{\{\hat{\sigma}\in\hat{\Sigma} \ : \ \hat{\sigma}\neq\hat{\sigma}_0\}}(\ell_{T(\hat{\sigma})}\otimes \ell_{T'(\hat{\sigma})}),
\]
by (\ref{eq-tensor-factorization-of-ell-T}).
The line $\tilde{\varphi}(\ell_T\otimes\ell_{T'})$ thus belongs to $F^{d(e-1)}(\wedge^*V_\CC)$ and projects onto the line\footnote{
It follows that the isomorphism $\phi:S^+_\CC\otimes S^+_\CC\rightarrow \wedge^*V_\CC$, of Lemma \ref{lemma-orlov-isomorphism-is-chevalley}, maps the line $\ell_T\otimes\ell_{T'}$
into a line in $F_d(\wedge^*V_\CC)$, which projects onto the line $\wedge^d(W_T\cap W_{T'})$ in $HW\otimes_\QQ\CC.$ 
}
$\wedge^{d(e-1)}[W_T+W_{T'}]$ in $\wedge^{d(e-1)}V_\CC$. The last statement follows also from the proof of \cite[III.3.3]{chevalley}.
We conclude that 
$
\tilde{\varphi}(BB_1)\otimes_\QQ\CC
$
is contained in $F^{d(e-1)}(\wedge^*V_\CC)$ and it projects onto
$
\oplus_{\sigma\in\Sigma}\wedge^{d(e-1)}[(V_\sigma)^\perp].
$
In particular, the image of the composition (\ref{eq-composition-from-BB-1-to-wedge-de-d})
%\begin{equation}
%\label{eq-composition-from-BB-1-to-wedge-de-d}
%BB_1\RightArrowOf{\tilde{\varphi}}F^{d(e-1)}(\wedge^*V_\QQ)\rightarrow \wedge^{d(e-1)}V_\QQ
%\end{equation}
is $e$-dimensional.
\end{proof}

\begin{rem}
Let $KB_1$ be the kernel of the composition (\ref{eq-composition-from-BB-1-to-wedge-de-d}).
Let us describe $KB_1$ explicitly.
It suffices to describe $KB_1\otimes_\QQ\tilde{K}$. Given $\sigma\in\Sigma$, let
$(\T_K\times \T_K)_{\sigma}$ be the set of ordered pairs $(T,T')\in \T_K\times \T_K$, such that $T\cap T'=\{\sigma\}$.
The equality $W_T+W_{T'}=\wedge^{d(e-1)}[\oplus_{\{\sigma'\in\Sigma \ : \sigma'\neq\sigma\}}V_{\sigma'}]=(V_{\sigma})^\perp$ holds, 
for each pair $(T,T')\in(\T_K\times \T_K)_{\sigma}$. The line $\ell_T\otimes\ell_{T'}$ is canonically isomorphic to the line 
$\wedge^{d(e-1)}(V_{\sigma}^\perp)$, for each pair $(T,T')\in(\T_K\times \T_K)_{\sigma}$, by
Lemma \ref{lemma-on-lines-tensor-ell-T-ell-T-prime}. In particular, the lines $\ell_{T_1}\otimes\ell_{T'_1}$
and $\ell_{T_2}\otimes\ell_{T'_2}$ are canonically isomorphic,
for two pairs $(T_1,T'_1)$ and $(T_2,T'_2)$ in $(\T_K\times \T_K)_{\sigma}$. Choosing a non-zero vector $t_{\sigma}$ in $\wedge^{d(e-1)}[(V_{\sigma})^\perp]$, for each
$\sigma\in\Sigma$, 
we get a non-zero vector $\lambda_{(T,T')}$ in $\ell_T\otimes\ell_{T'}$ mapping to $t_{\sigma}$ via (\ref{eq-composition-from-BB-1-to-wedge-de-d}), for each 
$(T,T')$ with $|T\cap T'|=1$. The subspace $KB_1\otimes_\QQ\tilde{K}$ is explicitly given by
\begin{equation}
\label{eq-explicit-equations-for-KB-1}
KB_1\otimes_\QQ\tilde{K} \ = \ 
\left\{
\sum_{\{(T,T') \ : \ |T\cap T'|=1\}}c_{(T,T')}\lambda_{(T,T')} \ : \
\sum_{(T,T')\in(\T_K\times \T_K)_{\sigma}}c_{(T,T')}=0, \ \ \forall \sigma\in\Sigma
\right\},
\end{equation}
where $c_{(T,T')}\in \tilde{K}$. 
\EndProof
\end{rem}

\begin{question}
(???) Let $m_{g_0}(1)\in[\wedge^{ev}_FH^1(X,\QQ)]\otimes_FK$ be the pure spinor of $W$.
Then $(id\otimes T(\hat{\sigma}))(m_{g_0}(1))$ is an element of 
$[\wedge^{ev}H^1_{\hat{\sigma}}(X)]\otimes_{\tilde{F}}\tilde{K}$ and 
\[
\otimes_{\hat{\sigma}\in\hat{\Sigma}}(id\otimes T(\hat{\sigma}))(m_{g_0}(1))
\]
is an explicit class $\lambda_T$ in $\ell_T$. 
%Hence, $\lambda_T\otimes\lambda_{T'}$ is an explicit class in $\ell_T\otimes\ell_{T'}$. 
Can we take $\lambda_{(T,T')}$ in the above remark to be $\lambda_T\otimes\lambda_{T'}$?
I.e., are the two cosets $\tilde{\varphi}(\lambda_{T_1}\otimes\lambda_{T'_1})+F^{d(e-1)}(\wedge^*V_\CC)$ and
$\tilde{\varphi}(\lambda_{T_2}\otimes\lambda_{T'_2})+F^{d(e-1)}(\wedge^*V_\CC)$
equal, for any two pairs $(T_i,T'_i)\in(\T_K\times\T_K)_\sigma$, $i=1,2$?
\end{question}

\begin{lem}
\label{lemma-on-the-phi-image-of-lines-tensor-ell-T-ell-T-prime}
If $|T\cap T'|=1$, then the line $\phi(\ell_T\otimes\ell_{T'})$ belongs\footnote{Should it be the line $(\phi\circ (id\otimes\tau))(\ell_T\otimes\ell_{T'})$?
}
 to $F_d(\wedge^*V_\CC)$ and it projects onto the line
$\wedge^d[W_T\cap W_{T'}]$ in $\wedge^dV_\CC$. The image of the composition 
\begin{equation}
\label{eq-composition-from-BB-1-to-wedge-d}
BB_1\RightArrowOf{\phi}F_d(\wedge^*V_\QQ)\rightarrow \wedge^dV_\QQ
\end{equation}
is $HW$.
%an $e$-dimensional subspace, which is the graded summand of degree $d$ of $\langle HW\rangle$.
\end{lem}

\begin{proof}
The statement follows from Lemma \ref{lemma-on-lines-tensor-ell-T-ell-T-prime}.
(???) We need to check that $\phi_\P\otimes\psi_{\P^{-1}[n]}$ restricts to each graded summand of $\wedge^*V_\QQ$ 
as Poincar\'{e} duality, up to sign (???). Poincar\'{e} duality maps $\langle HW\rangle_{d(e-1)}$ to $\langle HW\rangle_d$ and that it maps $\wedge^d[W_T+ W_{T'}]$ to $\wedge^d[W_T\cap W_{T'}]$. 
Indeed, Poincar\'{e} duality maps the top exterior power $\wedge^t_{\tilde{K}}Z$ of a $t$-dimensional subspace $Z$ of $V_{\tilde{K}}$ to the top exterior power 
$\wedge^{de-t}_{\tilde{K}}(Z^\perp)$, where $Z^\perp$ is the orthogonal complement with respect to the wedge product pairing on $\wedge^*V_{\tilde{K}}$.
Now note that 
$W_T\cap W_{T'}$ is the orthogonal complement of $W_T+ W_{T'}$ with respect to wedge product pairing and 
given $\sigma\in\Sigma$, the subspace $V_\sigma$ is the orthogonal complement of $\oplus_{\{\sigma'\in\Sigma \ : \ \sigma'\neq\sigma\}}V_{\sigma'}$.
The subspace $\langle HW\rangle_d$ is the direct sum of the top exterior powers of the former lines and 
$\langle HW\rangle_{d(e-1}$ is the direct sum of the top exterior powers of the latter lines.
\end{proof}

%**************
% Hide
%**************
\hide{
Consider the decreasing filtration $F_k(B\otimes_\QQ B):=\oplus_{i\geq k}BB_i$
and the increasing filtration $F^k\wedge^*V_\QQ:=\oplus_{i\leq k}\wedge^iV_\QQ$.
% and the increasing filtration $F^k(B\otimes_\QQ B):=\oplus_{i\leq k}BB_i$.
The isomorphism $\tilde{\varphi}:S_\QQ\otimes_\QQ S_\QQ\rightarrow \wedge^*V_\QQ$, given in (\ref{eq-tilde-varphi}),
maps $F_k(B\otimes_\QQ B)$ into $F^{d(e-k)}(\wedge^*V_\QQ)$, but the image of $F_k(B\otimes_\QQ B)$ is not contained in $F^{d(e-k)-1}(\wedge^*V_\QQ)$, by \cite[III.3.3]{chevalley}. 
We get the induced homomorphism $BB_k\rightarrow \wedge^{d(e-k)}V_\QQ$. 
The proof of \cite[III.3.3]{chevalley} shows furthermore that the image of $\ell_T\otimes\ell_{T'}$ in $\wedge^{d(e-k)}V_\QQ$ is the top exterior power of 
the subspace $W_T+ W_{T'}$ of $V_{\tilde{K}}$.
The associated homomorphism $BB_k\otimes_\QQ\tilde{K}\rightarrow \wedge^{d(e-k)}V_{\tilde{K}}$
is furthermore $\Spin(V_{\tilde{K}})_\eta$-equivariant. 
%Choose an ordering of the set $\T_K$. 
The image of the induced homomorphism 
$BB_1\otimes_\QQ\tilde{K}\rightarrow \wedge^{d(e-1)}V_{\tilde{K}}$ is thus 
\[
\bigoplus_{\sigma\in\Sigma}
\wedge^{d(e-1)}\left(\bigoplus_{\{\sigma' \ : \ \sigma'\neq \sigma\}}V_{\sigma'}
\right).
%=HW\otimes_\QQ\tilde{K}
\] 
%and $\tilde{\varphi}(BB_{e-1})=HW$. 
The isomorphism $\phi\circ (id\otimes\tau)$, $\phi$ as in Lemma \ref{lemma-orlov-isomorphism-is-chevalley}, maps 
$F_k(B\otimes_\QQ B)$ into $F_{dk}(\wedge^*V_\QQ)$ and the induced homomorphism on the graded summands maps 
$BB_1$ onto $HW$.
%since the characters $\ell_T\otimes \ell_{T'}$, . 

%**************
% End Hide
%**************
}

%Note that $\dim_\QQ(B)=2^{\dim_\QQ(F)}=2^{e/2}$ and so $\dim_\QQ(B\otimes_\QQ B)=2^e$. 
%The character $\det_\sigma$ of $\Spin(V_{\tilde{K}})_\eta$ appears in $BB_k\otimes_\QQ\tilde{K}$, 
%if and only if $k=1$ and its multiplicity in $BB_1\otimes_\QQ\tilde{K}$ is $2^{(\frac{e}{2}-1)}$. 
%We see that $\dim_\QQ(BB_1)=e2^{(\frac{e}{2}-1)}$.
%The character $\otimes_{\{\hat{\sigma} \ : \ \hat{\sigma}\neq \hat{\sigma}_0\}}\det_{T(\hat{\sigma})}$
%appears in $BB_k\otimes_\QQ\tilde{K}$, if and only if $k=\frac{e}{2}-1$ and it appears in $BB_{(\frac{e}{2}-1)}\otimes_\QQ\tilde{K}$ with multiplicity $2$.
%There are $(e/2)2^{(\frac{e}{2}-1)}$ characters of this type. 
%We see that $\dim_\QQ(BB_{(\frac{e}{2}-1)})=e2^{(\frac{e}{2}-1)}$.
The character $\det_{\sigma_1}\otimes \cdots \otimes \det_{\sigma_k}$, with $\hat{\sigma}_1, \dots, \hat{\sigma}_k$ distinct in $\hat{\Sigma}$,
appears only in $BB_k\otimes_\QQ\tilde{K}$ and it appears in $BB_k\otimes_\QQ\tilde{K}$ with multiplicity $2^{(\frac{e}{2}-k)}$.
Denote by 
\[
BB_{T\cap T'}
\] 
the subspace of $BB_k\otimes_\QQ\tilde{K}$, $k=|T\cap T'|$, such that $BB_{T\cap T'}$ is the direct sum of all 
characters of $\Spin(V_{\tilde{K}})_\eta$ isomorphic to $\otimes_{\{\hat{\sigma}\in T\cap T'\}}\det_{T(\hat{\sigma})}$. 
There are $2^k\Choose{e/2}{k}$ pairwise non-isomorphic such characters, so the dimension of $BB_k$ is $\Choose{e/2}{k}2^{e/2}$.
%***************************************************************************
% 
%***************************************************************************
\subsubsection{Hodge-Weil classes from pure spinors} 
Denote by $\langle HW\rangle$ the $\QQ$-subalgebra of $\wedge^*V_\QQ$ generated by the $e$-dimensional subspace $HW$ of $\wedge^dV_\QQ$. 
The algebra is graded and we set $\langle HW\rangle_k:=\langle HW\rangle\cap \wedge^kV_\QQ.$ The graded summand $\langle HW\rangle_k$ vanishes, if $d\not| \ k$, and
\[
\dim_\QQ \langle HW\rangle_{dk}=\Choose{e}{k},
\]
for $0\leq k\leq e$.
The multiplicity of $\det_\sigma$ in $\langle HW\rangle_{dk}\otimes_\QQ\tilde{K}$ is $0$, if $k$ is even, and it is $\Choose{(e/2)-1}{i}$, if $k=2i+1$, for $0\leq i\leq \frac{e}{2}-1$. Hence, the multiplicity of $\det_\sigma$ in $\langle HW\rangle\otimes_\QQ\tilde{K}$ is equal to its multiplicity in $B\otimes_\QQ B\otimes_\QQ\tilde{K}$. 
%Indeed, we have the following.

%**************
% Hide
%**************
\hide{
\begin{lem}
The isomorphism
\[
\tilde{\phi}:=\exp\left(\frac{1}{2}c_1(\P)\right)\cup \phi\circ (id\otimes \tau):H^*(X\times X,\QQ)\rightarrow H^*(X\times\hat{X},\QQ),
\]
which is  $\Spin(V)$-equivariant with respect to $(m\otimes m^\dagger,\rho)$ 
by Proposition \ref{prop-extension-class-of-decreasing-filtration-of-spin-V-representations}, 
%maps the $2^e$-dimensional subspace 
%$B\otimes_\QQ B$ of $S\otimes_\QQ S$ isomorphically onto 
%$\langle HW\rangle$. 
%Furthermore, 
satisfies
\begin{equation}
\label{equality-in-degree-de-over-2}
\tilde{\phi}(BB_{e/2})\subset \langle HW\rangle_{de/2},
\end{equation}
where the right hand side is the graded summand of $\langle HW\rangle$ in degree $de/2$. 
\end{lem}

\begin{proof}
The isomorphism $\tilde{\phi}$ extends to a $\Spin(V_{\tilde{K}})_\eta$-equivariant one from $S_\QQ\otimes_\QQ S_\QQ\otimes\tilde{K}$ onto
$\wedge^*V_{\tilde{K}}$. The subspace $BB_{e/2}\otimes_\QQ\tilde{K}$ is the sum of all 
characters of $\Spin(V_{\tilde{K}})_\eta$ isomorphic to $\ell_T\otimes\ell_T$, for some $T\in\hat{\Sigma}$.
The subspace $\langle HW\rangle_{de/2}\otimes_\QQ\tilde{K}$ of $\wedge^*V_{\tilde{K}}$ is the sum of all characters of $\Spin(V_{\tilde{K}})_\eta$ isomorphic to
tensor products $\otimes_{\sigma\in \Sigma'}\det_\sigma$, for some subset $\Sigma'\subset\Sigma$ of cardinality $e/2$, which projects onto (??? no ???) $\hat{\Sigma}$ under the restriction map $\Sigma\rightarrow\hat{\Sigma}$. The two sets of characters are equal, 
by Lemma \ref{lemma-tensor-product-of-two-ells}.
\end{proof}
%*********
% End Hide
%*********
}

%\begin{rem}
%Additional grading on $\wedge^*V_{\tilde{K}}$ is given by the action via $T\circ \eta$ of $K^\times$. 
%According to this grading $\langle HW\rangle\otimes_\QQ\tilde{K}$
%is the sum of all $d$-th powers of characters of $K^\times$.
%\end{rem}

Let $c$ be a class in $B\otimes_\QQ B$. 
Write $c=\sum_i c_i$, with $c_i$ in $BB_i$ and assume that $c_1$ does not belong to $KB_1$.
Let $\phi(c)_i$ be the graded summand of $\phi(c)$ in $\wedge^{2i}V_\QQ$. Set $r:=\phi(c)_0$, considered as a rational number via the isomorphism $\wedge^0V_\QQ\cong \QQ$. Assume that $r\neq 0$. Set
$\kappa(\phi(c)):=\phi(c)\exp\left(-\phi(c)_1/r\right)$. Assume that $d>2$. 

\begin{prop}
\label{prop-kappa-class-of-image-of-secant-class-yields-a-HW-class}
\begin{enumerate}
\item
\label{lemma-item-kappa-phi-c-is-invariant}
The class $\kappa(\phi(c))$ is $\Spin(V)_{\eta,B}$-invariant. 
\item
\label{lemma-item-kappa-phi-c-0-is-invariant}
The class $\kappa(\phi(c_0))$ belongs to $\R$. 
\item
\label{lemma-item-difference-projects-to-HW}
The difference
$\kappa(\phi(c))-\kappa(\phi(c_0))$ belongs to $F_d(\wedge^*V_\QQ)$ and its projection to $\wedge^dV_\QQ$ is a non-zero element of $HW$.
\end{enumerate}
\end{prop}

\begin{proof}
(\ref{lemma-item-kappa-phi-c-is-invariant}) 
The class $\kappa(\phi(c))$ is $\Spin(V)_{\eta,B}$-invariant, by the same argument proving Corollary \ref{cor-kappa-class-is-Spin-V-P-invariant}.

(\ref{lemma-item-kappa-phi-c-0-is-invariant})
The class $\phi(c_k)$ belongs to $F_{dk}(\wedge^*V_\QQ)$ and $d>2$, by assumption. Hence, $\phi(c)_0=\phi(c_0)_0$ and 
$\phi(c)_1=\phi(c_0)_1$. 
We thus have
\[
\kappa(\phi(c))=\kappa(\phi(c_0))+\phi(c-c_0)\exp\left(-\phi(c)_1/r\right).
\]
The class $\kappa(\phi(c_0))$ is $\Spin(V)_\eta$-invariant, by the same argument proving Corollary \ref{cor-kappa-class-is-Spin-V-P-invariant},
and is hence in $\R$, by Proposition \ref{prop-generator-for-the-invariant-subalgebra}(\ref{lemma-item-generatyors-for-the-invariant-subalgebra-R}). 

(\ref{lemma-item-difference-projects-to-HW})
The class $\phi(c-c_0)\exp\left(-\phi(c)_1/r\right)$ belongs to $F_d(\wedge^*V_\QQ)$
and its projection to $\wedge^dV_\QQ$ is equal to that of $\phi(c-c_0)$, which in turn is equal to the projection of $\phi(c_1)$. 
Now, the projection of $\phi(c_1)$ to $\wedge^dV_\QQ$ is a non-zero element of $HW$, by Lemma \ref{lemma-on-the-phi-image-of-lines-tensor-ell-T-ell-T-prime}.
\end{proof}

The following Corollary is an immediate consequence of Proposition \ref{prop-kappa-class-of-image-of-secant-class-yields-a-HW-class}. It is the reduction of the algebraicity of the Weil classes on abelian varieties with complex multiplication to the variational Hodge conjecture. 
(???) Generalize the following statement for deformations $(A,\eta',h')$ of $(X\times\hat{X},\eta,h)$ parametrized by points in the period domain $\Omega_B$  in Section 
\ref{sec-adjoint-orbit-n-Spin-V-RR-eta-B}. 

\begin{cor}
Assume that the class $c$ is algebraic. 
If the algebraic class $\kappa(\phi(c))$ in $H^d(X\times\hat{X},\QQ)$ remains algebraic in $H^d(A,\QQ)$, then every class in $HW(A,\eta)$ is algebraic.
\end{cor}

%*************
% Hide
%*************
\hide{
Let $c$ be a class in $B\otimes_\QQ B$. 
Write
$c=\sum_i c_i$, with $c_i$ in $BB_i$ and assume that $c_1$ does not belong to $KB_1$.

\begin{lem}
There exists a class $\delta\in \R$, such that 
$\tilde{\varphi}(c)-\delta$ belongs to $F^{d(e-1)}(\wedge^*V_\QQ)$ and its projection to $\wedge^{d(e-1)}V_\QQ$ is a non-zero class in 
$\langle HW\rangle_{d(e-1)}$. 
\end{lem}

\begin{proof}
The summand $c_0$ spans a trivial $\Spin(V_\QQ)_\eta$-character.
Hence, there exists a class $\alpha\in \R$, such that the $\Spin(V_\QQ)_\eta$-invariant class 
$\tilde{\varphi}(c_0)-\alpha$ belongs to $F^{d(e-1)}(\wedge^*V_\QQ)$, by Proposition \ref{prop-generator-for-the-invariant-subalgebra} (??? the lemma is in terms of the $\rho$-action, but $\tilde{\varphi}$ is not equivariant with respect to this action ???).
Let $\beta$ be the projection of $\tilde{\varphi}(c_0)-\alpha$ to $\wedge^{d(e-1)}V_\QQ$.
The class $\beta$ is $\Spin(V_{\tilde{K}})_\eta$-invariant and so in $\R$, by Lemma \ref{lemma-generators-of-the-module-of-Spin-V-B-invariant-classes}. 
The summand $\tilde{\varphi}(c_i)$ belongs to $F^{d(e-i)}(\wedge^*V_\QQ)$. Hence, $\tilde{\varphi}(c-c_0)$ belongs to $F^{d(e-1)}(\wedge^*V_\QQ)$
and its image in $\wedge^{d(e-1)}V_\QQ$ is equal to the image of $\tilde{\varphi}(c_1)$, 
which is an element $\gamma$ of $\langle HW\rangle_{d(e-1)}$, by Lemma \ref{lemma-on-lines-tensor-ell-T-ell-T-prime}. The element $\gamma$ 
%of $\langle HW\rangle_{d(e-1)}$ 
is non-zero,  by assumption. The 
vector $\tilde{\varphi}(c)-\alpha=(\tilde{\varphi}(c_0)-\alpha)+\tilde{\varphi}(c-c_0))$ is thus in $F^{d(e-1)}(\wedge^*V_\QQ)$ and it projects to $\wedge^{d(e-1)}V_\QQ$ as 
the sum of a $\Spin(V_{\tilde{K}})_\eta$-invariant class $\beta$ and a non-zero class $\gamma$ in $\langle HW\rangle_{d(e-1)}$.
The statement follows by setting $\delta=\alpha+\beta$.
\end{proof}
%*************
% End Hide
%*************
}

%Let $T$ and $T'$ be CM-types, such that the set 
%$\{\hat{\sigma}\in \hat{\Sigma} \ : \ T(\hat{\sigma})= T'(\hat{\sigma})\}$ consists of a single character $\hat{\sigma}_{T,T'}$. 
%Denote by $\sigma_{T,T'}\in\Sigma$ the common value of $T$ and $T'$ at $\hat{\sigma}_{T,T'}$. 
%The intersection $W_{T,\CC}\cap W_{T',\CC}$ is $V_{\sigma_{T,T'}}$ and their sum
%$W_{T,\CC}+ W_{T',\CC}$ is the direct sum of all $V_\sigma$, for $\sigma\neq \iota(\sigma_{T,T'})$. Now use \cite[III.3.3]{chevalley} 
%to recover $\wedge^{top}V_{\iota(\sigma_{T,T'})}$ from the tensor product of the two pure spinors $\ell_T\otimes \ell_{T'}$.

%Show  that for a generic pair of points $b_1, b_2\in B$, the image of 
%$b_1\otimes b_2$ in $\wedge^*V_\QQ$ via Orlov's cohomological isomorphism, normalized as in the definition of the $\kappa$ class, 
%has graded component in the sum of $\wedge^{d}V_\QQ$, and the subspace spanned by polynomials in divisor class, 
%and the projection $\kappa(b_1\otimes b_2)$ to $HW$ is non-zero. Then choose coherent sheaves $F_i$ with $ch(F_i)=b_i$ and 
%deduce that the classes in $HW$ are algebraic. Then deform 
%the image of $F_1\boxtimes F_2$ via Orlov's derived equivalence, to all abelian $2n$-folds with complex multiplication by $K$ deformation equivalent 
%to $(X\times\hat{X},\eta)$

%*******************************************************************************************************************
%
%*******************************************************************************************************************
\subsection{An example}
%\begin{example}
\label{sec-example-isometry-g-0}

The subspace $\wedge_F^kH^1(X,\QQ)$ consists of elements $\gamma$ of $\wedge_Q^kH^1(X,\QQ)\subset H^1(X,\QQ)^{\otimes k}$, such that the endomorphism 
 $1 \otimes \cdots \otimes \hat{\eta}_q \otimes \cdots \otimes 1$, with $\hat{\eta}_q$ in the $j$-th tensor factor, maps $\gamma$ to the same element $\hat{\eta}_q(\gamma)$, regardless of the choice of $j$. The condition is always satisfied for $k=1$, and so wedge products over $\QQ$ of such classes need not be such.
 %If $\alpha$ and $\beta$ are elements of $H^1(X)_{\hat{\sigma}}\cap H^1(X,\hat{\sigma}(F))$ and $F$ is Galois over $\QQ$, and we let $\Theta$ be  
 %the sum of the Galois orbit of $\alpha\wedge\beta$, then $\Theta$ belongs to $\wedge_F^2H^1(X,\QQ)$.

\begin{lem}
\label{lemma-wedge-2-F-H-1-contains-an-ample-class}
If $X$ is simple and $\End_\QQ(X)=F$, then 
$H^{1,1}(X,\QQ)$ is contained in the subspace $\wedge^2_FH^1(X,\QQ)$ of $H^2(X,\QQ)$. 
\end{lem}

\begin{proof}
%This is seen as follows. 
We have 
$\dim_\QQ H^{1,1}(X,\QQ)=\dim_\QQ(F)$, by
\cite[Prop. 5.5.7]{BL}. In this case $NS(X)\otimes_\ZZ\QQ$ is isomorphic to $\End_\QQ(X)$, by 
\cite[Prop. 5.2.1]{BL}. Let $L_0$ be a polarization on $X$ and let $L$ be a line bundle.
Denote their hermitian forms by $H_0$ and $H$. Let $X=V/\Lambda$ and $\hat{X}=\Omega/\hat{\Lambda}$, where $\Omega:=\Hom_{\bar{\CC}}(V,\CC)$ and
$\hat{\Lambda}=\{\omega \ : \ Im(\omega(\lambda))\in\ZZ, \forall \lambda\in\Lambda\}$.
Let $\phi_{H_0}:V\rightarrow\Omega$ be given by $\phi_{H_0}(v)=H_0(v,\bullet)$ and define $\phi_H$ similarly. Then $\phi_H$ is the analytic representation of $\phi_L$.
Set $\theta:=Im(H)$ and $\theta_0:=Im(H_0)$. Let $f:=\phi_H^{-1}\phi_{H_0}\in\End_\QQ(X)$.
Then $\phi_{H_0}(x)=\phi_H(f(x))$ and $H_0(x,\bullet)=H(f(x),\bullet)$, for all $x\in V$. Thus,
\begin{eqnarray*}
\theta(f(x),y)&=&Im H(f(x),y)=Im H_0(x,y)=-Im H_0(y,x) = -Im H(f(y),x)
\\
&=&Im H(x,f(y))=\theta(x,f(y)).
\end{eqnarray*}
So $\theta$ belongs to $\wedge^2_F \Lambda_\QQ^*=\wedge^2_FH^1(X,\QQ)$.
%Requiring the existence of a Hodge class in $\wedge^2_FH^1(X,\QQ)$ is in analogy to the case $F=\QQ$, 
%in which the construction was carried out only for projective compact complex tori. Note, however, that $\wedge^2_FH^1(X,\QQ)$
%is spanned by Hodge classes if $\dim_FH^1(X,\QQ)=2$ and $\dim H^{1,0}_{\hat{\sigma}}(X)=\dim H^{0,1}_{\hat{\sigma}}(X)$, 
%for all $\hat{\sigma}\in\hat{\Sigma}$. 
\end{proof}

Assume that $X$ is simple and $\End_\QQ(X)=F$.
 Then the subspace $\wedge^2_FH^1(X,\QQ)$ of $H^2(X,\QQ)$ contains a Hodge class $\Theta$, by Lemma \ref{lemma-wedge-2-F-H-1-contains-an-ample-class}.
Contraction with $\Theta$ induces an $\hat{\eta}(F)$-equivariant homomorphism $\Theta\Contract: H^1(X,\QQ)^*\rightarrow H^1(X,\QQ)$.
The element $g$ of $\Spin(V_\QQ)$ corresponding to cup product with $\exp(\Theta)$ acts on $V_\QQ:=H^1(X,\QQ)\oplus H^1(X,\QQ)^*$ 
by $\rho_g(w,\xi)=(w+\Theta\Contract \xi,\xi)$. It follows that $\rho_g$ is $\hat{\eta}(F)$-equivariant. Hence, the induced action of $\rho_g$ on $\wedge^*V_\QQ$ commutes with the induced action of $\hat{\eta}(F)$, where $\hat{\eta}_f$ acts via $\wedge^k\hat{\eta}_f$ on $\wedge^kV_\QQ$, for all $f\in F$.
Note that the nilpotent element of square zero $\rho_g-id_{V_\QQ}$ depends linearly on $\Theta$.
We thus define the element $g_0\in\Spin(V_{\hat{\eta},K})$ corresponding to $\exp(\sqrt{-q}\Theta)$ as a lift of the isometry of $V_\QQ\otimes_F K$ corresponding to $id_{V_\QQ\otimes_F K}+(\rho_g-id)\otimes\sqrt{-q}$, where $q$ is an element of $F$ such that $K=F[\sqrt{-q}]$. 
Explicitly, decompose elements of $V_\QQ\otimes_F K$ as  $(v_1,v_2\otimes\sqrt{-q})$ with $v_i\in V_\QQ$. Then 
\begin{eqnarray*}
(\rho_g-id)(v_1,v_2\otimes\sqrt{-q})&=& ((\rho_g-id)(v_1),(\rho_g-id)(v_2)\otimes\sqrt{-q})
\\
\sqrt{-q}(v_1,v_2\otimes\sqrt{-q})&=&(\hat{\eta}_{-q}(v_2),v_1\otimes\sqrt{-q}),
\\
(\sqrt{-q}\circ (\rho_g-id))(v_1,v_2\otimes\sqrt{-q})&=&(\hat{\eta}_{-q}((\rho_g-id)(v_2)),(\rho_g-id)(v_1)\otimes\sqrt{-q})
\end{eqnarray*}
Set 
\begin{eqnarray*}
\nu&:=&(\sqrt{-q}\circ (\rho_g-id)),
\\
\rho_{g_0}&:=&\exp(\nu).
\end{eqnarray*} 
In block matrix form we have:
\[
\nu=\left(
\begin{array}{cc}
0 & -\hat{\eta}_q(\rho_g-id))
\\
\rho_g-id & 0
\end{array}
\right), \ \mbox{and} \
\exp(\nu)=\left(
\begin{array}{cc}
1 & -\hat{\eta}_q(\rho_g-id))
\\
\rho_g-id & 1
\end{array}
\right).
\]
The lift $g_0$ of $\rho_{g_0}$ is unique up to multiplication by $-\one_S$ (acting as multiplication by $-1$ on the spin representation), and we choose $g_0$ so that $g_0-\one_S$ is nilpotent. The induced action of $\nu$ on $\wedge^*_{K}[V_\QQ\otimes_F K]$ is nilpotent, but no longer of square zero. The induced action of $\rho_{g_0}$ on $\wedge^*_{K}[V_\QQ\otimes_F K]$ is still $\exp(\nu)$. The maximal isotropic subspace $W=\rho_{g_0}(H^1(\hat{X},\QQ)\otimes_F K)$ is the graph of
$\restricted{\nu}{}:H^1(\hat{X},\QQ)\otimes_FK\rightarrow H^1(X,\QQ)\otimes_FK$ sending $\xi\in H^1(\hat{X},\QQ)$ to
$(\Theta\rfloor\xi)\otimes\sqrt{-q}$. Clearly, $W\cap \iota(W)=(0)$.

We get the induced action $m_{g_0}$ on $\wedge^*_K [H^1(X,\QQ)\otimes_F K]$. The pure spinor of the maximal isotropic subspace $H^1(\hat{X},\QQ)$ is $1\in H^{ev}(X,\QQ)$. 
The pure spinor of $W$ is $m_{g_0}(1)\in 
\wedge^*_K[H^1(X,\QQ)\otimes_FK]\subset H^*(X,\QQ)\otimes_F K$. The element $m_{g_0}(1)$
is given by
\begin{eqnarray}
\nonumber
m_{g_0}(1)&=&1+\Theta\otimes\sqrt{-q} -\hat{\eta}_q(\Theta\wedge_{K}\Theta)-\hat{\eta}_q(\Theta\wedge_{K}\Theta\wedge_{K}\Theta)\otimes\sqrt{-q}+\dots
\\
\label{eq-pure-spinor-rho-g-0-of-1}
&=& \left(1-\hat{\eta}_q(\Theta\wedge_{K}\Theta)+\dots\right)+
\left(\Theta-\hat{\eta}_q(\Theta\wedge_{K}\Theta\wedge_{K}\Theta)+\dots\right)\otimes\sqrt{-q}
\end{eqnarray}
The $j$-th wedge product $\Theta\wedge_{K}\cdots\wedge_{K} \Theta$ 
is equal to $\Theta\wedge_F\cdots\wedge_F\Theta$, where we consider $\wedge^{2j}_FH^1(X,\QQ)$
as an $F$-subspace of $\wedge^{2j}_K[H^1(X,\QQ)\otimes_FK]$ via the isomorphism 
\[
[\wedge^{2j}_FH^1(X,\QQ)]\otimes_F K\cong \wedge^{2j}_K[H^1(X,\QQ)\otimes_FK].
\]
The product $\Theta\wedge_F\cdots\wedge_F\Theta$ is the projection of $\Theta\wedge_\QQ\cdots\wedge_\QQ\Theta$ into the 
direct summand $\wedge^{2j}_FH^1(X,\QQ)$ of $\wedge^{2j}_\QQ H^1(X,\QQ)$. Write 
$\Theta=\sum_{\hat{\sigma}\in\hat{\Sigma}}\Theta_{\hat{\sigma}}$, where $\Theta_{\hat{\sigma}}$ is in
$\wedge^2H^1_{\hat{\sigma}}(X)$. Under the embedding of $\wedge^{2j}_\QQ H^1(X,\QQ)$ in $\wedge^{2j}_\RR H^1(X,\RR)$ we get the equality $\Theta\wedge_F\cdots\wedge_F\Theta=\sum_{\hat{\sigma}\in\hat{\Sigma}}\Theta_{\hat{\sigma}}\wedge_\RR\cdots\wedge_\RR\Theta_{\hat{\sigma}}$.

%********
% Hide
%********
\hide{
Note that for $\rho_{g_0}$ to act on $H^1(X\times\hat{X},\CC)$ we need to choose an embedding 
$H^*(X,\QQ)\otimes_F K\rightarrow H^1(X\times\hat{X},\CC)$ to determine the action of $\sqrt{-q}$. So whether $\rho_{g_0}$ is an automorphism of the Hodge structure would depend on this embedding. Our choice is such that $H^1(X\times\hat{X},\QQ)$ is invariant with respect to the $\rho_{g_0}$ action (??? NO ???).
This depends on equality (\ref{eq-W-T-and-its-conjugate-are-complementary}) which $\rho_{g_0}$ needs to satisfy.
We are thus able to induce an action on $H^*(X\times\hat{X},\QQ)\cong\wedge^*_\QQ V_\QQ$ and similarly on $H^*(X,\QQ)$. 
The construction is somewhat subtle. The element $\rho_{g_0}$ maps $H^1(\hat{X},\QQ)\otimes_F K$ to a maximal isotropic subspace $W$ of $V_\QQ\otimes_F K$ defined over $K$. We then use $W$ and a choice of CM-type $T$ to define the embedding of $H^1(X\times\hat{X},\QQ)\otimes_F K$ in $H^1(X\times\hat{X},\CC)$ so that $H^1(X\times\hat{X},\QQ)$ becomes a $K$ vector space. Only then we are able to lift $\rho_{g_0}$ to an element of $\Spin(V_\QQ)$ acting on $H^*(X\times\hat{X},\QQ)=\wedge^*_\QQ V_\QQ$ rather than an element of  $\Spin(V_{\hat{\eta},K})$ acting on $\wedge^*_K [V_\QQ\otimes_F K].$ 
%********
% End Hide
%********
}
%\hfill{\EndProof}
%********
% Hide
%********
\hide{
Let $g_0$ be cup product with $\exp(\sqrt{-q}\Theta)$, where $q$ is an element of $F$ such that $K=F[\sqrt{-q}]$. 
Note that cup product with $\Theta$ commutes 
with $\hat{\eta}_q$, by our choice of $\Theta$.
We have
\[
\exp(\sqrt{-q}\Theta)=
\left(
1-\hat{\eta}_q\circ\Theta^2/2+\hat{\eta}_q^2\circ\Theta^4/4!+\dots
\right)+
\sqrt{-q}\left(
\Theta-\hat{\eta}_q\circ\Theta^3/3!+\hat{\eta}_q^2\circ\Theta^5/5!+\dots
\right)
\]
Powers of $\Theta$ above denote the endomorphisms of $H^*(X,\QQ)$ of cup product with these powers and $\circ$ denotes composition of endomorphisms. 
(??? the above doesn't make sense, since the $\hat{\eta}_q$ action on $\wedge^kH^1(X,\QQ)$ does depend on the tensor factor in $H^1(X,\QQ)^{\otimes k}$ we place it. It acts only on the subspace $\wedge^k_FH^1(X,\QQ)$. On the other hand, the cup product with the exponential of a $2$-form in $\wedge^2_F(X,\QQ)$ induces an isometry of $H^1(X,\QQ)\oplus \Hom_F(H^1(X,\QQ),F)$, which induces an isometry of $V_\QQ$, which should correspond to an element of $\Spin(V_\QQ)$, which should then act on the spin representation $H^*(X,\QQ)$. Assume that $F$ is Galois over $\QQ$. An example of a class $\alpha$, which belongs to $\wedge^k_FH^1(X,\QQ)$ is the sum of a Galois orbit of classes $\alpha_{\hat{\theta}}$ in $\wedge^kH^1(X)_{\hat{\theta}}$. Then $\hat{\eta}_q(\alpha)=\sum_{\hat{\theta}\in\hat{\Sigma}}\hat{\theta}(q)\alpha_{\hat{\theta}}$. Wedge powers of $\alpha$ are not of this type. ???)
Note that for $\exp(\sqrt{-q}\Theta)$ to act on $H^{ev}(X,\CC)$ we need to choose an embedding 
$H^*(X,\QQ)\otimes_F K\rightarrow H^*(X,\CC)$ to determine the action of $\sqrt{-q}$. So whether $g_0$ is an automorphism of the Hodge structure would depend on this embedding.
%********
% End Hide
%********
}
%\end{example}

The secant space $B$ in this example is computed via the factorization (\ref{eq-tensor-factorization-of-B}). Write 
$\Theta=\sum_{\hat{\sigma}\in\hat{\Sigma}}\Theta_{\hat{\sigma}}$, where $\Theta_{\hat{\sigma}}$ is in
$\wedge^2H^1_{\hat{\sigma}}(X)$. 
$V_\CC=\oplus_{\hat{\sigma}\in\hat{\Sigma}}V_{\hat{\sigma},\RR}\otimes_\RR\CC$, where
$V_{\hat{\sigma},\RR}=H^1_{\hat{\sigma}}(X)\oplus H^1_{\hat{\sigma}}(\hat{X})$. Choose a $\sigma\in\Sigma$ restricting to $F$ as $\hat{\sigma}$. 
The isometry $\rho_{g_0}$ restricts to $V_{\hat{\sigma},\RR}\otimes_\RR\CC$ as $\exp(\sigma(\sqrt{-q})\Theta_{\hat{\sigma}})$. We get the pure spinor in
$S^+_{\hat{\sigma}}$ corresponding to the element $\exp(\sigma(\sqrt{-q})\Theta_{\hat{\sigma}})$. Given a CM-type $T$, we get the complex numbers
$[T(\hat{\sigma})(\sqrt{-q})]$ and 
the product in $S^+_\CC$
\begin{equation}
\label{eq-set-of-pure-spinnors-spanning-B}
\bigwedge_{\hat{\sigma}\in \hat{\Sigma}}\exp\left([T(\hat{\sigma})(\sqrt{-q})]\Theta_{\hat{\sigma}}\right)=
\exp\left(\sum_{\hat{\sigma}\in \hat{\Sigma}}[T(\hat{\sigma})(\sqrt{-q})]\Theta_{\hat{\sigma}}
\right)
\end{equation}
is a pure spinor spanning $\ell_T$. The secant space $B\otimes_\QQ\CC$ is the span of all these pure spinors in $S^+_\CC$. Equivalently, $B\otimes_\QQ\RR$ is the product $\otimes_{\hat{\sigma}\in\hat{\Sigma}}P_{\hat{\sigma}}$ in $S^+_\RR$, where 
\[
P_{\hat{\sigma}}\otimes_\RR\CC:=\span_\CC\{\exp(\sigma(\sqrt{-q})\Theta_{\hat{\sigma}}), \exp(\bar{\sigma}(\sqrt{-q})\Theta_{\hat{\sigma}})\}
\]
is a plane in $S^+_{\hat{\sigma}}$. The equality $\tau(P_{\hat{\sigma}})=P_{\hat{\sigma}}$, where $\tau$ is the dualization involution  as in (\ref{eq-Mukai-pairing}), implies the equality
\[
\tau(B)=B.
\]

We claim that $B$ is spanned by Hodge classes. It suffices to prove that 
the $2$-form $\Theta_{\hat{\sigma}}$ is a $(1,1)$-form, for every $\hat{\sigma}\in\hat{\Sigma}$. Now, the $F$ action on $H^1(X,\QQ)$ preserves the Hodge structure, by assumption.
Hence, $H^{1,1}(X,\CC)$ is $F^\times$-invariant, where the $F^\times$-action is via $\wedge^2\hat{\eta}$. The two form $\Theta$ is of type $(1,1)$, by assumption. Hence, $\Theta=\sum_{\{\hat{\sigma}_1,\hat{\sigma}_2\}\subset\hat{\Sigma}}\Theta_{\hat{\sigma}_1\otimes\hat{\sigma}_2}$, where 
the summand $\Theta_{\hat{\sigma}_1\otimes\hat{\sigma}_2}$ is in the isotypic subspace for the character $\hat{\sigma}_1\otimes\hat{\sigma}_2$
of $F^\times$ and is of type $(1,1)$. Our choice of $\Theta$ in $\wedge^2_FH^1(X,\CC)$ means that $\Theta_{\hat{\sigma}_1\otimes\hat{\sigma}_2}$ is non-zero, if and only if $\hat{\theta}_1=\hat{\theta}_2$, and the summand $\Theta_{\hat{\sigma}}$ in $\wedge^2H^1_{\hat{\sigma}}(X)$ is the summand 
$\Theta_{\hat{\sigma}\otimes\hat{\sigma}}$. Hence, indeed $\Theta_{\hat{\sigma}}$ is of type $(1,1)$.

\begin{example}
\label{example-K-is-extention-of-F-by-sqrt-of-a-negative-rational-number}
Consider the special case where the element $q$ of $F$ actually belongs to $\QQ$ (and is positive as $K$ is totally complex), though $e/2:=\dim_\QQ(F)>1$. 
Choose $\sqrt{-q}$ to be the square root in the upper half plane. 
The set (\ref{eq-set-of-pure-spinnors-spanning-B}) of $2^{e/2}$ pure spinors becomes
\[
\exp\left(\sqrt{-q}\sum_{\hat{\sigma}\in \hat{\Sigma}}\pm\Theta_{\hat{\sigma}}\right).
\]
Two of those, namely $\exp\left(\pm\sqrt{-q}\Theta\right)$, are defined over the quadratic imaginary number field $K_1:=\QQ(\sqrt{-q})$ contained in $K$ and 
span a secant plane $P_{K_1}$ in $S^+_\CC$ defined over $\QQ$.
\end{example}

\begin{example}
Following is a simple example of a class $c$ in $B\otimes B$ satisfying the hypothesis of Proposition \ref{prop-kappa-class-of-image-of-secant-class-yields-a-HW-class} for a CM-field $K$ with $[\QQ:K]=4$.
Keep the notation of Example \ref{example-K-is-extention-of-F-by-sqrt-of-a-negative-rational-number} and assume furthermore that $F=\QQ(\sqrt{t})$, where $t$ is a positive integer, which is not a perfect square and is relatively prime to $q$. Then $Gal(K/\QQ)$ is isomorphic to $(\ZZ/2\ZZ)^2$. Let $\gamma\in Gal(K/\QQ)$ send $\sqrt{t}$ to $-\sqrt{t}$ and
$\sqrt{-q}$ to itself. Then $\gamma$ interchanges $\Theta_{\hat{\sigma}_1}$ and $\Theta_{\hat{\sigma}_2}$, since $\Theta=\Theta_{\hat{\sigma}_1}+\Theta_{\hat{\sigma}_2}$ is $\gamma$-invariant. Hence,
$\sqrt{t}(\Theta_{\hat{\sigma}_1}-\Theta_{\hat{\sigma}_2})$ is rational, 
$\sqrt{-q}(\Theta_{\hat{\sigma}_1}-\Theta_{\hat{\sigma}_2})$ is $\iota\circ\gamma$ invariant, and the pure spinors 
$\exp(\pm \sqrt{-q}(\Theta_{\hat{\sigma}_1}-\Theta_{\hat{\sigma}_2}))$ are defined over the imaginary quadratic number subfield $K_2:=\QQ(\sqrt{-tq})$ of $K$.
Let $P_{K_2}$ be the plane spanned by these two spinors.
The $\QQ$-linear  $Gal(K/\QQ)$-action on the second factor $K$ of $B\otimes_\QQ K$ permutes\footnote{This permutation action on 
the four $K$-basis elements of $B\otimes_\QQ K$
defines 
a new $Gal(K/\QQ)$-action on $B\otimes_\QQ K$ via $K$-linear transformations. The latter is not natural, as it changes if we rescale the basis elements by a scalar in $K$.
} 
the four pure spinors (\ref{eq-set-of-pure-spinnors-spanning-B}). 
This is a lift of the action of $Gal(K/\QQ)$ on the set $\T_K$ of CM-types in the sense that the map $T\mapsto \ell_T$ is $Gal(K/\QQ)$-equivariant.
The set $\T_K$ consists of two $Gal(K/\QQ)$-orbits.
%The secant space $B$ is $4$-dimensional, 
The two latter pure spinors constitute the orbit of elements fixed by $\iota\circ\gamma$. 
%subspace $(B\otimes_\QQ K)^{\iota\circ\gamma}$, 
The secant plane $P_{K_1}$ in Example \ref{example-K-is-extention-of-F-by-sqrt-of-a-negative-rational-number} is 
spanned by the $Gal(K/\QQ)$-orbit of pure spinors fixed by $\gamma$. 

Let $\alpha_i$ be a rational class of non-zero rank in $P_{K_i}$, $i=1,2$, and set $c:=\alpha_1\otimes \alpha_2$. We claim that $c$ satisfies the hypothesis of Proposition \ref{prop-kappa-class-of-image-of-secant-class-yields-a-HW-class}, namely 
$c_1$ does not belong to $KB_1$. Indeed,
let $\T_K^\gamma$ be the set of CM-types fixed by $\gamma$ and define $\T_K^{\iota\circ\gamma}$ similarly. For each $\sigma\in\Sigma$ there exists a unique CM-type $T\in \T_K^\gamma$ of which $\sigma$ is a value. The same holds for $\T_K^{\iota\circ\gamma}$. 
Hence, the left hand side of each of the linear homogeneous equations of $KB_1$ in Equation (\ref{eq-explicit-equations-for-KB-1}), for each $\sigma\in\Sigma$, is a sum consisting of precisely one non-zero term $c_{(T,T')}$.
\EndProof
\end{example}

%**********
% End Hide
%**********
}

%*************
% Hide
%*************
\hide{

%****************************************************************
% 
%****************************************************************
\section{The moduli space $\M(2,0,n\Theta^2)$}
\label{subsec-n-equal-2-case}

Let $(X,\Theta)$ be a principally polarized abelian threefold with a cyclic Neron-Severi group.
A non-empty subset of the moduli space  $\M(2,0,n\Theta^2)$, $n\geq 1$, of rank $2$ slope-stable sheaves with trivial determinant and $c_2(E)=n\Theta^2$, consist of rank $2$ vector bundles, which are the extension
\begin{equation}
\label{eq-F-a}
0\rightarrow \StructureSheaf{X}\RightArrowOf{\iota_a} F_a \RightArrowOf{j_a} \Ideal{Y_a}(2\Theta)\rightarrow 0,
\end{equation}
where $Y_a$ is the disjoint union $\cup_{i=1}^nY_{a_i}$, and $Y_{a_i}=\Theta_{a_i}\cap \Theta_{-a_i}$, and
$a_i$, $1\leq i\leq n$ are generic points of $X$, by \cite[Theorem 2.3]{gulbrandsen}. 
The canonical line bundle of $Y_a$ is isomorphic to $\StructureSheaf{Y_{a}}(2\Theta)$.
Hence, $\deg(\omega_{Y_{a_i}})=\int_X2\Theta^3=12$ and each $Y_{a_i}$ is a  smooth curve of genus $7$. 

Consider, for simplicity, the case $n=2$. Given $x\in X$, set $\Theta_x:=\tau_x(\Theta)$.
Then $F_a$ is the middle cohomology of a complex $K_a$  given by 
\begin{equation}
\label{eq-monad}
\StructureSheaf{X}\RightArrowOf{\phi_a} 
\oplus_{i=1}^2 [\StructureSheaf{X}(\Theta_{-a_i})\oplus \StructureSheaf{X}(\Theta_{a_i})]
\RightArrowOf{\psi_a} \StructureSheaf{X}(2\Theta),
\end{equation}
with $\phi_a$ injective and $\psi_a$ surjective \cite[Prop. 3.2]{gulbrandsen}. Give $x\in X$, choose $\theta_{x}$ to be a non-zero section of $\StructureSheaf{X}(\Theta_x)$. The morphisms in the above complex are 
$\phi_a=(\theta_{-a_1},\theta_{a_1},\theta_{-a_2},\theta_{a_2})$ and
\[
\psi_a=(\theta_{a_1},-\theta_{-a_1},\theta_{a_2},-\theta_{-a_2}), 
\]
so that 
$\psi_a(c_1\theta_{-a_1},c_2\theta_{a_1},c_3\theta_{-a_2},c_4\theta_{a_2})=(c_1-c_2)\theta_{a_1}\theta_{-a_1}+(c_3-c_4)\theta_{a_2}\theta_{-a_2}$. 
The space 
$\Ext^2(F_a,\tau_x^*(F_b)\otimes L^{-1})$ is then the second hypercohomology group of the complex 
\[
\C_\bullet :=  \SheafHom(K_a,\tau_x^*(K_b)\otimes L).
\]
Note that $H^2(\C_0)$ vanishes, for generic $(x,L)$, since $H^2(L)=0$ for non-trivial $L$, and $H^2(L(\tau_x^*(\Theta_{\pm b_i})-\Theta_{\pm a_i}))$ vanishes if $L(\tau_x^*(\Theta_{\pm b_i})-\Theta_{\pm a_i})$ is non-trivial.
Furthermore, $H^1(\C_1)$ vanishes, for every $L$ and $x$, by Kodaira's Vanishing Theorem.
Note that $\C_2=\SheafHom(\StructureSheaf{X},L(2\Theta_{-x}))$, and so the space of degree $2$ complex homomorphisms $H^0(\C_2)$ is $8$ dimensional.

\begin{lem}
\label{lemma-rank-0}
For generic $(x,L)$ and $(a,b)$ all degree $2$ complex homomorphisms in $H^0(\C_2)$ are homotopic to zero. 
\end{lem}

\begin{proof}
Let $f_i^{\pm}$ be a non-zero section of $H^0(\tau_x^*(L(\Theta_{\pm b_i})))=H^0(L(\Theta_{-x\pm b_i}))$.
Let $g_i^{\pm}$ be a non-zero section of $H^0(L(2\Theta_{-x}-\Theta_{\pm a_i}))$.
A section $s\in H^0(L(2\Theta_{-x}))$ is homotopic to zero as a degree $2$ complex homomorphism, if and only if there exist homomorphisms 
\[
f=(c_1f_1^-,c_2f_1^+,c_3f_2^-,c_4f_2^+)\in \Hom(\StructureSheaf{X},\oplus_{i=1}^2[L(\Theta_{-x-b_i})\oplus L(\Theta_{-x+b_i})])
\]
and
\[
g=(d_1g^+_1,d_2g^-_1,d_3g^+_2,d_4g^-_2)\in
\Hom(\oplus_{i=1}^2 [\StructureSheaf{X}(\Theta_{-a_i})\oplus \StructureSheaf{X}(\Theta_{a_i})],L(2\Theta_{-x}))
\]
satisfying $s=\tau_x^*(\psi_b)\circ f+g\circ \phi_a$. Now,
\[
\tau_x^*(\psi_b)\circ f=
c_1f_1^-\tau_x^*(\theta_{b_1})-c_2f_1^+\tau_x^*(\theta_{-b_1})+
c_3f_2^-\tau_x^*(\theta_{b_2})-c_4f_2^+\tau_x^*(\theta_{-b_2})
\]
and
\[
g\circ \phi_a=
d_1g_1^+\theta_{-a_1}+d_2g_1^-\theta_{a_1}+d_3g_2^+\theta_{-a_2}+d_4g_2^-\theta_{a_2}.
\]
We see that $\tau_x^*(\psi_b)\circ f+g\circ \phi_a$ is a linear combination of the $8$ sections 
\begin{equation}
\label{eq-8-sections}
\{f_1^-\tau_x^*(\theta_{b_1}), f_1^+\tau_x^*(\theta_{-b_1}), f_2^-\tau_x^*(\theta_{b_2}), 
f_2^+\tau_x^*(\theta_{-b_2}), 
g_1^+\theta_{-a_1}, g_1^-\theta_{a_1}, g_2^+\theta_{-a_2}, g_2^-\theta_{a_2}
\}
\end{equation}
of $L(2\Theta_{-x})$. 

Let $\ell\in X$ be the point, such that $L$ is isomorphic to $\StructureSheaf{X}(\Theta_\ell-\Theta)$.
We have the isomorphisms
$L(2\Theta_{-x}-\Theta_{a_i})\cong 
\StructureSheaf{X}((\Theta_{\ell-x}-\Theta_{-x})+2\Theta_{-x}-\Theta_{a_i})\cong
\StructureSheaf{X}(\Theta_{\ell-a_i-2x})$. Hence, we can take $g^{\pm}_i=\theta_{\ell\mp a_i-2x}$.
Similarly, $f_i^\pm=\theta_{\ell-x\pm b_i}$. The eight sections displayed above are thus
\begin{eqnarray*}
\{
\theta_{\ell-x-b_1}\theta_{-x+b_1},\theta_{\ell-x+b_1}\theta_{-x-b_1},
\theta_{\ell-x-b_2}\theta_{-x+b_2}, \theta_{\ell-x+b_2}\theta_{-x-b_2},
\\
\theta_{\ell-a_1-2x}\theta_{a_1}, \theta_{\ell+a_1-2x}\theta_{-a_1},
\theta_{\ell-a_2-2x}\theta_{a_2}, \theta_{\ell+a_2-2x}\theta_{-a_2}
\}
\end{eqnarray*}
Consider the Kummer morphism $\kappa:X\rightarrow |L(2\Theta_{-x})|$ given by
$\kappa(y)=\theta_{\ell-x-y}\theta_{-x+y}$. It has degree $2$ and satisfies $\kappa(y)=\kappa(\ell-y)$.
Then the set of lines spanned by the above eight sections correspond to the points
\begin{equation}
\label{eq-8-points-in-the-image-of-the-kummer-map}
\{
\kappa(b_1),\kappa(-b_1),\kappa(b_2),\kappa(-b_2),\kappa(a_1+x),\kappa(-a_1+x),\kappa(a_2+x),\kappa(-a_2+x).
\}
\end{equation}
These $8$ sections are linearly independent, hence a basis, for generic $(x,L)$ and $(a,b)$, since the image of $\kappa$ in non-degenerate
(see the proof of \cite[Lemma 5.1]{gulbrandsen}).
\end{proof} 

Let $\ell\in X$ be the point, such that $L$ is isomorphic to $\StructureSheaf{X}(\Theta_\ell-\Theta)$.

\begin{lem}
\label{lemma-relative-extension-sheaf-does-not-vanish}
The sheaf $\G_{F_a,F_b}$ has rank $0$ and its support has co-dimension $1$. Its set theoretic support contains the following loci:
\begin{enumerate}
\item
The locus of points $(x,L)$, where
$x=\pm b_i\pm a_j$, $1\leq i,j\leq 2$ and $\ell$ is arbitrary, a total of 16 three dimensional abelian subvarieties of $X\times \hat{X}$.
\item
\label{lemma-item-second-locus}
 The locus
$x=\ell\pm b_i\pm a_j$, $1\leq i,j\leq 2$, a total of 16 three dimensional abelian subvarieties of $X\times \hat{X}$.
\item 
The locus of points $(x,\StructureSheaf{X})$, where $x$ is arbitrary.
\end{enumerate}
\end{lem}

\begin{proof}
The rank of $\G_{F_a,F_b}$ is zero, by Lemma \ref{lemma-rank-0}.
Set 
\[
L_{\pm b_i}:=\pi_{23,*}\left(
 \pi_{12}^*a^*\StructureSheaf{X}(\Theta_{\pm b_i})
%\oplus \StructureSheaf{X}(\Theta_{b_i})]
\otimes\pi_{13}^*\P
\right)
\]
\[
M_{\pm b_i}:=\pi_{23,*}\left(
 \pi_{12}^*a^*\StructureSheaf{X}(\Theta_{\pm b_i})
%\oplus \StructureSheaf{X}(\Theta_{b_i})]
%\otimes\pi_{13}^*\P
\right)
\]
\[
G_{\pm a_i}:=\pi_{23,*}\left(
 \pi_{12}^*a^*\StructureSheaf{X}(2\Theta)\otimes\pi_{13}^*\P \otimes \pi_1^*\StructureSheaf{X}(-\Theta_{\pm a_i})
 \right)
\]
Set $\Gamma_{\pm a_i}:=H^0(X,\StructureSheaf{X}(\Theta_{\pm a_i}))$.
The support of $\G_{F_a,F_b}$ contains the locus, where the set (\ref{eq-8-points-in-the-image-of-the-kummer-map}) is not linearly independent. The latter locus is the degeneracy locus of the tautological product followed by  evaluation homomorphism
$ev$ between two rank $8$ vector bundles, from 
\[
\oplus_{i=1}^2[(L_{-b_i}\otimes M_{b_i})\oplus (L_{b_i}\otimes M_{-b_i})]
\oplus
\oplus_{i=1}^2 [(G_{a_i}\otimes_\CC  \Gamma_{a_i})\oplus (G_{-a_i}\otimes_\CC  \Gamma_{-a_i})]
\]
%\[
%\pi_{23,*}\left(
%\left\{
%\oplus_{i=1}^2 \pi_{12}^*a^*[\StructureSheaf{X}(\Theta_{-a_i})\oplus \StructureSheaf{X}(\Theta_{a_i})]
%\oplus_{i=1}^2 \pi_{12}^*a^*[\StructureSheaf{X}(\Theta_{-b_i})\oplus \StructureSheaf{X}(\Theta_{b_i})]
%\right\}
%\otimes\pi_{13}^*\P
%\right)
%\]
to
\[
\pi_{23,*}\left(
 \pi_{12}^*a^*[\StructureSheaf{X}(2\Theta)]\otimes\pi_{13}^*\P
\right).
\]
Hence, if non-empty, the degeneracy locus of $ev$ is a divisor.
The set (\ref{eq-8-points-in-the-image-of-the-kummer-map}) is not linearly independent 
if $x=b_1-a_1$, since then $\kappa(a_1+x)=\kappa(b_1)$, the first and fifth lines are equal. 
Similarly, if $x=\ell-b_1-a_1$, then $\kappa(a_1+x)=\kappa(\ell-b_1)=\kappa(b_1)$ and again the first and fifth lines are equal. The rest of the cases again deal with cases where one of the first four lines is equal to one of the last four lines in the set (\ref{eq-8-points-in-the-image-of-the-kummer-map}).

The group $E_2^{0,2}=E_1^{0,2}=H^2(\C_0)$ is isomorphic to $E_\infty^{0,2}$, since the spectral sequence degenerate at the $E_3$ sheet, by the argument in the proof of \cite[Sec. 5.1]{gulbrandsen}, and the differentials 
$d_1^{-1,2}:E_1^{-1,2}\rightarrow E_1^{0,2}=H^2(\C_0)$ and
$d_2:E_2^{0,2}\rightarrow E_2^{2,1}$ both vanish by Kodaira's Vanishing Theorem and Serre's Duality.
When $L$ is trivial, then $\Ext^2(\StructureSheaf{X},L)=\Ext^2((K_a)_0,(\tau_x^*K_b\otimes L)_0)$ is a $3$-dimensional direct summand of $H^2(\C_0)$. Similarly, when $\ell=2x$,
then  $\Ext^2(\StructureSheaf{X}(2\Theta),L(2\Theta_{-x})=\Ext^2((K_a)_2,(\tau_x^*K_b\otimes L)_2)$ is a direct summand of dimension $3$. 
When $L$ is non-trial, 
\[
H^2(\C_0)\cong\Ext^2(\oplus_{i=1}^2[\StructureSheaf{X}(\Theta_{-a_i})\oplus \StructureSheaf{X}(\Theta_{a_i})],
\oplus_{i=1}^2[\StructureSheaf{X}(\Theta_{\ell-x-b_i})\oplus \StructureSheaf{X}(\Theta_{\ell-x+b_i})]).\]
The locus where $H^2(\C_0)$ does not vanish is equal to the 16 abelian subvarieties in the locus (\ref{lemma-item-second-locus}) above, and the rank of $H^2(\C_0)$ at the generic point of each is $3$.
\end{proof}

(???) 
%In the statement of Lemma \ref{lemma-relative-extension-sheaf-does-not-vanish} 
%we need also to add the locus where  
%the set (\ref{eq-8-points-in-the-image-of-the-kummer-map}) is linearly dependent, 
%even though there aren't any repetitions in the set.
We need to add the locus where $H^3(\C_{-1})$ does not vanish, and which contribute to 
$\Ext^2(F_a,\tau_x^*(F_b)\otimes L)$. For the latter, use the version of Serre's Duality
in \cite[Lemma 4.2]{gulbrandsen} to relate the $E_\infty^{-1,3}$ term of the spectral sequence for $\Ext^2(F_a,\tau_x^*(F_b)\otimes L)$ and  the $E_\infty^{1,0}$ term for $\Ext^1(\tau_x^*(F_b)\otimes L,F_a)\cong
\Ext^1(F_b,\tau_{-x}^*(F_a)\otimes L^{-1})$. Here we deal with degree $1$ morphisms of complexes, and the condition that a pair of homomorphisms 
\[
f\in \Hom(\StructureSheaf{X},\oplus_{i=1}^2[\tau_{-x}^*(L(\Theta_{-b_i}))\oplus \tau_{-x}^*(L(\Theta_{b_i}))])
\]
and 
\[
g\in \Hom(\oplus_{i=1}^2[\StructureSheaf{X}(\Theta_{-a_i})\oplus \StructureSheaf{X}(\Theta_{a_i})],\tau_{-x}^*(L^{-1}(2\Theta)))
\]
 yield a morphism of complexes, again boils down to linear dependence of a set analogous to 
 (\ref{eq-8-sections}), which holds on a codimension $1$ locus.

%****************************************************************
% 
%****************************************************************
\subsection{Some computations, which may not be needed}

\begin{lem}
\label{lemma-cohomologies-of-ideal-sheaf}
$\dim H^i(\Ideal{Y_a}\otimes L)=0$, for $i\neq 2$, and $\dim H^2(\Ideal{Y_a}\otimes L)=6n$.
\end{lem}

\begin{proof}
Use the long exact sequence of sheaf cohomology associated to the short exact sequence
$0\rightarrow\Ideal{Y_a}\otimes L \rightarrow L \rightarrow \restricted{L}{Y_a}$ and the vanishing of $H^i(L)$, for all $i$, to get the isomorphism
$H^2(\Ideal{Y_a}\otimes L)\cong H^1(Y_a,\restricted{L}{Y_a})$. Next use the fact that $Y_a$ consists of $n$ connected components, each of which is a curve of genus $7$ and Riemann-Roch.
\end{proof}

\begin{lem}
$\dim H^i(X,F_b\otimes L)=0$, if $i\neq 1$, and $\dim H^1(X,F_b\otimes L)=6n-8$.
\end{lem}

\begin{proof}
We have the short exact sequence
\[
0\rightarrow L(-2\Theta)\rightarrow F_b(-2\Theta)\otimes L \rightarrow \Ideal{Y_b}\otimes L \rightarrow 0.
\]
The cohomologies $H^i(L(-2\Theta))$ vanish, for $i\neq 3$, and $\dim H^3(L(-2\Theta))=8$. Using the isomorphisms
$H^3(F_b(-2\Theta)\otimes L)\cong H^0(F_b^*(2\Theta)\otimes L^{-1})^*\cong H^0(F_b\otimes L^{-1})^*=0,$
we get the short exact sequence
\[
0\rightarrow H^2(F_b(-2\Theta)\otimes L) \rightarrow H^2(\Ideal{Y_b}\otimes L) \rightarrow H^3(L(-2\Theta))\rightarrow 0.
\]
We get that $\dim H^2(F_b(-2\Theta)\otimes L)=6n-8$, by Lemma \ref{lemma-cohomologies-of-ideal-sheaf}.
The statement follows from the isomorphism $F_b(-2\Theta)\cong F_b^*$ and Serre's Duality.
\end{proof}

\begin{question}
\label{question-dimenstion-over-Y-a}
Note that $\chi\left(F_b\otimes \restricted{L}{Y_a}\right)=0$, so 
\[
\dim H^0(Y_a,F_b\otimes \restricted{L}{Y_a})=\dim H^1(Y_a,F_b\otimes \restricted{L}{Y_a}).
\]
What is that dimension?
\end{question}

\begin{proof}
The monad description of $F_b$ in \cite[Prop. 3.2]{gulbrandsen} should be relevant. 
\end{proof}

\begin{question}
Let $F_a$ and $F_b$ be two rank $2$ slope stable vector bundles of the form (\ref{eq-F-a})
and $L\in \Pic^0(X)$ a non-trivial line-bundle. Compute
$\dim\Ext^1(F_a,F_b\otimes L).$ 
{\bf A partial answer is given in section \ref{subsec-n-equal-2-case} when $n=2$ and the dimension jumps as we vary $L$.}
\end{question}

Let $\epsilon_a\in \Ext^1(\Ideal{Y_a}(2\Theta),\StructureSheaf{X})$ be the extension class of (\ref{eq-F-a}). We have the long exact sequence
\begin{eqnarray}
\label{eq-long-exact-sequence-of-exts}
0&\rightarrow& \Hom(\Ideal{Y_a}(2\Theta),F_b\otimes L) \RightArrowOf{j_a^*}
\Hom(F_a,F_b\otimes L) \RightArrowOf{\iota_a^*}
H^0(F_b\otimes L) \RightArrowOf{\circ\epsilon_a}
\\
\nonumber
&\rightarrow & \Ext^1(\Ideal{Y_a}(2\Theta),F_b\otimes L)  \rightarrow 
\Ext^1(F_a,F_b\otimes L) \rightarrow H^1(F_b\otimes L) \rightarrow 
\\
\nonumber
&\rightarrow & \Ext^2(\Ideal{Y_a}(2\Theta),F_b\otimes L)  \rightarrow 
\Ext^2(F_a,F_b\otimes L) \rightarrow H^2(F_b\otimes L) \rightarrow 
\\
\nonumber
&\rightarrow & \Ext^3(\Ideal{Y_a}(2\Theta),F_b\otimes L)  \rightarrow 
\Ext^3(F_a,F_b\otimes L) \rightarrow H^3(F_b\otimes L) \rightarrow 0
\end{eqnarray}
Note the isomorphism $F_b^*(2\Theta) \cong F_b$. Serre's Duality yields
\[
\Ext^i(\Ideal{Y_a}(2\Theta),F_b\otimes L)^* \cong
\Ext^{3-i}(F_b\otimes L,\Ideal{Y_a}(2\Theta)) \cong
H^{3-i}(F_b\otimes\Ideal{Y_a}\otimes L^{-1}).
\]
%In particular, $\Ext^3(\Ideal{Y_a}(2\Theta),F_b\otimes L)^* \cong H^0(F_b\otimes\Ideal{Y_a}\otimes L^{-1})$, 
%and the latter is a subspace of $H^0(F_b\otimes L^{-1})$, which vanishes.

Consider the long exact sequence
\begin{eqnarray*}
0&\rightarrow & H^0(F_b\otimes\Ideal{Y_a}\otimes L^{-1})\rightarrow H^0(F_b\otimes L^{-1})\rightarrow 
H^0(Y_a,F_b\otimes \restricted{L}{Y_a})
\\
&\rightarrow & 
H^1(F_b\otimes\Ideal{Y_a}\otimes L^{-1})\rightarrow H^1(F_b\otimes L^{-1})\rightarrow 
H^1(Y_a,F_b\otimes \restricted{L}{Y_a})
\\
&\rightarrow & 
H^2(F_b\otimes\Ideal{Y_a}\otimes L^{-1})\rightarrow H^2(F_b\otimes L^{-1})\rightarrow 
0
\\
&\rightarrow & 
H^3(F_b\otimes\Ideal{Y_a}\otimes L^{-1})\rightarrow H^3(F_b\otimes L^{-1})\rightarrow 
0
\end{eqnarray*}
If in Question (\ref{question-dimenstion-over-Y-a}) $\dim H^i(Y_a,F_b\otimes \restricted{L}{Y_a})=0$, for $i=1,2$,
then 
\[
\dim\Ext^1(\Ideal{Y_a}(2\Theta),F_b\otimes L)=\dim H^{2}(F_b\otimes\Ideal{Y_a}\otimes L^{-1})=
\dim H^2(F_b\otimes L^{-1})=0.
\] 
Similarly,
$\dim\Ext^2(\Ideal{Y_a}(2\Theta),F_b\otimes L)=\dim H^{1}(F_b\otimes\Ideal{Y_a}\otimes L^{-1})=
\dim H^1(F_b\otimes L^{-1})=6n-8.$ The following is thus a piece of the long exact sequence
(\ref{eq-long-exact-sequence-of-exts}):
\[
0\rightarrow \Ext^1(F_a,F_b\otimes L) \rightarrow H^1(F_b\otimes L) \RightArrowOf{\epsilon_a\circ}
 \Ext^2(\Ideal{Y_a}(2\Theta),F_b\otimes L)  \rightarrow 
\Ext^2(F_a,F_b\otimes L) \rightarrow 0,
\]
where the middle homomorphism $\epsilon_a\circ$ is between two vector spaces of dimension $6n-8$ and we identify $H^1(F_b\otimes L)$ with $\Ext^1(\StructureSheaf{X},F_b\otimes L)$ when we compose with the extension class $\epsilon_a\in\Ext^1(\Ideal{Y_a}(2\Theta),\StructureSheaf{X})$.

%*************
% End Hide
%*************
}

%***************************************************************************
% 
%***************************************************************************
\section{Abelian sixfolds of Weil type associated to a coherent sheaf with a positive Igusa invariant on an abelian threefold}
\label{sec-Igusa-invariant}
We used an ad hoc construction of secant sheaves on abelian $3$-folds in Lemma \ref{lemma-ch-F-i-is-on-secant-to-spinor-variety}. We provide a more conceptual construction of secant sheaves using results of Igusa. The results of this section are not used in the paper. 

%***************************************************************************
% 
%***************************************************************************
\subsection{The Igusa quartic}
Let $X$ be an abelian $3$-fold, so that $V$ has rank $12$. 
\label{sec-coherent-sheaves-with-positive-Igusa-invariant}
The group $\Spin(V)$ acts naturally on the coordinate polynomial ring $\Sym((S^+_\QQ)^*)$ of the half-spin representation $S^+_\QQ$. 
The ring $\Sym((S^+_\QQ)^*)^{\Spin(V)}$ of invariant polynomials is generated by a single polynomial $J$ of degree $4$, by \cite[Prop. 3]{igusa}. The hyperplane $V(J)$ in $\PP(S^+_\CC)$ defined by $J$
is the tangential variety of the spinor variety \cite[Remark 2.1.1]{abuaf}. In other words, $V(J)$ is the union of lines in $\PP(S^+_\CC)$ tangent to the spinor variety. 

We recall next the explicit formula for Igusa's quartic.
Let $\{e_1, \dots, e_6\}$ be a basis of $H^1(X,\Integers)$ with
$\int_X e_1\wedge e_2\wedge \cdots\wedge e_6=1$.
%and let $\{e_7, \dots e_{12}\}$ be a dual basis of $H^1(\hat{X},\Integers)$. 
Set $[pt_X]:=e_1\wedge e_2\wedge \cdots\wedge e_6$.
For $1\leq i<j\leq 6$, set 
%$e_{ij}:=e_i\wedge e_j$ and
\[
e_{ij}^*:=(-1)^{i+j-1}e_1\wedge \cdots \wedge e_{i-1}\wedge e_{i+1} \wedge \cdots \wedge e_{j-1}\wedge e_{j+1}\wedge \cdots \wedge e_6,
\]
so that $e_i\wedge e_j\wedge e_{ij}^*=[pt_X]$. Given an alternating matrix $A$ of rank $2r$, denote by $Pf(A)$ the Pfaffian of $A$ normalized so that $Pf\left(\begin{array}{cc}0 & I_r\\-I_r & 0\end{array}\right)=1$,
where $I_r$ is the $r\times r$ identity matrix. An element $x\in S^+_\QQ$ is a linear combination
\[
x=x_0+\sum_{i<j}x_{ij}e_i\wedge e_j+\sum_{i<j}y_{ij}e_{ij}^*+y_0[pt_X]
\]
with rational coefficients $x_0$, $x_{ij}$, $y_{ij}$, and $y_0$.
Let $X_{ij}$ be the matrix obtained from $(x_{ij})$ by crossing out the $i$-th and $j$-th rows and columns. 
Define $Y_{ij}$ similarly in terms of $(y_{ij})$.
Then the Igusa quartic is
\begin{equation}
\label{eq-J}
J(x)=
x_0Pf((y_{ij})) + y_0Pf((x_{ij})) + \sum_{i<j}Pf(X_{ij})Pf(Y_{ij})-(1/4)(x_0y_0-\sum_{i<j}x_{ij}y_{ij})^2,
\end{equation}
\cite[Prop. 3]{igusa}.

Given a non-zero element $w\in S^+_\CC$ denote by $[w]\in\PP(S^+_\CC)$ the line spanned by $w$.

\begin{lem}
\label{lemma-secant-to-spinor-variety}
For each $[w]\in \PP(S^+_\CC)\setminus V(J)$ there exists a unique plane
$P_w\subset S^+_\CC$, containing $w$, such that the line $\PP(P_w)$ 
%is a line secant to $Im(\eta)$ and $\PP(P_w)$ 
intersects the even spinor variety along two pure spinors corresponding to two transversal maximal isotropic subspaces $W_i$, $i=1,2$, of $V_\CC$. 
The homomorphism $\rho:\Spin(V_\CC)\rightarrow SO(V_\CC)$ maps the
stabilizer $\Spin(V_\CC)_w$, of an element $w\in S^+_\CC$ spanning $[w]$, isomorphically onto the image of the embedding 
\begin{equation}
\label{eq-embedding-e-of-stabilizer-of-w}
e:SL(W_1)\rightarrow SO(V_\CC)
\end{equation}
%of $SL(W_1)$ in $\Spin(V_\CC)$ 
acting on $W_2$ via the isomorphism $W_1^*\cong W_2$ induced  by the pairing (\ref{eq-pairing-on-V}). If $w$ belongs to $S^+_\QQ$, then the plane $P_w$ is defined over $\QQ$.
\end{lem}

\begin{proof}
The homomorphism $\rho$ restricts to an injective homomorphism from $\Spin(V_\CC)_w$ into $SO(V_\CC)$.
Indeed, the kernel of $\rho$ has order $2$, generated by $-1\in C(V_\CC)$, and the latter
acts as $-id_{S_\CC}$ on the spin representation. Hence, the kernel of $\rho$ intersects $\Spin(V_\CC)_w$ trivially.

Let $\tilde{V}(J)\subset S^+_\CC$ be the cone over $V(J)$.
The complement $S^+_\CC\setminus \tilde{V}(J)$ intersects each fiber of $J:S^+_\CC\rightarrow \CC$ in a single $\Spin(V_\CC)$-orbit, by \cite[Prop. 3]{igusa}. 
Hence, it suffices to prove the statement for one such $w\in S^+_\ComplexNumbers$ in each fiber. 
Let $w=1+d[pt_X]$, $d\neq 0$, where $[pt_X]\in H^6(X,\Integers)$ is the class Poincar\'{e} dual to a point. 
Then $w$ belongs to the plane spanned by the two pure spinors $1$ and $[pt_X]$ and $J(w)=-(1/4)d^2$. 
The kernel of $m_1:V_\CC\rightarrow S^-_\CC$ is the maximal isotropic subspace $W_1:=H^1(\hat{X},\CC)$ and 
the kernel of $m_{[pt_X]}:V_\CC\rightarrow S^-_\CC$ is $W_2:=H^1(X,\CC)$. The description of the stabilizer of 
$w$ as the embedding (\ref{eq-embedding-e-of-stabilizer-of-w}) 
is given in this case in \cite[Lemma 2]{igusa}. Clearly, $W_1$ and $W_2$ are the only two maximal isotropic subspaces invariant under $e(SL(W_1))$. Hence, the two lines spanned by $1$ and $[pt_X]$ are the only pure spinor lines in $S^+_\CC$ stabilized under $\Spin(V_\CC)_w$ and the line
$\PP(\mbox{span}\{1, \ [pt_X]\})$ is the unique line secant to the spinor variety and passing through $w$, every point of which 
is $\Spin(V_\CC)_w$-invariant. We conclude, 
more generally, that for every $w$ with $[w]\in \PP(S^+_\CC)\setminus V(J)$, the  stabilizer $\Spin(V)_w$ fixes precisely two lines of even pure spinors, $[u_1]$, $[u_2]$, 
and that  $w$ belongs to the plane $P_w:=\mbox{span}\{u_1, u_2\}$. Furthermore, setting $W_i:=\ker(m_{u_i})$, 
$i=1,2$, the stabilizer $\Spin(V)_w$ is isomorphic to the stated embedding (\ref{eq-embedding-e-of-stabilizer-of-w}) of $SL(W_1)$. 

Assume next that $w$ belongs to $S^+_\QQ$. Then the stabilizer $\Spin(V_\CC)_w$ is defined over $\QQ$. 
The latter determines $P_w$. Hence, $P_w$ is defined over $\QQ$ as well.
%******
%Hide
%******
\hide{
to 
the pair $\{W_1, W_2\}$ is determined by the stabilizer $\Spin(V)_w$, and hence so does the plain $P_w$ spanned by the 

The secant variety of the spinor variety is the whole of
$\PP(S^+_\CC)$, by dimension count. 
The complement $\PP(S^+_\CC)\setminus V(J)$ is a single $\Spin(V)$-orbit, by \cite[Prop. 3]{igusa}, hence there exists a secant with transversal maximal isotropic subspaces, for every $w\in \PP(S^+_\CC)\setminus V(J)$.
It remains to prove the uniqueness of the plane $P_w$.
Indeed, given one 
such plane $P$, the intersection of $\PP(P)$ with the even spinor variety consists of two points $u_1$ and $u_2$, such that 
the line $\PP(P)$ intersects $V(J)$ along these two points with multiplicity $2$, since the even spinor variety is contained in the singular locus of $V(J)$. Set $W_i:=\ker(m_{u_i})$. The $W_1\cap W_2=(0),$ by assumption.
The stabilizer $\Spin(V)_w$ of $w$ in $\Spin(V)$ is 
isomorphic to the the embedding of $SL(W_1)$ in $\Spin(V)$ 
%leaving each $W_i$ invariant, $i=1,2$, and the 
acting on $W_2$ leaving invariant the isomorphism $W_1\cong W_2^*$ induced  by the pairing (\ref{eq-Mukai-pairing}) \cite[Prop. 3]{igusa}. Hence, $\Spin(V)_w$ fixes $u_1$ and $u_2$ and acts as the identity on the lines $\tilde{u}_i\subset S^+_\CC$ and so on $P$. Now, clearly, the subgroup of 
%******
% End Hide
%******
}
\end{proof}

%The existence of the plane $P_w$ in the  above statement will be explained also via a decomposition of representations in 
%Equation (\ref{eq-restriction-of-positive-half-spin-decomposes}) below.

\begin{rem}
\begin{enumerate}
\item
%The Igusa quartic is defined over $\QQ$. 
For any subfield $F$ of $\CC$, 
the level set $J^{-1}(d)\subset S^+_F$, $d\in F$, consists of a single $\Spin(V_F)$-orbit, by \cite[Prop. 3]{igusa}. 
Furthermore, if $w=1+e_{14}^*+e_{25}^*+de_{36}^*$, $d\in F$, then $J(w)=d$.
\item
Note that the secant $\PP(P_w)$ to the spinor variety in Lemma \ref{lemma-secant-to-spinor-variety} intersects the quartic $V(J)$ along the same two points along which it intersected the spinor variety, each with multiplicity $2$, since the even spinor variety is contained in the singular locus of $V(J)$. The subspace $P_w$ is defined over $\QQ$, and so the length two subscheme of the intersection of $\PP(P_w)$ with the spinor variety is defined over $\QQ$. 
\end{enumerate}
\end{rem}

%***************************************************************************
% 
%***************************************************************************
\subsection{Complex multiplication}
If $w$ belongs to $S^+_\QQ$, then $d:=J(w)$ is a rational number, since $J$ is defined over $\QQ$. Set $K:=\QQ[\sqrt{-d}]$. 
%If $-d$ is a square of a rational number, then $P_w$ is the image of $\mbox{span}\{1,[pt_X]\}$ under an element of $\Spin(V_\QQ)$.
% and we are done.
Assume that $-d$ is not a square of a rational number.
Let $\sigma: K\rightarrow K$ be the involution in $Gal(K/\QQ)$. Denote by $\sigma$ also the induced involution on $S^+_K$, $S^-_K$, and $V_K$. 
%Then $\sigma(g(P_w))=\sigma\left(\mbox{span}\{1,[pt_X]\}\right)=\mbox{span}\{1,[pt_X]\}=g(P_w)$.
%Given $g\in\Spin(V_K)$, set $g^\sigma:=\sigma g\sigma$. 
Let $Nm:K\rightarrow\QQ$ be the norm map $Nm(\lambda)=\lambda\sigma(\lambda)$.
%Then $\sigma(g)(w)=\sigma(g(w))=1-2\sqrt{-d}[pt_X]$ and so  
%$\sigma(g)(P_w)=P_{(\sigma(g))(w)}=P_{g(w)}=g(P_w)=\sigma(g(P_w))$.
%Hence, $\sigma(P_w)=\sigma((\sigma(g^{-1}))(P_{(\sigma(g))(w)}))$

Let $\tilde{O}(V_\QQ)$ be the group of rational similarities of $V_\QQ$, defined in (\ref{eq-group-of-similarities}).
%\[
%\tilde{O}(V_\QQ):=\{g\in GL(V_\QQ) \ : \ (g(v_1),g(v_2))=c(v_1,v_2), \ \mbox{for some} \ c\in \QQ, c>0\}.
%\]

\begin{lem}
\label{lemma-imaginary-quadratic-field-is-centralizer-sixfold-case}
Let $w\in S^+_\QQ$ be a rational class with $d:=J(w)> 0$. Set $K:=\QQ[\sqrt{-d}]$. 
Then the two maximal isotropic subspaces $W_1$ and $W_2$ of $V_\CC$ invariant under $\Spin(V_\QQ)_w$ are defined over $K$, but not over $\QQ$, and $\sigma(W_1)=W_2$. The centralizer of $\rho(\Spin(V_\QQ)_w)$ in $\tilde{O}(V_\QQ)$
is isomorphic to $K^\times$.
%$\QQ^\times$, if the two even pure spinors invariant under $\Spin(V_\CC)_w$ are defined over $\QQ$. 
%In this case $\Spin(V_\QQ)_w$ is isomorphic to $SL(6,\QQ)$.
%Otherwise, $Z$ is an imaginary quadratic number field $K$ and 
\end{lem}

\begin{proof} 
There exists an element $g\in \Spin(V_K)$, such that $g(w)=1+2\sqrt{-d}[pt_X]$, by the last paragraph in the proof of \cite[Prop. 3]{igusa}. Then $P_{g(w)}=\mbox{span}\{1,[pt_X]\}$ and so
$W_1:=g^{-1}(H^1(X,\QQ))$ and $W_2:=g^{-1}(H^1(\hat{X},\QQ))$ are defined over $K$ and are the two maximal isotropic subspaces of $V_K$ invariant under $\Spin(V_\QQ)_w$. The vanishing $W_1\cap W_2=(0)$ follows from the vanishing $H^1(X,\QQ)\cap H^1(\hat{X},\QQ)=(0)$ of the intersection in $V_\QQ$.
%Lemma \ref{lemma-P-is-non-isotropic-iff-W-1-and-W-2-are-transversal}.
The maximal isotropic subspaces $W_1$ and $W_2$ are not defined over $\QQ$, since otherwise the plane $P_w$ would belong to the $\Spin(V_\QQ)$ orbit of $\mbox{span}\{1,[pt_X]\}$ and $J$ would have non-positive values on rational points of $P_w$, by formula (\ref{eq-J}). This would contradict the assumption that $J(w)>0$.
The pair $\{W_1, W_2\}$ is invariant under $\sigma$, since $P_w$ is, by Lemma \ref{lemma-secant-to-spinor-variety}. We conclude that $W_2=\sigma(W_1)$.

The equality $\Spin(V_\QQ)_w=\Spin(V_\QQ)_P$ holds, by Remark \ref{remark-stabilizer-of-w-may-have-two-connected-components}.
Let $\cm:K^\times\rightarrow GL(V_\QQ)$ be the homomorphism given in (\ref{eq-action-by-imaginary-quadratic-number-field}).
We have seen that the centralizer of $\rho(\Spin(V_\QQ)_P)$ in $\tilde{O}(V_\QQ)$
is equal to $\cm(K^\times)$ in Lemma \ref{lemma-centralizer-of-rho-Spin-V-P}.
%***************
% Hide
%***************
\hide{
The subset $V_\QQ$ of $V_K$ is equal to $\{v_1+\sigma(v_1) \ : \ v_1\in W_1\}$.
%Given $\lambda\in K^\times$, 
Let $\cm:K^\times\rightarrow GL(V_\QQ)$ be the homomorphism given in (\ref{eq-action-by-imaginary-quadratic-number-field}).
%act on $W_1$ by multiplication by $\lambda$ and on $W_2$ by multiplication by $\sigma(\lambda)$. 
%Then $\cm_\lambda$ leaves $V_\QQ$ invariant and we get the homomorphism ()
%\begin{equation}
%\label{eq-action-by-imaginary-quadratic-number-field}
%\cm:K^\times\rightarrow GL(V_\QQ)
%\end{equation}
%sending $\lambda$ to the restriction of $\cm_\lambda$ to $V_\QQ$.
The image of $\cm$ clearly centralizes $e(SL(W_1))$, where $e$ is given in (\ref{eq-embedding-e-of-stabilizer-of-w}), and so it centralizes $\rho(\Spin(V_\QQ)_w)$, by Lemma \ref{lemma-secant-to-spinor-variety}. 
Given $v_1, v_1'\in W_1$, set $v=v_1+\sigma(v_1)$ and $v'=v_1'+\sigma(v_1')$. 
%Every element of $V_\QQ$ is of this form. 
Then 
\begin{eqnarray*}
(\cm_\lambda(v),\cm_\lambda(v'))_V&=&(\lambda v_1+\sigma(\lambda)\sigma(v_1),\lambda v_1'+\sigma(\lambda)\sigma(v_1'))_V
\\
&=&
Nm(\lambda)\left[(v_1,\sigma(v_1'))_V+(v_1',\sigma(v_1))_V\right]
=  Nm(\lambda)(v,v').
\end{eqnarray*}
Now, $Nm(\lambda)$ is positive, for all $\lambda\in K^\times$.
Hence, the image of $\cm$ is contained in $\tilde{O}(V_\QQ)$.

Conversely, if $g\in\tilde{O}(V_\QQ)$  centralizes $\rho(\Spin(V_\QQ)_w)$,
then $W_1$ and $W_2$ are $g$ invariant and $g$ acts on $W_i$ via multiplication by a scalar $\lambda_i\in K^\times$. Given $v_1\in W_1$, we have
\[
\lambda_1v_1=g(v_1)=(\sigma g\sigma)(v_1)=\sigma(\lambda_2\sigma(v_1))=\sigma(\lambda_2)v_1.
\]
Hence, $\lambda_2=\sigma(\lambda_1)$ and $g=\cm(\lambda_1)$.
%***************
% End Hide
%***************
}
\end{proof}

Given $w\in S^+_\QQ$ as in Lemma \ref{lemma-imaginary-quadratic-field-is-centralizer-sixfold-case} and an orientation of the unique secant $P$ through $w$ (so a choice of one of the two maximal isotropic subspaces $W_i$, $i=1,2$), we get the injective group homomorphism $\cm:K^\times\rightarrow GL(V_\QQ)$  given in (\ref{eq-action-by-imaginary-quadratic-number-field}). The image of $\cm$ is contained in $\tilde{O}(V_\QQ)$ and is equal to the centralizer of $\rho(\Spin(V_\QQ)_w)$, by Lemma \ref{lemma-centralizer-of-rho-Spin-V-P}. When $P$ is spanned by Hodge classes, then $\cm$ defines a complex multiplication on $X\times\hat{X}$ and a $2$-dimensional subspace of Hodge-Weil classes, by Lemma \ref{lemma-decomposition-into-4-direct-summands} and Corollary \ref{cor-plane-of-Hodge-Weil-classes}.

\begin{example}
\label{example-gulbrandsen}
Moduli spaces of rank $2$ vector bundles with trivial determinant on a principally polarized abelian threefold $(X,\Theta)$ were studied in \cite{gulbrandsen}. Gulbrandsen proves the non-emptiness of the moduli space $\M(2,0,d\Theta^2)$ 
of slope-stable vector bundles $F$ of rank $2$ with $c_1(F)=0$ and $c_2(F)=d\Theta^2$, for every positive integer $d$
\cite[Theorem 2.3]{gulbrandsen}. He also proves that 
the moduli space with $c_2=\Theta^2$ contains a Zariski open subset of dimension $13$ birational to a $\PP^1$ bundle over a finite quotient of
$X^2\times_X X^2\times_X X^2$, where $X^2$ is considered a variety over $X$ via the group operation
\cite[Theorem 6.1]{gulbrandsen}.
The Chern character of a vector bundle $F$ in $\M(2,0,d\Theta^2)$
is $w:=2-d\Theta^2$. We claim that $J(w)=16d^3$, so that $K=\QQ(\sqrt{-d})$. Indeed, there is a basis $\{e_1, e_2, \dots, e_6\}$ of $H^1(X,\Integers)$, such that $\Theta=e_1\wedge e_4+e_2\wedge e_5+ e_3\wedge e_6$. So 
\[
\Theta^2=2\left[e_1\wedge  e_4\wedge  e_2 \wedge  e_5+
e_1\wedge  e_4\wedge  e_3\wedge  e_6+
e_2\wedge  e_5\wedge  e_3\wedge  e_6\right]
= (-2)(e_{36}^*+e_{25}^*+e_{14}^*),
\] 
and $w=2+2d(e_{36}^*+e_{25}^*+e_{14}^*)$, so $J(w)=2^4d^3J(1+(e_{36}^*+e_{25}^*+e_{14}^*))=16d^3$. %Hence,  In that case
%$ch(F)=2-d\Theta^2=2+2d(e_{36}^*+e_{25}^*+e_{14}^*)$, $J(ch(F))=16d^3$, and
%$K=\QQ(\sqrt{-d})$.
Note that $w=2\alpha$, where $\alpha$ is the class in Lemma \ref{lemma-ch-F-i-is-on-secant-to-spinor-variety}.
\end{example}

\begin{example}
The Igusa invariant of the Chern character of the ideal sheaf of a length $n$ zero dimensional subscheme is
$J(1-n[pt_X])=-(1/4)n^2=-(n/2)^2$, and so $K=\QQ$
\end{example}

%**********
% Hide
%**********
\hide{
%***************************************************************************
% 
%***************************************************************************
\section{Some representation theory}
We have the isomorphisms of $\Spin(V_\QQ)$ representations
\[
S_\QQ\otimes S_\QQ \cong \End(S_\QQ)\cong  C(V_\QQ) \cong \oplus_{i=0}^{12}\wedge^iV_\QQ.
\]
The exterior power $\wedge^iV_\QQ$ is irreducible, for $i\neq 6$, while 
$\wedge^6V_\QQ$ is the direct sum of two irreducible representations of $\Spin(V_\QQ)$, 
$\wedge^6_+V_\QQ$ spanned by top exterior powers of even maximal isotropic subspaces of $V_\QQ$
and $\wedge^6_-V_\QQ$ spanned by top exterior powers of odd maximal isotropic subspaces of $V_\QQ$
\cite[III.4.5]{chevalley}.
The tensor square $S_\QQ\otimes S_\QQ$ decomposes into the direct sum of $S^+_\QQ\otimes S^+_\QQ$, 
$S^-_\QQ\otimes S^-_\QQ$, each decomposing further into symmetric and alternating squares, 
$S^+_\QQ\otimes S^-_\QQ$, and $S^-_\QQ\otimes S^+_\QQ$. The decomposition of each into irreducible representations is given below \cite[Sec. III.3.4]{chevalley}.
\begin{equation}
\label{eq-decomposition-of-sym-2-S-plus-into-spin-12-irreducible-reps}
\Sym^2(S^+_\QQ) \cong \wedge^2V_\QQ\oplus \wedge^6_+V_\QQ.
\end{equation}
\[
\Sym^2(S^-_\QQ) \cong \wedge^2V_\QQ\oplus \wedge^6_-V_\QQ.
\]
\begin{equation}
\label{eq-wedge-2-S-+}
\wedge^2S^+_\QQ \cong \wedge^2S^-_\QQ \cong 1\oplus \wedge^4V_\QQ.
\end{equation}
\[
S^+_\QQ\otimes S^-_\QQ \cong V_\QQ\oplus\wedge^3V_\QQ\oplus \wedge^5V_\QQ.
\]
%***************************************************************************
% 
%***************************************************************************
\subsection{Restriction to the stabilizer $\Spin(V_\QQ)_w$}
Keep the notation of Lemma \ref{lemma-imaginary-quadratic-field-is-centralizer-sixfold-case}.
The representation $V_K$ of $\Spin(V_K)_w$ decomposes as
\[
V_K\cong W_1\oplus W_2,
\]
and $W_2\cong W_1^*$.
The half-spin representation $S^+_\QQ$ admits the filtration 
\[
\mbox{span}\{w\}\subset w^\perp\subset S^+_\QQ, 
\]
where $w^\perp$ is the symplectic-orthogonal subspace with respect to the pairing (\ref{eq-Mukai-pairing}). The plane $P_w\subset S^+_\QQ$ contains $w$ and it is not contained in $w^\perp$, since the two even maximal isotropic subspaces $W_1$ and $W_2$ are transversal, and so their even pure spinors pair non-trivially, by \cite[III.2.4]{chevalley}.
Hence, $P_w$ projects onto $S^+_\QQ/w^\perp$, 
yielding the isomorphism $P_w^\perp\cong w^\perp/\mbox{span}\{w\}$ and the 
decomposition into a direct sum of  $\Spin(V_\QQ)_w$-representations
\begin{equation}
\label{eq-restriction-of-positive-half-spin-decomposes}
S^+_\QQ\cong P_w\oplus (w^\perp/\mbox{span}\{w\}).
\end{equation}
The half-spin representation $S^+_K$ is isomorphic to the even exterior algebra of each maximal isotropic subspace of $V_K$. Hence, 
over $K$, the summand $w^\perp/\mbox{span}\{w\}$ further decomposes as the direct sum $\wedge^2W_1\oplus \wedge^4 W_1$ of two dual representations.
\[
S^+_K\cong P_w\oplus \wedge^2W_1\oplus \wedge^4W_1.
\]
We conclude that the dimension of the subspace of $\Sym^2(S^+_K)$ of $\Spin(V_\QQ)_w$--invariant vectors is $4$, the sum of the $3$-dimensional $\Sym^2(P_w)$ and a one-dimensional subspace of $\wedge^2W_1\otimes \wedge^4W_1$. Clearly, the direct summand $\wedge^2V_K$ in 
(\ref{eq-decomposition-of-sym-2-S-plus-into-spin-12-irreducible-reps}) contributes a one-dimensional subspace of $\Spin(V_\QQ)_w$-invariant vectors. Hence, the subspace of $\wedge^6_+V_K$ of $\Spin(V_\QQ)_w$--invariant vectors is 
$3$-dimensional.

We have
\[
\wedge^6V_K\cong \oplus_{i=0}^6\left[\wedge^iW_1\otimes\wedge^{6-i}W_1^*\right].
\]
The trivial representation of $\Spin(V_K)_w$ appears with multiplicity $3$ as the summands
$\wedge^0W_1\otimes\wedge^6W_1^*$, 
$\wedge^6W_1\otimes\wedge^0W_1^*$, and once in
$\wedge^3W_1\otimes\wedge^3W_1^*\cong\End(\wedge^3W_1).$
Hence, every $\Spin(V_K)_w$-invariant vector in $\wedge^6V_K$ belongs to $\wedge^6_+V_K$.
This $3$-dimensional trivial subrepresentation of $\wedge^6V_K$ is defined over $\QQ$, since the direct sum 
$[\wedge^0W_1\otimes\wedge^6W_2]\oplus 
[\wedge^6W_1\otimes\wedge^0W_2]$ is, as is the third power of the invariant line  in $\wedge^2V_\QQ$ dual to the one spanned by $(f(\bullet),\bullet)$ in $\wedge_2V_\QQ^*$, where $f$ is given in Equation (\ref{eq-f}). 

The second partial $J_{ww}$ of the Igusa quartic with respect to $w$ is an element of 
$(\Sym^2S^+_\QQ)^*\cong \Sym^2(S^+_\QQ)$. 
%If, for example, $w=1+[pt_X]$, then
%\[
%J_{ww}(x)=\left(\frac{\partial}{\partial x_0}+\frac{\partial}{\partial y_0}\right)^2J(x)=\sum_{i<j}x_{ij}y_{ij}-2x_0y_0-%\frac{1}{2}(x_0^2+y_0^2).
%\]

\begin{lem}
\label{lemma-hat-J}
The projection $\hat{J}_{ww}$ of $J_{ww}$ to the direct summand $\wedge^6_+V_\QQ$ in 
(\ref{eq-decomposition-of-sym-2-S-plus-into-spin-12-irreducible-reps}) is non-trivial. 
The subset $\{\hat{J}_{ww},\Theta_w^3\}$ of $\wedge^6_+V_\QQ$ is linearly independent.
\end{lem}

\begin{proof}
We have the $\Spin(V_\QQ)$-equivariant isomorphism $J:\Sym^2(S^+_\QQ)\rightarrow \Sym^2(S^+_\QQ)^*$ and the isomorphism $\Sym^2(S^+_\QQ)^*\cong \Sym^2(S^+_\QQ)$ induced by the bilinear pairing (\ref{eq-Mukai-pairing})
The subspace $\wedge^6_+ V_\QQ$ of $\Sym^2(S^+_\QQ)$ is spanned by squares of pure spinors \cite[III.3.2 and III.4.5]{chevalley}.
Taking partials with respect to the coordinate $x_0$  of $(S^+_\QQ)^*$ 
corresponds to contraction with a pure spinor associated to the maximal isotropic subspace $H^1(X,\QQ)$ of $V_\QQ$. The second partial  $J_{x_0x_0}=-y_0^2/2$ does not vanish, and so $J$ restricts to an isomorphism
from $\wedge^6_+ V_\QQ$ onto $(\wedge^6_+ V_\QQ)^*$, as the two direct summands in 
(\ref{eq-decomposition-of-sym-2-S-plus-into-spin-12-irreducible-reps}) are distinct irreducible $\Spin(V_\QQ)$-representations. The non-vanishing of $\hat{J}_{ww}$ would follow once we 
prove that the element $w^2$  projects to a non-zero element in $\wedge^6_+ V_\QQ$.
The bilinear pairing (\ref{eq-Mukai-pairing})
induces a $\Spin(V_\QQ)$-invariant pairing on $\Sym^2(S^+_\QQ)$ with respect to which the direct summands of 
(\ref{eq-decomposition-of-sym-2-S-plus-into-spin-12-irreducible-reps}) are orthogonal. Hence, it suffices to prove that $w^2$ pairs non-trivially with the square of some even pure spinor.
Indeed, $w=au_1+bu_2$, where $u_1$ and $u_2$ are even pure spinors and $a,b$ are non-zero elements of $K$ and so $(w^2,u_2^2)=(a^2u_1^2,u_2^2)+(2abu_1u_2,u_2^2)+(b^2u_2^2,u_2^2)=a^2(u_1,u_2)_S^2$
and $(u_1,u_2)_S$ does not vanish, since $W_1$ and $W_2$ are transversal \cite[III.2.4]{chevalley}.

We prove next that $\{\hat{J}_{ww},\Theta_w^3\}$ is linearly independent.
Let $\{e_1, \dots, e_6\}$ be a basis of $W_2$. We get the non-zero element  $e_{W_2}:=e_1\wedge \cdots \wedge e_6$ of $\wedge^6W_2$.
The bilinear pairing (\ref{eq-pairing-on-V}) induces a symmetric bilinear pairing on $\wedge^iV_\QQ$,
which we denote by $(\bullet,\bullet)_V$ as well.
%On $\wedge^6_+V_\QQ$ this pairing must be proportional to the wedge product pairing 
%$\wedge:\Sym^2(\wedge^6_+V_\QQ) \rightarrow \wedge^{12}V_\QQ\cong \QQ$, 
%since $\wedge^6_+V_\QQ$ is an irreducible representation of $\Spin(V_\QQ)$. 
%We conclude that $(\Theta_w^3,\Theta_w^3)_V\neq 0$, since $\Theta_w^6\neq 0$.
%It remains to prove the vanishing of $(\Theta_w^3,\hat{J}_{ww})_V$.
On $\wedge^6_+V_\QQ$ this pairing must be proportional to the one considered above and induced from $\Sym^2S^+_\QQ$ by 
(\ref{eq-Mukai-pairing}) and (\ref{eq-decomposition-of-sym-2-S-plus-into-spin-12-irreducible-reps}), since $\wedge^6_+V_\QQ$ is an irreducible representation of $\Spin(V_\QQ)$. 
We conclude that $(\hat{J}_{ww},e_{W_2})_V\neq 0$, since we  have seen that $(w^2,u_2^2)$ does not vanish.
Now $(\Theta_w^3,e_{W_2})_V$ is proportional to the product
$(\Theta_w,e_1\wedge e_2)_V(\Theta_w,e_3\wedge e_4)_V(\Theta_w,e_5\wedge e_6)_V$, which vanishes, since 
$(\Theta_w,e_i\wedge e_j)_V$ is proportional to $\Theta_w(e_i,e_j)=(f(e_i),e_j)_V$, which in turn vanishes, since 
$f(e_i)=-\sqrt{-q}e_i$ and $W_2$ is an isotropic subspace of $V_K$.
\end{proof}

\begin{rem}
Clearly, the element $\hat{J}_{ww}$ of $\wedge^6_+V_\QQ$ is $\Spin(V_\QQ)_w$-invariant.
Note that $\hat{J}_{ww}$ is defined over $\QQ$.  
%It should play a role similar to the role the Cayley class plays for abelian fourfolds of Weil type.
\end{rem}

\begin{rem}
Consider more generally the partials of $J$ with respect to elements of $\Sym^2P_w$ and relate them to the three dimensional $\Spin(V)_w$-invariant subspace of $\wedge^6V_\QQ$. The latter has the canonical line spanned by $\Theta_w^3$ and the former the canonical line spanned by the product of the two pure spinors $u_1$ and $u_2$ in $P_w$. So $\Theta_w^3$ should be proportional to $\hat{J}_{u_1u_2}$ (???)
\end{rem}

%**********
% End Hide
%**********
}

%***************************************************************************
% 
%***************************************************************************
%******
% Hide
%*****
\hide{

\subsection{Miscellaneous}
We get also the decomposition
\[
\wedge^2S^+_\QQ\cong \wedge^2P_w\oplus [P_w\otimes P_w^\perp]\oplus \wedge^2(P_w^\perp).
\]
Now, the line $\wedge^2P_w$ is different from the trivial $\Spin(V_\QQ)$ subrepresentation
appearing in (\ref{eq-wedge-2-S-+}) and hence project non-trivially to $\wedge^4 V_\QQ$.
The summand $\wedge^2(P_w^\perp)$ is isomorphic to its dual, which contains the restriction
of the pairing  (\ref{eq-Mukai-pairing}). Hence, it decomposes as 
\[
\wedge^2(P_w^\perp)\cong 1\oplus \wedge^2_0(P_w^\perp),
\]
where 
$\wedge^2_0(P_w^\perp)$ is the intersection $\wedge^2(P_w^\perp)\cap \wedge^4 V_\QQ$ in $\wedge^2S^+_\QQ$.
We get a decomposition
\[
\wedge^4 V_\QQ\cong 1\oplus \wedge^2_0(P_w^\perp)\oplus [P_w\otimes P_w^\perp].
\]

As a representation of the stabilizer $\Spin(V_\QQ)_w$, the space
$\wedge^6_+V_\QQ$ contains the two dimensional trivial representation $P$, such that 
$P\otimes_\QQ K=\wedge^6W_1+\wedge^6 W_2.$
%******
% End Hide
%*****
}

%****************************************************************
% 
%****************************************************************
\section{Appendix}
\label{appendix}

Let $X$ be a smooth $3$-dimensional variety and $C$ and $\Sigma$ smooth curves on $X$, such that the subscheme $C\cap \Sigma$ is zero-dimensional. Let
\begin{equation}
\label{eq-short-exact-sequence-of-ideal-sheaf-of-C}
0\rightarrow \Ideal{C}\rightarrow\StructureSheaf{X}\rightarrow \StructureSheaf{C}\rightarrow 0
\end{equation}
be the short exact sequence of the ideal sheaf of $C$.

\begin{lem}
\label{lemma-vanishing-of-tor-sheaves}
\begin{enumerate}
\item
\label{lemma-item-Tor-1-embeds-in-tensor-product}
Let $\Ideal{C,C\cap \Sigma}$ be the ideal sheaf of $C\cap \Sigma$ as a subscheme of $C$. The following sequence is short exact.
\[
0\rightarrow \SheafTor_{-1}(\StructureSheaf{C},\StructureSheaf{\Sigma})
\rightarrow \StructureSheaf{C}\otimes \Ideal{\Sigma}\rightarrow \Ideal{C,C\cap \Sigma} \rightarrow 0.
\]
In particular, the  sheaf $\StructureSheaf{C}\otimes \Ideal{\Sigma}$ has a $1$-dimensional support and a subsheaf with a $0$-dimensional support.
\item
The torsion sheaves $\SheafTor_{-k}(\Ideal{C},\Ideal{\Sigma})$ vanish, for $k>0$, and tensoring (\ref{eq-short-exact-sequence-of-ideal-sheaf-of-C}) with $\Ideal{\Sigma}$ we get the short exact sequence
\begin{equation}
\label{eq-short-exact-sequence-of-tensor-product-of-two-ideal-sheaves}
0\rightarrow \Ideal{C}\otimes\Ideal{\Sigma}\rightarrow \Ideal{\Sigma}\rightarrow \StructureSheaf{C}\otimes\Ideal{\Sigma}\rightarrow 0.
\end{equation}
\end{enumerate}
\end{lem}

The lemma implies that $\Ideal{C}\otimes\Ideal{\Sigma}$ is the ideal sheaf of a subscheme $Z$ 
with embedded points along $C\cap\Sigma$ whose reduced induced subscheme is $C\cup\Sigma$ and $\Ideal{C}\otimes\Ideal{\Sigma}$ represents also the left derived tensor product.

\begin{proof}
\underline{Step 1:}
We prove part (\ref{lemma-item-Tor-1-embeds-in-tensor-product}) as well as that the sheaf $\SheafTor_{-k}(\StructureSheaf{C},\StructureSheaf{\Sigma})$ vanishes, for $k\geq 2$.
Let $p$ be a point of intersection of $C$ and $\Sigma$. Let $x$, $y$ be a regular sequence in the stalk $\Ideal{C,(p)}$. Consider the Koszul complex resolving the stalk $\StructureSheaf{C,(p)}$.
\[
0\rightarrow \StructureSheaf{X,(p)}\RightArrowOf{(y,-x)} \StructureSheaf{X,(p)}\oplus \StructureSheaf{X,(p)}\RightArrowOf{(x,y)}
 \StructureSheaf{X,(p)}\rightarrow \StructureSheaf{C,(p)}.
\]
Tensoring with $\StructureSheaf{\Sigma,(p)}$ we get the complex
\[
\StructureSheaf{\Sigma,(p)}\RightArrowOf{(\bar{y},-\bar{x})} \StructureSheaf{\Sigma,(p)}\oplus \StructureSheaf{\Sigma,(p)}\RightArrowOf{(\bar{x},\bar{y})}\StructureSheaf{\Sigma,(p)}
\]
whose homology sheaves represent the stalk of $\SheafTor_{-k}(\StructureSheaf{C},\StructureSheaf{\Sigma})$.
The left arrow is injective, since at least one of the restrictions $\bar{x}$ and $\bar{y}$ of $x$ and $y$ to $\Sigma$ does not vanish. Hence, $\SheafTor_{-k}(\StructureSheaf{C},\StructureSheaf{\Sigma})$ vanishes, for $k\geq 2$.

Choose next a regular sequence $(x',y')$ in the stalk $\Ideal{\Sigma,(p)}$ and consider the diagram, whose top row is a locally free  resolution.
\[
\xymatrix{
0\ar[r] & \StructureSheaf{X,(p)}\ar[r]^-{(y',-x')} \ar[d]^{=}& \StructureSheaf{X,(p)}\oplus \StructureSheaf{X,(p)}\ar[r] \ar[d]^{=}&\Ideal{\Sigma,(p)}\ar[r] \ar[d]&0
\\
0\ar[r] & \StructureSheaf{X,(p)}\ar[r]^-{(y',-x')} & \StructureSheaf{X,(p)}\oplus \StructureSheaf{X,(p)}\ar[r]^-{(x',y')} &
\StructureSheaf{X,(p)}  
}
\]
Tensoring with $\StructureSheaf{C,(p)}$ we get again that $(\bar{y}',-\bar{x}')$ is injective and its cokernel is the stalk of $\Ideal{\Sigma}\otimes\StructureSheaf{C}$.
% $\SheafTor_{-1}(\StructureSheaf{C},\StructureSheaf{\Sigma})$ as a subsheaf,  . 
We conclude that the stalk at $p$ of 
$\SheafTor_{-1}(\StructureSheaf{C},\StructureSheaf{\Sigma})$, which is $\ker(\bar{x}',\bar{y}')/Im(\bar{y}',-\bar{x}'))$, 
embeds as a submodule  of the stalk of $\Ideal{\Sigma}\otimes\StructureSheaf{C}$, which is the cokernel of $(\bar{y}',-\bar{x}')$. This completes the proof of part (\ref{lemma-item-Tor-1-embeds-in-tensor-product}).

\underline{Step 2:} We prove next the vanishing of $\SheafTor_{-k}(\StructureSheaf{C},\Ideal{\Sigma})$, for $k\geq 1$.
The exactness of the sequence (\ref{eq-short-exact-sequence-of-tensor-product-of-two-ideal-sheaves}) follows.
Tensoring (\ref{eq-short-exact-sequence-of-ideal-sheaf-of-C}) with a coherent sheaf $F$ we get the long exact sequence
%\begin{equation}
%\label{eq-long-exact-sequence-of-torsion-sheaves}
\[
\xymatrix{
& \cdots \ar[r]& 0 \ar[r]& \SheafTor_{-2}(\StructureSheaf{C},F)
\\
\ar[r]&\SheafTor_{-1}(\Ideal{C},F) \ar[r] & 0 \ar[r] & \SheafTor_{-1}(\StructureSheaf{C},F) 
\\
\ar[r] & \Ideal{C}\otimes F \ar[r] & \StructureSheaf{X}\otimes F \ar[r] & \StructureSheaf{C}\otimes F \ar[r] & 0.
}
\]
%\end{equation}
We get the isomorphisms $\SheafTor_{-k}(\Ideal{C},F)\rightarrow \SheafTor_{-k-1}(\StructureSheaf{C},F)$, for $k\geq 1$,
and the analogous isomorphism for $\Sigma$.
%Similarly, we get the isomorphisms
\begin{eqnarray}
\label{eq-isomorphisms-of-torsion-sheaves1}
\SheafTor_{-k}(\Ideal{C},F)&\IsomRightArrow &\SheafTor_{-k-1}(\StructureSheaf{C},F),
\\
\label{eq-isomorphisms-of-torsion-sheaves2}
\SheafTor_{-k}(\Ideal{\Sigma},F)&\IsomRightArrow& \SheafTor_{-k-1}(\StructureSheaf{\Sigma},F),
\end{eqnarray}
for $k\geq 1$. Hence, we have
\[
\SheafTor_{-k}(\StructureSheaf{C},\Ideal{\Sigma})\cong 
\SheafTor_{-k}(\Ideal{\Sigma},\StructureSheaf{C})\stackrel{(\ref{eq-isomorphisms-of-torsion-sheaves2})}{\cong} \SheafTor_{-k-1}(\StructureSheaf{\Sigma},\StructureSheaf{C})=0,
\]
for $k\geq 1$, where the right vanishing is by Step 1.

\underline{Step 3:}
The isomorphism (\ref{eq-isomorphisms-of-torsion-sheaves1}) with $F=\Ideal{\Sigma}$ yields
%long exact sequence (\ref{eq-long-exact-sequence-of-torsion-sheaves}) with $F=\Ideal{\Sigma}$.
$\SheafTor_{-k}(\Ideal{C},\Ideal{\Sigma})\cong \SheafTor_{-k-1}(\StructureSheaf{C},\Ideal{\Sigma})$, for $k\geq 1$,
and the right hand side vanishes, by Step 2. Hence, $\SheafTor_{-k}(\Ideal{C},\Ideal{\Sigma})$ vanishes, for $k>0$.
\end{proof}

%*****************************************************************
%
%*****************************************************************
\section{Notation}

\hspace{1ex}
\small

\begin{longtable}{l l l}
\\
$X$ & a complex abelian variety & 
\\
$\hat{X}$ & the abelian variety $\Pic^0(X)$ dual to $X$ &
\\
$\P$ & the Poincar\'{e} line bundle over $X\times\hat{X}$. & 
\\
$V$ & the lattice $H^1(X,\Integers)\oplus H^1(\hat{X},\Integers)$ & 
\\
$V_{\bullet}$ & the vector space $V\otimes_\ZZ\bullet$ over a field $\bullet$ &
\\
$(\bullet,\bullet)_V$ & the bilinear pairing on the lattice $V$ & (\ref{eq-pairing-on-V})
\\
$\Spin(V)$ & the subgroup of $\Spin(V_\QQ)$ preserving the lattice $V$ &
\\
$S$ & $H^*(X,\ZZ)$ regarded as the spin representation of $\Spin(V)$ &
\\
$S^+$ & $H^{ev}(X,\ZZ)$ regarded as the half spin representation of $\Spin(V)$
\\
$S^-$ & $H^{odd}(X,\ZZ)$ regarded as the half spin representation of $\Spin(V)$
\\
$(\bullet,\bullet)_S$ & the Mukai pairing on the spin representation & (\ref{eq-Mukai-pairing})
\\
$\tau$ &  the main anti-automorphism of $S$ & (\ref{eq-Mukai-pairing})
\\
$\tau_{\bullet}$ & translation automorphism of $X$ or $X\times\hat{X}$ by an element $\bullet$
\\
$C(V)$ & the Clifford algebra of $V$ & (\ref{eq-Clifford-relation})
\\
$G(V)$ & the Clifford group & Sec. \ref{sec-clifford-algebra}
\\
$\rho$ & the standard representation of $G(V)$ and of $\Spin(V)$ & Sec. \ref{sec-clifford-algebra}
\\
$\rho$ & denotes also the extension of $\rho$ to a representation on $\wedge^*V$ & (\ref{eq-rho-extended-to-exterior-algebra})
\\
$\rho'$ & a representation of $\Spin(V)$ on $\wedge^*V$ & (\ref{eq-rho-prime})
\\
$m$ & the spin representation of $C(V)$ and $\Spin(V)$ & Sec. \ref{sec-clifford-algebra}
\\
$m_\bullet$ & value of $m$ at $\bullet\in \Spin(V)$ & 
\\
$m_{\bullet}^\dagger$ & $\tau m_{\bullet}\tau$ & (\ref{eq-m-dagger})
\\
$IGr(2n,V_\CC)$ & the maximal isotropic grassmannian of $V_\CC$ & Sec. \ref{sec-pure-spinors}
\\
$IGr_+(2n,V_\CC)$ & the even connected component of $IGr(2n,V_\CC)$ & Sec. \ref{sec-pure-spinors}
\\
$P$ & a plane in $S^+_\QQ$ corresponding to a secant to $IGr_+(2n,V_\CC)$ & Sec. \ref{sec-pure-spinors}
\\
$\Spin(V)_P$ & subgroup of $\Spin(V)$ leaving invariant all elements of $P\subset S^+_\QQ$ & (\ref{eq-Spin(V)_P})
\\
$K$ & an imaginary quadratic number field & \hspace{1ex}
\\
$\Spin(V_K)_{\ell_1,\ell_2}$ & subgroup of $\Spin(V_K)$ leaving invariant two pure spinors  & (\ref{eq-spin-V-K-ell-1-ell-2})
\\
$\det_i$ & a $K^\times$ valued character of $\Spin(V_K)_{\ell_1,\ell_2}$, $i=1,2$ & Sec \ref{sec-pure-spinors}
\\ 
$F^\bullet(\wedge^*V)$ & the increasing filtration $\oplus_{i=0}^\bullet\wedge^iV$ of $\wedge^*V$ & Sec. \ref{subsection-the-isomorphism-tilde-varphi}
\\
$F_\bullet(\wedge^*V)$ & the decreasing filtration $\oplus_{i\geq\bullet}\wedge^iV$ of $\wedge^*V$ & (\ref{eq-decreasing-weight-filtration})
\\
$HW_P$ & $2$-dimensional subspace of $S^+_\QQ\otimes S^+_\QQ$ corresponding to $\ell_1^2\oplus \ell_2^2$ & Sec. \ref{sec-HW-classes-from-squares-of-pure-spinors}
\\
$\hat{HW}$ &  subspace of Weil classes in the middle cohomology& Sec. \ref{sec-introduction-abelian-varieties-of-Weil-type}
\\
$\hat{HW}_P$ &  subspace of Weil classes in  $H^{2n}(X\times\hat{X},\QQ)$ & Sec. \ref{sec-abelian-varieties-of-Weil-type-introduction}
\\
$\cm$ & an embedding of $K$ in $\End(V_\QQ)$ & (\ref{eq-action-by-imaginary-quadratic-number-field})
\\
$\Xi_P$ & a $(1,1)$ form on $X\times \hat{X}$ & (\ref{eq-Xi})
\\
$\varphi$ & the isomorphism of $S\otimes_\ZZ S\rightarrow C(V)$ & (\ref{eq-Chevalley-varphi}) 
\\
$\tilde{\varphi}$ & the isomorphism of $S\otimes_\ZZ S\rightarrow \wedge^*V$ & (\ref{eq-tilde-varphi})
\\
$SO_+(V)$ & the kernel of the spinor norm $SO(V)\rightarrow \{\pm 1\}$,  $=\rho(\Spin(V))$  & %(\ref{})
\\
$SO_+(V_\QQ)_f$ & the image of $\Spin(V_\QQ)_P$ in $SO_+(V_\QQ)$, here $f=\cm(\sqrt{-d})$ & (\ref{eq-so-f})
\\
$H(\bullet,\bullet)$ & a $K$-valued Hermitian form on $V_\QQ$ & (\ref{eq-H})
\\
%$\Xi_P$ & a $(1,1)$-form on $V_\RR/V$ & (\ref{eq-Theta-P})
%\\
$\Omega_P$ & a period domain of abelian varieties of Weil type & (\ref{eq-Omega})
\\
$\nu(I)$ & a discrete invariant of a complex structure $I$ in $\rho(\Spin(V_\RR)_P)$ & Lem. \ref{lemma-3-dimensional-eigenvalues}
\\
$\nu$ & the isomorphism $H^*(X\times X,\ZZ)\rightarrow H^*(\hat{X}\times X)$ & Lem. \ref{lemma-nu-equal-tilde-varphi}
\\
$\Phi_{\bullet}$ & integral transform $D^b(X)\rightarrow D^b(Y)$ with kernel $\bullet\in D^b(X\times Y)$ & Sec. \ref{sec-cohomological-action-of-derived-equivalences}
\\
$\Psi_{\bullet}$ &integral transform $D^b(Y)\rightarrow D^b(X)$ with kernel $\bullet\in D^b(X\times Y)$ & Sec. \ref{sec-cohomological-action-of-derived-equivalences}
\\
$\Phi$ & Orlov's derived equivalence $\Phi:D^b(X\times X)\rightarrow D^b(X\times\hat{X})$ & (\ref{eq-Orlov-derived-equivalence-from-XxX-to-X-times-hat-X})
\\
$\phi$ & the isomorphism $S\otimes_\ZZ S\rightarrow \wedge^*V$ induced by $\Phi$ & (\ref{eq-Orlov-cohomological-isomorphism})
\\
$\kappa(\bullet)$ & the characteristic class $ch(\bullet)\exp(-c_1(\bullet))$ & Def. \ref{def-kappa-B}
\\
$at_\bullet$ & the Atiyah class of a coherent (possibly twisted) sheaf $\bullet$ & Def. \ref{def-atiyah-class-of-twisted-sheaf}
\\
$\sigma$ & in Sections \ref{section-abelian-2n-folds-of-Weil-type-from-rational-secants} and \ref{sec-Hermitian-form}:  Galois involution of the number field $K$ &
\\
$\sigma$ & in Sections \ref{section-semiregular-twisted-sheaves} to \ref{section-Jacobians-of-genus-3-curves} it denotes the semi-regularity map & (\ref{eq-semiregularity-map})
\\
$HH^i(X)$ & the $i$-th Hochschild cohomology of $X$ & Sec. \ref{sec-rank-6-obstruction-map}
\\
$ev_F$ & the evaluation homomorphism $ev_F:HH^*(X)\rightarrow \Ext^*(F,F)$ & (\ref{eq-ev-F})
\\
$ob_F$ & $ev_F:HH^2(X)\rightarrow \Ext^2(F,F)$ & Sec. \ref{sec-rank-6-obstruction-map}
\\
$HT^i(X)$ & $\oplus_{j=0}^i H^j(X,\wedge^{i-j}TX)$ & Sec. \ref{sec-rank-6-obstruction-map}
\\
$\iota$ & in Section \ref{section-Jacobians-of-genus-3-curves}: the involution of $X=\Pic^2(C)$, $L\mapsto \omega_C\otimes L^{-1}$ & Sec. \ref{section-Jacobians-of-genus-3-curves}
\\
$\iota$ & our default notation for an inclusion or an embedding

\end{longtable}

\normalsize
%*****************************************************************
%
%*****************************************************************
{\bf Acknowledgements:}
This work was partially supported by a grant  from the Simons Foundation (\#962242). 
I thank Nick Addington, Alexander Perry and Jonathan Pridham for helpful communications.

%****************************************************************
% Bibliography:
%****************************************************************


\begin{thebibliography}{B-N-R}

\bibitem[Ab]{abuaf} Abuaf, R.:
{\em On quartic double fivefolds and the matrix factorization of exceptional quaternionic representations.\/} arXiv:1709.05217v2.

\bibitem[Ar]{artamkin} Artamin, I. V.: {\em On deformation of sheaves.\/} Math. USSR Izvestiya Vol. 32(1989), No. 3.

\bibitem[At]{atiyah} Atiyah, M. F.: {\em Complex analytic connections in fibre bundles.\/}
Transactions of the AMS , 1957, Vol. 85, No. 1, pp. 181--207.

\bibitem[B]{beauville} Beauville, A.: {\em Fibr\'{e}s de rang deau sur une courbe, fibr\'{e} d\'{e}terminant et functions th\^{e}ta. II.\/}
Bull. de la S. M. F. tome 119, No 3 (1991), p. 259--291.

\bibitem[BF1]{buchweitz-flenner} Buchweitz, R-O., Flenner, H.: {\em A semi-regularity map for modules and applications to deformations.\/} Compositio Math. 137(2003), no.2, 135--210.

\bibitem[BF2]{buchweitz-flenner-HH} Buchweitz, R-O., Flenner, H.: {\em The global decomposition theorem for Hochschild (co-)homology of singular spaces via the Atiyah-Chern character.\/} Adv. Math.   217 (2008), no. 1, 243--281.

\bibitem[BL]{BL} Birkenhake, C., Lange, H.: {\em Complex abelian varieties.\/} 2nd Edition, Springer (2010). 


\bibitem[Ca1]{caldararu-thesis} C\u{a}ld\u{a}raru, A.:
{Derived categories of twisted sheaves on
Calabi-Yau manifolds.\/} Thesis, Cornell Univ., May 2000.


\bibitem[Ca2]{caldararu-I} C\u{a}ld\u{a}raru, A.:
{\em The Mukai pairing I. The Hochschild structure.\/} Preprint arXiv:0308079v2.

\bibitem[Ca3]{caldararu-II} C\u{a}ld\u{a}raru, A.:
{\em The Mukai pairing II: the Hochschild-Kostant-Rosenberg isomorphism. \/}
Adv. in Math. 194 (2005) 34--66.



\bibitem[CBR]{calaque-et-al} Calaque, D., Rossi, C., Van den Bergh, M.:
{\em C\u{a}ld\u{a}raru conjecture and Tsygan formality.\/} Ann. of Math. (2) 176 (2012), no. 2, 865--923. 

\bibitem[Ch]{chevalley} Chevalley, C.:
{\em The algebraic theory of spinors.\/}
Columbia Univ. Press 1954.

\bibitem[Ce]{ceresa} Ceresa, G.: {\em $C$ is not algebraically equivalent to $C^{-}$ in its  Jacobian.\/} Ann. of Math. (2)  117  (1983), no.2, 285--291.

\bibitem[DM]{deligne-milne} Deligne, P., Milne, J.S.: {Hodge cycles on abelian varieties.\/} in {\em Hodge cycles, motives, and Shimura varieties\/}
Deligne, Pierre; Milne, James S.; Ogus, Arthur; Shih, Kuang-yen
Lecture Notes in Math., 900
Springer-Verlag, Berlin-New York, 1982. A revised version is available at \url{https://www.jmilne.org/math/Documents/index.html}

\bibitem[G]{gulbrandsen} Gulbrandsen, M.:
{\em Vector bundles and monads on abelian threefolds.\/} Comm. Algebra 41 (2013), no. 5, 1964--1988.

\bibitem[GLO]{golyshev-luntz-orlov}
Golyshev, V., Luntz, V., and Orlov, D.:
{\em Mirror symmetry for abelian varieties.\/}
J. Alg. Geom. 10 (2001) 433-496.

\bibitem[vG1]{van-Geemen} van Geemen, B.: {\em An introduction to the Hodge Conjecture for abelian varieties.\/}
Algebraic cycles and Hodge theory (Torino, 1993), 233--252,
Lecture Notes in Math., 1594, Springer, Berlin, 1994. 

\bibitem[vG2]{van-Geemen-theta} van Geemen, B.
{\em Theta functions and cycles on some abelian fourfolds.\/}
Math. Z. 221 (1996), no. 4, 617--631. 

\bibitem[Ha]{hartshorne} Hartshorne, R.: {\em Algebraic geometry.\/} 
Grad. Texts in Math., No. 52 Springer-Verlag, New York-Heidelberg, 1977.
 
\bibitem[H1]{huybrechts-complex-geometry-book} Huybrechts, D.:
{\em Complex Geometry, An Introduction.\/} Springer-Verlag (2005).

\bibitem[H2]{huybrechts-derived-categories-book} Huybrechts, D.:
{\em Fourier-Mukai Transforms in Algebraic Geometry.\/}
Oxford University Press, 2006.

%\bibitem[H3]{huybrechts-k3-book} Huybrechts, D.: {\em Lectures on $K3$ surfaces.\/} Cambridge University Press (2016).

\bibitem[HL]{huybrechts-lehn} Huybrechts, D., Lehn, M: {\em
The geometry of moduli spaces of sheaves.\/} Second edition. Cambridge University Press, Cambridge, 2010.

\bibitem[HP]{perry-hotchkiss} Hotchkiss, J., Perry, A.: {\em The period-index conjecture for abelian threefolds and Donadldson-Thomas theory.\/} arXiv:2405.03315v2.

\bibitem[Hua]{Huang} Huang, S.: {\em A note on a question of Markman.\/} J. Pure Appl. Algebra 225 (2021), no. 9.

\bibitem[I]{igusa} Igusa, J-I.: {\em A Classification of Spinors Up to Dimension Twelve.\/} Amer. J. of Math. Vol. 92, No. 4, 997--1028 (1970).

\bibitem[dJ]{de-Jong} de Jong, A. J.: {\em A result of Gabber.\/} Preprint, 
\url{https://www.math.columbia.edu/~dejong/}

\bibitem[K]{koike} Koike, K.: {\em Algebraicity of some Weil Hodge classes.\/} Canad. Math. Bull.47(2004), no.4, 566--572.

%\bibitem[La]{lam} Lam, T. Y.: {\em Introduction to quadratic forms over fields.\/}
%Grad. Stud. Math., 67 American Mathematical Society, Providence, RI, 2005.

\bibitem[Li]{lieblich} Lieblich, M.:
{\em Compactified moduli of projective bundles.\/} Algebra Number Theory 3 (2009), no.6, 653--695.

\bibitem[M1]{markman-BBF} Markman, E.:
{\em The Beauville-Bogomolov class as a characteristic class.\/}  J. of Algebraic Geometry, 29 (2020) 199--245.
%Preprint arXiv:1105.3223.v4


\bibitem[M2]{markman-generalized-kummers}  Markman, E.: {\em The monodromy of generalized Kummer varieties and algebraic cycles on their intermediate Jacobians.\/} J. Eur. Math. Soc. (JEMS) 25 (2023), no. 1, 231--321.


\bibitem[Mi1]{milne} Milne, J.:{\em \'{E}tale Cohomology.\/} Princeton Math. Ser., No. 33 Princeton University Press, Princeton, NJ, 1980.

%\bibitem[Mi2]{milne-algebraic-groups} Milne, J.:{\em Algebraic Groups.\/} Cambridge Stud. Adv. Math., 170
%Cambridge University Press, Cambridge, 2017. 

%\bibitem[MM]{markman-mehrotra} Markman, E., Mehrotra, S.:
%{\it Integral transforms and deformations of $K3$ surfaces.\/} Electronic preprint, arXiv: 1507.03108v1.

\bibitem[MZ1]{Moonen-Zarhin-Weil-Hodge-Tate-classes} Moonen, B., Zarhin, Y.: {\em Hodge classes and Tate classes on simple abelian fourfolds.\/} Duke Math. J. Vol. 77 No 3, (1995) 553-581.

\bibitem[MZ2]{Moonen-Zarhin-Weil-classes} Moonen, B., Zarhin, Y.: {\em Weil classes on abelian varieties.\/} J. reine angew. Math. 496 (1998), 83--92.

\bibitem[MZ3]{moonen-zarhin-low-dimension} Moonen, B., Zarhin, Y.: {\em Hodge classes on abelian varieties of low dimension.\/} Math. Ann. 315, 711-733 (1999).

\bibitem[Mu1]{mukai-duality} Mukai, S.:
{\em Duality between $D(X)$ and $D(\hat{X})$ with its application to Picard sheaves.\/} Nagoya Math. J. 81 (1981), 153--175.

\bibitem[Mu2]{mukai-spin} Mukai, S.: 
{\em Abelian varieties and spin representations.\/} Preprint of
Warwick Univ. (1998)(English translation from Proceedings of the symposium
``Hodge Theory and algebraic geometry'', Sapporo, 1994, pp. 110-135).

%\bibitem[Mu3]{mukai-fourier-functor-and-its-applications} Mukai, S.: 
%{\em  Fourier functor and its application to the moduli of bundles on an abelian variety.\/}
%Adv. Studies in Pure Math., 10, 515--550 (1987).

\bibitem[Mu4]{mukai-semihomogeneous} Mukai, S.: 
{\em  Semi-homogeneous vector bundles on an Abelian variety.\/}
J. Math. Kyoto Univ. 18 (1978), no. 2, 239--272.

\bibitem[Mum]{mumford-abelian-varieties} Mumford, D.: {Abelian varieties.\/} Tata Inst. Fund. Res. Stud. Math., 5 
Published for the Tata Institute of Fundamental Research, Bombay; by Hindustan Book Agency, New Delhi, (2008).

\bibitem[NR]{NR} Narasimhan, M. S., Ramanan, S: {\em $2\theta$-Linear systems on abelian varieties.\/} Vector bundles on algebraic varieties (Bombay, 1984), 415--427.
Tata Inst. Fund. Res. Stud. Math., 11
Published for the Tata Institute of Fundamental Research, Bombay, 1987.

\bibitem[Ob]{obata} Obata, M.: {\em On n-dimensional homogeneous spaces of Lie groups of dimension greater than n(n-1)/2.\/}
 J. Math. Soc. Japan 7 (1955), 371--388.
 
 
 \bibitem[O'G]{ogrady} O'Grady, K.: 
{\em Compact tori associated to hyperkaehler manifolds of Kummer type.\/}
Int. Math. Res. Notices (2021), 12356--12419.

\bibitem[Or]{orlov-abelian-varieties} Orlov, D. O.: 
{\em 
Derived categories of coherent sheaves on abelian varieties and equivalences between them.\/} Izv. Math. 66 (2002), no. 3, 569--594. 

\bibitem[Pa]{pauly} Pauly, C.: {\em Self-Duality of Coble's Quartic hypersurface and applications.\/}
Michigan Math. J. 50 (2002).

\bibitem[Pr]{pridham} Pridham, J.: {\em Semiregularity as a consequence of Goodwillie's theorem.\/} Electronic preprint arXiv:1208.3111.v4.
%\url{https://www.maths.ed.ac.uk/~jpridham/semireg2.pdf}

\bibitem[R]{ramon-mari} Ram\'{o}n Mari, J.: 
{\em On the Hodge conjecture for products of certain surfaces.\/} Collect. Math. 59 (2008), no. 1, 1--26. 

\bibitem[S1]{schoen1}
Schoen, C.: {\em Hodge classes on self-products of a variety with an
 automorphism.\/}  Compositio Math. 65 (1988), no. 1, 3--32. 

\bibitem[S2]{schoen}
Schoen, C.: {\em Addendum to: ``Hodge classes on self-products of a variety with an
 automorphism''.\/} Compositio Math. 114 (1998), no. 3, 329--336.


\bibitem[Siu]{siu} Siu, Y.: 
{\em Extension of locally free analytic sheaves.\/} Math. Ann. 179, 285--294 (1969).


\bibitem[T]{toda} Toda, Y.: {\em Deformations and Fourier-Mukai transforms.\/}
J. Differential Geom. 81 (2009), no. 1, 197--224.

\bibitem[TT]{trautman} Trautman, A., Trautman, K: {\em Generalized pure spinors.\/} J. Geom. Phys. 15 (1994) 1--22.

\bibitem[Ve]{verbitsky} Verbitsky, M.: {\em Ergodic complex structures on hyperkahler manifolds: an erratum.\/} arXiv.math:1708.05802. 

\bibitem[Vo]{voisin} Voisin, C.: {\em The Hodge Conjecture.\/} Open problems in mathematics, 521--543. Springer,  2016.

\bibitem[W]{weil} Weil, A.: {\em Abelian varieties and the Hodge ring.\/} Collected papers, Vol. III, 421--429. Springer Verlag (1980).


\end{thebibliography}
\end{document}